%% file: ThurstonNormViaSpunNormalImmersions.tex
\documentclass[12pt]{gtart}

\usepackage[inner=20mm, outer=20mm, textheight=245mm]{geometry}

\usepackage{amsmath, amssymb, graphicx, epstopdf}
\usepackage{latexsym}
\usepackage{amscd}
\usepackage{euscript}
\usepackage{enumitem}
\usepackage{epsfig}
\usepackage{psfrag}
\usepackage[makeroom]{cancel}
\usepackage{caption}
\usepackage{multirow}
\usepackage{multicol}
\usepackage[c]{esvect}
\usepackage{stix}
\usepackage[colorlinks,citecolor=blue,linkcolor=blue, backref=page]{hyperref}
\usepackage[nameinlink]{cleveref}
\usepackage{comment}
\usepackage{relsize}

\usepackage{listings}
\usepackage{color}

\usepackage{verbatimbox}
\newcommand\Includegraphics[2][]{\addvbuffer[3pt 0pt]{\includegraphics[#1]{#2}}}

\usepackage[utf8]{inputenc}

\DeclareFixedFont{\ttb}{T1}{txtt}{bx}{n}{12} 
\DeclareFixedFont{\ttm}{T1}{txtt}{m}{n}{12}  

\usepackage{color}
\definecolor{deepblue}{rgb}{0,0,0.5}
\definecolor{deepred}{rgb}{0.6,0,0}
\definecolor{deepgreen}{rgb}{0,0.5,0}
\definecolor{dkgreen}{rgb}{0,0.6,0}
\definecolor{gray}{rgb}{0.5,0.5,0.5}
\definecolor{mauve}{rgb}{0.58,0,0.82}

\usepackage{listings}
\usepackage{textcomp}

\newcommand\pythonstyle{\lstset{
language=Python,
upquote=True,
frame=leftline,
aboveskip=3mm,
belowskip=3mm,
showstringspaces=false,
columns=flexible,
basicstyle={\footnotesize\ttfamily},
morekeywords={self,sage},              
keywordstyle=\color{blue},
emph={MyClass,__init__},          
emphstyle=\ttb\color{deepred},    
stringstyle=\color{dkgreen},                        
showstringspaces=false,
commentstyle=\color{gray},
breaklines=true,
breakatwhitespace=true,
tabsize=4,
xleftmargin=15pt,
xrightmargin=15pt,
deletekeywords={from, not},
deletekeywords=[2]{any,enumerate}
}}

\lstnewenvironment{python}[1][]
{
\pythonstyle
\lstset{#1}
}
{}

\renewcommand{\tt}{\texttt}

\usepackage{subcaption}
\usepackage{color}
\def\tet{\tau}
\def\face{t}

\def\simplex{\Delta}

\theoremstyle{plain}
\newtheorem{theorem}{Theorem}
\newtheorem{lemma}[theorem]{Lemma}
\newtheorem{proposition}[theorem]{Proposition}
\newtheorem{corollary}[theorem]{Corollary}

\newtheorem{question}[theorem]{Question}

\theoremstyle{definition}
\newtheorem{definition}[theorem]{Definition}
\newtheorem*{definition*}{Definition}

\theoremstyle{remark}

\newtheorem{remark}[theorem]{Remark}

\numberwithin{equation}{section}

\definecolor{bettergreen}{rgb}{0,0.6,0.4}

\def\Q{\mathbb{Q}} 
\def\QQ{\mathbb{Q}}
 
\def\RR{\mathbb{R}}
\def\SS{\mathbb{S}}
\def\ZZ{\mathbb{Z}}
\def\NN{\mathbb{N}}

\def\Quad{q}
\def\v{\mathfrak{v}}
\def\t{\mathfrak{t}}
\def\a{\mathfrak{a}}

\def\tri{\mathcal{T}}

\def\simplex{\Delta}

\def\N{\mathcal{N}}

\def\R3{\mathbb{R}^3}

\def\qtons{\vv{Q}}
\def\Pqtons{P\vv{Q}}
\def\qns{Q}
\def\LT{\mathfrak{D}}

\DeclareMathOperator{\Int}{int}
\DeclareMathOperator{\im}{im}

\def\define{\textbf}
\def\del{\partial}

\def\hom{\mathcal{H}}

\DeclareMathOperator{\wt}{wt}

\def\as{\alpha}
\def\B{\mathcal{B}}
\def\e{\mathfrak{e}}

\def\Mod{\mathrm{Mod}}
\def\F{\mathcal{F}}
\def\genus{\mathrm{genus}}

\newcommand{\hull}[1]{\widehat{#1}}

\newcommand{\itoverbar}[1]{\,\overline{\!{#1}}}

\newcommand{\Mbar}{\itoverbar{M}}
\newcommand{\Nbar}{\itoverbar{N}}

\begin{document}

\title{The Thurston norm via spun-normal immersions}
\author{Daryl Cooper, Stephan Tillmann and William Worden}


\begin{abstract}
A theory of transversely oriented spun-normal immersed surfaces in ideally triangulated 3--manifolds is developed in this paper, including linear functionals determining the boundary curves, Euler characteristic and homology class of these immersions. This is used to develop and implement an algorithm to compute the unit ball of the Thurston norm for cusped hyperbolic 3--manifolds of finite volume. As an application of independent interest, we give an upper bound on the minimal entropy of pseudo-Anosov maps of surfaces with number of cusps bounded linearly in genus.
\end{abstract}







\makeshorttitle


\input{intro02.tex}
\input{tons04.tex}
\input{angles01.tex}
\input{homology02.tex}
\input{algo01.tex}

\input{implementation03.tex}
\input{apps02}


\subsection*{Acknowledgements} 

The authors thank Eriko Hironaka, Craig Hodgson, Saul Schleimer, Henry Segerman and Mehdi Yazdi for comments and discussions related to this work.
Research of the third author is supported in part under the Australian Research Council's ARC Future Fellowship FT170100316.

%
%

\bibliographystyle{amsplain}
\bibliography{thurston}

\bigskip


\address{Department of Mathematics, University of California Santa Barbara, CA 93106, USA}
\email{cooper@math.ucsb.edu}

\address{School of Mathematics and Statistics, The University of Sydney, NSW 2006, Australia} 
\email{stephan.tillmann@sydney.edu.au} 

\address{Department of Mathematics, Rice University, Houston, TX 77005, USA}
\email{william.worden@rice.edu}

\Addresses

\input{appendix.tex}

\end{document}

%% file: intro02.tex


\section{Introduction}\label{sec:intro}

Let $M$ be a cusped hyperbolic 3--manifold of finite volume. The motivation for this paper is to use an ideal triangulation of $M$ and spun-normal surfaces to give an algorithm to compute the unit ball of the Thurston norm of $M.$ A similar algorithm for closed 3--manifolds using closed normal surfaces was given by the first two authors in \cite{Cooper-norm-2009}, and finer structure results were obtained by Tollefson and Wang~\cite{tollefson96-taut}.

To begin with, let $M$ be the interior of the compact, orientable 3--manifold $\overline{M}$ with boundary a finite union of tori, and let $\tri$ be an arbitrary ideal triangulation of $M.$ Our main technical contribution is the realisation of certain transversely oriented normal $Q$--coordinates by transversely oriented immersed normal surfaces (see \Cref{sec:spun_normal}). This requires an adaptation of the \emph{neat position} of \cite{Cooper-norm-2009}. The coordinates have the property that surfaces with the same coordinate may be topologically distinct and may have different boundary slopes, but they always have the same image in second homology and the same Euler characteristic. The $Q$--coordinates form a positive linear cone $\qtons(\tri)$ defined by linear equations arising from \emph{shift conditions}, analogous to the conditions given by Dunfield and Garoufalidis~\cite{Dunfield-incompressibility-2012} in the standard setting.
Based on the discussion by Garoufalidis, Hodgson, Hoffman and Rubinstein, we describe an Euler characteristic function on the space of spun-normal surfaces. This is given in \Cref{sec:euler} using the existence of a generalised angle structure with special properties. 

We describe a linear map $\hom\co\qtons(\tri)\to H_2(\Mbar,\del\Mbar;\RR)$ in \Cref{subsec:simplicial}, which takes the coordinate of a transversely oriented immersion to the corresponding homology class. 
Under the additional assumption that $M$ is the complement of a link in a homology sphere, we show in \Cref{subsec:peripheral} how to determine the homology map using the boundary map of \cite{tillmann08-normal}. Finally, an algorithm to compute the unit ball of the Thurston norm is devised in \Cref{sec:algorithm}, resulting in our main theorem as follows:

\begin{theorem}\label{thm:main}
	Let $\Mbar$ be a compact, orientable 3-manifold with non-empty boundary and suppose that $M=\mathrm{int}(\Mbar)$ has a complete hyperbolic structure of finite volume. Suppose a $0$--efficient ideal triangulation $\tri$ of $M$ is given. Then there is an algorithm using transversely oriented normal surface theory to determine the unit ball of the Thurston norm on $H_2(\Mbar,\del \Mbar;\RR)$.
\end{theorem}

The unit norm ball can also be computed (in principle) from Heegaard Floer homology~\cite{Alishanti-bordered-2019, Lipshitz-computing-2014} or twisted Alexander polynomials~\cite{friedl-thurston-2015}. We are not aware of implementations of these algorithms. The implementation of our algorithm is based on \texttt{Regina}~\cite{burton12-regina} and \texttt{SnapPy}~\cite{SnapPy}. Notes on implementation and computational results can be found in \Cref{sec:implementation and examples}.


We apply our algorithm in \Cref{sec:applications} to obtain a result concerning the entropy of pseudo-Anosov maps of cusped surfaces that is of independent interest. This result gives a partial answer to a question posed to us by Eriko Hironaka, which was originally asked by Mehdi Yazdi (see \cite[Question 4.13]{Yaz20}).

Let $S=\Sigma_{g,n}$ be a surface of genus $g$ with $n$ punctures, and let $\varphi\in \Mod(S)$ be a pseudo-Anosov mapping class. Let $\lambda_\varphi$ denote the dilatation of $\varphi$, and define $l_{g,n}=\min\{\log(\lambda_\varphi)\mid \varphi\in \Mod(S) \text{ is pseudo-Anosov}\}$. 

Penner \cite{Pen91} has shown that $ l_{g,0}\asymp \frac{1}{g}$, and that for any surface with $\chi(S)<0$:
$$
\frac{\log 2}{12g-12+4n} \;\le\; l_{g,n}
$$ 
Combining this lower bound with the upper bound $l_{0,n}\le \frac{2\log(2+\sqrt{3})}{n-3}$ for $n\ge 4$ of Hironaka and Kin~\cite{HK06}, it follows that $ l_{0,n}\asymp \frac{1}{n}$. Note also that Penner's bound implies that $\frac{c_1}{|\chi(S)|}\le l_{g,n}$, for $c_1=\frac{\log 2}{6}$, as long as $\chi(S)<0$.

For fixed $g\ge 2$, Tsai \cite{Tsa09} has shown that $l_{g,n}\asymp \frac{\log n}{n}$. Moreover, Tsai \cite{Tsa09} shows that there exists a constant $c_2$  such that for any $g\ge 2$ and $n\ge 0$ one has
$$
l_{g,n}\le c_2\cdot g\frac{\log|\chi(S)|}{|\chi(S)|}
$$

\begin{figure}[h]
 	\centering
   	\includegraphics{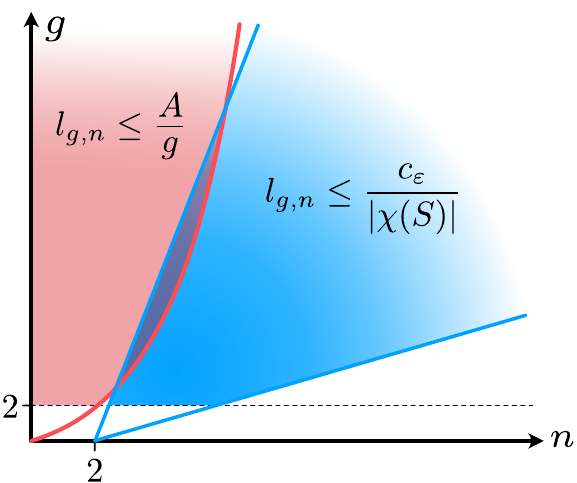}
   	\caption{Qualitative sketch of the region of the $(g,n)$-plane covered by the bound of \Cref{thm:entropy} (in blue), and the region covered by the bound of  \cite[Theorem 1.2]{Yaz20} (in red).}
   	\label{fig:regions}
\end{figure}

Work of Yazdi \cite{Yaz18} shows that for any positive real number $\alpha$, there exists a constant $c_\alpha$ such that for any $g\ge 2$, $n\ge 0$ there is the following lower bound:
$$
 \frac{c_\alpha}{g^{2+\alpha}}\frac{\log|\chi(S)|}{|\chi(S)|} \; \le \; l_{g,n}
$$
This is a strengthening of Penner's bound in the case where $n\gg g$. The above work of Tsai implies that an analogous upper bound does not hold if we allow $n$ to grow independently of $g.$ However, 
as an application of our new algorithm, we are able to make the following new observation in this line of work:

\begin{theorem}\label{thm:entropy}
Fix any $1 >  \varepsilon>0$. Then there exists a constant $c_\varepsilon$ depending on $\varepsilon$ such that for all $g\ge 2$ and for all $n$ satisfying $\varepsilon g +2 \le n \le \frac{1}{\varepsilon}g+2$ we have
$$
l_{g,n} \le \frac{c_\varepsilon}{|\chi(S)|}
$$
where $\chi(S)$ is the Euler characteristic of $S=\Sigma_{g,n}$.
\end{theorem}

This theorem complements recent work of Yazdi~\cite{Yaz20}, in which it is shown that there exist positive constants $A,B$, and $C$ so that for all $n\ge 1$ and $g\ge C n \log^2n$ one has $\frac{B}{g}\le l_{g,n}\le \frac{A}{g}$. Notably, the constants in Yazdi's bound are independent of $g$ and $n$, as is the constant $c_\epsilon$ in our bound. At first glance it may seem that our bound is an improvement on Yazdi's upper bound, but we note that the set of all $(g,n)$ to which both results apply is finite for any fixed $\epsilon$. \Cref{fig:regions} shows for fixed $\epsilon$ the regions of the $(g,n)$-plane on which each result applies.


%% file: tons04.tex
%


\section{Transversely oriented spun-normal surfaces}\label{sec:spun_normal}

A new idea in \cite{Cooper-norm-2009} for normal surfaces in closed 3--manifolds  was the use of transversely oriented normal discs in order to define an orientation on a normal surface, and hence a map to the second homology of the 3--manifold. A key result was the construction of an immersed transversely oriented normal surface for each admissible coordinate. This was achieved through a placement of the transversely oriented normal discs that was called \emph{neat position}. In this section, we generalise this set-up to spun-normal surfaces in ideally triangulated, orientable 3--manifolds with torus cusps. As in \cite{tillmann08-normal}, with a little more notation and additional cases, the results of this section can be generalised to arbitrary ends and non-orientable 3--manifolds.
All manifolds and maps are understood to be PL in this paper.


\subsection{Ideal triangulation}
\label{sec:Ideal triangulation}

The notation of \cite{Jaco-0-efficient-2003, tillmann08-normal} is used in this paper. 
Let $M$ be the interior of an orientable compact manifold with non-empty boundary consisting of a finite union of pairwise disjoint tori. 
An \define{ideal triangulation}, $\tri,$ of $M$ consists of the union $\widetilde{\Delta}=\cup_i \simplex_i^3$ of pairwise disjoint regular Euclidean 3--simplices of edge length 100, a set $\Phi$ of isometries between pairs of faces, the natural quotient map $p\co \widetilde{\Delta} \to \widetilde{\Delta} / \Phi = P,$ and a homeomorphism $h \co P \setminus P^{(0)} \to M,$ between $M$ and the complement of the 0--skeleton in $P.$ The quotient space $P$ is usually called a \define{pseudo-manifold} and the \define{end-compactification} of $M.$

For brevity, we refer to a 3--manifold $M$ imbued with an ideal triangulation $\tri=(\widetilde{\Delta}, \Phi, h)$ as an \define{ideally triangulated 3--manifold}. Throughout, we assume that $P$ and $M$ are \emph{oriented}, that $h$ is orientation preserving, and that all singular 3--simplices in $P$ are oriented coherently and the tetrahedra in $\widetilde{\Delta}$ are given the induced orientation. There is a natural bijection between tetrahedra in $\widetilde{\Delta}$ and singular 3--simplices in $P$.

A \define{normal corner} in a 1--simplex is an interior point of a 1--simplex. A \define{normal arc} in a 2--simplex is a properly embedded arc whose boundary consists of normal corners on distinct edges of the 2--simplex. A \define{normal disc} in a 3--simplex is a properly embedded triangle or quadrilateral whose vertices are normal corners on different edges of the 3--simplex. The boundary edges of a normal disc are thus normal arcs on distinct faces of the 3--simplex. The normal discs (arcs, corners) in $P$ are the images of normal discs (arcs, corners) in $\widetilde{\Delta}.$

A \define{normal isotopy} (resp.\thinspace \define{homotopy}) of $P$ is an isotopy (resp.\thinspace homotopy) keeping each (singular) simplex of each dimension invariant.

Let $S$ be a surface and $f\co S \to M$ be a proper immersion. We say that $f$ (or $S$) is \define{normal} with respect to $\tri$ if $f$ is transverse to the 2--skeleton of $\tri$ and the pull-back of the intersection with the 2--skeleton is a cell decomposition of $S$ consisting of triangles or quadrilaterals. In particular, each such 
triangle or quadrilateral in $S$ maps to a normal triangle or quadrilateral in $M,$ and the cell decomposition of each connected component of $S$ has at most finitely many quadrilaterals (but possibly infinitely many triangles). We often indicate that a normal surface has infinitely many triangles by calling it a \define{spun-normal surface}.

A \define{transversely oriented normal immersion} $(S, \nu)\to M$ is a normal immersion $f\co S\to M$ with a transverse orientation $\nu_i$ of each connected component $S_i$ of $S$. Whence each normal disc in $S$ inherits a transverse orientation. We now describe a special set of transversely oriented normal discs in a 3--simplex, before we return to a global analysis of such normal immersions.


\subsection{Neat position in a 3--simplex}

Let $\simplex^3$ be an oriented regular Euclidean 3--simplex of edge length 100. We now describe a special set consisting of infinitely many normal discs with transverse orientation in $\simplex^3.$ The concept of neat position naturally leads to a generalisation to transversely oriented spun-normal surfaces of the \emph{shift} that Dunfield-Garoufalidis~\cite{Dunfield-incompressibility-2012} introduced for ordinary spun-normal surface coordinates.


\subsubsection{Transversely oriented normal discs}

A \define{transversely oriented normal disc} $(D, \nu_D)$ in $\simplex^3$ is a normal disc $D$ with a \define{transverse orientation} $\nu_D$. We regard $\nu_{D}$ as a function on the components of $\simplex^3\setminus D$ that sends one component to $+1$ and the other to $-1.$ The former component is denoted by $\simplex^3(D, \nu_D)$ and the maximal subsimplex of $\simplex^3$ contained in it is \define{dual} to $(D, \nu_D)$ and denoted $\sigma(D, \nu_D)$. Analogous notions and notation apply to \define{transversely oriented normal arcs} on 2--simplices and \define{transversely oriented normal corners} on 1--simplices. Moreover, the transverse orientation of $D$ induces transverse orientations on the normal arcs and normal corners in the boundary of $D.$

Two transversely oriented normal discs are equivalent if there is a normal isotopy of $\simplex^3$ that takes one disc with its transverse orientation to the other. An equivalence class is called a \define{transversely oriented disc type}. There is a natural bijective correspondence between transversely oriented normal disc types and non-empty, proper subsimplices of $\Delta.$ Hence there are $14$ equivalence classes. If $\sigma(D, \nu_D)$ is a 0--simplex (resp.\thinspace 2--simplex), then  $(D, \nu_D)$ is a \define{small} (resp.\thinspace \define{large}) \define{triangle}. 

When we refer to a \define{normal isotopy class} it is often convenient to be able to refer to it \emph{as a} corner, arc or disc in the given simplex. As $\simplex^3$ is a regular Euclidean simplex, a canonical choice is to represent the class of a normal corner by the barycentre of a 1--simplex, and the class of any other normal object as the Euclidean convex hull of its vertices, all placed at barycentres of 1--simplices. In particular, a normal triangle type is represented by an equilateral triangle and a normal quadrilateral type by a square. \emph{From now on, whenever we refer to a normal isotopy class, we refer to its canonical representative.}


\subsubsection{Neat position}
\label{sec:neat in simplex}

We say that $(D, \nu_D)$ is at distance $k$ from $\sigma(D, \nu_D)$ if $D$ is the intersection of $\simplex^3$ with a plane and each vertex of $D$ is at distance $k$ from a 0--simplex of $\sigma(D, \nu_D)$, see \Cref{fig:distance_k}.

\begin{figure}[h]
    \centering
    \begin{subfigure}{.35\textwidth}
    	\vspace{.2cm}\quad\\
        \includegraphics[scale=.7]{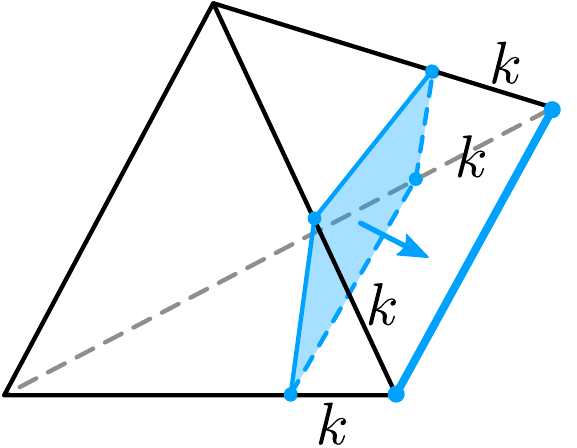}\\
        \vspace{.4cm}\quad\\
        \includegraphics[scale=.7]{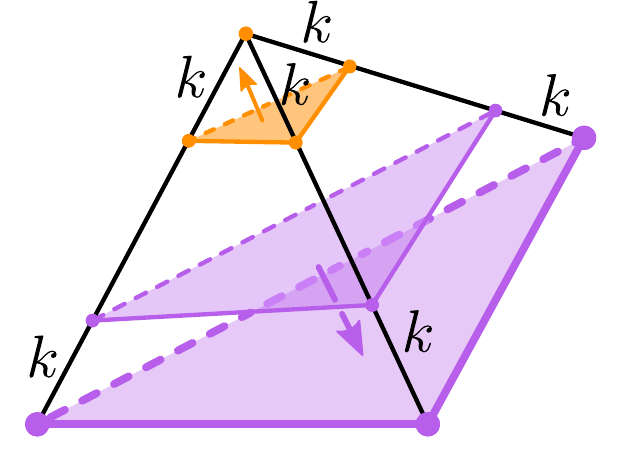}
        \caption{Discs at distance $k$.}
        \label{fig:distance_k}
    \end{subfigure}
    \begin{subfigure}{.63\textwidth}
        \includegraphics[scale=.9]{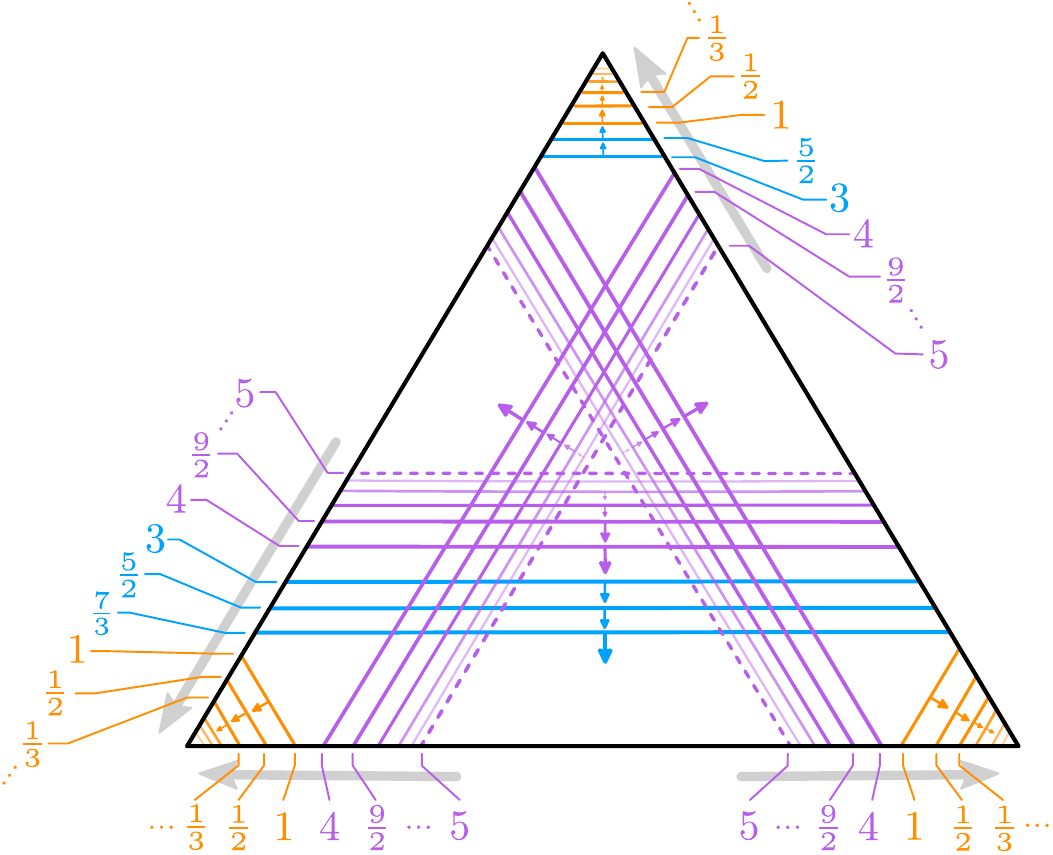}
    	\caption{Neat arcs on a 2--simplex. There are 5 quads in this tetrahedron.}
        \label{fig:distances_all}
    \end{subfigure}
     \caption{The placement of transversely oriented normal discs is shown. In (b), gray arrows indicate the direction of measurement and the transverse orientation of corners, and lower left labels are symmetric with lower right labels (not shown). Similarly, upper right labels are symmetric with upper left labels (not shown). Dotted lines indicate the simplices at distance 5 that large triangles limit to.}
    \label{fig:distances}
\end{figure}

Dual to each of the 0--simplices of $\simplex^3,$ we now place small triangles in $\simplex^3$ at distances $\frac{1}{n}$ from the 0--simplex for each $n \in \NN$ (here we use the convention that the smallest natural number is 1).

Dual to each of the 2--simplices of $\simplex^3,$ we place large triangles in $\simplex^3$ at distances $5-\frac{1}{n}$ from the 2--simplex for each $n \in \NN.$

Last, we pick a pair $\simplex^1_1$ and $\simplex^1_2$ of opposite 1--simplices of $\simplex^3,$ and for each $\simplex^1_i$ we choose a positive integer (possibly zero) $n_i$. If $n_i\neq 0,$ we place $n_i$ transversely oriented quadrilaterals dual to $\simplex^1_i$ at distance $2+\frac{1}{n}$ from the 1--simplex for each $1\le n \le n_i.$

The infinite collection of transversely oriented normal discs thus placed in $\simplex^3$ is said to be in \define{neat position}. Note that on each 1--simplex, the transversely oriented corners of discs are within distance 5 from its endpoints, and the transverse orientation of all corners is towards the closest vertex. The transversely oriented corners with the same orientation have two accumulation points along a 1--simplex, namely a vertex of $\Delta$ and the corner at distance 5 from the vertex. See \Cref{fig:distances_all}.

On each 2--simplex, we have one infinite family of transversely oriented normal arcs for each proper non-empty subsimplex. The short arcs accumulate on the dual 0--simplex, and the long arcs accumulate on the long arc at distance 5 from the dual 1--simplex.

Let $\simplex^2$ be a 2--simplex of $\simplex^3.$ 
We describe a natural labelling scheme for the transversely oriented normal arcs on $\simplex^2$ with the elements of $\ZZ.$ Non-negative labels are given to short arcs and negative labels to long arcs. Consider the collection of all short arcs dual to a vertex $\simplex^0$ of $\simplex^2.$ The short arc furthest from $\simplex^0$ is given label 0, the next furthest label 1, etc. Similarly, consider the collection of all long arcs dual to an edge $\simplex^1$ of $\simplex^2.$ The long arc closest to $\simplex^1$ is given label $-1$, the second closest is given label $-2, $ etc. Note that if no quadrilateral in $\simplex^3$ has a short (resp. long) arc of the given type, then the short (resp. long) arc with label $n\in \ZZ$ is at distance $\frac{1}{n+1}$ (resp. $5+\frac{1}{n}$) from its dual simplex.


\begin{figure}[h]
\centering
  \begin{subfigure}{.5\textwidth}
    \centering
        \includegraphics[scale=.55]{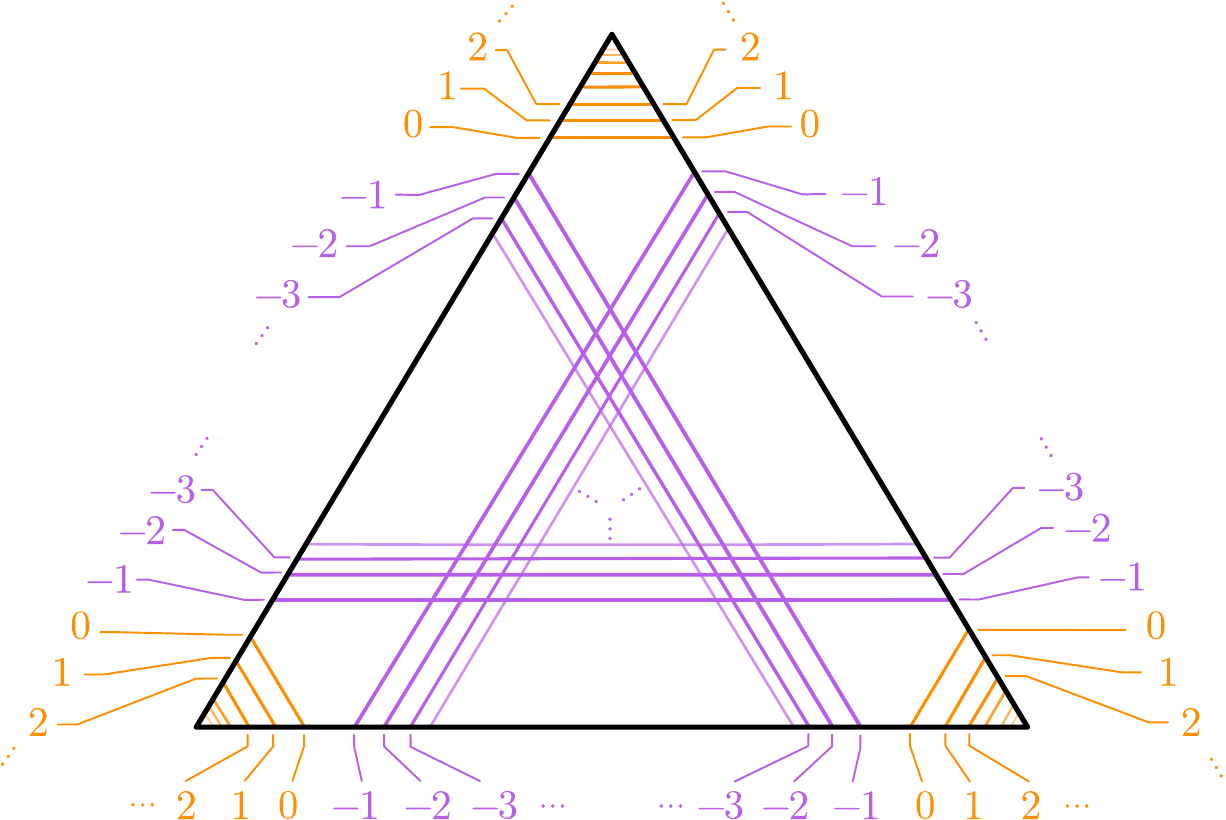}
        \caption{The labelling scheme for transversely oriented corners}
        \label{fig:signs_labels}
    \end{subfigure}
    \hspace{0.5cm}
    \begin{subfigure}{.45\textwidth}
   \centering
        \includegraphics[scale=.55]{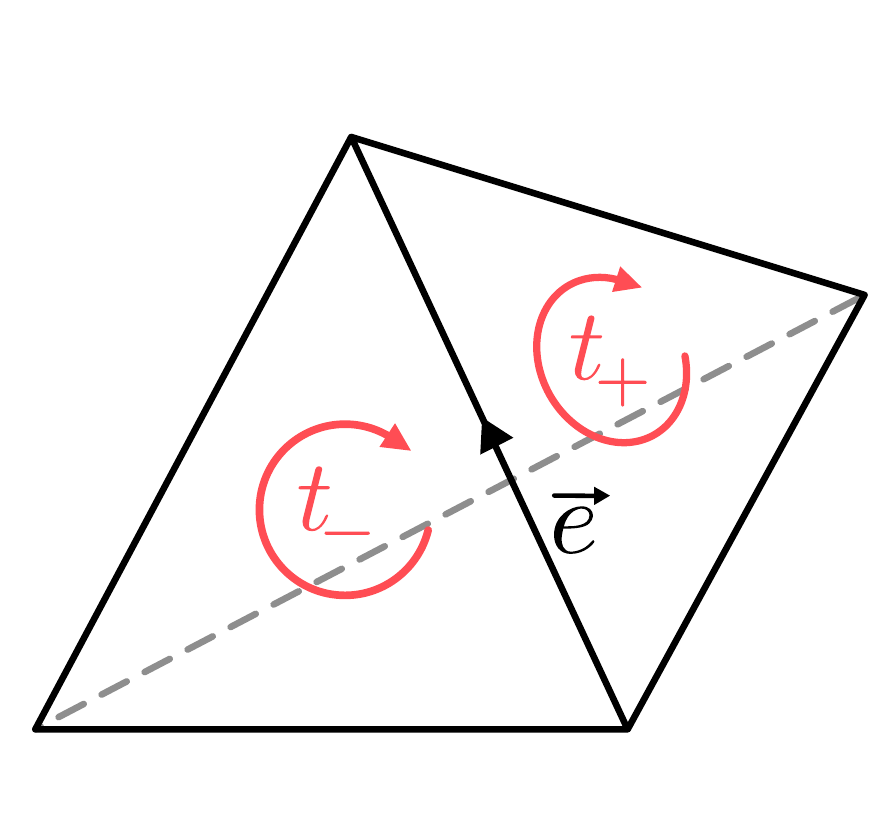}
        \caption{The positive and negative faces $\face_+$ and $\face_-$}
        \label{fig:signs}
    \end{subfigure}
          \caption{Setting up the translation for an oriented edge $\vec{{e}}$}
    \label{fig:signs_translation}
\end{figure}

The labelling scheme now also induces labels to transversely oriented corners on a 1--simplex \emph{relative to} a 2--simplex, as shown in \Cref{fig:signs_labels}. Note that for each normal isotopy class of transversely oriented normal corner, there is a bijection between the corners of this type and the integers, and that the labels of the transversely oriented normal corners along the edge respect the natural order of the integers. The labels of transversely oriented corners depend on the reference 2--simplex, as can be seen from the cases shown in \Cref{fig:translations}.

So for a given 1--simplex $e$ of $\simplex^3$, there are two labellings for each transversely oriented corner $(c, \nu_c)$. As the labels respect the linear order of the integers, the labellings for each transversely oriented corner differ by a translation $T_{\vv{e}} \co \ZZ \to \ZZ$ which depends only on the normal isotopy class of the normal corner and the following orientation convention. The transverse orientation $\nu_c$ gives $e$ an orientation, and we denote the thus oriented edge by $\vv{e}$. The orientation of $\simplex^3$ induces orientations on the 2--simplices containing $e,$ and these in turn induce opposite orientations on $e.$ The face $\face_+$ that induces $\nu_c$ is the positive face, and the other, $\face_-$, is the negative face for $\vv{e}.$ See \Cref{fig:signs}. The translation $T_{\vv{e}}$ takes the labelling with respect to the positive face to the labelling with respect to the negative face. The different cases, where the translation is by a positive integer, a negative integer or zero are shown in \Cref{fig:translations}. This is a natural generalisation of the shift of \cite{Dunfield-incompressibility-2012}.

There is an alternative interpretation of $T_{\vv{e}}$ in terms of \emph{slope} and \emph{weight} associated with quadrilateral disc types, 
known to readers familiar with Tollefson's $Q$--matching equations~\cite{tollefson98-quadspace, tillmann08-normal}. We discuss this viewpoint in \Cref{sec:matching}.

\begin{figure}[h]
    \centering
    \begin{subfigure}{.27\textwidth}
        \includegraphics[scale=1.1]{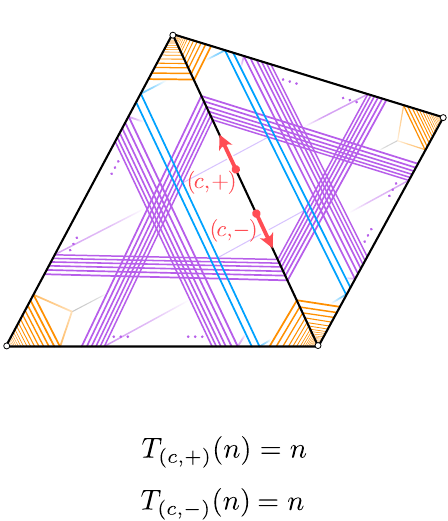}
        \caption{Zero translation}
        \label{fig:Zero_translation}
    \end{subfigure}
    \begin{subfigure}{.38\textwidth}
        \includegraphics[scale=1.1]{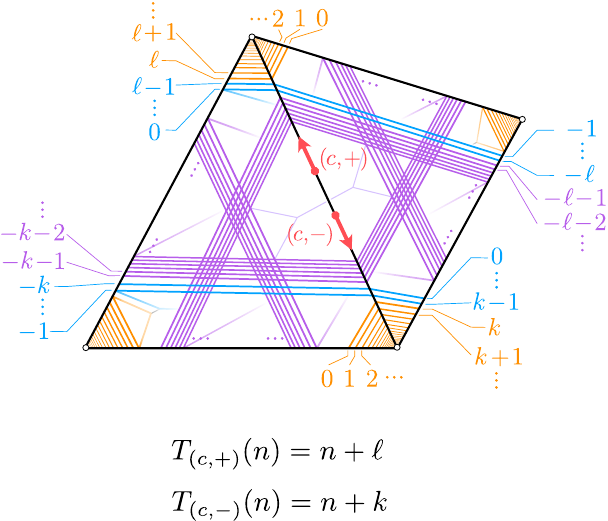}
    	\caption{Positive translation}
        \label{fig:Positive_translation}
    \end{subfigure}
    \begin{subfigure}{.33\textwidth}
    	\includegraphics[scale=1.1]{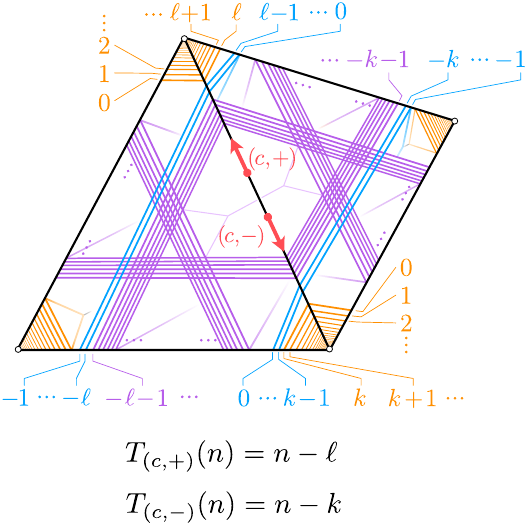}
    	\caption{Negative translation}
        \label{fig:Negative_translation}
    \end{subfigure}
    \caption{Translations}
    \label{fig:translations}
\end{figure}

Suppose $\tet$ is an arbitrary Euclidean 3--simplex and $\mathcal{F}$ is an infinite collection of transversely oriented normal discs in $\tet.$ Then $\mathcal{F}$ is said to be in \define{neat position} in $\tet$ if there is a homeomorphism 
$h\co \tet \to \simplex^3$ that takes the $k$--skeleton of $\tet$ to the $k$--skeleton of $\simplex^3$ for each $k\in \{0, 1, 2\}$, and has the property that $h(\mathcal{F})$ is in neat position in $\simplex^3.$

This completes the local consideration for a single 3--simplex.


\subsection{The matching equations}

This section defines the matching equations for transversely oriented quadrilaterals, and the solutions space $\qtons(\tri)$ as well as the projective solution space $\Pqtons(\tri)$. It is also shown how neat position naturally gives transversely oriented immersions for certain vectors in $\qtons(\tri).$ In addition, linear functionals determining the number and slopes of the ends of the immersions are described and the example of the figure-8 knot complement illustrates a few subtle points.


\subsubsection{Neat position of normal subset in pseudo-manifold}
\label{sec:Neat position of normal subset in pseudo-manifold}

Now suppose $p\co \widetilde{\Delta} \to \widetilde{\Delta} / \Phi = P$ is the quotient map described above, and each 3--simplex $\simplex^3_k$ in $\widetilde{\Delta}$ contains an infinite collection $\mathcal{F}_k$ of transversely oriented normal discs in neat position. 

Note that if no $\mathcal{F}_k$ contains transversely oriented normal quadrilaterals, then we can apply the convention of \Cref{sec:neat in simplex} to each $\simplex^3_k$ and choose isometries as face pairings. Hence in $P$ we obtain infinite collections of embedded vertex linking surfaces with orientation towards the dual vertices, and  infinite collections of immersed surfaces that are homotopic to vertex linking surfaces with orientation away from the dual vertices.

In the presence of quadrilateral discs, we need to apply isotopies to the $\mathcal{F}_k$ so that they glue up across 2--simplices.
Suppose $\varphi\in \Phi$ is a face pairing between faces $t_i$ and $t_j$ of the 3--simplices $\simplex^3_i$ and $\simplex^3_j$ in $\widetilde{\Delta}.$ After possibly pre-composing $\varphi$ with a normal isotopy, we may assume that $\mathcal{F}_i\cap t_i$ is mapped bijectively to $\mathcal{F}_j\cap t_j.$ Indeed, this isotopy is essentially determined by the labelling scheme for the transversely oriented normal arcs, and 
we may assume that this normal isotopy is 
\begin{enumerate}
\item[(a)] supported in a small neighbourhood of the face $t_i$ of $\simplex^3_i,$ and
\item[(b)] is the identity on each 1--simplex except for the distance 5 neighbourhood of its endpoints.
\end{enumerate}
It follows that there is an isotopy of $\widetilde{\Delta}$ supported outside of a small neighbourhood $\text{nhd}(\widetilde{\Delta}^{(1)})$ of the 1--skeleton of $\widetilde{\Delta}$  such that the image of $\cup_k\mathcal{F}_k \setminus \text{nhd}(\widetilde{\Delta}^{(1)})$ is the image of an immersed transversely oriented surface in $P \setminus p(\text{nhd}(\widetilde{\Delta}^{(1)})).$

Under what conditions can we ensure that the isotopies can be chosen to agree on the 1--skeleton?
When we identify two faces of 3--simplices in $\widetilde{\Delta}$, the natural bijection between the labelling schemes for transversely oriented normal arcs is the identity map. As we go around an edge of a 3--simplex, we have described a translation that gives a shift in this coordinate system arising from the quadrilaterals. Hence as we go around an edge in $P,$ we obtain a successive application of translations that have to combine to the identity as we return to the initial starting point. To state this more formally, 
let $e$ be an edge in $P$ and suppose $p^{-1}(e) = \cup_{j=1}^k e_j \subset \widetilde{\Delta}^{(1)}.$ An orientation $\vv{e}$ induces an orientation $\vv{e}_j$ of each $e_j.$ Then we require that $$T_{\vv{e}}(n) := T_{\vv{e}_k}\circ T_{\vv{e}_{k-1}} \circ \ldots T_{\vv{e}_{1}}(n) = n$$ for one (and hence all) $n \in \ZZ.$

We will now describe this as a system of equations in the coordinate system used in computation.


\subsubsection{Matching equations}\label{sec:matching}

Let $\square$ and $\vv{\square}$ respectively denote the set of all normal isotopy classes
of normal quadrilaterals in $\tri$ and transversely oriented normal quadrilaterals in $\tri$,
so that $|{\square}| = 3t$ and $|\vv{\square}| = 6t$ where $t$ is the number of tetrahedra in $\tri$.

There is a natural map $\vv{\square} \to \square$ which amounts to forgetting the transverse orientation.

We identify $\RR^{\square}$ with $\RR^{3t}$ and $\RR^{\vv{\square}}$ with $\RR^{6t}.$
Given a transversely oriented normal surface $(S, \nu),$ let $\vv{x}(S, \nu) \in \RR^{\vv{\square}} = \RR^{6t}$
denote the integer vector for which each coordinate $\vv{x}(S, \nu)(q, \mu)$
counts the number of transversely oriented normal quadrilaterals in $S$ of type $(q, \mu) \in \vv{\square}$.

Similarly, let $x(S , \nu) = x(S) \in \RR^\square = \RR^{3t}$ denote the integer vector for which each coordinate $x(S)(q)$ counts the number of quadrilateral discs in $S$ of type $q \in \square$.

The map $\vv{\square} \to \square$ induces a natural map $\RR^{\vv{\square}} \to \RR^{\square}$ giving 
$\vv{x}(S, \nu) \mapsto x(S).$

The \define{transversely oriented normal $Q$--coordinate} $\vv{x}(S, \nu)$ satisfies the
following two algebraic conditions.

First, $\vv{x}(S, \nu)$ is admissible.
A vector $x \in \RR^\square$ is \define{admissible} if
$x \ge 0$, and for each tetrahedron $x$ is non-zero
on at most one of its three quadrilateral types. 
This reflects the fact that an embedded surface cannot contain
two different types of quadrilateral in the same tetrahedron.
We say that $\vv{x}\in \RR^{\vv{\square}} $ is admissible if its image  $x \in \RR^\square$ is admissible.

Second, $\vv{x}(S, \nu)$ satisfies a linear equation for each ideal
edge $e$ in $M$ and each orientation chosen on this edge. We call this a \define{$\vv{Q}$--matching equation}. 
The equation arises from the fact that  the surface $S$ meets $e$ locally in a 2--disc, and each transversely oriented normal disc in $S$ induces the same orientation on $e.$ Moreover, the link of $e$ is like an equator for a neighbourhood of $e,$ and the boundary of the 2--disc crosses this equator from north to south
as often as it crosses from south to north.
We now give the precise form of these equations.
Recall that $M$ is oriented and all tetrahedra are given
the induced orientation; see \cite[\S2.9]{tillmann08-normal} for
details.

\begin{figure}[h]
    \centering
    \begin{subfigure}{.33\textwidth}
        \includegraphics[scale=1]{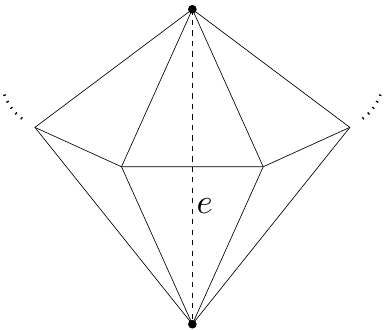}
        \caption{The abstract neighbourhood $B(e)$}
        \label{fig:matchingquadbdry}
    \end{subfigure}
    \begin{subfigure}{.32\textwidth}
        \includegraphics[scale=1]{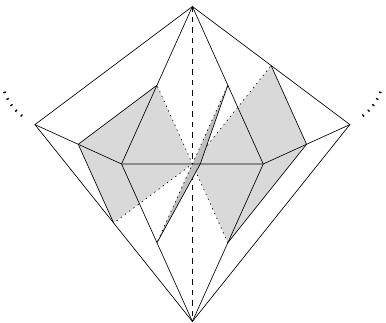}
    	\caption{Positive slope $+1$}
        \label{fig:matchingquadpos}
    \end{subfigure}
    \begin{subfigure}{.33\textwidth}
    	\includegraphics[scale=1]{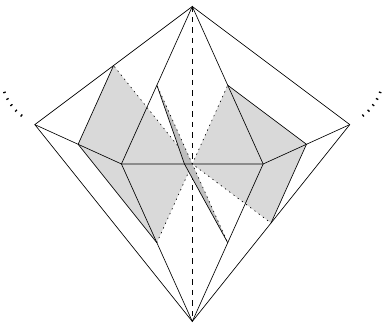}
    	\caption{Negative slope $-1$}
        \label{fig:matchingquadneg}
    \end{subfigure}
    \caption{Slopes of quadrilaterals.}
    \label{fig:slopes}
\end{figure}

Consider the collection $\mathcal{C}$ of all ideal tetrahedra meeting at an edge $e$ in $M$ (including $k$ copies of a tetrahedron $\tau$ if $e$ occurs $k$ times as an edge in $\tau$).  We form the \define{abstract neighbourhood $B(e)$} of $e$ by pairwise identifying faces of tetrahedra in $\mathcal{C}$ such that there is a well defined quotient map from $B(e)$ to the neighbourhood of $e$ in $M$; see \Cref{fig:matchingquadbdry} for an illustration.  Then  $B(e)$ is a ball (possibly with finitely many points missing on its boundary).  We think of the (ideal) endpoints of $e$ as the poles of its boundary sphere, and the remaining points as positioned on the equator. The ball $B(e)$ obtains a natural orientation from $M$, and as in \cite[\S2.4]{tillmann08-normal} we give the equator the orientation induced from the southern hemisphere.

Let $\tau$ be a tetrahedron in $\mathcal{C}$. The boundary square of a transversely oriented normal quadrilateral of type $(q, \mu)$ in $\tau$ meets the equator of $\partial B(e)$ if and only it has a transversely oriented corner on $e$. In this case, it has a slope ${s(q,e) = \pm1}$ of a well--defined sign on $\partial B(e)$ which is independent of the orientation of $e$ and the transverse orientation of $q.$ Refer to \Cref{fig:matchingquadpos,fig:matchingquadneg}, which show quadrilaterals with \define{positive} and \define{negative slopes} respectively. If $D$ is a connected transversely oriented normal surface in $B(e)$ that meets $e$, then $D$ is a disc. Indeed, $D$ is the join of $\partial D$ to $D\cap e$. Now $\partial D$ can be oriented such that each slope of a quadrilateral in $D$ corresponds to the signed intersection number of $\partial D$ with the equator. In particular, the total sum is zero.

Since tetrahedra in $M$ can appear with multiplicity in $B(e)$, the same is true for normal quadrilateral discs. 
We thus define the \define{total weight} $\wt_{\vv{e}}(q, \nu)$ of $q$ at $\vv{e}$ as a sum over contributions from the four corners of the quadrilateral discs $q.$ If a corner of $q$ is not on $e,$ then its contribution to $\wt_{\vv{e}}(q, \nu)$ is zero. If a corner of $q$ is on $e$ but the transverse orientation of $q$ does not induce the orientation on $\vv{e}$ then the contribution is zero. If a corner is on $e$ and the transverse orientation of $q$ induces the orientation on $\vv{e}$ then the contribution is $s(q,e).$

Given the oriented edge $\vv{e}$ in $M,$ the $\vv{Q}$--matching equation of $\vv{e}$
is then defined by 
\[0 = \sum_{(q, \nu)\in \vv{\square}}\; \wt_{\vv{e}}(q, \nu)\;\vv{x}(q, \nu)\]
where $\vv{x}(q, \nu)$ is the modulus in $\vv{x}$ of the $(q,\nu)$ quadrilateral type, i.e., $\vv{x}(q, \nu)$ counts the number of $(q,\nu)$ quads in $\vv{x}$. We emphasise again that each edge is counted twice in this setting, once with each orientation.

The set of all $\vv{x}\in \RR^{\vv{\square}}$ with the property that
(i)~$\vv{x}\ge 0$ and (ii)~$\vv{x}$ satisfies the $\vv{Q}$--matching equations is denoted
$\qtons(\tri)$ and called the \define{transversely oriented normal surface cone}. This
naturally is a polyhedral cone, but
the set of all admissible $\vv{x}\in \RR^{\vv{\square}}$
typically meets $\qtons(\tri)$ in a non-convex set.

The projective solution space $\Pqtons(\tri)$ is the intersection of $\qtons(\tri)$ with the standard simplex in $\RR^{\vv{\square}}.$

\begin{theorem}\label{thm:admissible integer t.o. solution gives normal immersed}
For each $\vv{x}\in \RR^{\vv{\square}}$ with the properties that $\vv{x}$ has integral coordinates, is admissible, and satisfies the $\vv{Q}$--matching equations, there is a (possibly
non-compact) properly immersed transversely oriented normal surface $(S, \nu)\to M$ such that $\vv{x} = \vv{x}(S, \nu).$ Moreover, $S$ is a topologically finite surface and we may choose a compact core of $S$ such that the immersion is an embedding on each end. Conversely, if $(S, \nu)\to M$ is a transversely oriented normal immersion, then its coordinate $\vv{x}(S, \nu)$ satisfies the $\vv{Q}$--matching equations.
\end{theorem}

The analogous result for standard coordinates of \emph{closed} transversely oriented normal surfaces is \cite[Proposition 8]{Cooper-norm-2009}.
The analogous result for unoriented ${Q}$--matching equations of \cite{Kang-normal-2005, tillmann08-normal}  postulates the existence of a \emph{unique} (possibly non-compact) properly embedded normal surface up to normal isotopy with a given $Q$--coordinate. In particular, if the surface $(S, \nu)$ in the conclusion of the above theorem can be chosen to be embedded, then it is obtained from the unique embedded representative of $x(S)$ by choosing a transverse orientation on it. The example in \Cref{sec:example fig8} shows that there may be no embedded transversely oriented normal surface satisfying the conclusion of the theorem. Indeed, in that example $(S, \nu)$ is an immersed thrice punctured sphere but the unique representative of $x(S)$ is a once-punctured Klein bottle. \Cref{cor:euler} shows that the (possibly immmersed) representative of $\vv{x}(S, \nu)$ and the unique embedded representative of $x(S)$ have the same Euler characteristic, and we describe in \Cref{sec:implementation} how we determine the topology of these surfaces.

\begin{proof}
Let $\vv{x}\in \RR^{\vv{\square}}$ have the properties that $\vv{x}$ has integral coordinates, is admissible, and satisfies the $\vv{Q}$--matching equations. The key observation is that the $\vv{Q}$--matching equation for $\vv{e}$ is equivalent to the translation $T_{\vv{e}}$ being the identity map. This follows from the definition of slopes and weights of quadrilateral discs.

As in \Cref{sec:Neat position of normal subset in pseudo-manifold}, we therefore construct an infinite normal subset in neat position in the end-compactification $P$ of $M.$ The face pairings give identifications between edges of normal discs, and hence a map from a transversely oriented surface $(S, \nu) \to M \subset P$ that is an immersion.
Since the translation associated to each oriented edge in $P$ is the identity, the cell decomposition of $S$ is locally finite.

Now discard any component of $S$ that does not contain any quadrilateral discs, and denote the resulting surface again by $S.$ Hence $S$ is non-empty has finitely many components. If $S$ contains only finitely many large triangles, then $(S, \nu) \to M$ is a properly immersed surface and clearly an embedding of each end (after choosing a suitable compact core).

Hence assume $S$ contains infinitely many large triangles. These triangles accumulate on an immersed surface $F$ consisting only of large triangles. We now describe an explicit homotopy of $(S, \nu) \to M$ that takes these triangles to triangles accumulating at the 0--skeleton of $P.$ 
Each component $F_\v$ of $F$ can be viewed as the immersion of a vertex linking surface at distance $95$ from a vertex $\v\in P^{(0)}.$ Choose a compact core of $S$ that contains all quadrilateral discs. A normal homotopy of the map $(S, \nu) \to M$ that fixes the compact core and all small triangles is uniquely determined by describing where the corners of large triangles are mapped. The large triangles are then mapped to the Euclidean triangles with these endpoints.
There is $1> \varepsilon>0$ such that $N_\varepsilon(F_\v)$ only contains large triangles not contained in the compact core. An oriented edge $\vv{e}$ ending in $\v$ is naturally parameterised by $[0,100]$ with $100$ corresponding to $\v.$ There is a homotopy of $\vv{e}$ which is the identity on $[0, 5 - \varepsilon],$ takes $[5 - \varepsilon, 5]$ to $[5 - \varepsilon, 100]$ by the linear map $x\mapsto \frac{1}{\varepsilon}[(95+\varepsilon)x - 95 (5-\varepsilon)]$ and maps $[5, 100]$ to $100.$  This is the homotopy we apply to each corner $(c, \nu_c)$ of a large triangle, where $c\in e$ and $\nu_c$ induces the \emph{opposite} orientation of $\vv{e}.$ The resulting map is a proper immersion, and it is an embedding of each end since any two large triangles in $S$ close to a vertex will have pairwise disjoint interior.

This proves the claim that each solution is realised by a proper immersion with embedded ends. The converse part follows from the description of the matching equations.
\end{proof}


\begin{remark}[(Notation)]
We usually write $\widetilde{\Delta} = \{ \tau_i\}$ and label the vertices of each $\tau_i$ by $0_i, 1_i, 2_i, 3_i$ with the natural order corresponding to the chosen orientation. A quadrilateral type partitions the vertices into two pairs, and a transverse orientation is specified by only indicating the pair of vertices on the positive side. Thus, the transversely oriented normal quadrilateral types in $\tau_i$ are written in the form $q^i_{ab},$ where {$0\le a< b\le  3.$} 

We simplify notation by writing $q_{ab}$ 
in sums over all simplices or if $\tau_i$ is understood.
\end{remark}


\subsection{Boundary curve map}
\label{sec:boundary curve map}

Following \cite{tillmann08-normal}, it is straightforward to determine the transversely oriented boundary curves of a transversely oriented normal surface. The construction is essentially the same as in \cite{tillmann08-normal}, except that the boundary arcs of the quadrilaterals are partitioned into two sets. One set, the \emph{inward boundary}, corresponds to the ends of the surface that consist of large triangles, and the other set, the \emph{outward boundary}, corresponds to the ends that consist of small triangles.

The methods of \cite{tillmann08-normal} are applied independently to each set. Thus we obtain two boundary curve maps. The image of each is a union of pairwise disjoint embedded simple closed curves in $\del\Mbar$, but the two sets may intersect non-trivially, as happens in the example given in the next section. 

For each vertex $\v$ of $\tri$, let $N_\v$ be an open regular neighborhood of $\v$, and let $B_\v=\del \Nbar_\v$. We may assume that the closures of the vertex neighbourhoods are pairwise disjoint.
Hence  identify $\Mbar$ with $M\setminus (\cup_\v N_\v)$, so that the vertex links $B_\v$ are components of $\del \Mbar$. Fix $\v$ and let $\gamma$ be an oriented normal closed curve in $B_\v$. We start by constructing two linear functionals $w_\gamma^+$ and $w_\gamma^-$ on $\qtons(\tri)$ by defining them on quadrilaterals, then extending linearly. To this end, let $\Delta^3$ be a tetrahedron in $\tri$, and let $q$ be a quadrilateral in $\Delta^3$. For any normal arc $\a$ in a face of $\Delta^3$ there is a natural transverse orientation, pointing into the triangle cut off by $\a$. A normal arc $\a$ in $\del q$ is in the \define{outward boundary} $\del q^+$ if the two transverse orientations on $\a$ agrees with the transverse orientation on $q$, and is in the \define{inward boundary} $\del q^-$ otherwise. Now orient $q$ in such a way that the new orientation, taken with the transverse orientation of $q$, induces a left-handed orientation on $M$ (see \Cref{fig:intersection_map}). This orientation of $q$ induces an orientation on $\del q$, and in particular on $\a$. Let $\t$ be the unique normal triangle in $\Delta^3\cap \del\Mbar$ whose boundary contains an arc $\a_\t$ normally isotopic to $\a$. Orient $\a_\t$ so that its orientation agrees with that of $\a$. Recall that for such an arc $\a_\t$ and curve $\gamma$, the algebraic intersection number $\iota(\a_\t,\gamma)$ is $+1$ if $\a_\t$ crosses $\gamma$ from left to right, and $-1$ otherwise (see \Cref{fig:intersection}).
Now define
$$
w^+_\gamma(q) = \sum_{\a \in \del q^+} \iota(\a_\t,\gamma)\qquad \text{and}\qquad
w^-_\gamma(q) = \sum_{\a \in \del q^-} \iota(\a_\t,\gamma)
$$
where the algebraic intersection number is defined with respect to the transverse orientation on $B_\v$ pointing away from $\v$. Note that if $\t$ is not a triangle in the triangulation of $B_\v$, then $\iota(\gamma,\a_\t)=0$ since $\gamma\subset B_\v$. Letting $\vv{x}\in \qtons(\tri)$ and extending the above definition linearly gives
$$
w^+_\gamma(\vv{x}) = \sum_{\Delta^
3\subset \tri} \left(\sum_{0\le i<j\le 3} \vv{x}(q_{ij})w^+_\gamma(q_{ij})\right) \qquad \text{and}\qquad
w^-_\gamma(\vv{x}) = \sum_{\Delta^
3\subset \tri} \left(\sum_{0\le i<j\le 3} \vv{x}(q_{ij})w^-_\gamma(q_{ij})\right)
$$

\begin{figure}
    \begin{subfigure}{.7\textwidth}
    \centering
        \includegraphics[scale=2.2]{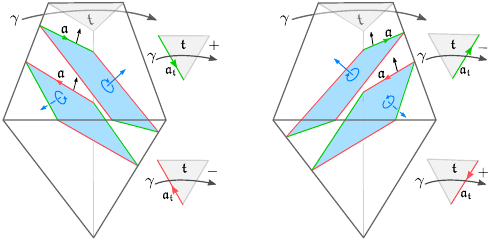}
        \caption{Inward and outward boundary intersections}
        \label{fig:intersection_map}
    \end{subfigure}
    \begin{subfigure}{.29\textwidth}
    	\centering
    	\vspace{1cm}\quad\\
        \includegraphics[scale=2.2]{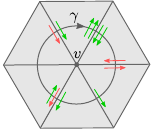}\\
        \vspace{1cm}\quad
    	\caption{A vertex loop}
        \label{fig:matching_loop}
    \end{subfigure}
    \caption{(a): The outward boundary $\del q^+$ of $q$ is shown in green, and the inward boundary $\del q^-$ is shown in red. If $\a_\t$ is in $\del q^+$, a curve $\gamma$ crossing through $\t$ contributes $+1$ to $w_\gamma^+(q)$ if it enters via the edge $\a_\t$, and contributes $-1$ if it leaves via $\a_\t$. Similarly for $\a_\t$ in $\del q^-$, but with signs switched. (b): For a loop around a vertex $v$ of the triangulation of $B_\v$, we must have $w_\gamma^+(\vec{x})= w_\gamma^-(\vec{x}) = 0$.
    }
    \label{fig:intersection}
\end{figure}

where $\vv{x}(q_{ij})$ is the modulus in $\vv{x}$ of the $q_{ij}$ quadrilateral type (see \Cref{sec:matching}). If $v$ is a vertex in the triangulation of $B_\v$ and $\gamma$ is a small oriented circle around $v$, then for any $\vv{x}\in\qtons(\tri)$ we have $w^+_\gamma(\vv{x})=w^-_\gamma(\vv{x})=0$. This is because for this choice of $\gamma$, these are exactly the two $\qtons$-matching equations around the 1-simplex $\Delta^1$ of $\tri$ containing $v$. This fact is illustrated in \Cref{fig:matching_loop}, where blue arrows correspond to normal arcs in $\del q^+$ and blue arrows correspond to arcs in $\del q^-$ (for the quadrilateral $q$ containing the arc). It follows that for any closed curve $\gamma$ in $B_\v$, $w_\gamma(\vv{x})$ depends only on the homotopy class of $\gamma$. 

Let $(\mu_\v,\lambda_\v)$ be closed normal curves representing a basis for $H_1(B_{\v};\ZZ)$. Define maps $w^+_\v,w^-_\v:\qtons(\tri)\to \RR^2$ by
$$
w^+_\v(\vv{x})=(-w^+_{\lambda_\v}(\vv{x}),w^+_{\mu_\v}(\vv{x})) \qquad \text{and}\qquad
w^-_\v(\vv{x})=(-w^-_{\lambda_\v}(\vv{x}),w^-_{\mu_\v}(\vv{x}))
$$
Now consider the canonical isomorphism $\ZZ^2\cong H_1(B_\v;\ZZ)$ which assigns to a vector $(x,y)\in \ZZ^2$ the class ${[\gamma]\in H_1(B_\v;\ZZ)}$ satisfying $\iota(\gamma,\lambda_\v)=-x$ and $\iota(\gamma,\mu_\v)=y$. This isomorphism extends uniquely to an isomorphism $\phi_\v:\RR^2\to H_1(B_\v;\RR)$. The map $\phi_\v$ takes basis vectors $(1,0)$ and $(0,1)$ to $\mu_\v$ and $\lambda_\v$, respectively. Finally, we define maps $\del^+_\v,\del^-_\v:\qtons(\tri)\to H_1(B_\v;\RR)$ by $\del^+_\v=\phi_\v\circ w^+_\v$, $\del^-_\v=\phi_\v\circ w^-_\v$, and define the \define{outward and inward boundary curve maps} $\del^+,\del^-: \qtons(\tri)\to H_1(\del\Mbar;\RR)$ to be the respective direct sums over the vertices $\v$ of $\tri$ of the maps $\del^+_\v$ and the maps $\del^-_\v$. That is, 
$$
\del^+ := \oplus_\v \del^+_\v: \qtons(\tri)\to \oplus_\v H_1(B_\v;\RR)=H_1(\del \Mbar;\RR)
$$
and similarly for $\del^-$. 
It follows from the construction that $\del^+$ and $\del^-$ are linear maps. The following proposition, which is a corollary of \cite[Proposition 3.3]{tillmann08-normal}, shows that these maps are appropriately named:
\begin{proposition}\label{prop:bdy_map}
	Let $\vv{x}\in \qtons(\tri)$, and let $(f,\nu_S):S\to M$ be the transversely oriented normal immersion given by \Cref{thm:admissible integer t.o. solution gives normal immersed}. Fix a vertex $\v$ of $\tri$, and suppose $\del_\v^+(\vv{x})= k_1(p_1\mu_\v+ q_1\lambda_\v)$ and $\del_\v^-(\vv{x})= k_2(p_2\mu_\v+ q_2\lambda_\v)$, where $\gcd(p_i,q_i)=1$ for $i=1,2$. Then there are $k_i$ components $\beta$ of $f^{-1}(f(S)\cap B_\v)$ such that $[f(\beta)] = p_i\mu_\v+ q_i\lambda_\v$ in $H_1(B_\v;\RR)$, and $S$ has $k_1+k_2$ cusps on the cusp of $M$ corresponding to $\v$.
\end{proposition}

\begin{proof}
Let $\triangle_+$ denote the union of all small normal triangles in $f(S)$, and $\triangle^-$ the union of all large normal triangles. It follows from the neat position construction that $\triangle^+$ is embedded, and that every normal arc in $\del \triangle^+$ is in $\bigcup_{q\in \vv{\square}} \del q^+$. Fix a vertex $\v$ of $\tri$, and let $\triangle_\v^+$ consist of the connected components of $\triangle^+$ that are isotopic into $B_\v$. The boundary $\del \triangle_\v^+$ of $\triangle_\v^+$ consists of oriented arcs $\a$ along which quadrilaterals glue to small triangles. Each such arc is normally isotopic to an arc $\a_\t$ in $B_\v$, and these arcs taken together form a 1-chain $c$. It follows from the construction of $\del_\v^+$ that $\del_\v^+(\vv{x})=[c]\in H_1(B_\v;\RR)$, where $[c]$ is the class of $c$.
\end{proof}


\subsection{Example: figure-8 knot}
\label{sec:example fig8}

The figure-8 knot complement $M=\SS^3\setminus 4_1$ has Seifert surface a once punctured torus $S$. Since $M$ is fibered with fiber $S$, \Cref{thm:walsh} does not guarantee the existence of an embedded spun-normal surface isotopic to $S$. Indeed, a simple calculation shows that there is no such spun-normal surface for the standard two-tetrahedron triangulation (see \cite{Kang-normal-2003}). However, we demonstrate here that there is a properly immersed spun-normal surface $F$ that is homologous to $S$, which is a thrice-punctured sphere (hence it is norm-minimizing). We discovered this using the computer program \texttt{Tnorm}, which is described in \Cref{sec:implementation and examples}. On the other hand, in \Cref{sec:implementation and examples} we also find a triangulation of the figure-8 knot complement by \emph{five} tetrahedra, with respect to which the fiber \emph{is} realized as an embedded spun-normal surface.

\begin{figure}[h]
    \begin{subfigure}{.49\textwidth}
    \centering
        \includegraphics[scale=1.3]{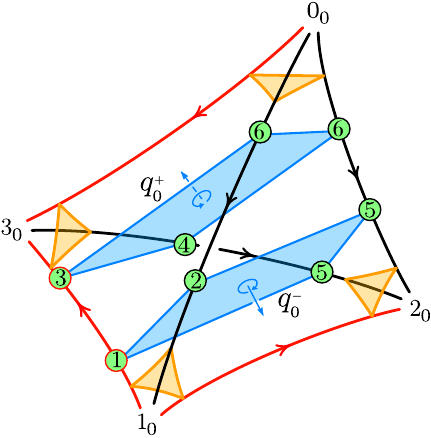}
        \caption{Tetrahedron 0}
        \label{fig:tet_left}
    \end{subfigure}
    \begin{subfigure}{.49\textwidth}
    \centering
        \includegraphics[scale=1.3]{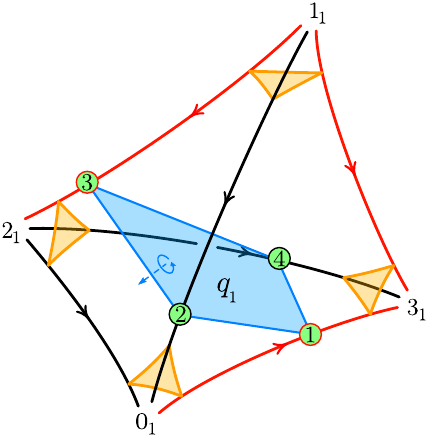}
    	\caption{Tetrahedron 1}
        \label{fig:tet_left}
    \end{subfigure}
    \caption{The figure-8 knot complement is triangulated by the two tetrahedra shown, with face pairings determined by edge colors and orientations. }
    \label{fig:tets}
\end{figure}

\Cref{fig:tets} shows Thurston's famous triangulation of the figure-8 knot complement by two tetrahedra, which we denote by $\tri_8$. Denote by $\tet_0$ the tetrahedron to the left and by $\tet_1$ the tetrahedron to the right. There are two edge classes, which are colored red and black in the figure, with orientations assigned to uniquely determine face gluings. Let $\vv{x}\in\RR^{\vv{\square}}$ be the vector with $\vv{x}(q_{03}^0)=\vv{x}(q_{12}^0)=\vv{x}(q_{02}^1)=1$, and $\vv{x}(q_{ij}^k)=0$ for all other quadrilaterals. Going forward, we denote the two quadrilaterals of $\tet_0$ by $q_0^+$ and $q_0^-$, and the quadrilateral of $\tet_1$ by $q_1$. These are shown in \Cref{fig:tets}, with transverse orientations indicated. 
The first important thing to note is that the $Q$-matching equations are satisfied: for each transverse orientation, there are two quadrilaterals of positive slope and two quadrilaterals of negative slope. Thus $\vv{x}\in\qtons(\tri_8)$, and $\vv{x}$ is admissible by definition. 

The identifications of the transversely oriented normal corners of the discs can either be determined by examination of the abstract edge neighbourhood, or from the induced triangulations of the transversely oriented vertex links. In \Cref{fig:tets} we have indicated the identifications by labelling the corners of the quadrilaterals arbitrarily with $1\le i \le 6$. This uniquely determines the resulting properly immersed normal surface $F$ obtained by adding infinitely many triangles in neat position. 
The three quadrilaterals are glued as shown in \Cref{fig:glued_quads}, and the resulting surface is a thrice-punctured sphere.

\begin{figure}
	\begin{subfigure}{0.49\textwidth}
		\centering
		\vspace{.6cm}\quad\\
		\includegraphics[scale=1.3]{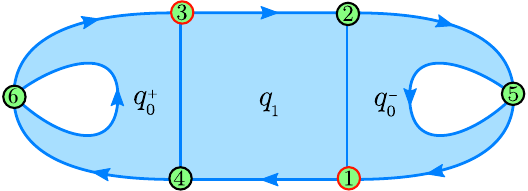}
		\vspace{.6cm}\quad
		\caption{Quadrilaterals in $F$} 
		\label{fig:glued_quads}
	\end{subfigure}
	\begin{subfigure}{0.49\textwidth}
		\centering
		\includegraphics[width=5cm]{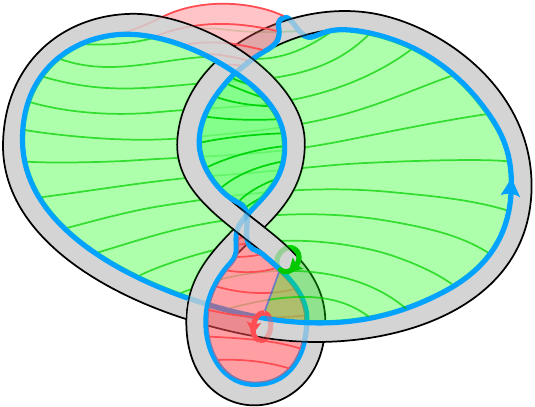}
		\caption{Viewed as immersed spanning disc for the knot} 
		\label{fig:glued_with_tris}
	\end{subfigure}
	\caption{The properly immersed thrice-punctured spun-normal sphere $F$}
	\label{fig:normal_surface}
\end{figure}

Let $\Mbar$ be the manifold with boundary satisfying $\mathrm{int}(\Mbar)=M$. The boundary slopes of the three boundary components of $F$ are determined as follows. Each quadrilateral of $\vv{x}$ can be oriented so that the given orientation, taken with the transverse orientation on the quad, induces a left-hand orientation on $\Mbar$. This in turn induces an orientation on the boundary of each quad. Let $T$ be the triangulation of $\del\Mbar$ induced by $\tri$. $T$ consists of $8$ triangles, which we have drawn near the corners of the tetrahedra in \Cref{fig:tets}, and which glue up into the fundamental domain of $\del\Mbar$ shown in \Cref{fig:meridians}. Each oriented edge of a quadrilateral is normally isotopic to an edge of the triangulation $T$, so the quadrilateral edges induce oriented curves on $\del\Mbar$. One of these curves is shown in \Cref{fig:boundary}, and the other two are shown in \Cref{fig:meridians}. Since the meridian and the longitude of the cusp of the figure-8 knot are the curves $\mu$ and $\lambda$ indicated in \Cref{fig:boundary} (see \cite[Section 4.6]{thurston78-lectures}), the three boundary curves of $F$ in these coordinates are $2\mu+\lambda$ (shown in \Cref{fig:boundary}), and two copies of $-\mu$ (\Cref{fig:meridians}). Thus $[\del F]=[\lambda]$ as elements of $H_1(\del\Mbar)$, so $F$ is homologous to the Seifert surface $S$ (see \Cref{subsec:peripheral} for homology map details).

\begin{figure}[h]
    \centering
    \begin{subfigure}{.49\textwidth}
        \includegraphics[scale=1.3]{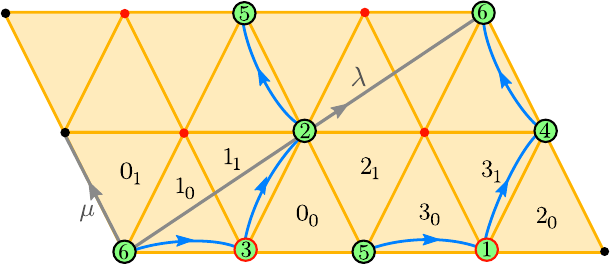}
        \caption{The boundary curve $2\mu + \lambda$}
        \label{fig:boundary}
    \end{subfigure}
    \begin{subfigure}{.49\textwidth}
    	\vspace{.5cm}\quad\\
        \includegraphics[scale=1.3]{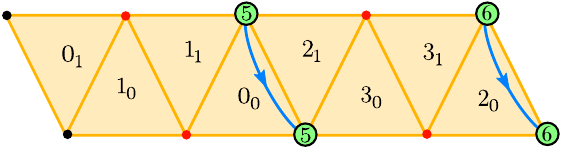}\\
        \vspace{.5cm}\quad
    	\caption{Two parallel $-\mu$ boundary curves}
        \label{fig:meridians}
    \end{subfigure}
    \caption{The triangulation $T$ of the boundary $\partial \Mbar$ of $M$, shown with the three boundary curves of $F$. }
    \label{fig:bdys}
\end{figure}


%% file: angles01.tex

\section{Euler characteristic and generalised angle structures}\label{sec:euler}

We describe how one can determine the Euler characteristic of a transversely oriented normal immersion from its quadrilateral coordinate using certain angle structures that have \emph{vanishing peripheral holonomy}. This material relies on \cite{neumann-combinatorics-1992, Luo-angle-2008, GHHR16}.


\subsection{Angle structures and rotational holonomy}

Let $e\in \tri$ be an edge in $M$ and $q\in \square$ be a quadrilateral type, contained in the 3--simplex $\simplex^3.$ 
We now define the index $i(e,q).$ Informally, this is the number of times $q$ \emph{faces} $e.$ 
If $e$ is not incident with $\simplex^3$, then we define $i(e,q)=0.$ Otherwise, we let $i(e,q)$ be the number of edges in the preimage of $e$ in $\widetilde{\Delta}$ that are incident with the preimage of $\simplex^3$ and disjoint from the lift of $q$. It follows that $i(e,q)\in\{0,1,2\}$. 

\begin{definition}
	A \define{generalized angle structure} on a triangulated manifold $(M,\tri)$ is a function $\as:\square \to \RR$ such that for each tetrahedron $\simplex^3\in \tri$
	\begin{equation}\label{eq:pi}
	\sum_{q\subset \simplex^3} \as(q)=\pi
	\end{equation}
	and for every edge $e\in \tri$
	\begin{equation}\label{eq:2pi}
	\sum_{q\in\square} i(\e,q)\as(q)=2\pi
	\end{equation}
	
	
	In the case that $\alpha(q)\in[0,\pi]$ for all $q\in \square$, we say that $\alpha$ is a \define{semi-angle structure}. If $\alpha(q)\in(0,\pi)$ for all $q\in \square$, then $\alpha$ is a \define{strict angle structure}.
\end{definition}

We can alternatively think of an angle structure $\as$ as an assignment of dihedral angles to the tetrahedra of $\tri$:
if $e$ is an edge of $\simplex^3$, then the dihedral angle at $e$ is $\as(q)$, where $q\subset \simplex^3$ is the quadrilateral type disjoint from $e.$ It follows that opposite edges of a 3--simplex have the same dihedral angle. Given a normal surface $S$, we then get an angle structure on $S$ induced by $\as$ as follows. If $t$ is a normal triangle in a tetrahedron $\simplex^3$, then the dihedral angles of $\simplex^3$ given by $\as$ induces angle assignments on $t$: the vertex of $t$ meeting the edge $e$ of $\simplex^3$ is assigned the angle at $e$. In the same way we also get angle assignments for vertices of normal quads. It follows from \Cref{eq:pi} that the three dihedral angles meeting a vertex of $\simplex^3$ sum to $\pi$ (equivalently, the total angle in any triangle of $S$ is $\pi$), and it follows from \Cref{eq:2pi} that the total dihedral angle around an edge of $\simplex^3$ is $2\pi$ (equivalently, the total angle around a vertex of $S$ is $2\pi$).

For each vertex $\v$ of $M$ let $N_\v$ be an open neighborhood of $\v$, chosen so that for distinct vertices the neighborhoods are disjoint, and so that $B_\v:=\del N_\v$ is a vertex linking normal surface. Let $M^c=M\setminus \bigcup N_\v$, so that $\del M^c=\bigcup B_\v$. For any isotopy class of simple closed curves on $B_\v$, there is a representative $\gamma$ that is \define{normal} with respect to the triangulation of $B_\v$ by normal triangles. That is, for any normal triangle $t$ of $B_\v$, $\gamma\cap t$ is a properly embedded arc cutting off a corner of $t$.

\begin{definition}
Let $\gamma$ be an oriented simple closed normal curve on $B_\v$ consisting of normal arcs $a_1,\dots,a_k$.  The \define{rotational holonomy} of $\as$ with respect to $\gamma$ is defined to be
$$
h_\gamma(\as)=\sum_{i=1}^k \epsilon_i \alpha_i
$$
where $\alpha_i$ is the angle associated to the corner of the triangle cut off by $a_i$, and $\epsilon_i=1$ if $a_i$ cuts off a corner to its left, $\epsilon_i=-1$ if $a_i$ cuts off a corner to its right.
\end{definition}

We remark that if complex shape parameters are associate to the ideal tetrahedra \cite{thurston78-lectures, Neumann-volumes-1985}, then the rotational holonomy of a peripheral curve $\gamma$ is the imaginary part of the logarithm of the derivative of its holonomy. If $h_\gamma(\as)=0$ for every normal curve $\gamma$ in $B_\v$ and every vertex $\v$, then we say that $\as$ has \define{vanishing peripheral rotational holonomy}. 

The rotational holonomy has the following natural properties. 
\begin{enumerate}
\item $h_{-\gamma}(\as) = -h_\gamma(\as),$ where $-\gamma$ denotes $\gamma$ with opposite orientation;
\item $h_{\gamma+\sigma}(\as) \equiv  h_{\gamma}(\as) + h_{\sigma}(\as)$ modulo $2\pi$, where addition of curves is oriented Haken sum.
\end{enumerate}

The following proposition is stated in \cite{GHHR16}. We supply some details that are not explicitly mentioned in the proof and which are of particular interest for implementation.

\begin{proposition}\cite[Proposition A.2]{GHHR16}\label{prop:Gara1}
Let $\tri$ be an ideal triangulation of a manifold $M$ with torus cusps. Then there exists a generalized angle structure on $\tri$ with vanishing peripheral rotational holonomy.
\end{proposition}

For example, if $M$ admits a complete hyperbolic structure and $\tri$ is a geometric ideal triangulation of $M$, then the angle structure induced by the hyperbolic structure has vanishing peripheral rotational holonomy.

\begin{proof}
The proof of \cite[Proposition A.2]{GHHR16} provides a generalised angle structure $\as$ with the following property. For each peripheral torus $B_\v$, one has a basis $\{\lambda_\v, \mu_\v\}$ with the property that $h_{\mu_\v}(\as) = h_{\lambda_\v}(\as)=0.$ To complete the proof of the proposition, it needs to be shown that $h_{p\mu_\v+q\lambda_\v}(\as) =0$ for all $p, q\in \ZZ.$

The argument given in \cite[pp.248-249]{neumann-combinatorics-1992} implies that this is the case so long as the normal curves representing the basis elements are \emph{simple}.\footnote{SnapPy will always choose such representatives for a basis when given a triangulation or a link diagram.} One way to choose such a basis algorithmically is as follows. Choose simple closed oriented curves $l_\v$ and $m_\v$ in the 1--skeleton of $B_\v$ with the property that $B_\v \setminus (l_\v \cup m_\v)$ is a disc and $l_\v \cap m_\v$ consists of a single vertex. The boundary of a small regular neighbourhood $N(l_\v)$ (resp. $N(m_\v)$)  of $l_\v$ (resp. $m_\v$) consists of two normal curves, and we give these curves the orientation induced from the orientation of $N(l_\v)$ as a subsurface of $B_\v.$ We then choose $\lambda_\v$ (resp. $\mu_\v$) to be the boundary component representing the same class in $H_1(B_\v)$ as $l_\v$ (resp. $m_\v$). 
\end{proof}

\subsection{Combinatorial Gauss--Bonet and Euler characteristic}\label{sec:as_compact}

Let $S$ be a surface with cell
decomposition $C$. Here we allow $S$ to be non-compact and to have boundary, and we allow the decomposition $C$ to have infinitely many cells.  Suppose further that there is a combinatorial angle structure on $(S,C)$, that is, a function that assigns to each corner $c$ of each cell in $S$ a real number $a(c)$, called the interior angle at $c.$ We note that an interior angle can be negative or zero.
Following Thurston~\cite[Chapter 13]{thurston78-lectures}, the
following combinatorial quantities are defined:

\begin{itemize}
\item If $c$ is a $n$--gon in $C$, its \define{combinatorial hyperbolic
    area} $A(c)$ is defined to be $(n-2)\pi -$(sum of interior angles
  of $c$).
\item For each vertex $v$ in $C$ in the interior of $S$, define its
  \define{combinatorial cone curvarture} to be $K(v) = 2\pi -$(sum of
  angles at $v$).
\item For each boundary curve $s$ of $S$, define a \define{combinatorial
    boundary curvature} by $K_{g}(s) = \sum_{v \in s} \pi - $(angle
  sum at $v$).
\end{itemize}

In the case when $S$ is compact and $C$ is finite, we have the following, which is a direct consequence of the definitions.

\begin{theorem}[Combinatorial Gau\ss--Bonnet formula]\label{thm:comb_GB}
  Let $S$ be a compact surface with a finite cell decomposition $C$
  and an angle structure. Then:
  \begin{equation}
   2 \pi \chi (S) = \sum_{v \in int(S)} K(v) - \sum_{c\in C} A(c) + 
     \sum_{s\subseteq \partial S} K_{g}(s)
  \end{equation}
\end{theorem}

Now suppose $M$ is the interior of an orientable compact manifold with non-empty boundary consisting of a union of tori, and $\tri$ an ideal triangulation with generalised angle structure.

If $f\co S\to M$ is a normal immersion, then the Euler characteristic of $S$ can be computed using the induced angle structure on $S$. If $S$ is non-compact, choose a compact core $\overline{S}$ of $S$ with the property that the annuli in the closure of $S \setminus \overline{S}$ only contain normal triangles.
The combinatorial area of each triangle is zero, and the combinatorial area of a quadrilateral $\Quad$ in a tetrahedron $\tet$ is equal to twice the angle $\as(\Quad)$ of the edge of $\tet$ which it misses. Moreover, the angle sum at each vertex in the cell decomposition equals $2\pi$. Thus one obtains:

\begin{equation}\label{eq:EC}
2 \pi \chi (S) = -\sum_{c\in C} A(c) + 
     \sum_{s\subseteq \partial \overline{S}} K_{g}(s) = -2 \sum_{\Quad\in C} \as(\Quad)+ 
     \sum_{s\subseteq \partial \overline{S}} K_{g}(s)
\end{equation}
where $C$ is the cell decomposition of $S$ into normal disks. For $S$ a closed surface we therefore have

\begin{equation}
\chi (S) = -\frac{1}{\pi} \sum_{\Quad\in C} \as(\Quad)
\end{equation}

Let $\as$ be a generalized angle structure on $\tri$ with vanishing peripheral rotational holonomy, as provided by \Cref{prop:Gara1}.
 For each transversely oriented normal quad $q\in \square$, define $\chi^\ast(\Quad)=-\frac{1}{\pi}\alpha(\Quad)$, and extend linearly over $\qns(\tri)$ to get a map $\chi^\ast:\qns(\tri)\to \RR$. 
 
The proof of the following proposition now follows by observing that the boundary curvature always vanishes due to the vanishing peripheral rotational holonomy:

\begin{proposition}\cite[Proposition A.3]{GHHR16}\label{prop:Gara2}
Let $\chi^\ast:\qns(\tri)\to \RR$ be the linear map defined above. If $f:S\to M$ is a spun normal immersion, then $\chi^\ast(x(S))=\chi(S)$. Specifically,
\begin{equation}
\chi (S) = \sum_{\Quad\in \square} \chi^\ast(\Quad) \cdot x(S)(q)
\end{equation}
\end{proposition}
We remark that the result was only stated for embeddings in \cite{GHHR16}, but the proof also applies to immersions.
%
%
%
The following corollary is immediate.

\begin{corollary}\label{cor:euler}
	The map $\chi^\ast:\qtons(\tri)\to \RR$, defined by $\chi^\ast (\vv{x}(S,\nu_S))=\chi^\ast(x(S))$, where $x(S)$ is obtained from $\vv{x}(S,\nu_S)$ by forgetting orientation, satisfies: if $(f,\nu_S):S\to M$ is a transversely oriented spun normal (unbranched) immersion, then $\chi^\ast(\vv{x}(S,\nu_S))=\chi(S)$.	
\end{corollary}

%
%
%
%
%
%


%% file: homology02.tex

\section{Homology map and the algorithm}
\label{sec:Homology map}

Throughout this section, $M$ is the interior of the compact, orientable 3--manifold $\overline{M}$ with boundary a finite union of tori, and $\tri$ an ideal triangulation of $M.$
We first define the homology map $\hom\co\qtons(\tri)\to H_2(\Mbar,\del\Mbar;\RR)$. 
This is essentially the simplicial homology map described in \cite{Cooper-norm-2009}, adapted to our setting.
For a link complement in a rational homology 3--sphere, we can pick out a canonical basis and use the positive and negative boundary maps defined in \Cref{sec:boundary curve map} to define the homology map. This is inspired by the examples of link complements in the 3--sphere analysed by Thurston~\cite{Thurston-norm-1986}.
With the homology map and our definitions in hand, \Cref{thm:ball} is \Cref{thm:main} with a more precise statement. 


\subsection{Simplicial homology map}\label{subsec:simplicial}

For each vertex $\v$ of the pseudo-manifold $P$ of \Cref{sec:Ideal triangulation}, let $N_\v$ be an open regular neighborhood of $\v$, and let $B_\v=\del \Nbar_\v$. As before, we assume that the regular neighborhoods are pairwise disjoint and identify $\Mbar = P \setminus N,$ where $N = \bigcup_{\v\in P^{(0)}} N_\v.$ We also let $N' \subset N$ be a smaller regular neighbourhood of the 0-skeleton with the property that $\overline{N'} \subset \Int(N).$ 

We have isomorphisms $H_2(P;\RR)\xrightarrow{g_1} H_2(P,N;\RR)\xrightarrow{g_2} H_2(\Mbar,\del\Mbar;\RR)$, the first coming from the long exact sequence of the pair $(P,N)$ and the fact that each connected component of $N$ retracts to a point, and the second coming from the excision theorem and the fact that each connected component of $N\setminus N'$ retracts to a vertex link.

The triangulation $\tri$ gives $P$ the structure of a simplicial complex, and using this we can compute the homology group $H_2(P,\RR)$. Let $C_2(P)$ be the 2-dimensional chain group, which is generated by oriented triangular faces of the two-skeleton of $\tri$. 

\begin{remark}
The homology map for closed 3--manifolds in \cite{Cooper-norm-2009} uses the fact that a closed normal immersion only has finitely many discs and hence one can homotope a transversely oriented immersion in the direction of the transverse orientation into the 2--skeleton. Small triangles map to vertices, quadrilaterals to edges, and the homology map thus
 essentially factors through a projection onto the coordinates of the large triangles. Here, we construct a different map using only the quadrilateral coordinates --- we map a quadrilateral dual to an edge of a 3--simplex to the two faces of the simplex containing that edge, taken with appropriate signs. It is interesting to note that this new map can also be used for closed 3--manifolds. The upshot is that one can determine the homology class of a closed transversely oriented normal immersion in a closed 3--manifold from the small triangles alone, from the large triangles alone, and from the quadrilaterals alone. In particular, if such a surface represents a non-trivial homology class, it must have at least one disc of each type.
\end{remark}

Let $\Delta^3$ be a tetrahedron in $\tri$, with vertices labeled $v_i$, $0\le i\le 3$. Recall that $q_{ij}$, $0\le i<j\le 3$, is the transversely oriented normal quad that separates the edge $e_{ij}$ between $v_i$ and $v_j$ from the opposite edge, and such that the transverse orientation on $q_{ij}$ points toward $e_{ij}$. Let $t_i$ be the face of $\Delta^3$ across from $v_i$. 
First we define $\varphi\co\qtons(\tri)\to C_2(P)$ for oriented quads $q_{ij}\in\vv{\square}$ by 
$$\varphi(q_{ij})=\sum_{k\in\{0,\dots,3\}\setminus\{i,j\}}\epsilon_k t_k$$
where $\epsilon_k=1$ if the orientation on $q_{ij}$ (induced by the transverse orientation and the orientation on $M$) agrees with the orientation on $t_k\in C_2(P)$, and $\epsilon_k=-1$ otherwise. We now define $\varphi$ on all of $\qtons(\tri)$ be extending linearly and restricting.

Let $\del_i$ be the chain map from the simplicial chain group $C_i(P)$ to $C_{i-1}(P)$, $i=2,3$, and let $h\co\ker\del_2 \to \ker \del_2 /\im \del_3=H_2(P;\RR)$ be the canonical homomorphism.

The heart of the proof of the following proposition is a construction of Segerman~\cite{Seg09}, which shows that an \emph{embedded} normal surface can be isotoped such that it coincides in $\Mbar$ with a surface made up of twisted squares that appears in work of Yoshida~\cite{Yoshida-ideal-1991}. The twisted squares and the normal quadrilaterals of the isotopic surfaces are encoded by the same $Q$--coordinate, and the correspondence is shown in \Cref{fig:twisted-square}. The map from transversely oriented twisted squares to oriented triangles is then straightforward.

\begin{figure}[h]
    \centering
    \begin{subfigure}{.24\textwidth}
        \includegraphics[width=4cm]{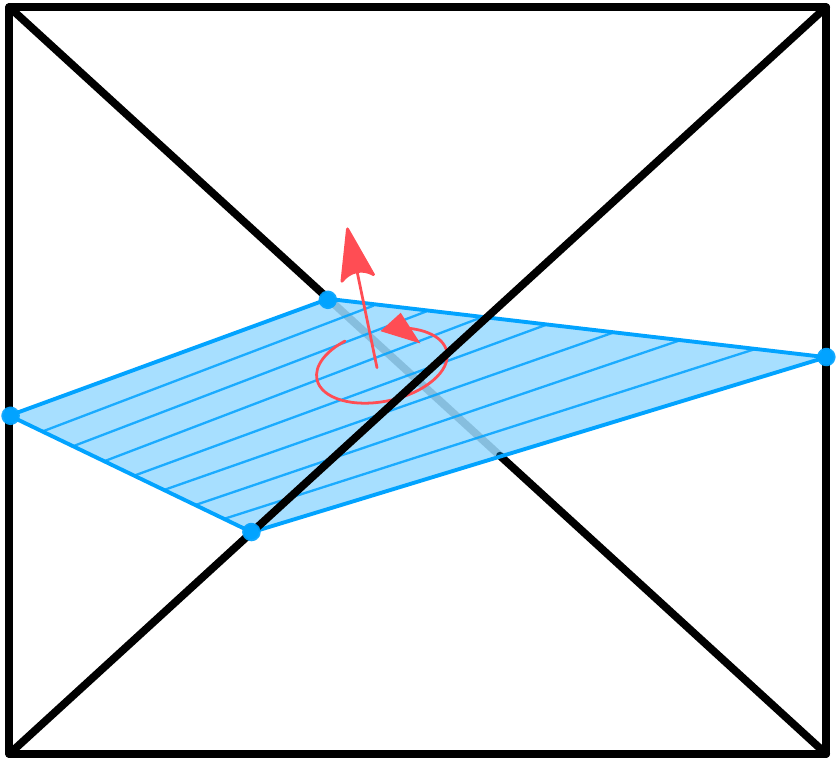}
        \caption{}
        \label{fig:ts1}
    \end{subfigure}
    \begin{subfigure}{.24\textwidth}
        \includegraphics[width=4cm]{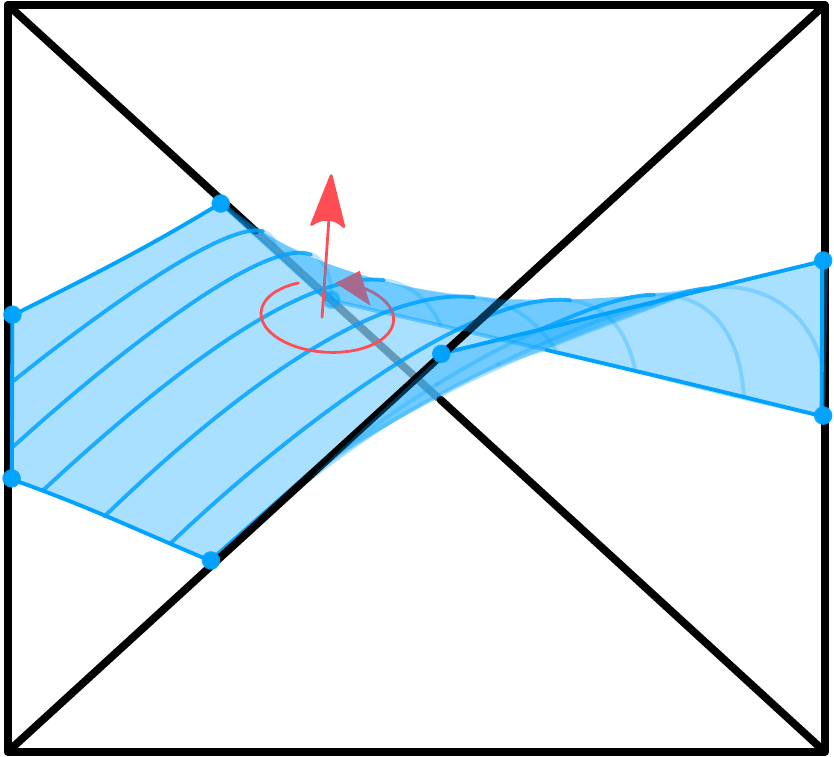}
    	\caption{}
        \label{fig:ts2}
    \end{subfigure}
        \begin{subfigure}{.24\textwidth}
        \includegraphics[width=4cm]{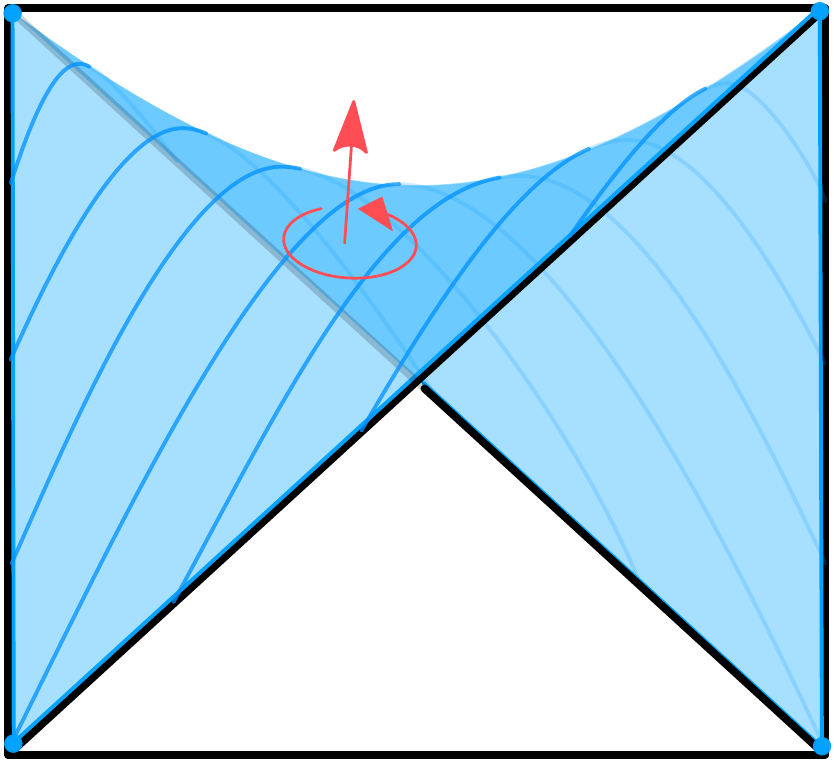}
        \caption{}
        \label{fig:ts3}
    \end{subfigure}
    \begin{subfigure}{.24\textwidth}
        \includegraphics[width=4cm]{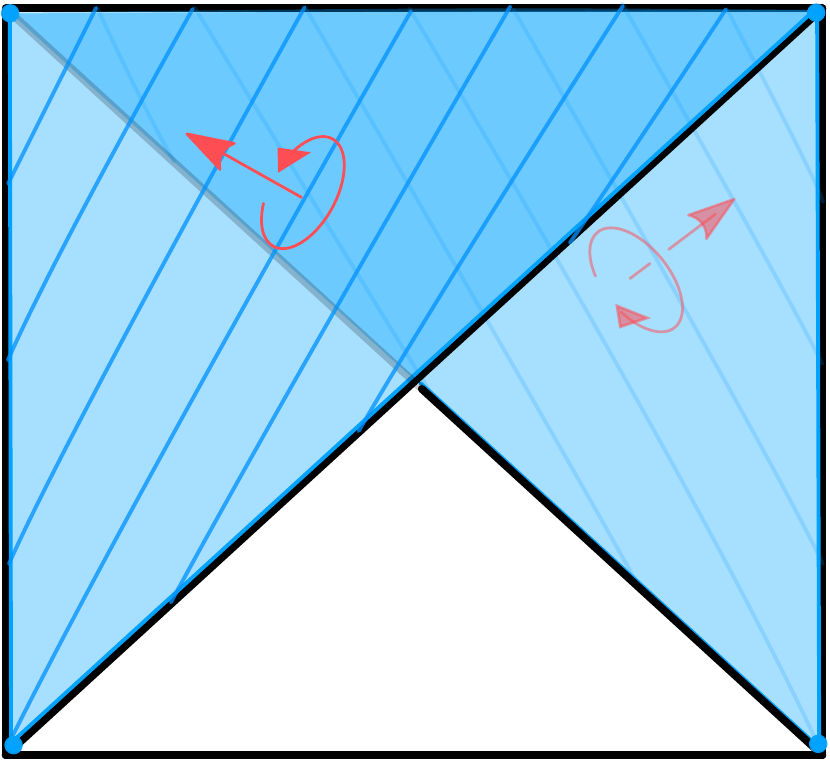}
    	\caption{}
        \label{fig:ts4}
    \end{subfigure}
    \caption{From transversely oriented normal quadrilaterals via twisted squares to oriented triangles}
    \label{fig:twisted-square}
\end{figure}

\begin{proposition}
	 Let $\Mbar$ be a manifold with non-empty torus boundary and hyperbolic interior. Let $\hom=g_2\circ g_1 \circ h\circ\varphi$, with $\varphi$, $h$, $g_1$, and $g_2$ as defined above. Then $\hom\co\qtons(\tri)\to H_2(\Mbar,\del\Mbar;\RR)$ is a linear map and satisfies the following property: If $(f,\nu_S)\co S\to M$ is a transversely oriented spun-normal immersion, then $f_\ast ([S])=\hom(\vv{x}(S,\nu_S))$. Moreover, if $b_1(\Mbar)\ge 2$, then $\varphi$ is surjective.
\end{proposition}

\begin{proof}
	Let $\vv{x}\in \qtons(\tri)$ and let $(S,\nu_S)$ be the properly immersed normal surface given by \Cref{thm:admissible integer t.o. solution gives normal immersed}. If $S$ is an embedding, then by \cite[Theorem 3.1]{Seg09}, $S$ can be isotoped to a surface $S'$ so that in $P\setminus L$ it agrees with the corresponding twisted square surface, also described in \cite{Seg09}, and such that all triangles of $S'$ are in $L$. Note that although \cite{Seg09} does not explicitly address non-embedded surfaces, it is nonetheless the case that such a homotopy also exists for properly immersed surfaces. Indeed, the isotopy described in \cite{Seg09} can be realized as a composition of isotopies, each of which is the identity outside of $A=\mathrm{int}(B(e))\setminus N$ for some edge $e$, where $B(e)$ is the abstract neighborhood $B(e)$ of $e$ (as defined in \Cref{sec:matching}). Components of $S\cap A$ are embedded, so within $A$ we can isotope a component $F$ of $S\cap A$ exactly as shown in \cite[Figure 4]{Seg09}.  From the isotopy in \cite{Seg09} that describes how to obtain $S'$ from $S$, we see that a quadrilateral $q$ of $S$ is isotoped to a twisted square $s$ that can be assumed to be isotopic to $\varphi(q)$ inside $P\setminus N$. Hence within $P\setminus N$, $S'$ is homotopic to some orientation consistent gluing $X$ of the double-triangles of $\varphi(\vv{x})$, and $\varphi(\vv{x})\in \ker \del_2$ since $X$ has no boundary. 
	
It follows that in the relative homology group $H_2(P,L;\RR)$, $g_1\circ h\circ \varphi(\vv{x})=f_\ast([S])$. 
Hence when considered in $H_2(\Mbar,\del\Mbar;\RR)$, $f_\ast([S])=\hom(\vv{x})=\hom(\vv{x}(S,\nu_S))$.
	
If $b_1(\Mbar)\ge 2$, the vertices of the norm ball in $H_2(\Mbar,\del\Mbar)$ are realizable as embedded spun-normal surfaces. Thus the points $\hom(\vv{x})$ such that $\vv{x}\in\qtons(\tri)$ span $H_2(\Mbar,\del\Mbar;\RR)$, and so the images $\phi(\vv{x})$ span $\ker \del_2$, and it follows that $\varphi$ is surjective.
\end{proof}



\subsection{Peripheral homology map}\label{subsec:peripheral}

We are now interested in the situation where the homology map factors through the boundary map.

The long exact sequence of the pair $(\Mbar,\del\Mbar)$ gives a homomorphism $\varphi:H_2(\Mbar,\del\Mbar;\RR)\to H_1(\del\Mbar;\RR)$, which for a surface $S\subset \Mbar$ representing a class $[S]\in H_2(\Mbar,\del \Mbar;\RR)$, maps $[S]$ to $[\del S]\in H_1(\del\Mbar;\RR)$. Combined with the homology map, we obtain
\[ 
\varphi\circ\hom: \qtons(\tri)\to H_2(\Mbar,\del\Mbar;\RR)\to H_1(\del\Mbar;\RR)
\]
We also let $\del :\qtons(\tri)\to H_1(\del\Mbar;\RR)$ be the sum of the maps $\del^+,\del^-:\qtons(\tri)\to H_1(\del\Mbar;\RR)$ defined in \Cref{sec:boundary curve map}. That is, we define $\del(\vv{x})=\del^+(\vv{x})+\del^-(\vv{x})$. The definitions of the maps imply $\varphi\circ\hom = \del.$

Using the universal coefficient theorem and Poincar\'e-Lefschetz duality, we have $H_2(\Mbar,\del\Mbar;\RR) \cong H^1(\Mbar;\RR) \cong H_1(\Mbar;\RR).$
The standard half-lives half-dies argument implies for homology with rational coefficients that the image of the map $H_1(\del\Mbar;\RR) \to H_1(\Mbar;\RR)$ induced by the inclusion map has rank $k,$ where $k$ is the number of boundary components of $\Mbar.$ Hence $\varphi$ is an isomorphism onto its image if $H_1(\del\Mbar;\RR) \to H_1(\Mbar;\RR)$ is a surjection. 

So assume that $H_1(\del\Mbar;\RR) \to H_1(\Mbar;\RR)$ is a surjection. 
In particular, we can make an identification $\Mbar=S_{\QQ}\setminus \N(L)$, where $S_\QQ$ is a rational homology sphere, and $L=L_1\cup\dots\cup L_k$ is a link in  $S_\QQ$. Here, $\N(L)$ is an open regular neighbourhood of $L$, and we let $\N(L_i)$ denote the connected component that is a regular neighbourhood of $L_i$ and $B_i$ denote the frontier of $\N(L_i).$


Consider $M_i = S_\QQ\setminus \N(L_i).$ Then $\partial M_i = B_i.$
The half-lives half-dies argument implies that there is a basis $\langle \mu_i, \lambda_i \rangle= H_1(B_i, \mathbb{Z})$ with the property that $\mu_i$ maps to an element of infinite order whilst $\lambda_i$ maps to an element of finite order under the inclusion map to $H_1(M_i, \mathbb{Z}).$ In particular, $\lambda_i$ is uniquely determined up to sign, whilst $\mu_i$ is only well-defined up to sign and a power of $\lambda_i.$ 
It follows from the Mayer-Vietoris sequence that $\mu_i$ may be chosen as the geometric meridian for $L_i$ and that $\lambda_i$ is isotopic to the core curve $L_i$ of $\N(L_i).$

The set $\{\mu_1,\lambda_1,\dots,\mu_k,\lambda_k\}$ is a basis for $H_1(\del\Mbar;\RR)$. 
Let $p:H_1(\del\Mbar;\RR)\to \Lambda \subset H_1(\del\Mbar;\RR)$ be the projection to the subspace spanned by $\lambda_1,\dots, \lambda_k$. We emphasise that this depends on the identification of $M$ as a link complement in a rational homology sphere.

Suppose the order of  $\lambda_i$ is $n_i$ in $H_1(M_i, \mathbb{Z}).$ A construction due to Stallings~\cite{Stallings-fibering-1962} produces a properly embedded, connected, oriented surface $S_i$ in $M_i = S_\QQ\setminus \N(L_i)$ with $[\partial S_i] = n_i \lambda_i.$ We may isotope this surface to be transverse to all other components $L_j$ of $L$. Hence we obtain a properly embedded surface $\mathring{S}_i=S_i\cap \Mbar$ in $\Mbar$ with the property that $\partial \mathring{S}_i$ is the union of $n_i$ copies of $\lambda_i$ on $B_i$ and a number of meridians on the other boundary components of $\Mbar.$ 
Hence $p\circ \varphi([\mathring{S}_i])=n_i \lambda_i$. In particular, the subspace in $H_2(\Mbar,\del\Mbar;\RR)$ spanned by the classes represented by $\mathring{S}_1,\ldots,\mathring{S}_k$ is mapped onto $\Lambda$ by $p\circ\varphi.$ This implies that $\langle [\mathring{S}_1],\ldots,[\mathring{S}_k] \rangle$ has dimension $k$, and since the dimensions of $\Lambda$ and $H_2(\Mbar,\del\Mbar;\RR)$ are both $k$, $p\circ \varphi$ is an isomorphism. In particular:
\[ 
H_2(\Mbar,\del\Mbar;\RR) = \langle [\mathring{S}_1],\ldots,[\mathring{S}_k] \rangle  \cong \langle \lambda_1,\ldots,\lambda_k\rangle = \Lambda 
\]
We therefore have an isomorphism $\psi\co\Lambda \to H_2(\Mbar,\del\Mbar;\RR)$ of $\RR$-vector spaces given by 
$\psi(\lambda_i) = \frac{1}{n_i}[\mathring{S}_i]$, and it follows from this discussion that $\hom = \psi\circ p\circ\del.$

%

%% file: algo01.tex

\subsection{The algorithm}
\label{sec:algorithm}

For correctness of our algorithm, we use the following result.

\begin{theorem}[Walsh]\cite[Theorem 1.6]{Wal11} \label{thm:walsh}
Let $M$ be an atoroidal, acylindrical, irreducible, compact 3-manifold with torus cusps, and let $\tri$ be an ideal triangulation of $M$ such that no edge of $\tri$ is isotopic into a cusp. Let $S\subset M$ be a properly embedded, 2-sided, incompressible surface that is not a virtual fiber. Then there is an embedded spun normal surface in $\tri$ that is properly isotopic to $S$.
\end{theorem}

For $M$ a 3-manifold, let $b_1(M)$ denote the first Betti number.

\begin{theorem}\label{thm:ball}
	Let $\Mbar$ be a 3-manifold with non-empty torus boundary such that $M=\mathrm{int}(\Mbar)$ is hyperbolic, and let $\tri$ be an ideal triangulation of $M$ that is $0$-efficient. Let $\Pqtons(\tri)$ be the projective solution space and let $B$ be the finite set of points
	
	$$
	\left\{
	\frac{v}{|\chi^\ast(v)|}: v \text{ is an admissible vertex of $\Pqtons(\tri)$ with } \chi^\ast(v)<0
	\right\}
	$$

If $b_1(M)\ge 2$, then the convex hull $\hull{\hom(B)}$ of $\hom(B)$ is the unit ball $\B$ of the Thurston norm on $H_2(\Mbar,\del \Mbar;\RR)$, where $\hom$ is the homology map defined in \Cref{sec:Homology map}.
\end{theorem}

The proof of the theorem is almost identical to the proof of the corresponding theorem in \cite{Cooper-norm-2009} and given here for completeness.

\begin{proof}

First we show that $\hull{\hom(B)}\subset \B$. Since $\B$ is convex, it is enough to show that  $\hom(B)\subset \B$. Let $v\in B$, and observe that $v$ has rational coordinates since the vertices of $\Pqtons(\tri)$ have integral coordinates. Let $t>0$ be minimal so that $x=t\cdot v$ is an integral point. Since $v\subset \vv{Q}(\tri)$ and is admissible, $x\in \vv{Q}(\tri)$ is admissible. By \Cref{thm:admissible integer t.o. solution gives normal immersed}, there is a transversely oriented normal immersion $(f,\nu_S):S\to M$ such that $\vv{x}(S,\nu_S)=x$, and by \Cref{prop:bdy_map} we have $\hom(x) = f_\ast([S])\in H_2(\Mbar,\del \Mbar)$. Since $S$ is immersed, \Cref{cor:euler} gives $\chi^\ast(x)=\chi(S)$.

If $S$ is not connected, then $x$ is the sum over all components $S_i\subset S$ of $\vv{x}(S_i,\nu_{S_i})\in \qtons(\tri)$. Since $v$ is a vertex, it is not a non-trivial convex combination of points in $\Pqtons(\tri)$. On the other hand,

$$
\vv{x}(S,\nu_S)=\sum_i \vv{x}(S_i,\nu_{S_i})
$$

It follows that $\vv{x}(S,\nu_S)$ is a multiple of $\vv{x}(S_i,\nu_{S_i})$, where $S_i$ is any component of $S$. By minimality of $t$ it follows that $S$ is connected. Since $\chi(S)=\chi^\ast(x)=t\chi^\ast(v)<0$, we have $-\chi_-(S)=\chi(S)=\chi^\ast(x)$. By \cite[Corollary 6.18]{Gabai-foliations-1983}, the Thurston norm for embedded surfaces is equal to the norm for singular surfaces, so 
$$
\|\hom(x)\| \le \chi_-(S)=|\chi(S)|=|\chi^\ast(x)|
$$
from which it follows that
$$
\|\hom(v)\|= \frac{\|\hom(x)\|}{|\chi^\ast(v)|}\le 1
$$

Since $\hull{\hom(B)}$ and $\B$ are both convex, to show that $\B\subset\hull{\hom(B)}$ it is enough to show that the vertices of $\B$ are in $\hull{\hom(B)}$. To this end, let $y\in\B$, and let $S$ be a norm minimizing transversely oriented surface so that $y=[S]/|\chi(S)|$. Since $b_1(M)\ge 2$, vertices of $\B$ are not top-dimensional faces, so $S$ does not lie over a top-dimensional face. Thus by \cite[Theorem 3]{Thurston-norm-1986}, $S$ is not homological to a fiber. Now suppose that $p:N\to M$ is a finite cover, and denote the respective Thurston norm balls by $\B_N$ and $\B_M$. Identifying $H_2(\Mbar,\del \Mbar)$ with $H^1(\Mbar)$, \cite[Corollary 6.13]{Gabai-foliations-1983} implies that $p^\ast(\B_M)=\B_N\cap p^\ast(H^1(\Mbar))$, so vertices of $\B_M$ do not lift to top-dimensional faces of $\B_N$ (again using that $b_1(M)\ge 2$). It follows then that $S$ cannot be homological to a virtual fiber. Thus by \Cref{thm:walsh} $S$ can be isotoped into spun-normal form. 

Let $x=\vv{x}(S,\nu_S)$. The ray through $x$ intersects $\Pqtons(\tri)$ in a point $v$, which is admissible since $x$ is. The point $v$ can be expressed as a convex linear combination of vertices of $\Pqtons(\tri)$, so that we have $v=\sum t_i v_i$ where each $v_i$ is a vertex and $0\le t_i\le 1$. For each $v_i$, let $\bar{v}_i$ be the smallest multiple of $v_i$ that is an integral point. Since $\tri$ is 0-efficient, none of the $\bar{v}_i$ are normal 2-spheres. Since $M$ is hyperbolic it is $\del$-incompressible, so any disk in $M$ is isotopic into a cusp. It follows that $\chi^\ast(\bar{v}_i)\ne 1$, and consequently that $\chi^\ast(v_i)\le 0$ for all $i$. Since $v$ is admissible, each $v_i$ is also admissible. If $\chi^\ast(v_i)=0$, then $\bar{v}_i$ is the normal coordinate of a connected, immersed surface of zero Euler characteristic, so $\hom(\bar{v}_i)=0$ since $M$ is atoroidal.

Let $\hat{v}=\sum t_i \hat{v}_i$, where $\hat{v}_i=v_i$ if $\chi^\ast(v_i)<0$ and $\hat{v}_i=0$ otherwise. Then $\hom(\hat{v})=\hom(v)$, $\chi^\ast(\hat{v})=\chi^\ast(v)$ and each non-zero $\hat{v}_i$ is a rational multiple of a point in $B$. If follows that $\hat{v}_i/|\chi^\ast(\hat{v}_i)|\in B$, and therefore that $\hom(\hat{v}_i/|\chi^\ast(\hat{v}_i|)=\hom(\hat{v}_i)
/|\chi^\ast(\hat{v}_i)| \in \hom(B)$. 
Since $\chi^\ast(\hat{v})=\sum t_i\cdot\chi^\ast(\hat{v}_i)$ and $\chi^\ast(\hat{v}_i)<0$ for all $i$, we obtain $1= \sum t_i\cdot\frac{|\chi^\ast(\hat{v}_i)|}{|\chi^\ast(\hat{v})|}$. Thus we have 

$$
\sum t_i\cdot\frac{|\chi^\ast(\hat{v}_i)|}{|\chi^\ast(\hat{v})|}\frac{\hom(\hat{v}_i)}{|\chi^\ast(\hat{v}_i)|}=\sum t_i\cdot \frac{\hom(\hat{v}_i)
}{|\chi^\ast(\hat{v})|}\in \hull{\hom(B)}
$$
which implies that
$$
y=\frac{[S]}{|\chi(S)|}=\frac{\hom(x)}{|\chi^\ast(x)|}=\frac{\hom(v)}{|\chi^\ast(v)|}=\frac{\hom(\hat{v})}{|\chi^\ast(\hat{v})|}=\sum t_i\cdot \frac{\hom(\hat{v}_i)
}{|\chi^\ast(\hat{v})|}\in \hull{\hom(B)}
$$
We thus have $\B\subset \hull{\hom(B)}$, which concludes the proof.
\end{proof}


We are now ready to prove \Cref{thm:main}, which was stated in \Cref{sec:intro}.

\begin{proof}[Proof of \Cref{thm:main}]
Let $b_1(M)$ denote the first betti number of $M$. If $b_1(M)\ge 2$, then \Cref{thm:ball} provides an algorithm via normal surfaces. Now suppose that $b_1(M)=1$. Note that \Cref{thm:ball} fails for 1-dimensional homology because in that case the top-dimensional faces are vertices, and therefore vertices of the norm ball can be virtual fibers, which are not always realizeable as normal surfaces. Also note that for the 1-dimensional case the norm ball is determined by two vertices $\pm t \in \RR \cong H_2(\Mbar,\del \Mbar;\RR)$.

Since $\Mbar$ has non-empty boundary and $b_1(M)=1$, it must have a single boundary component $T$. Let $\gamma\in H_1(T;\ZZ)$ be the homological longitude (i.e., the unique curve trivial in $H_1(\Mbar;\QQ)$). The curve $\gamma$ is the boundary of a connected surface $S$ representing a non-trivial class in $H_2(\Mbar,\del\Mbar;\RR)$. Our goal is to determine $|\chi(S)|=\|[S]\|_T$, which we will do by Dehn filling along the boundary of $S$ then applying the algorithm for closed manifolds of Cooper--Tillmann \cite{Cooper-norm-2009}. 

Let $\Mbar_\gamma$ be the Dehn filling of $\Mbar$ along $\gamma$. Let $\tau$ be the order of the torsion subgroup of $H_1(\Mbar)$, and let $\tau_\gamma$ be the order of the torsion subgroup of $H_1(\Mbar_\gamma)$. In order to use the Thurston norm computations for $\Mbar_\gamma$ to get the norm of $[S]$, we will need to know how many boundary components $S$ has. We claim that $|\del S|=\tau/\tau_\gamma$. This is because if $|\del S|=k$, then the homological longitude $\gamma$ has order $k$ in $H_1(\Mbar,\ZZ)$, since $k$ copies of it bound a 2-cycle, and $\gamma$ is killed under Dehn filling. To ensure that Dehn filling does not kill any other part of the torsion subgroup, we can use the Mayer--Vietoris sequence of $M_\gamma=M\cup \mathbf{T}$, where $\mathbf{T}$ is the filling solid torus. The relevant part of the sequence is
$$
0\to H_2(\Mbar_\gamma)\xrightarrow{\,\,\del\,\,} H_1(\del \Mbar)\xrightarrow{\,\,\varphi\,\,} H_1(\Mbar)\oplus H_1(\mathbf{T})\xrightarrow{\,\,\psi\,\,} H_1(\Mbar_\gamma)
$$
With $\overline{S}$ denoting the closed surface in $\Mbar_\gamma$ resulting from capping off $\del S$, we have 
$$\ker\varphi=\im(\del)=\del([\overline{S}])=\del S=k\cdot \gamma \in H_1(\del\Mbar)\cong \ZZ\langle \gamma\rangle\times\ZZ\langle\lambda\rangle
$$ 
for appropriately chosen $\lambda\in H_1(\del \Mbar)$. Thus 
$\ker\psi=\im\varphi = \ZZ\langle(\lambda,-\lambda)\rangle\times \ZZ_k\langle(\gamma,0)\rangle
$. 
Letting $\psi_0$ be the restriction of $\psi$ to $H_1(\Mbar)\oplus \{0\}$, we get $\ker\psi_0 = \ZZ_k$, since it must be a subgroup of both $\ker \psi$ and $H_1(\Mbar)$.

We can now complete the proof. After Dehn filling along $\gamma$, \cite{Cooper-norm-2009} allows us to find $|\chi(\overline{S})|$, and hence we can compute $\genus(\overline{S})=\genus(S):= g$. Then $|\chi(S)|=2g-2+|\del S|=2g-2+\tau/\tau_\gamma$.
%
\end{proof}


%% file: implementation03.tex

\section{Implementation and examples}
\label{sec:implementation and examples}

\subsection{Implementation}
\label{sec:implementation}

The algorithm described in \Cref{sec:algorithm}, along with the algorithm for closed manifolds given in \cite{Cooper-norm-2009}, has been implemented in the computer program \texttt{Tnorm} \cite{tnorm20}, available at \url{https://pypi.org/project/tnorm/}. In its current form, \texttt{Tnorm} requires working installations of both \texttt{Regina} and \texttt{SnapPy} inside \texttt{Sage}. We hope to remove the dependence on \texttt{Sage} in a future release. 

In addition to computing the Thurston norm ball of a hyperbolic 3-manifold, \tt{Tnorm} has several other features that are useful when working with spun-normal surfaces. Most notably \tt{Tnorm} can check for embeddedness, connectedness, and orientability, and can compute the Euler characteristic and boundary slopes of a spun-normal surface. These features are also implemented for transversely oriented spun-normal surfaces. Although Regina can compute boundary slopes of (non-transversely oriented) spun-normal surfaces, none of the other features listed are currently implemented in Regina for spun-normal surfaces. Thus \tt{Tnorm} is useful as a companion to \tt{Regina} if one is interested in studying these surfaces. 

To compute the boundary slopes of a spun-normal surface, \tt{Tnorm} uses the basis for $H_1(\del M;\ZZ)$ computed by \tt{SnapPy}. If $M = \SS^3\setminus L$, $L=\cup_i L_i$, is a link complement and \tt{SnapPy} has a link diagram for $L$ (for example, this is the case if $M$ is a link from one of \tt{SnapPy}'s link censuses), then this is the \emph{knot-theoretic basis}. That is, the basis used for the component $L_i$ will consist of the meridian and the longitude that is trivial in $H_1(\SS^3\setminus L_i)$. Otherwise, provided a hyperbolic structure is found for $M$, \tt{SnapPy} defaults to a \emph{geometric basis} consisting of the two shortest simple closed curves on each cusp cross-section.

Before moving on to some examples, we say a bit more about how we check that a transversely oriented normal surfaces is embedded. Given $\vv{x}\in\qtons(\tri)$, we would like to determine if there exists a transversely oriented normal surface $(S,\nu)$ that is embedded and such that $\vv{x}=\vv{x}(S,\nu)$. Thus we begin by fixing $\vv{x}\in \qtons(\tri)$, then obtain the normal quadrilateral coordinate $x$ by forgetting orientations on $\vv{x}$. There is a unique embedded normal surface $S'$ with quadrilateral coordinate $x$ by \cite{Kang-normal-2005, tillmann08-normal}. For a connected component $S_i'$ of $S'$, we attempt to put a transverse orientation on $S_i'$ by putting an orientation on some initial quadrilateral $q_0$, then iteratively assign compatible orientations to all quadrilaterals adjacent to an oriented quadrilateral, until either all quadrilaterals are oriented, or some quadrilateral is adjacent to two oriented quadrilaterals with incompatible orientations. If successful, we get two transverse orientations $\pm \nu_i'$ for each connected component, and hence $2^k$ possible orientations on $S'$. If among these $2^k$ orientations on $S'$ there is an orientation $\nu'$ such that $\vv{x}=\vv{x}(S', \nu')$, then $(S',\nu')$ is an embedded representative for $\vv{x}$. If none of the orientations found match the orientation for $\vv{x}$, or if $S'$ is not orientable, then since $S'$ is the unique embedded normal surface with quadrilateral coordinate $x$, $\vv{x}$ cannot have an embedded representative. For if it did, then such a surface would be an orientable embedded normal surface with quadrilateral coordinate $x$ but not normally isotopic to $S'$, contradicting uniqueness of $S'$. 


\subsection{Examples}

\texttt{Tnorm} can be used as an application with a graphical user interface via the terminal command 
\begin{center}
\texttt{\$ sage -python -m tnorm.app}
\end{center}
or can be imported as a module in a \texttt{Sage} session. Here we will demonstrate some of the features of \tt{Tnorm} with some examples, using the command line interface.

We start with the figure-8 knot complement. The initial function \tt{tnorm.load()} will take as input any input that \tt{SnapPy} will accept, as well as \tt{SnapPy} manifolds and triangulations, and \tt{Regina} triangulations. After loading the figure-8 knot complement (a.k.a., \tt{`4\_1'}), we see that \tt{Tnorm} confirms what we found in \Cref{sec:example fig8}. Namely, the generator of its second homology (rel boundary) is represented by an immersed thrice-punctured sphere. There is an outward boundary component with slope $2\mu+\lambda$, and two inward boundary components each with slope $-\mu$. Since the outward boundary slope and the negative of the inward boundary are not equal, the union of the ends of the surface cannot be normally homotoped to be embedded. Note that although there are other transversely oriented normal surfaces that represent the same homology class, \tt{Tnorm} will always check for an embedded normal surface to represent the primitive integer point over a vertex of the norm ball.

\begin{python}
sage: import tnorm
sage: W=tnorm.load('4_1')
Enumerating quad transversely oriented normal surfaces (qtons)... Done.
sage: B=W.norm_ball
Computing Thurston norm unit ball... Done.
sage: B.vertices()
[Vertex 0: represented by S_0,3 at (-1), mapped from immersed qtons with index 0,
 Vertex 1: represented by S_0,3 at (1), mapped from immersed qtons with index 1]
sage: W.boundary_slopes(1)  # compute the boundary slopes of qtons with index 1
{'outward':[(2, 1)], 'inward':[(-2, 0)]}
sage: W.is_embedded(1), W.ends_embedded(1)
False, False
\end{python}

In the case of the immersed thrice-punctured sphere realizing the homology class of the fiber of the figure-8 knot, the ends are not embedded, and this obstructs embeddedness of the thrice-punctured sphere. We also find examples where the norm minimizer is not realized by an embedded spun normal surface, but the ends are embedded. In particular this is the case for census knot \tt{L7a1} with the census triangulation:
\begin{python}
sage: W=tnorm.load('K7a1',quiet=True); B=W.norm_ball
sage: B.vertices()
[Vertex 0: represented by (1/3)*S_2,1 at (-1/3), mapped from immersed qtons with index 6,
 Vertex 1: represented by (1/3)*S_2,1 at (1/3), mapped from immersed qtons with index 1]
sage: W.ends_embedded(6), W.is_embedded(6)
True, False
\end{python}

Both of the above examples are fibered knots, as we would expect given \Cref{thm:walsh}, which guarantees that surfaces that are not fibers or virtual fibers can be realized as embedded normal surfaces. In fact, we find that among the first 268 hyperbolic knots in the census of Hoste and Thistlethwaite, there is no fibered knot with fiber that is realized by an embedded normal surface in the census triangulation. On the other hand, it is possible that for some other triangulation the fiber can be normalized. For example, although the census triangulation for the figure-8 knot complement does not contain an embedded spun-normal fiber (see \Cref{sec:example fig8}), we demonstrate below another triangulation in which the fiber can be normalized while staying embedded:

\begin{python}
sage: import tnorm; import regina; import snappy
sage: path=[(2, 1), (2, 1), (2, 6), (2, 5), (1, 5), (2, 9), (1, 0)]
sage: M=snappy.Manifold('K4a1'); T=regina.Triangulation3(M._to_string())
sage: for dim,index in path:
....:     T.pachner(T.face(dim,index))
sage: T.size()   #number of tetrahedra in T
5
sage: N=snappy.Manifold(T.snapPea())
sage: N.set_peripheral_curves('shortest') #compute a nice peripheral basis
sage: W=tnorm.load(N); B=W.norm_ball
sage: B.vertices()
[Vertex 0: represented by S_1,1 at (-1), mapped from embedded qtons with index 10,
 Vertex 1: represented by S_1,1 at (1), mapped from embedded qtons with index 22]
\end{python}

This example gives us some hope that the following question could have a positive answer:

\begin{question}
For a fibered knot complement $M=\SS^3\setminus K$ or fibered once-cusped hyperbolic 3--manifold $M$, is there always some ideal triangulation of $M$ such that the fiber is realized as an embedded spun-normal surface?
\end{question}

As discussed in the proof of \Cref{thm:main}, when $M$ is a knot complement it may happen that the generator of $H_2(M;\del M)$ is not realized by \emph{any} norm-minimizing transversely oriented normal surface (immersed or embedded). In fact this is the case for the census knot \tt{K8a16} with the census triangulation. Nonetheless, \tt{Tnorm} is still able to determine the norm ball, which it does by computing the genus of the knot using the method described in the proof of \Cref{thm:main}:

\begin{python}
sage: W=tnorm.load('K8a16',quiet=True); B=W.norm_ball
sage: B.vertices()
[Vertex 0: represented by (1/5)*S_3,1 at (-1/5,), not represented by any qtons,
 Vertex 1: represented by (1/5)*S_3,1 at (1/5,), not represented by any qtons]
\end{python}

\begin{figure}[h]
    \begin{subfigure}{.233\textwidth}
    \centering
        \includegraphics[scale=.2098]{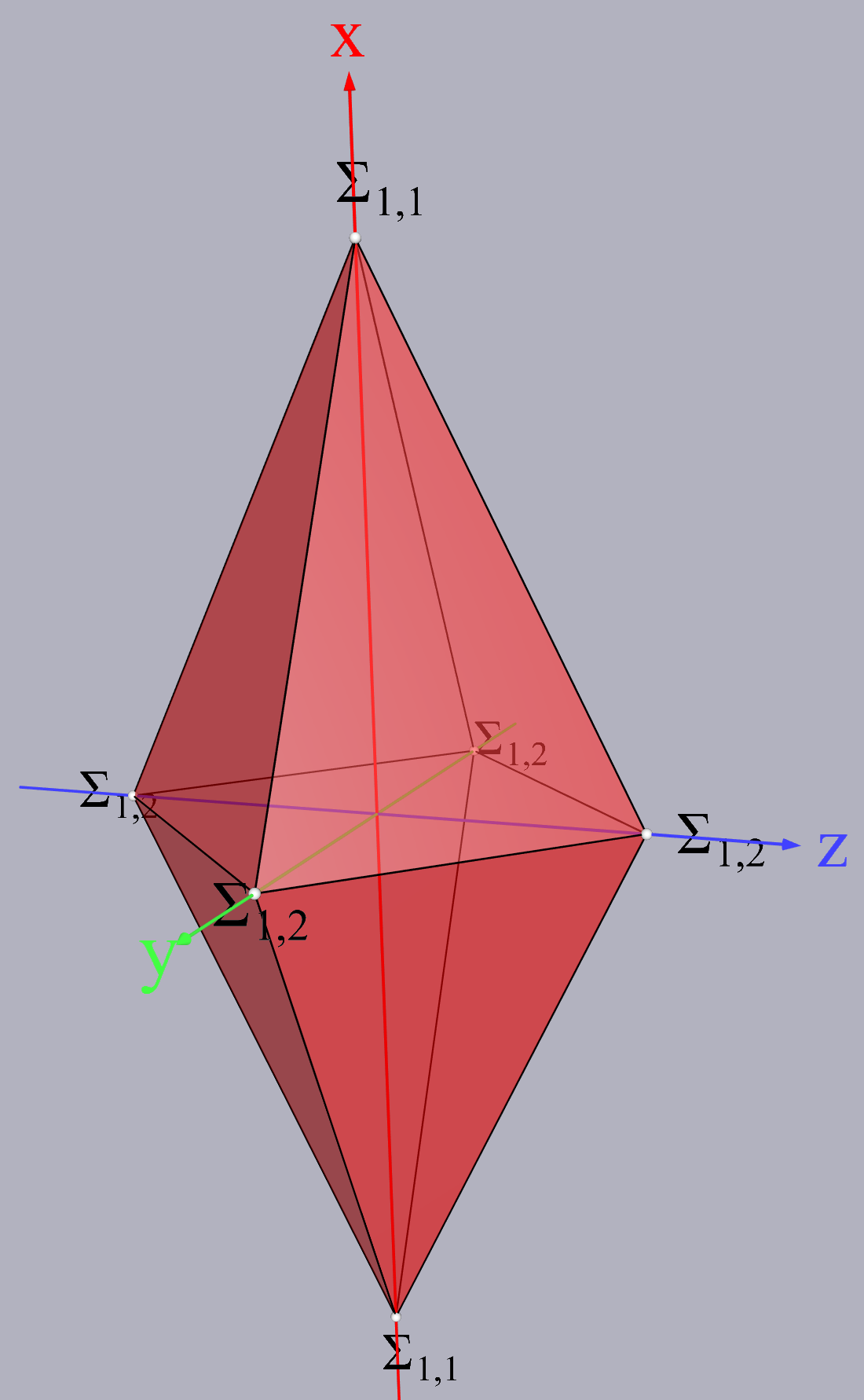}
        \caption{\tt{L9a46}}
        \label{fig:L9a46}
    \end{subfigure}
    \begin{subfigure}{.342\textwidth}
    	\centering
        \includegraphics[scale=.2]{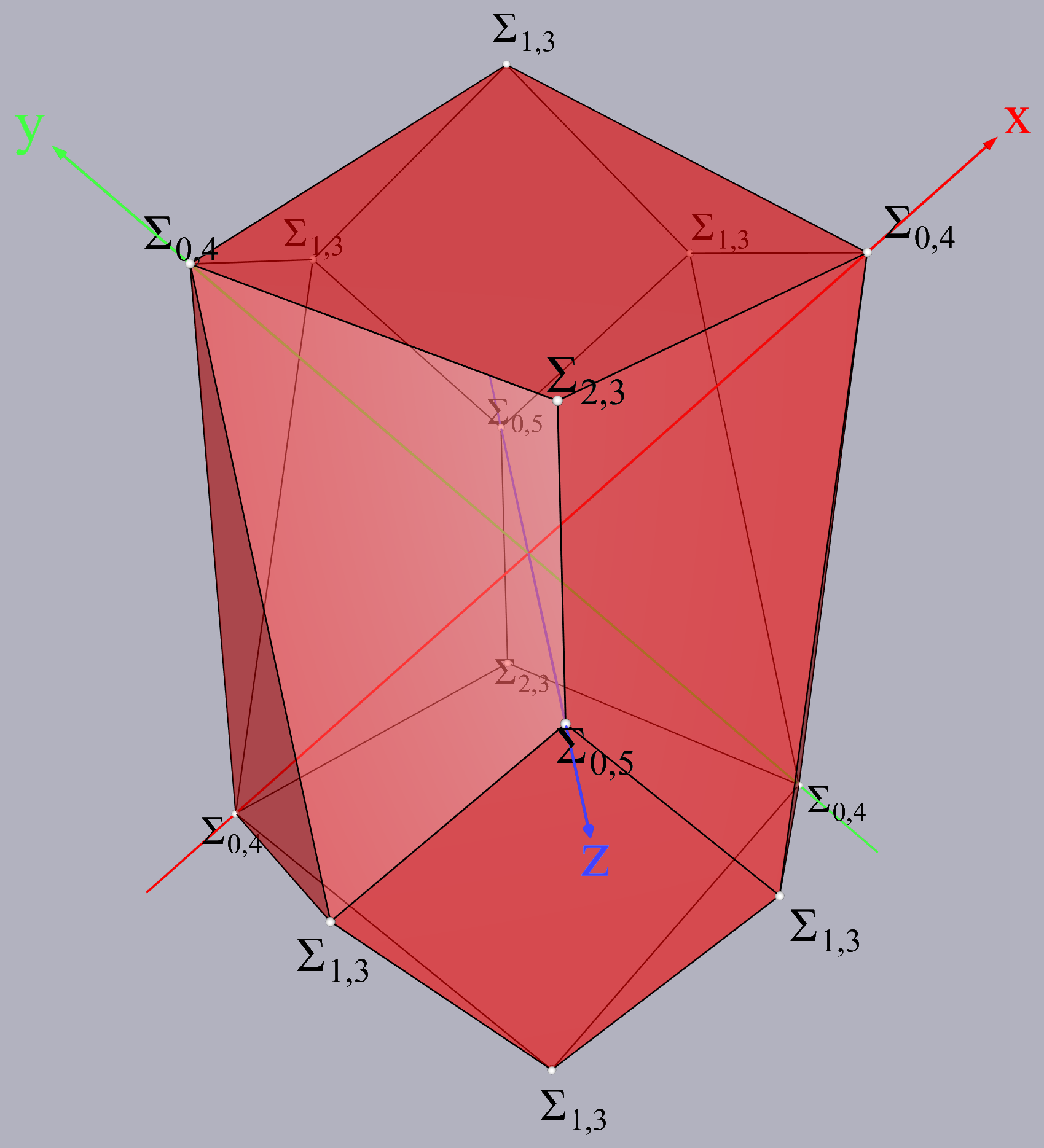}
    	\caption{\tt{L10a159}}
        \label{fig:L10a159}
    \end{subfigure}
    \begin{subfigure}{.385\textwidth}
    	\centering
        \includegraphics[scale=.2111]{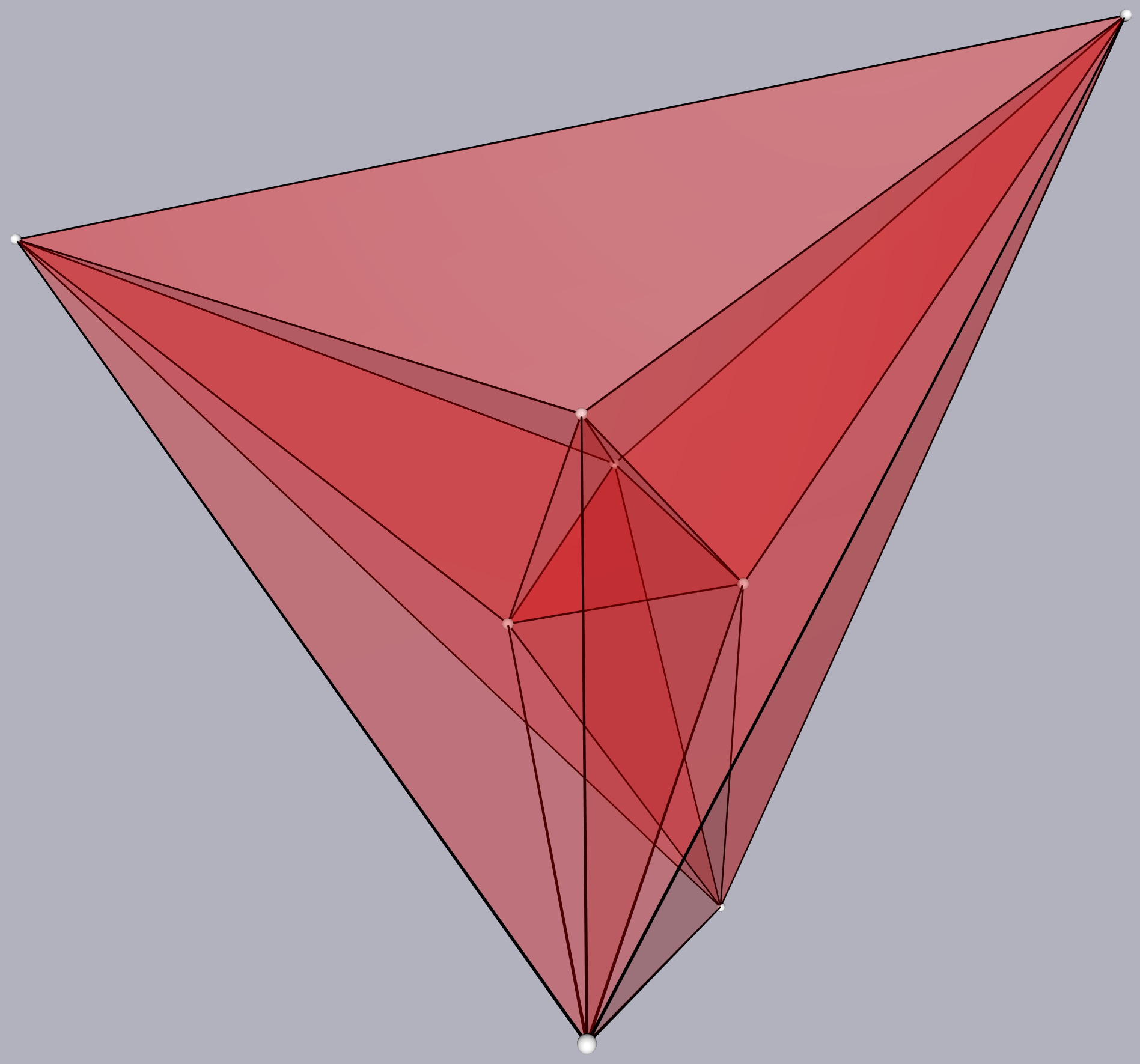}
    	\caption{\tt{L9a55} (Schlegel projection)}
        \label{fig:L9a55}
    \end{subfigure}
    \caption{Three Thurston norm unit balls computed by \tt{Tnorm}.}
    \label{fig:norm_balls}
\end{figure}

The above phenomenon appears to be rare---in particular, in the census of ideal triangulations of the first 268 hyperbolic knot complements, only eight have no spun-normal surface realizing a generator of $H_2(M;\del M)$. If a hyperbolic manifold $M$ is not the complement of a knot in a homology sphere, then the primitive integer homology class over a Thurston norm ball vertex will always be realized by some embedded normal surface. Norm balls for a few hyperbolic link complements are shown in \Cref{fig:norm_balls}. These are produced by \tt{Tnorm} as 3D manipulables, viewable in any modern web browser. The first two, in \Cref{fig:L9a46,fig:L10a159}, are three component links (and so their norm balls are 3-dimensional). For these two examples the vertices of the norm balls are labelled by the topological type of the normal surface realizing the primitive integer class lying above the vertex. The norm ball in \Cref{fig:L9a55} is the 4-component census link \tt{L9a55}. The norm ball for this link is 4-dimensional, and the figure shows a Schlegel projection of the 4-dimensional polytope to $\RR^3$. All vertices of the norm ball are realized by normal surfaces that are thrice-punctured spheres, but we have suppressed the labels in this case.

For a non-hyperbolic manifold $M$ the Thurston norm may only be a semi-norm and the normalisation result of incompressible surfaces due to Walsh~\cite{Wal11} does not apply. However, if every surface in $M$ that is non-trivial in homology has negative Euler characteristic, then the Thurston norm is a genuine norm. If, in addition, all vertices of the norm ball are realizable as immersed normal surfaces, then \tt{Tnorm} will correctly compute the norm ball. In the below example we compute the Thurston norm ball for the non-hyperbolic manifold obtained by Dehn filling the third cusp of \tt{L12n1738} along its longitude. Since the filling is along a boundary slope, the resulting manifold has internal homology, so \tt{Tnorm} uses the simplicial homology map (see \Cref{sec:Homology map}).

\begin{python}
sage: W=tnorm.load('L12n1738(0,0)(0,0)(0,1)')
sage: B=W.norm_ball; B.vertices()
[Vertex 0: represented by S_0,3 at (1, 0, 0), mapped from embedded qtons with index 15,
 Vertex 1: represented by S_0,3 at (0, 1, 0), mapped from embedded qtons with index 222,
 Vertex 2: represented by (1/2)*S_2,0 at (0, 0, 1/2), mapped from embedded qtons with index 193,
 Vertex 3: represented by (1/2)*S_2,0 at (0, 0, -1/2), mapped from embedded qtons with index 155,
 Vertex 4: represented by S_0,3 at (0, -1, 0), mapped from embedded qtons with index 148,
 Vertex 5: represented by S_0,3 at (-1, 0, 0), mapped from embedded qtons with index 9]	
\end{python}

We finish this section with a demonstration of how \tt{Tnorm} can be used to compute topological properties of spun-normal surfaces enumerated by \tt{Regina}. In the example below, we find that the spun-normal surface $S$ is a non-orientable connected surface with two boundary components and Euler characteristic -5. When we take the Haken sum of $S$ with the thrice-punctured sphere $R$, we get a connected, non-orientable surface $F$ with three boundary components. The \define{spinning slopes} of these surfaces are the boundary slopes described in \cite{tillmann08-normal}, for which the orientation on a boundary component is determined by the direction that it spins into the cusp, rather than on an orientation of the surface. For an oriented surface such as $R$, we can also compute the \define{oriented boundary slopes}, for the which the orientation on the boundary is induced by the orientation on the surface.

\begin{python}
sage: import regina; import snappy
sage: M=snappy.Manifold('L13n124'); T=regina.SnapPeaTriangulation(M._to_string())
sage: NS=regina.NormalSurfaces.enumerate(T,regina.NS_QUAD,regina.NS_VERTEX)
sage: S=NS.surface(500); R=NS.surface(127)
sage: tnorm.extras.is_connected(S), tnorm.extras.is_connected(R)
(True, True)
sage: tnorm.extras.is_orientable(S), tnorm.extras.is_orientable(R)
(False, True)
sage: tnorm.extras.euler_char(S), tnorm.extras.euler_char(R)
(-5, -1)
sage: tnorm.extras.num_boundary_comps(S), tnorm.extras.num_boundary_comps(R)
(2, 3)
sage: tnorm.extras.spinning_slopes(S), tnorm.extras.spinning_slopes(R)
([(-2, 1), (-8, 1)], [(0, 1), (-2, 0)])
sage: F=tnorm.extras.haken_sum(R,S)
sage: tnorm.extras.is_connected(F), tnorm.extras.is_orientable(F)
(True, False)
sage: tnorm.extras.euler_char(F), tnorm.extras.num_boundary_comps(F)
(-6, 3)
sage: tnorm.extras.spinning_slopes(F)
[(-2, 2), (-10, 1)]
sage: tnorm.extras.oriented_boundary_slopes(R)[0]
{'outward': [(0, 1), (-1, 0)], 'inward': [(0, 0), (1, 0)]}
\end{python}



%% file: apps02.tex

\section{An application}
\label{sec:applications}

Our proof of \Cref{thm:entropy} relies on the following result of Fried \cite{Fri82}, for which we need some notation. Let $\F_\Q$ be the set of rational points on a fibered face of a Thurston norm unit ball. For $x\in \F_\Q$, let $F_x$ be the fiber over $x$, and let $\lambda_x$ be the dilatation of the associated monodromy. Let $L_\F:\F_\Q \to \RR$ be defined by $L_\F(x)=|\chi(F_x)|\log(\lambda_x)$.
\begin{theorem}[Fried \cite{Fri82}]
The function $L_\F$ extends to a continuous convex function on $\F$ going to infinity toward the boundary of $\F$.
\end{theorem}

We prove \Cref{thm:entropy} by finding a particular fibered link for which the Thurston unit norm ball has a fibered face containing many fibers. The following proposition describes precisely the properties such a link (and its unit norm ball) needs to have.

\begin{proposition}\label{prop:good_face}
Let $M=\SS^3\setminus (L=L_1\sqcup L_2\sqcup L_3)$ be the complement of a fibered 3-component link, and let $\F$ be a fibered face of the Thurston norm unit ball for $M$ with vertices $x_1,x_2$ and $x_3$ having the following properties:
\begin{itemize}
\item[(1)] $2x_1,2x_2$ and $x_3$ have norm minimizing representatives $S_1=\Sigma_{1,2}$, $S_2=\Sigma_{1,2}$, and $S_3=\Sigma_{1,1}$, respectively.
\item[(2)] $S_1$ and $S_2$ each have a single boundary component on $L_1$, and a single boundary component on $L_2$, and no boundary component on $L_3$.
\item[(3)] The slopes for $S_1$ and $S_2$ are distinct on $L_1$ and on $L_2$.
\item[(4)] $S_3$ has no boundary components on $L_1$ or $L_2$, and has one boundary component on $L_3$.
\end{itemize}
Then for any $0<\varepsilon<1$, there is a compact subset $K_\varepsilon\subset\F$ such that $\RR^+K_\varepsilon$ contains a fiber $S_{g,n}$ for every $g\ge 2$ and every $n$ satisfying $\varepsilon g + 2\le n\le \frac{g}{\varepsilon}+2$.
\end{proposition}

\begin{proof}

Suppose we have such a link $L$ and fibered face $\F$. We first prove the following:

\emph{Claim 1:} For every integer $g\ge 2$, there are positive integers $p,q$ such that $p+q=g$, $\gcd(p,q)=1$, and $|p-q|\le 4$.

\emph{Proof of Claim 1:} 
Given $g$, choose $p$ and $q$ as follows. If $g=2$, set $p=q=1$. If $g$ is odd, let $p=\frac{g-1}{2}+1$ and let $q=\frac{g-1}{2}$, so that $q=p+1$ and hence $\gcd(p,q)=1$. If $g\equiv 0\mod 4$, choose $p=\frac{g}{2}-1$ and $q=\frac{g}{2}+1$, so that $q=p+2$ and $p$ and $q$ are both odd (since $\frac{g}{2}$ is even). Then $\gcd(p,q)\mid 2$, but since $p$ and $q$ are odd we must have $\gcd(p,q)=1$. If $g\equiv 2\mod 4$, choose $p=\frac{g}{2}-2$ and $q=\frac{g}{2}+2$, so that $q=p+4$. Then $\gcd(p,q)\mid 4$, but $\frac{g}{2}$ is odd, so $p$ and $q$ must also be odd and hence $\gcd(p,q)=1$ is the only possibility, thus concluding the proof of Claim 1.

Fix $0<\varepsilon <1$, and define $K_\varepsilon\subset \F$ as follows:
$$
K_\varepsilon:= \left\{x\in \F \,\,\Big |\,\, x= \frac{px_1+qx_2+rx_3}{|px_1+qx_2+rx_3|},\,\, where \,\,\, \frac{\varepsilon}{6}\le \frac{p}{r},\frac{q}{r},\frac{p}{q}\le \frac{6}{\varepsilon}\right\}
$$
Since $H_2(M,\del M; \RR)\cong \RR^3$ and $x_1,x_2,x_3$ are linearly independent, we may identify $\frac{x_i}{\|x_i\|}$, $i=1,2,3$,  with the standard basis of $\RR^3$, so that $\F$ is the standard 2-simplex. The conditions in the definition of $K_\varepsilon$ say that the projection of the ray determined by $x=px_1+qx_2+rx_3$ to each coordinate plane has slope bounded away from $0$ and $\infty$, so that the intersection of this ray with $\F$ stays a bounded distance (in the $\RR^3$ metric) from the boundary of $\F$. It follows that $K_\varepsilon$ is a compact subset of $\F$.

We first show that $\RR^+\F$ contains a fiber $F=\Sigma_{g,n}$ for every $g\ge 2$ and $n\ge 3$. To this end, fix such a pair $(g,n)$ and choose $p,q$ as in Claim 1. Let $S'=pS_1+qS_2$. Since $S_1$ and $S_2$ have distinct slopes on $L_1$ and $L_2$ and $\gcd(p,q)=1$, $S'$ has exactly two boundary components (one on $L_1$ and one on $L_2$). Since the Thurston norm is linear over the fibered face we have
$$
2+2\cdot \genus(S')-2=|\chi(S')|=\|S'\|=p\|S_1\|+q\|S_2\|=2p+2q \implies \genus(S')=p+q=g
$$
Letting $F=S'+rS_3$ with $r=n-2$, we compute similarly
$$
2\cdot\genus(F)-2+|\del F|=|\chi(F)|=\|F\|=\|S'\|+r\|S_3\|=2p+2q+r
$$
Since the two boundary components of $S'$ are on $L_1$ and $L_2$, and the $r$ boundary components of $rS_3$ are on $L_3$ (i.e., disjoint from the boundary of $S'$), $F$ has a total of $r+2=n$ boundary components, which in turn implies that the genus of $F$ is also $g=p+q$. Since there is only one boundary component on, e.g, $L_1$, $F$ cannot be a multiple of a fiber, and must therefore be a fiber. Since $p,q,r$ are positive, $F$ is in the interior of the face $\F$.

Now let $F$ be as above, and assume that $n$ satisfies $\varepsilon g + 2\le n\le \frac{g}{\varepsilon}+2$. Since $\varepsilon<1$ and $|p-q|\le 4$, we immediately have $\frac{\varepsilon}{6}\le \frac{1}{5}\le \frac{p}{q}\le 5\le \frac{6}{\varepsilon}$. Using $r=n-2$ and $g=p+q$ we have
$$
\varepsilon g + 2\le n\le \frac{g}{\varepsilon}+2 \implies \varepsilon\le\frac{p+q}{r}\le \frac{1}{\varepsilon} \implies \frac{p}{r},\frac{q}{r}\le \frac{1}{\varepsilon}\le \frac{6}{\varepsilon}
$$
Since $|p-q|\le 4$ we have $p\le q+4\le 5q $ and $q\le p+4\le 5p$. This gives
$$
\varepsilon \le \frac{p+q}{r}\le \frac{6p}{r} \implies \frac{\varepsilon}{6}\le \frac{p}{r}
$$
and similarly for $\frac{q}{r}$. It follows that $F\in \RR^+K_\varepsilon$ and the proposition is proved.
\end{proof}

\begin{figure}
 	\centering
   	\includegraphics[width=.85\textwidth]{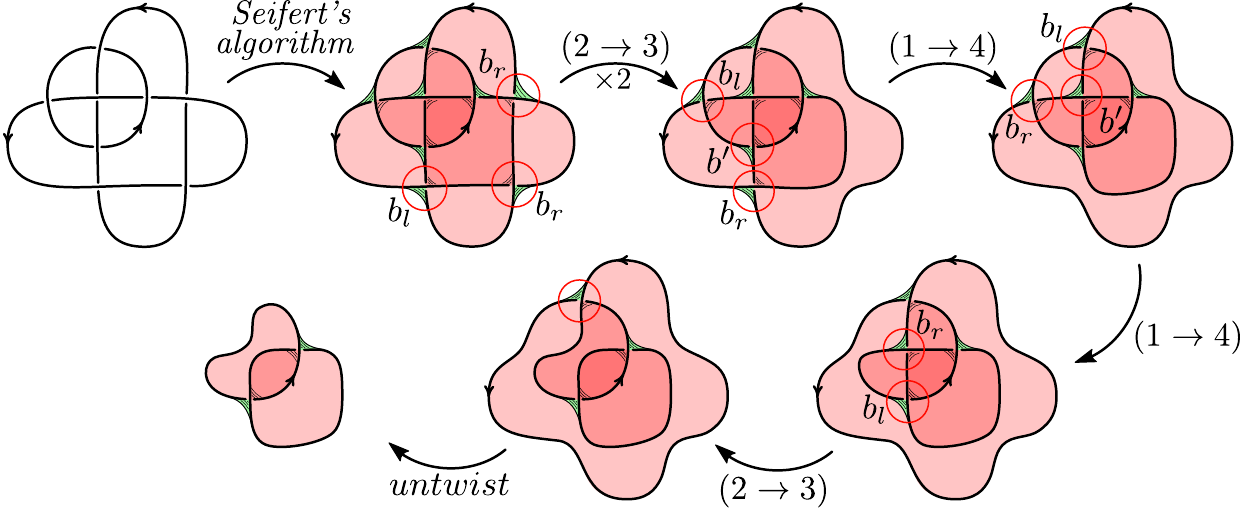}
   	\caption{We successively deplumb Hopf bands from the Seifert surface for the oriented link \tt{L9a46}, until we are left with a single Hopf band.}
   	\label{fig:fibered}
\end{figure}

\begin{lemma}\label{lem:L6a46}
Let $L$ be the census link \texttt{L9a46}, pictured in \Cref{fig:fibered} (upper left). Then $L$ is fibered, and at least one face of its Thurston norm unit ball for $\SS^3\setminus L$ satisfies the conditions of \Cref{prop:good_face}.
\end{lemma}

\begin{proof}
First we show that $L$ is fibered. We will use a criterion of Gabai, given in \cite[Theorem 5.1]{Gab86-detecting}, which states that an oriented alternating link $L$ is fibered if and only if the surface $R$ obtained by Seifert's algorithm is a plumbing of Hopf bands. This means that $R$ can be obtained by taking successive Murasugi 4-sums of Hopf bands. Or, equivalently, that we can successively remove Hopf bands from $R$ by desumming, until we obtain a disjoint union of Hopf bands. For a discussion of Murasugi sums see \cite{Gab86-detecting}. Note that for a link, Seifert's algorithm depends on the orientations of the components. If we choose the wrong orientation on $L$ we may get a Seifert surface which is not a plumbing of Hopf bands, but it may be that $L$ is fibered with respect to some other orientation. With the orientation shown in the upper left of \Cref{fig:fibered}, we obtain the Seifert surface $R$ shown at upper middle-left of the same figure. In general, a Seifert surface is some number of disks joined by twisted bands. If two such bands join the same two Seifert disks and twist in the same direction, then we say they are \define{parallel}. We first show the following:

\emph{Claim 1:} If $b_l$ and $b_r$ are parallel bands of a Seifert surface $R$ that are separated by at most one other band $b'$, then we can deplumb a Hopf band $H$ from $R$ in such a way that either $B_l$ or $B_r$ is removed from $R$.

\emph{Proof of Claim 1:} Refer to \Cref{fig:deplumb}. In this figure we remove $b_r$. In general, whether we can remove $b_r$ or $b_l$ depends on the direction of the twist of $b_r$ and $b_l$, and on the position of $b'$. Since we will only need the case shown in \Cref{fig:deplumb}, we will not discuss the other cases (though they are similar). For our case of interest, we (1) slide $b'$ out of the way (down and to the left), then (2) slide $b_r$ into position to be desummed, then (3) deplumb $b_r$, and finally we (4) slide $b'$ back to its original position.

With the claim proved, we can now show that $L$ is fibered by following the steps of \Cref{fig:fibered}. In the first step (after Seifert's algorithm), we identify three parallel bands, and remove two of them (i.e., we do steps $(2\to 3)$ of \Cref{fig:deplumb} twice). In the next two steps we identify three bands $b_r$, $b_l$, and $b'$, and remove $b_r$ using steps $(1\to 4)$ of \Cref{fig:deplumb}. Finally, we remove another parallel band, then untwist the last band. We are left with a Hopf band, as desired.

Finally, the claim that the Thurston norm ball for $\SS^3\setminus L$ has a face satisfying the conditions of \Cref{prop:good_face} is proved by explicitly computing the norm ball. Using the implementation described in \Cref{sec:implementation and examples}, we find that the norm ball is an octahedron, with vertices $x_1,x_2,x_3$ as shown in \Cref{fig:L9a46}, in which $x_1$ and $x_2$ are the respective vertices on the positive $y$- and $z$-axes, and $x_3$ is the vertex on the positive $x$-axis. The points $2x_1$ and $2x_2$ lying over $x_1$ and $x_2$ are represented by surfaces $S_1$ and $S_2$, both of which are twice-puntured tori. The vertex $x_3$ is represented by a once-punctured torus $S_3$. The boundary slopes of $S_1$ are $[(0,1),(-1,0),(0,0)]$, the boundary slopes of $S_2$ are $[(-1,0),(0,1),(0,0)]$, and the boundary slopes of $S_3$ are $[(0,0),(0,0),(0,1)]$. Since every face has vertices $\pm x_1, \pm x_2, \pm x_3$ for some choice of signs, all faces of the octahedron satisfy conditions (1) though (4) of \Cref{prop:good_face}. Since the link is fibered, at least one of these faces is a fibered face, so the hypotheses of \Cref{prop:good_face} are satisfied. 
\end{proof}

\begin{figure}
 	\centering
   	\includegraphics[width=.95\textwidth]{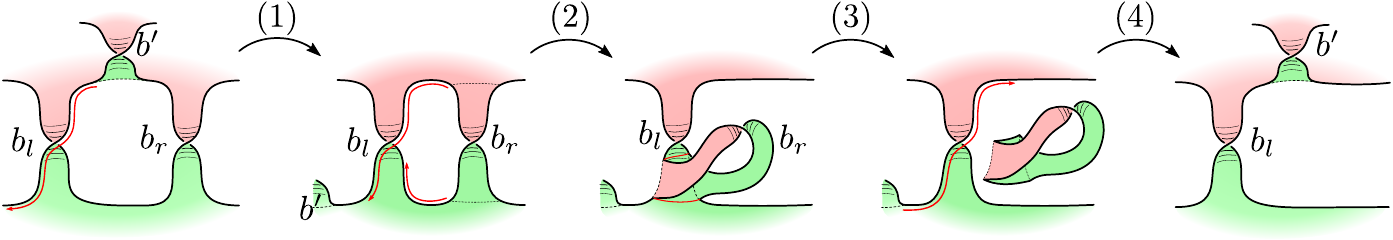}
   	\caption{Deplumbing a Hopf band.}
   	\label{fig:deplumb}
\end{figure}

We are now ready to prove \Cref{thm:entropy}. 

\begin{proof}[Proof of \Cref{thm:entropy}.]
By \Cref{lem:L6a46} there is a fibered link satisfying the conditions of \Cref{prop:good_face}. For $\varepsilon>0$ fixed, let $c_\varepsilon$ be the maximum of Fried's function $L_\F$ on $K_\varepsilon$. By \Cref{prop:good_face}, for every $g\ge 2$ and $n$ satisfying $\varepsilon g + 2\le n\le \frac{g}{\varepsilon}+2$, $\RR^+K_\varepsilon$ contains an element $x$ represented by a fiber $F\cong \Sigma_{g,n}$. Thus 
$$l_{g,n}\le \log(\lambda_x)=\frac{L_\F(x)}{|\chi(F)|}\le \frac{c_\varepsilon}{|\chi(F)|}$$
This is the desired upper bound.
\end{proof}


%% file: appendix.tex


\appendix
\section{Norm balls of some 2-component links}
\label{sec:appendix}

In this appendix we give the Thurston norm balls for the first one-hundred hyperbolic 2-component links in the census of Hoste and Thistlethwaite, along with the isometry group of the link complement. See \Cref{sec:implementation and examples} for notes on implementation and more detailed examples.

\input{latex_table.tex}



%% file: latex_table.tex
\footnotesize 
 \begin{tabular}{|c|c|c|c|} 
 \hline 
 Link & Norm Ball & Link & Norm Ball \\ 
 \hline 
\quad & \multirow{6}{*}{\Includegraphics[width=1.8in]{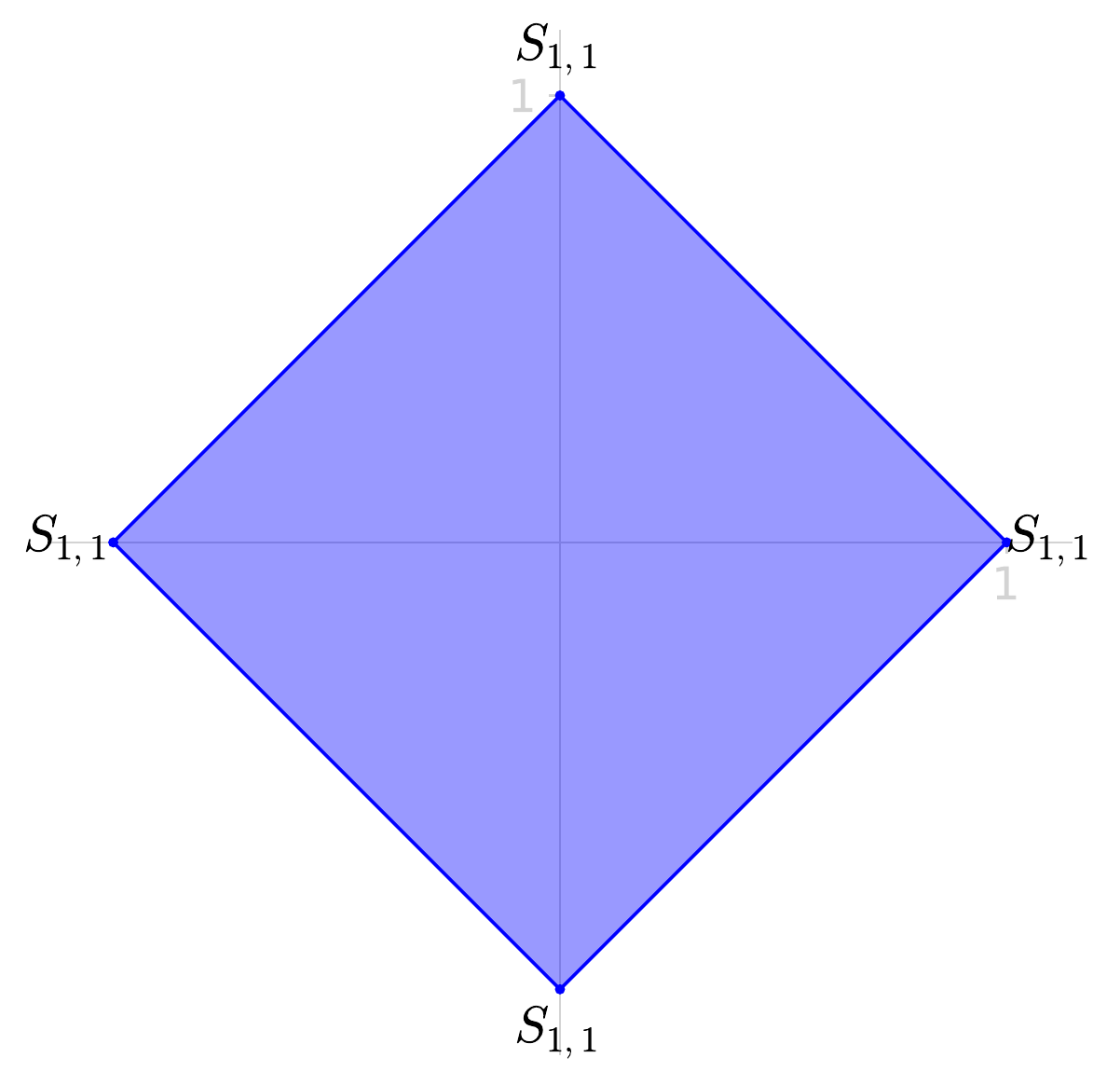}} & \quad & \multirow{6}{*}{\Includegraphics[width=1.8in]{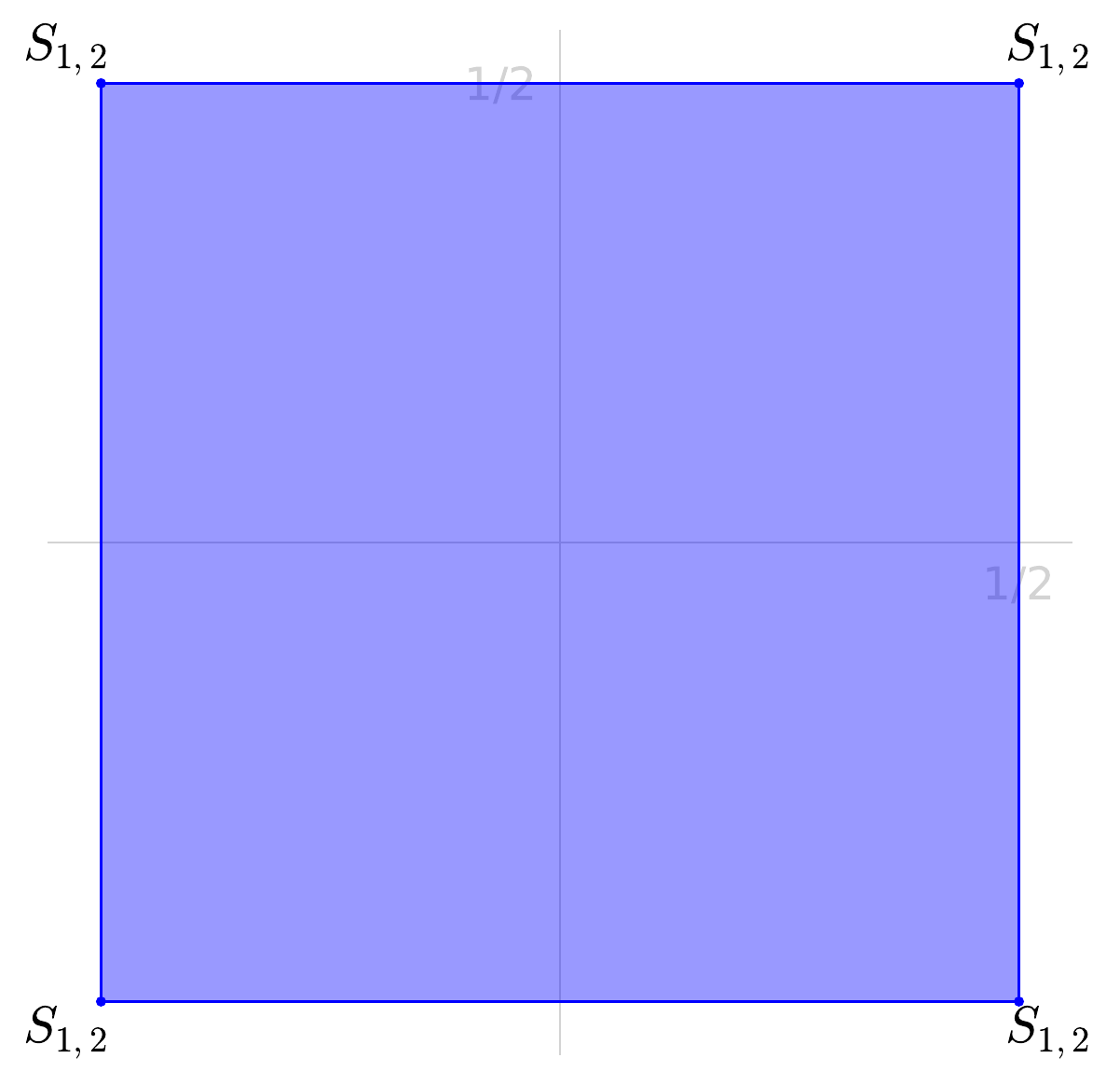}} \\ 
 $L=5^{{2}}_{{1}}$ & & $L=6^{{2}}_{{2}}$ & \\ 
 \quad & & \quad & \\ $\mathrm{Isom}(\mathbb{S}^3\setminus L) = D_4$ & & $\mathrm{Isom}(\mathbb{S}^3\setminus L) = D_4$ & \\ 
 \quad & & \quad & \\ 
 \includegraphics[width=1in]{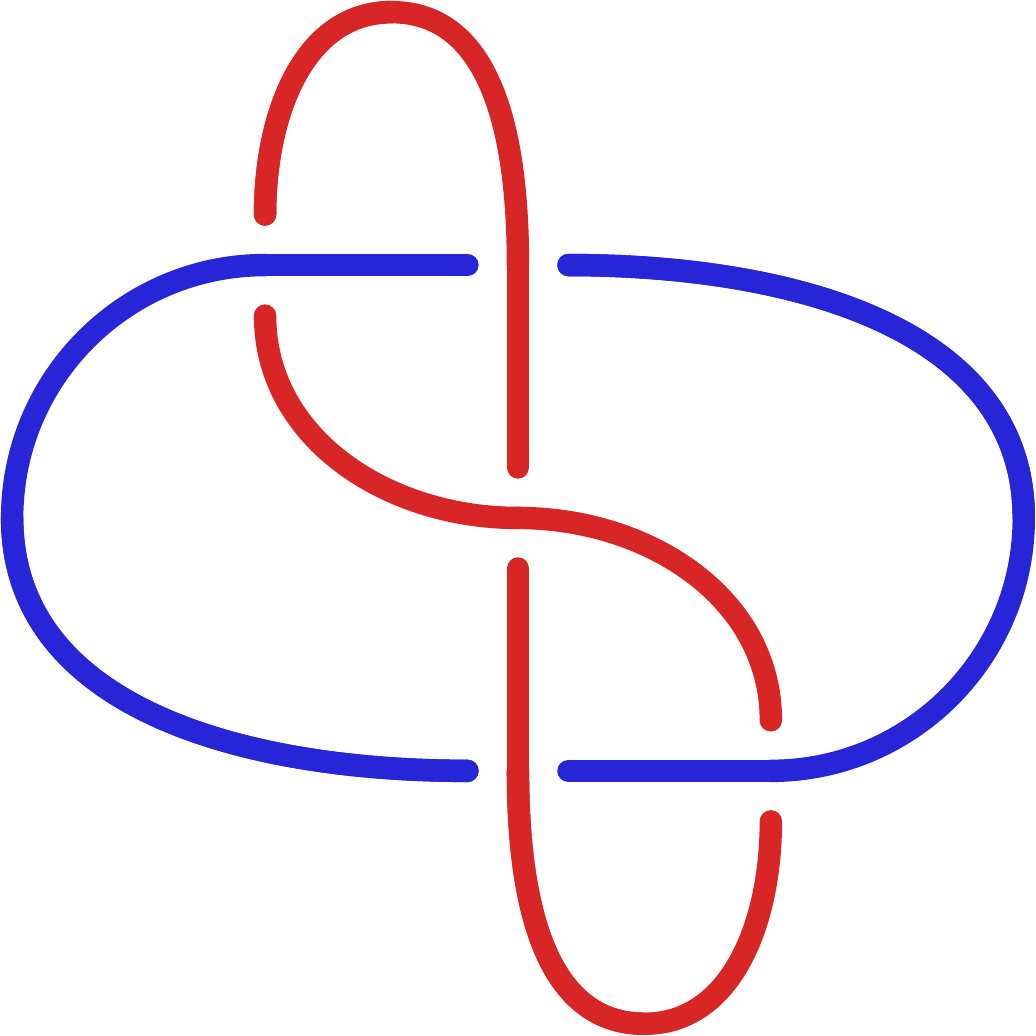}  & & \includegraphics[width=1in]{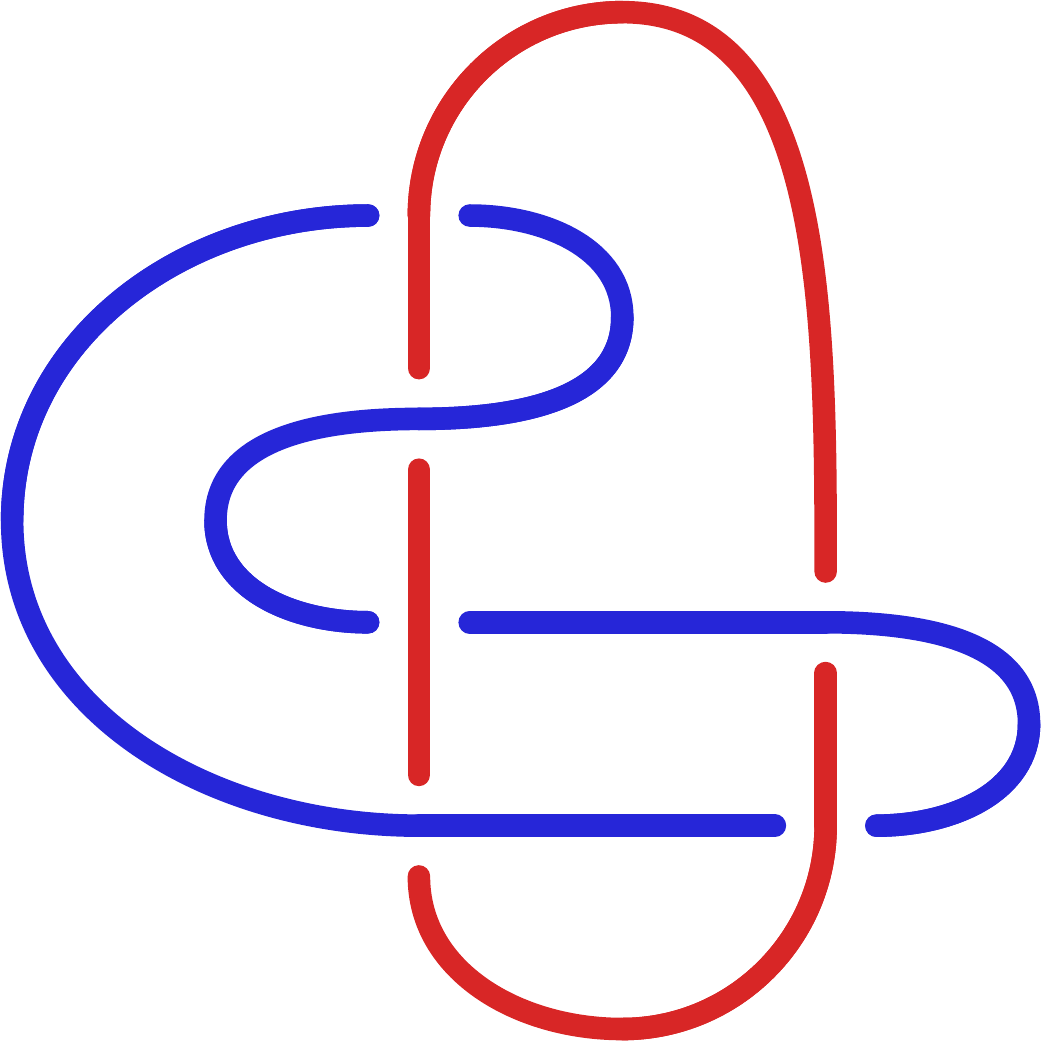} & \\ 
 \quad & & \quad & \\ 
 \hline  
\quad & \multirow{6}{*}{\Includegraphics[width=1.8in]{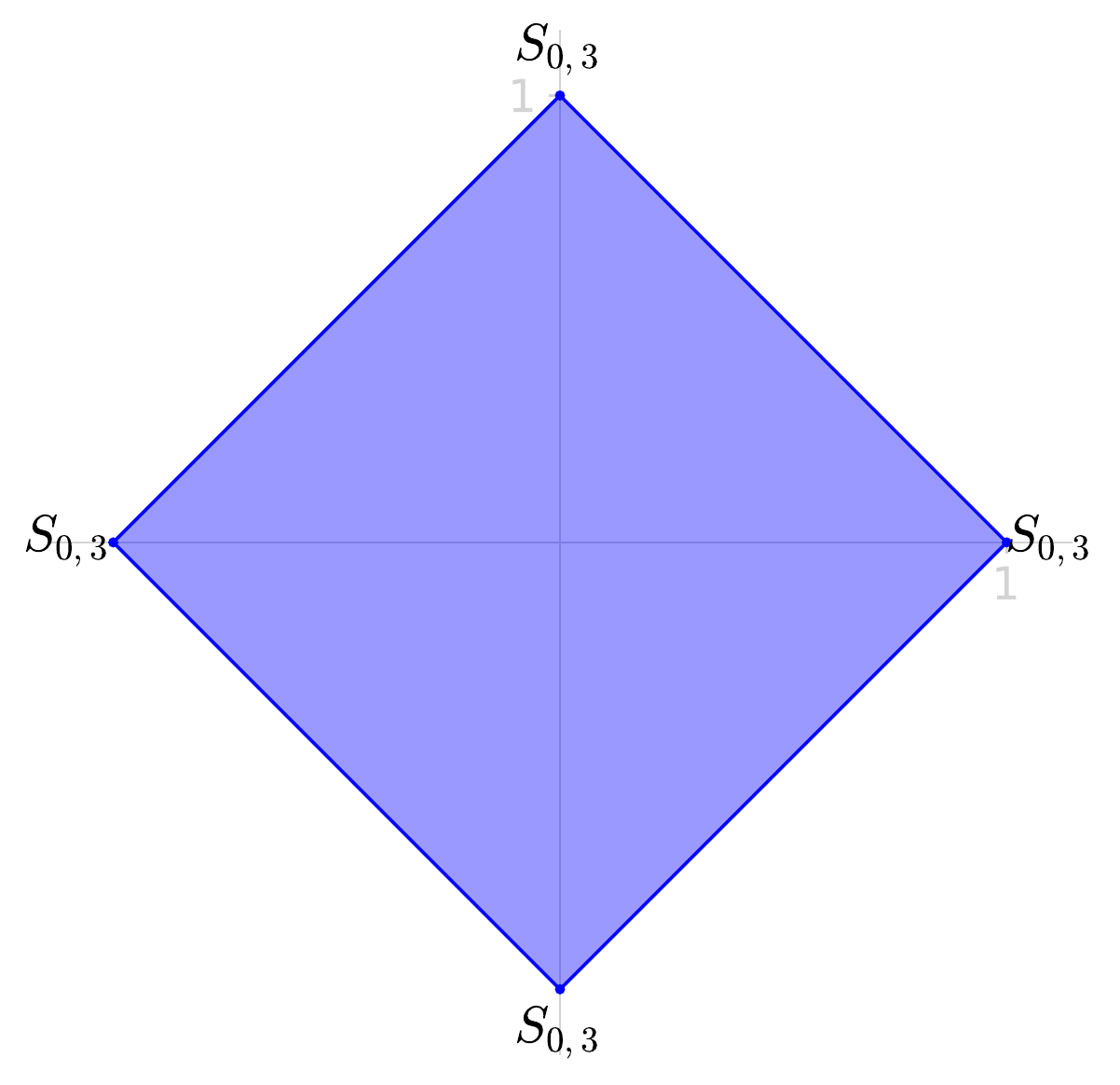}} & \quad & \multirow{6}{*}{\Includegraphics[width=1.8in]{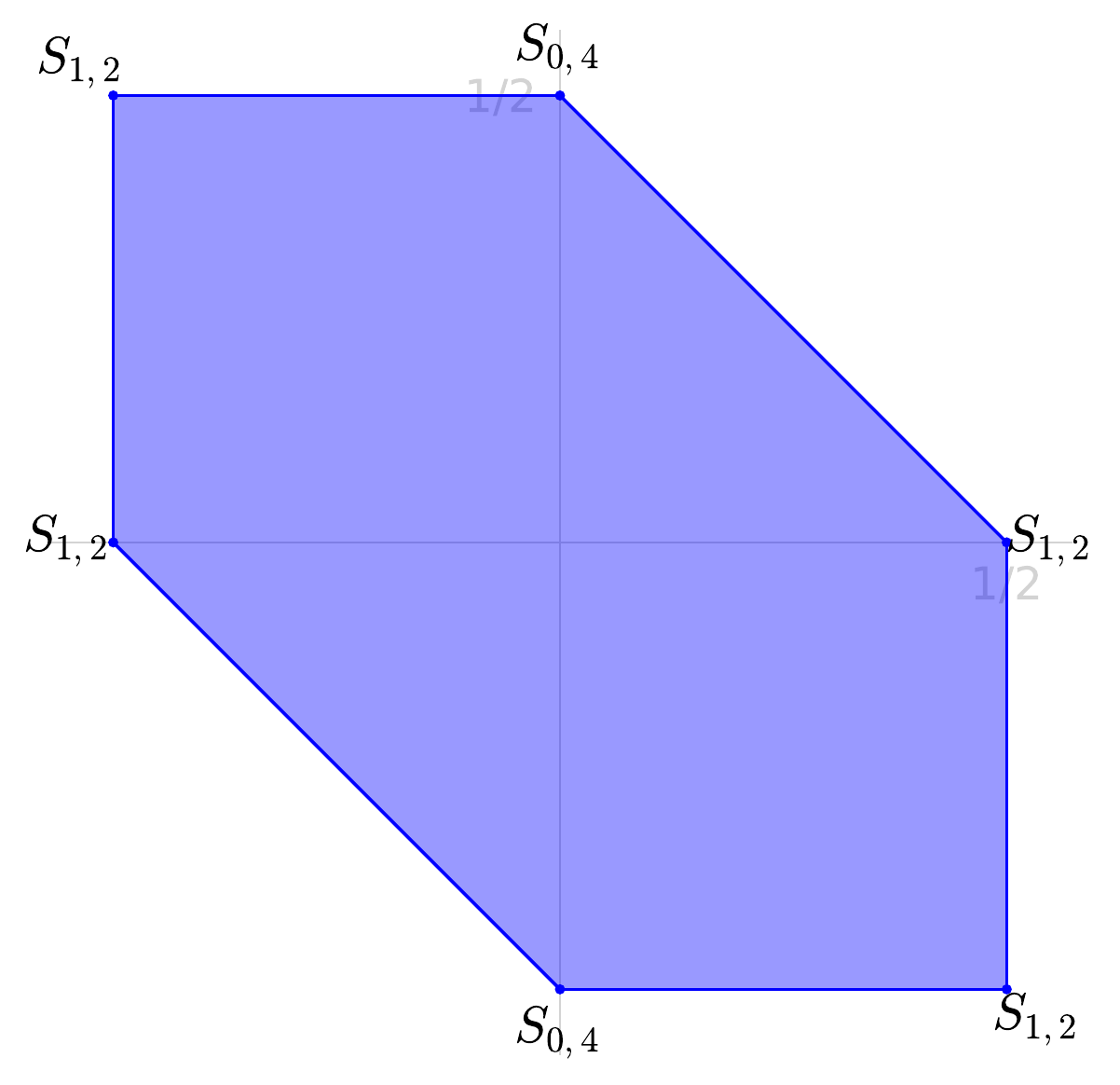}} \\ 
 $L=6^{{2}}_{{3}}$ & & $L=7^{{2}}_{{1}}$ & \\ 
 \quad & & \quad & \\ $\displaystyle\mathrm{Isom}(\mathbb{S}^3\setminus L) = \bigoplus_{i=1}^3 \mathbb{Z}$ & & $\mathrm{Isom}(\mathbb{S}^3\setminus L) = \mathbb{{Z}}_2\oplus\mathbb{{Z}}_2$ & \\ 
 \quad & & \quad & \\ 
 \includegraphics[width=1in]{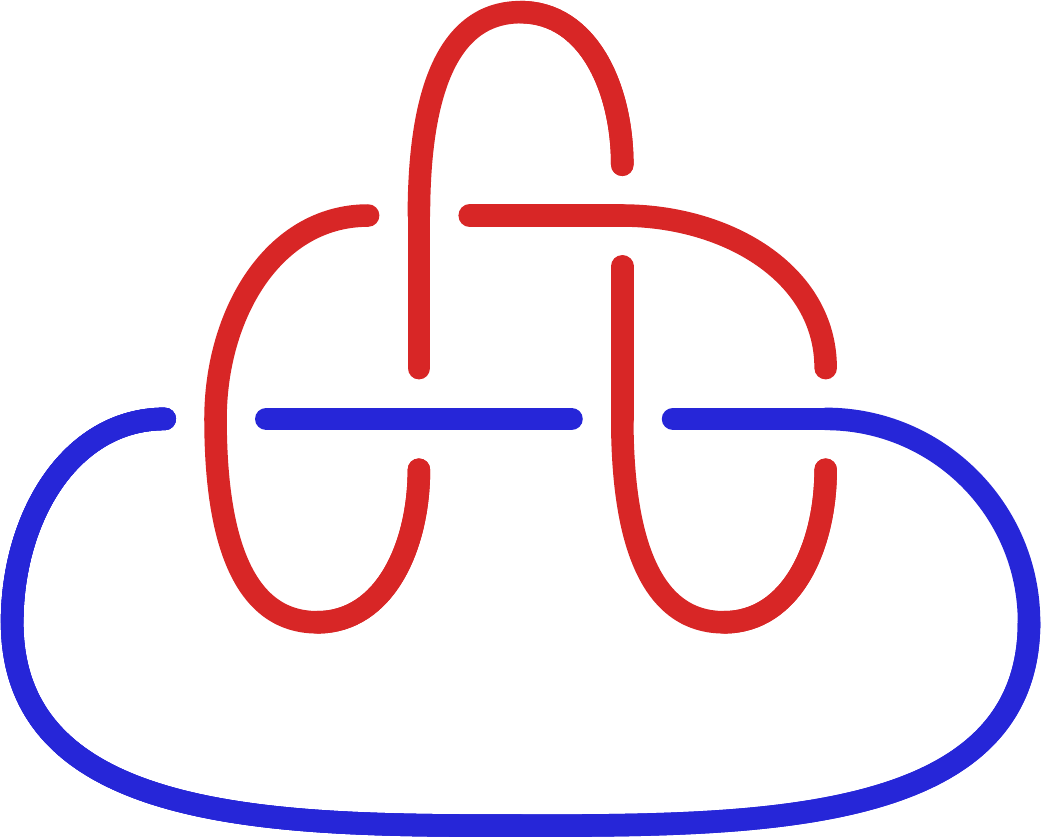}  & & \includegraphics[width=1in]{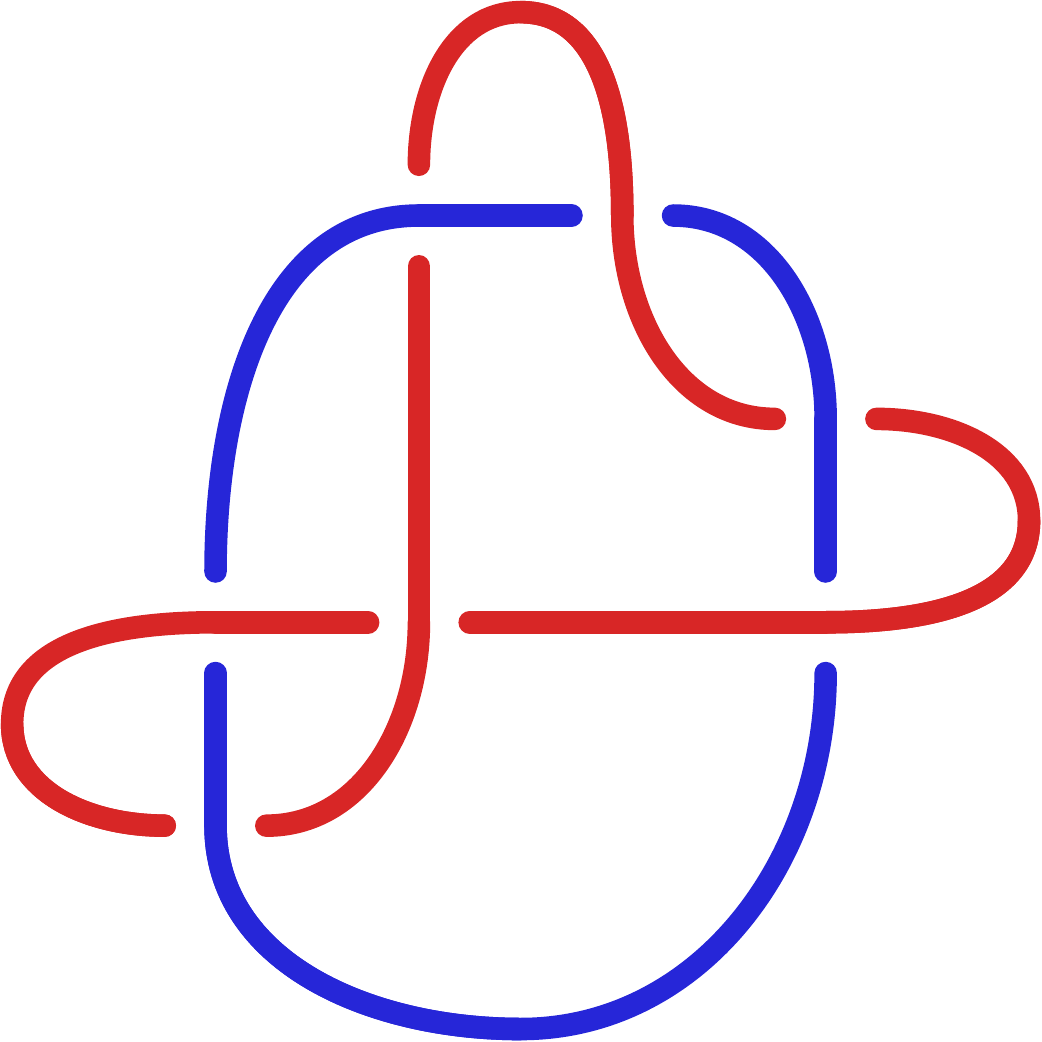} & \\ 
 \quad & & \quad & \\ 
 \hline  
\quad & \multirow{6}{*}{\Includegraphics[width=1.8in]{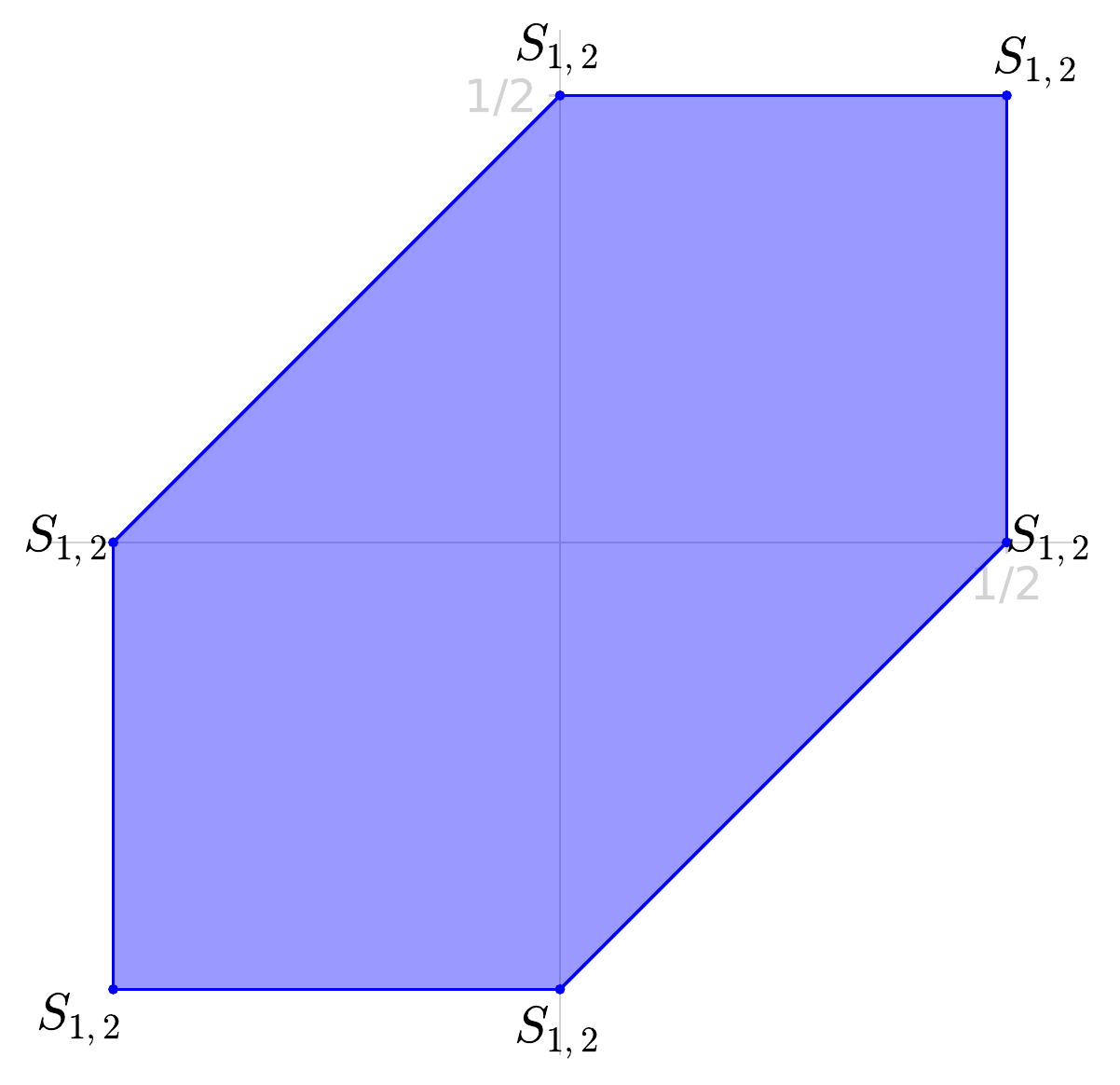}} & \quad & \multirow{6}{*}{\Includegraphics[width=1.8in]{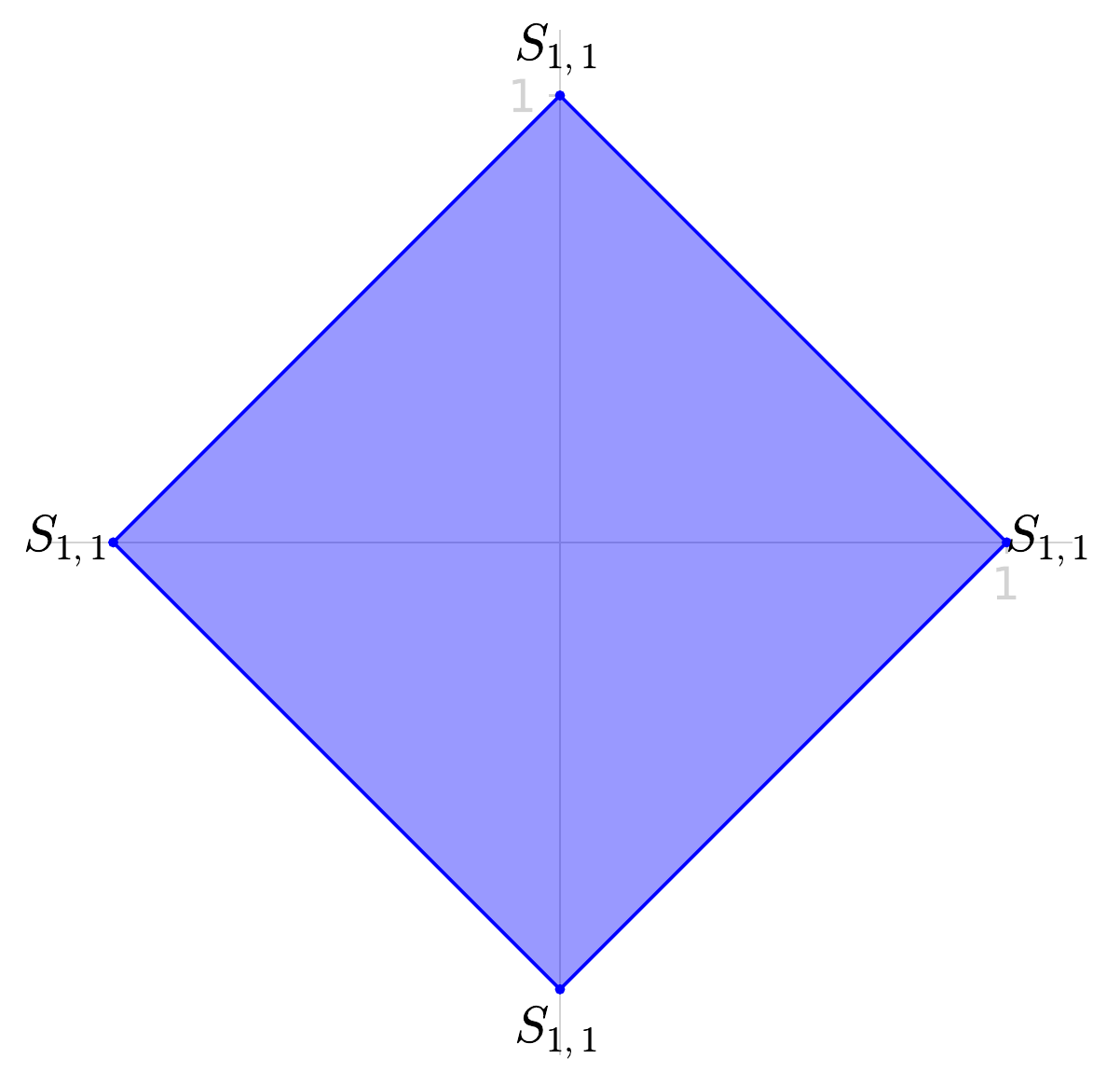}} \\ 
 $L=7^{{2}}_{{2}}$ & & $L=7^{{2}}_{{3}}$ & \\ 
 \quad & & \quad & \\ $\mathrm{Isom}(\mathbb{S}^3\setminus L) = \mathbb{{Z}}_2\oplus\mathbb{{Z}}_2$ & & $\mathrm{Isom}(\mathbb{S}^3\setminus L) = D_4$ & \\ 
 \quad & & \quad & \\ 
 \includegraphics[width=1in]{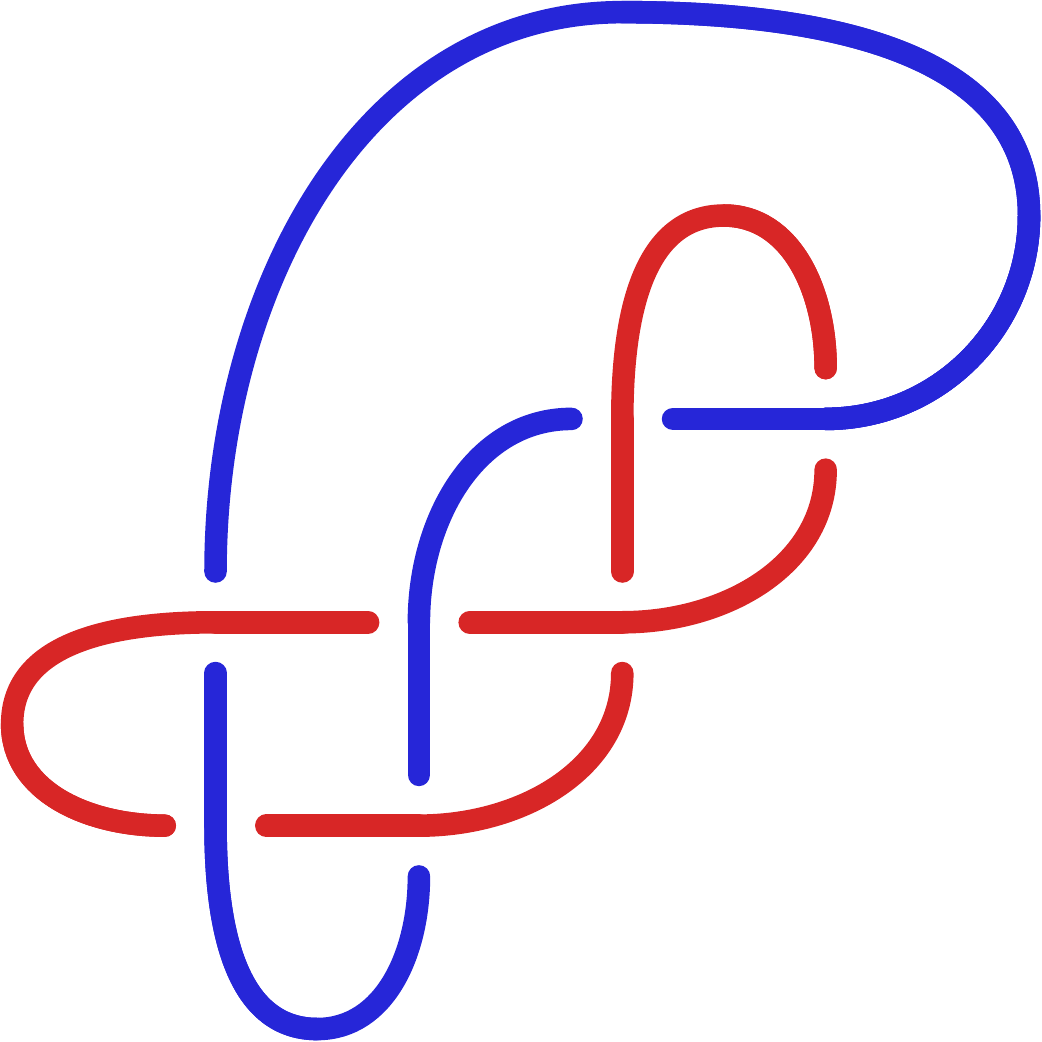}  & & \includegraphics[width=1in]{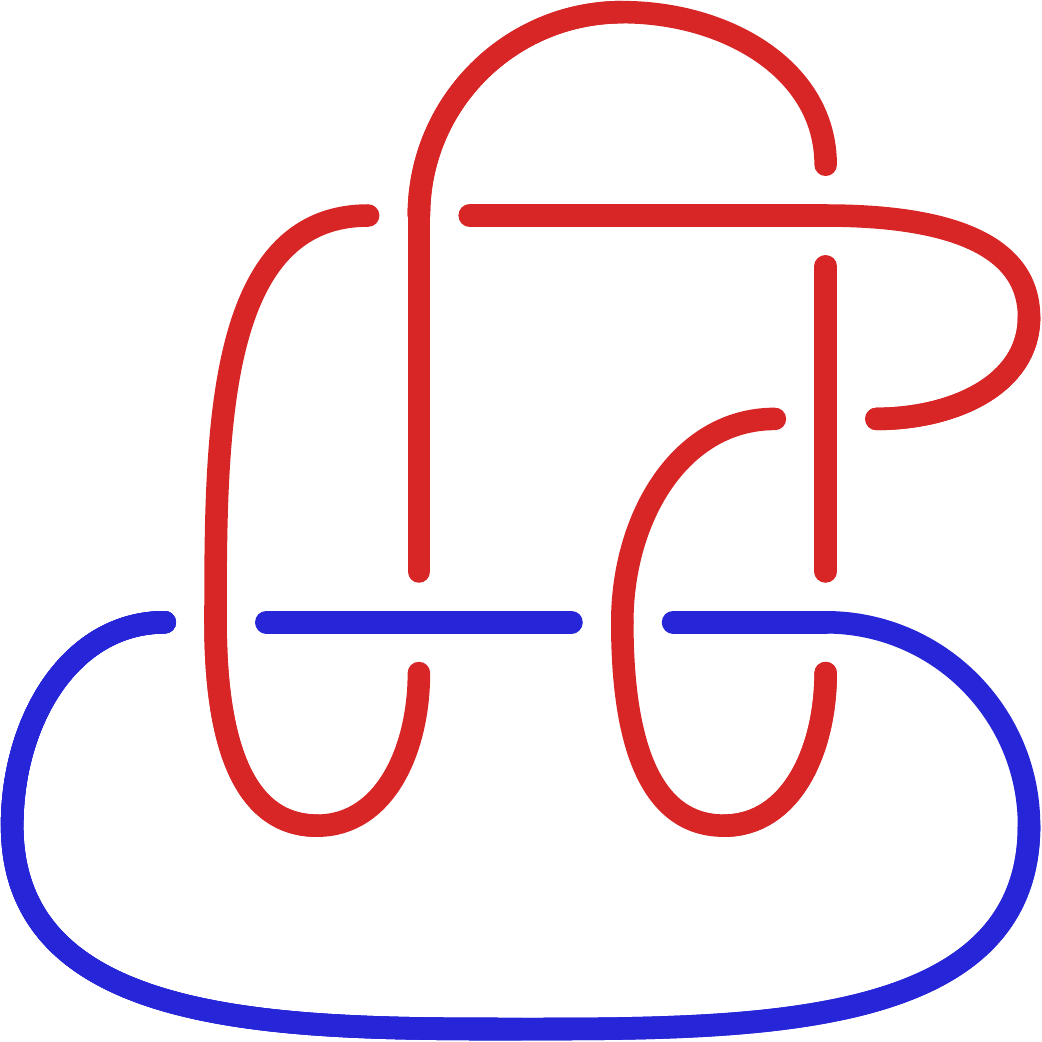} & \\ 
 \quad & & \quad & \\ 
 \hline  
\end{tabular} 
 \newpage \begin{tabular}{|c|c|c|c|} 
 \hline 
 Link & Norm Ball & Link & Norm Ball \\ 
 \hline 
\quad & \multirow{6}{*}{\Includegraphics[width=1.8in]{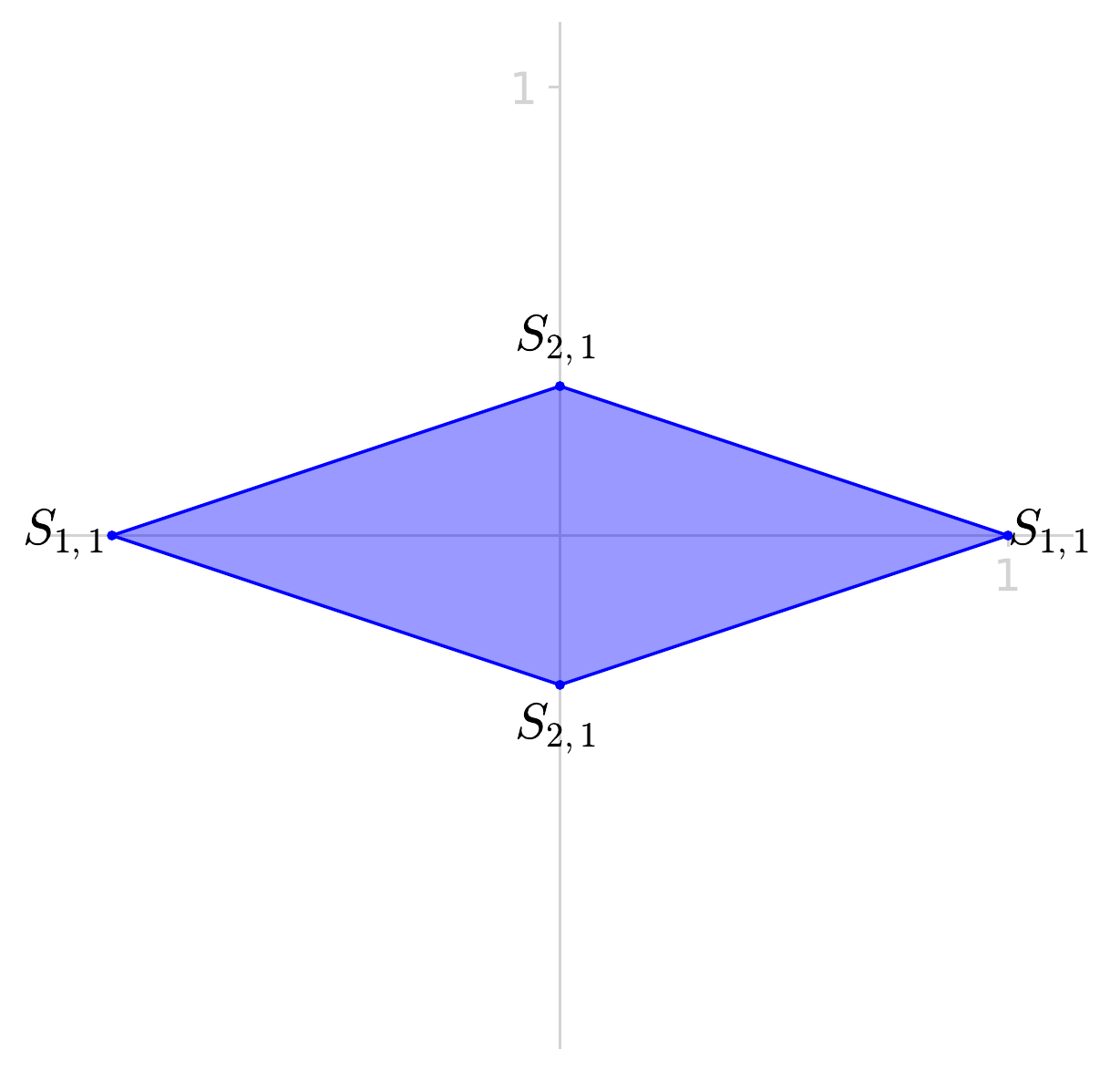}} & \quad & \multirow{6}{*}{\Includegraphics[width=1.8in]{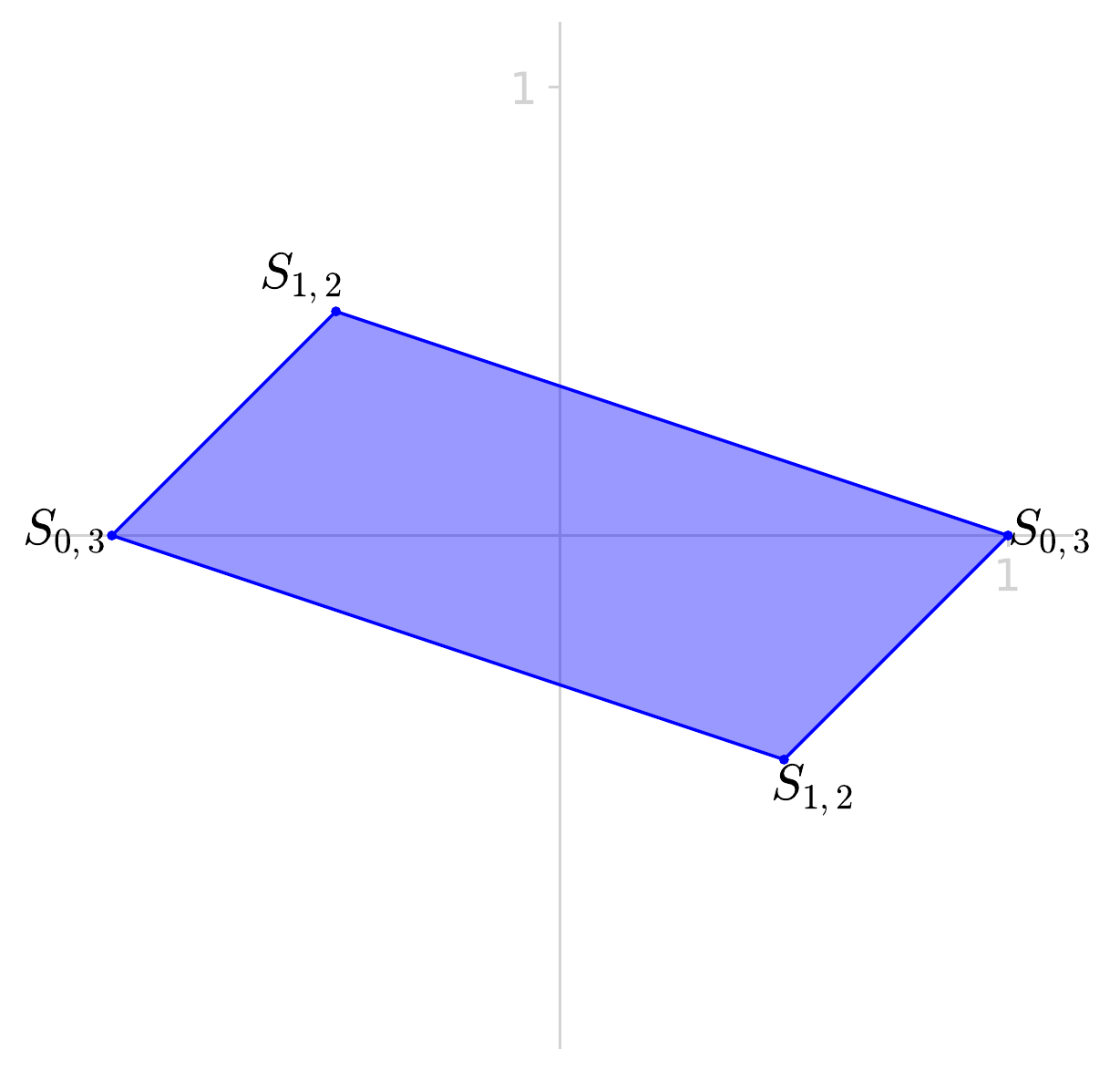}} \\ 
 $L=7^{{2}}_{{4}}$ & & $L=7^{{2}}_{{5}}$ & \\ 
 \quad & & \quad & \\ $\mathrm{Isom}(\mathbb{S}^3\setminus L) = \mathbb{{Z}}_2\oplus\mathbb{{Z}}_2$ & & $\mathrm{Isom}(\mathbb{S}^3\setminus L) = \mathbb{{Z}}_2\oplus\mathbb{{Z}}_2$ & \\ 
 \quad & & \quad & \\ 
 \includegraphics[width=1in]{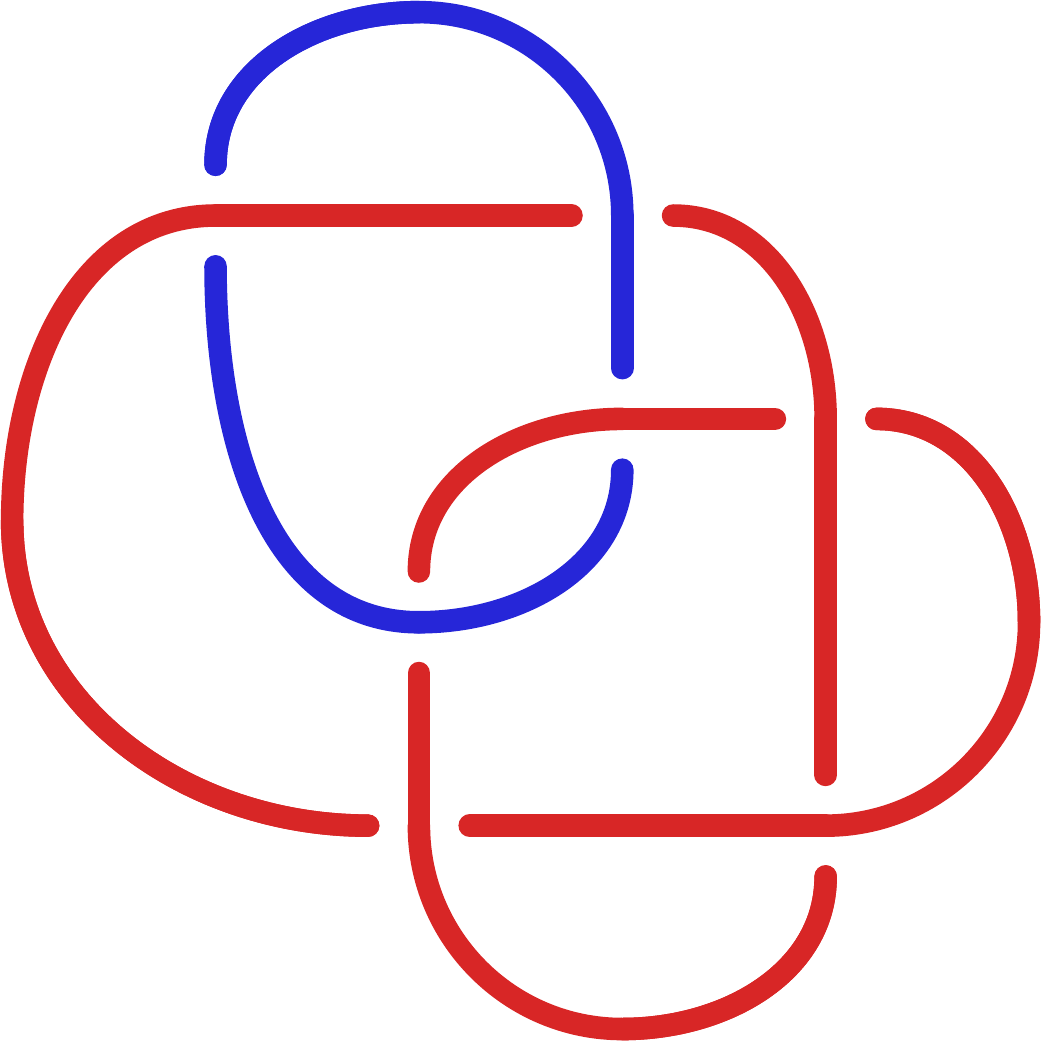}  & & \includegraphics[width=1in]{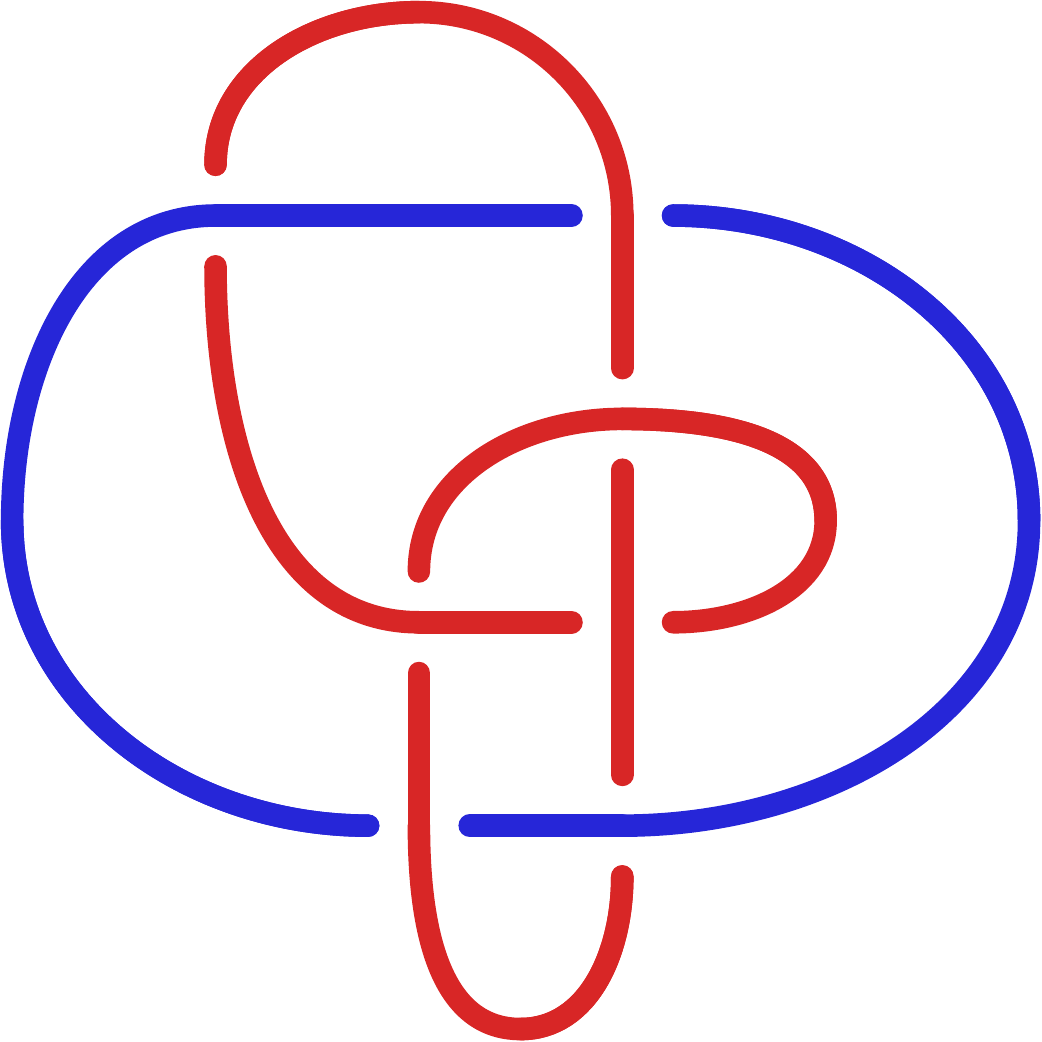} & \\ 
 \quad & & \quad & \\ 
 \hline  
\quad & \multirow{6}{*}{\Includegraphics[width=1.8in]{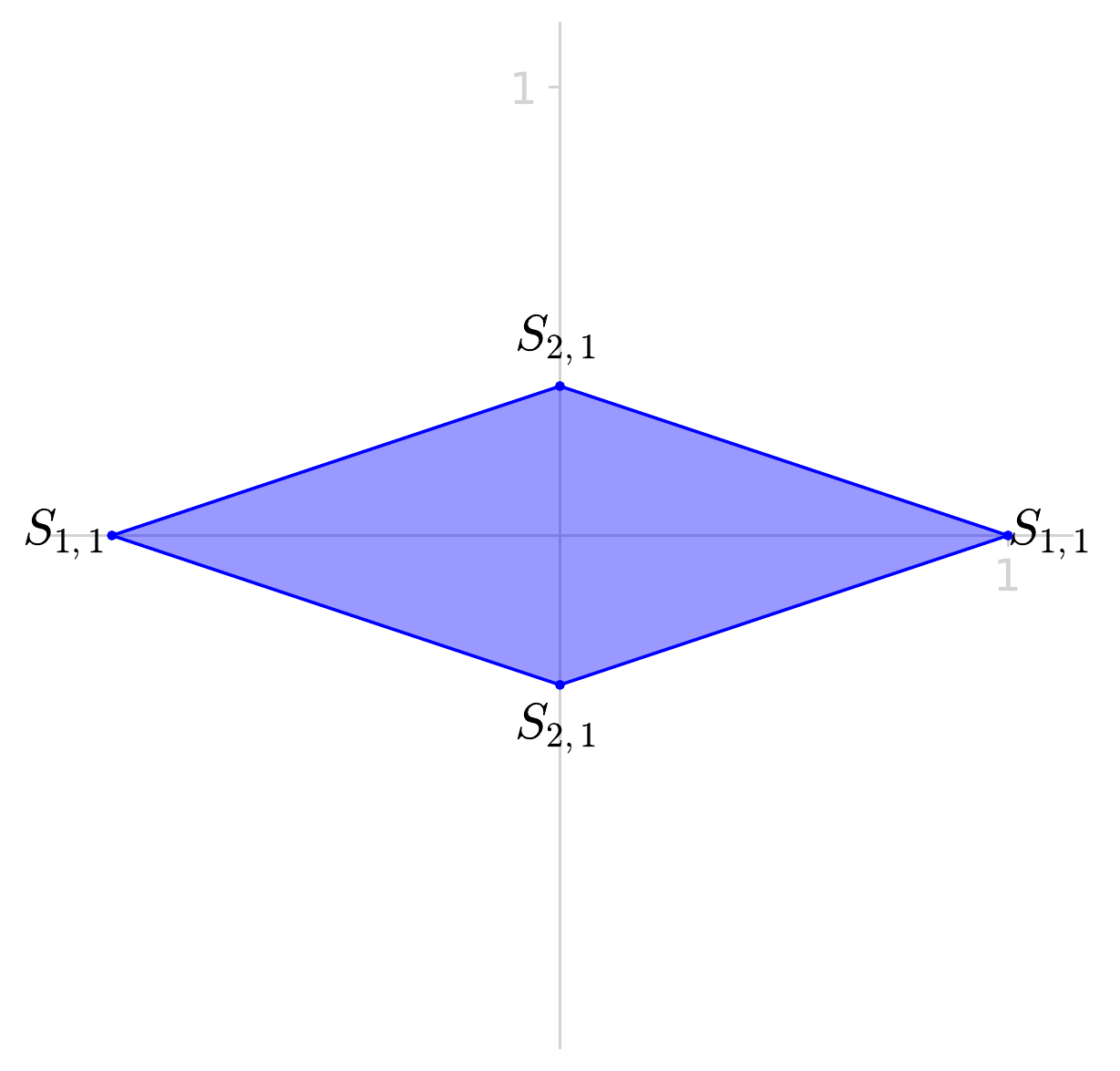}} & \quad & \multirow{6}{*}{\Includegraphics[width=1.8in]{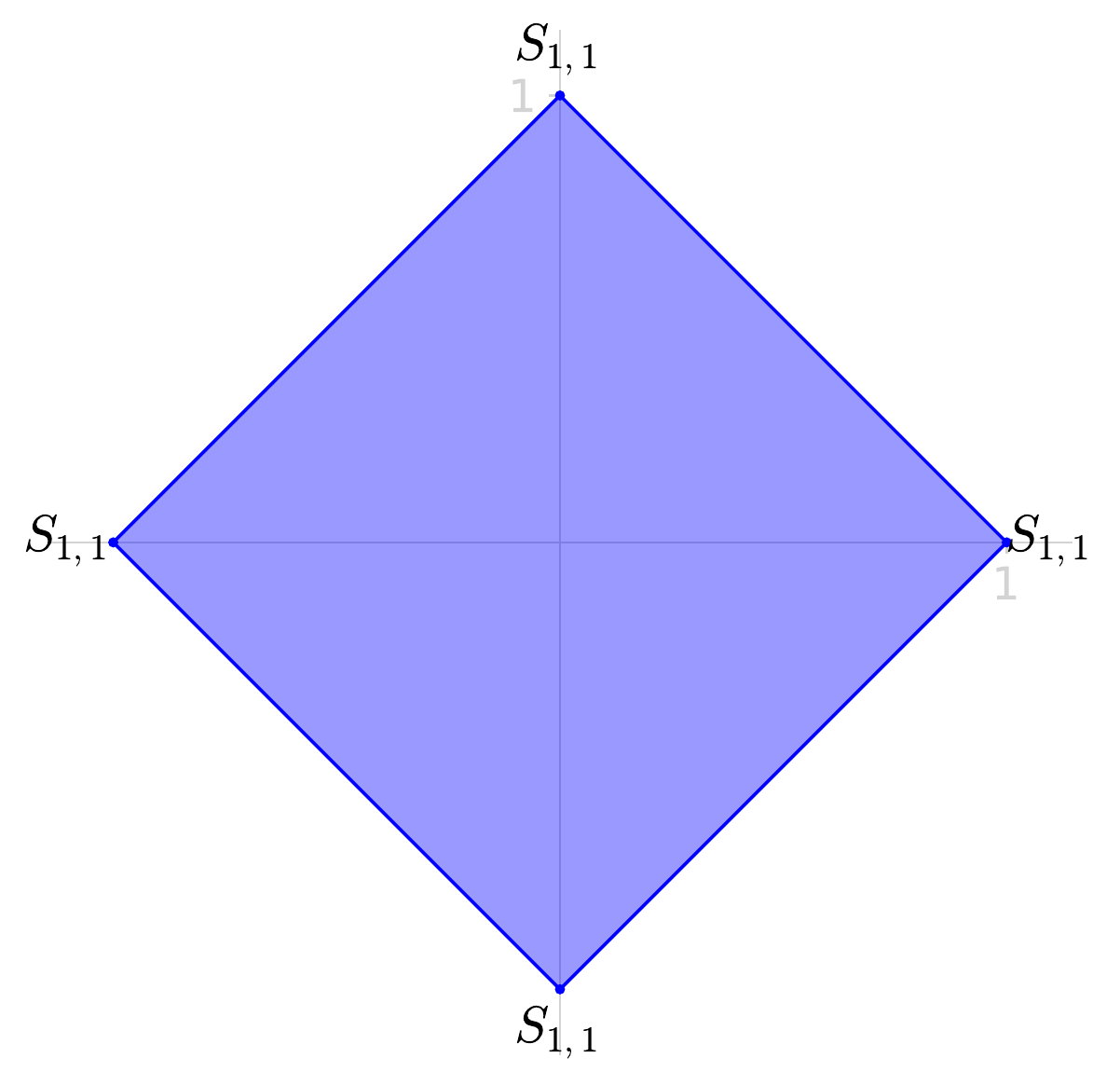}} \\ 
 $L=7^{{2}}_{{6}}$ & & $L=7^{{2}}_{{8}}$ & \\ 
 \quad & & \quad & \\ $\mathrm{Isom}(\mathbb{S}^3\setminus L) = \mathbb{{Z}}_2\oplus\mathbb{{Z}}_2$ & & $\mathrm{Isom}(\mathbb{S}^3\setminus L) = D_4$ & \\ 
 \quad & & \quad & \\ 
 \includegraphics[width=1in]{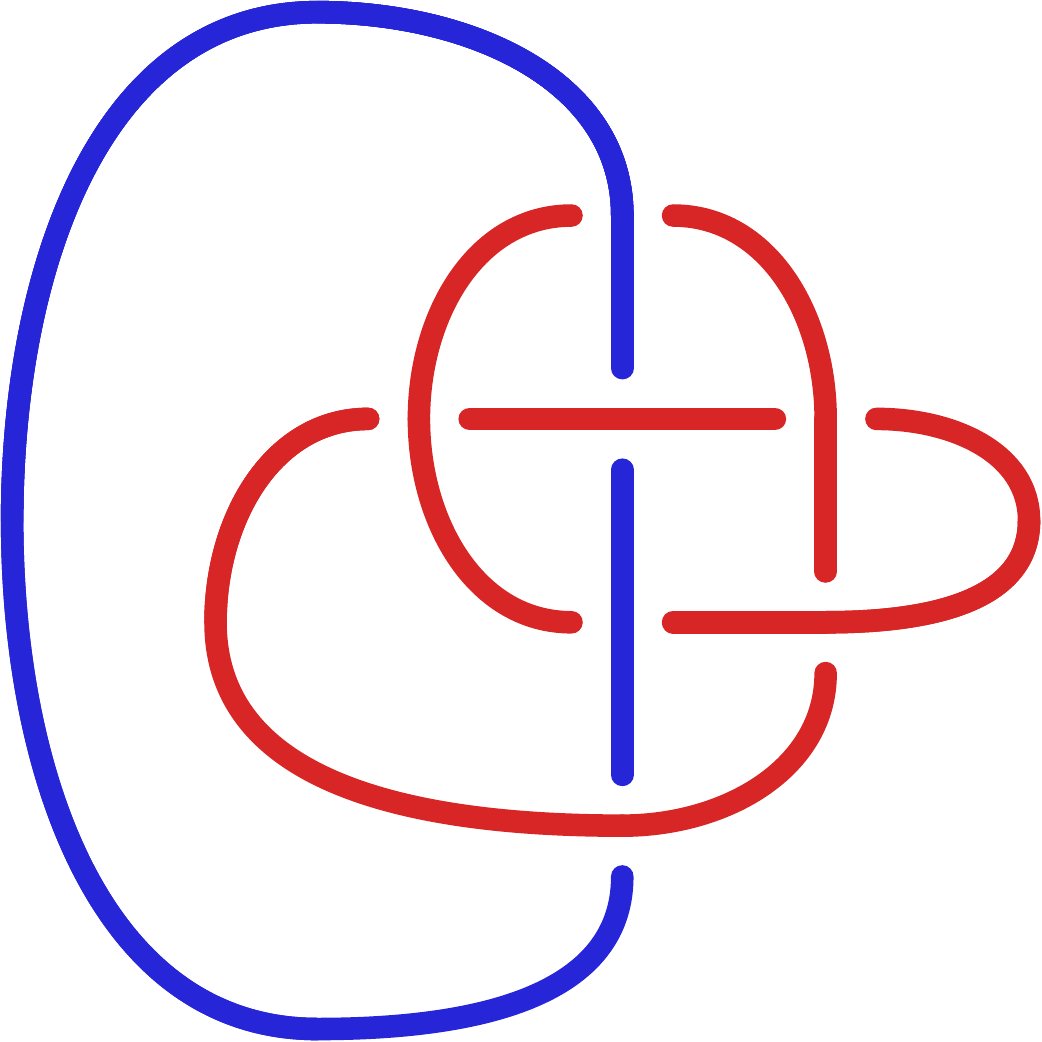}  & & \includegraphics[width=1in]{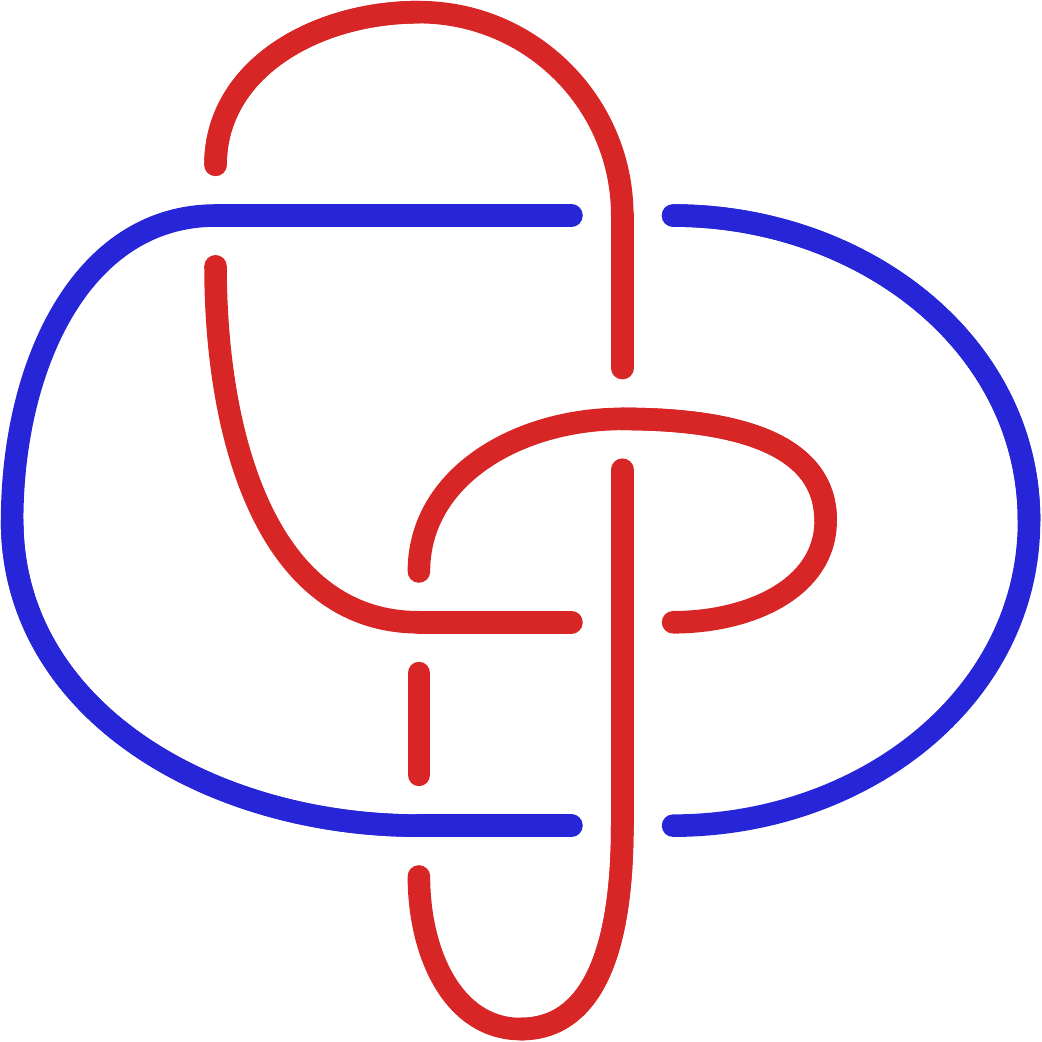} & \\ 
 \quad & & \quad & \\ 
 \hline  
\quad & \multirow{6}{*}{\Includegraphics[width=1.8in]{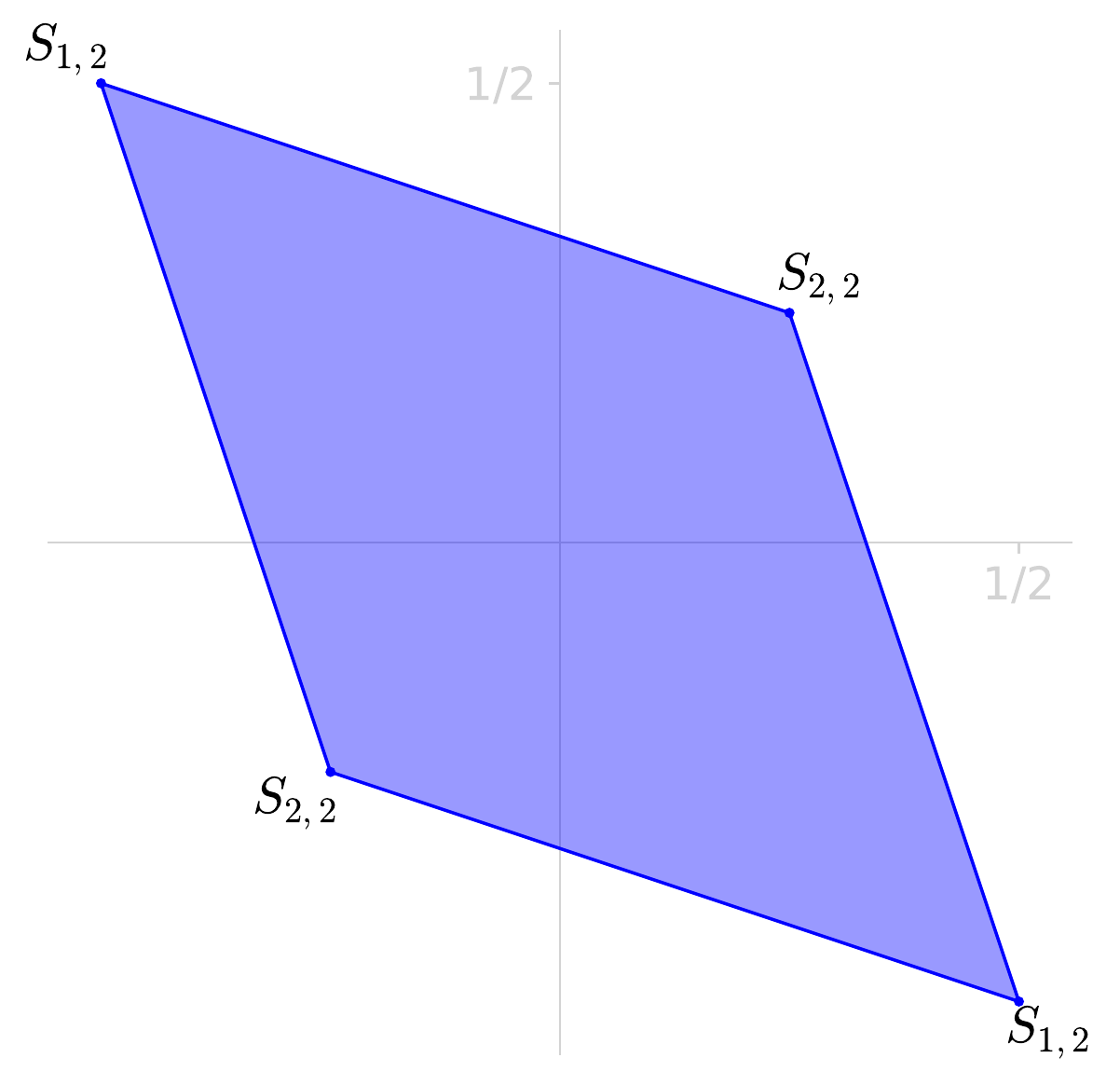}} & \quad & \multirow{6}{*}{\Includegraphics[width=1.8in]{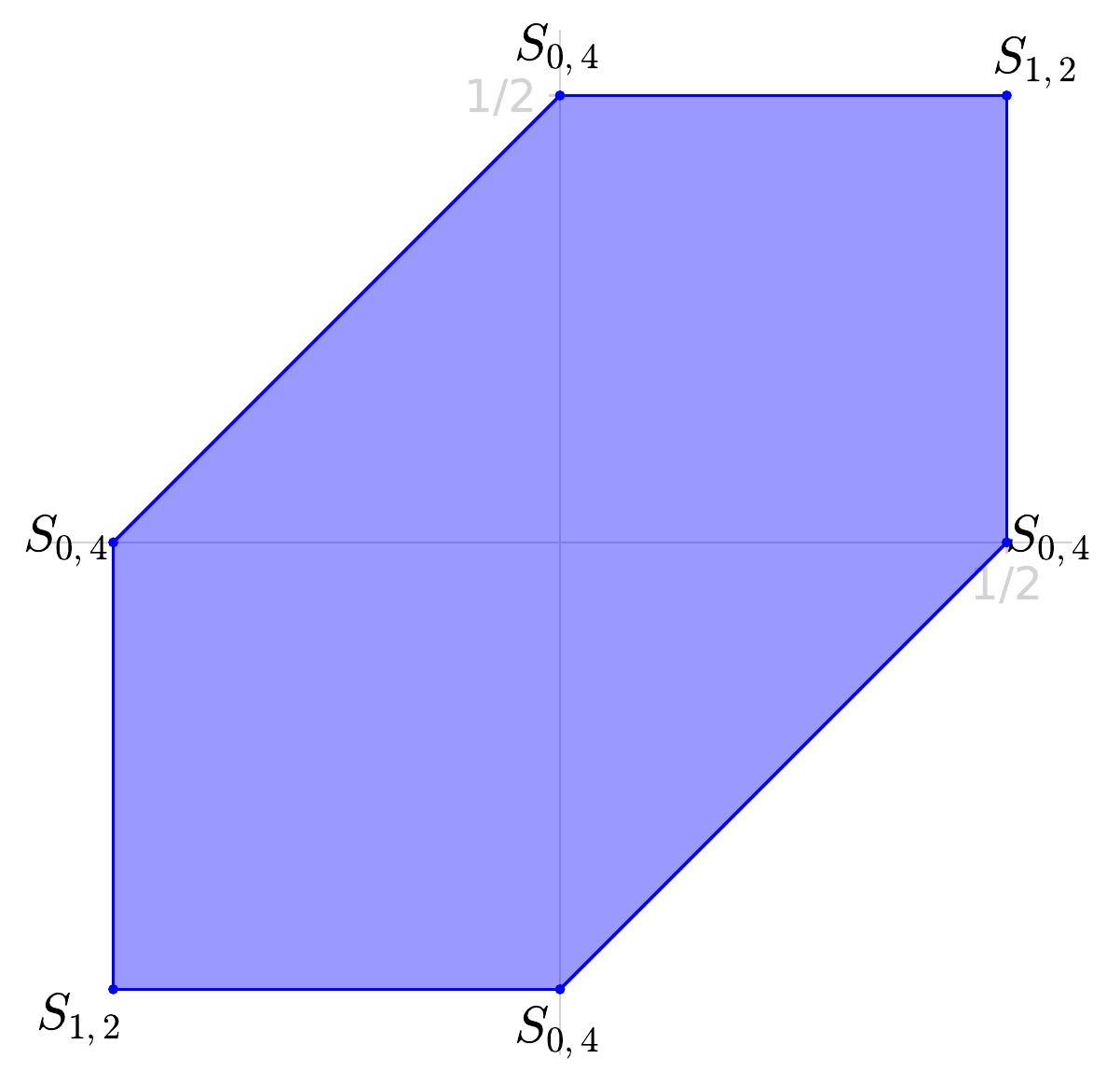}} \\ 
 $L=8^{{2}}_{{2}}$ & & $L=8^{{2}}_{{3}}$ & \\ 
 \quad & & \quad & \\ $\mathrm{Isom}(\mathbb{S}^3\setminus L) = \mathbb{{Z}}_2\oplus\mathbb{{Z}}_2$ & & $\mathrm{Isom}(\mathbb{S}^3\setminus L) = \mathbb{{Z}}_2\oplus\mathbb{{Z}}_2$ & \\ 
 \quad & & \quad & \\ 
 \includegraphics[width=1in]{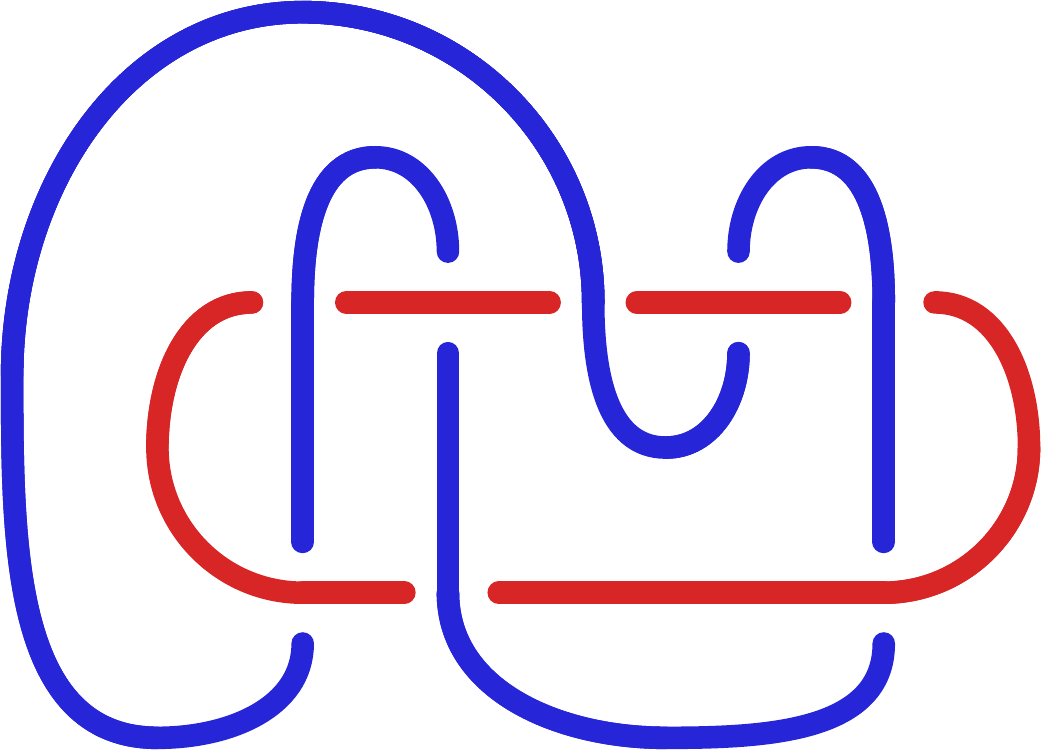}  & & \includegraphics[width=1in]{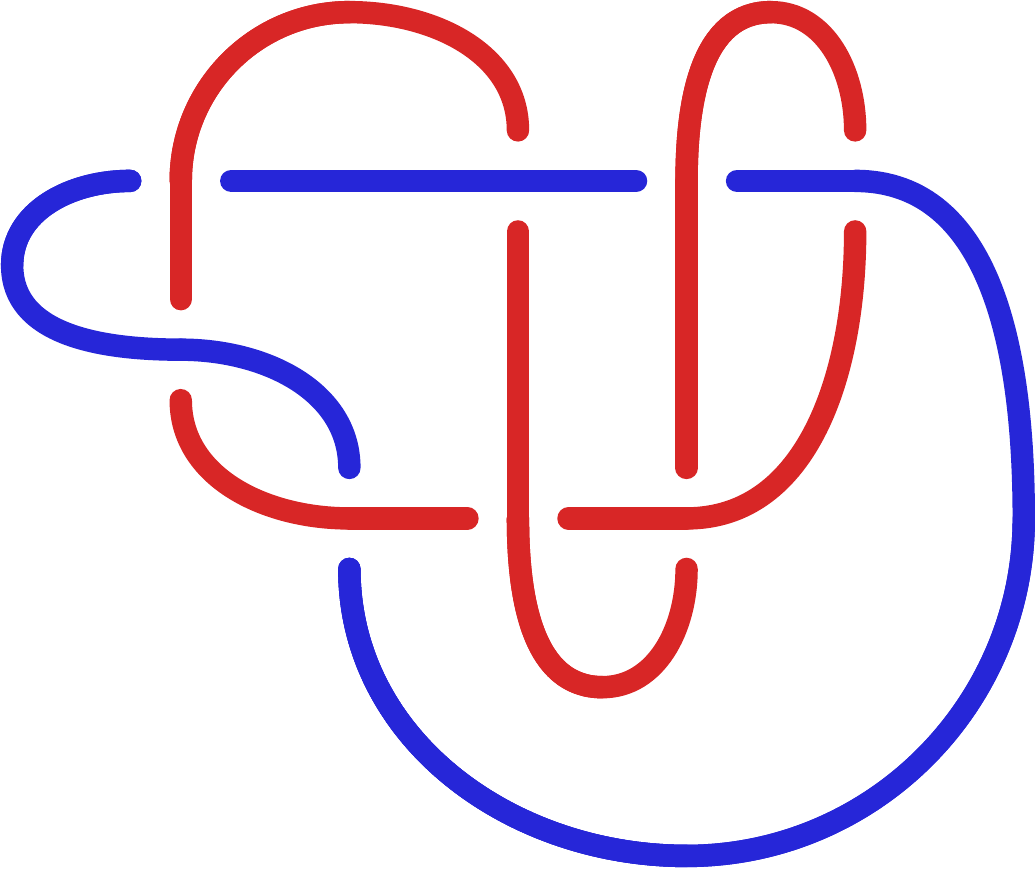} & \\ 
 \quad & & \quad & \\ 
 \hline  
\quad & \multirow{6}{*}{\Includegraphics[width=1.8in]{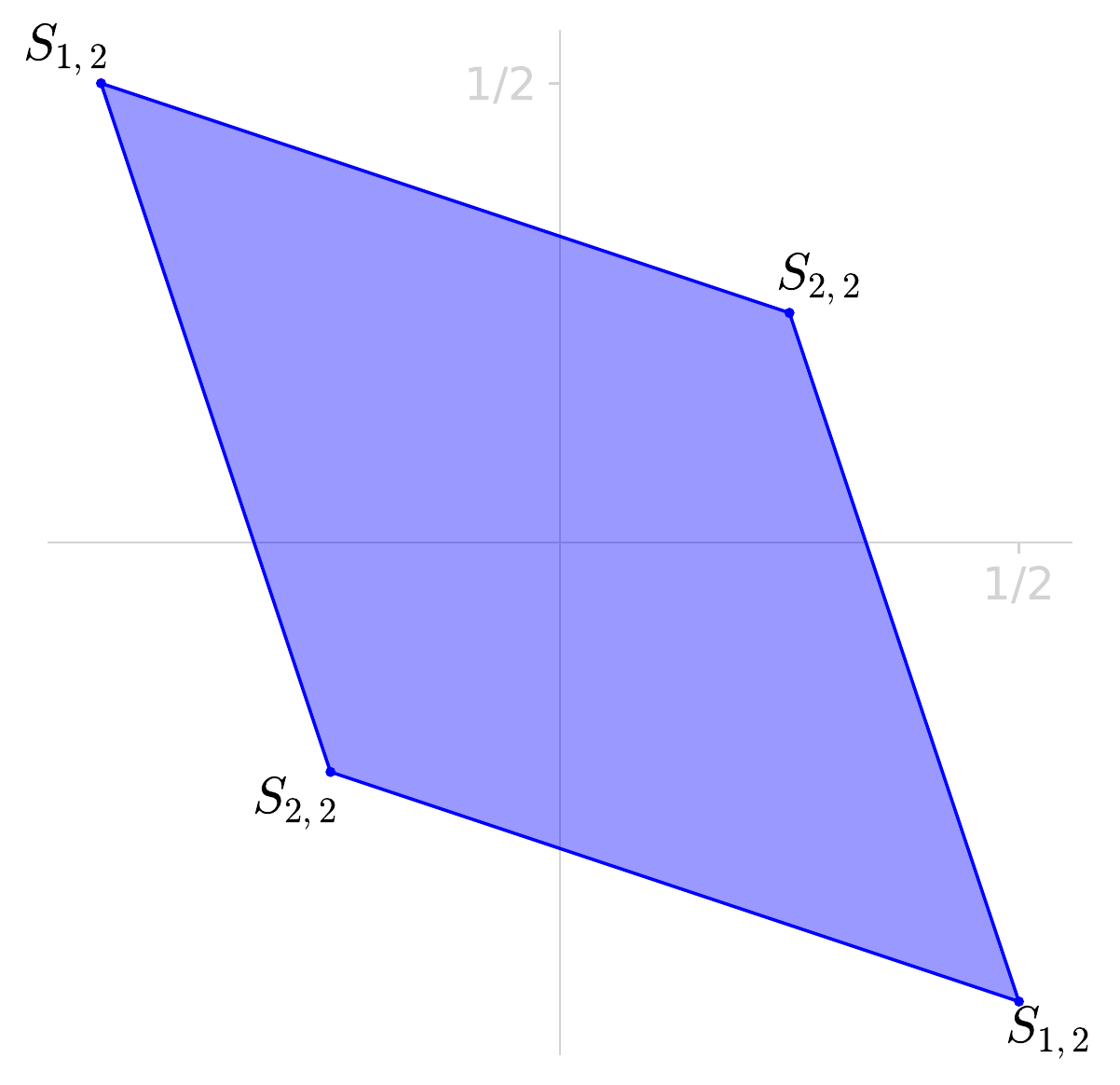}} & \quad & \multirow{6}{*}{\Includegraphics[width=1.8in]{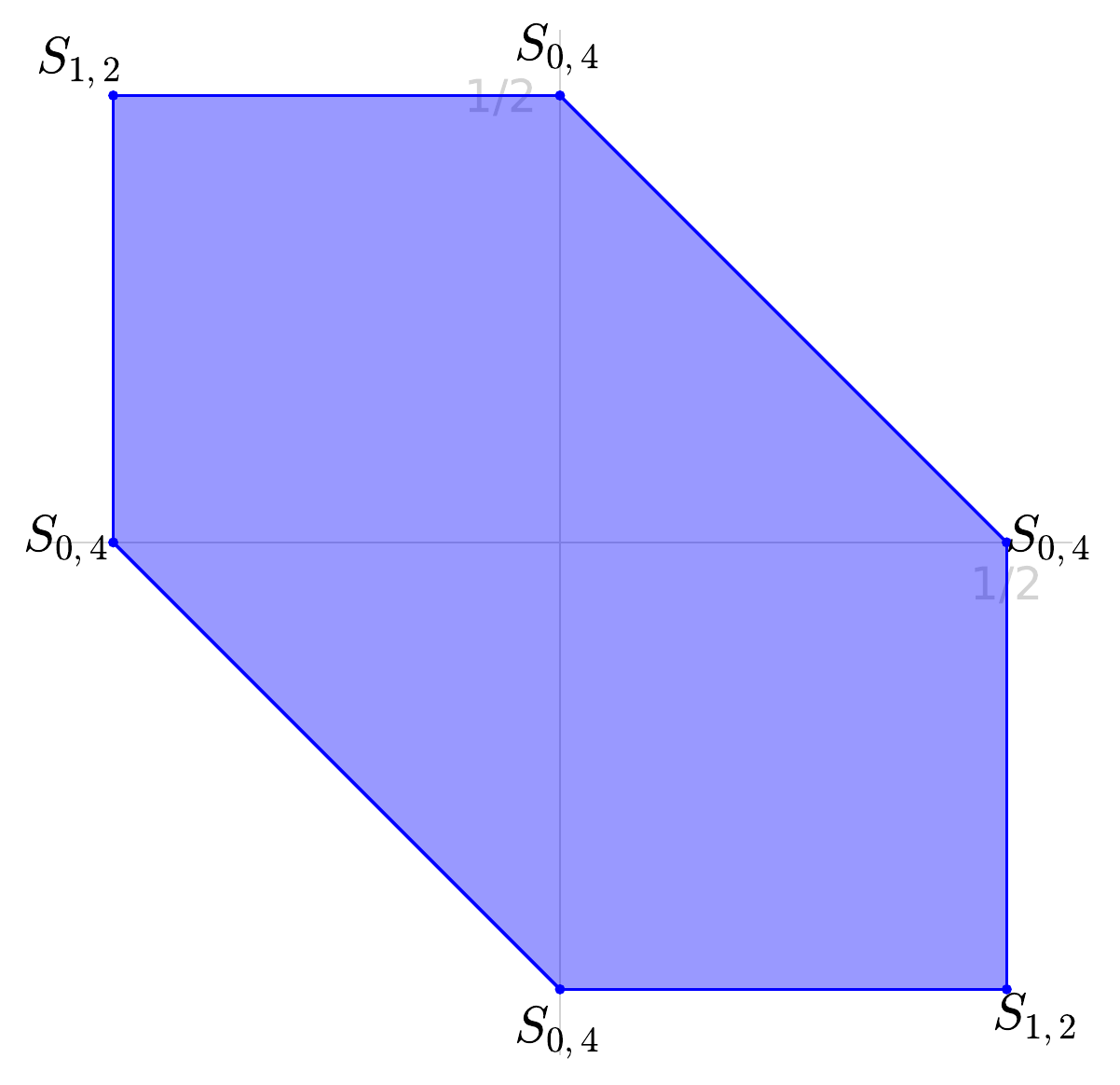}} \\ 
 $L=8^{{2}}_{{4}}$ & & $L=8^{{2}}_{{5}}$ & \\ 
 \quad & & \quad & \\ $\mathrm{Isom}(\mathbb{S}^3\setminus L) = \displaystyle\bigoplus_{i=1}^3 \mathbb{Z}$ & & $\mathrm{Isom}(\mathbb{S}^3\setminus L) = \mathbb{{Z}}_2\oplus\mathbb{{Z}}_2$ & \\ 
 \quad & & \quad & \\ 
 \includegraphics[width=1in]{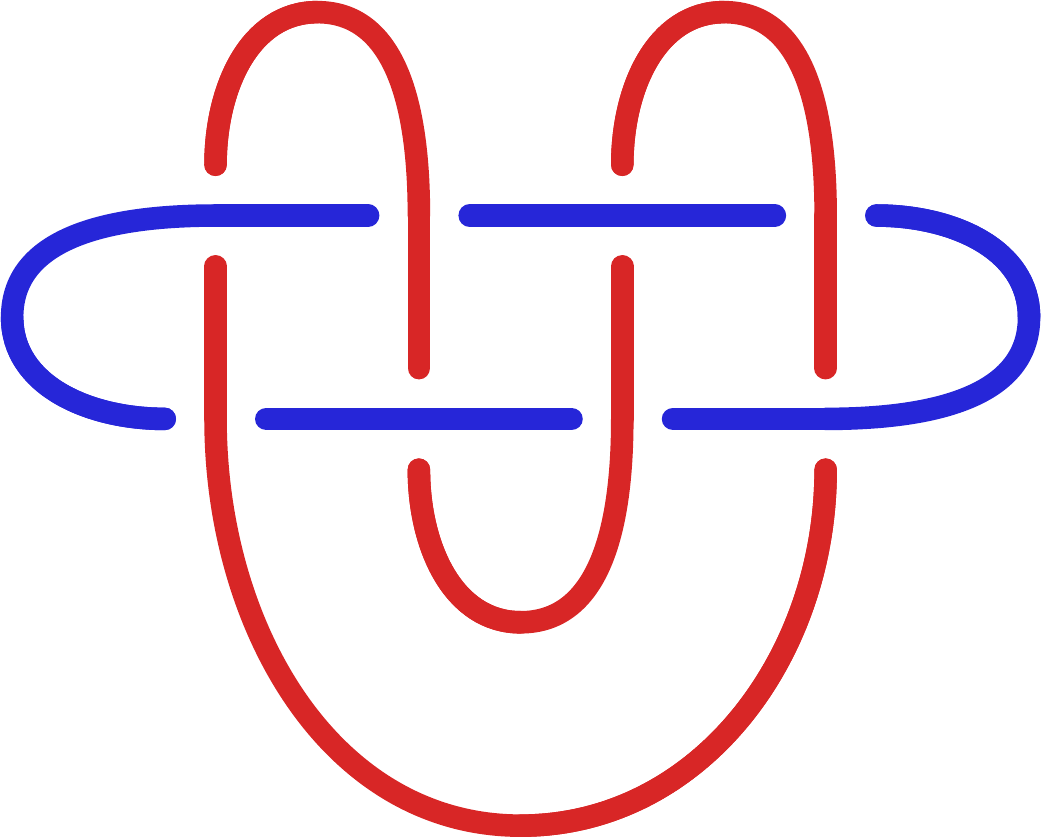}  & & \includegraphics[width=1in]{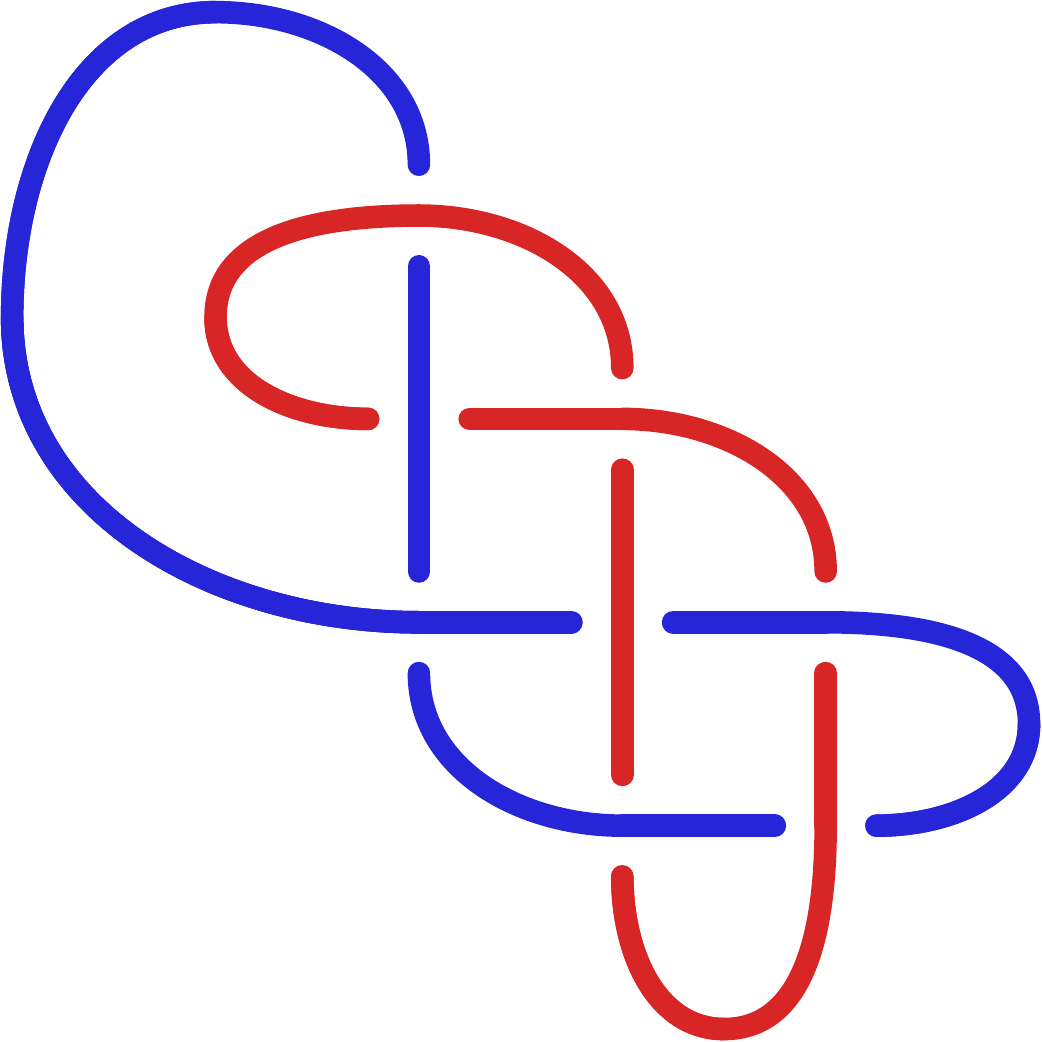} & \\ 
 \quad & & \quad & \\ 
 \hline  
\quad & \multirow{6}{*}{\Includegraphics[width=1.8in]{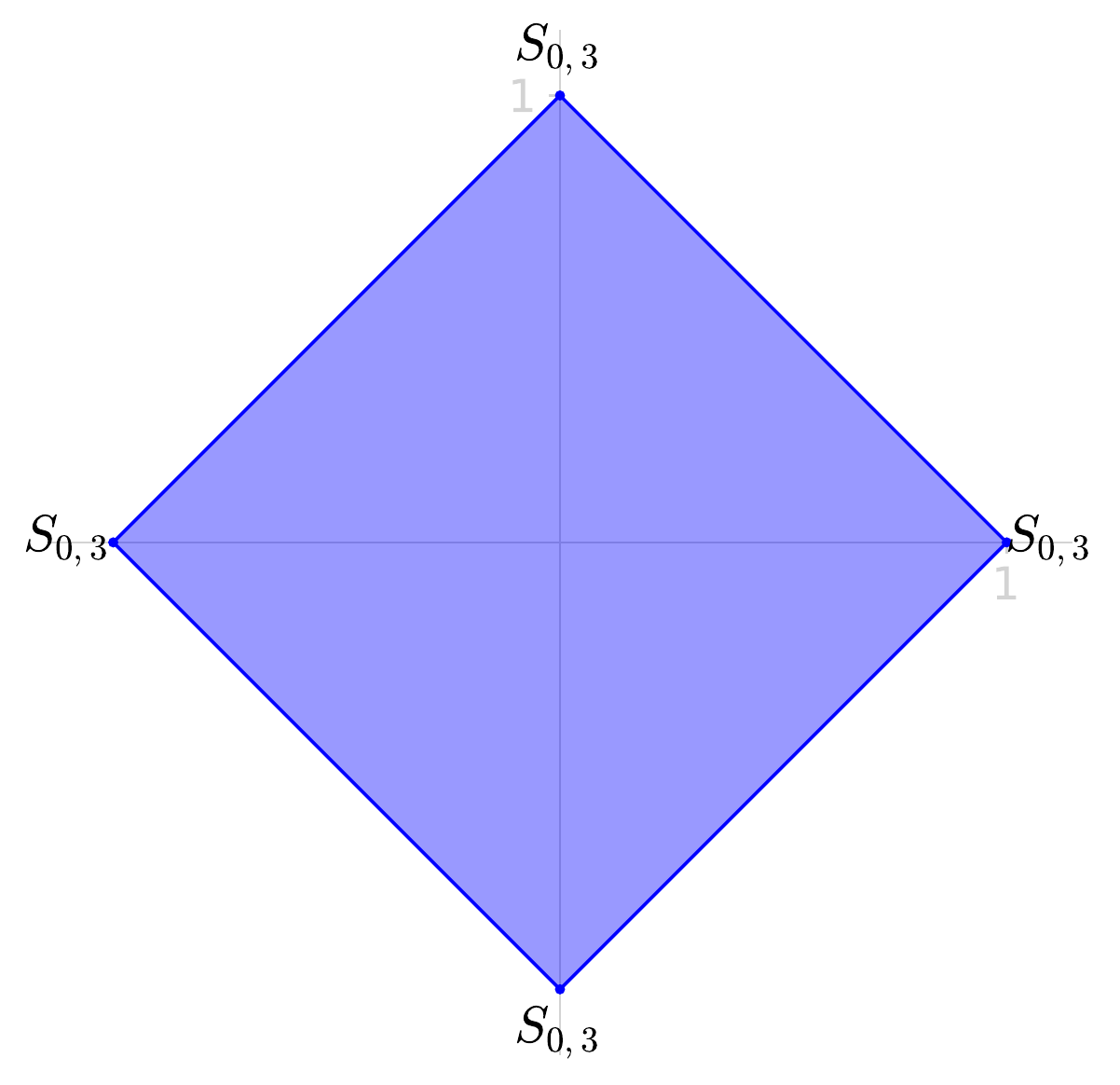}} & \quad & \multirow{6}{*}{\Includegraphics[width=1.8in]{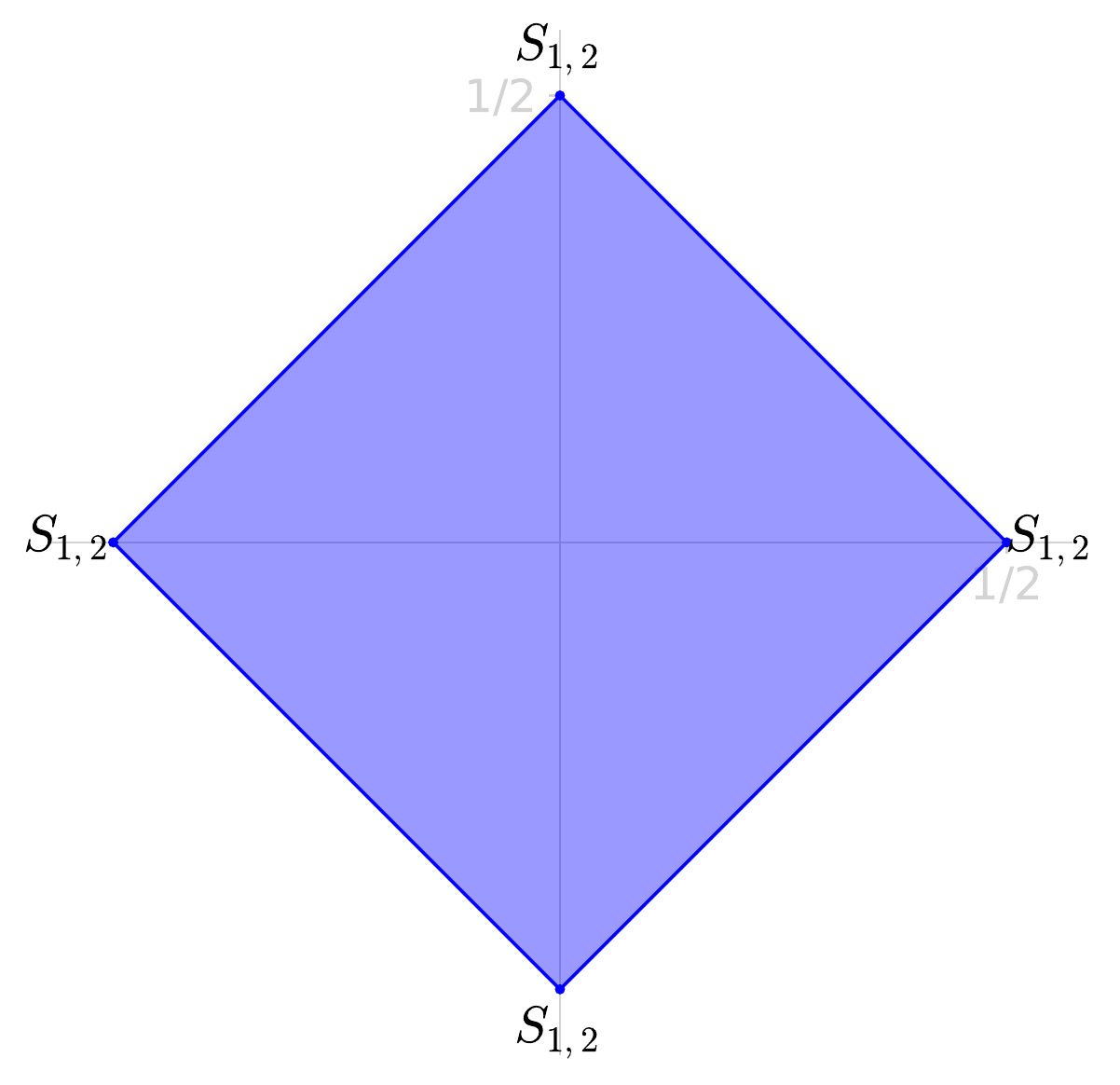}} \\ 
 $L=8^{{2}}_{{6}}$ & & $L=8^{{2}}_{{7}}$ & \\ 
 \quad & & \quad & \\ $\mathrm{Isom}(\mathbb{S}^3\setminus L) = \displaystyle\bigoplus_{i=1}^3 \mathbb{Z}$ & & $\mathrm{Isom}(\mathbb{S}^3\setminus L) = \displaystyle\bigoplus_{i=1}^3 \mathbb{Z}$ & \\ 
 \quad & & \quad & \\ 
 \includegraphics[width=1in]{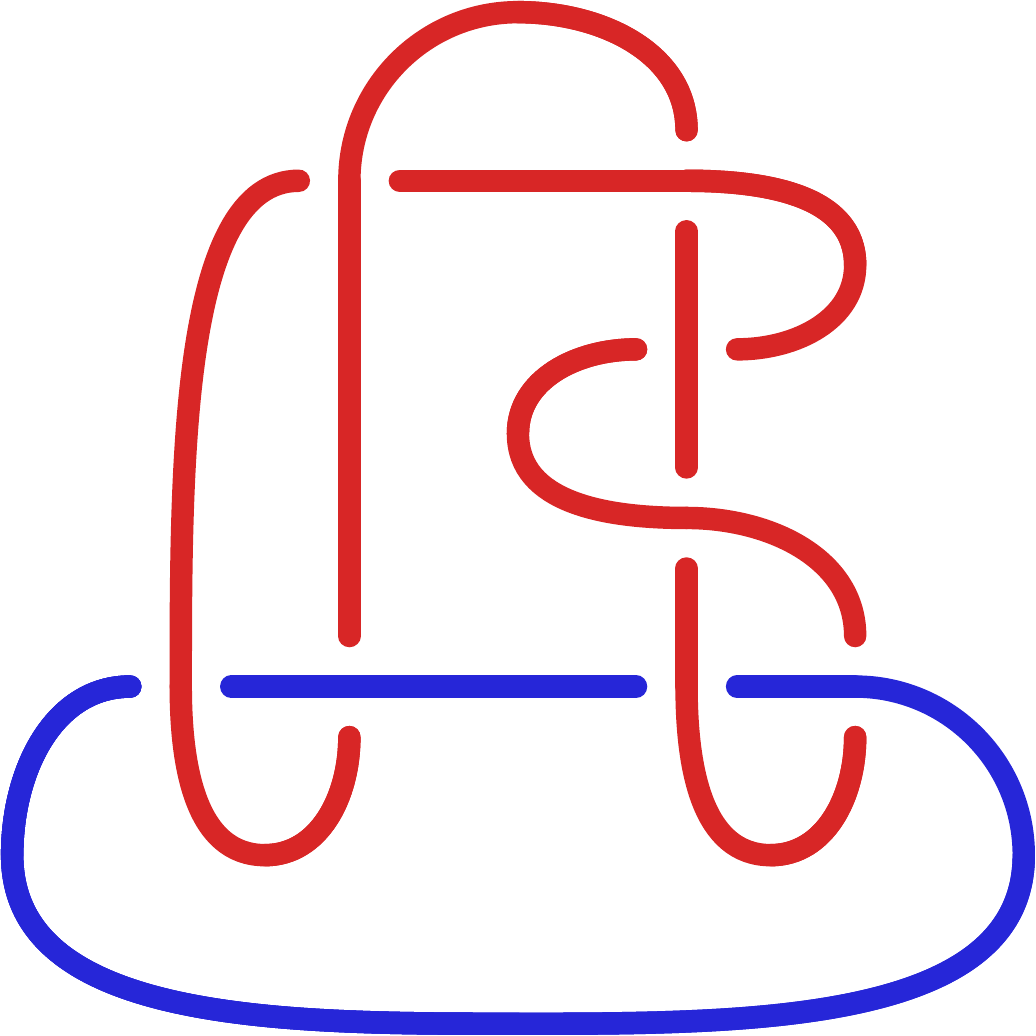}  & & \includegraphics[width=1in]{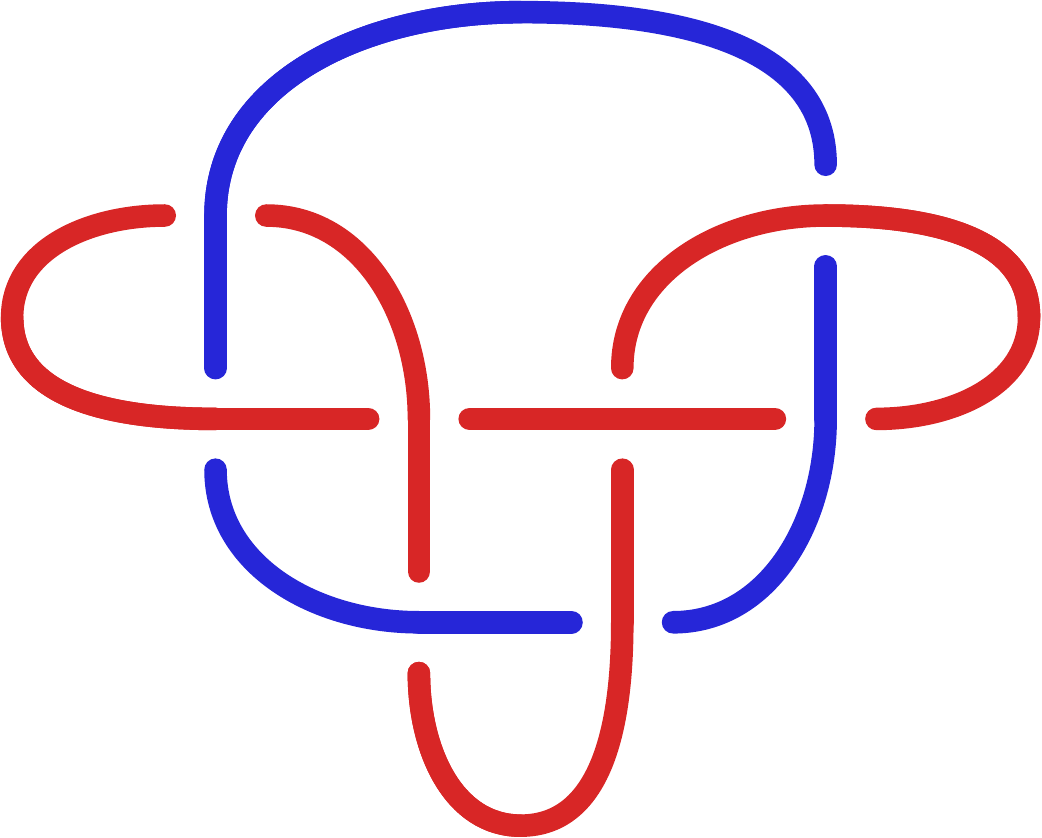} & \\ 
 \quad & & \quad & \\ 
 \hline  
\end{tabular} 
 \newpage \begin{tabular}{|c|c|c|c|} 
 \hline 
 Link & Norm Ball & Link & Norm Ball \\ 
 \hline 
\quad & \multirow{6}{*}{\Includegraphics[width=1.8in]{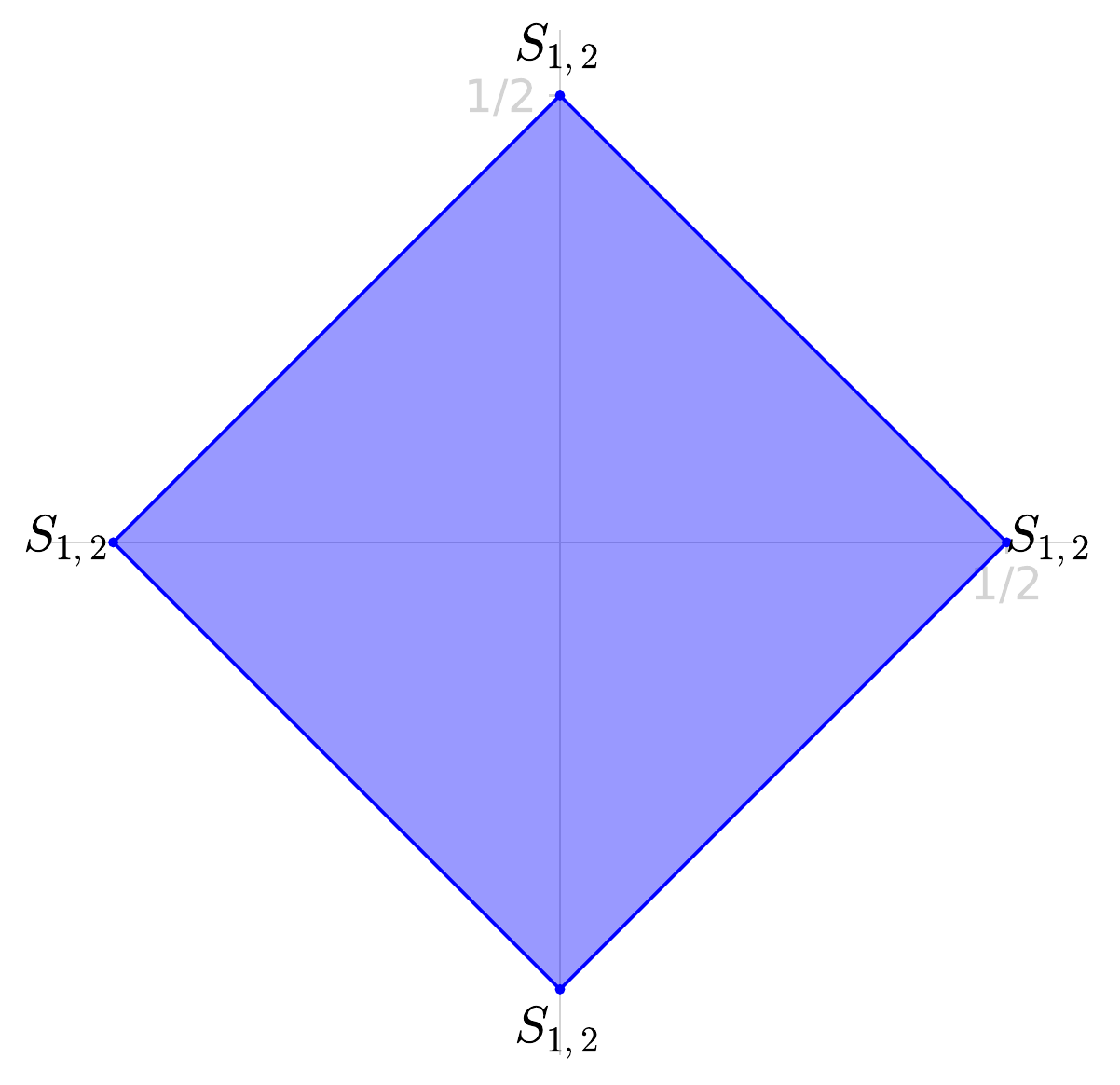}} & \quad & \multirow{6}{*}{\Includegraphics[width=1.8in]{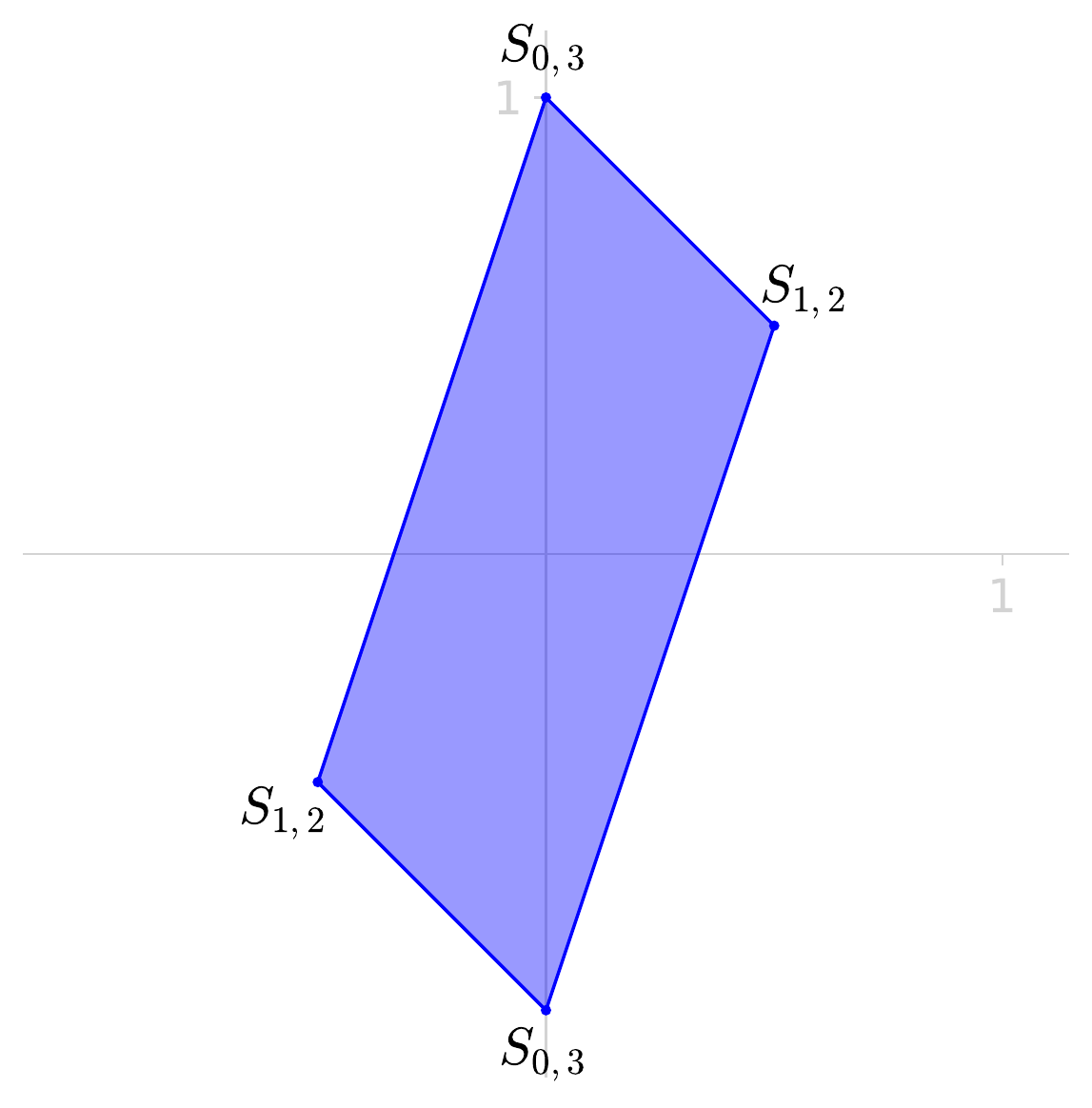}} \\ 
 $L=8^{{2}}_{{8}}$ & & $L=8^{{2}}_{{9}}$ & \\ 
 \quad & & \quad & \\ $\mathrm{Isom}(\mathbb{S}^3\setminus L) = D_4$ & & $\mathrm{Isom}(\mathbb{S}^3\setminus L) = \mathbb{{Z}}_2\oplus\mathbb{{Z}}_2$ & \\ 
 \quad & & \quad & \\ 
 \includegraphics[width=1in]{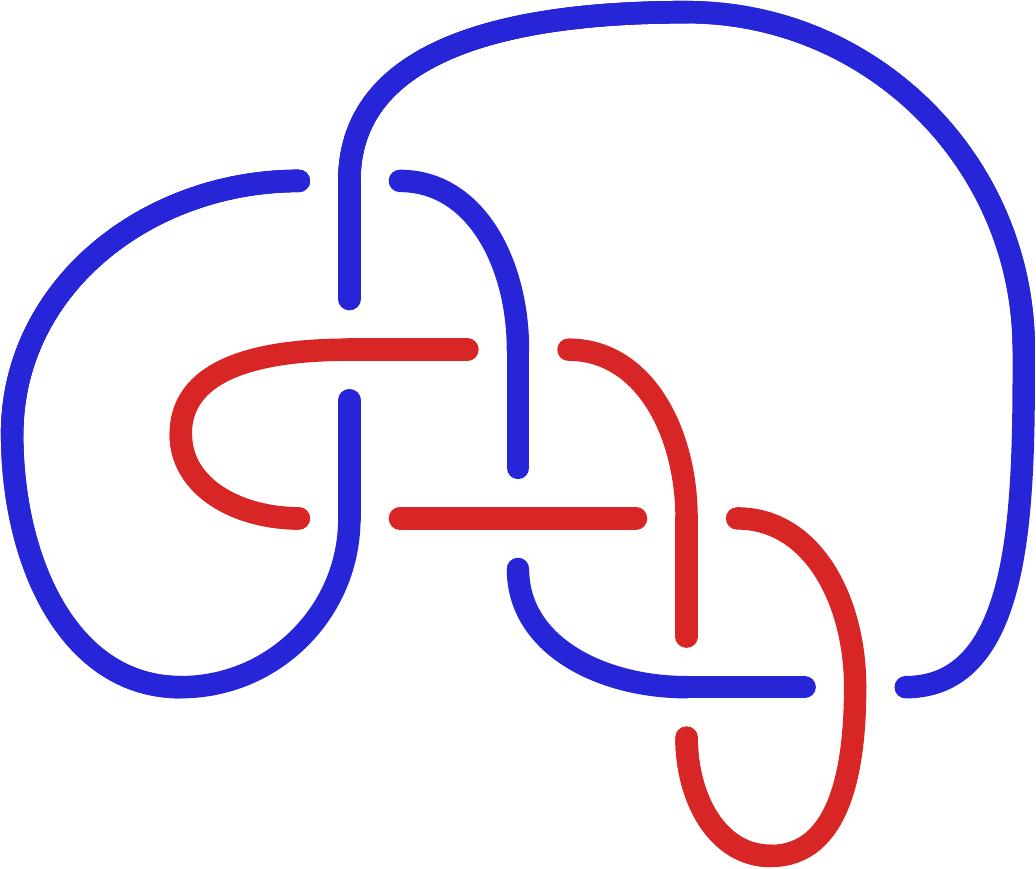}  & & \includegraphics[width=1in]{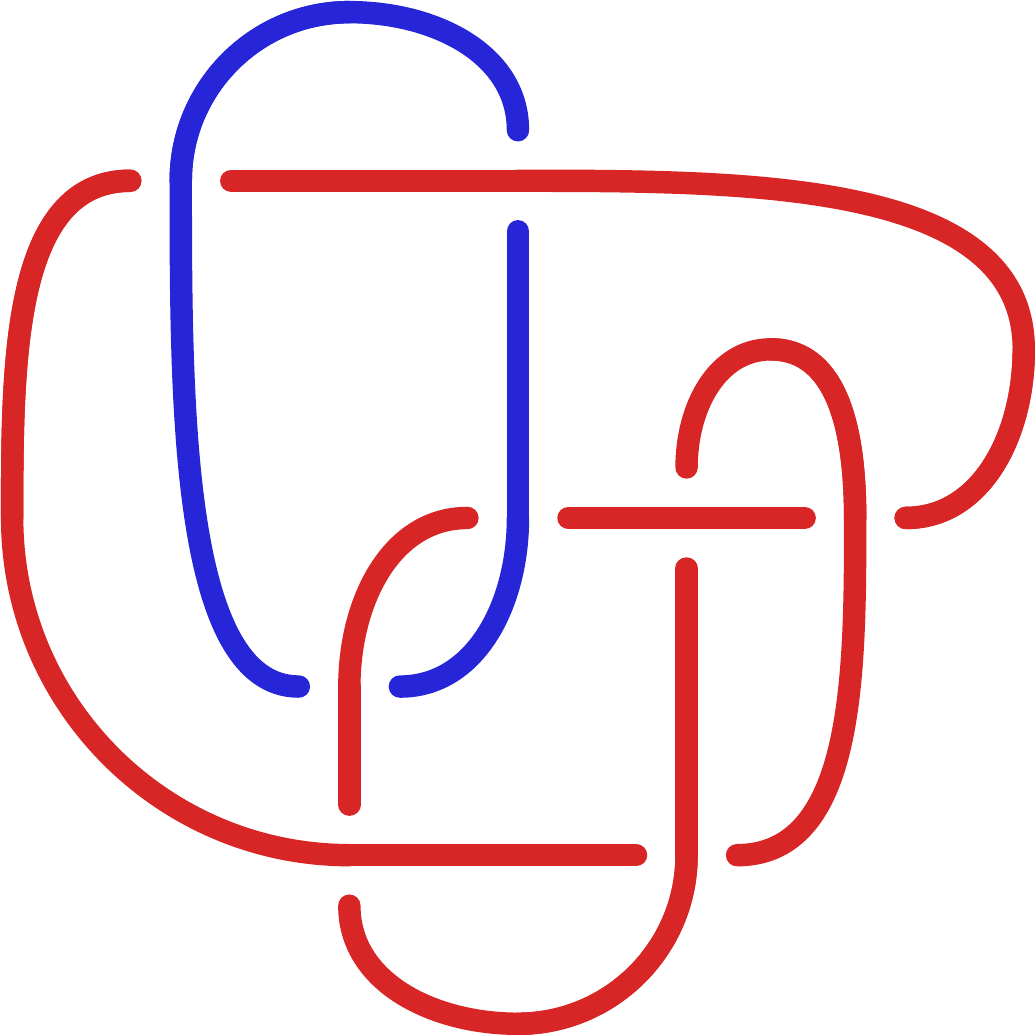} & \\ 
 \quad & & \quad & \\ 
 \hline  
\quad & \multirow{6}{*}{\Includegraphics[width=1.8in]{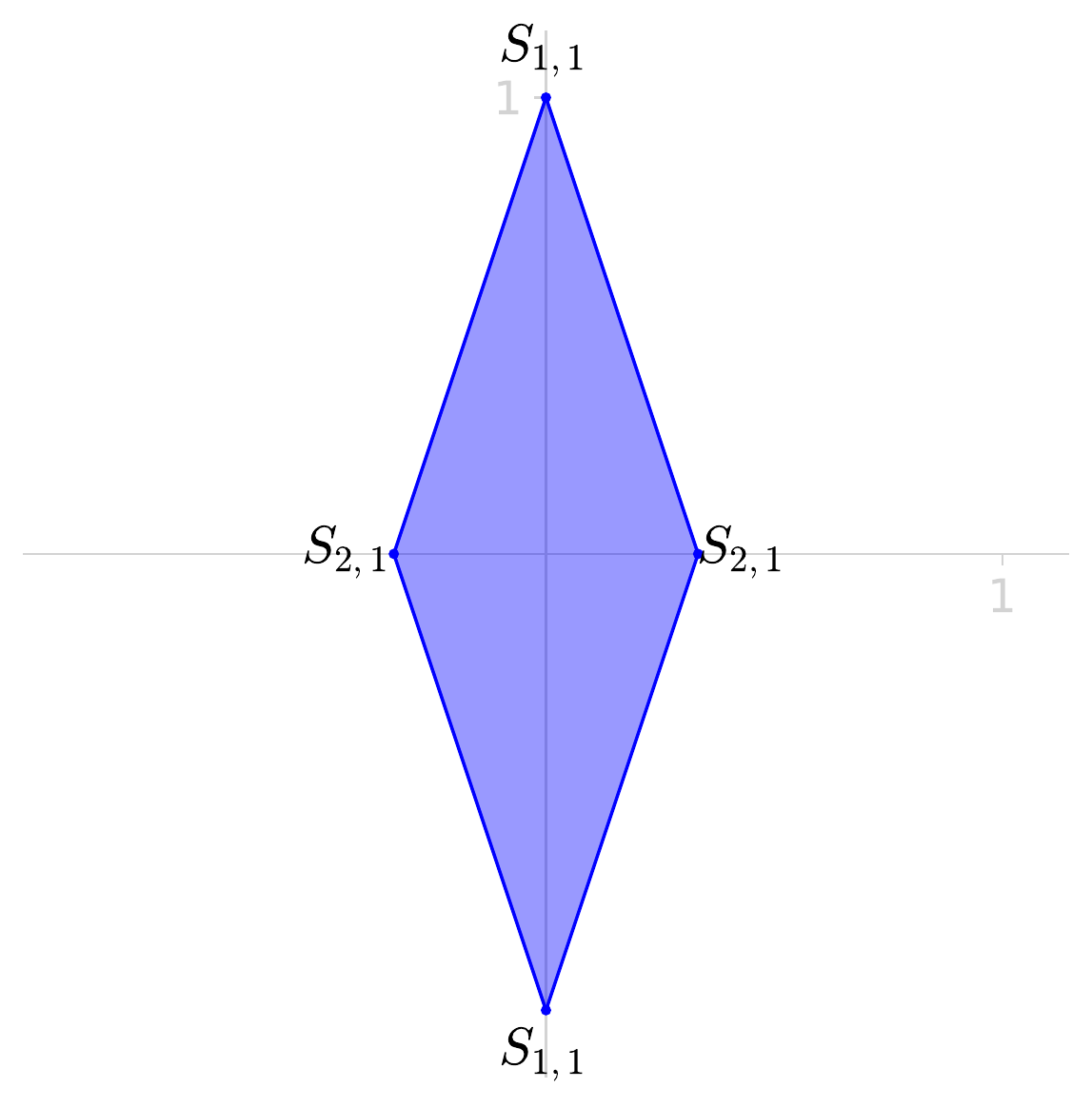}} & \quad & \multirow{6}{*}{\Includegraphics[width=1.8in]{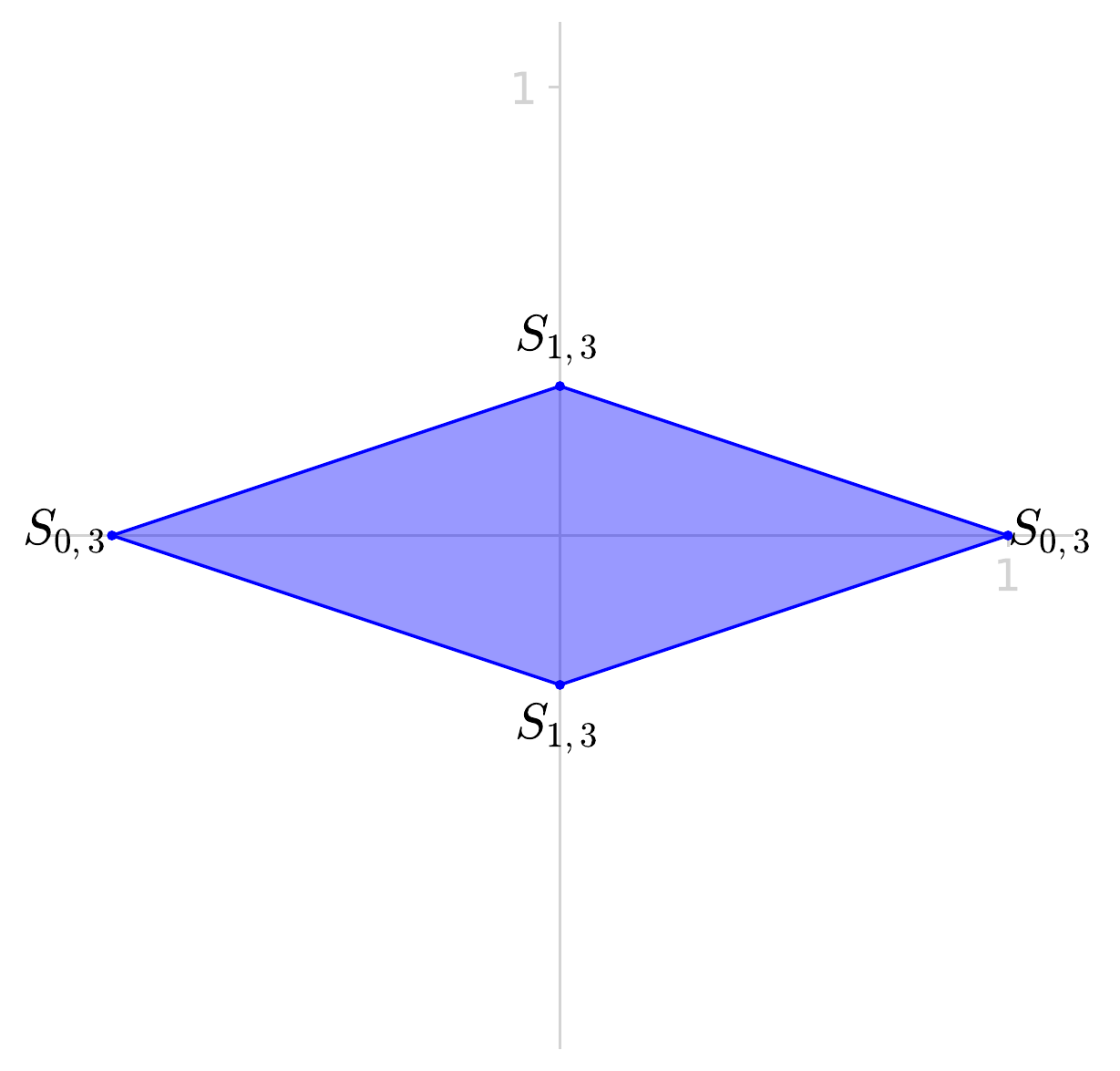}} \\ 
 $L=8^{{2}}_{{10}}$ & & $L=8^{{2}}_{{11}}$ & \\ 
 \quad & & \quad & \\ $\mathrm{Isom}(\mathbb{S}^3\setminus L) = \mathbb{{Z}}_2\oplus\mathbb{{Z}}_2$ & & $\mathrm{Isom}(\mathbb{S}^3\setminus L) = \mathbb{{Z}}_2\oplus\mathbb{{Z}}_2$ & \\ 
 \quad & & \quad & \\ 
 \includegraphics[width=1in]{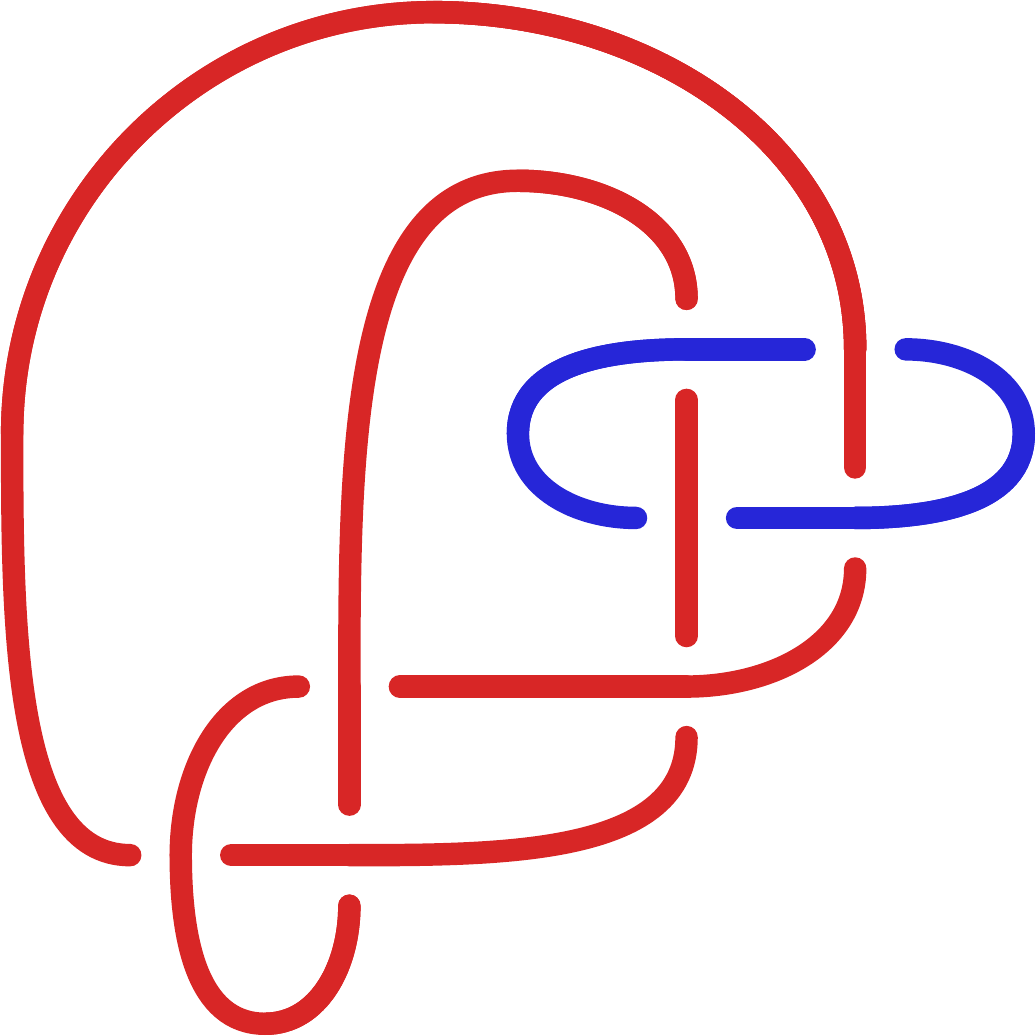}  & & \includegraphics[width=1in]{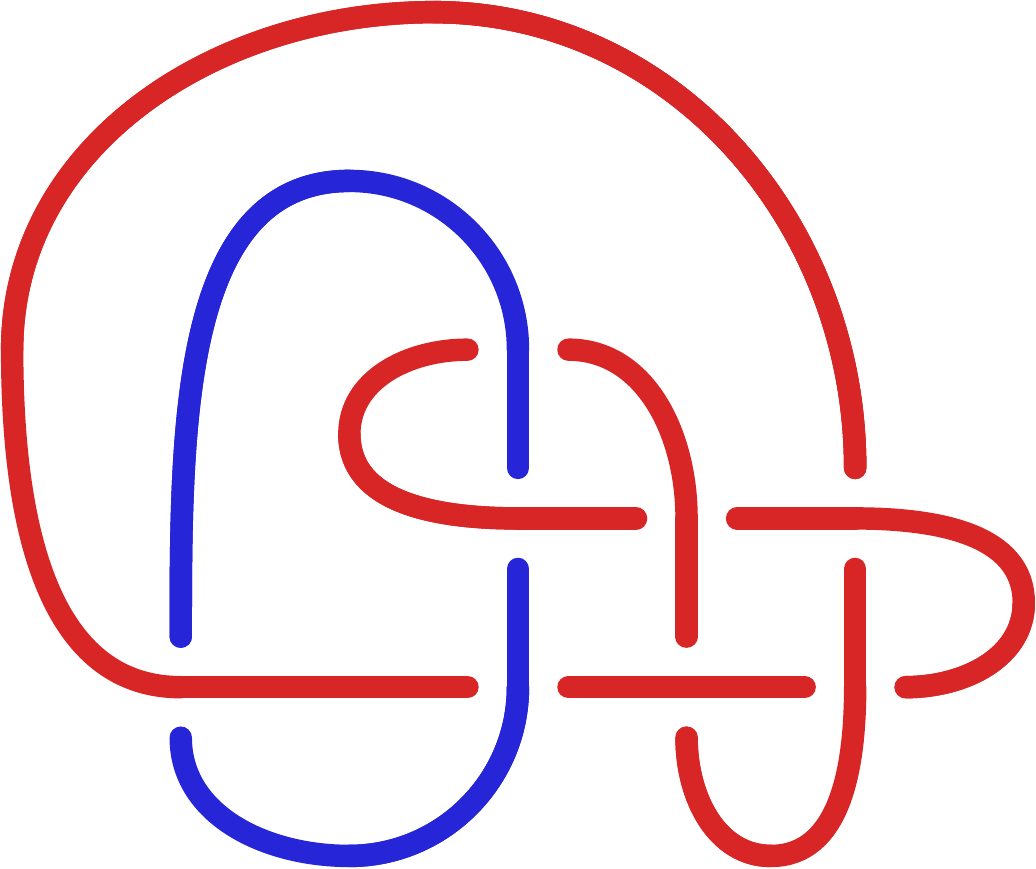} & \\ 
 \quad & & \quad & \\ 
 \hline  
\quad & \multirow{6}{*}{\Includegraphics[width=1.8in]{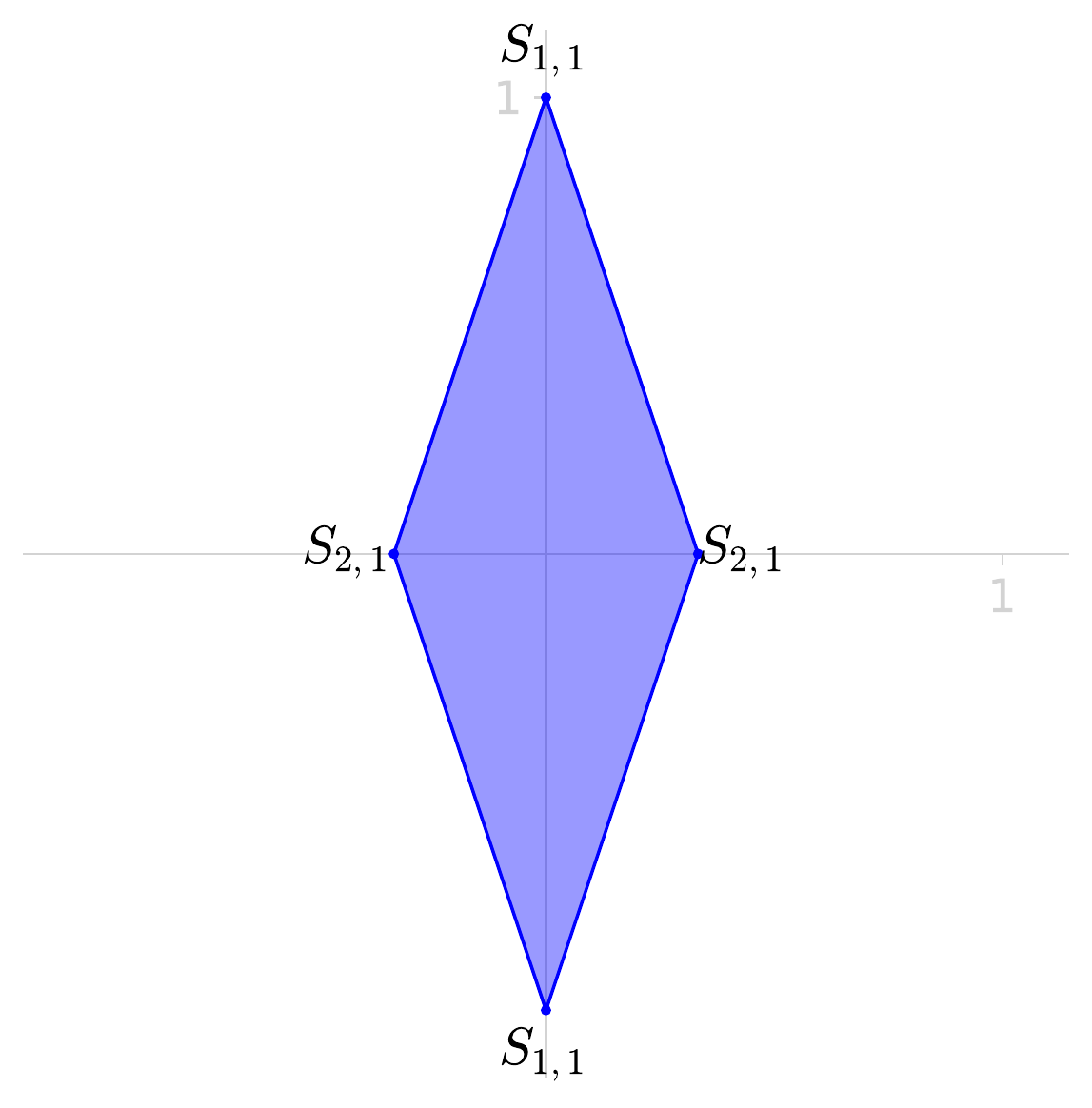}} & \quad & \multirow{6}{*}{\Includegraphics[width=1.8in]{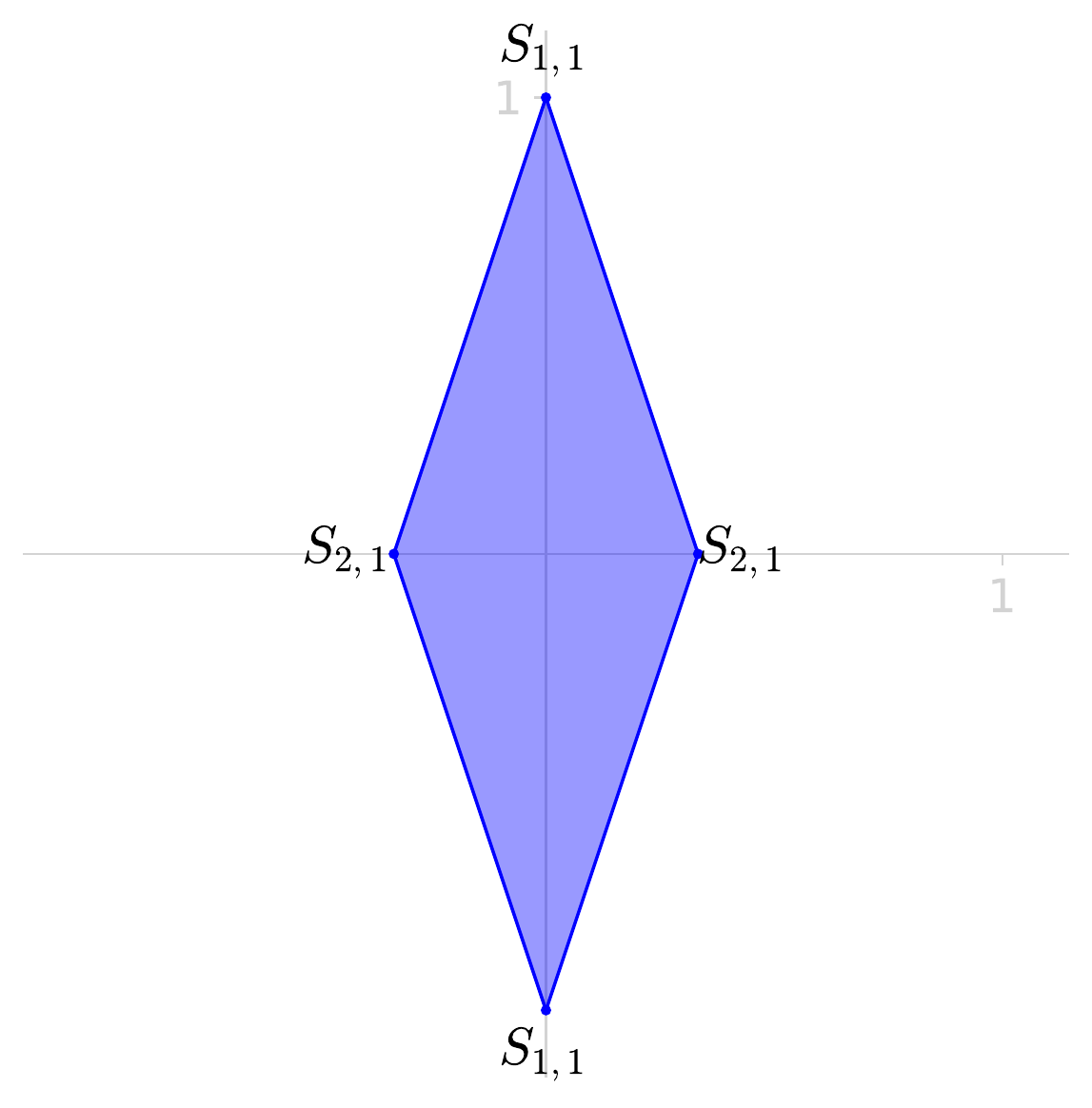}} \\ 
 $L=8^{{2}}_{{12}}$ & & $L=8^{{2}}_{{13}}$ & \\ 
 \quad & & \quad & \\ $\mathrm{Isom}(\mathbb{S}^3\setminus L) = \mathbb{{Z}}_2\oplus\mathbb{{Z}}_2$ & & $\mathrm{Isom}(\mathbb{S}^3\setminus L) = \mathbb{{Z}}_2\oplus\mathbb{{Z}}_2$ & \\ 
 \quad & & \quad & \\ 
 \includegraphics[width=1in]{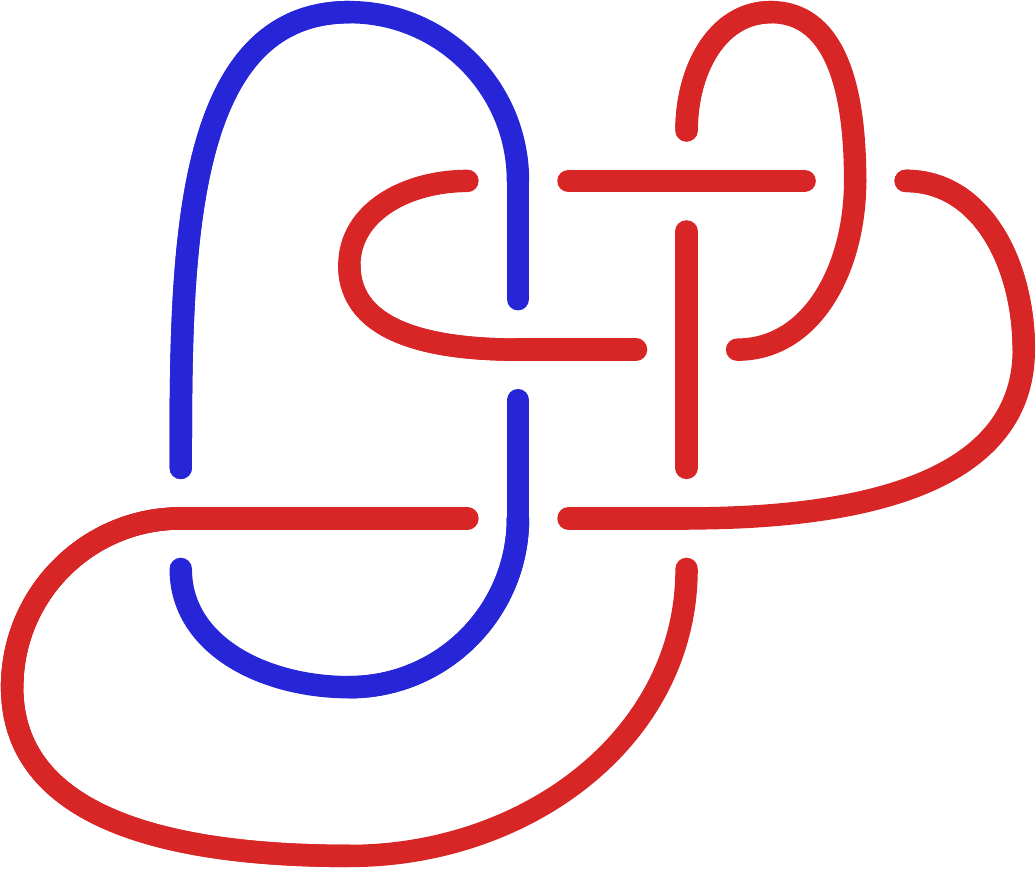}  & & \includegraphics[width=1in]{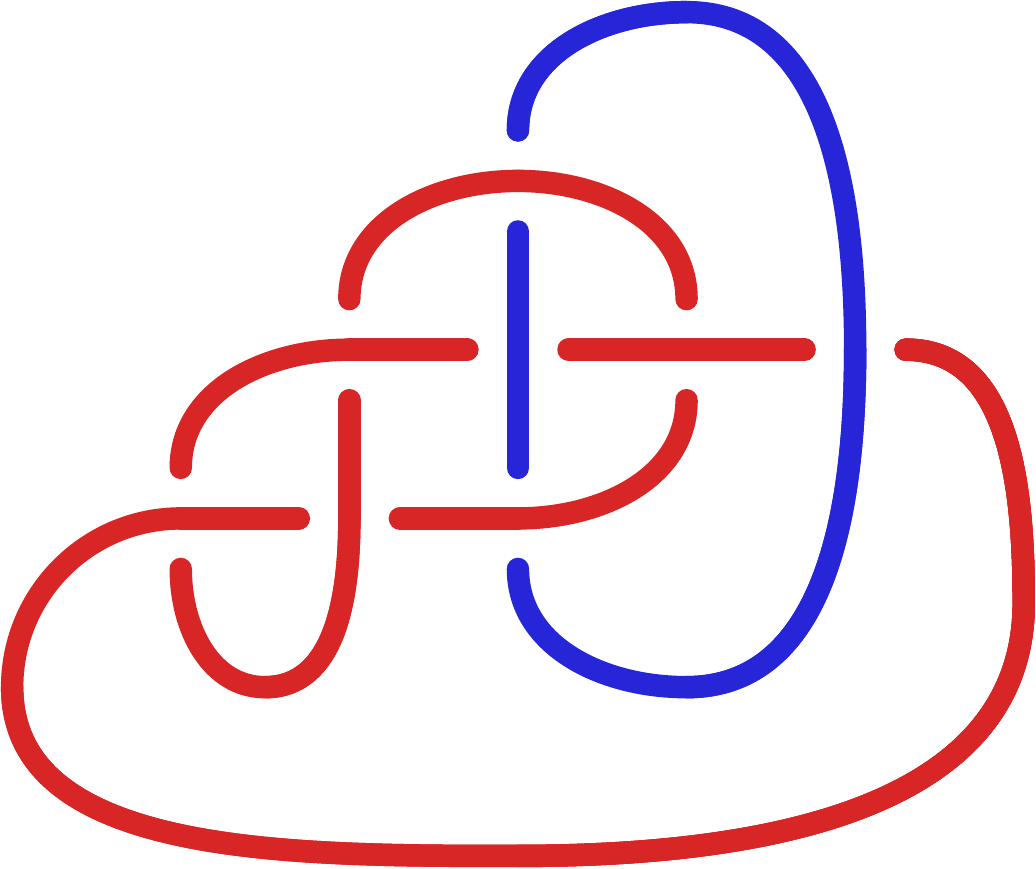} & \\ 
 \quad & & \quad & \\ 
 \hline  
\quad & \multirow{6}{*}{\Includegraphics[width=1.8in]{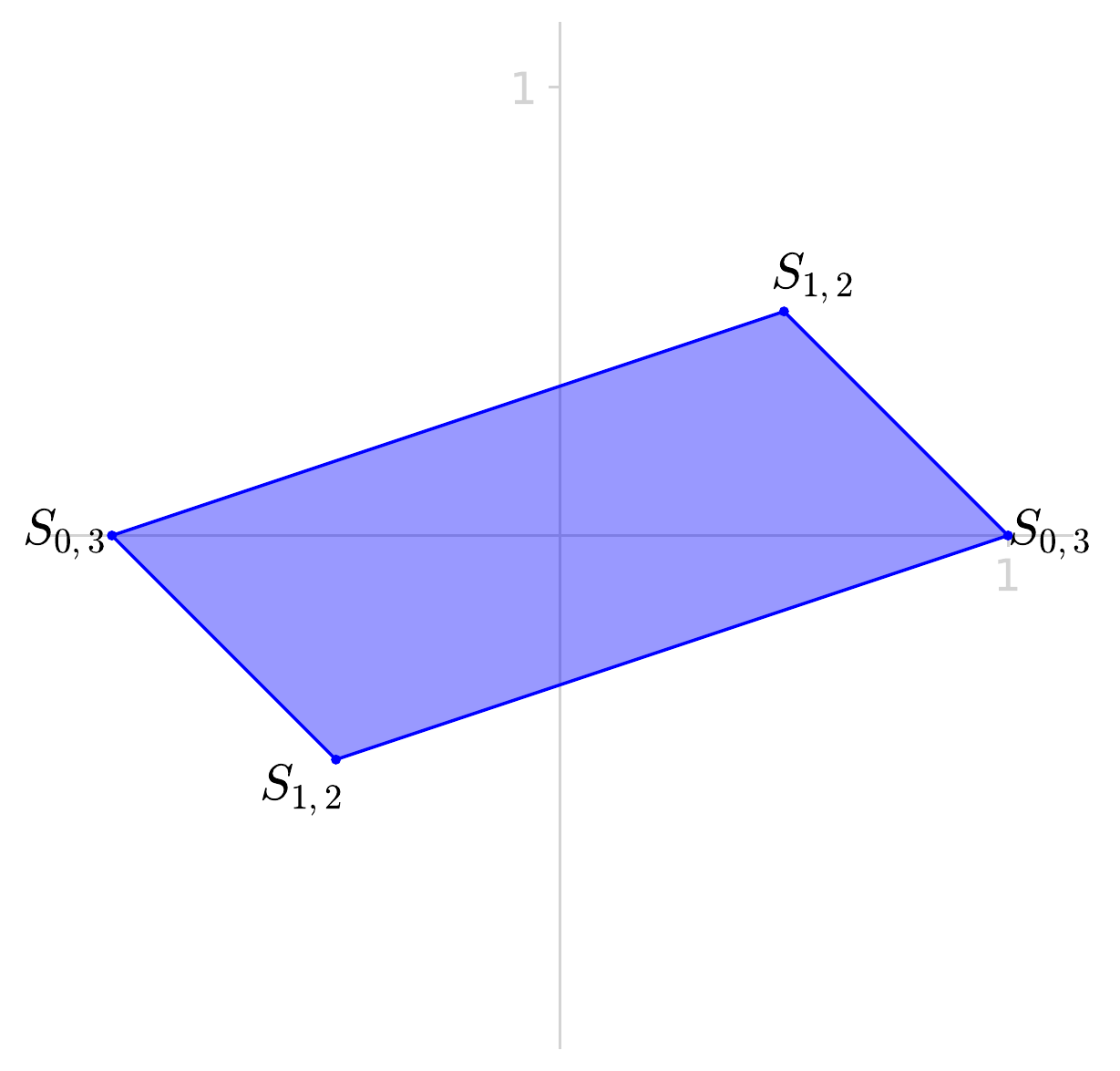}} & \quad & \multirow{6}{*}{\Includegraphics[width=1.8in]{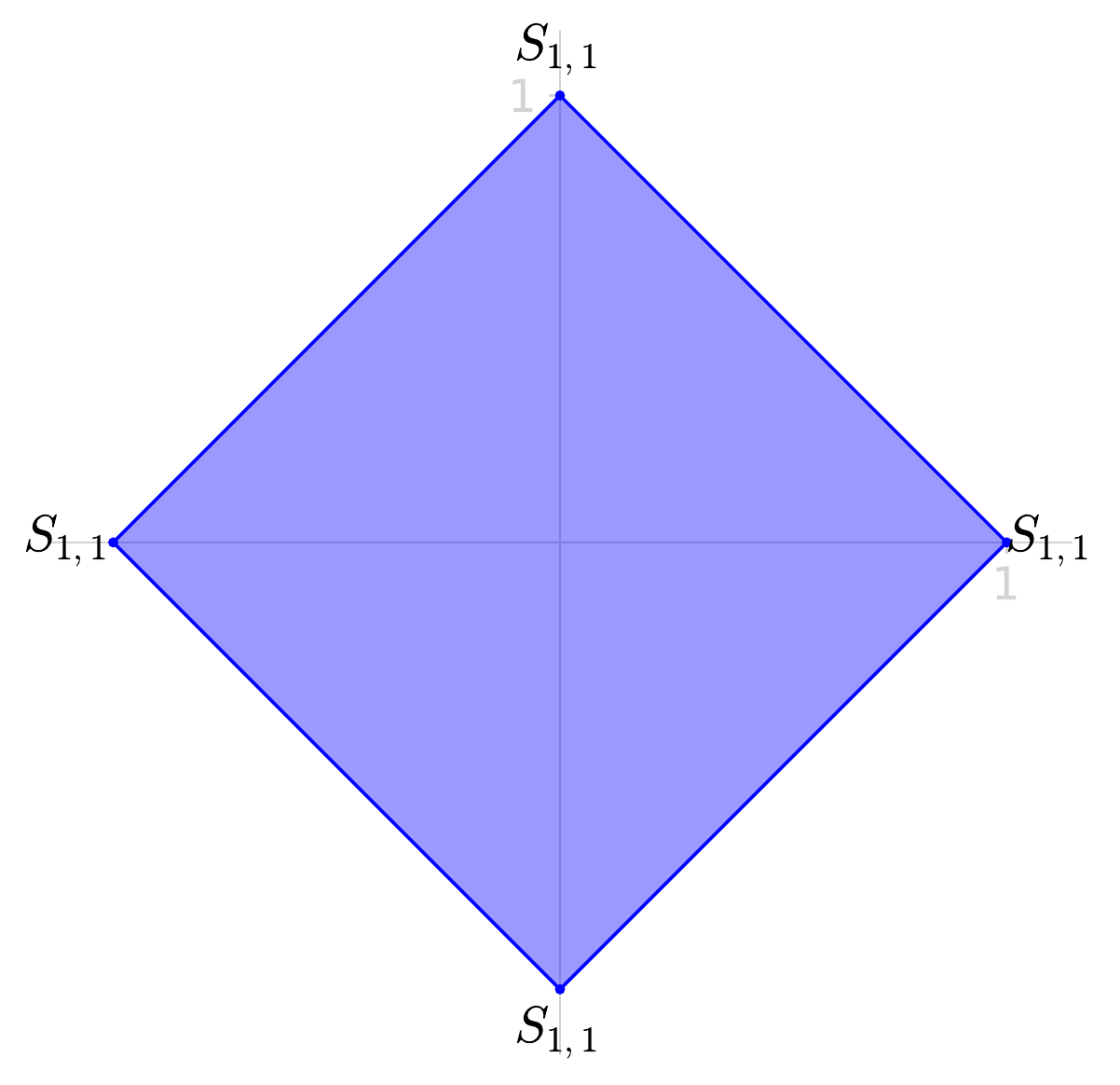}} \\ 
 $L=8^{{2}}_{{14}}$ & & $L=8^{{2}}_{{15}}$ & \\ 
 \quad & & \quad & \\ $\mathrm{Isom}(\mathbb{S}^3\setminus L) = D_4$ & & $\mathrm{Isom}(\mathbb{S}^3\setminus L) = D_4$ & \\ 
 \quad & & \quad & \\ 
 \includegraphics[width=1in]{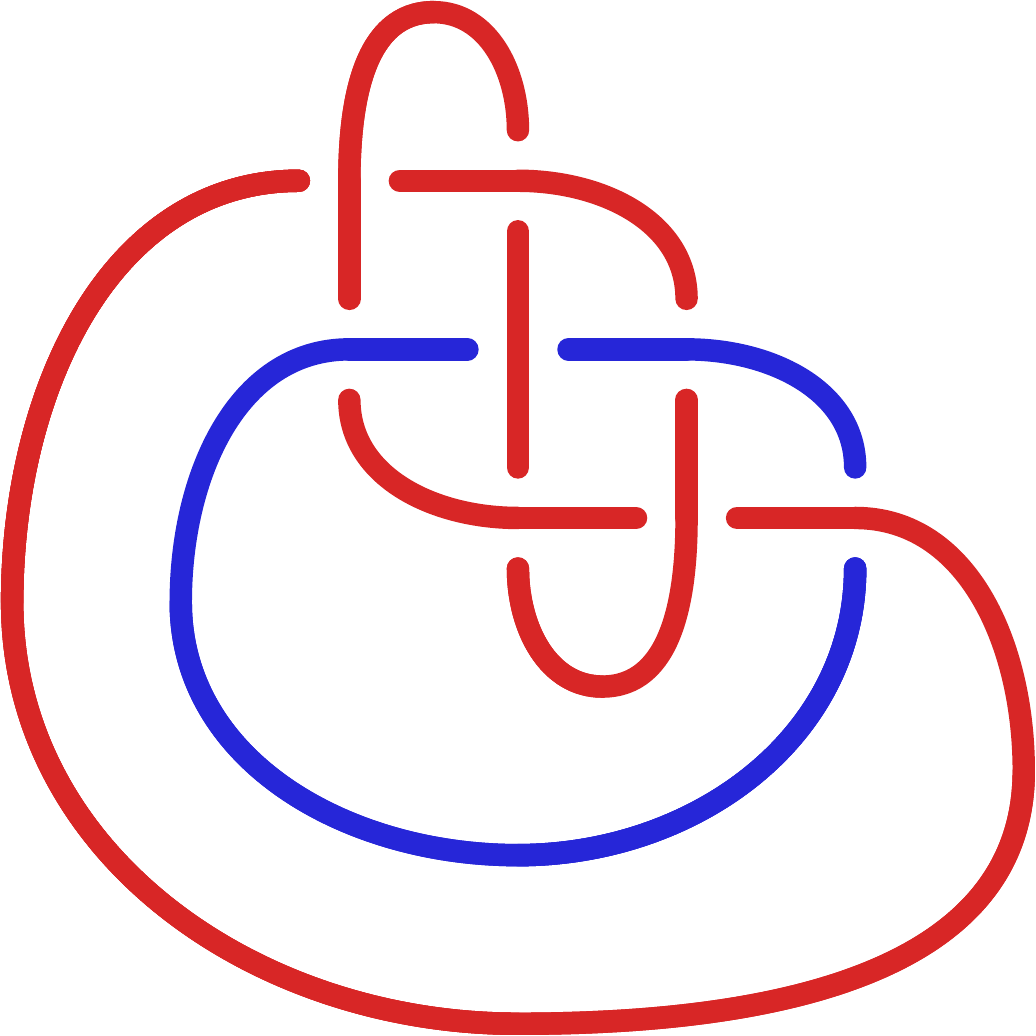}  & & \includegraphics[width=1in]{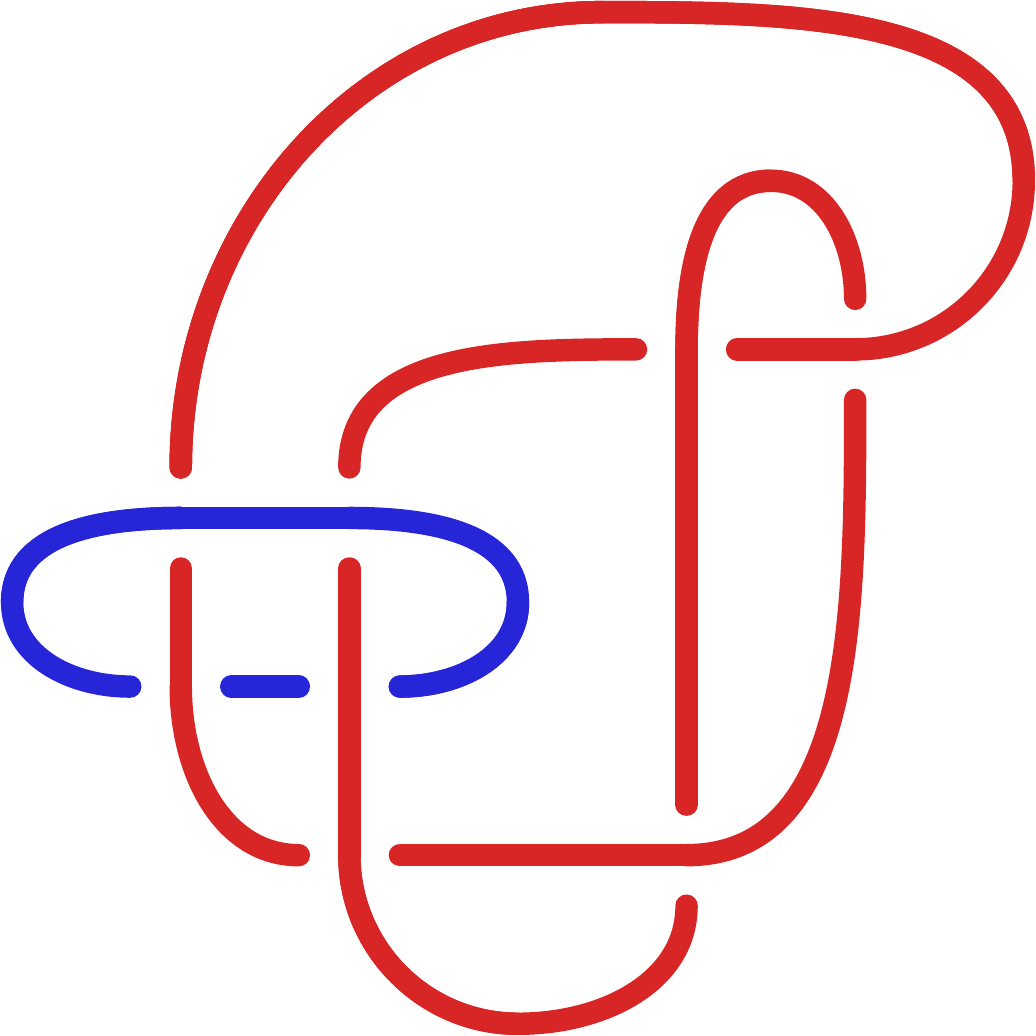} & \\ 
 \quad & & \quad & \\ 
 \hline  
\quad & \multirow{6}{*}{\Includegraphics[width=1.8in]{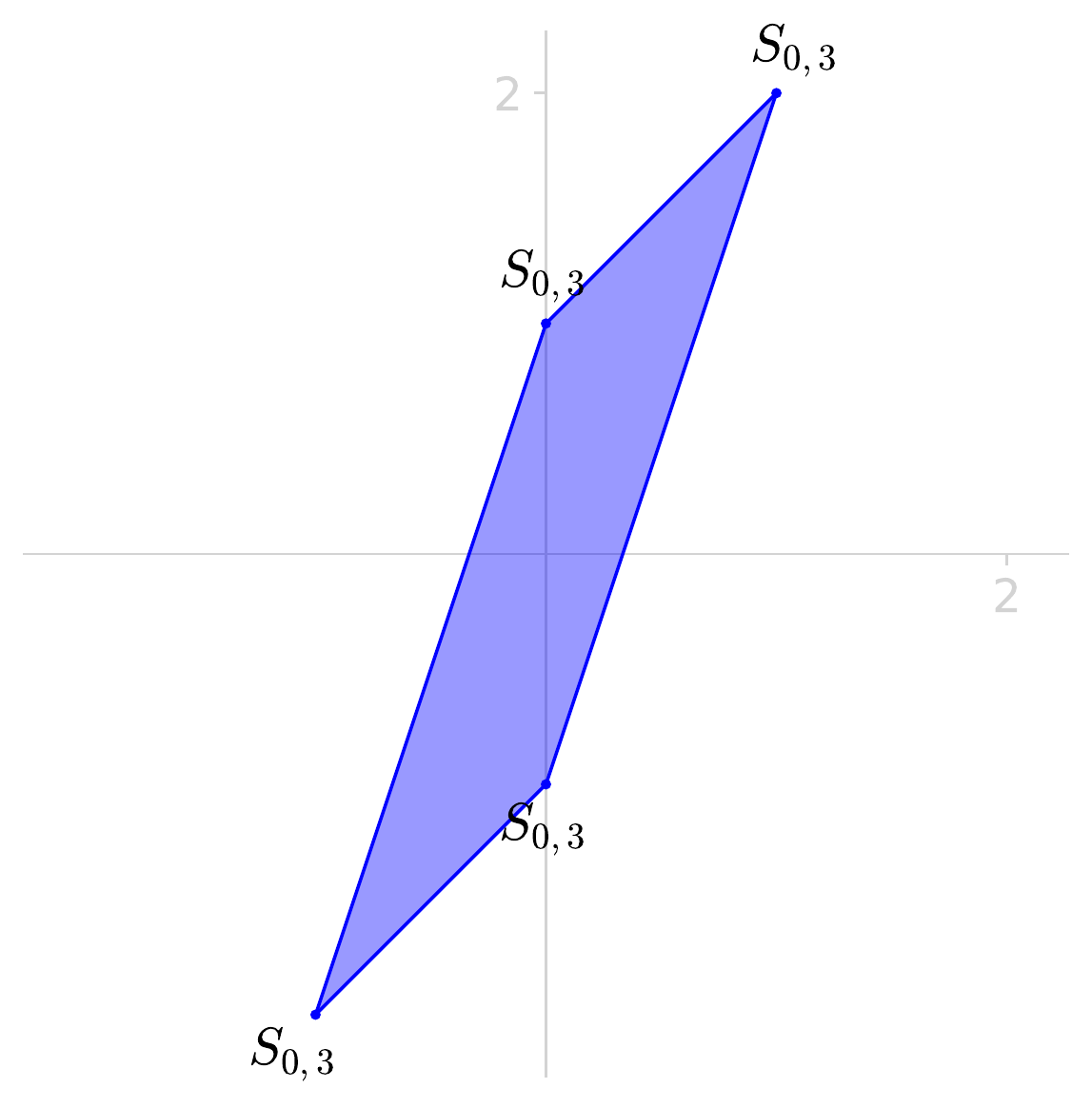}} & \quad & \multirow{6}{*}{\Includegraphics[width=1.8in]{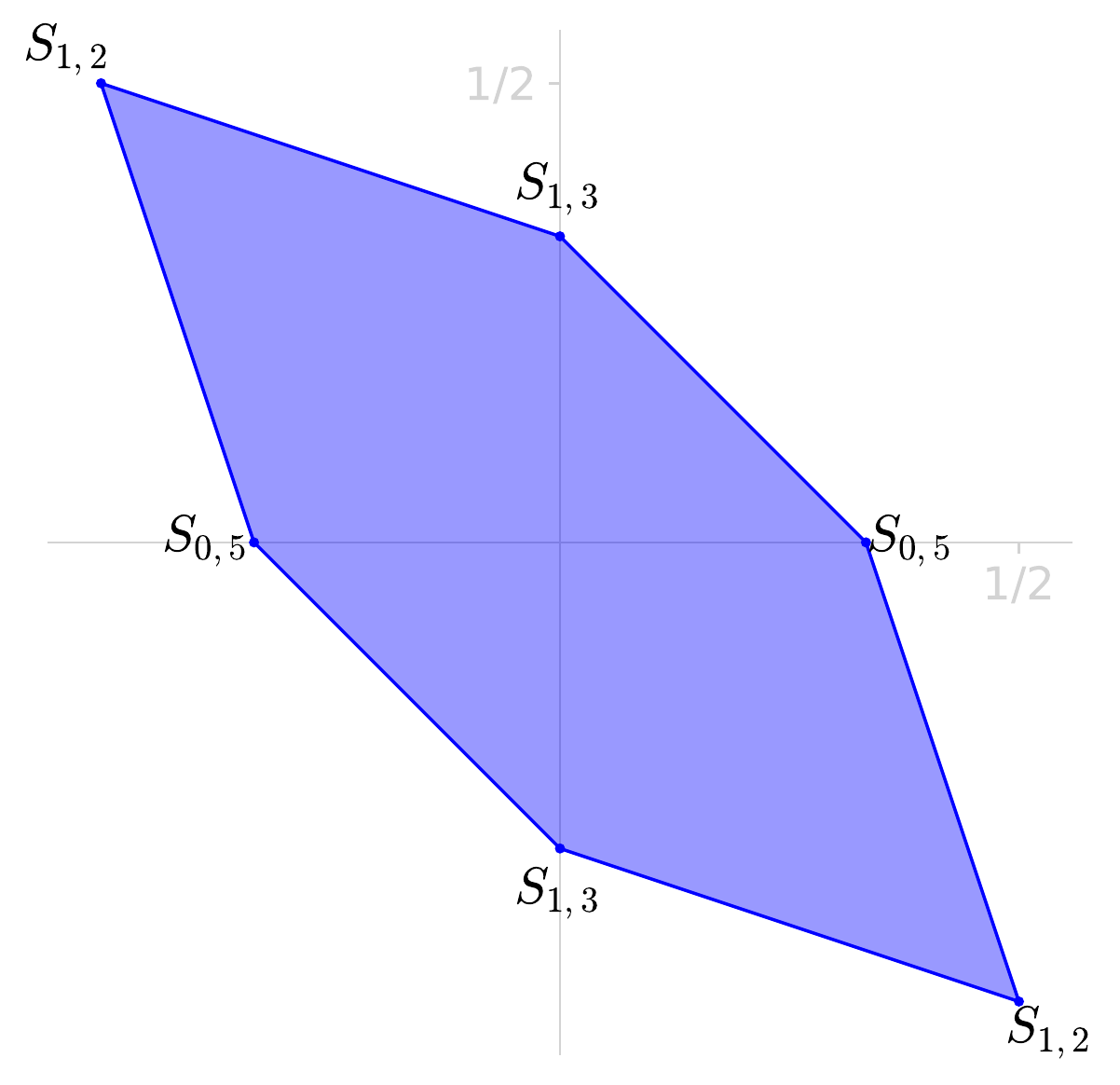}} \\ 
 $L=8^{{2}}_{{16}}$ & & $L=9^{{2}}_{{1}}$ & \\ 
 \quad & & \quad & \\ $\mathrm{Isom}(\mathbb{S}^3\setminus L) = \displaystyle\bigoplus_{i=1}^3 \mathbb{Z}$ & & $\mathrm{Isom}(\mathbb{S}^3\setminus L) = \mathbb{{Z}}_2\oplus\mathbb{{Z}}_2$ & \\ 
 \quad & & \quad & \\ 
 \includegraphics[width=1in]{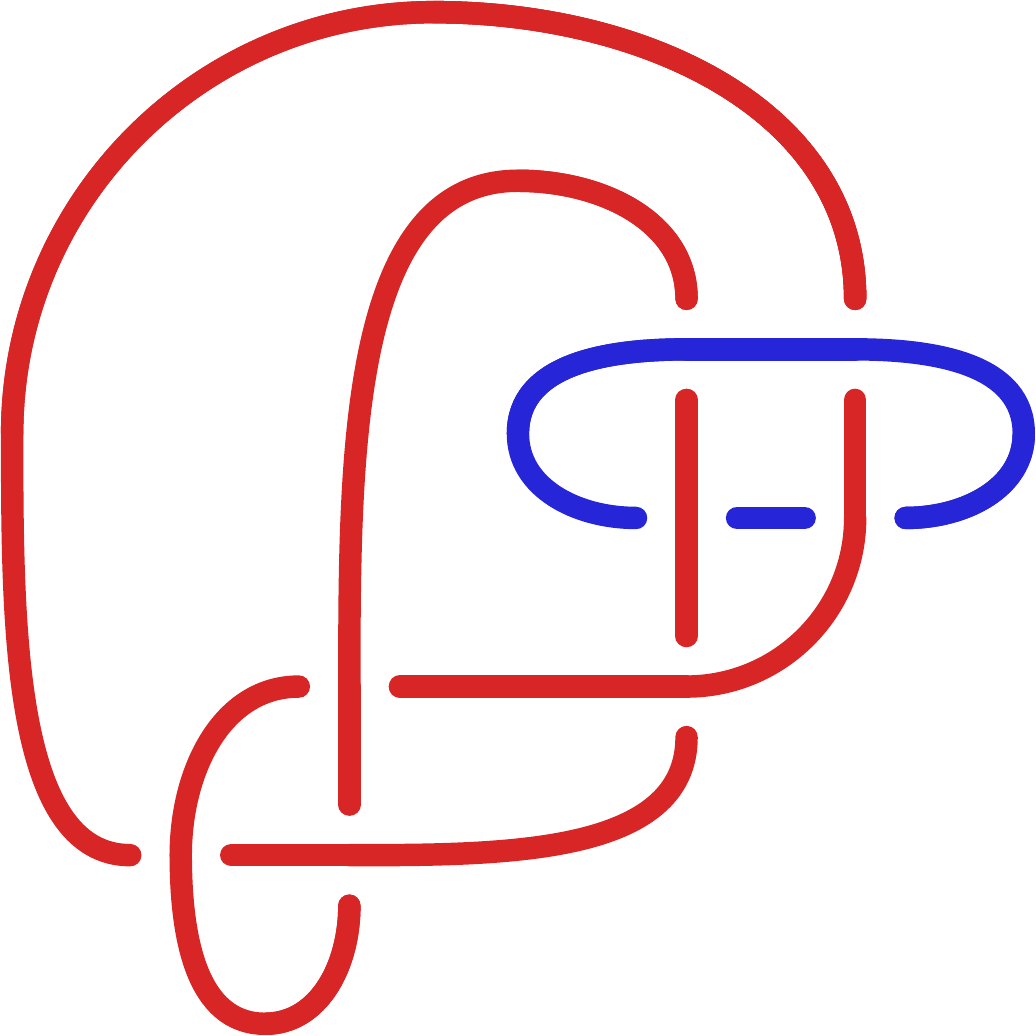}  & & \includegraphics[width=1in]{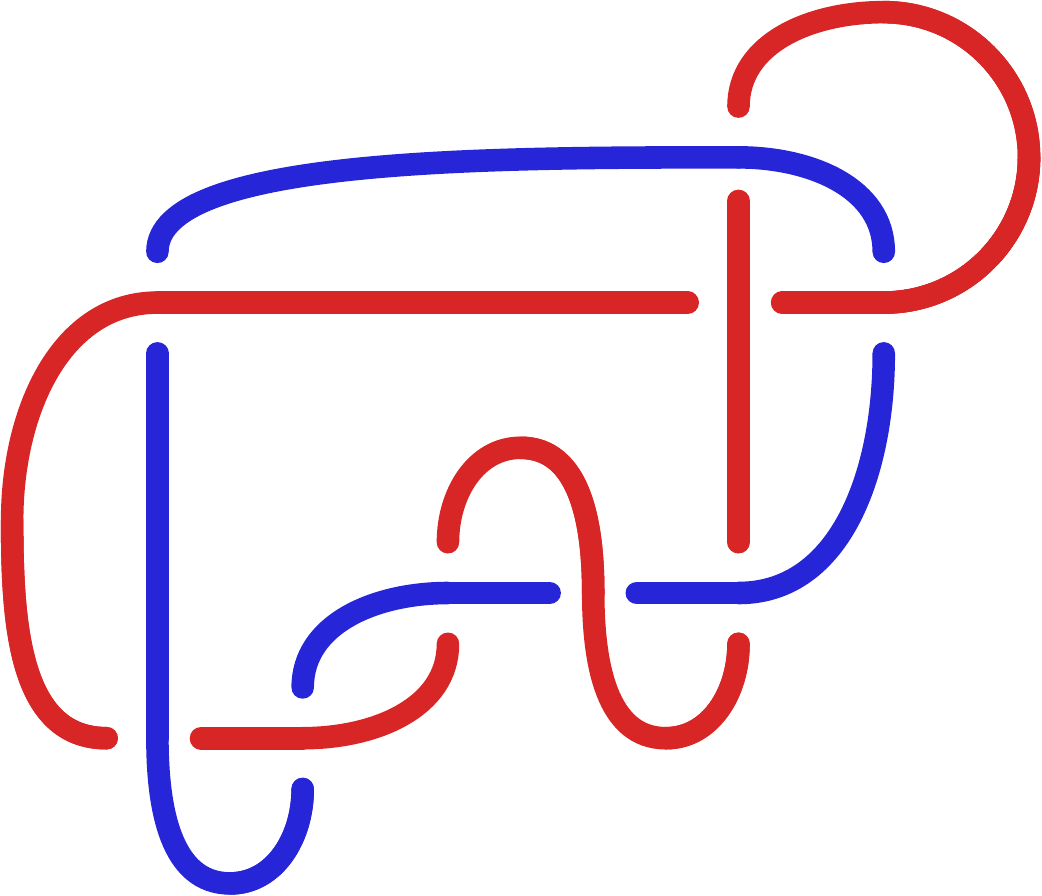} & \\ 
 \quad & & \quad & \\ 
 \hline  
\end{tabular} 
 \newpage \begin{tabular}{|c|c|c|c|} 
 \hline 
 Link & Norm Ball & Link & Norm Ball \\ 
 \hline 
\quad & \multirow{6}{*}{\Includegraphics[width=1.8in]{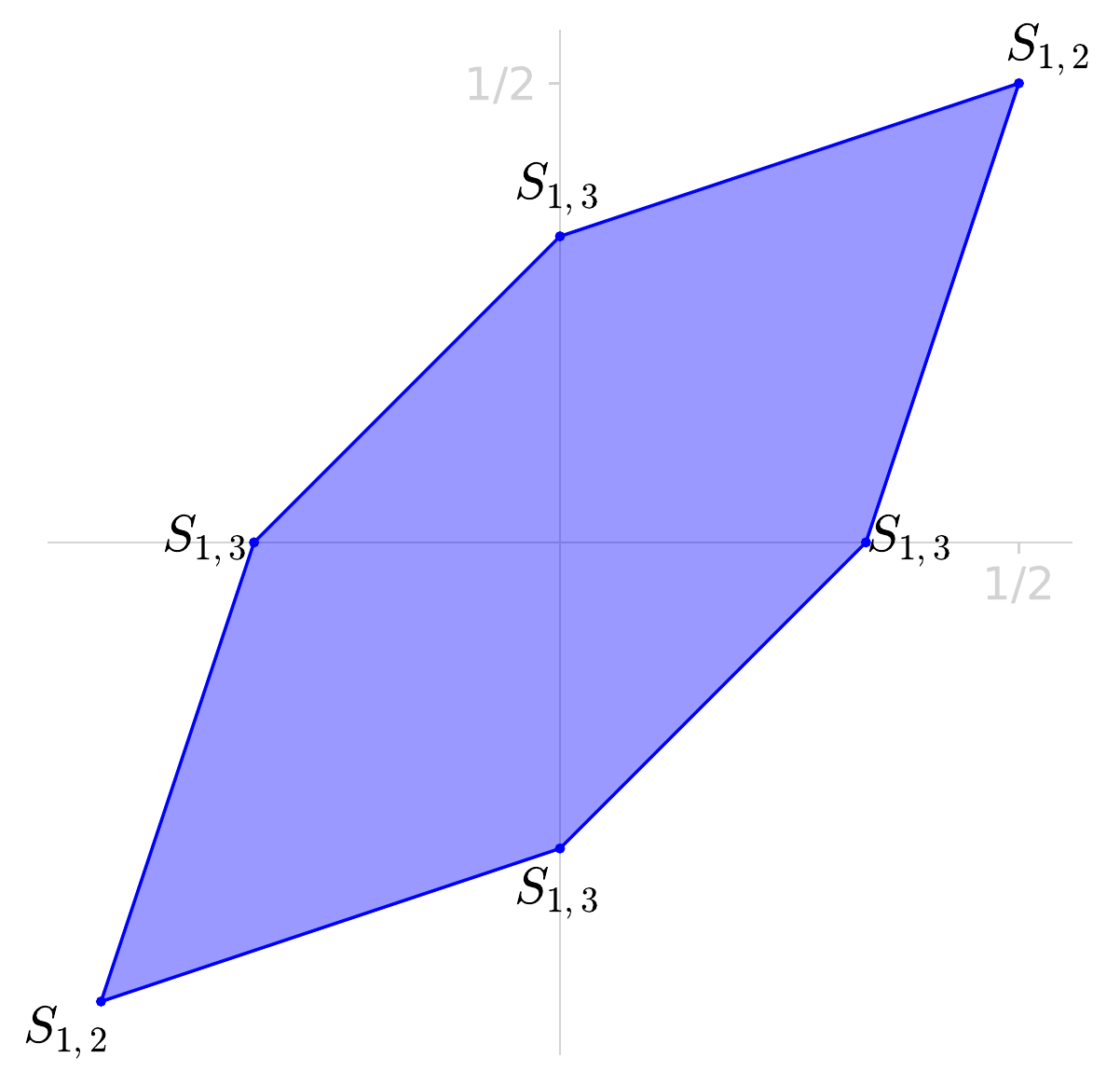}} & \quad & \multirow{6}{*}{\Includegraphics[width=1.8in]{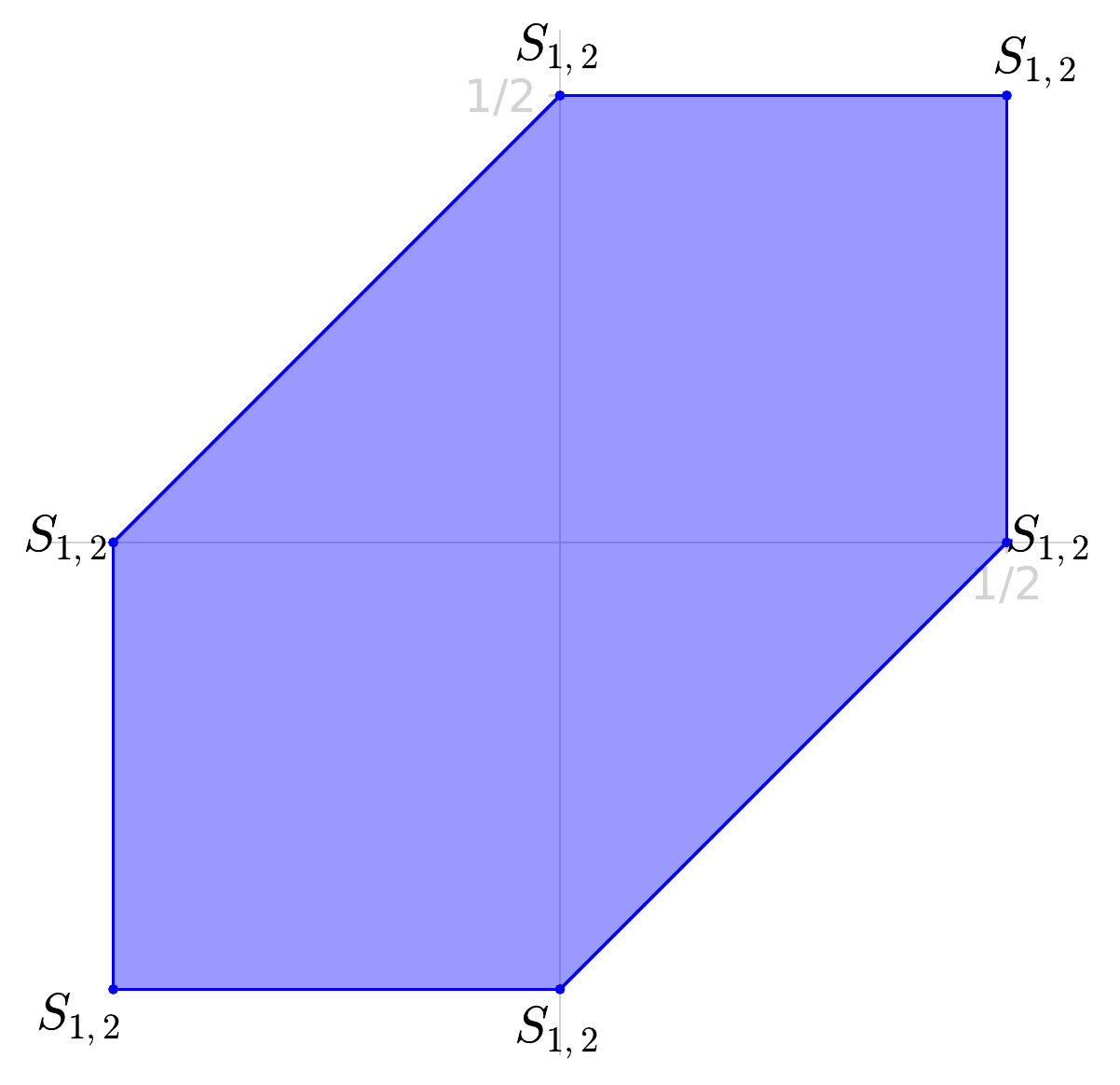}} \\ 
 $L=9^{{2}}_{{2}}$ & & $L=9^{{2}}_{{3}}$ & \\ 
 \quad & & \quad & \\ $\mathrm{Isom}(\mathbb{S}^3\setminus L) = \mathbb{{Z}}_2\oplus\mathbb{{Z}}_2$ & & $\mathrm{Isom}(\mathbb{S}^3\setminus L) = \mathbb{{Z}}_2\oplus\mathbb{{Z}}_2$ & \\ 
 \quad & & \quad & \\ 
 \includegraphics[width=1in]{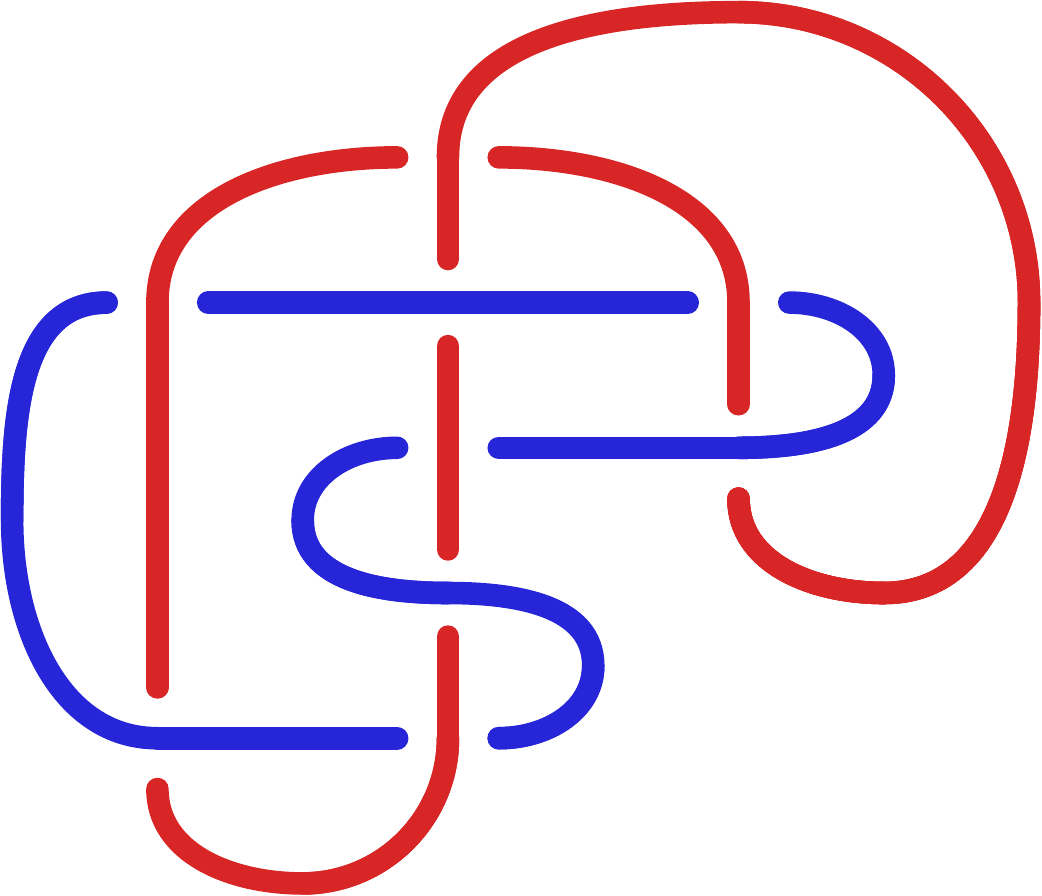}  & & \includegraphics[width=1in]{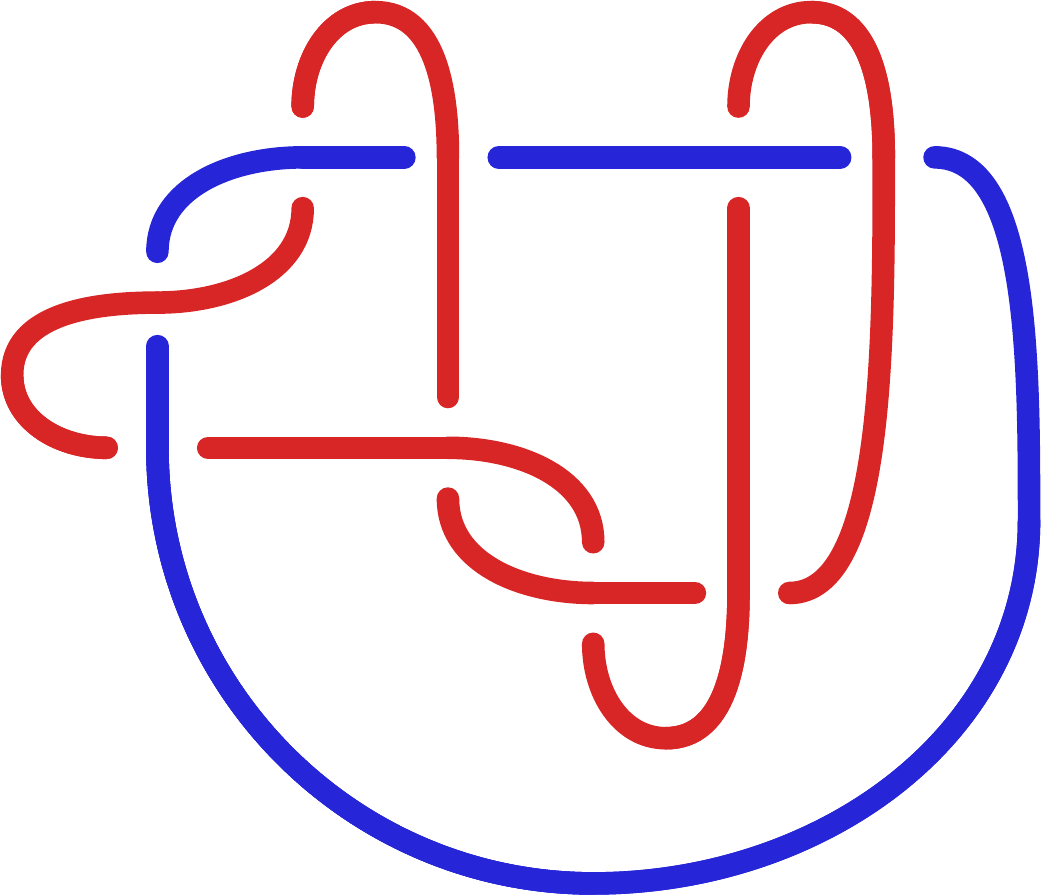} & \\ 
 \quad & & \quad & \\ 
 \hline  
\quad & \multirow{6}{*}{\Includegraphics[width=1.8in]{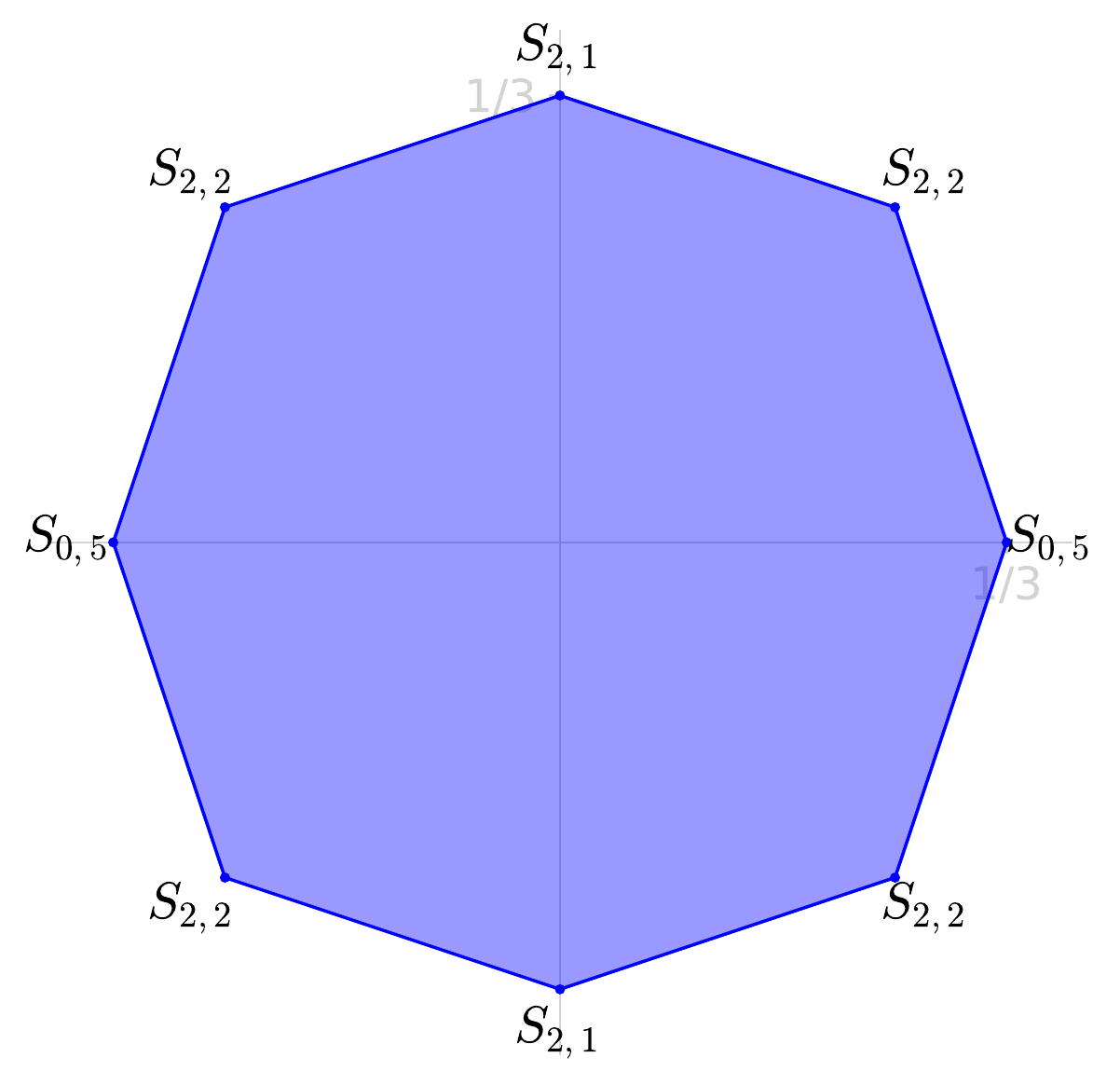}} & \quad & \multirow{6}{*}{\Includegraphics[width=1.8in]{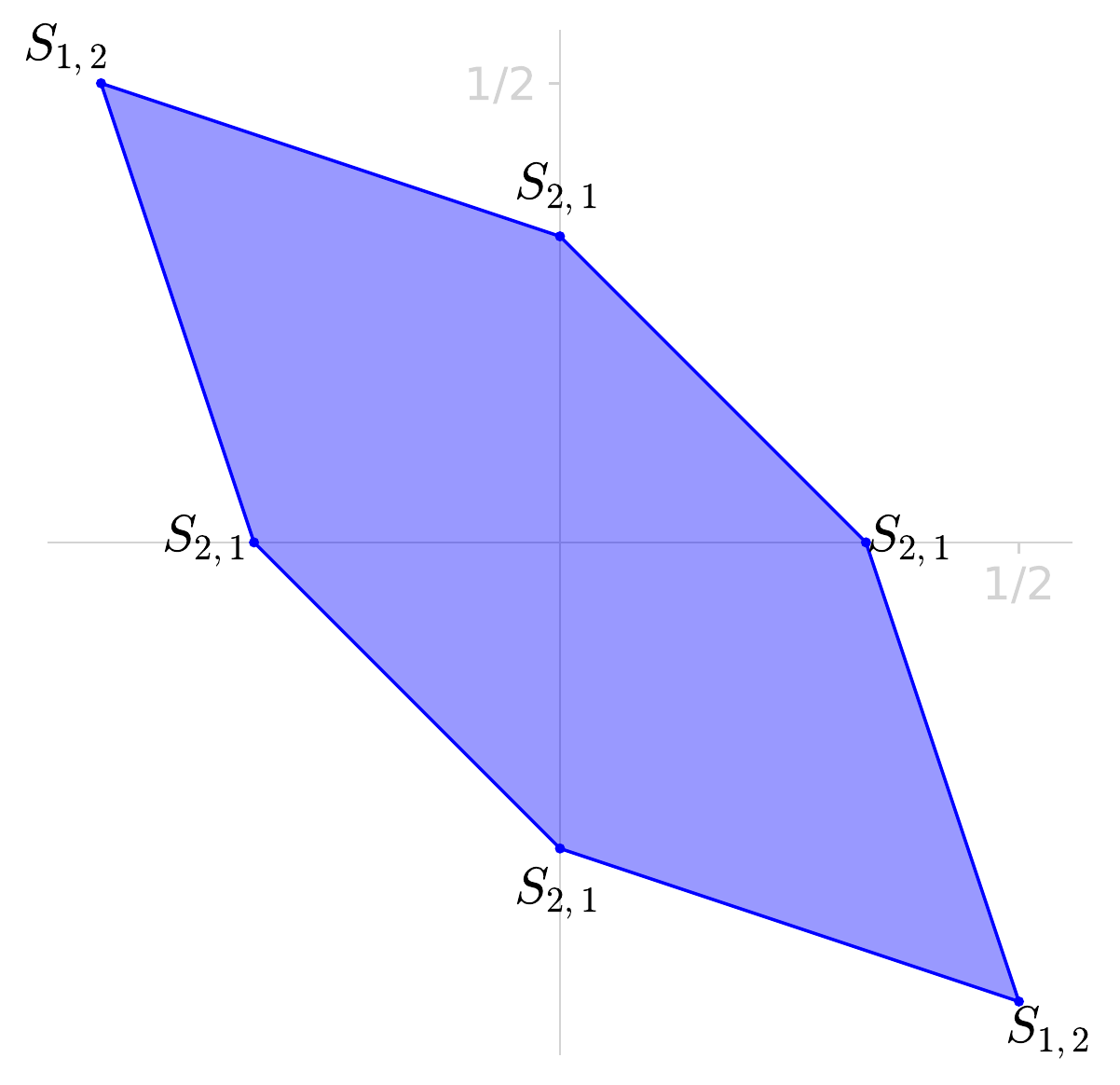}} \\ 
 $L=9^{{2}}_{{4}}$ & & $L=9^{{2}}_{{5}}$ & \\ 
 \quad & & \quad & \\ $\mathrm{Isom}(\mathbb{S}^3\setminus L) = D_4$ & & $\mathrm{Isom}(\mathbb{S}^3\setminus L) = \mathbb{{Z}}_2\oplus\mathbb{{Z}}_2$ & \\ 
 \quad & & \quad & \\ 
 \includegraphics[width=1in]{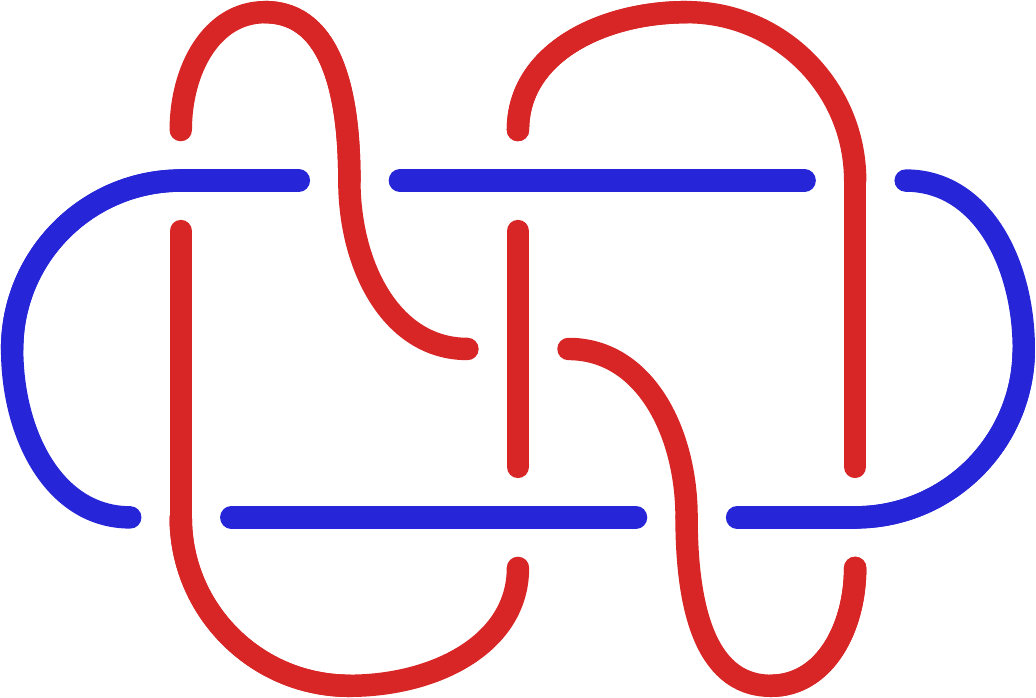}  & & \includegraphics[width=1in]{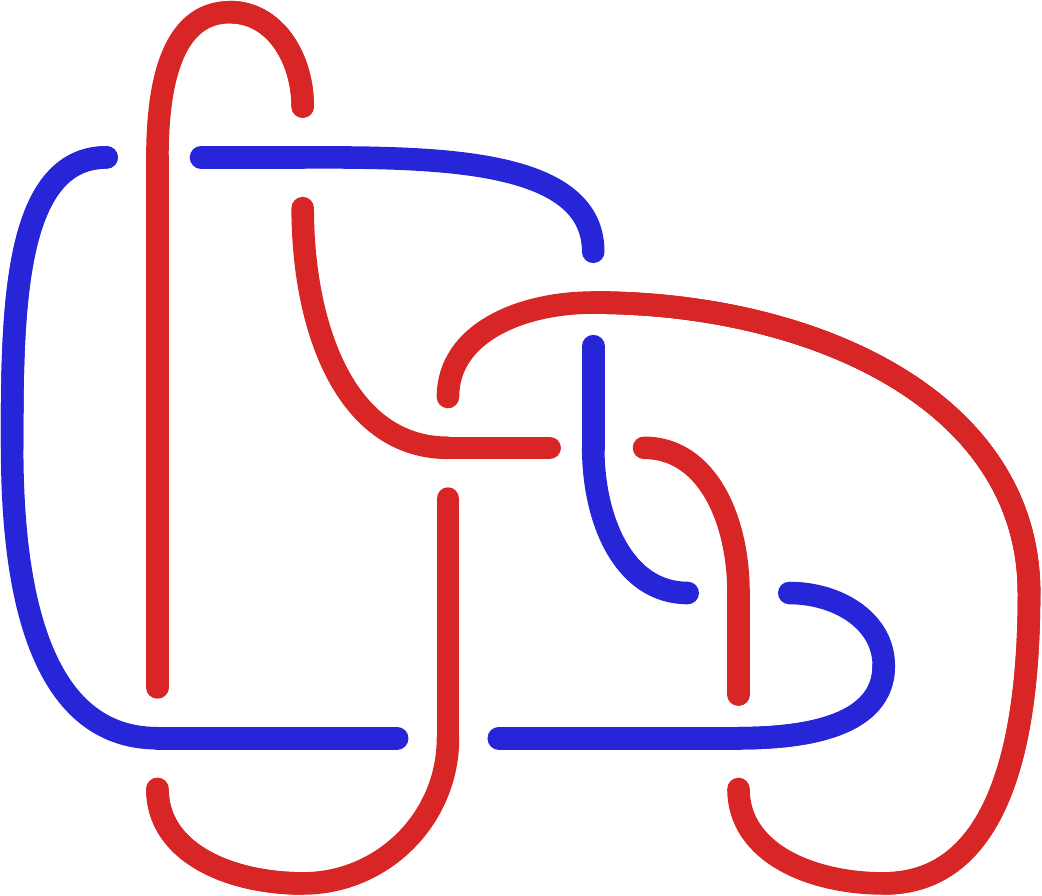} & \\ 
 \quad & & \quad & \\ 
 \hline  
\quad & \multirow{6}{*}{\Includegraphics[width=1.8in]{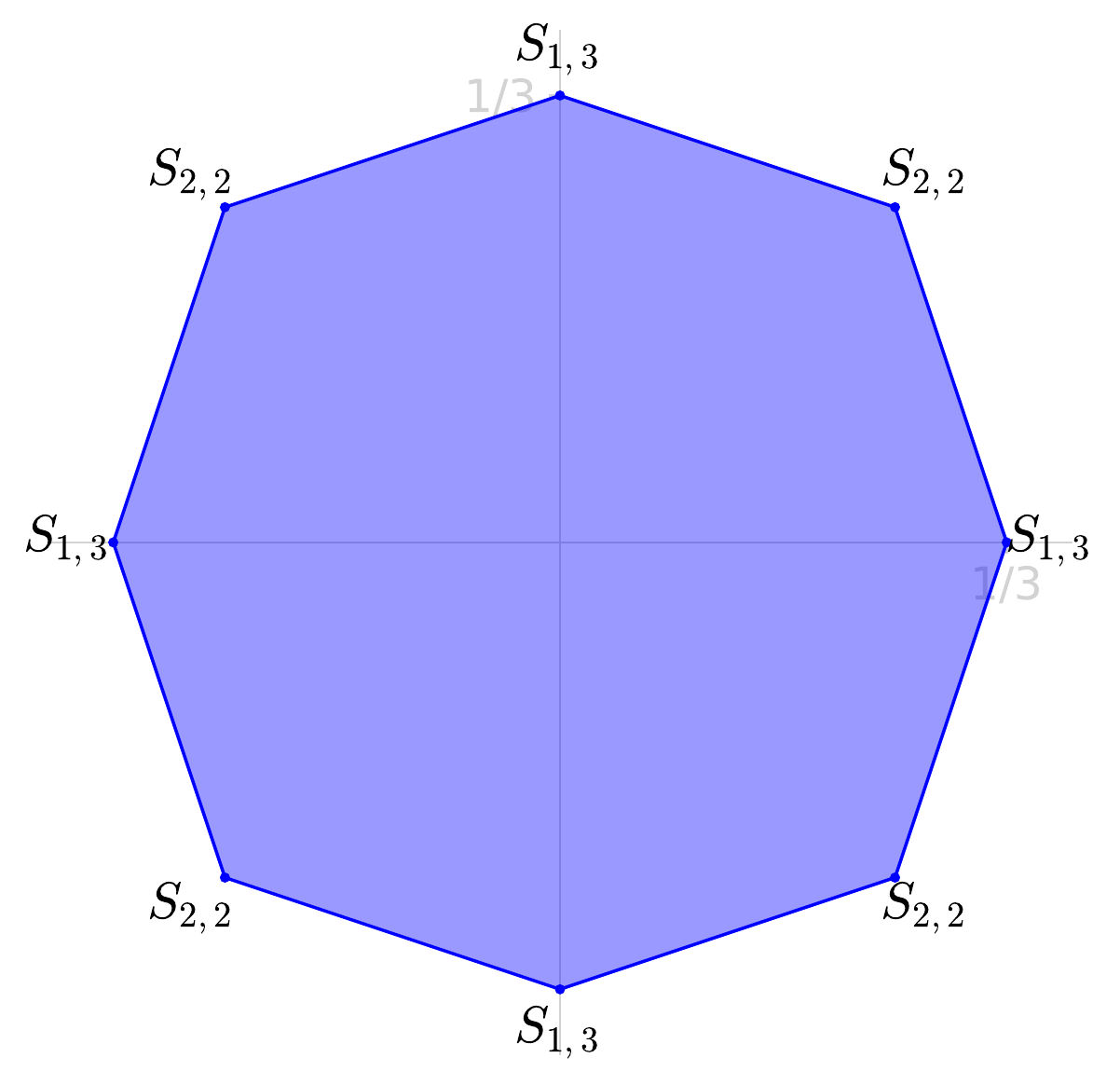}} & \quad & \multirow{6}{*}{\Includegraphics[width=1.8in]{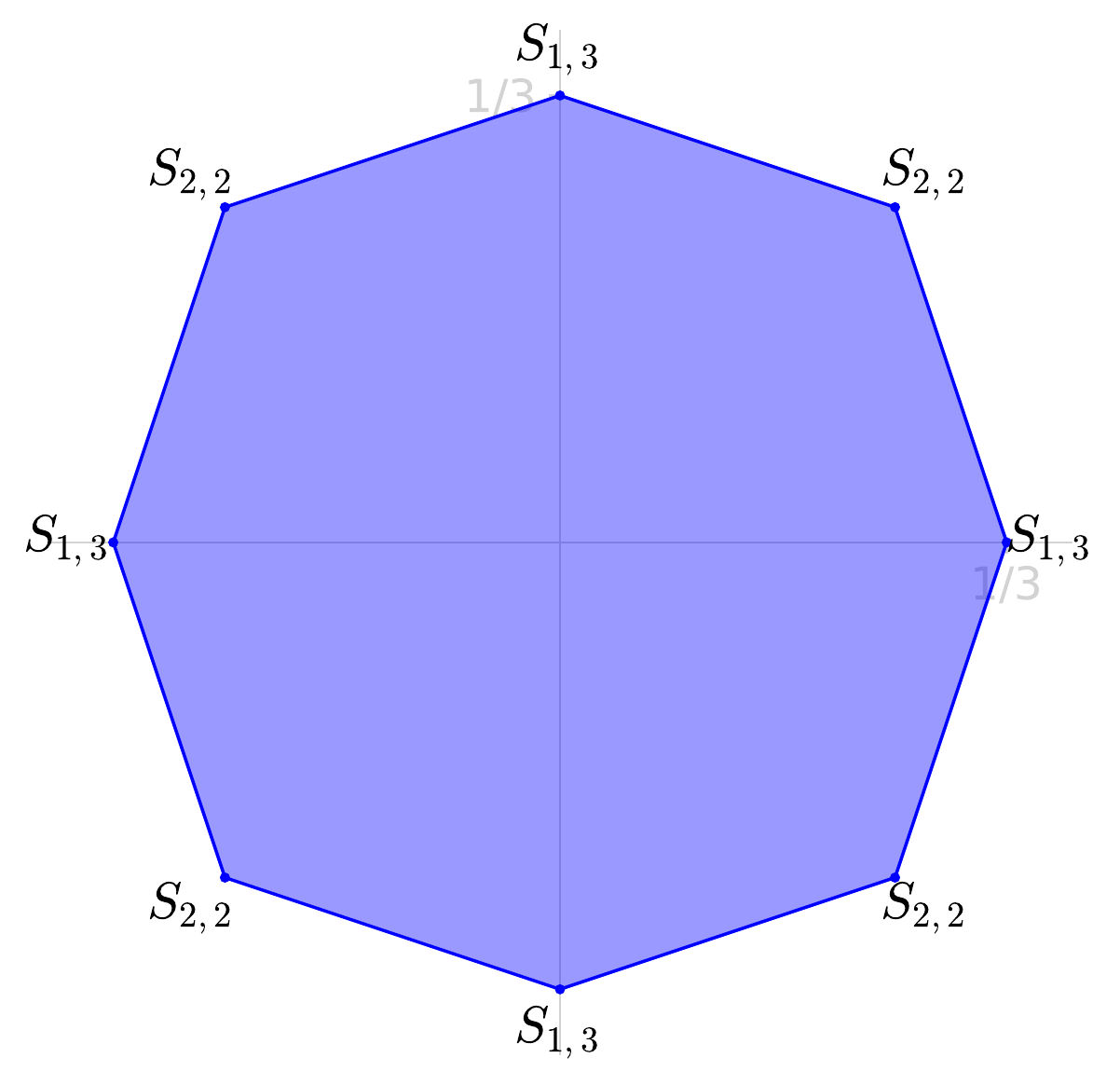}} \\ 
 $L=9^{{2}}_{{6}}$ & & $L=9^{{2}}_{{7}}$ & \\ 
 \quad & & \quad & \\ $\mathrm{Isom}(\mathbb{S}^3\setminus L) = \mathbb{{Z}}_2\oplus\mathbb{{Z}}_2$ & & $\mathrm{Isom}(\mathbb{S}^3\setminus L) = \mathbb{{Z}}_2\oplus\mathbb{{Z}}_2$ & \\ 
 \quad & & \quad & \\ 
 \includegraphics[width=1in]{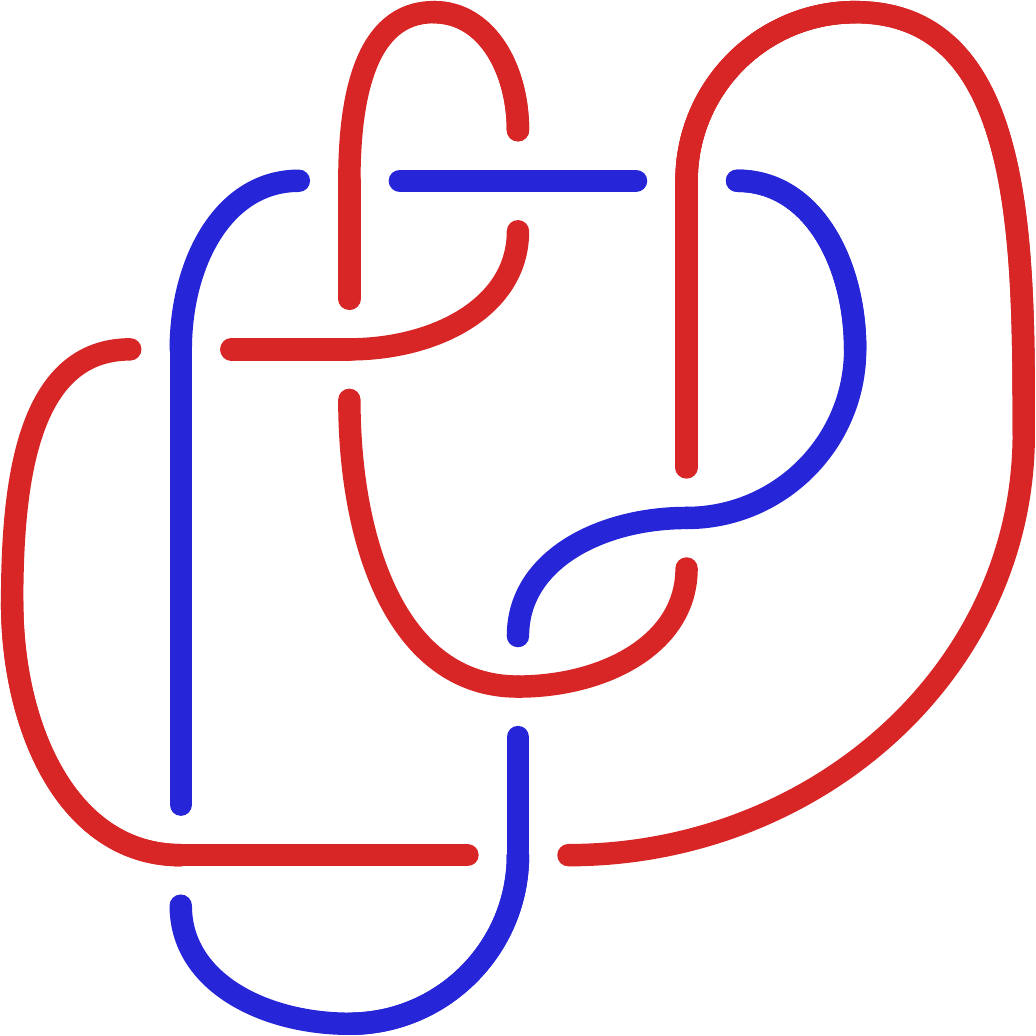}  & & \includegraphics[width=1in]{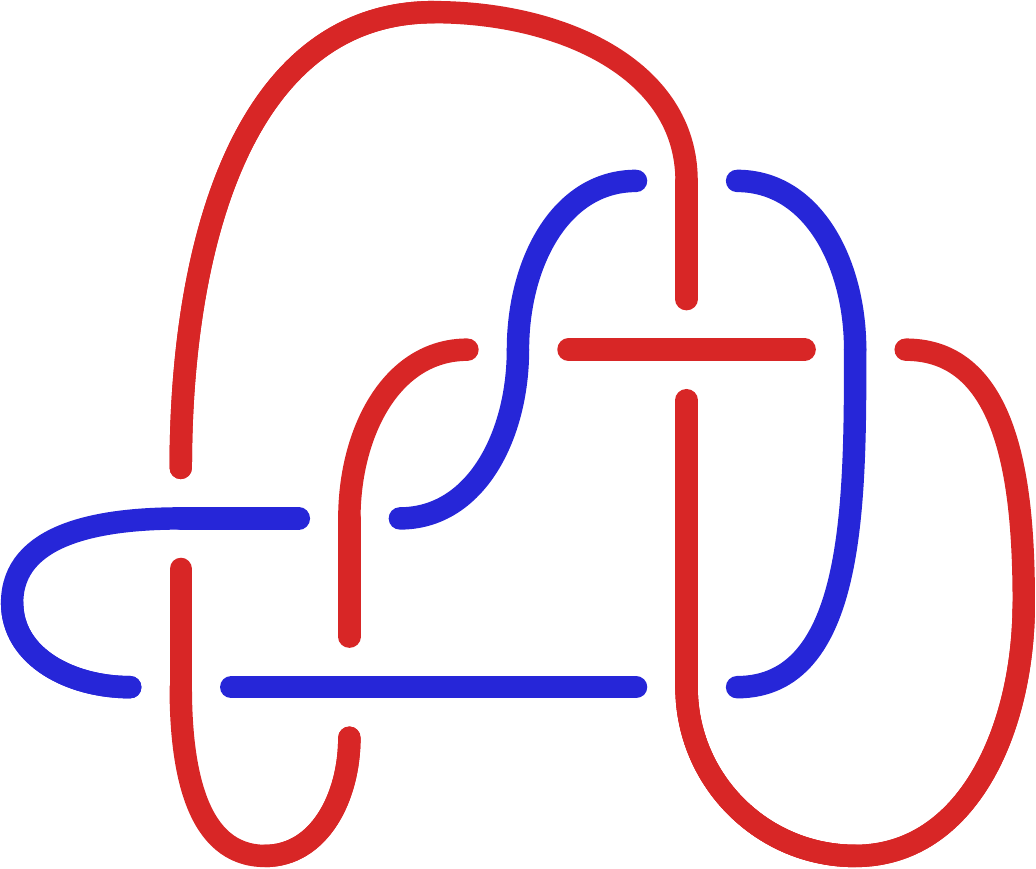} & \\ 
 \quad & & \quad & \\ 
 \hline  
\quad & \multirow{6}{*}{\Includegraphics[width=1.8in]{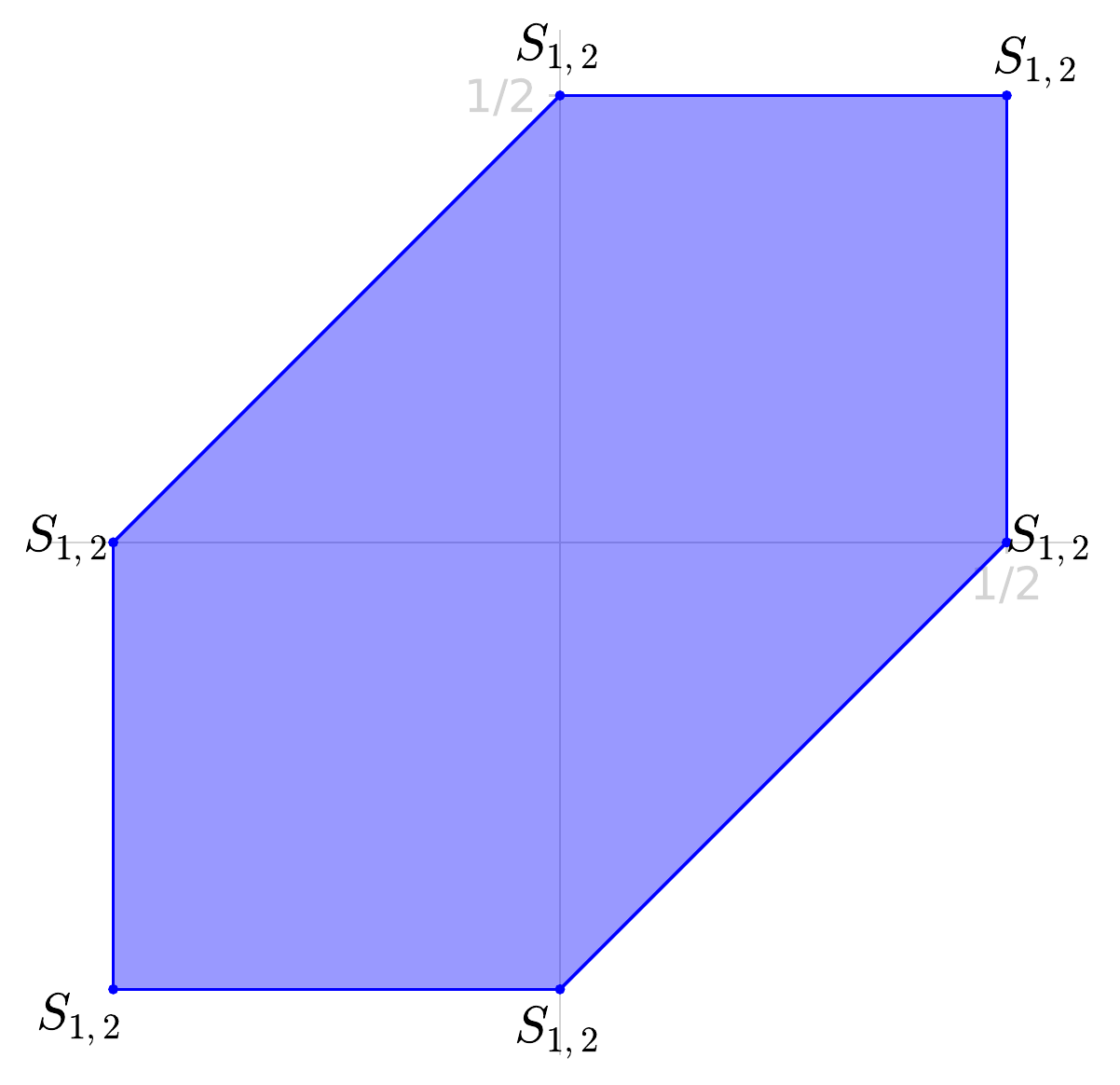}} & \quad & \multirow{6}{*}{\Includegraphics[width=1.8in]{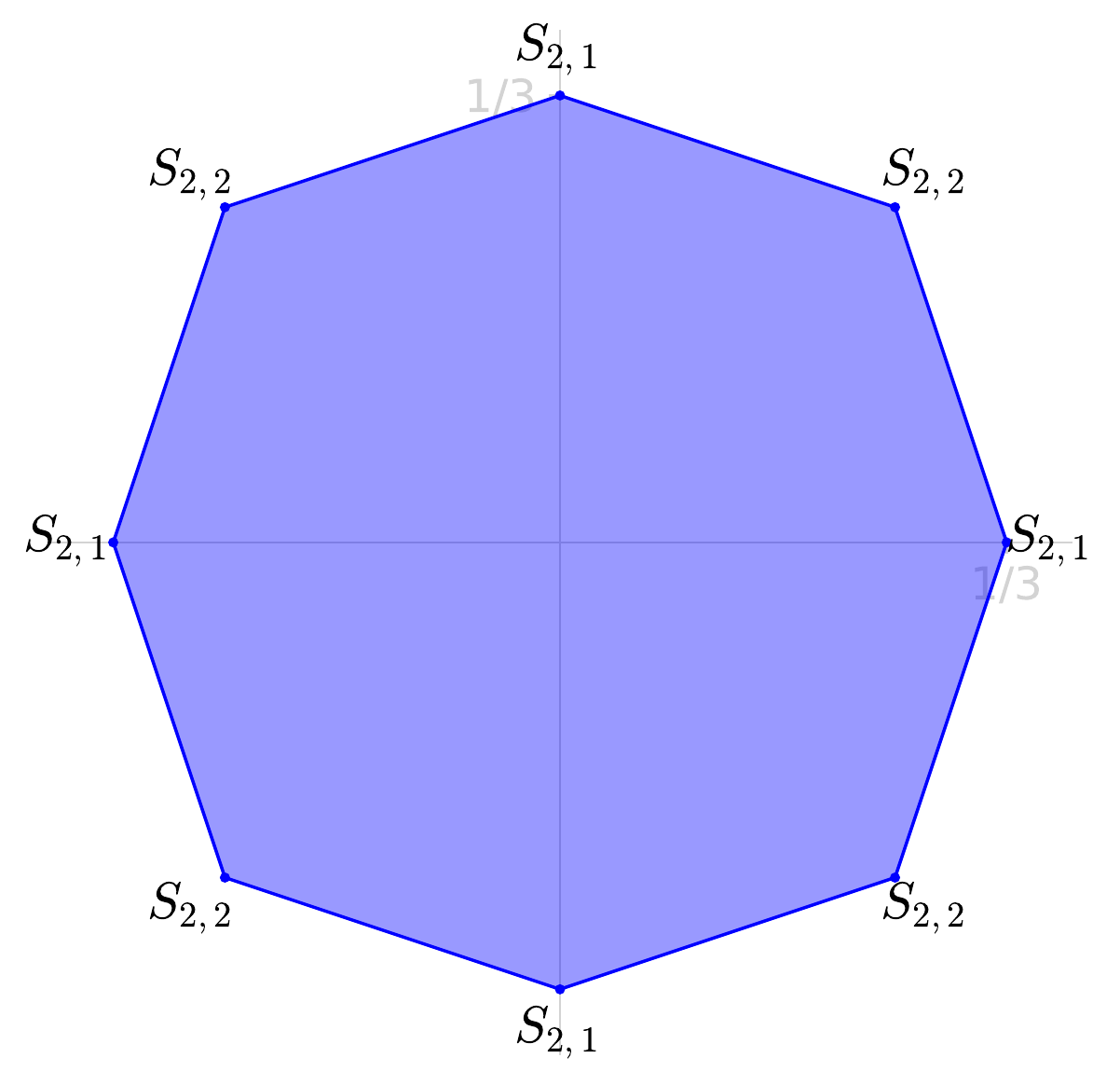}} \\ 
 $L=9^{{2}}_{{8}}$ & & $L=9^{{2}}_{{9}}$ & \\ 
 \quad & & \quad & \\ $\mathrm{Isom}(\mathbb{S}^3\setminus L) = \mathbb{{Z}}_2\oplus\mathbb{{Z}}_2$ & & $\mathrm{Isom}(\mathbb{S}^3\setminus L) = D_4$ & \\ 
 \quad & & \quad & \\ 
 \includegraphics[width=1in]{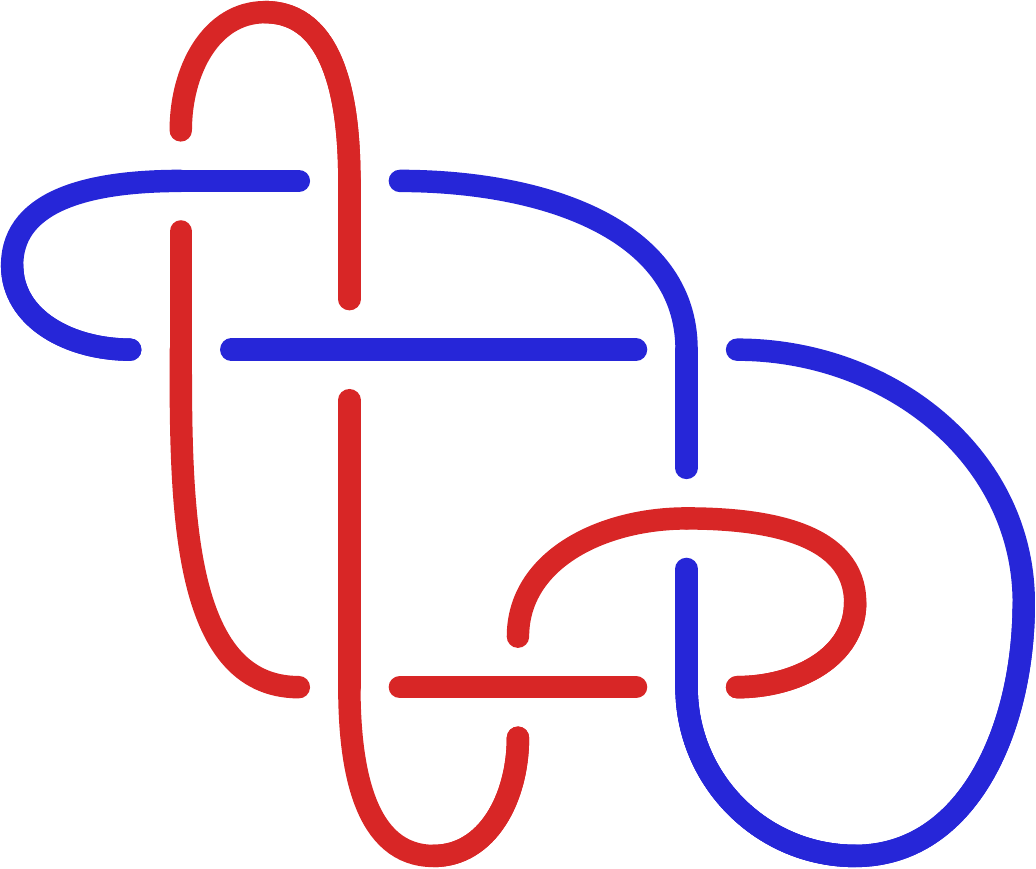}  & & \includegraphics[width=1in]{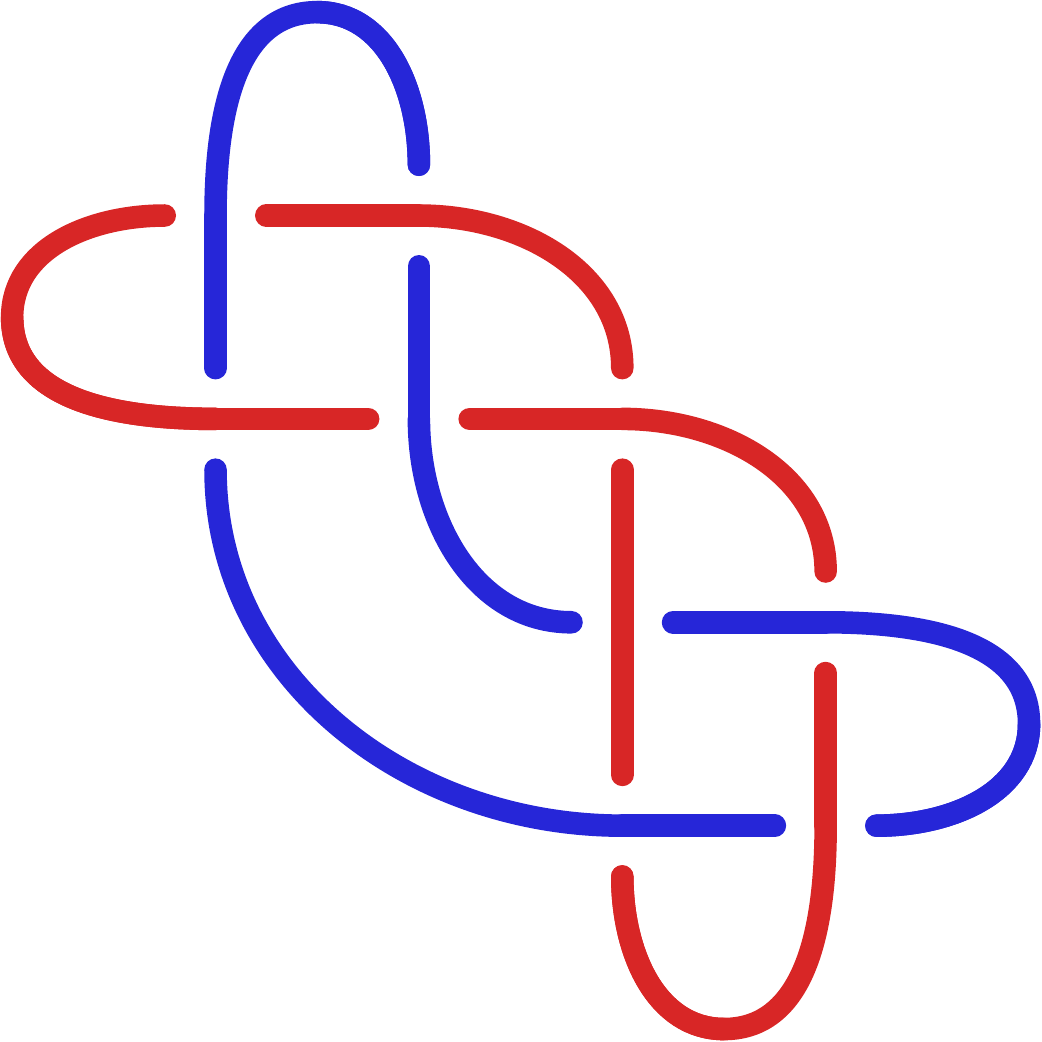} & \\ 
 \quad & & \quad & \\ 
 \hline  
\quad & \multirow{6}{*}{\Includegraphics[width=1.8in]{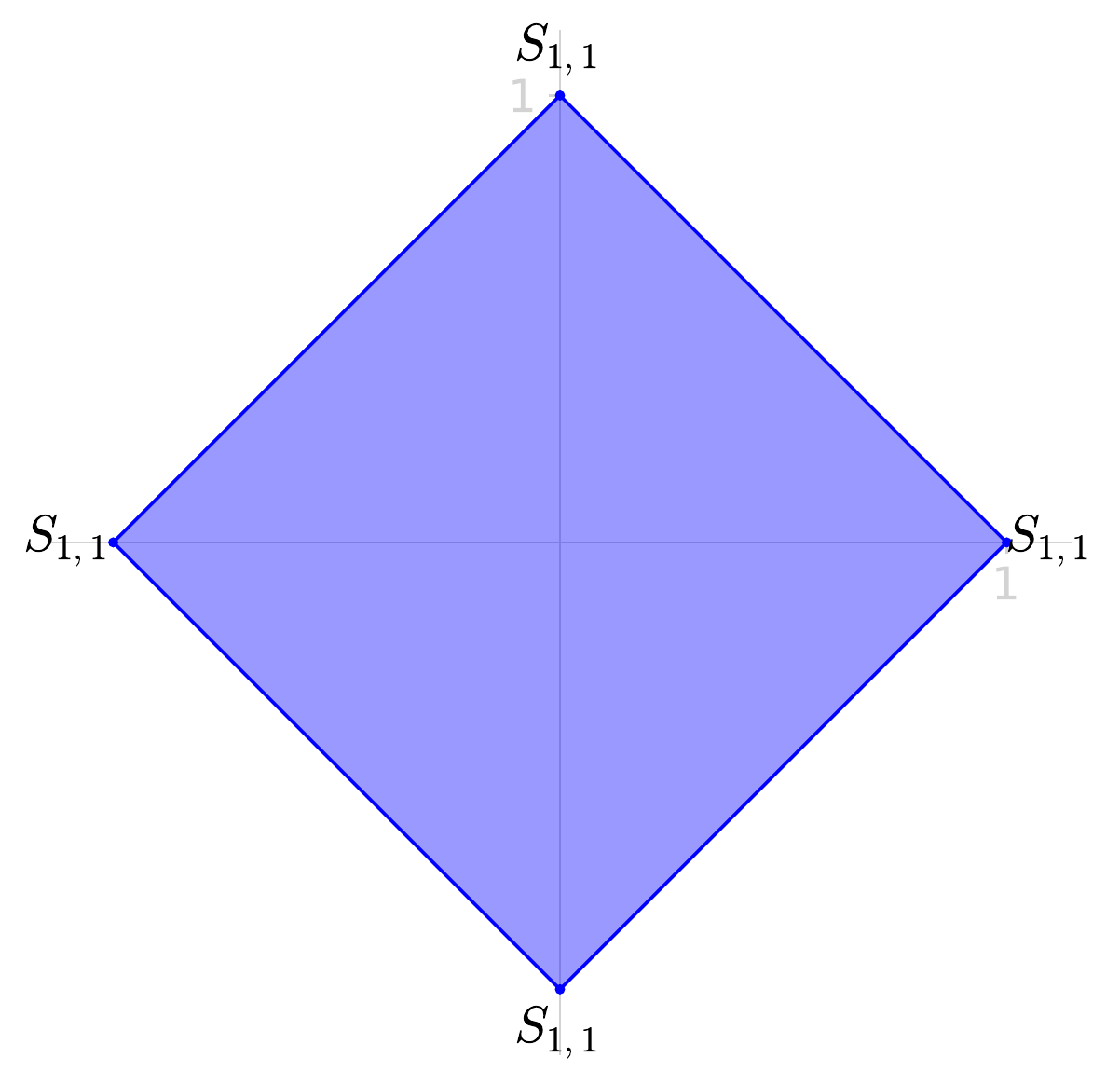}} & \quad & \multirow{6}{*}{\Includegraphics[width=1.8in]{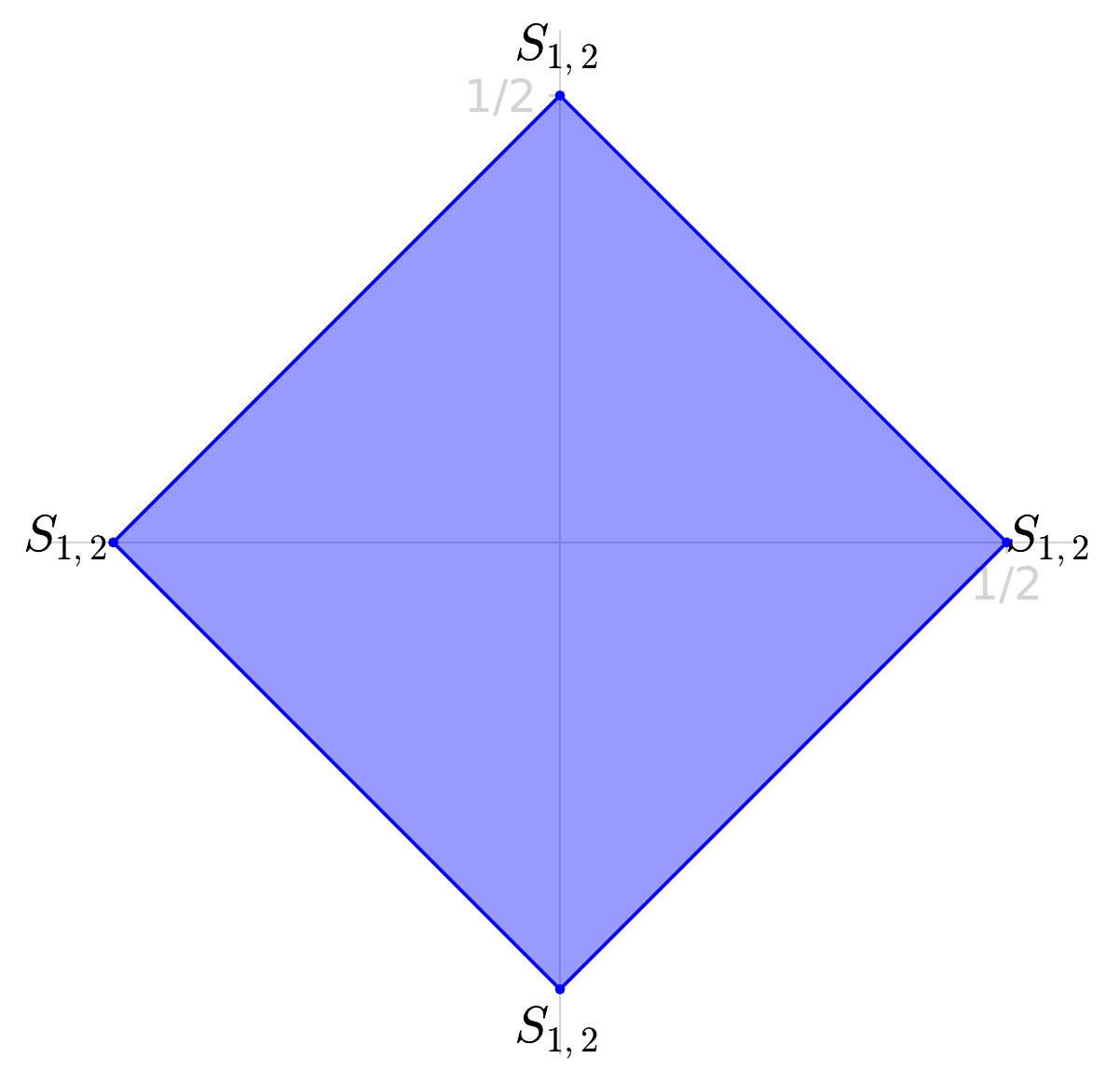}} \\ 
 $L=9^{{2}}_{{10}}$ & & $L=9^{{2}}_{{11}}$ & \\ 
 \quad & & \quad & \\ $\mathrm{Isom}(\mathbb{S}^3\setminus L) = D_4$ & & $\mathrm{Isom}(\mathbb{S}^3\setminus L) = \mathbb{{Z}}_2\oplus\mathbb{{Z}}_2$ & \\ 
 \quad & & \quad & \\ 
 \includegraphics[width=1in]{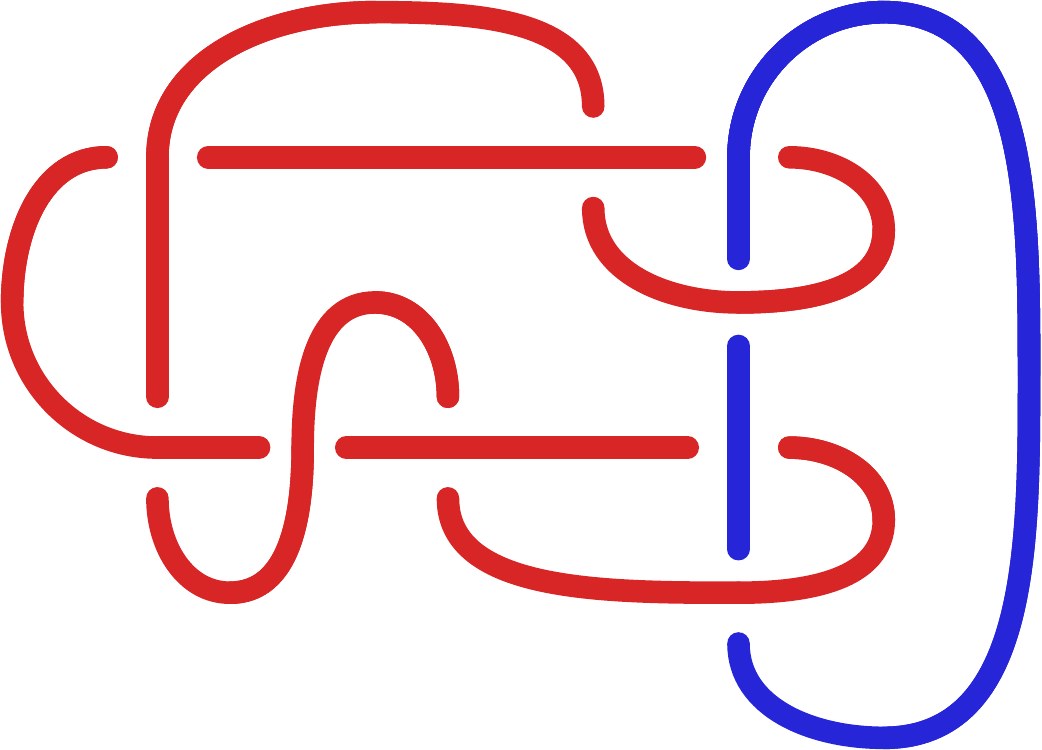}  & & \includegraphics[width=1in]{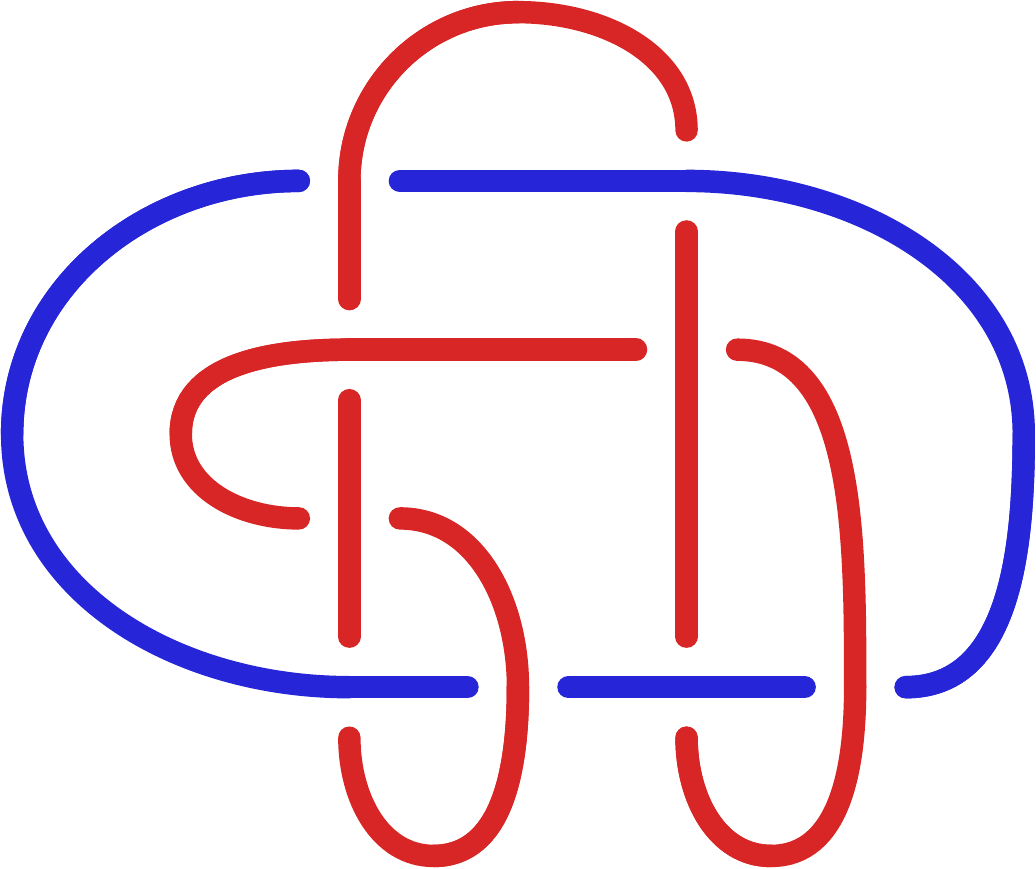} & \\ 
 \quad & & \quad & \\ 
 \hline  
\end{tabular} 
 \newpage \begin{tabular}{|c|c|c|c|} 
 \hline 
 Link & Norm Ball & Link & Norm Ball \\ 
 \hline 
\quad & \multirow{6}{*}{\Includegraphics[width=1.8in]{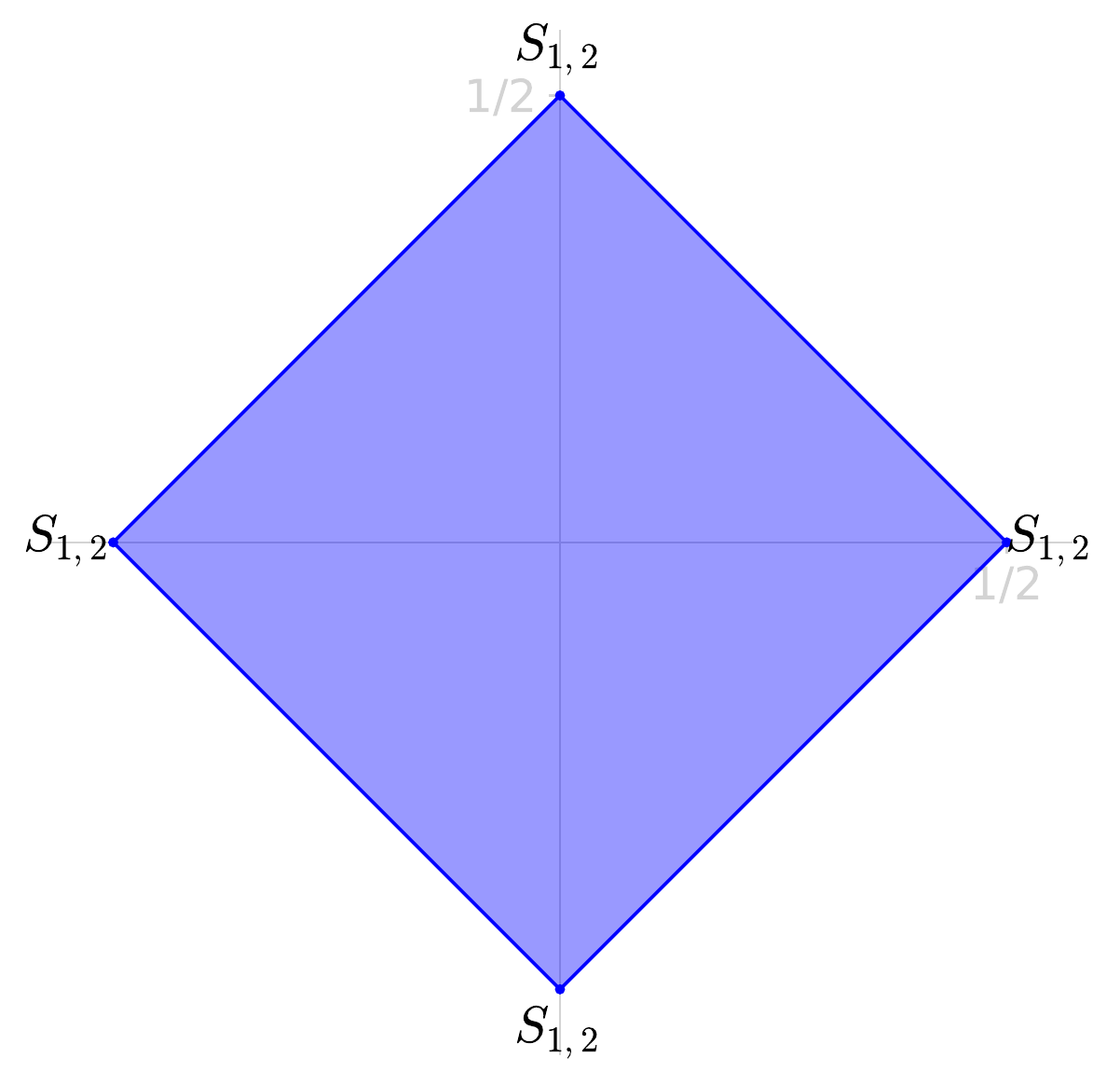}} & \quad & \multirow{6}{*}{\Includegraphics[width=1.8in]{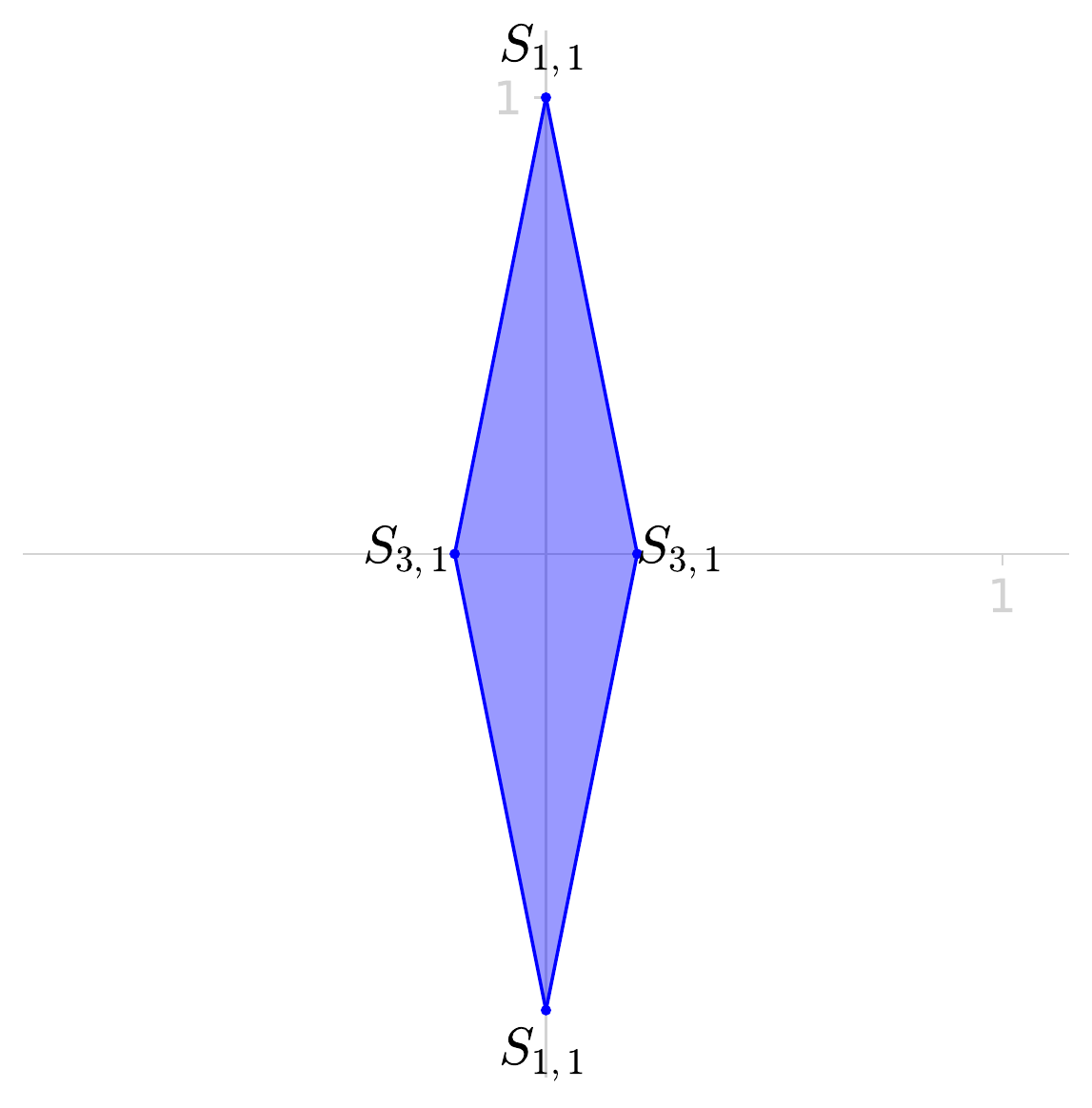}} \\ 
 $L=9^{{2}}_{{12}}$ & & $L=9^{{2}}_{{13}}$ & \\ 
 \quad & & \quad & \\ $\mathrm{Isom}(\mathbb{S}^3\setminus L) = \mathbb{{Z}}_2\oplus\mathbb{{Z}}_2$ & & $\mathrm{Isom}(\mathbb{S}^3\setminus L) = \mathbb{{Z}}_2\oplus\mathbb{{Z}}_2$ & \\ 
 \quad & & \quad & \\ 
 \includegraphics[width=1in]{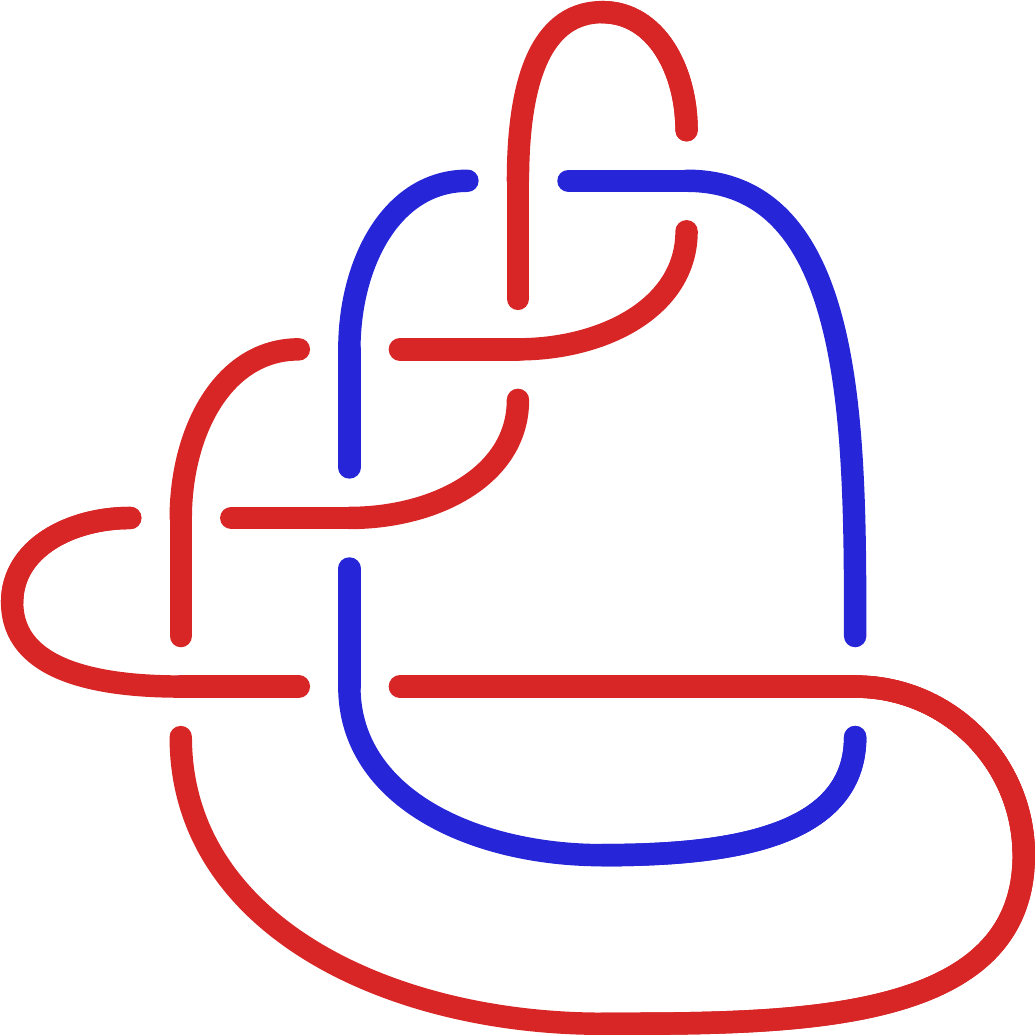}  & & \includegraphics[width=1in]{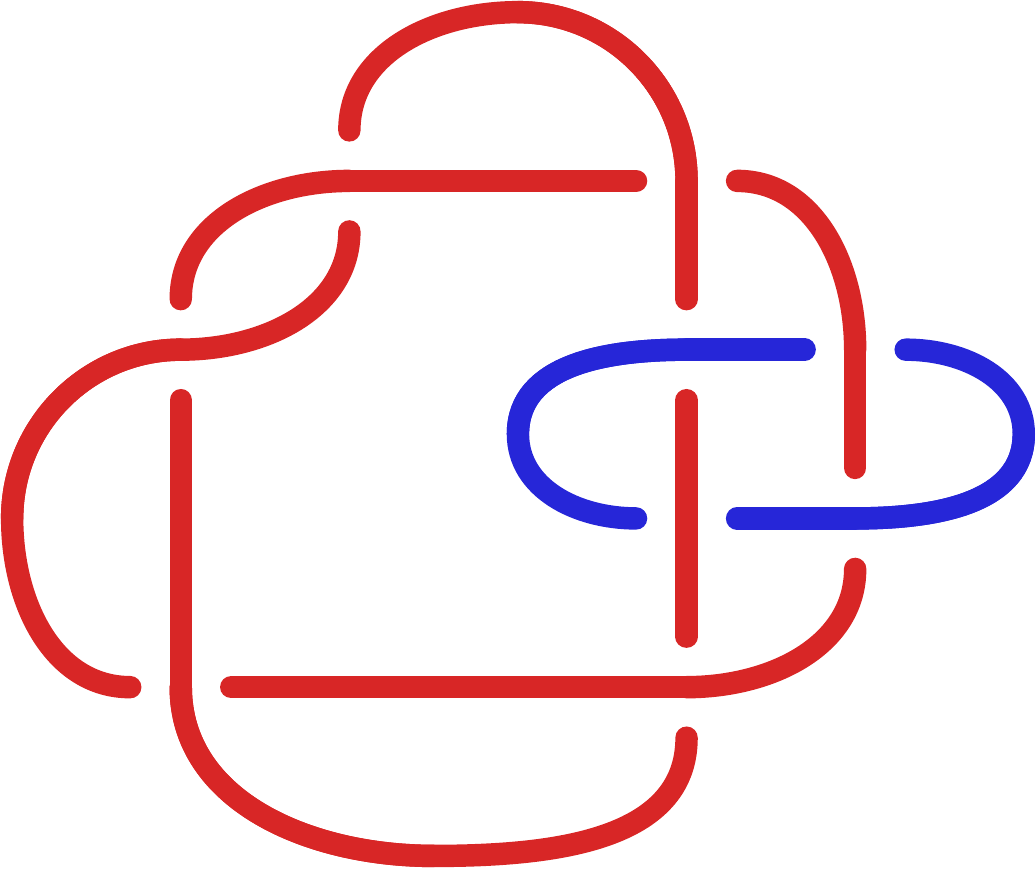} & \\ 
 \quad & & \quad & \\ 
 \hline  
\quad & \multirow{6}{*}{\Includegraphics[width=1.8in]{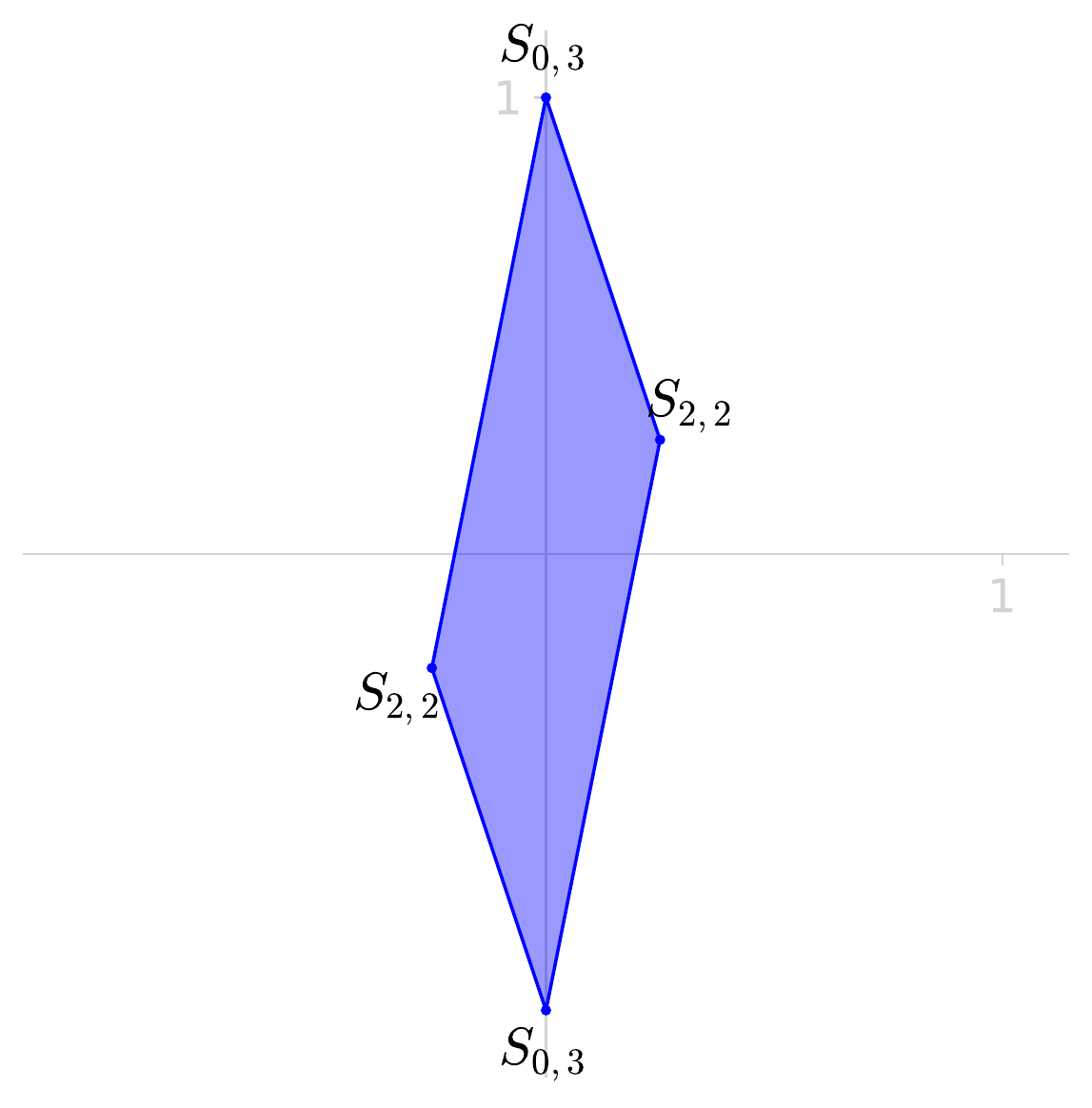}} & \quad & \multirow{6}{*}{\Includegraphics[width=1.8in]{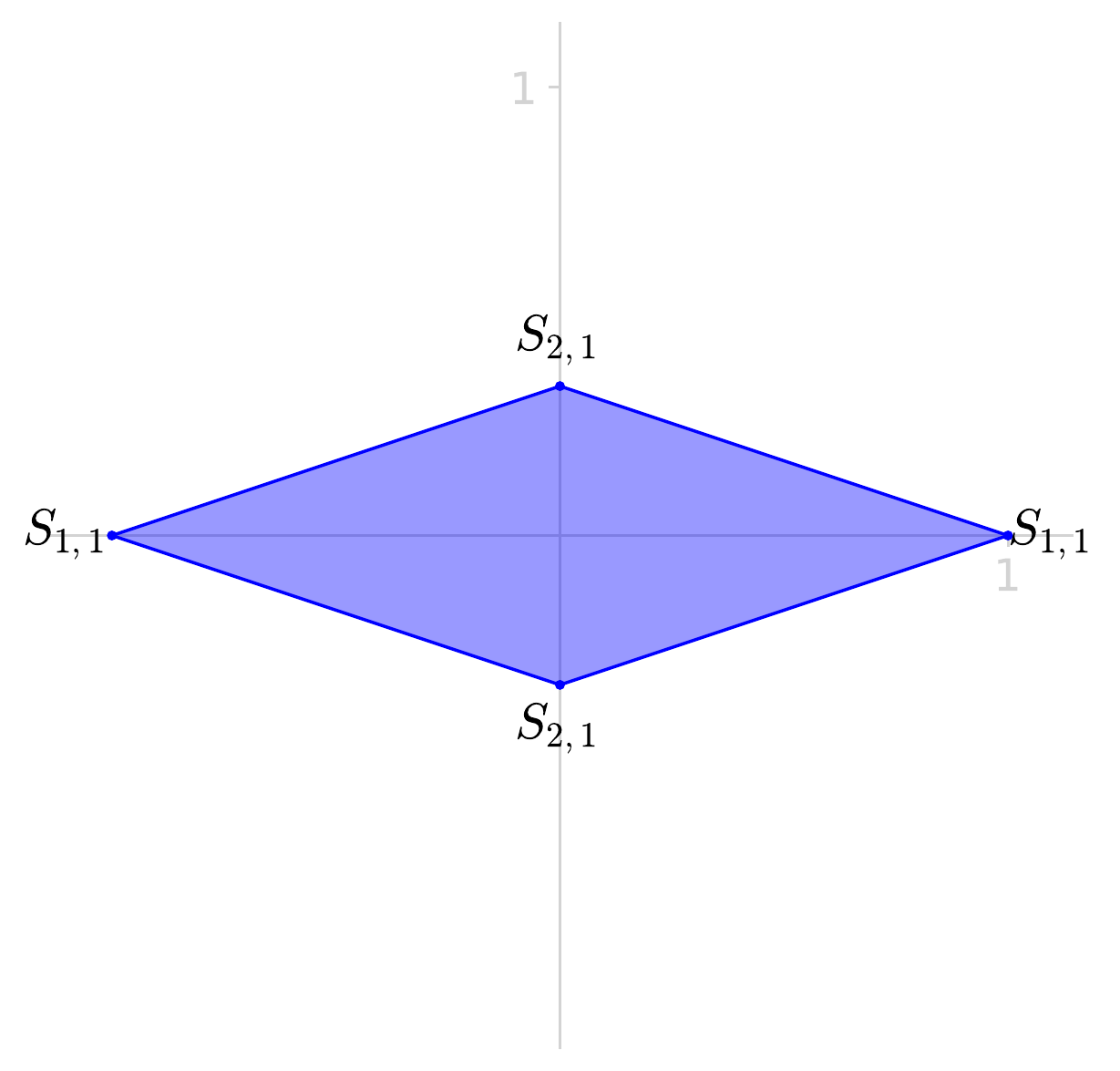}} \\ 
 $L=9^{{2}}_{{14}}$ & & $L=9^{{2}}_{{15}}$ & \\ 
 \quad & & \quad & \\ $\mathrm{Isom}(\mathbb{S}^3\setminus L) = \mathbb{{Z}}_2\oplus\mathbb{{Z}}_2$ & & $\mathrm{Isom}(\mathbb{S}^3\setminus L) = \mathbb{{Z}}_2\oplus\mathbb{{Z}}_2$ & \\ 
 \quad & & \quad & \\ 
 \includegraphics[width=1in]{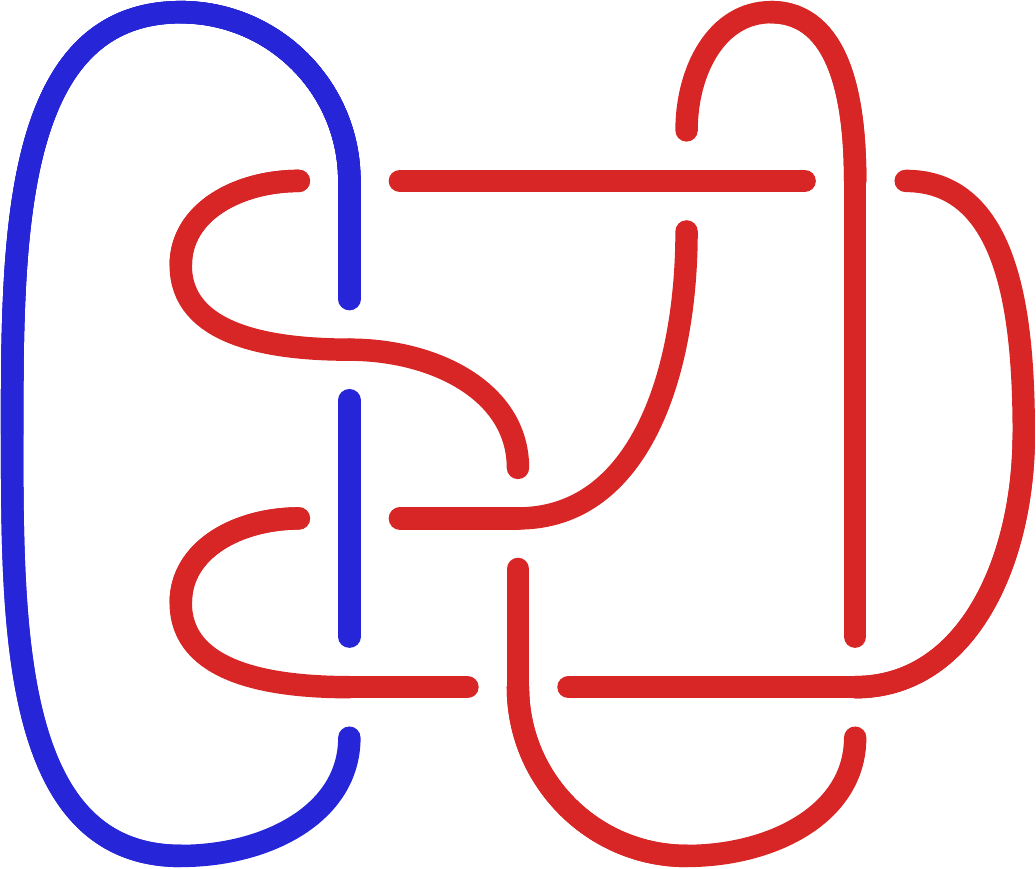}  & & \includegraphics[width=1in]{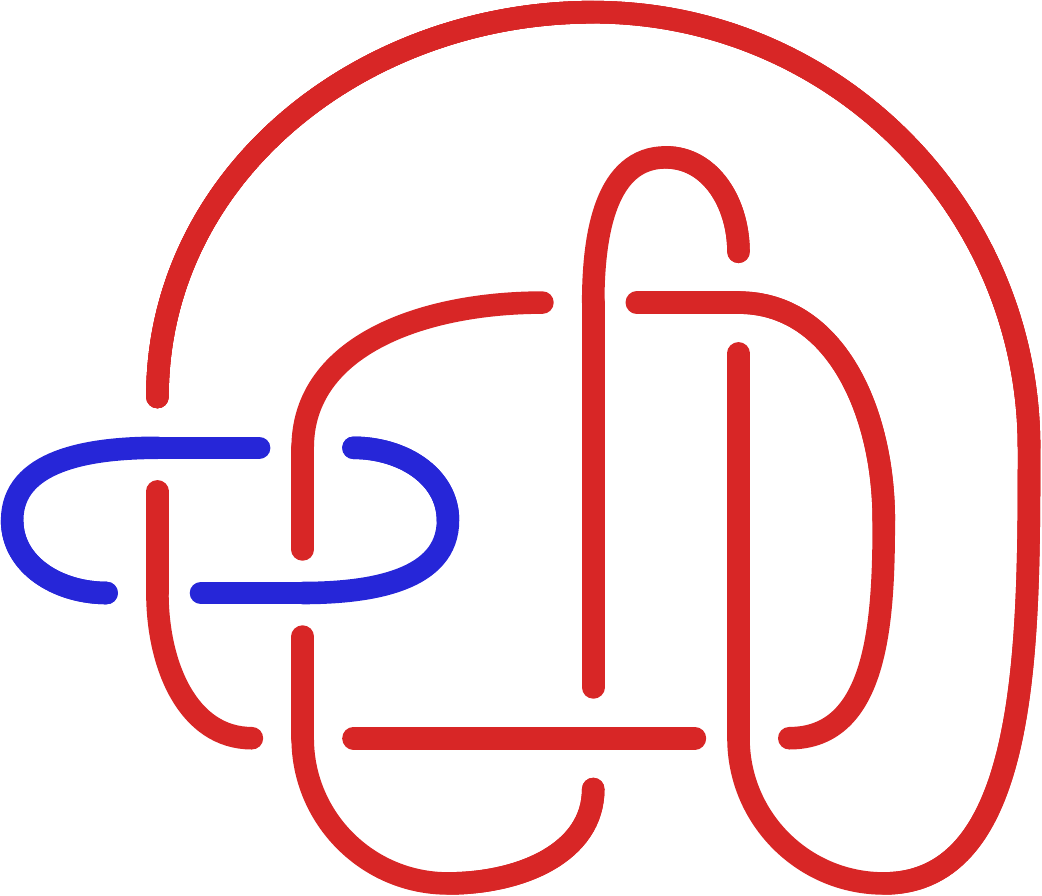} & \\ 
 \quad & & \quad & \\ 
 \hline  
\quad & \multirow{6}{*}{\Includegraphics[width=1.8in]{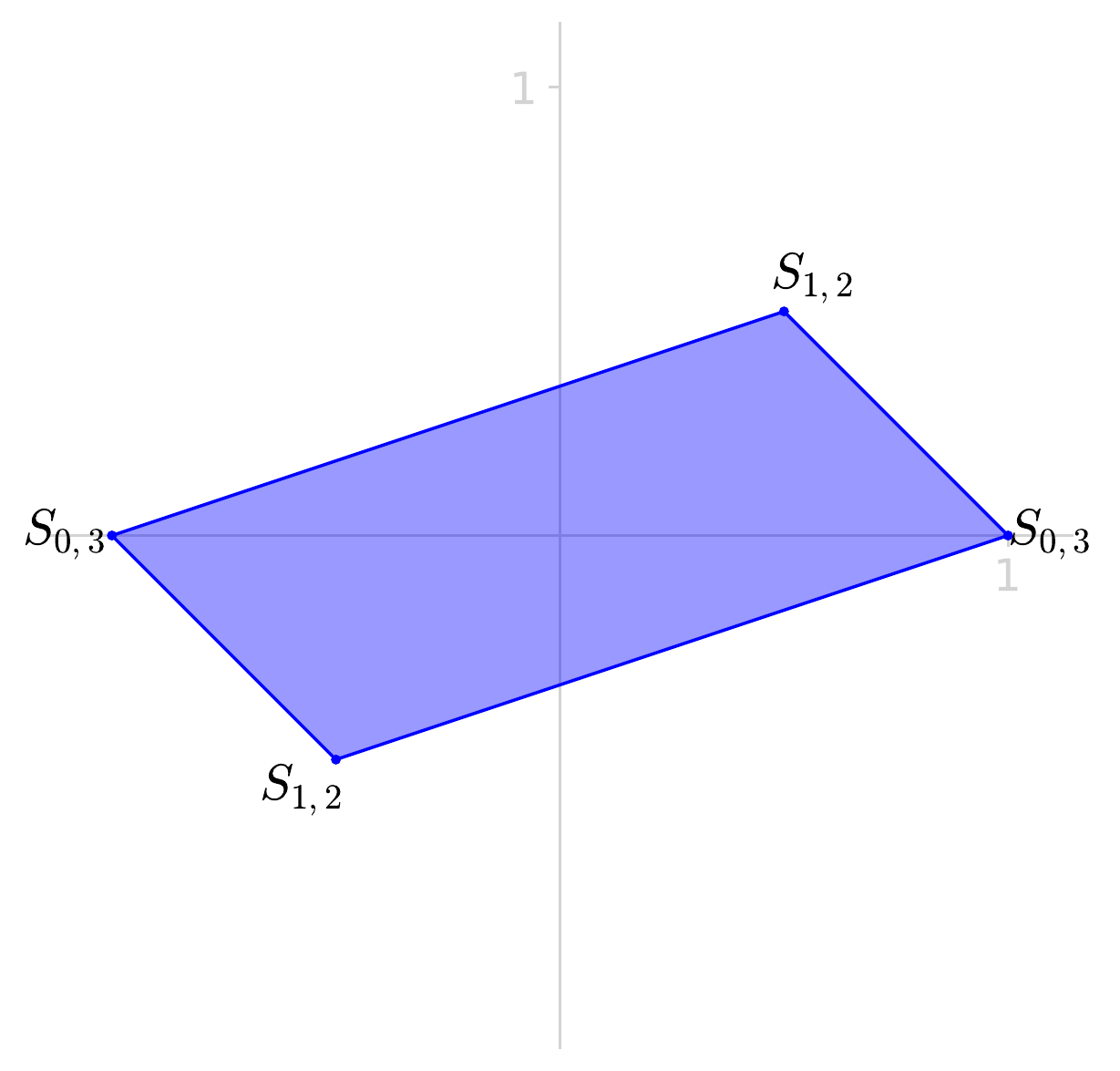}} & \quad & \multirow{6}{*}{\Includegraphics[width=1.8in]{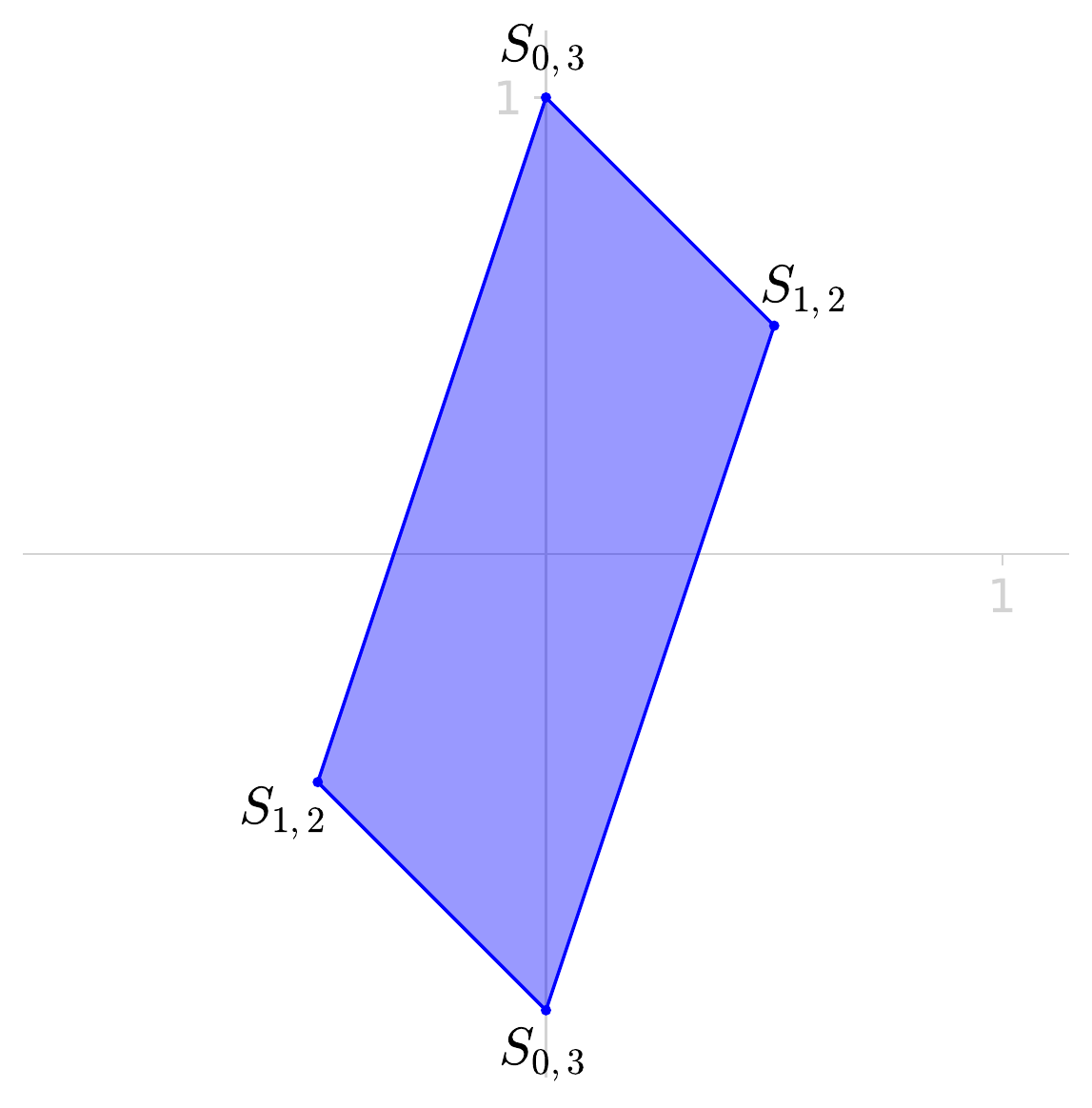}} \\ 
 $L=9^{{2}}_{{16}}$ & & $L=9^{{2}}_{{17}}$ & \\ 
 \quad & & \quad & \\ $\mathrm{Isom}(\mathbb{S}^3\setminus L) = \mathbb{{Z}}_2\oplus\mathbb{{Z}}_2$ & & $\mathrm{Isom}(\mathbb{S}^3\setminus L) = \mathbb{{Z}}_2\oplus\mathbb{{Z}}_2$ & \\ 
 \quad & & \quad & \\ 
 \includegraphics[width=1in]{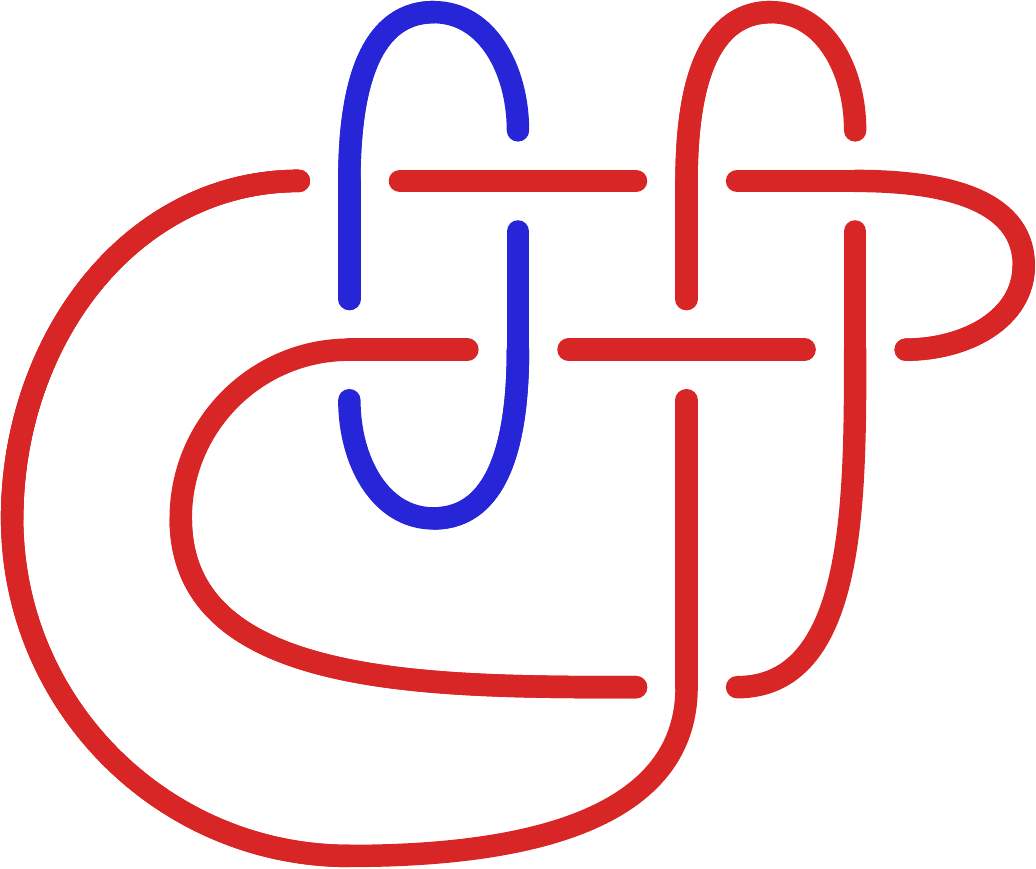}  & & \includegraphics[width=1in]{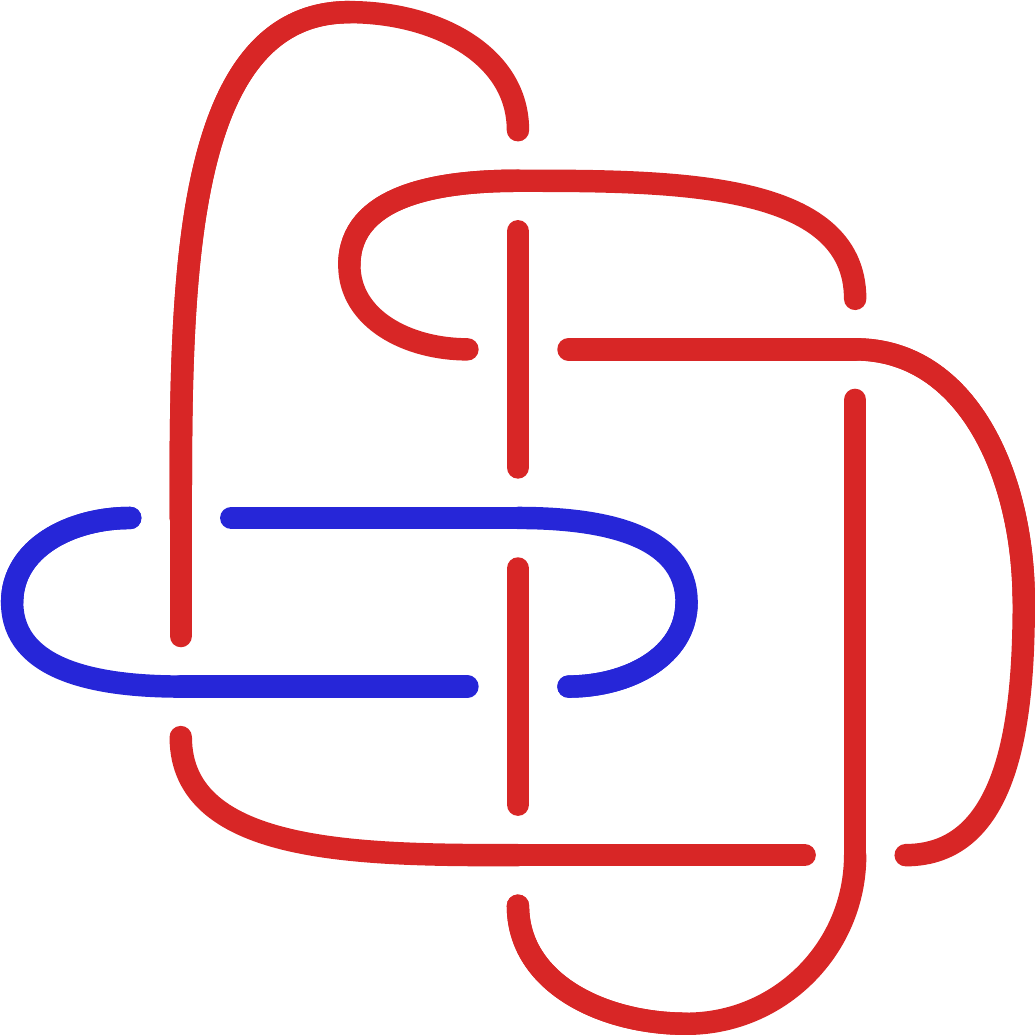} & \\ 
 \quad & & \quad & \\ 
 \hline  
\quad & \multirow{6}{*}{\Includegraphics[width=1.8in]{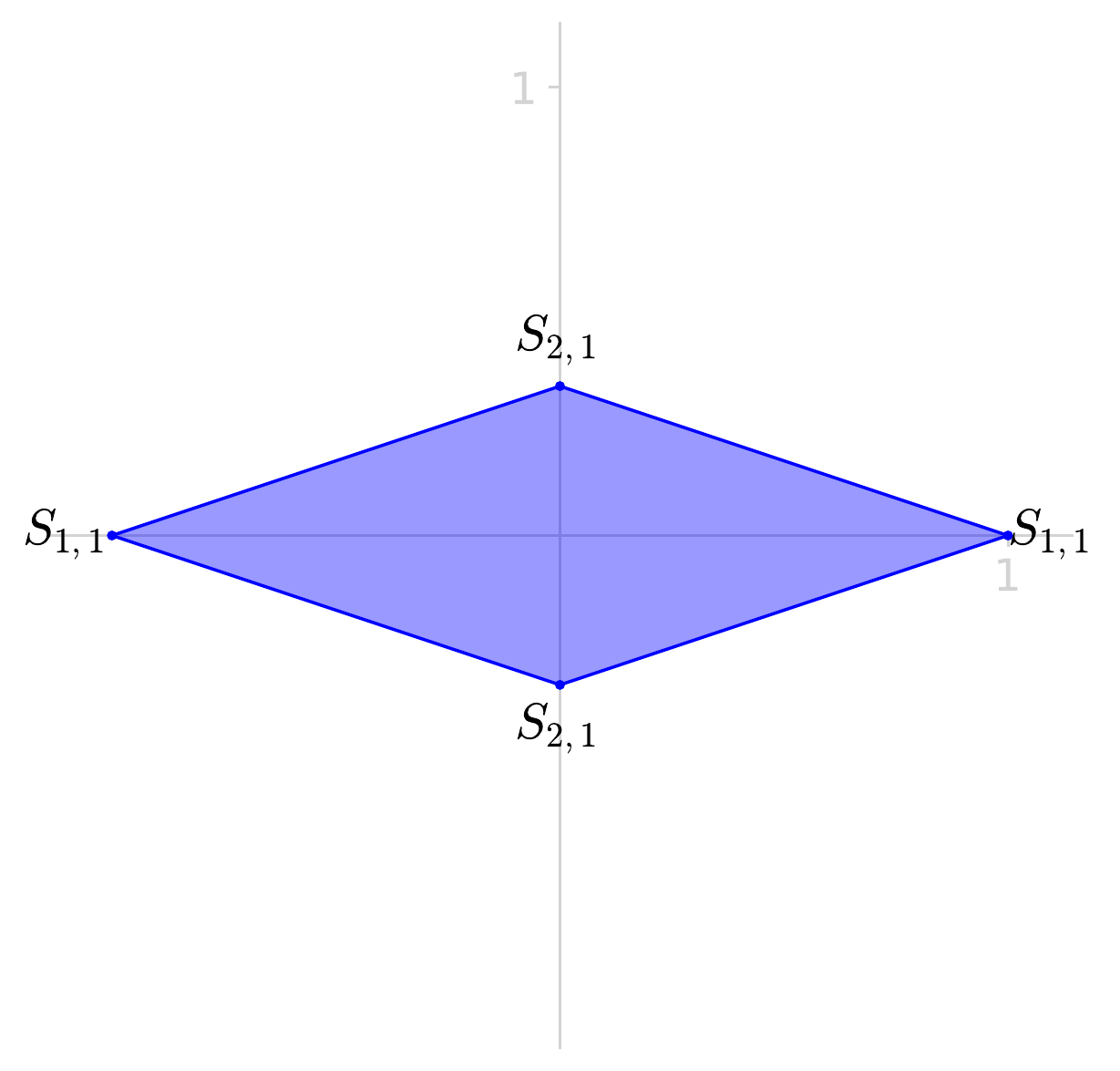}} & \quad & \multirow{6}{*}{\Includegraphics[width=1.8in]{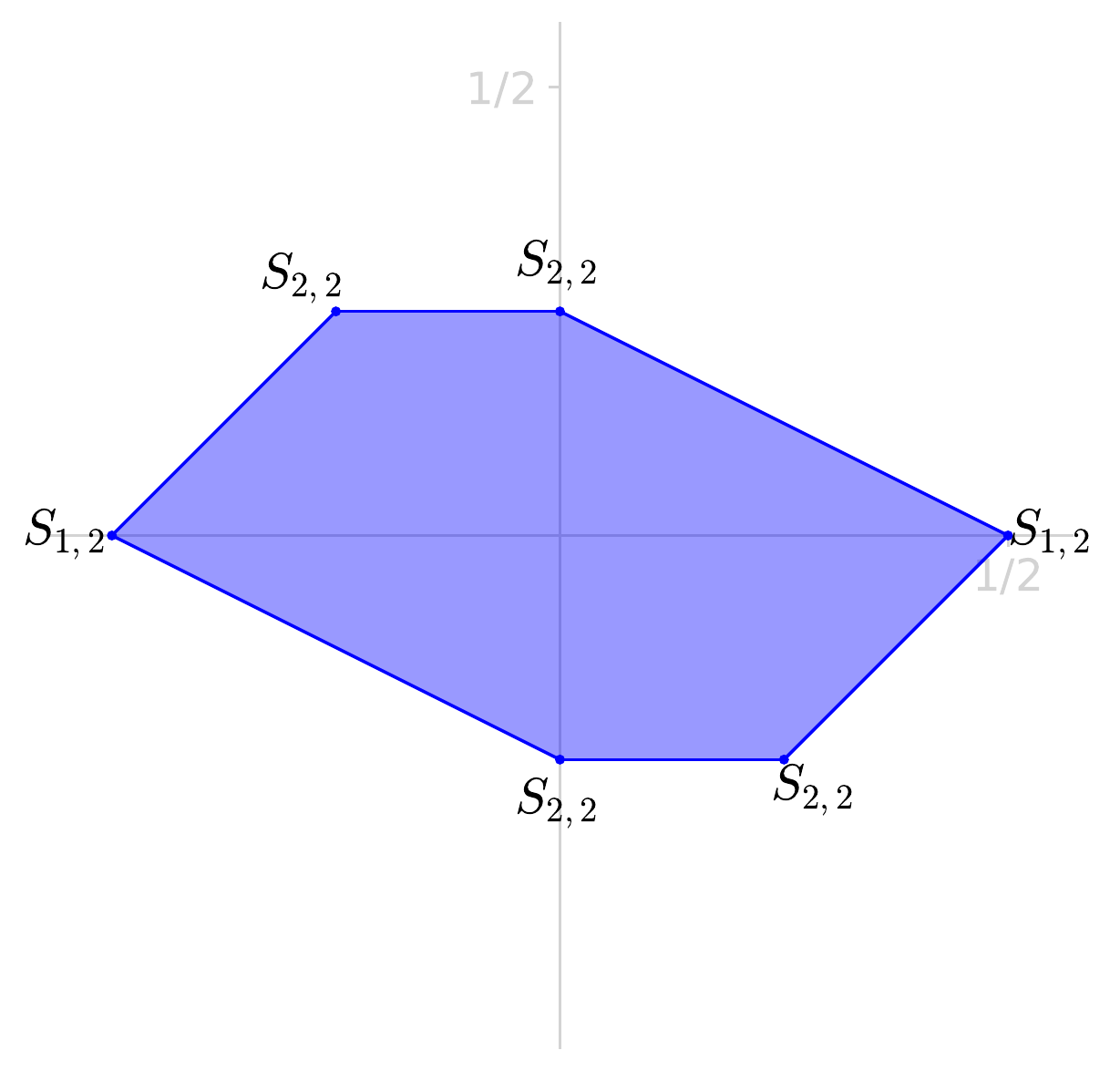}} \\ 
 $L=9^{{2}}_{{18}}$ & & $L=9^{{2}}_{{19}}$ & \\ 
 \quad & & \quad & \\ $\mathrm{Isom}(\mathbb{S}^3\setminus L) = \mathbb{{Z}}_2\oplus\mathbb{{Z}}_2$ & & $\mathrm{Isom}(\mathbb{S}^3\setminus L) = \mathbb{{Z}}_2$ & \\ 
 \quad & & \quad & \\ 
 \includegraphics[width=1in]{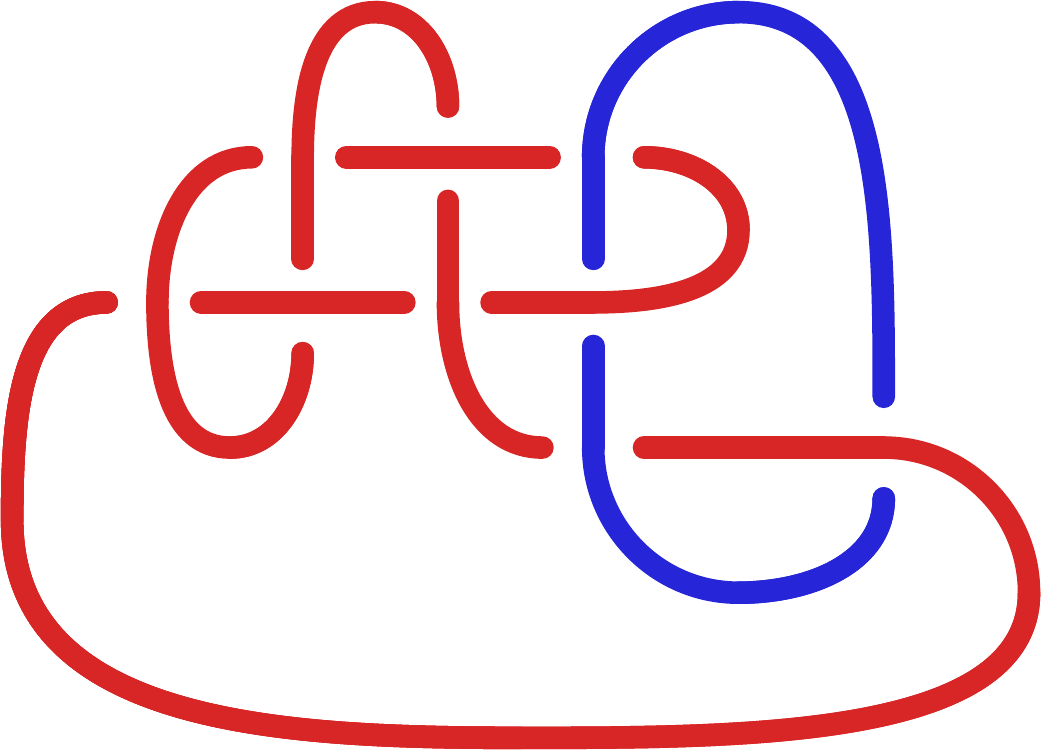}  & & \includegraphics[width=1in]{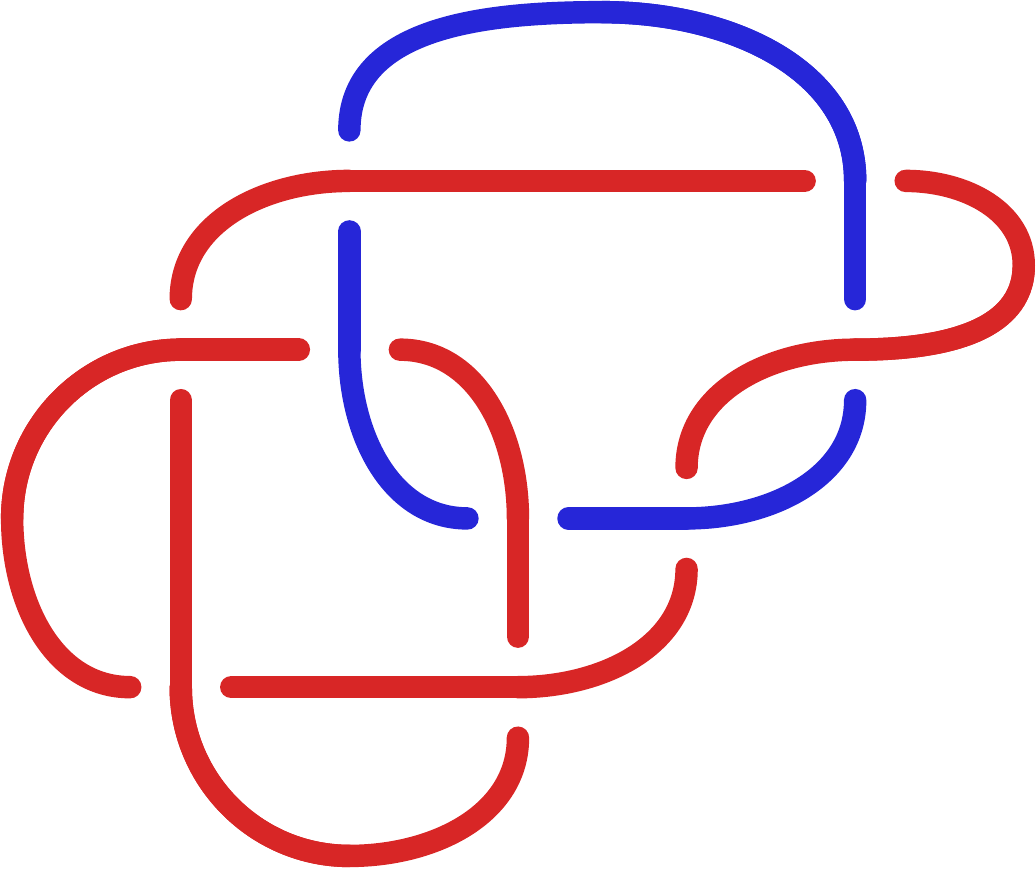} & \\ 
 \quad & & \quad & \\ 
 \hline  
\quad & \multirow{6}{*}{\Includegraphics[width=1.8in]{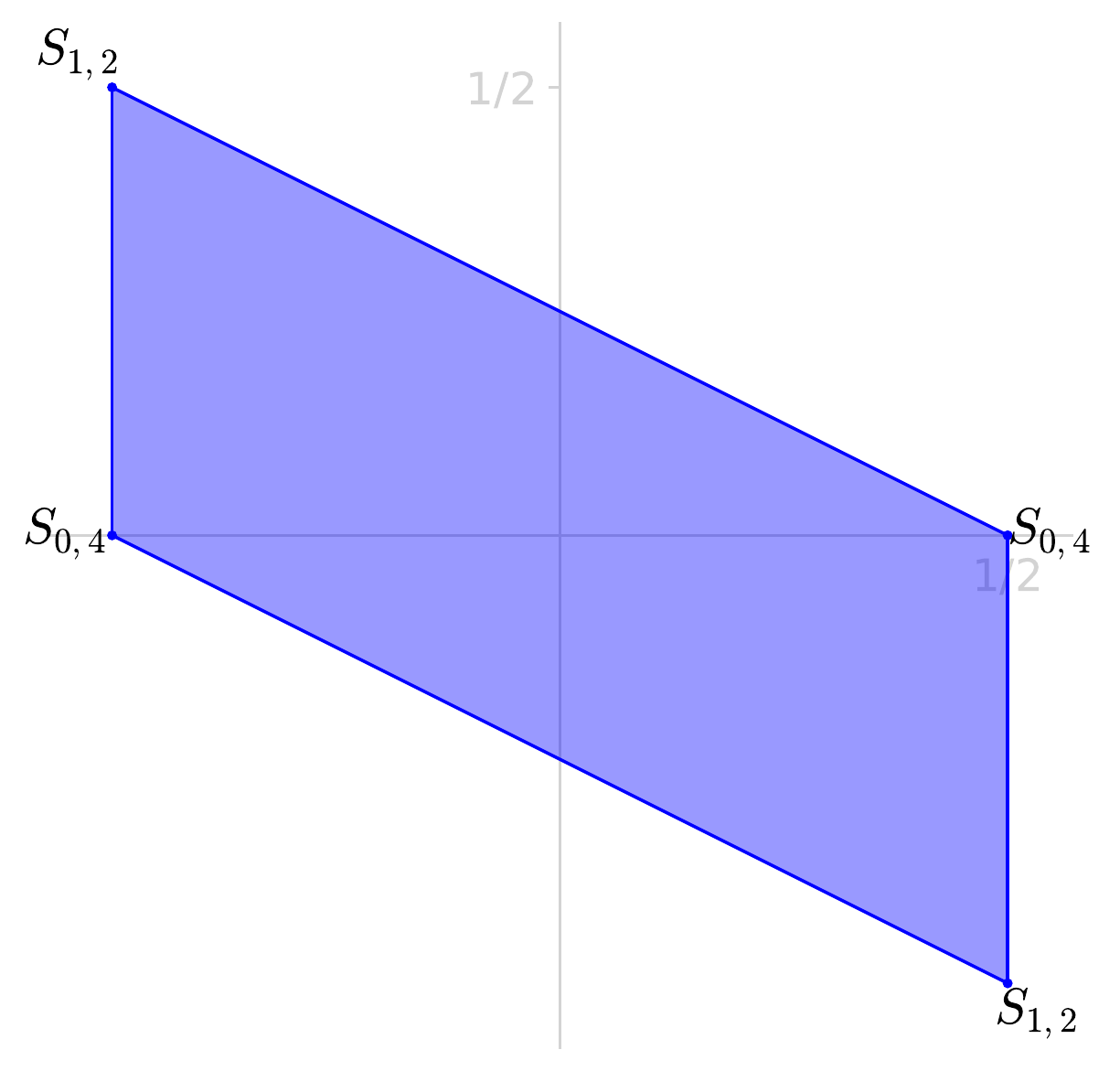}} & \quad & \multirow{6}{*}{\Includegraphics[width=1.8in]{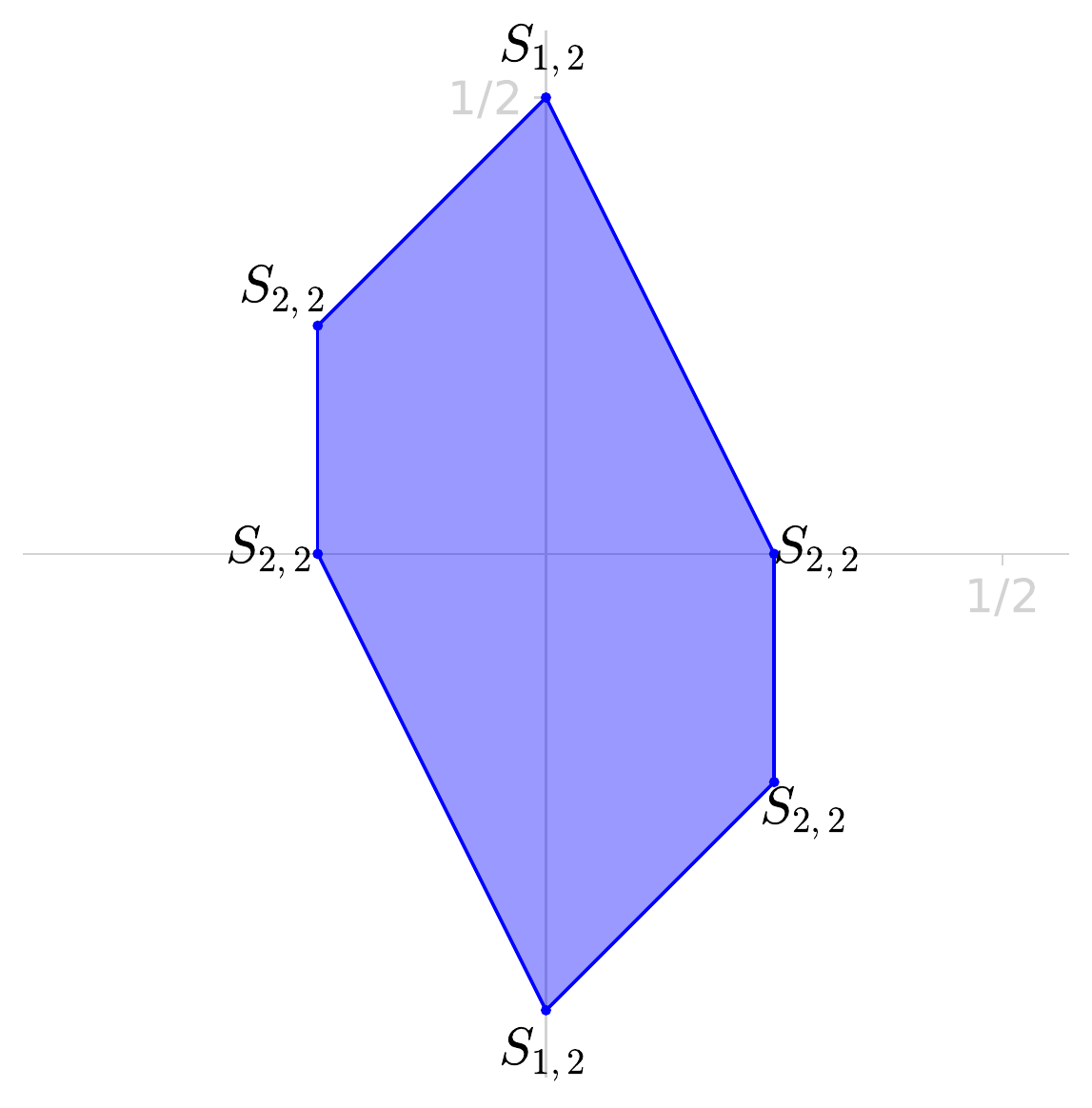}} \\ 
 $L=9^{{2}}_{{20}}$ & & $L=9^{{2}}_{{21}}$ & \\ 
 \quad & & \quad & \\ $\mathrm{Isom}(\mathbb{S}^3\setminus L) = \mathbb{{Z}}_2$ & & $\mathrm{Isom}(\mathbb{S}^3\setminus L) = \mathbb{{Z}}_2$ & \\ 
 \quad & & \quad & \\ 
 \includegraphics[width=1in]{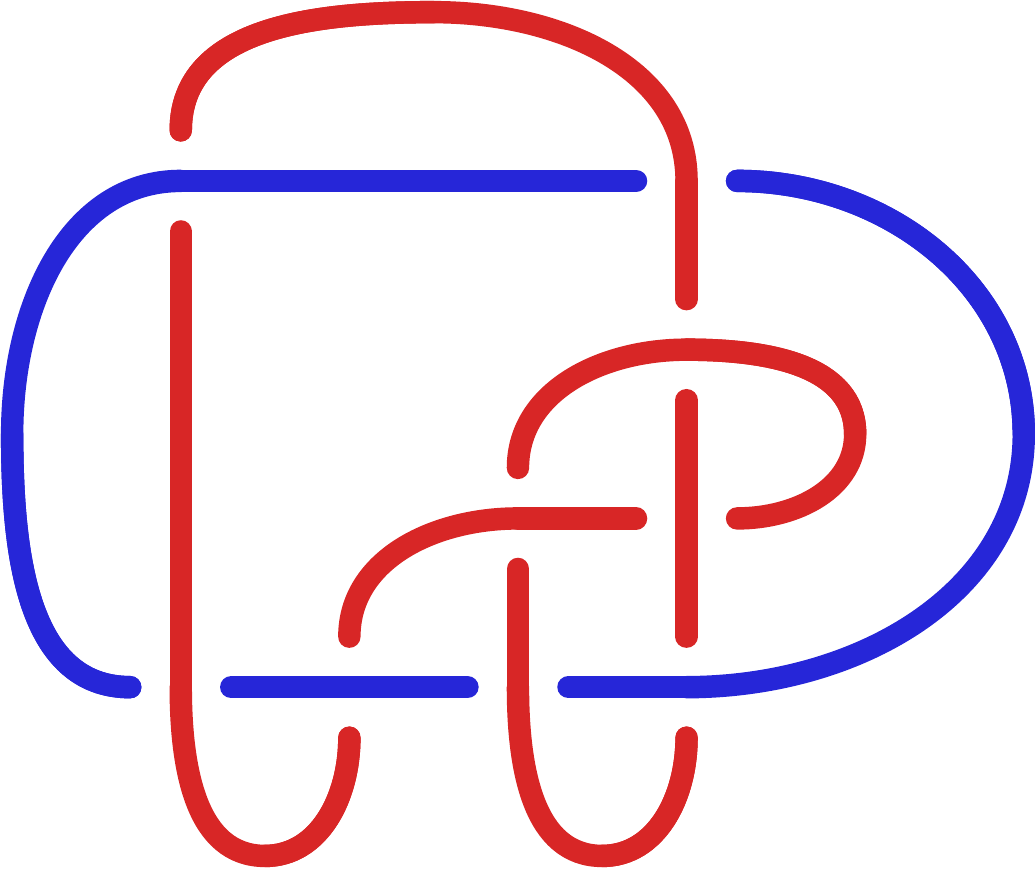}  & & \includegraphics[width=1in]{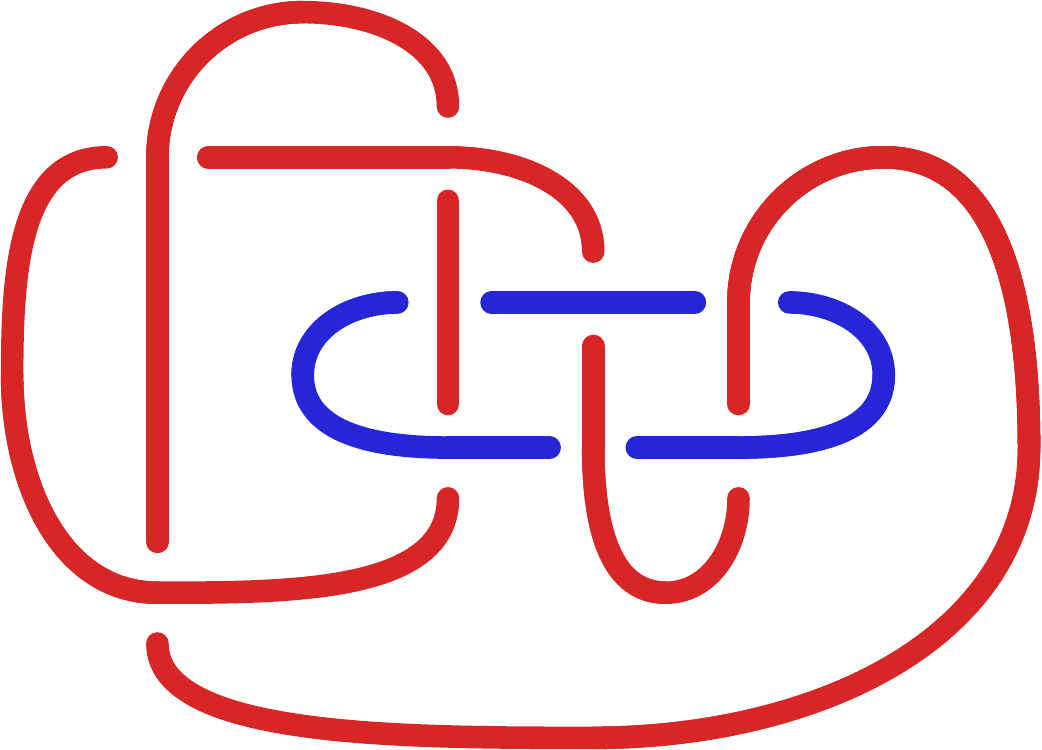} & \\ 
 \quad & & \quad & \\ 
 \hline  
\end{tabular} 
 \newpage \begin{tabular}{|c|c|c|c|} 
 \hline 
 Link & Norm Ball & Link & Norm Ball \\ 
 \hline 
\quad & \multirow{6}{*}{\Includegraphics[width=1.8in]{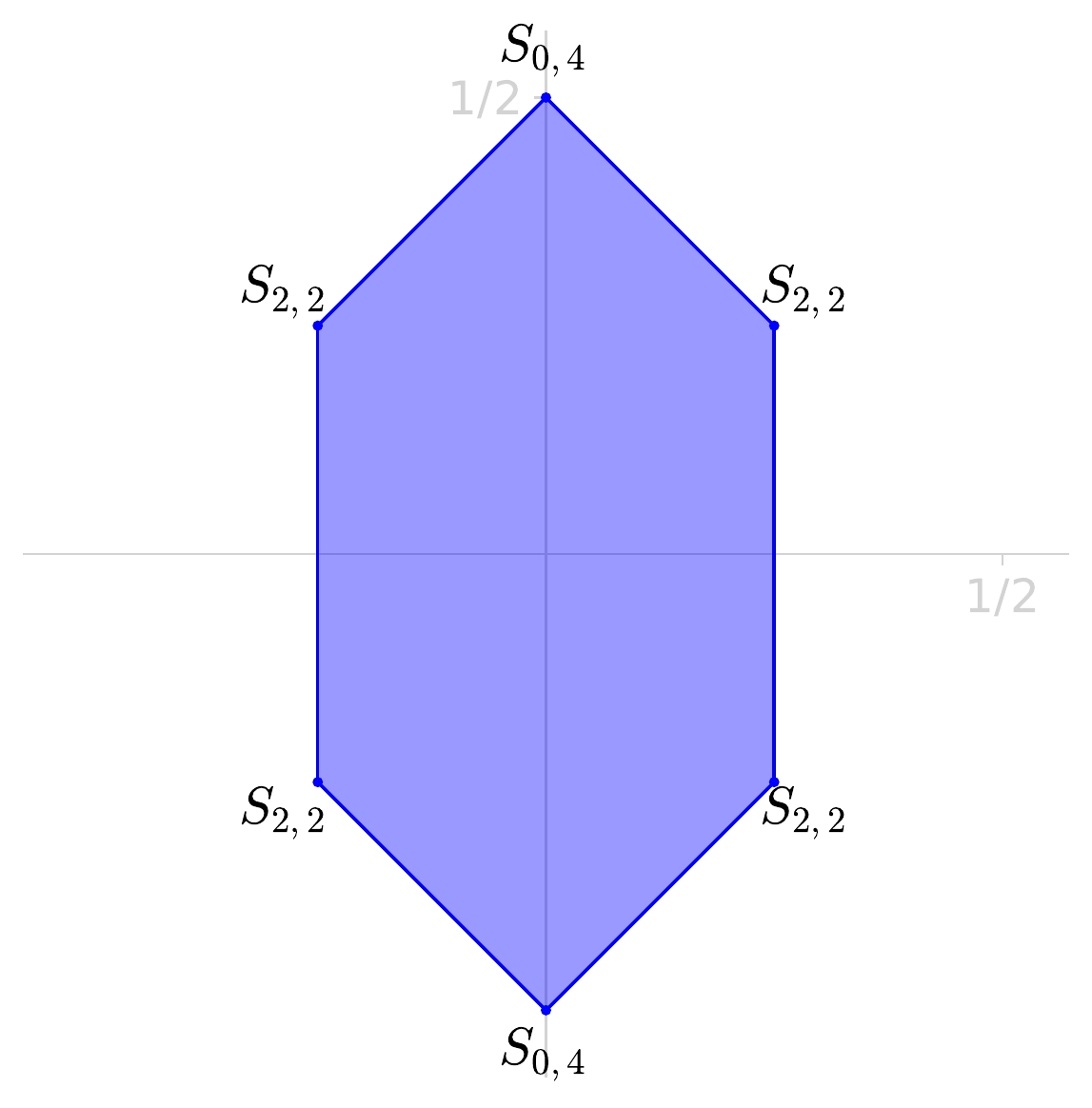}} & \quad & \multirow{6}{*}{\Includegraphics[width=1.8in]{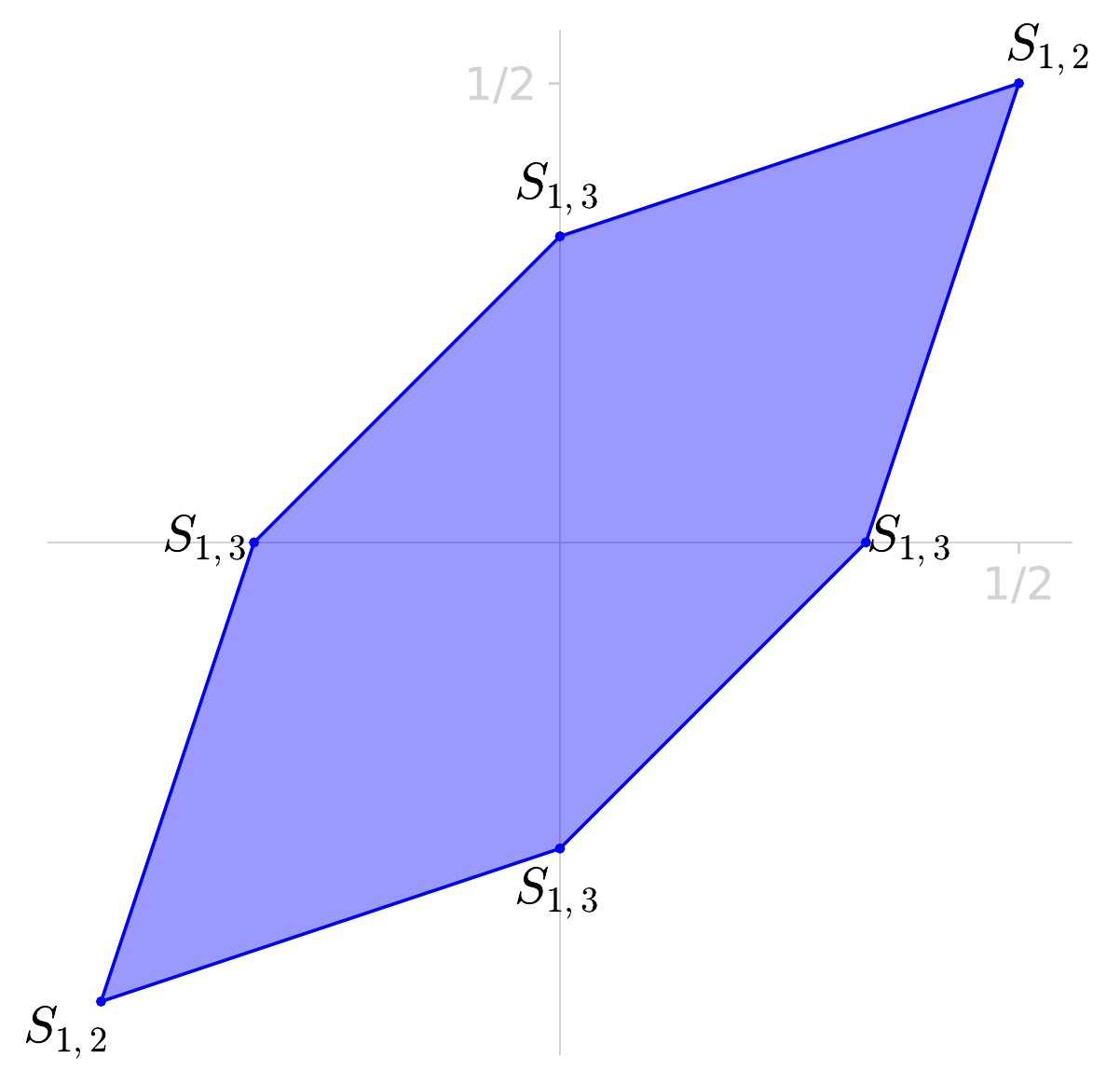}} \\ 
 $L=9^{{2}}_{{22}}$ & & $L=9^{{2}}_{{23}}$ & \\ 
 \quad & & \quad & \\ $\mathrm{Isom}(\mathbb{S}^3\setminus L) = \mathbb{{Z}}_2$ & & $\mathrm{Isom}(\mathbb{S}^3\setminus L) = \mathbb{{Z}}_2\oplus\mathbb{{Z}}_2$ & \\ 
 \quad & & \quad & \\ 
 \includegraphics[width=1in]{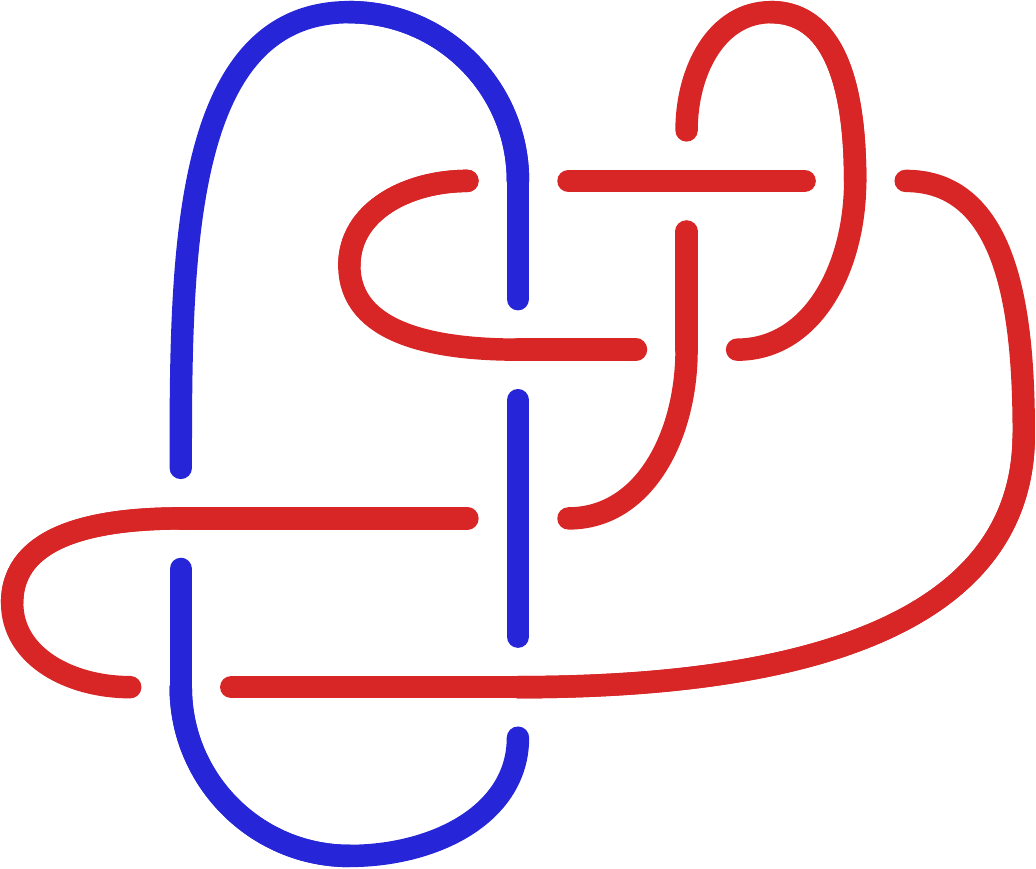}  & & \includegraphics[width=1in]{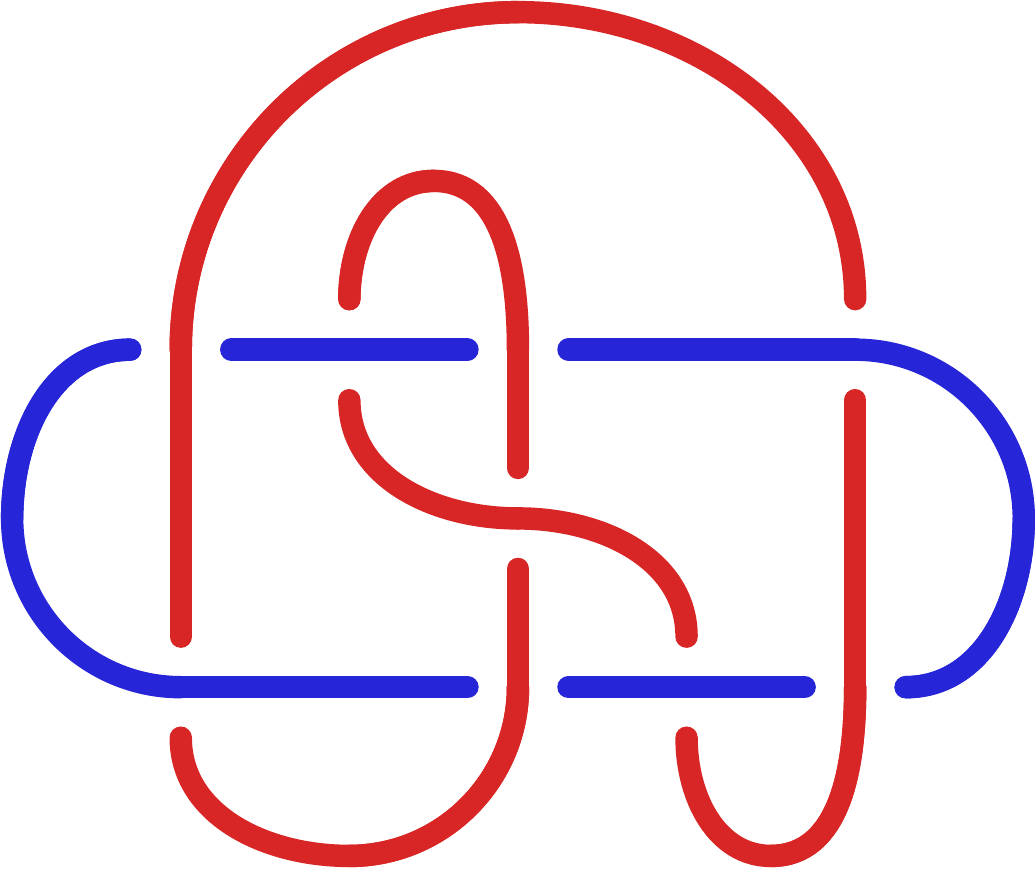} & \\ 
 \quad & & \quad & \\ 
 \hline  
\quad & \multirow{6}{*}{\Includegraphics[width=1.8in]{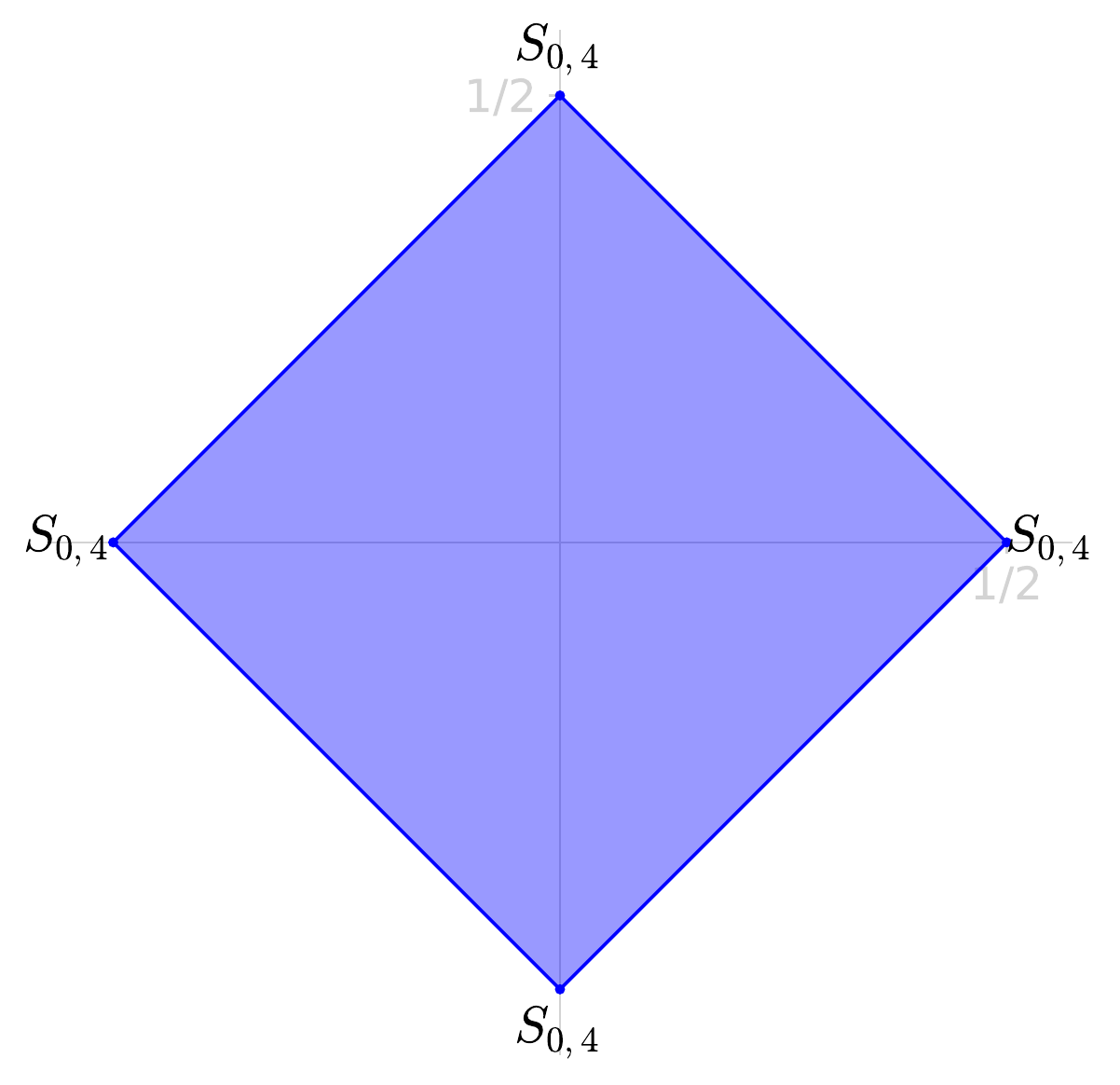}} & \quad & \multirow{6}{*}{\Includegraphics[width=1.8in]{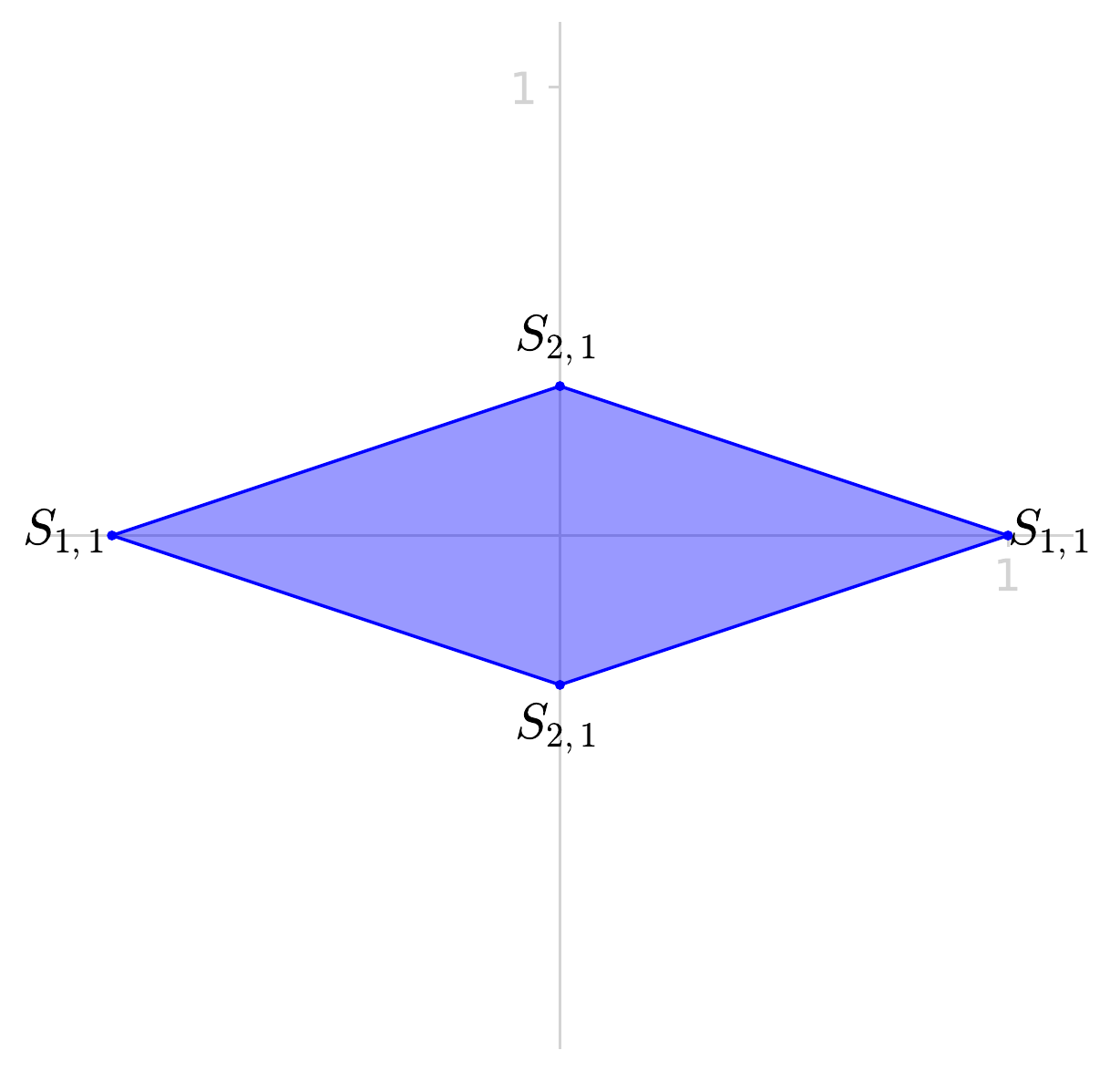}} \\ 
 $L=9^{{2}}_{{24}}$ & & $L=9^{{2}}_{{25}}$ & \\ 
 \quad & & \quad & \\ $\mathrm{Isom}(\mathbb{S}^3\setminus L) = D_6$ & & $\mathrm{Isom}(\mathbb{S}^3\setminus L) = \mathbb{{Z}}_2\oplus\mathbb{{Z}}_2$ & \\ 
 \quad & & \quad & \\ 
 \includegraphics[width=1in]{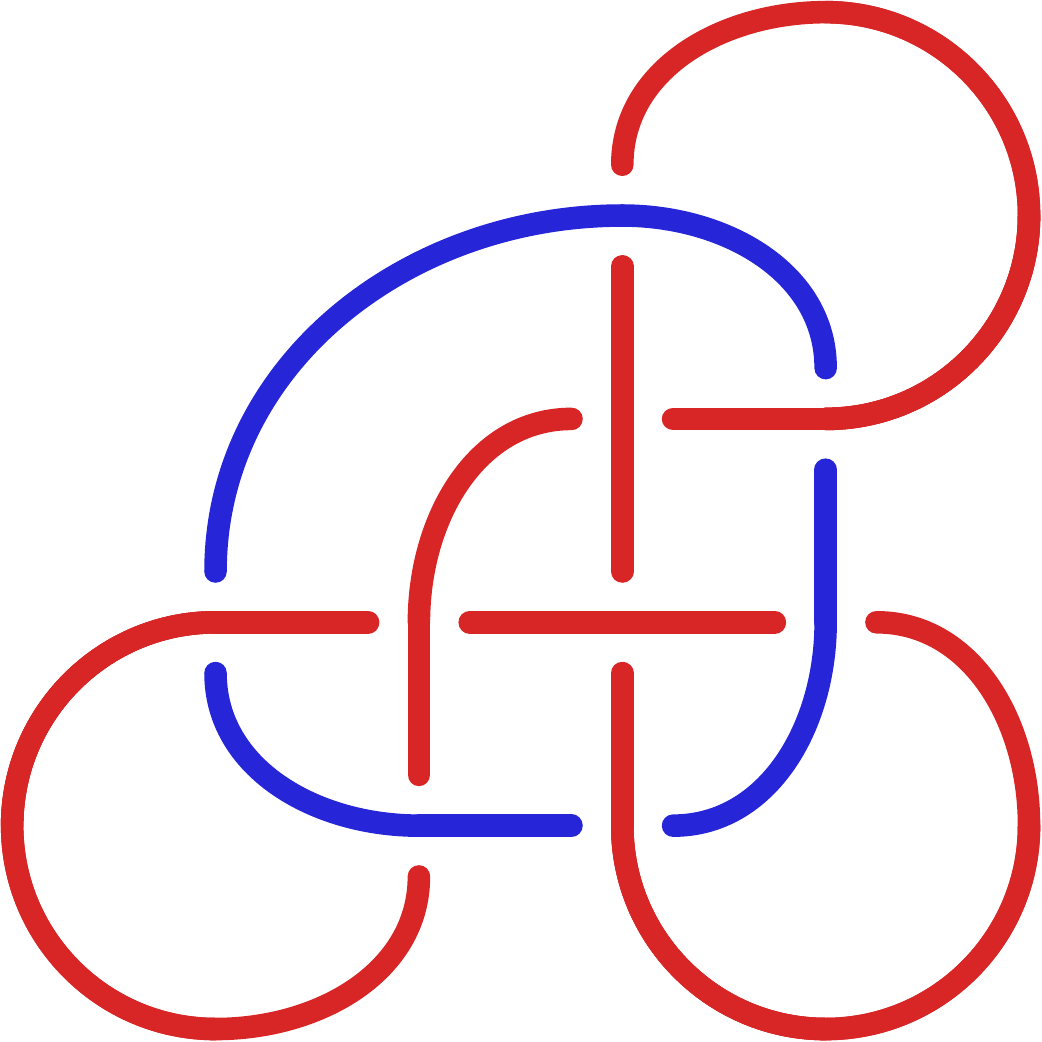}  & & \includegraphics[width=1in]{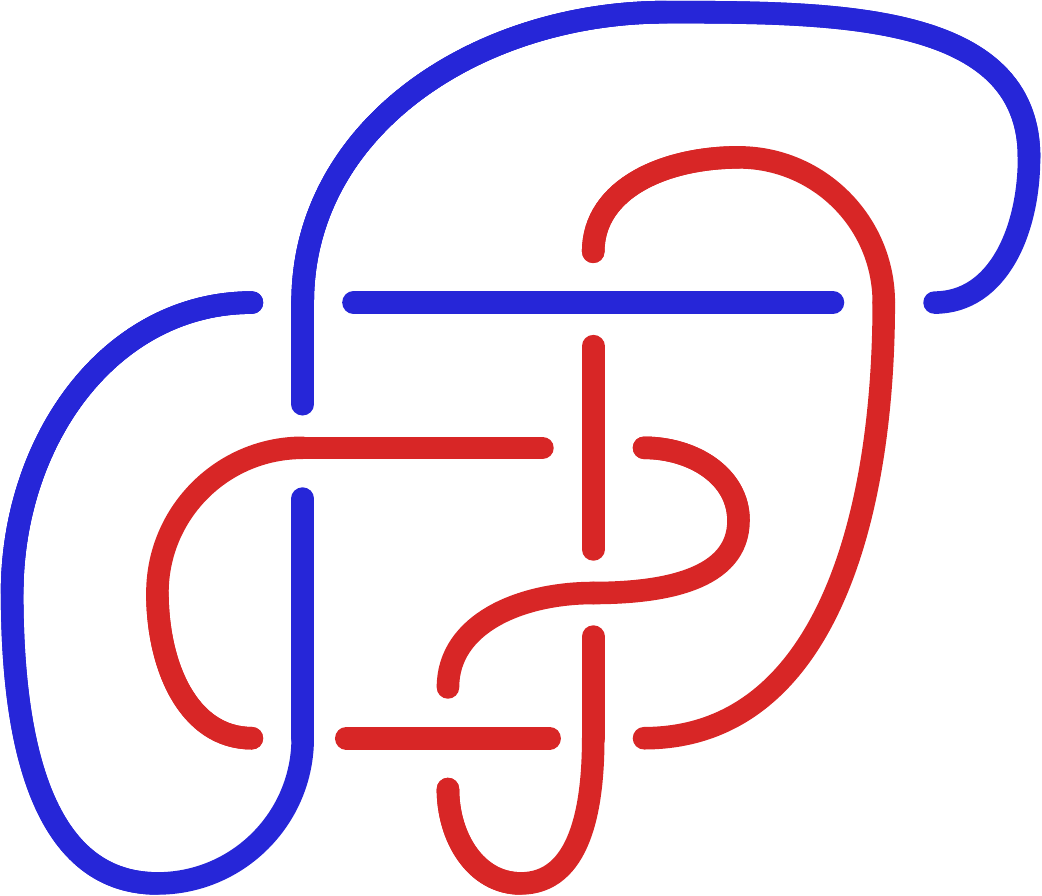} & \\ 
 \quad & & \quad & \\ 
 \hline  
\quad & \multirow{6}{*}{\Includegraphics[width=1.8in]{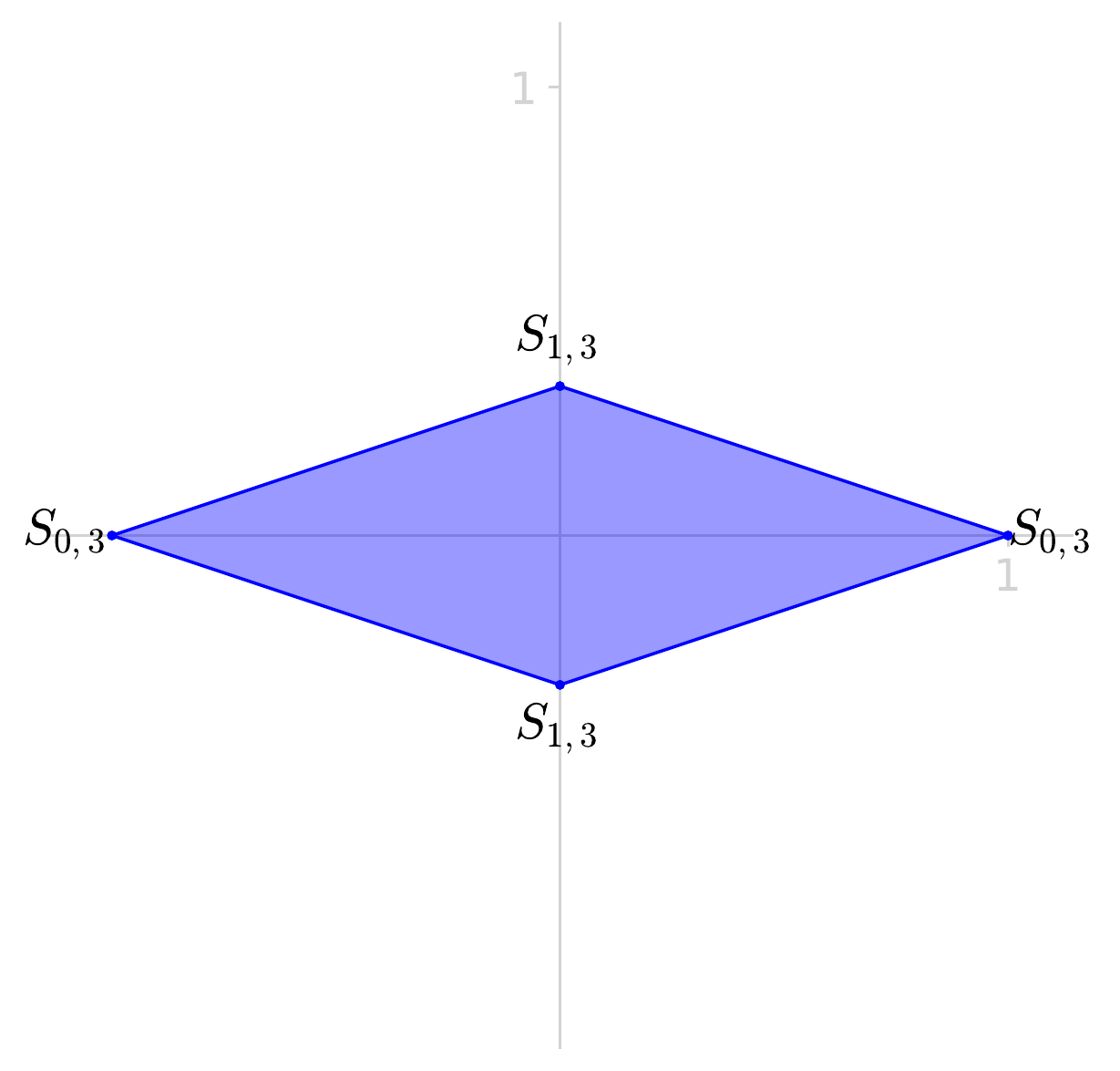}} & \quad & \multirow{6}{*}{\Includegraphics[width=1.8in]{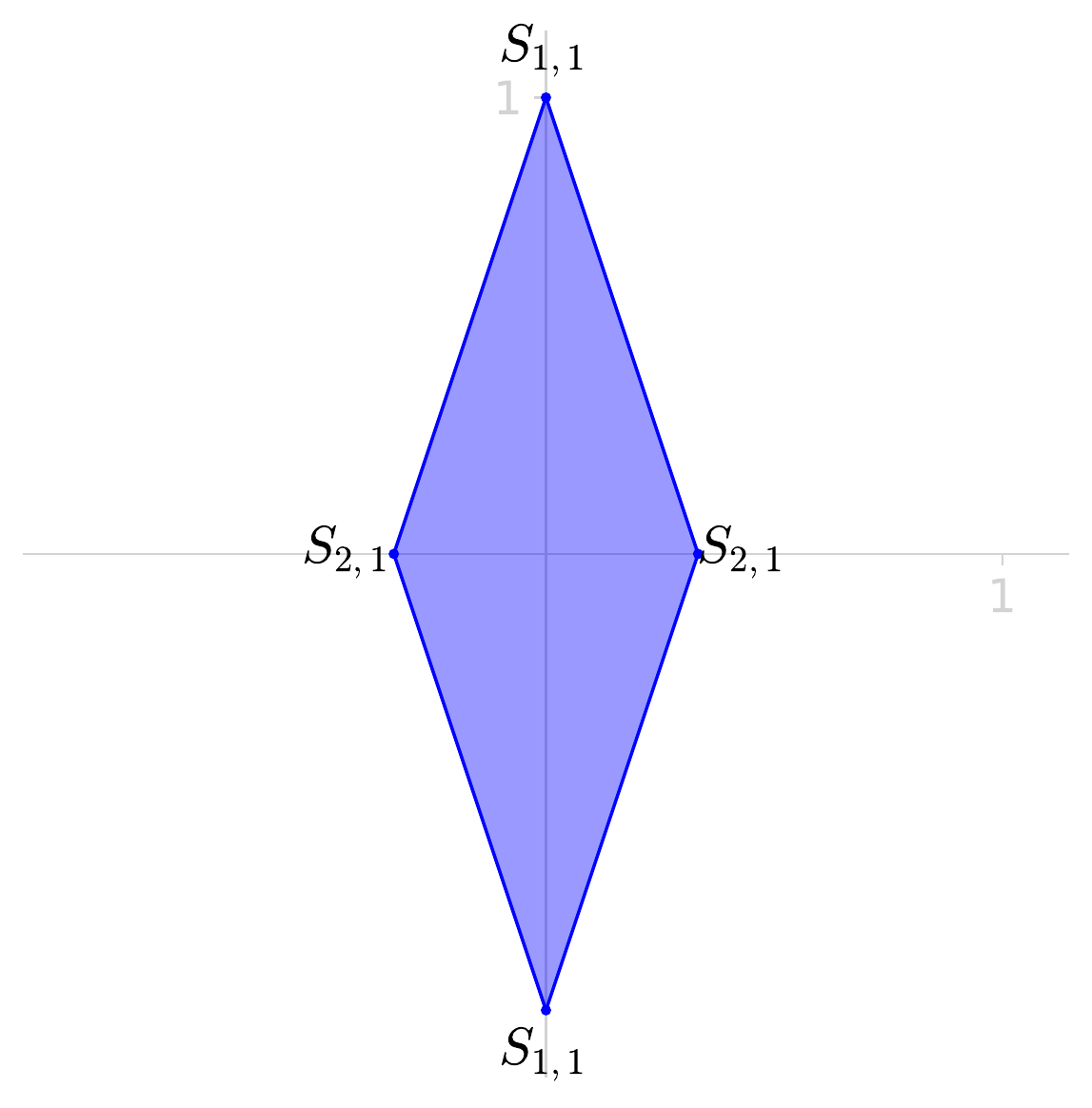}} \\ 
 $L=9^{{2}}_{{26}}$ & & $L=9^{{2}}_{{27}}$ & \\ 
 \quad & & \quad & \\ $\mathrm{Isom}(\mathbb{S}^3\setminus L) = \mathbb{{Z}}_2\oplus\mathbb{{Z}}_2$ & & $\mathrm{Isom}(\mathbb{S}^3\setminus L) = \mathbb{{Z}}_2\oplus\mathbb{{Z}}_2$ & \\ 
 \quad & & \quad & \\ 
 \includegraphics[width=1in]{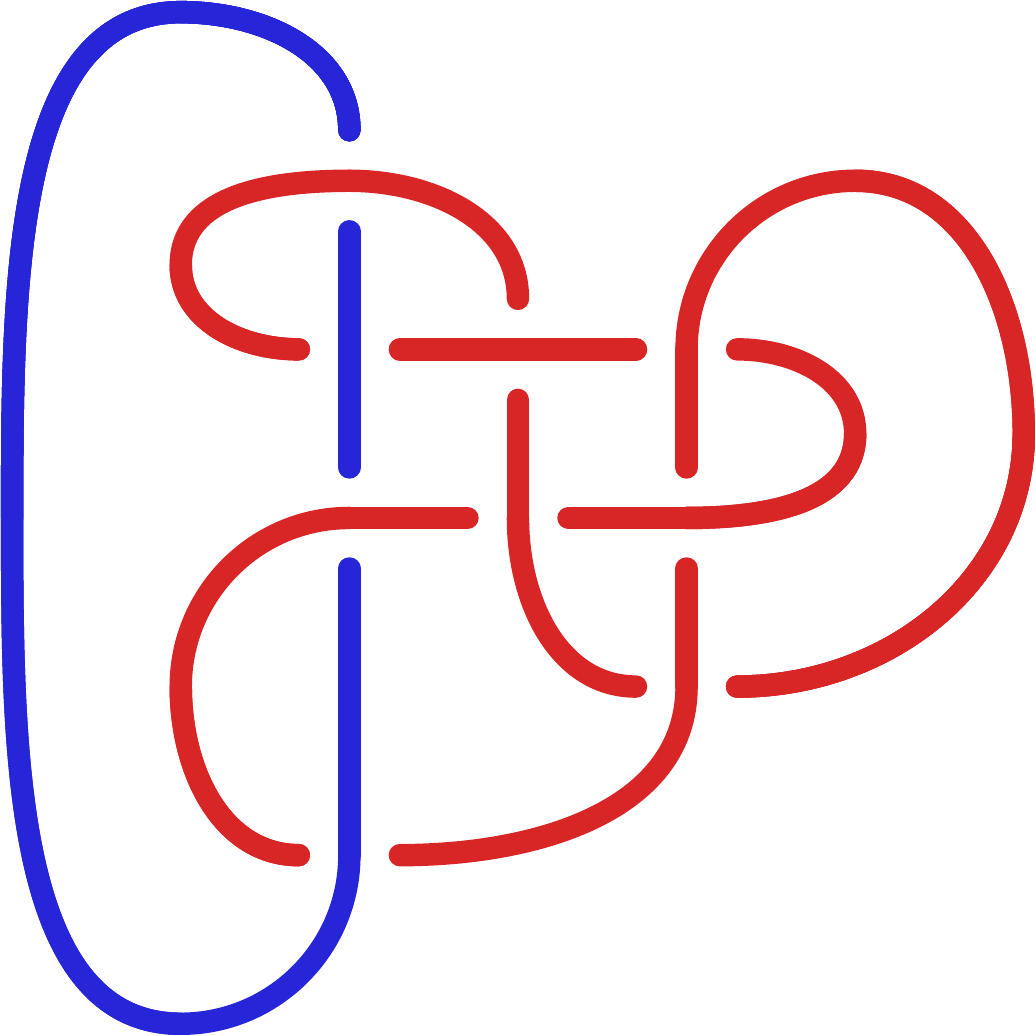}  & & \includegraphics[width=1in]{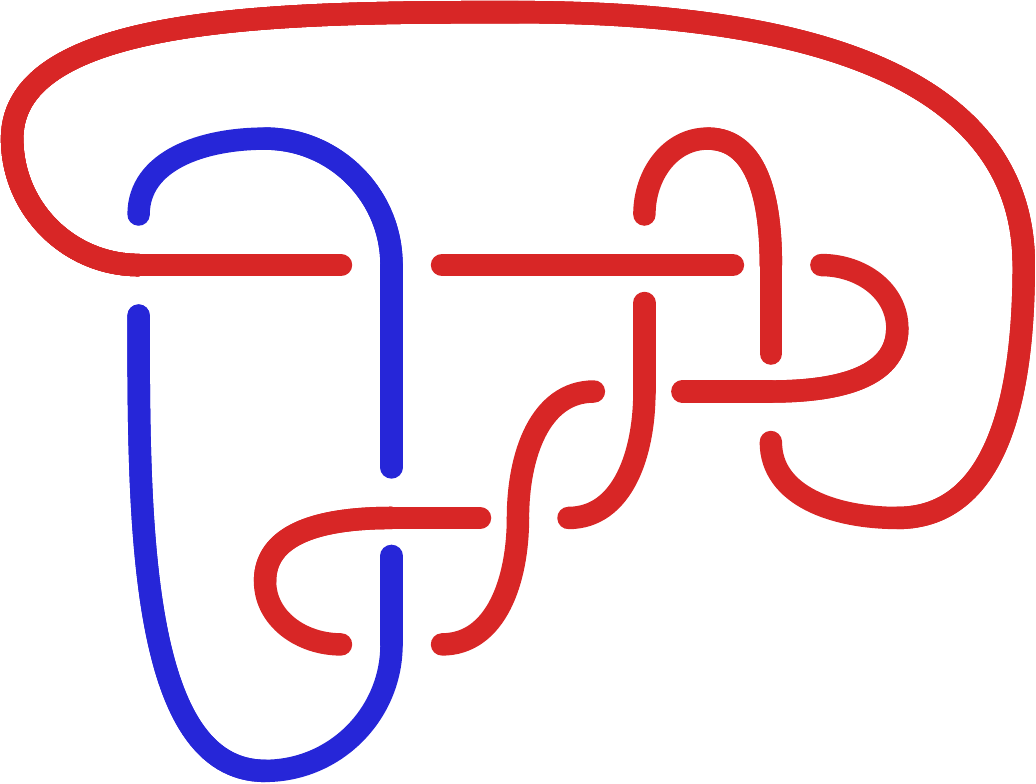} & \\ 
 \quad & & \quad & \\ 
 \hline  
\quad & \multirow{6}{*}{\Includegraphics[width=1.8in]{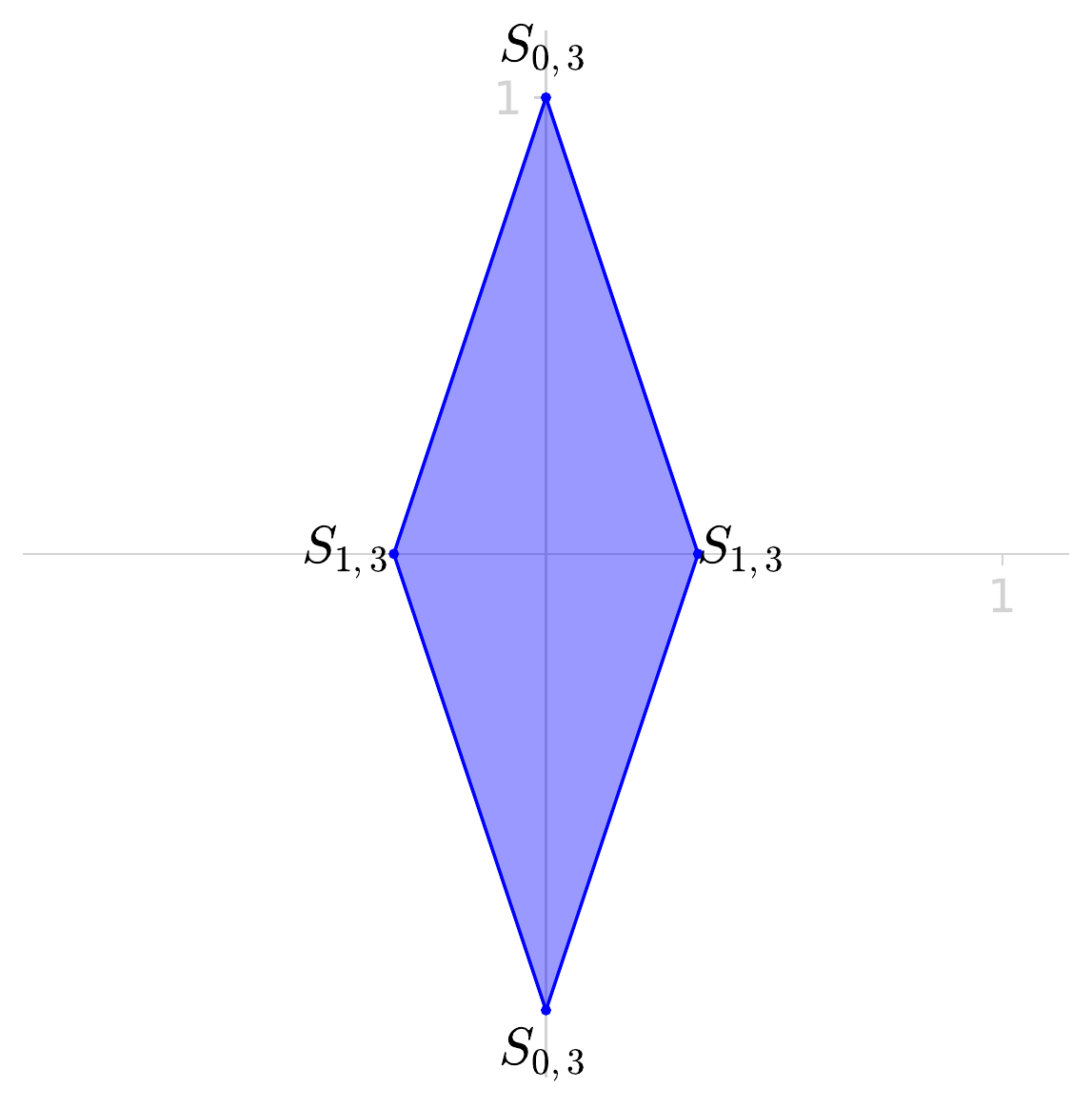}} & \quad & \multirow{6}{*}{\Includegraphics[width=1.8in]{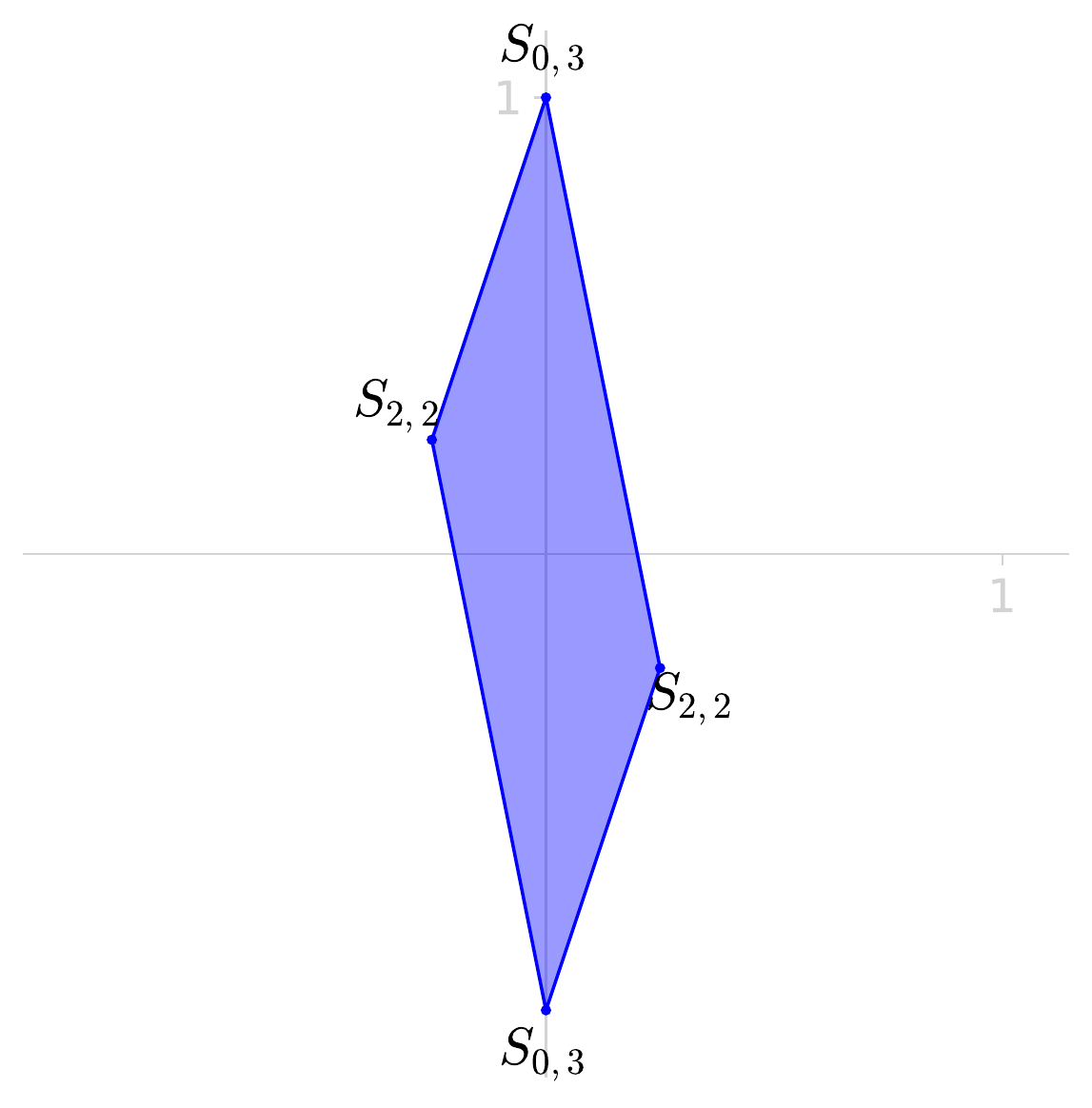}} \\ 
 $L=9^{{2}}_{{28}}$ & & $L=9^{{2}}_{{29}}$ & \\ 
 \quad & & \quad & \\ $\mathrm{Isom}(\mathbb{S}^3\setminus L) = \mathbb{{Z}}_2\oplus\mathbb{{Z}}_2$ & & $\mathrm{Isom}(\mathbb{S}^3\setminus L) = \mathbb{{Z}}_2$ & \\ 
 \quad & & \quad & \\ 
 \includegraphics[width=1in]{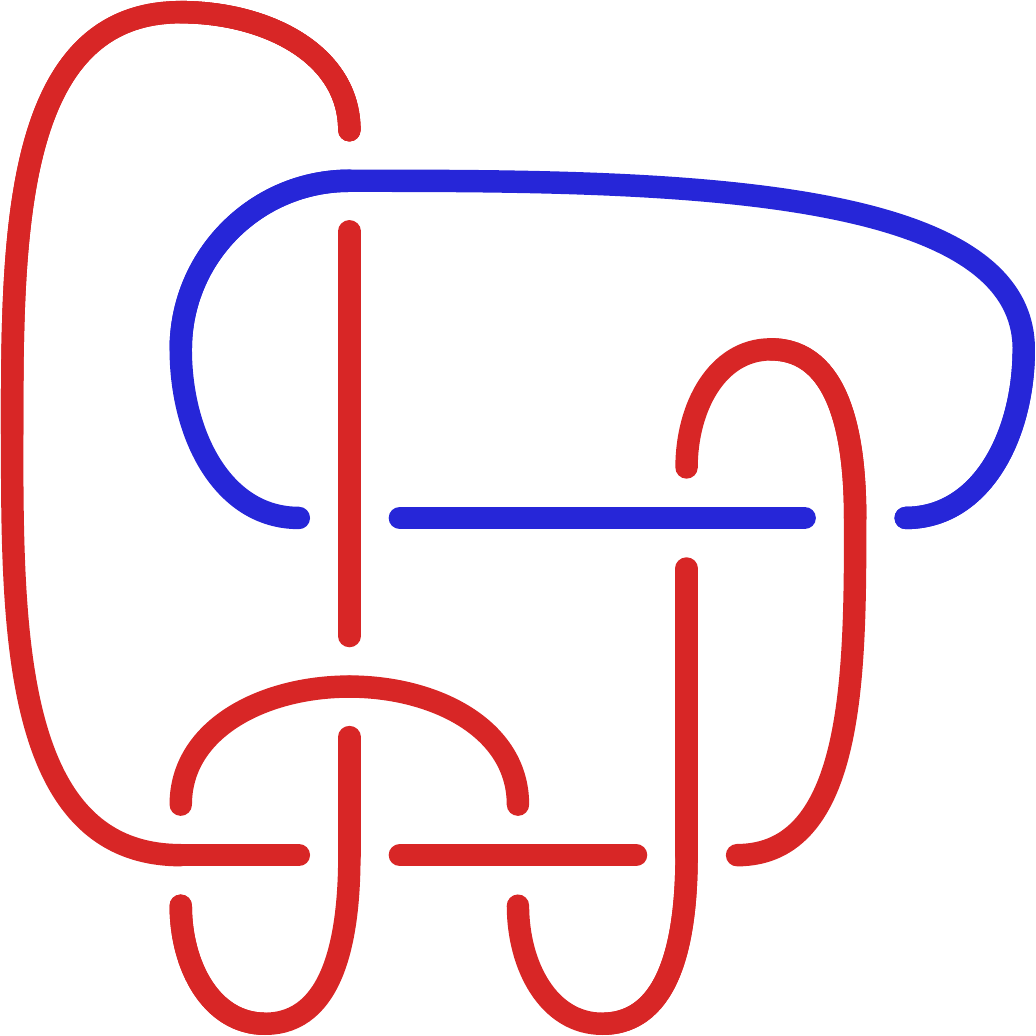}  & & \includegraphics[width=1in]{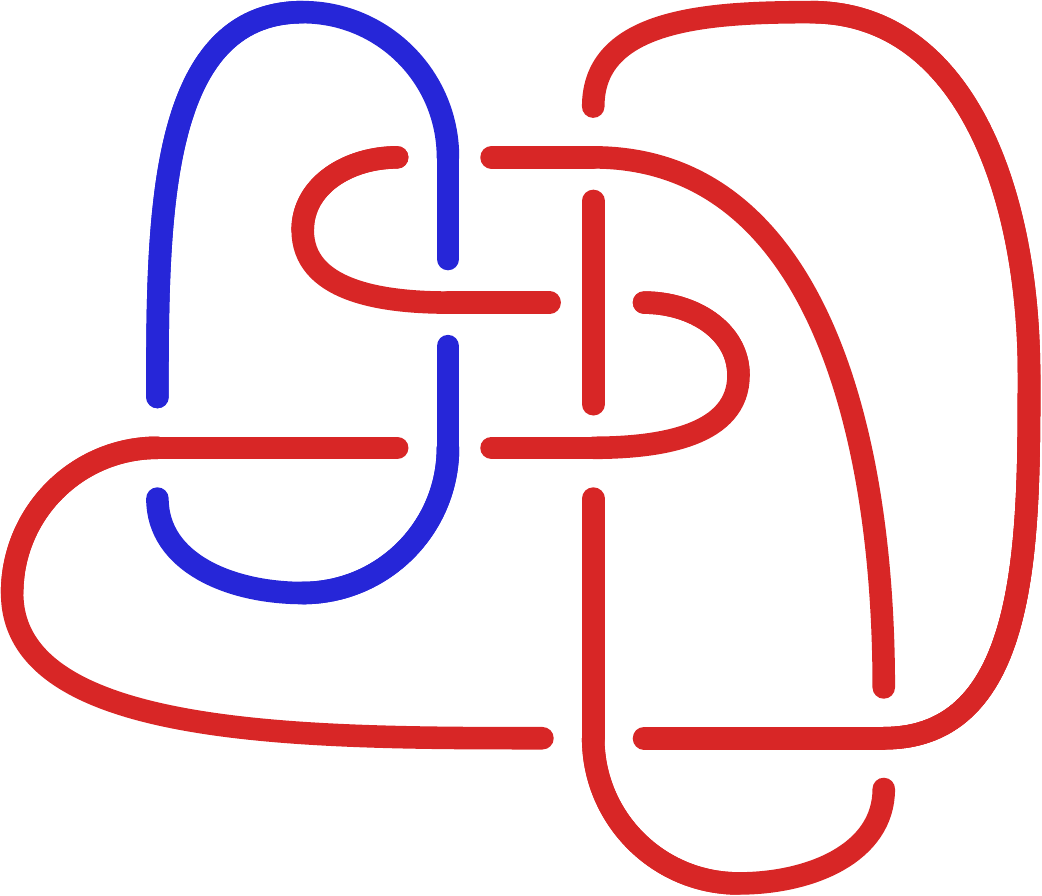} & \\ 
 \quad & & \quad & \\ 
 \hline  
\quad & \multirow{6}{*}{\Includegraphics[width=1.8in]{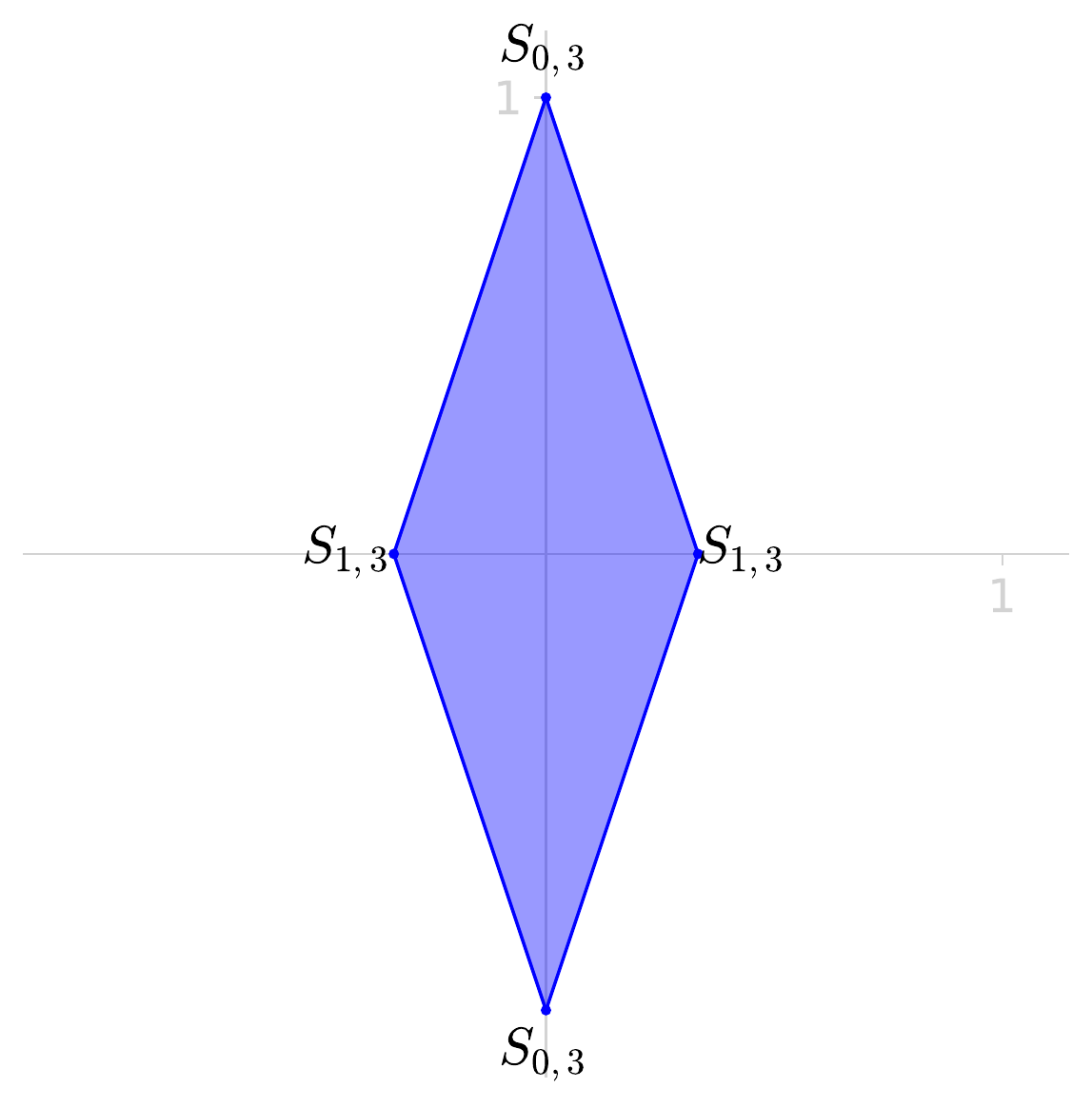}} & \quad & \multirow{6}{*}{\Includegraphics[width=1.8in]{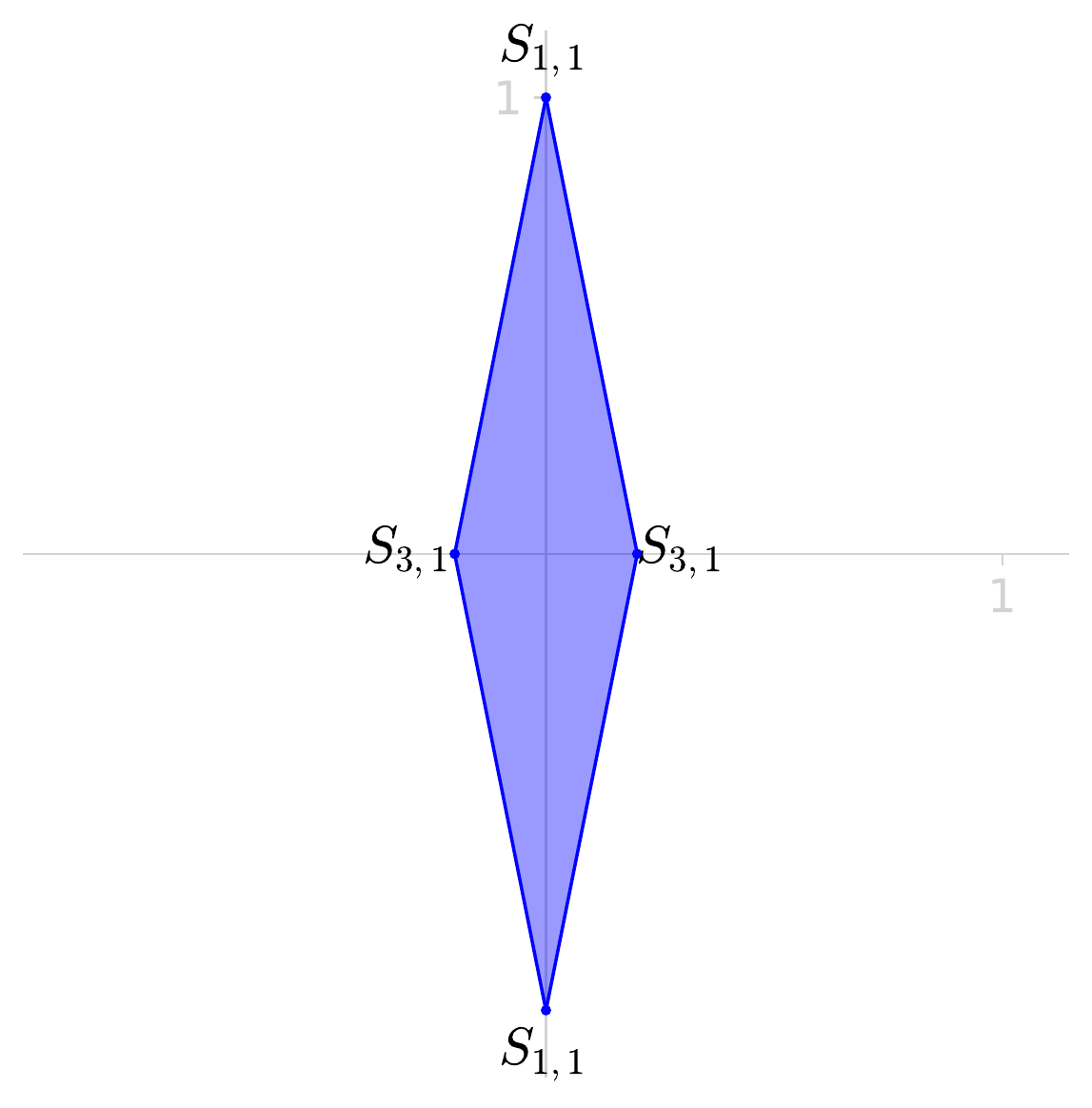}} \\ 
 $L=9^{{2}}_{{30}}$ & & $L=9^{{2}}_{{31}}$ & \\ 
 \quad & & \quad & \\ $\mathrm{Isom}(\mathbb{S}^3\setminus L) = \mathbb{{Z}}_2$ & & $\mathrm{Isom}(\mathbb{S}^3\setminus L) = \mathbb{{Z}}_2\oplus\mathbb{{Z}}_2$ & \\ 
 \quad & & \quad & \\ 
 \includegraphics[width=1in]{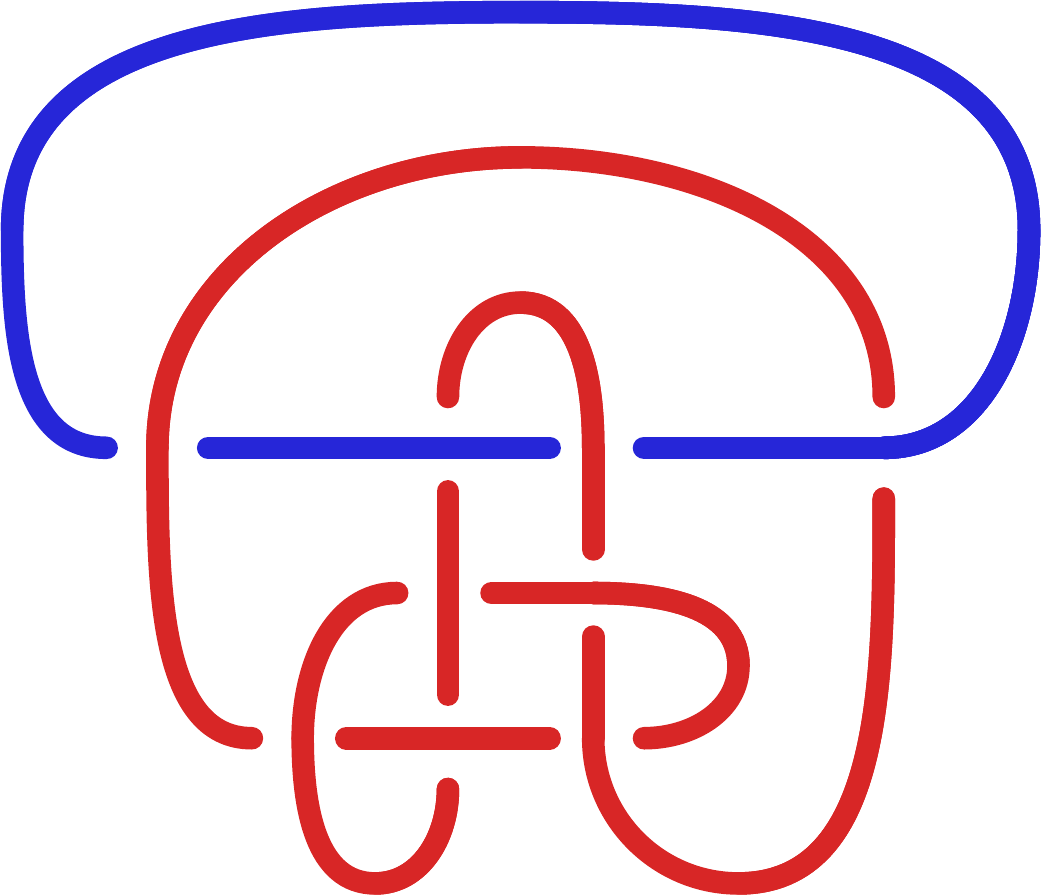}  & & \includegraphics[width=1in]{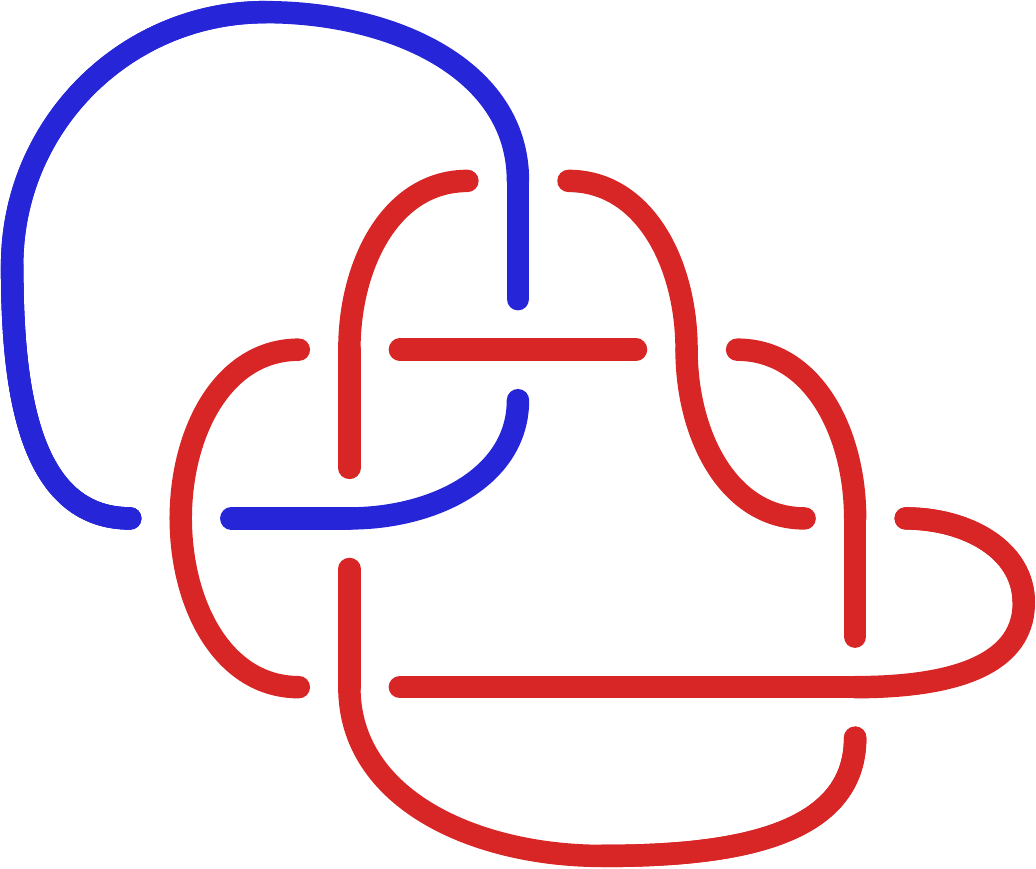} & \\ 
 \quad & & \quad & \\ 
 \hline  
\end{tabular} 
 \newpage \begin{tabular}{|c|c|c|c|} 
 \hline 
 Link & Norm Ball & Link & Norm Ball \\ 
 \hline 
\quad & \multirow{6}{*}{\Includegraphics[width=1.8in]{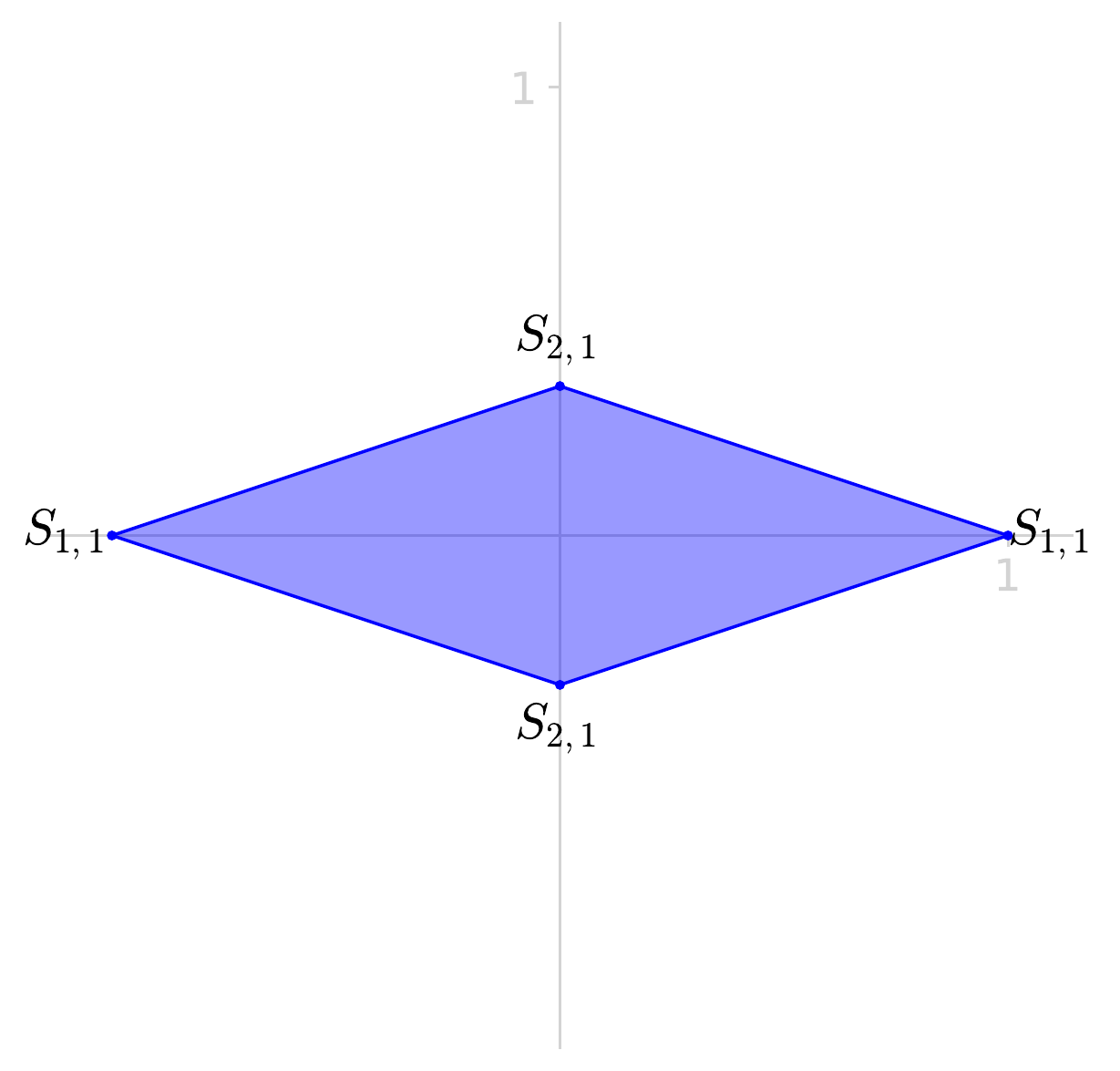}} & \quad & \multirow{6}{*}{\Includegraphics[width=1.8in]{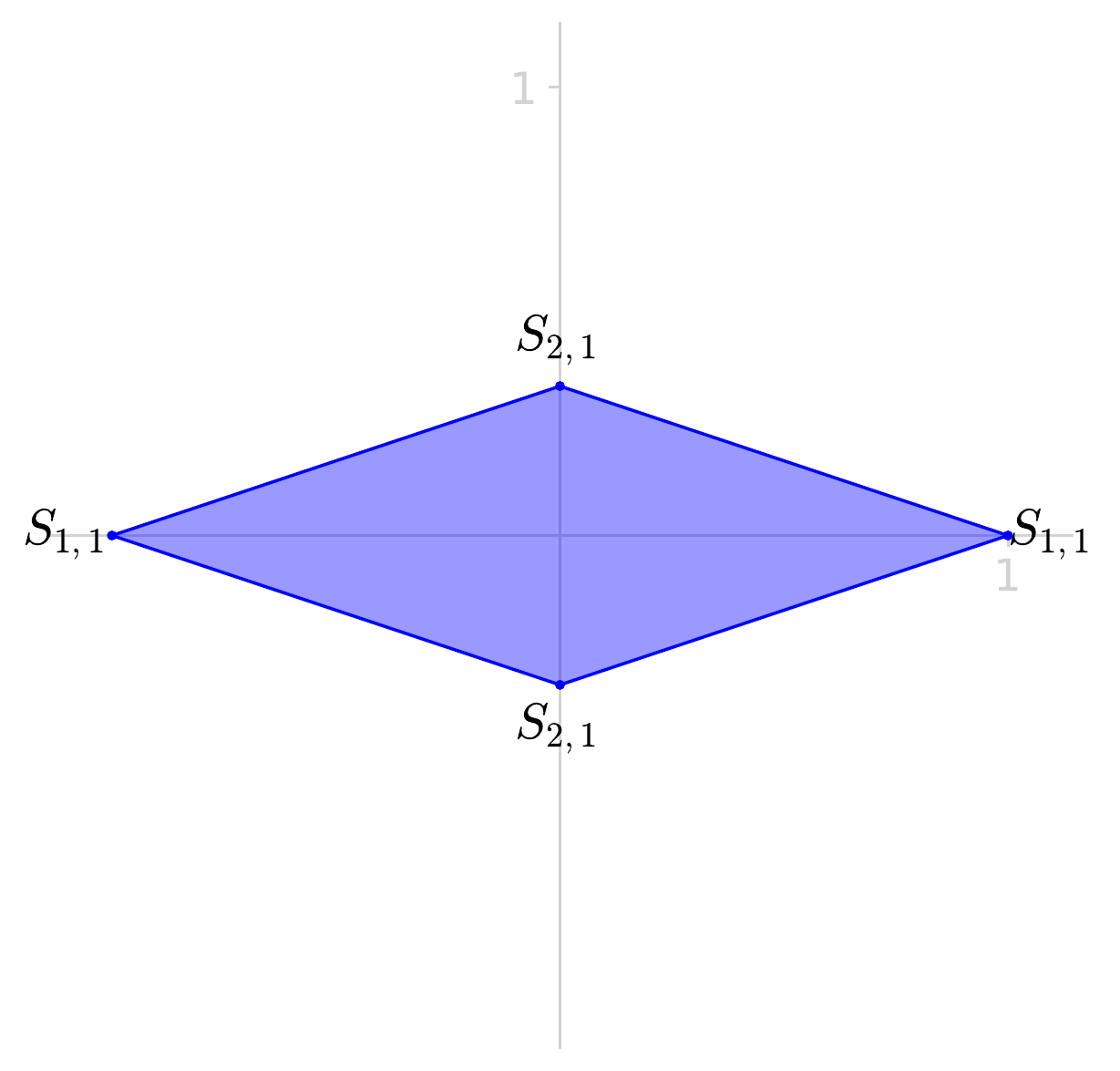}} \\ 
 $L=9^{{2}}_{{32}}$ & & $L=9^{{2}}_{{33}}$ & \\ 
 \quad & & \quad & \\ $\mathrm{Isom}(\mathbb{S}^3\setminus L) = \mathbb{{Z}}_2\oplus\mathbb{{Z}}_2$ & & $\mathrm{Isom}(\mathbb{S}^3\setminus L) = \mathbb{{Z}}_2\oplus\mathbb{{Z}}_2$ & \\ 
 \quad & & \quad & \\ 
 \includegraphics[width=1in]{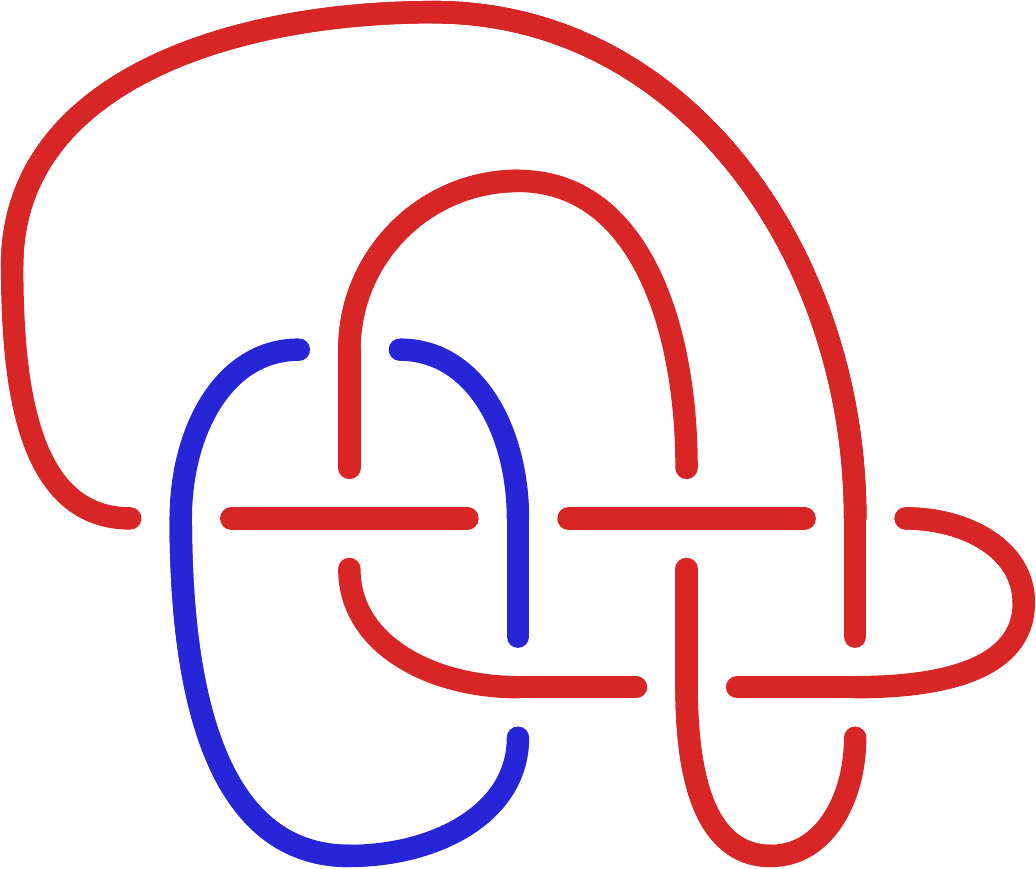}  & & \includegraphics[width=1in]{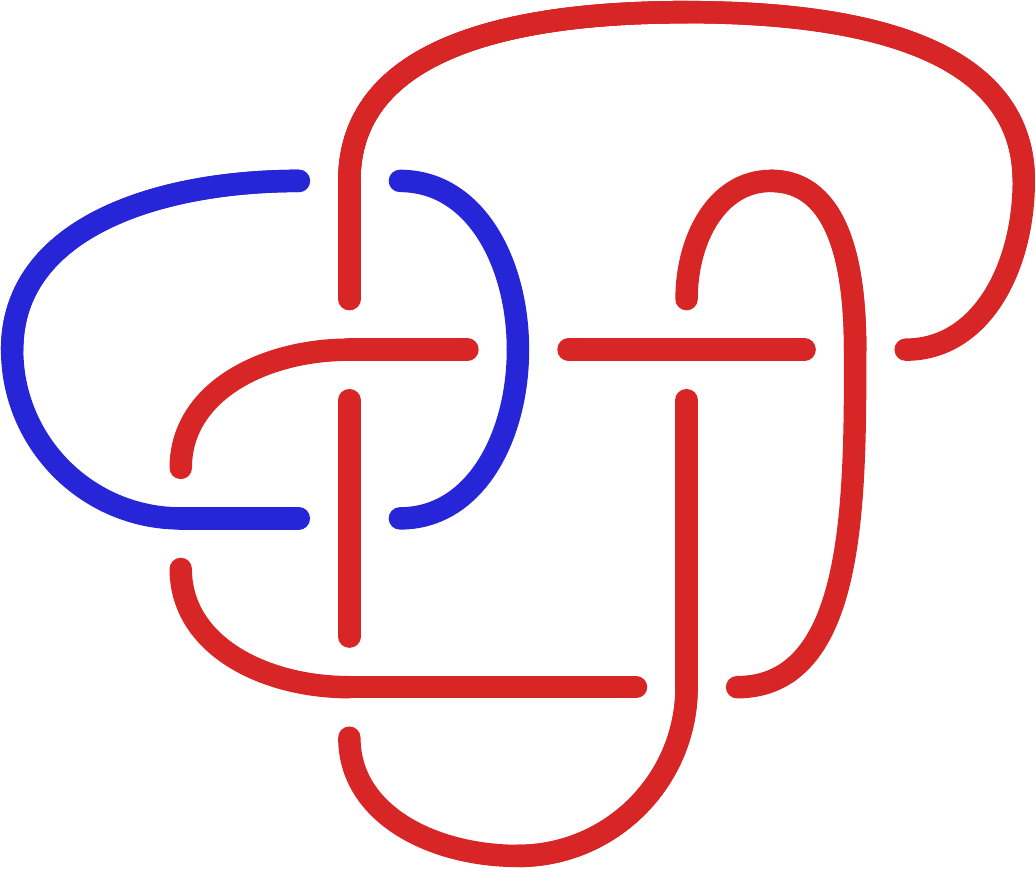} & \\ 
 \quad & & \quad & \\ 
 \hline  
\quad & \multirow{6}{*}{\Includegraphics[width=1.8in]{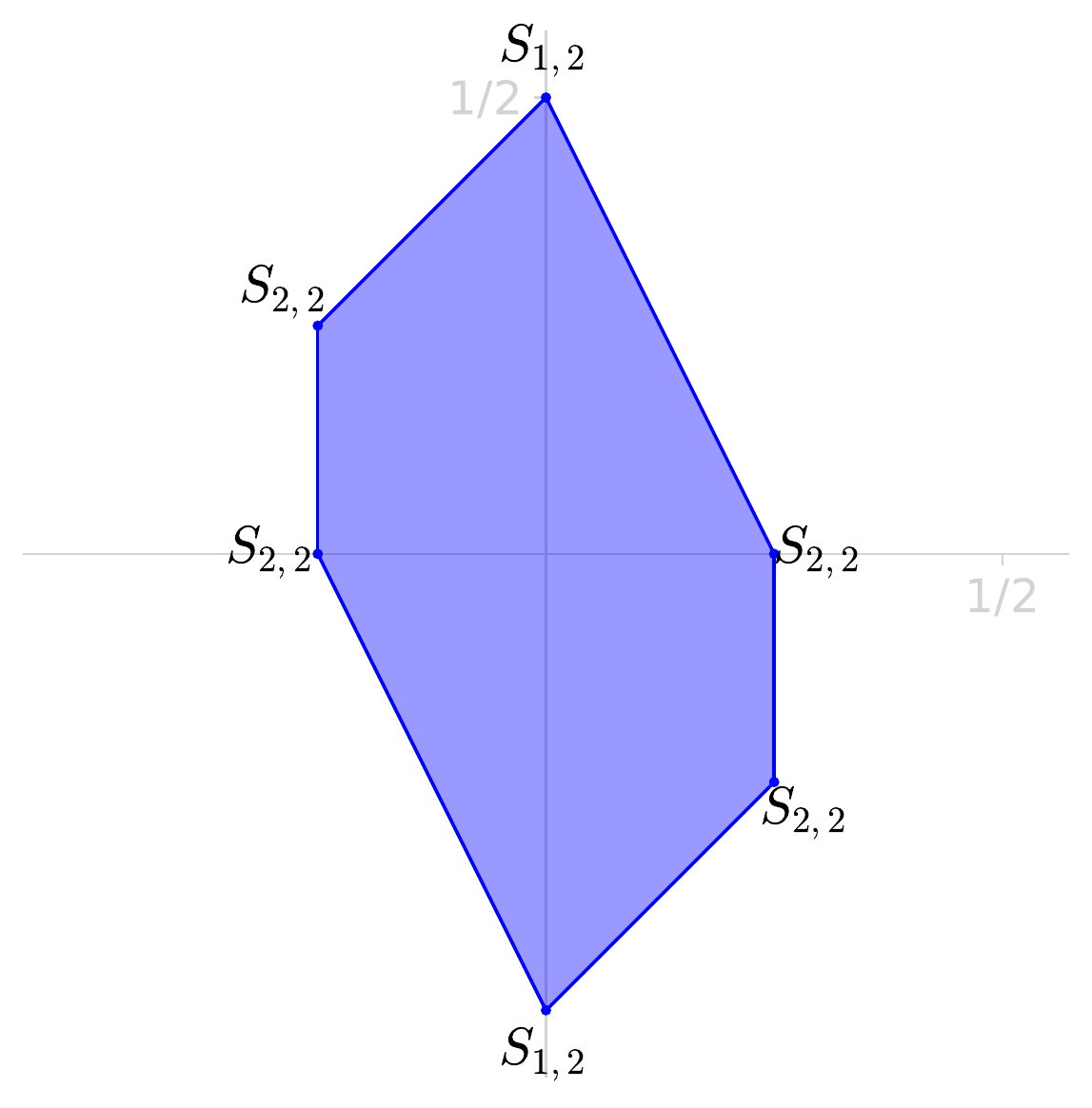}} & \quad & \multirow{6}{*}{\Includegraphics[width=1.8in]{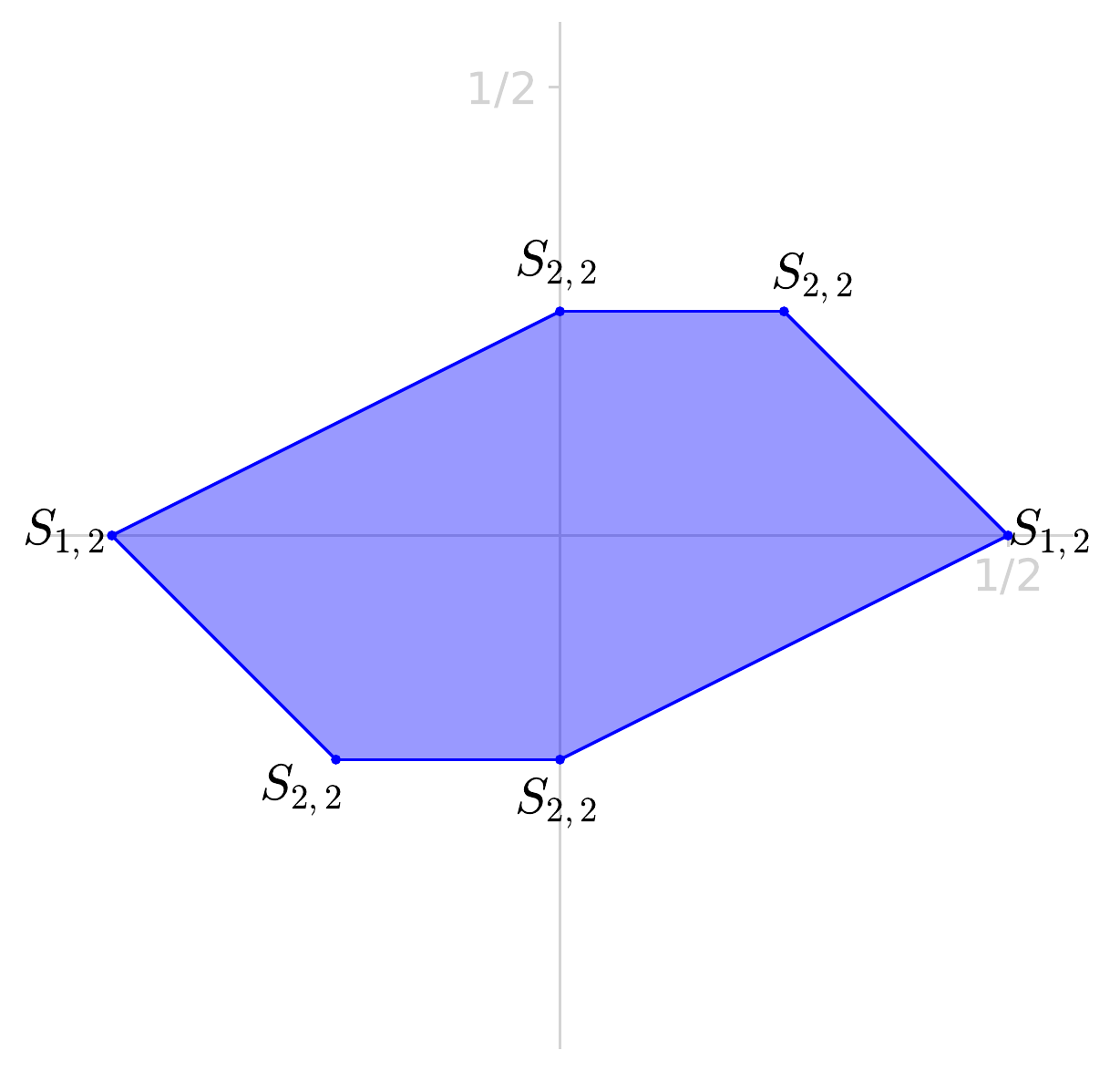}} \\ 
 $L=9^{{2}}_{{34}}$ & & $L=9^{{2}}_{{35}}$ & \\ 
 \quad & & \quad & \\ $\mathrm{Isom}(\mathbb{S}^3\setminus L) = 0$ & & $\mathrm{Isom}(\mathbb{S}^3\setminus L) = 0$ & \\ 
 \quad & & \quad & \\ 
 \includegraphics[width=1in]{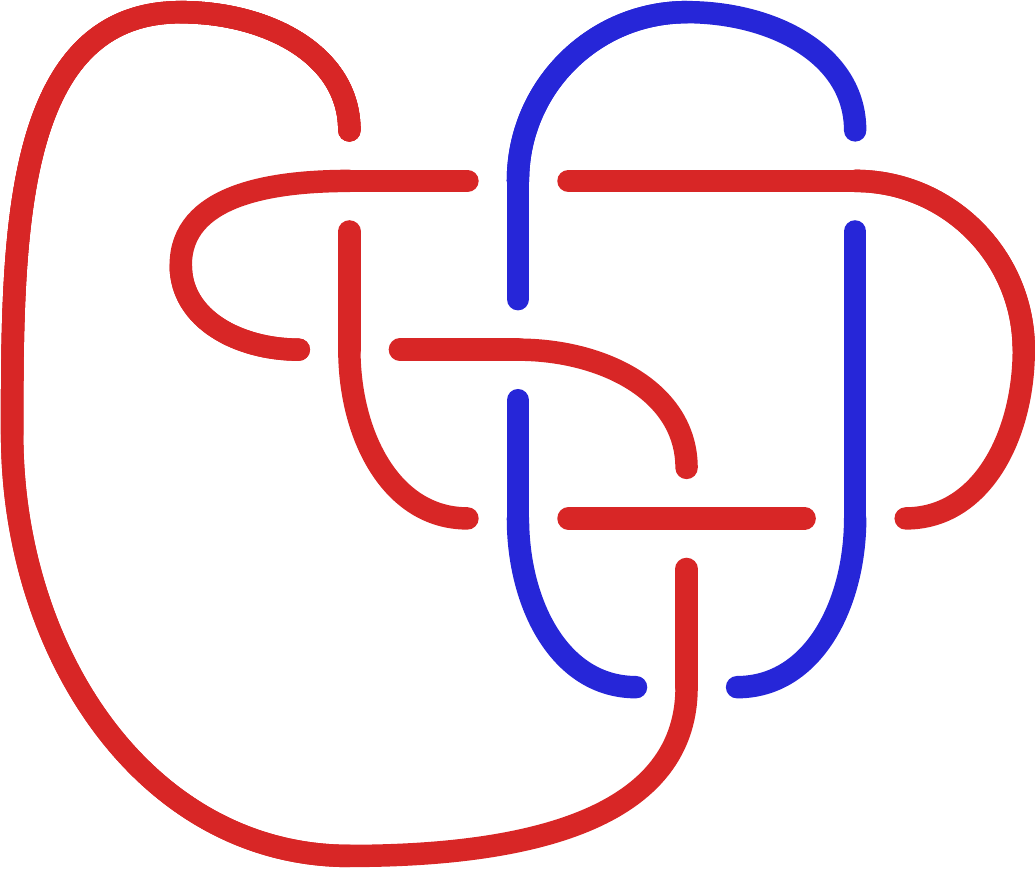}  & & \includegraphics[width=1in]{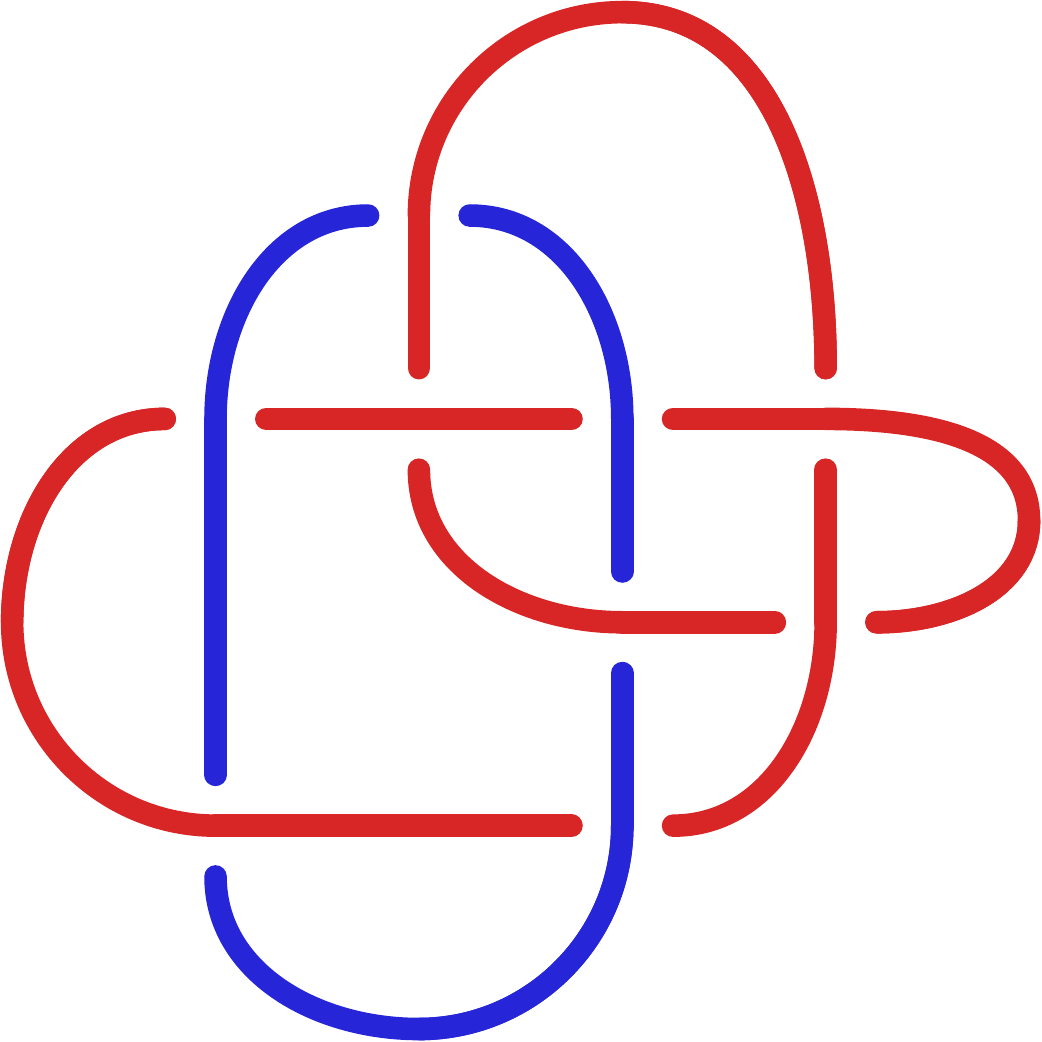} & \\ 
 \quad & & \quad & \\ 
 \hline  
\quad & \multirow{6}{*}{\Includegraphics[width=1.8in]{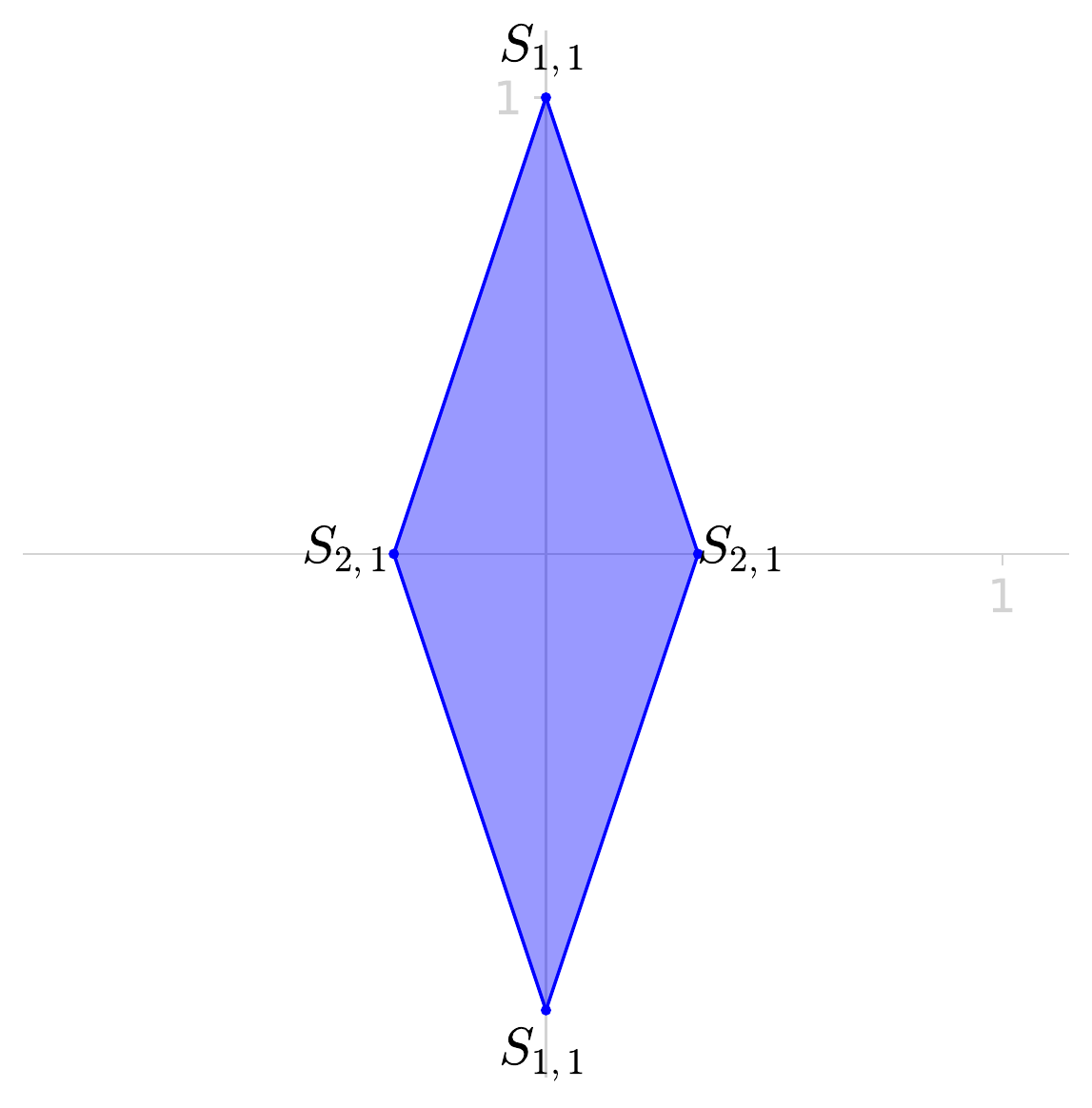}} & \quad & \multirow{6}{*}{\Includegraphics[width=1.8in]{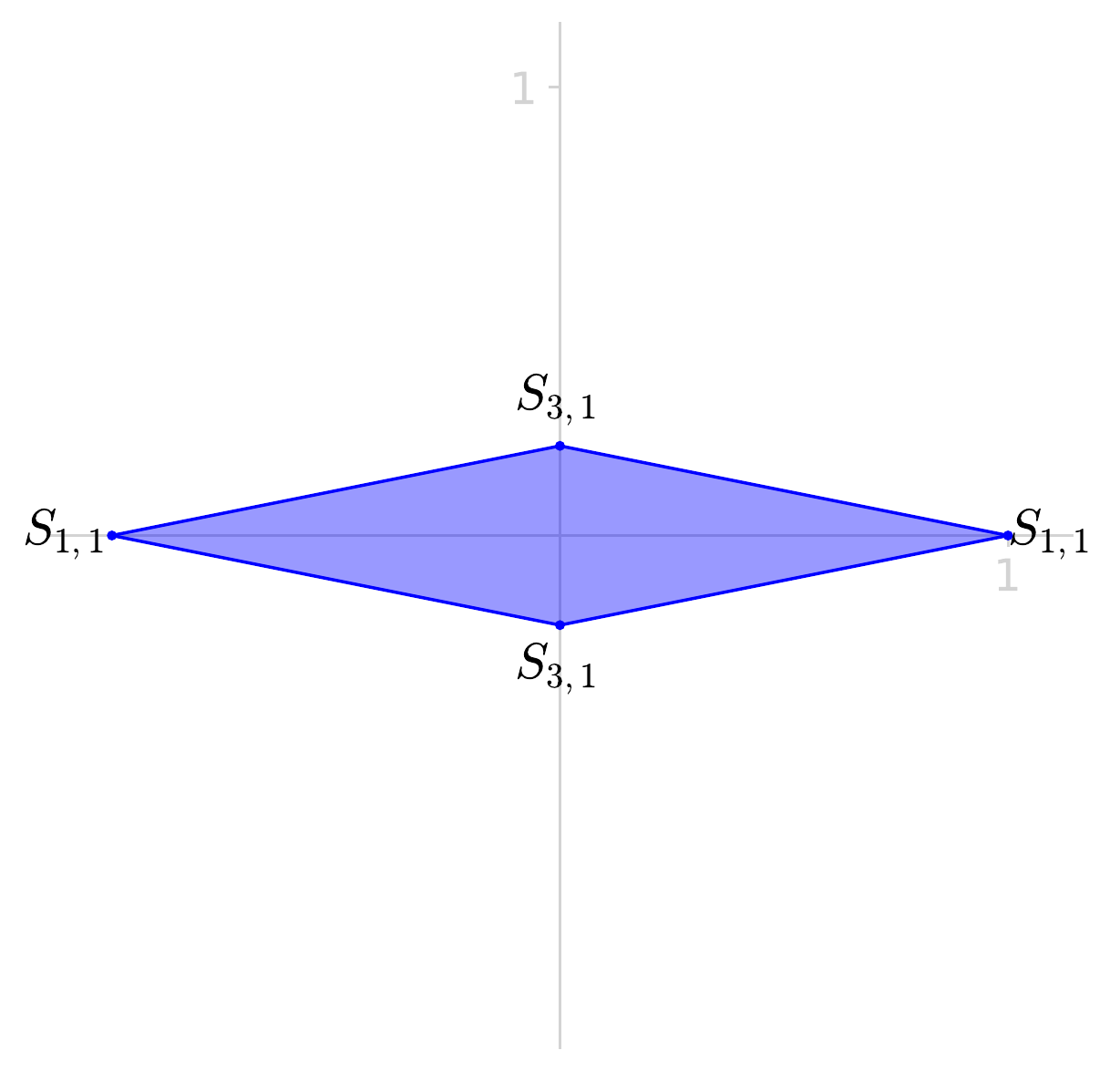}} \\ 
 $L=9^{{2}}_{{36}}$ & & $L=9^{{2}}_{{37}}$ & \\ 
 \quad & & \quad & \\ $\mathrm{Isom}(\mathbb{S}^3\setminus L) = \mathbb{{Z}}_2\oplus\mathbb{{Z}}_2$ & & $\mathrm{Isom}(\mathbb{S}^3\setminus L) = \mathbb{{Z}}_2\oplus\mathbb{{Z}}_2$ & \\ 
 \quad & & \quad & \\ 
 \includegraphics[width=1in]{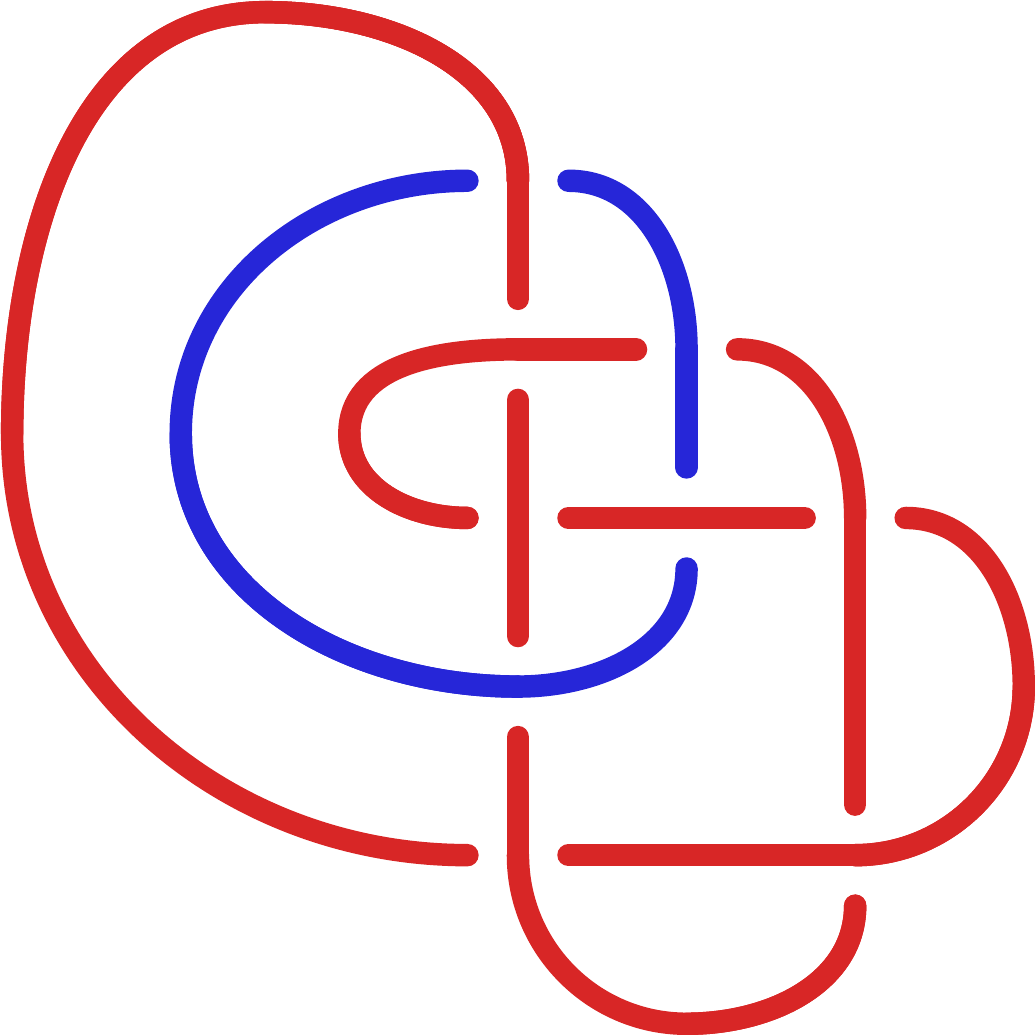}  & & \includegraphics[width=1in]{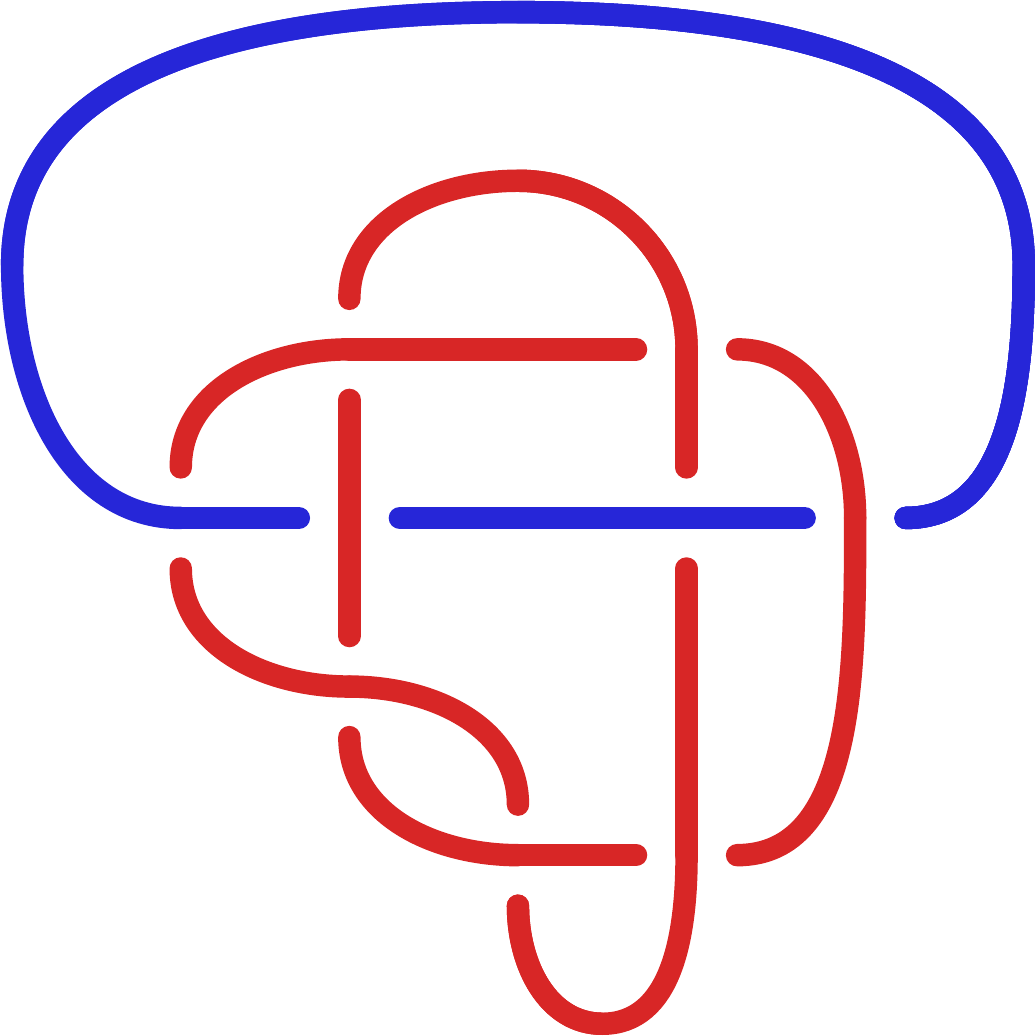} & \\ 
 \quad & & \quad & \\ 
 \hline  
\quad & \multirow{6}{*}{\Includegraphics[width=1.8in]{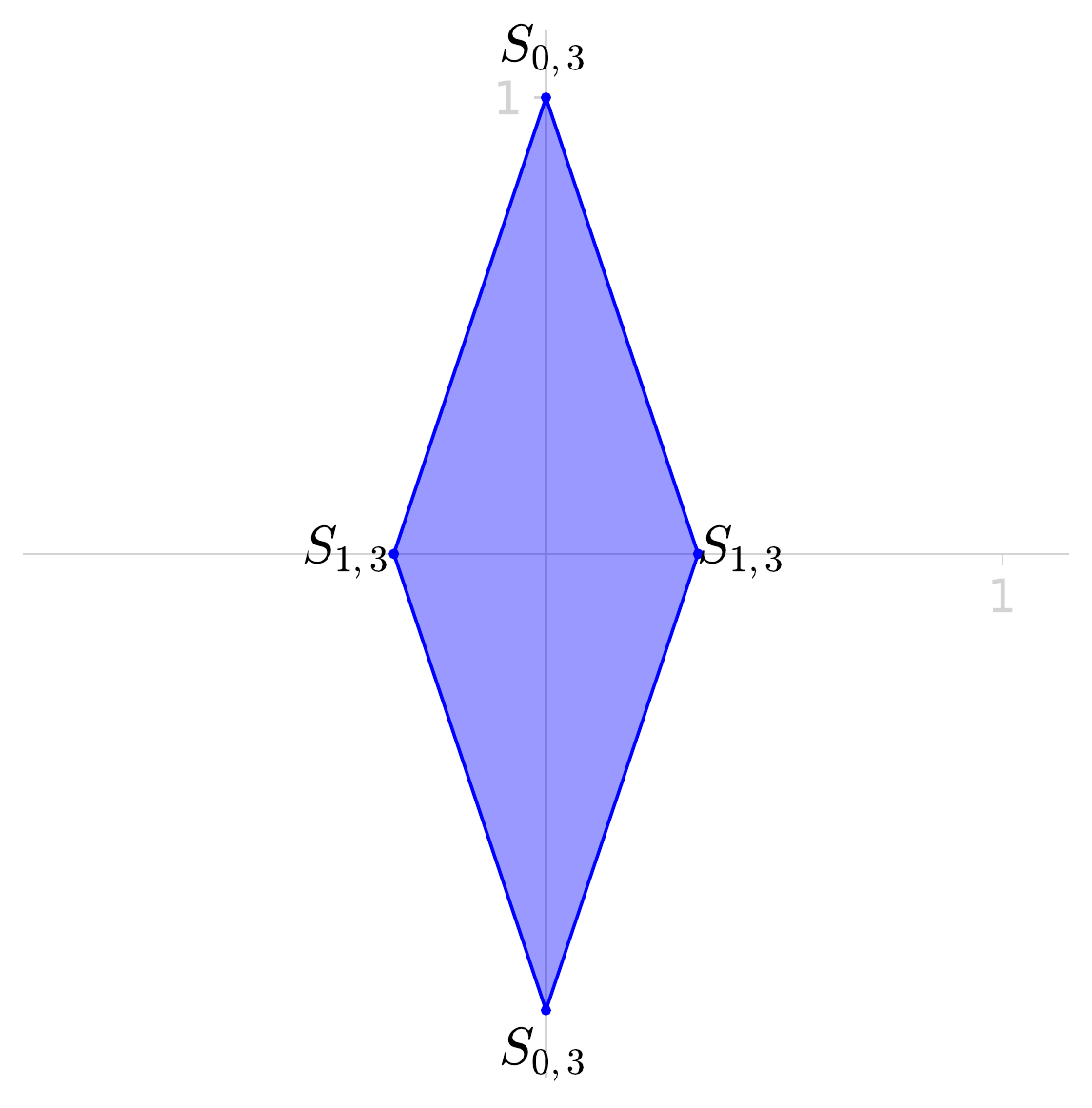}} & \quad & \multirow{6}{*}{\Includegraphics[width=1.8in]{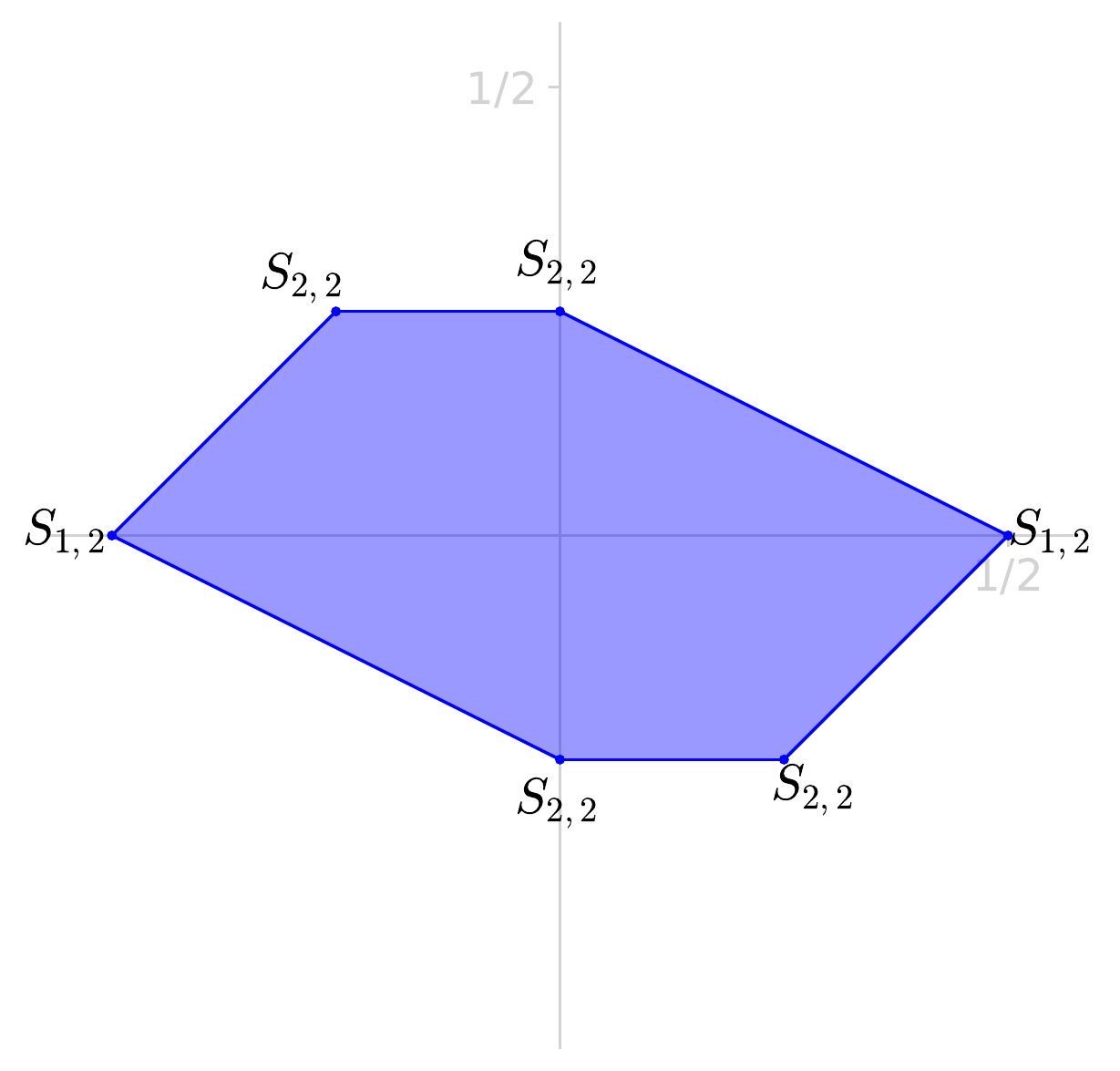}} \\ 
 $L=9^{{2}}_{{38}}$ & & $L=9^{{2}}_{{39}}$ & \\ 
 \quad & & \quad & \\ $\mathrm{Isom}(\mathbb{S}^3\setminus L) = \mathbb{{Z}}_2\oplus\mathbb{{Z}}_2$ & & $\mathrm{Isom}(\mathbb{S}^3\setminus L) = 0$ & \\ 
 \quad & & \quad & \\ 
 \includegraphics[width=1in]{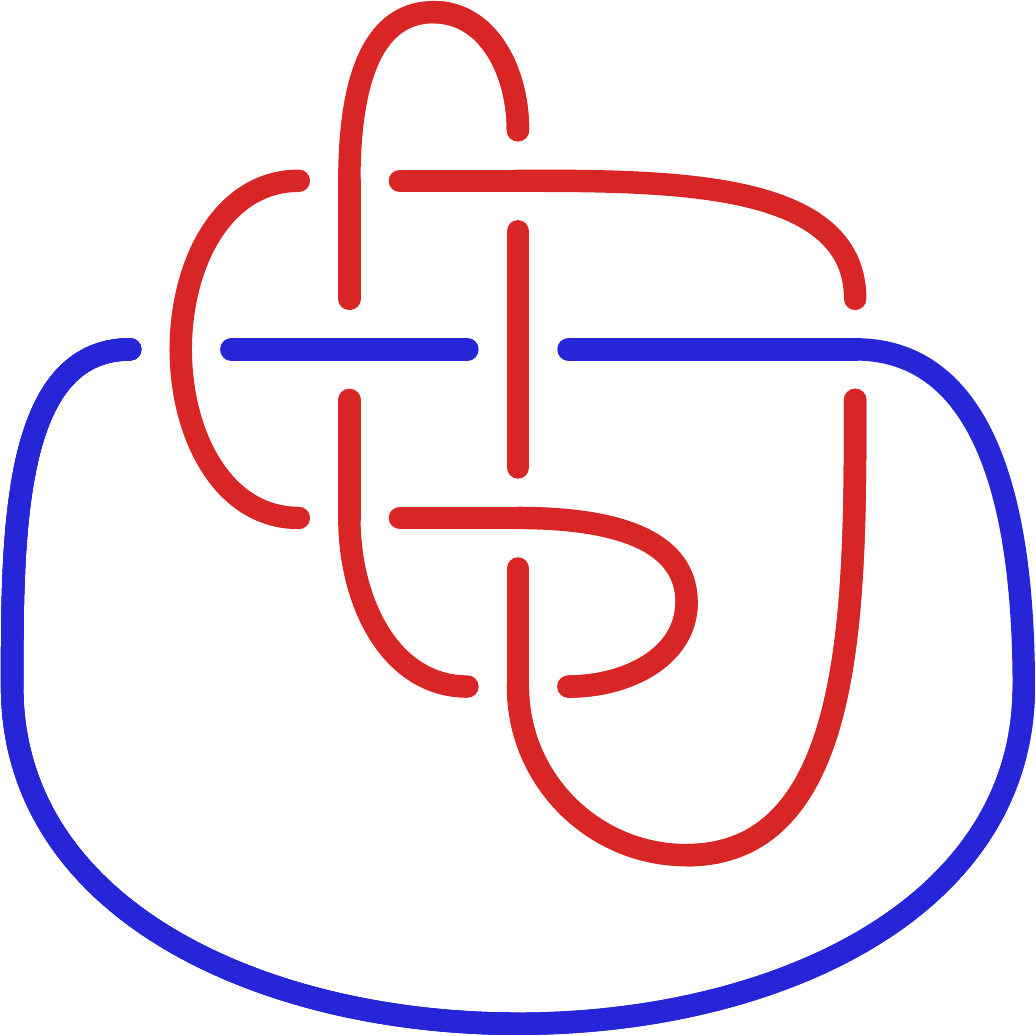}  & & \includegraphics[width=1in]{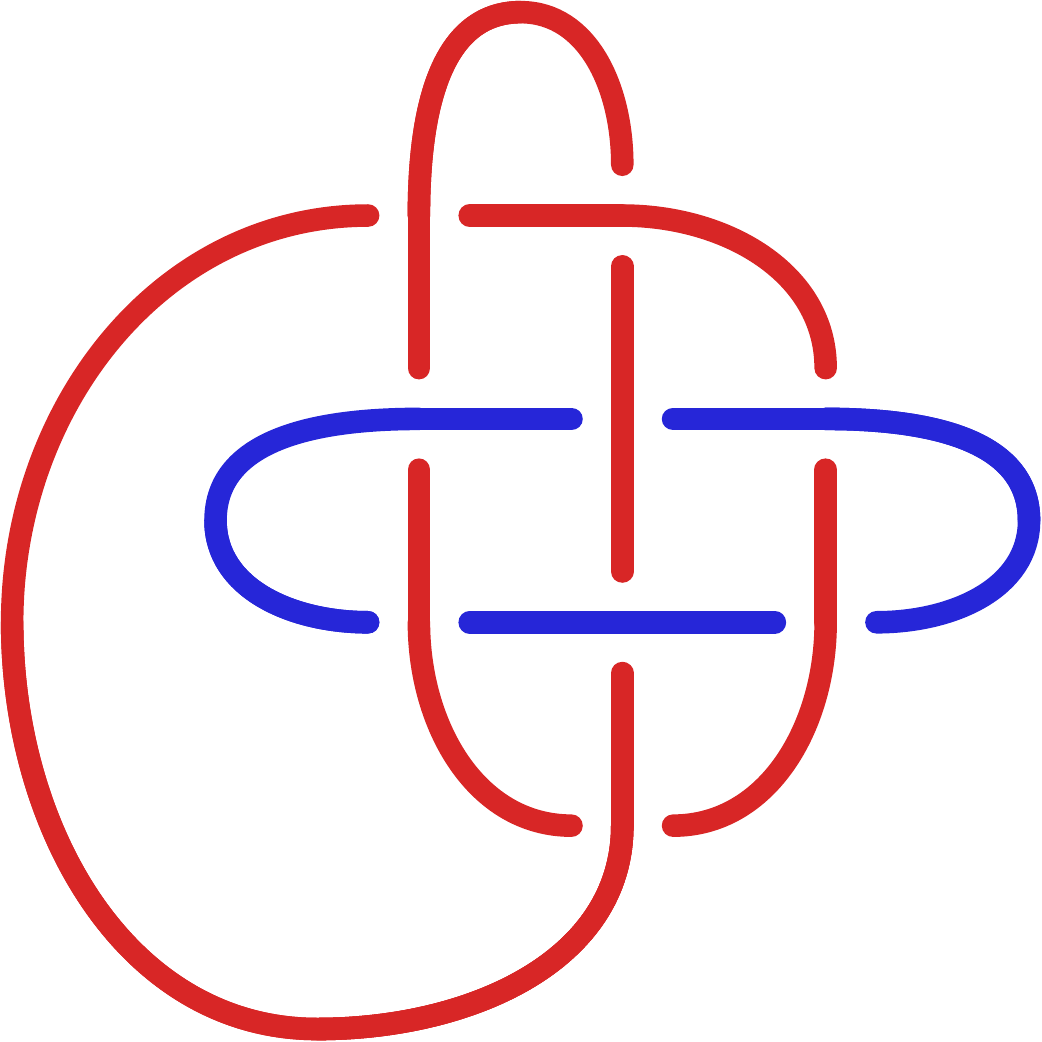} & \\ 
 \quad & & \quad & \\ 
 \hline  
\quad & \multirow{6}{*}{\Includegraphics[width=1.8in]{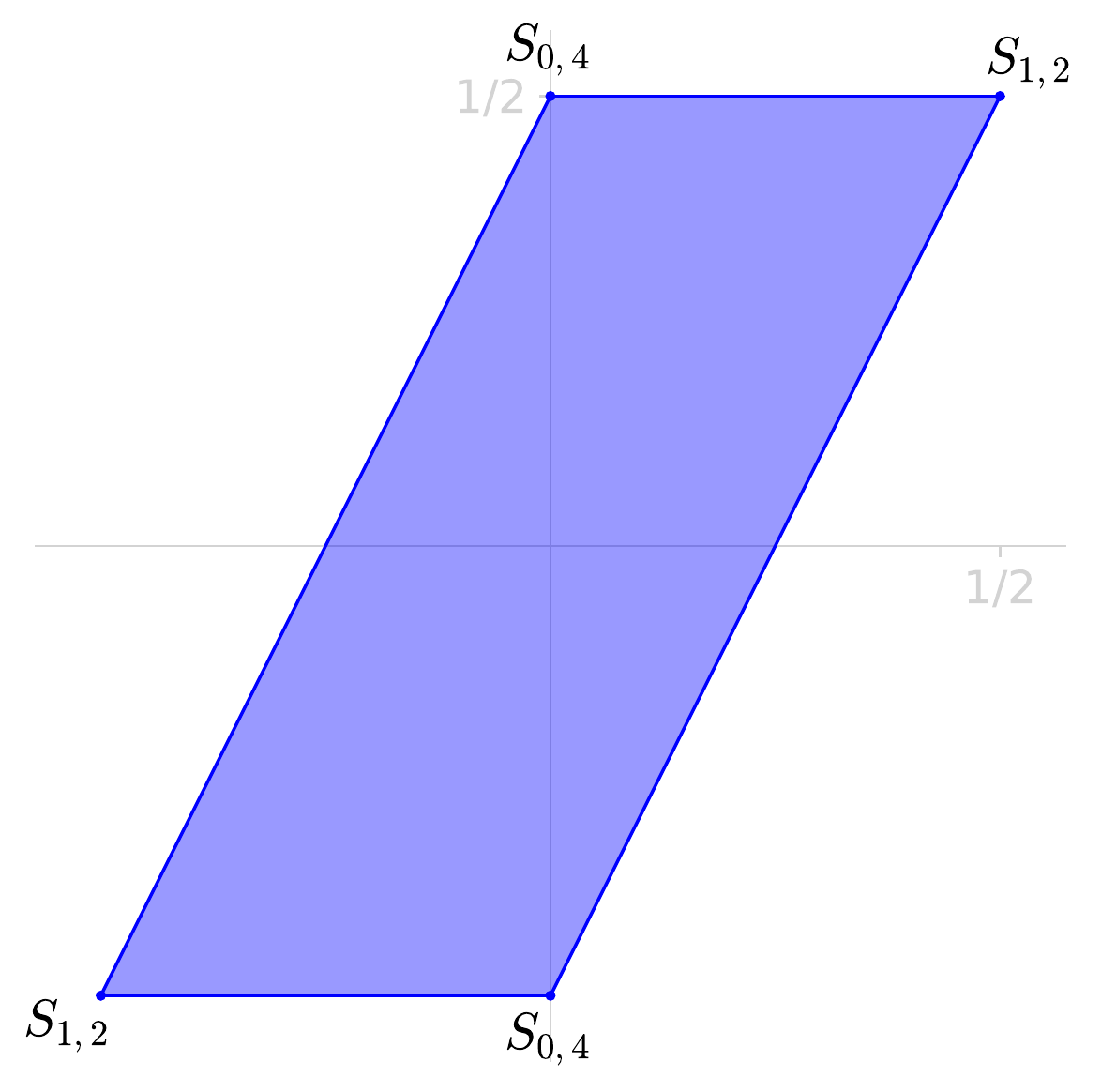}} & \quad & \multirow{6}{*}{\Includegraphics[width=1.8in]{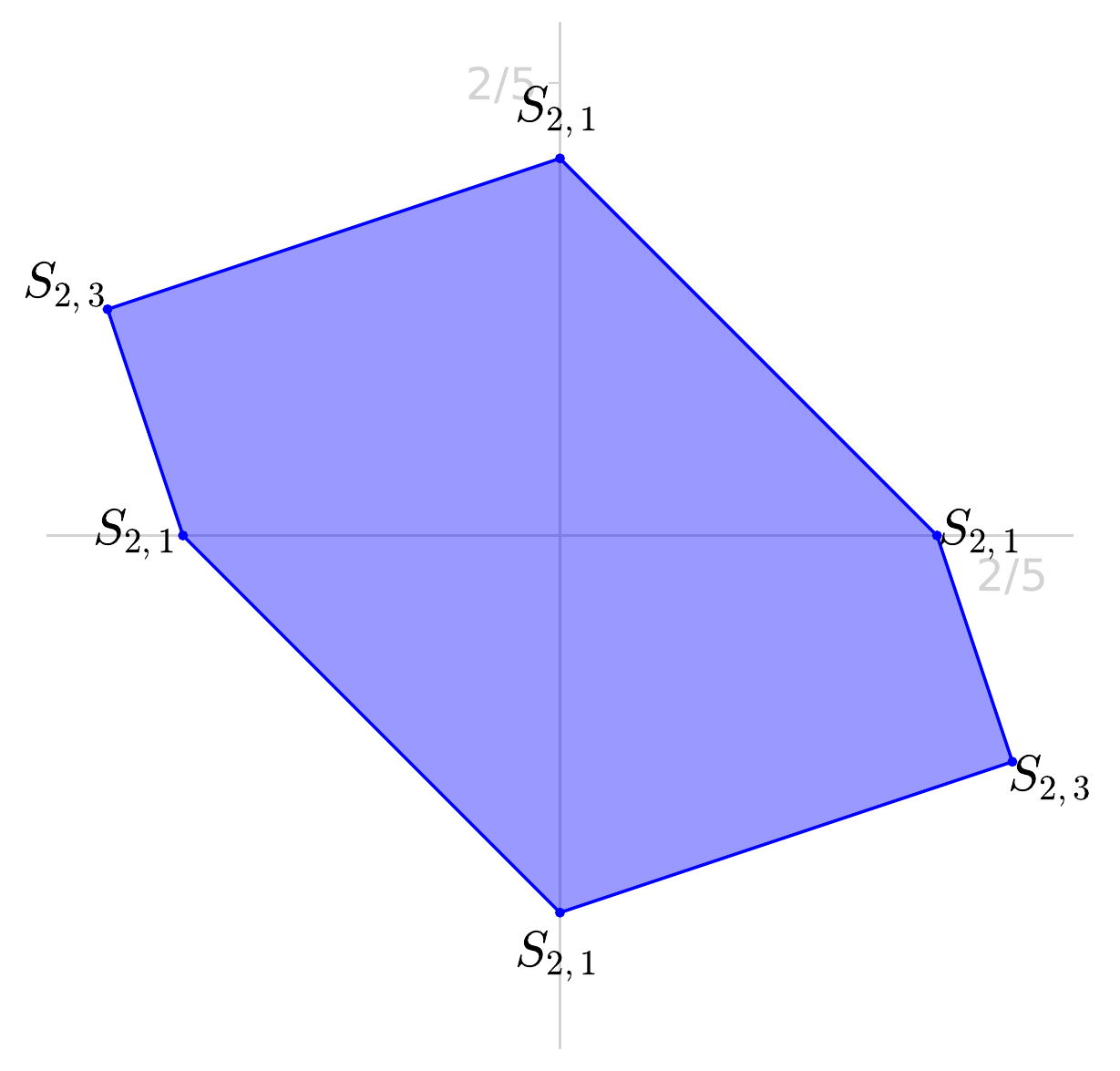}} \\ 
 $L=9^{{2}}_{{40}}$ & & $L=9^{{2}}_{{41}}$ & \\ 
 \quad & & \quad & \\ $\mathrm{Isom}(\mathbb{S}^3\setminus L) = D_3$ & & $\mathrm{Isom}(\mathbb{S}^3\setminus L) = \mathbb{{Z}}_2$ & \\ 
 \quad & & \quad & \\ 
 \includegraphics[width=1in]{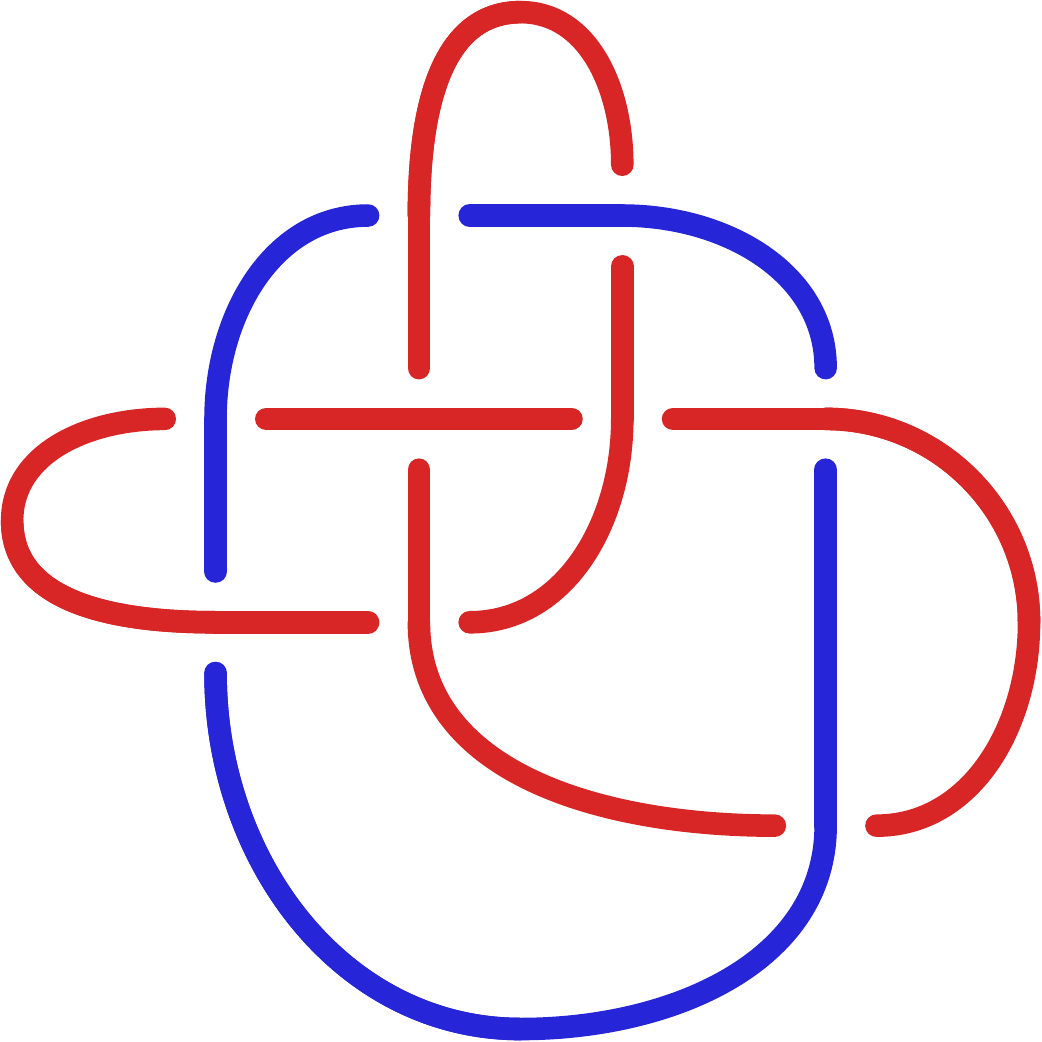}  & & \includegraphics[width=1in]{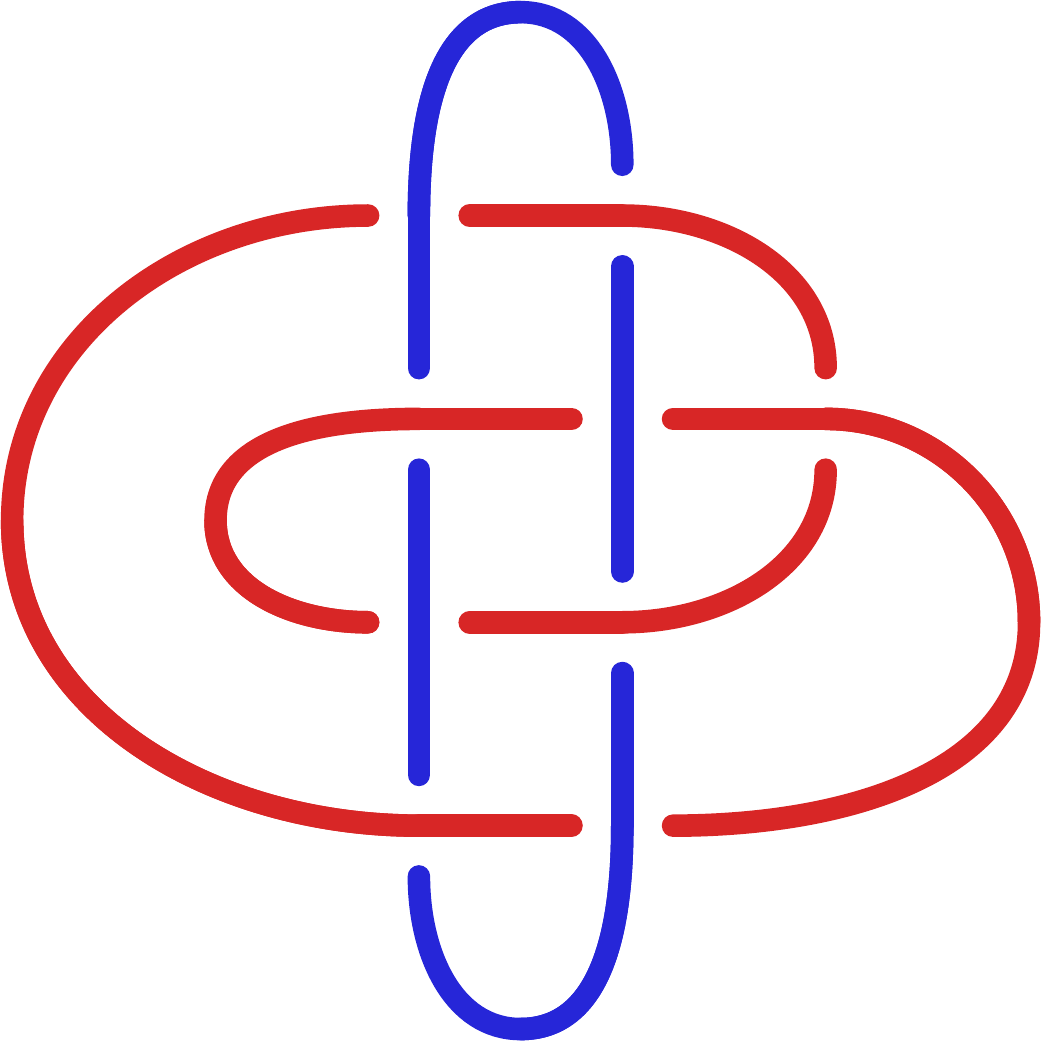} & \\ 
 \quad & & \quad & \\ 
 \hline  
\end{tabular} 
 \newpage \begin{tabular}{|c|c|c|c|} 
 \hline 
 Link & Norm Ball & Link & Norm Ball \\ 
 \hline 
\quad & \multirow{6}{*}{\Includegraphics[width=1.8in]{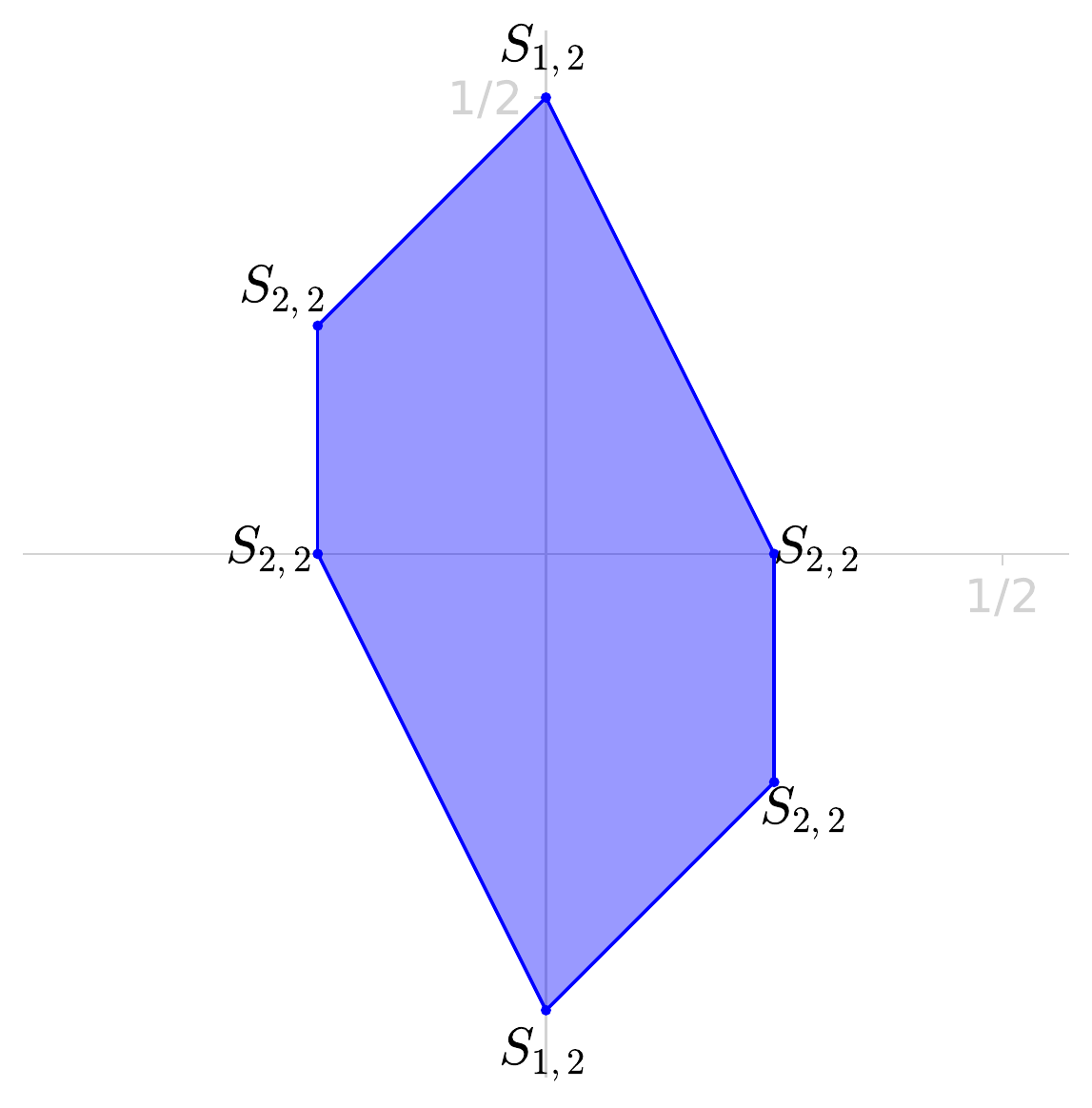}} & \quad & \multirow{6}{*}{\Includegraphics[width=1.8in]{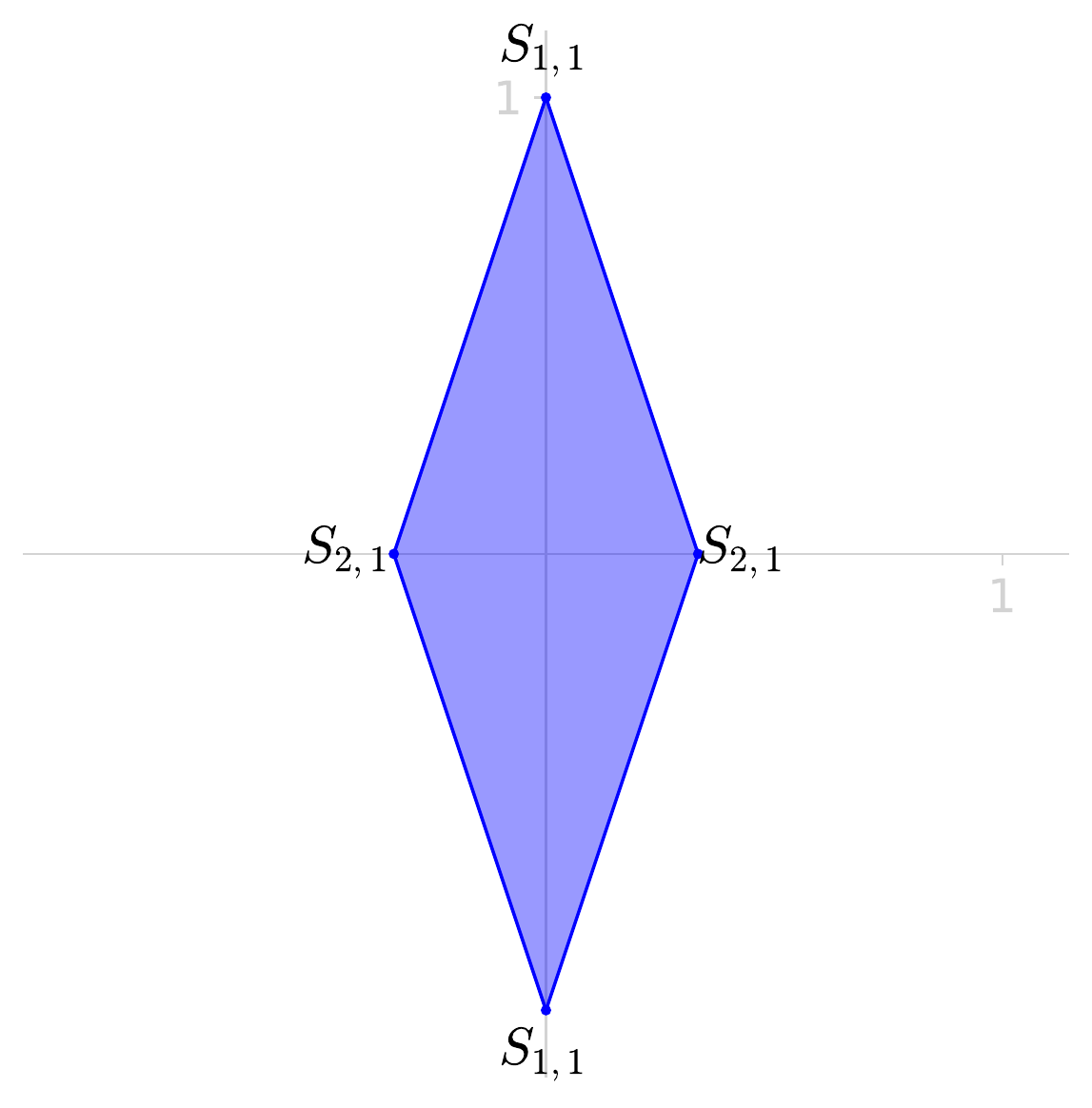}} \\ 
 $L=9^{{2}}_{{42}}$ & & $L=9^{{2}}_{{44}}$ & \\ 
 \quad & & \quad & \\ $\mathrm{Isom}(\mathbb{S}^3\setminus L) = \mathbb{{Z}}_2$ & & $\mathrm{Isom}(\mathbb{S}^3\setminus L) = \mathbb{{Z}}_2\oplus\mathbb{{Z}}_2$ & \\ 
 \quad & & \quad & \\ 
 \includegraphics[width=1in]{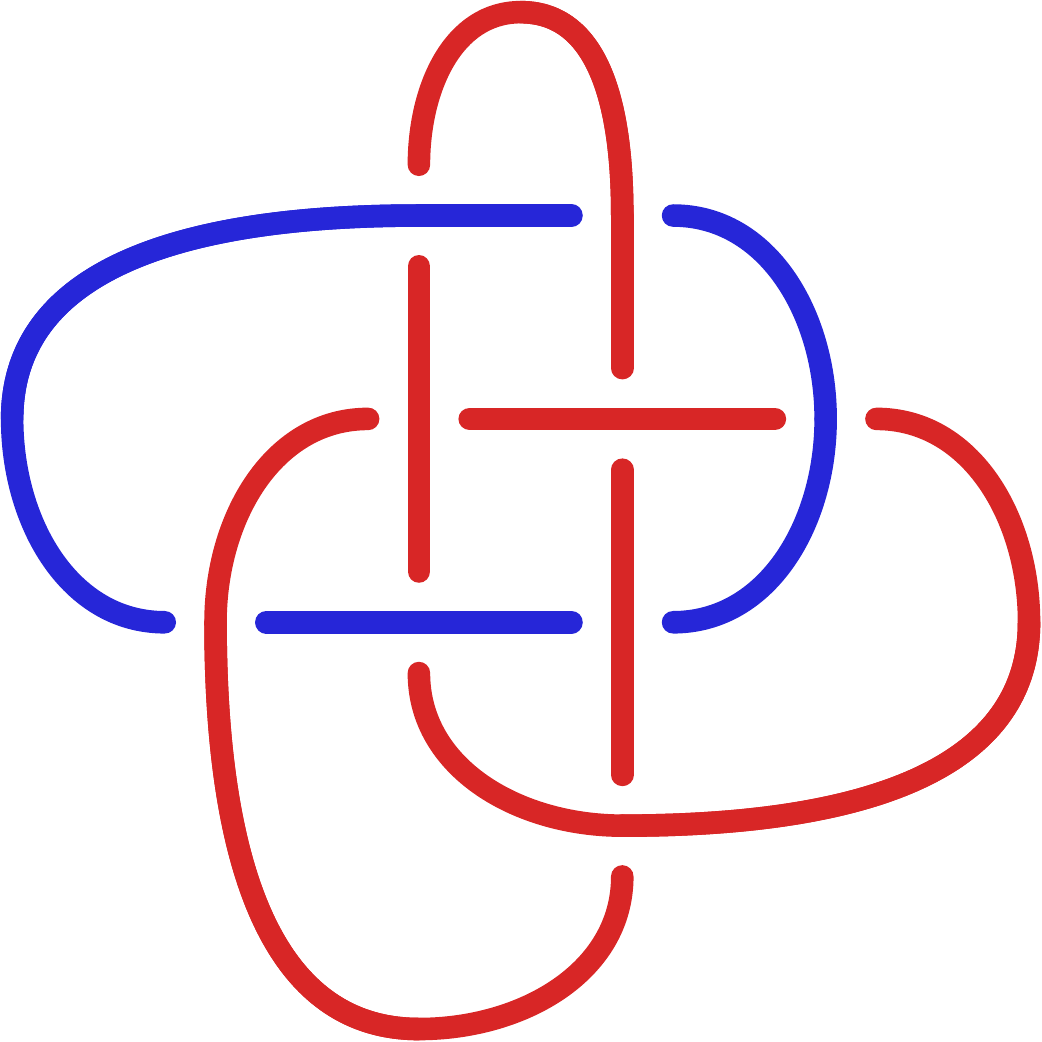}  & & \includegraphics[width=1in]{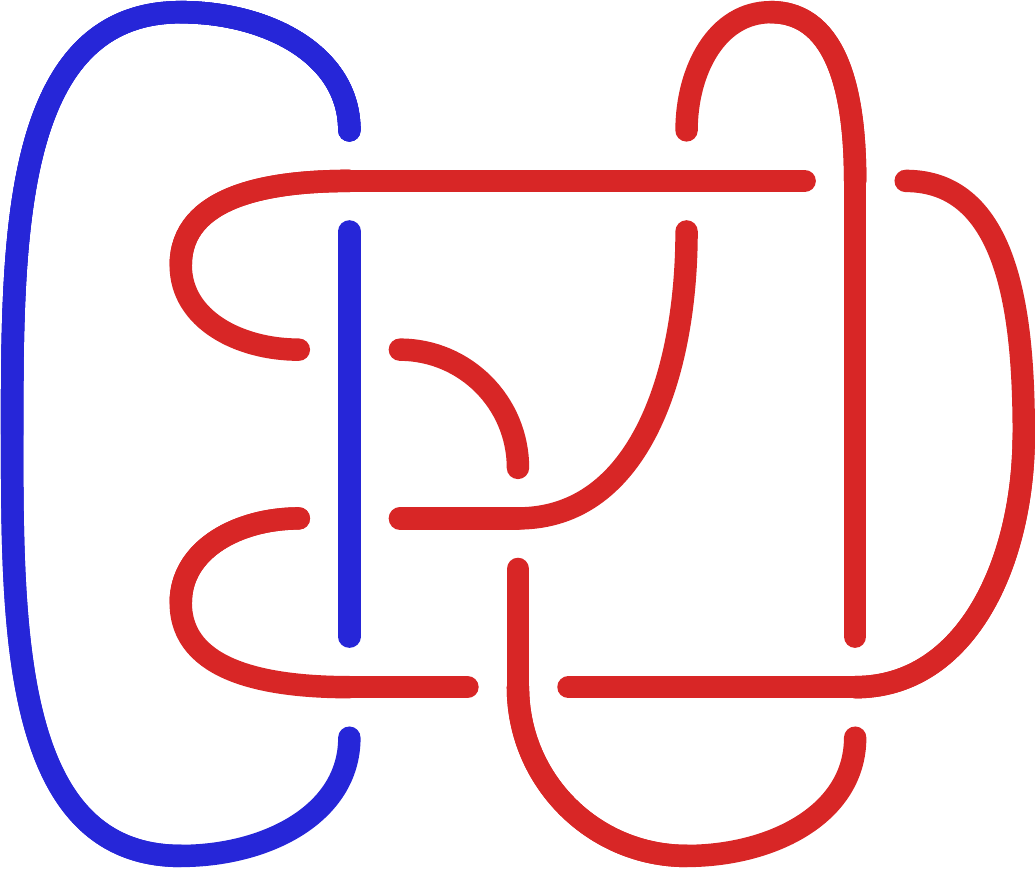} & \\ 
 \quad & & \quad & \\ 
 \hline  
\quad & \multirow{6}{*}{\Includegraphics[width=1.8in]{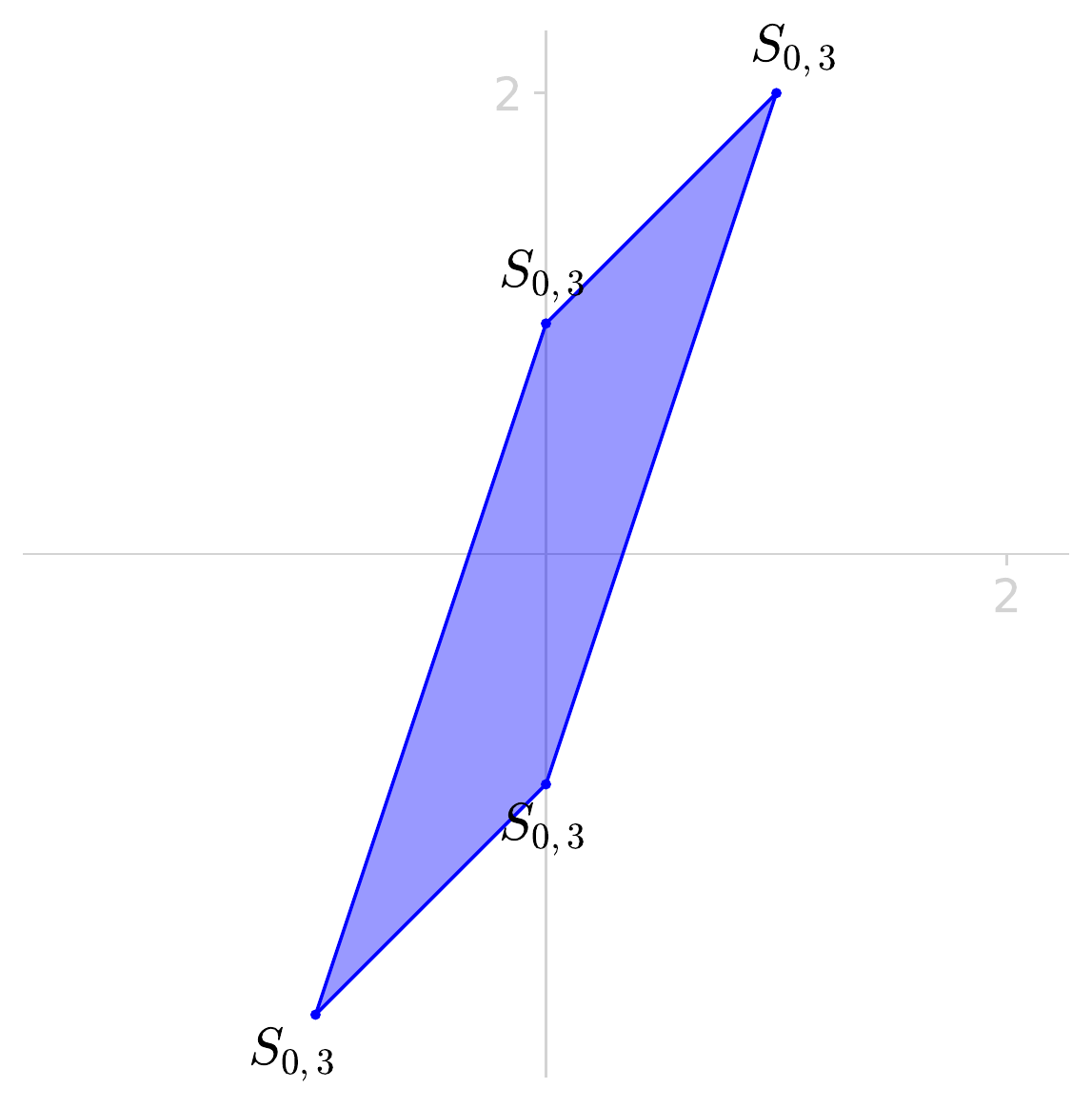}} & \quad & \multirow{6}{*}{\Includegraphics[width=1.8in]{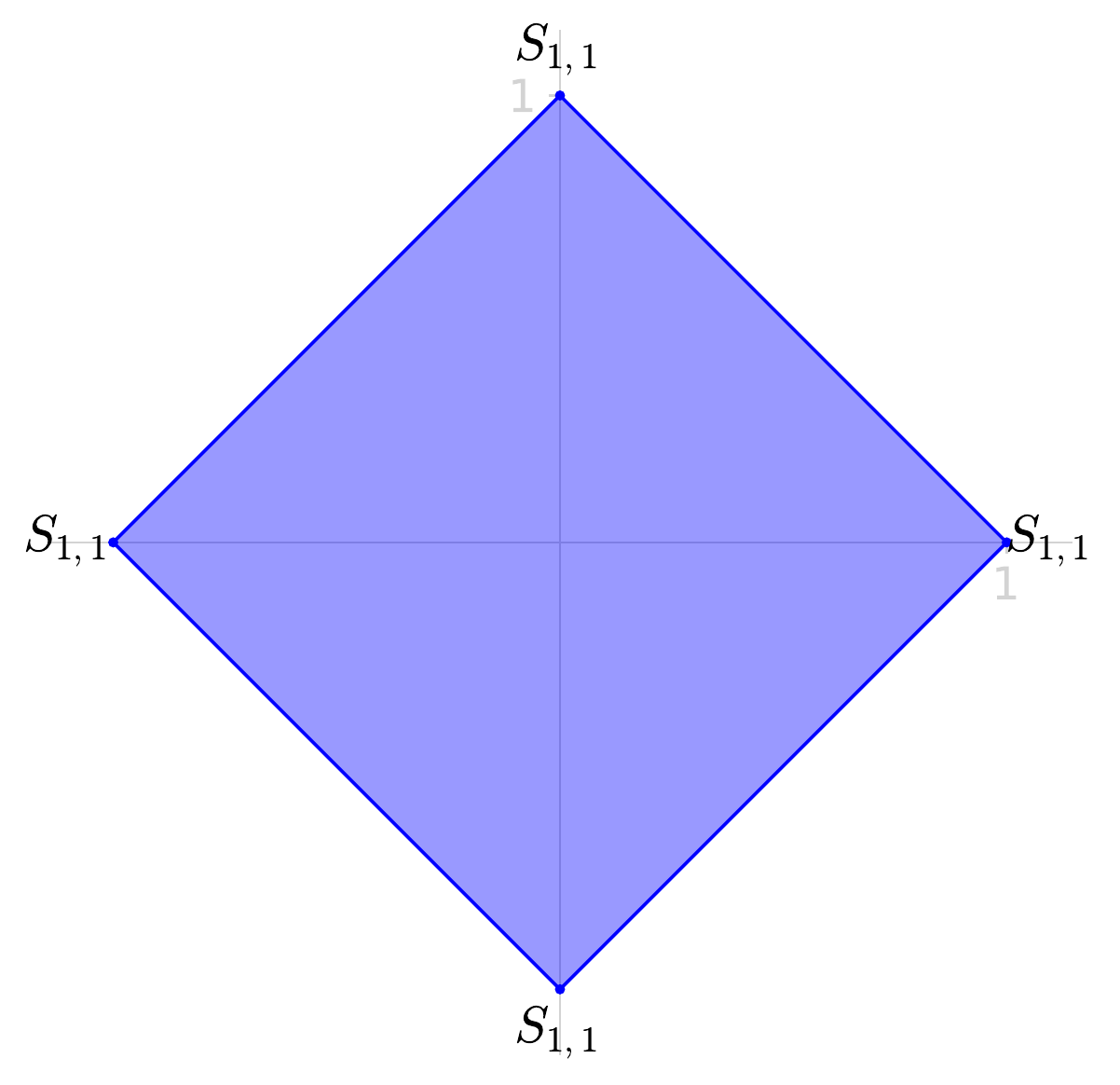}} \\ 
 $L=9^{{2}}_{{45}}$ & & $L=9^{{2}}_{{46}}$ & \\ 
 \quad & & \quad & \\ $\mathrm{Isom}(\mathbb{S}^3\setminus L) = \displaystyle\bigoplus_{i=1}^3 \mathbb{Z}$ & & $\mathrm{Isom}(\mathbb{S}^3\setminus L) = D_4$ & \\ 
 \quad & & \quad & \\ 
 \includegraphics[width=1in]{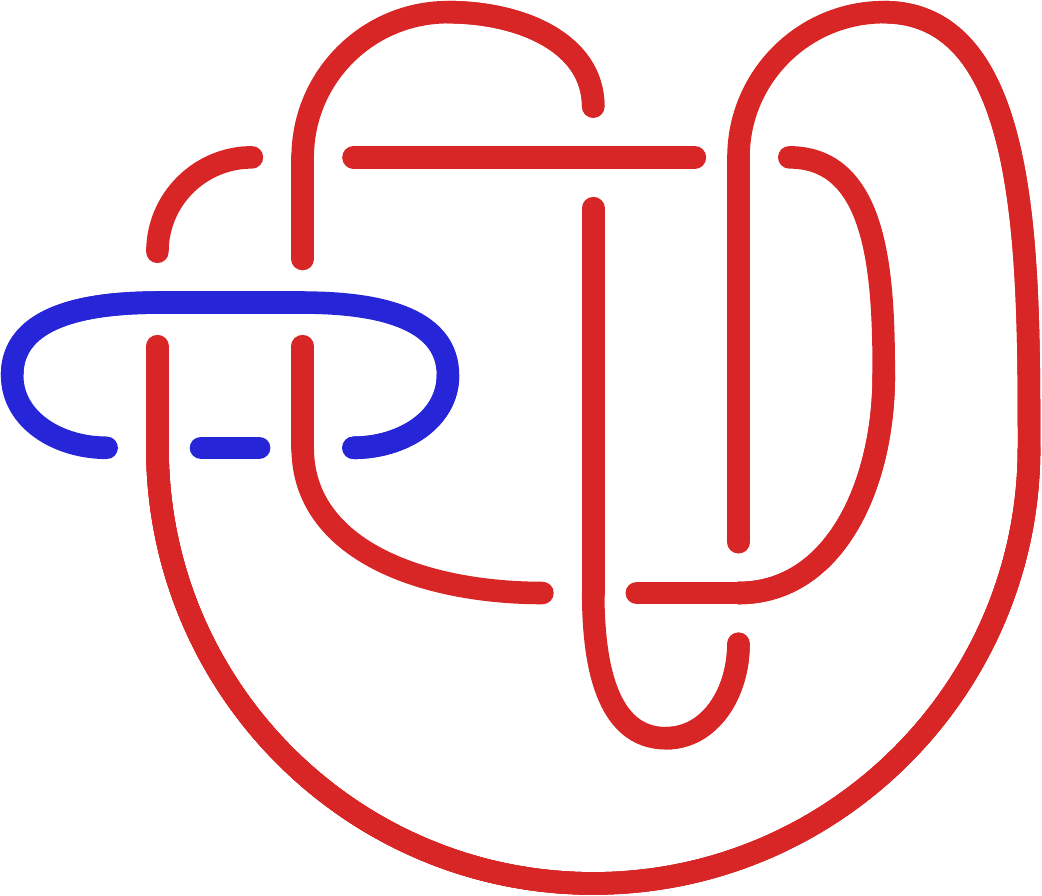}  & & \includegraphics[width=1in]{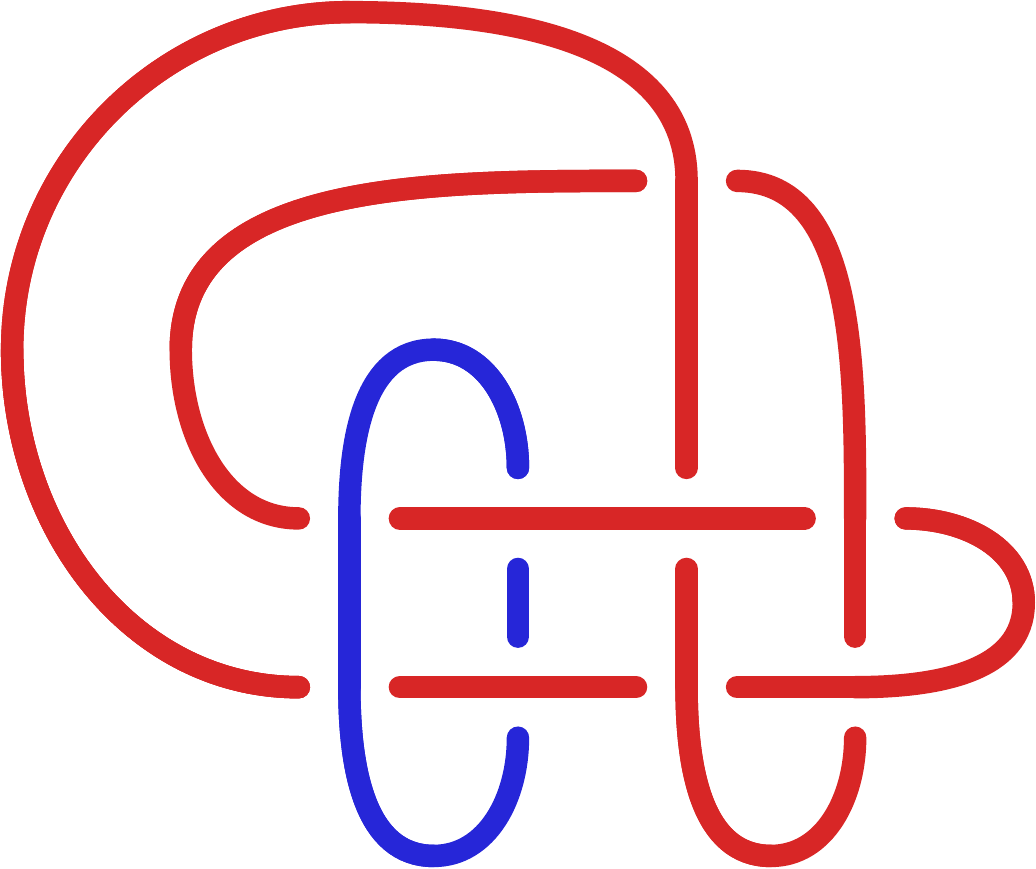} & \\ 
 \quad & & \quad & \\ 
 \hline  
\quad & \multirow{6}{*}{\Includegraphics[width=1.8in]{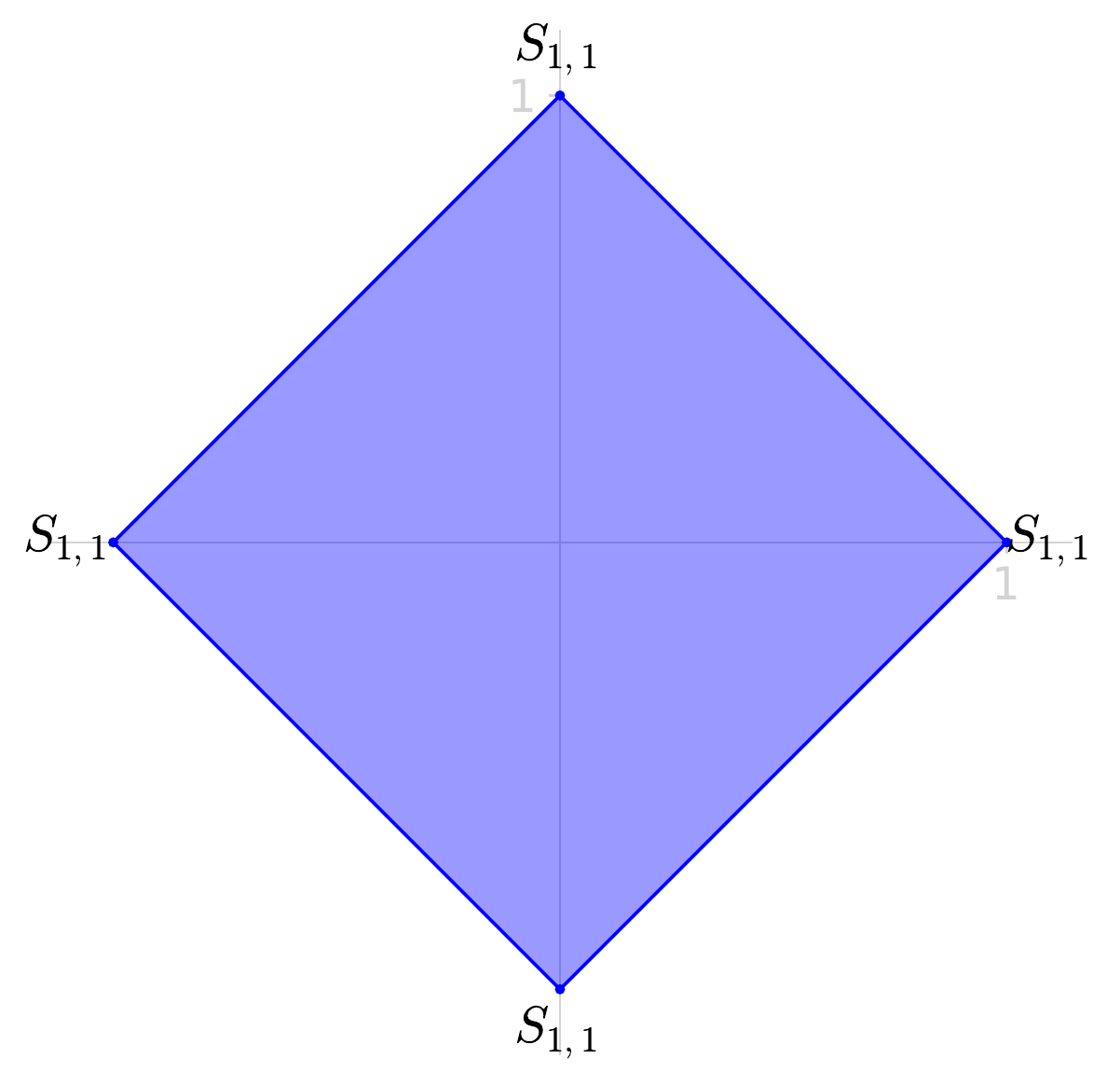}} & \quad & \multirow{6}{*}{\Includegraphics[width=1.8in]{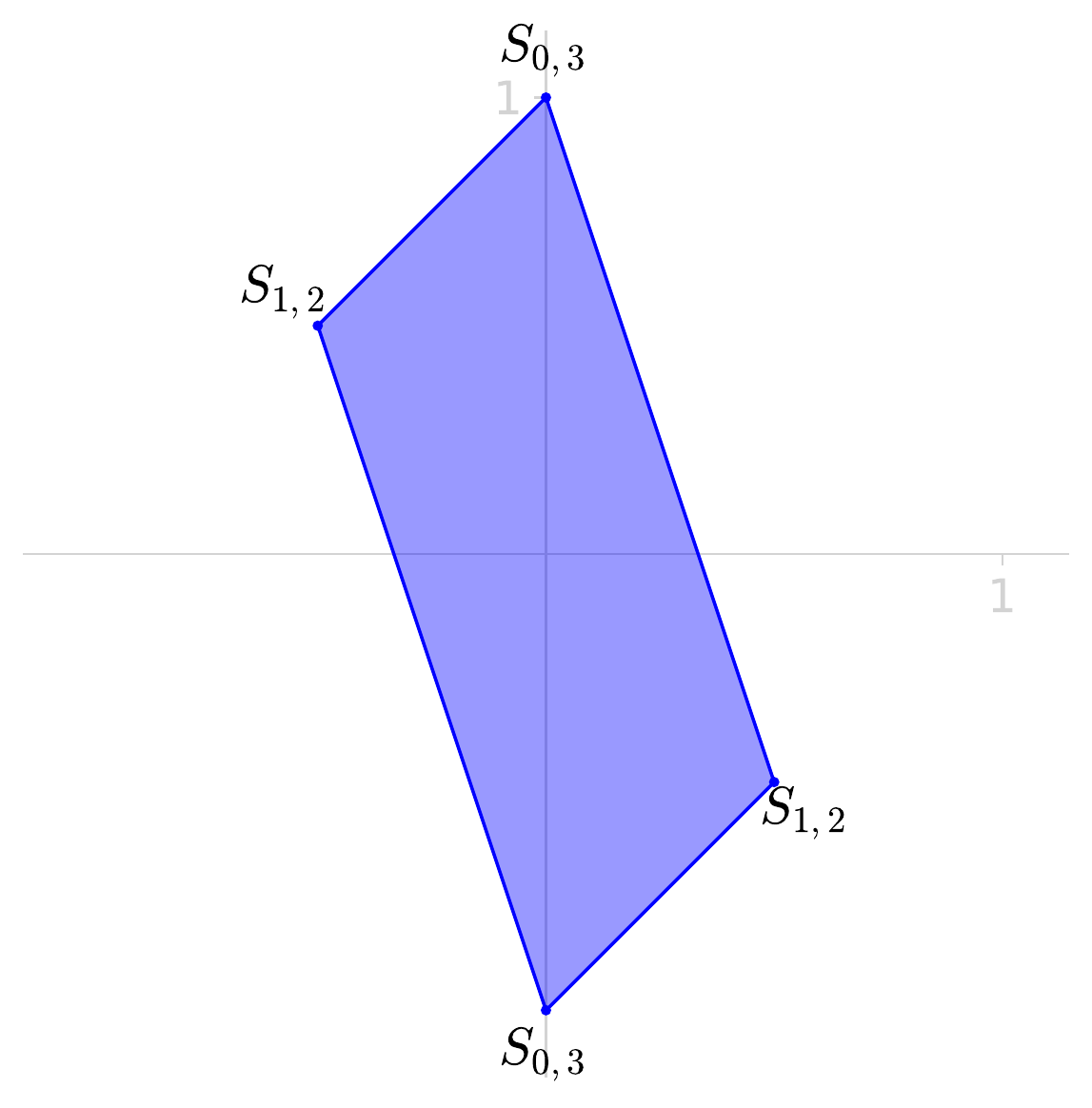}} \\ 
 $L=9^{{2}}_{{47}}$ & & $L=9^{{2}}_{{48}}$ & \\ 
 \quad & & \quad & \\ $\mathrm{Isom}(\mathbb{S}^3\setminus L) = D_4$ & & $\mathrm{Isom}(\mathbb{S}^3\setminus L) = \mathbb{{Z}}_2\oplus\mathbb{{Z}}_2$ & \\ 
 \quad & & \quad & \\ 
 \includegraphics[width=1in]{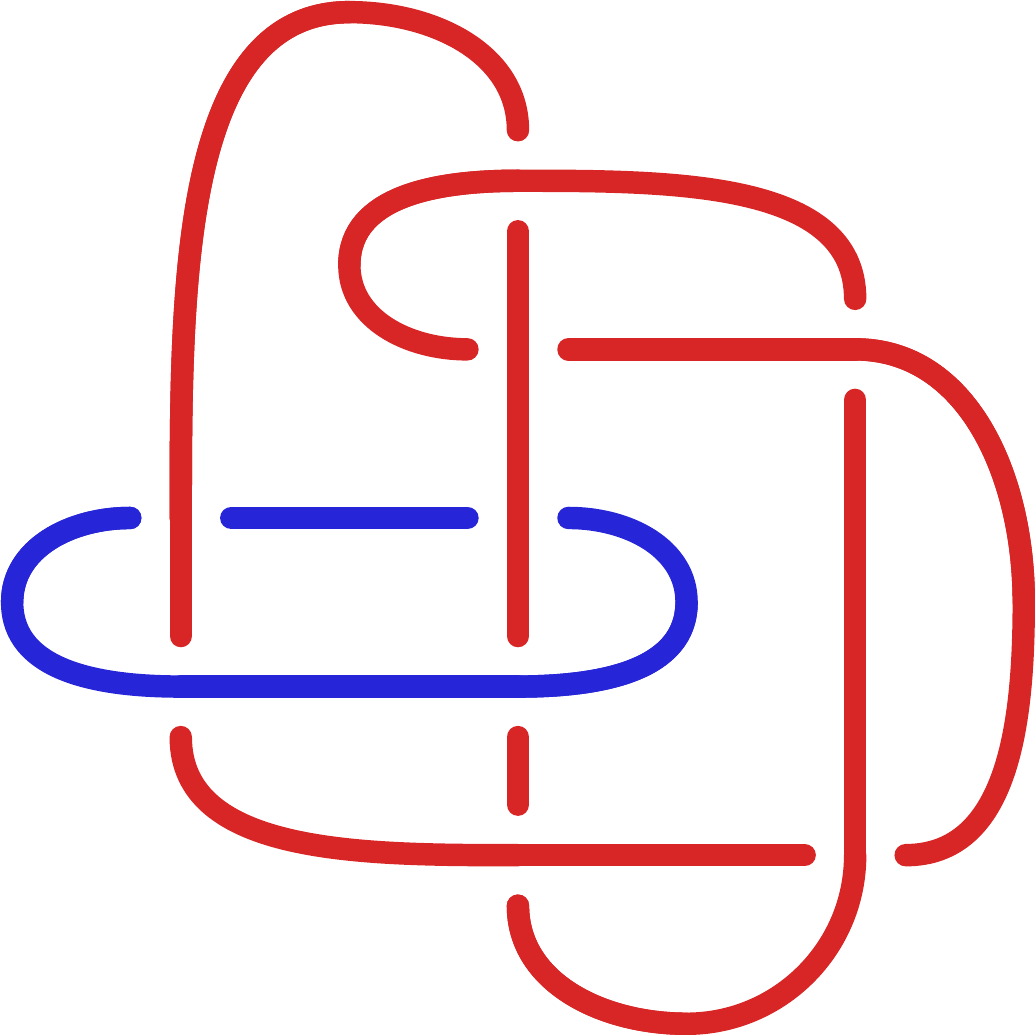}  & & \includegraphics[width=1in]{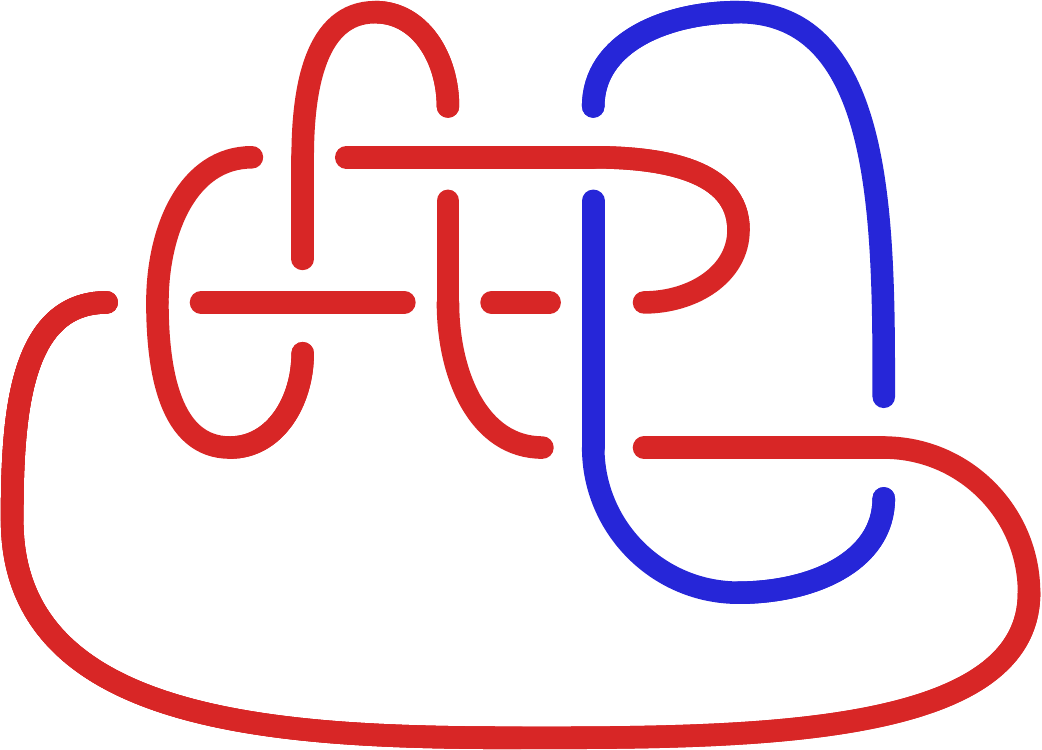} & \\ 
 \quad & & \quad & \\ 
 \hline  
\quad & \multirow{6}{*}{\Includegraphics[width=1.8in]{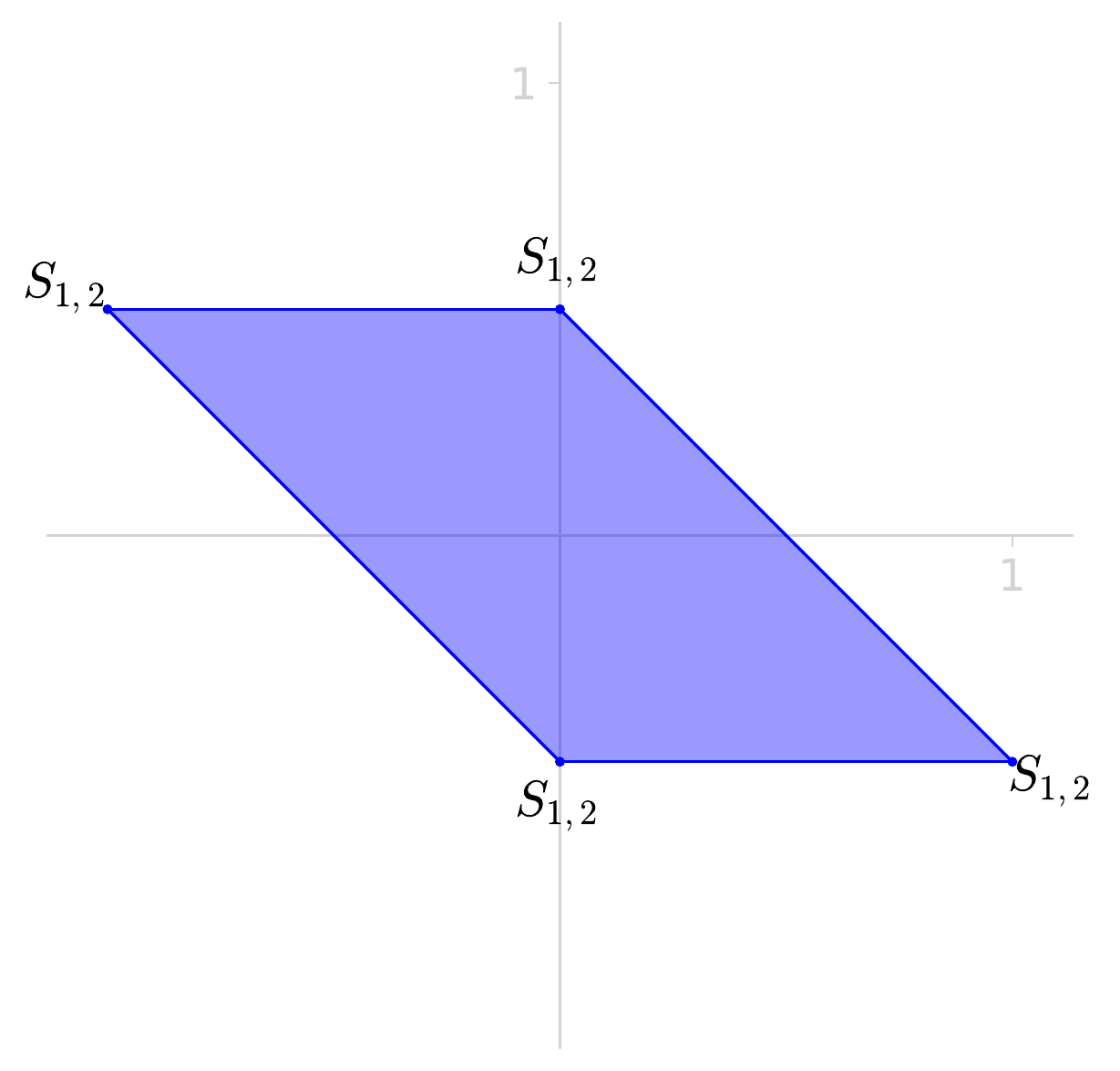}} & \quad & \multirow{6}{*}{\Includegraphics[width=1.8in]{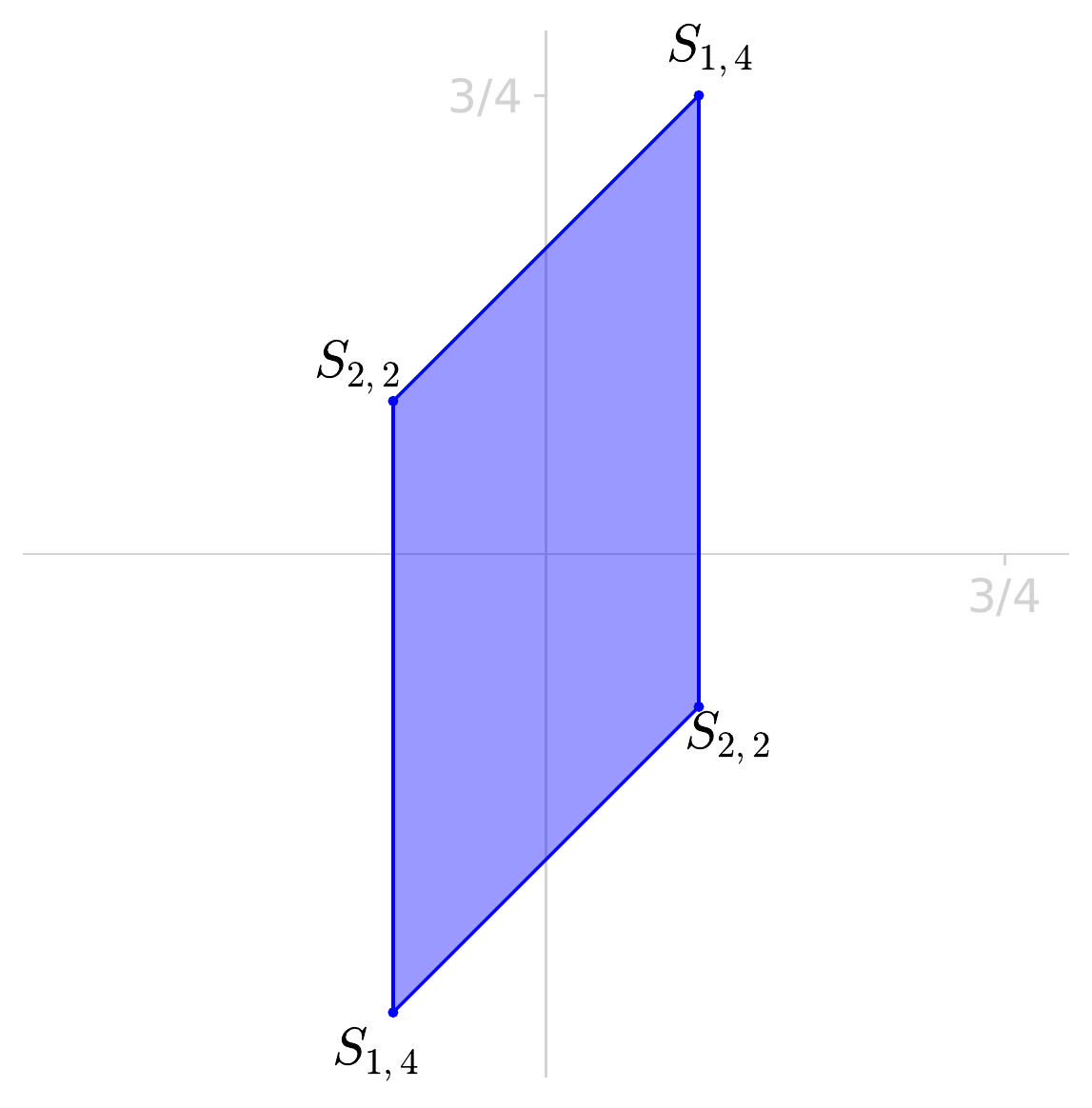}} \\ 
 $L=9^{{2}}_{{50}}$ & & $L=9^{{2}}_{{51}}$ & \\ 
 \quad & & \quad & \\ $\mathrm{Isom}(\mathbb{S}^3\setminus L) = \mathbb{{Z}}_2\oplus\mathbb{{Z}}_2$ & & $\mathrm{Isom}(\mathbb{S}^3\setminus L) = \mathbb{{Z}}_2$ & \\ 
 \quad & & \quad & \\ 
 \includegraphics[width=1in]{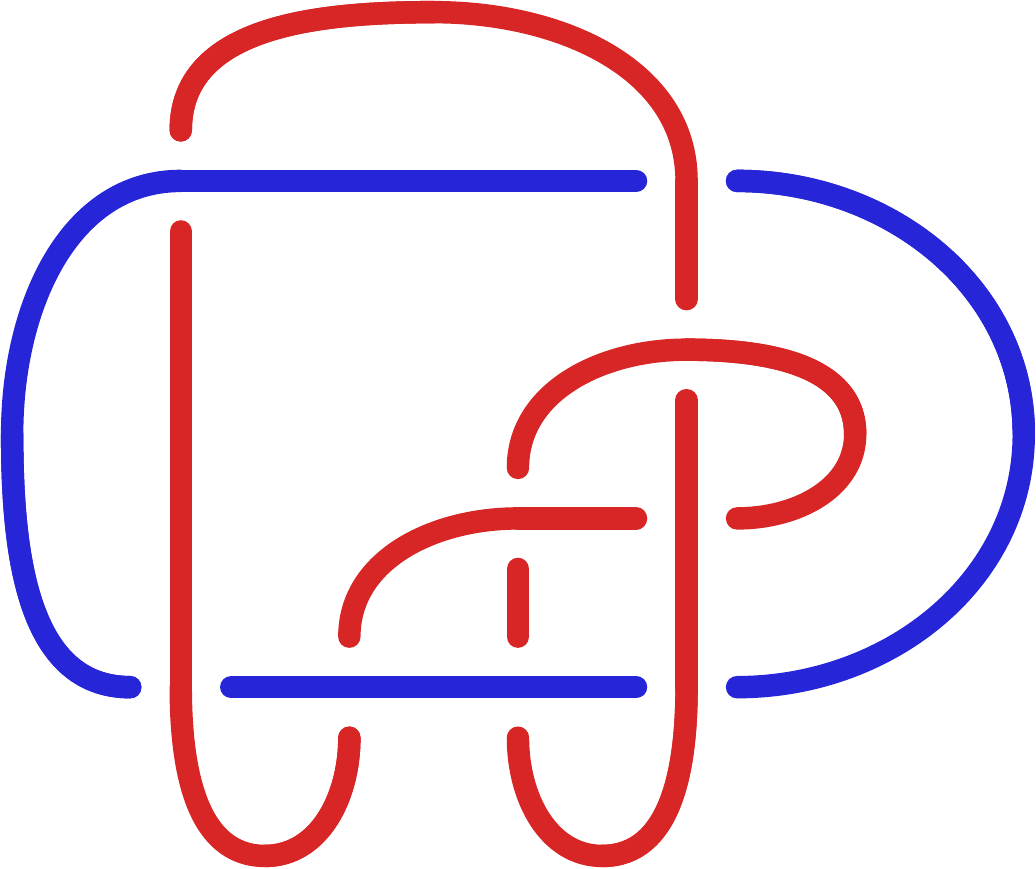}  & & \includegraphics[width=1in]{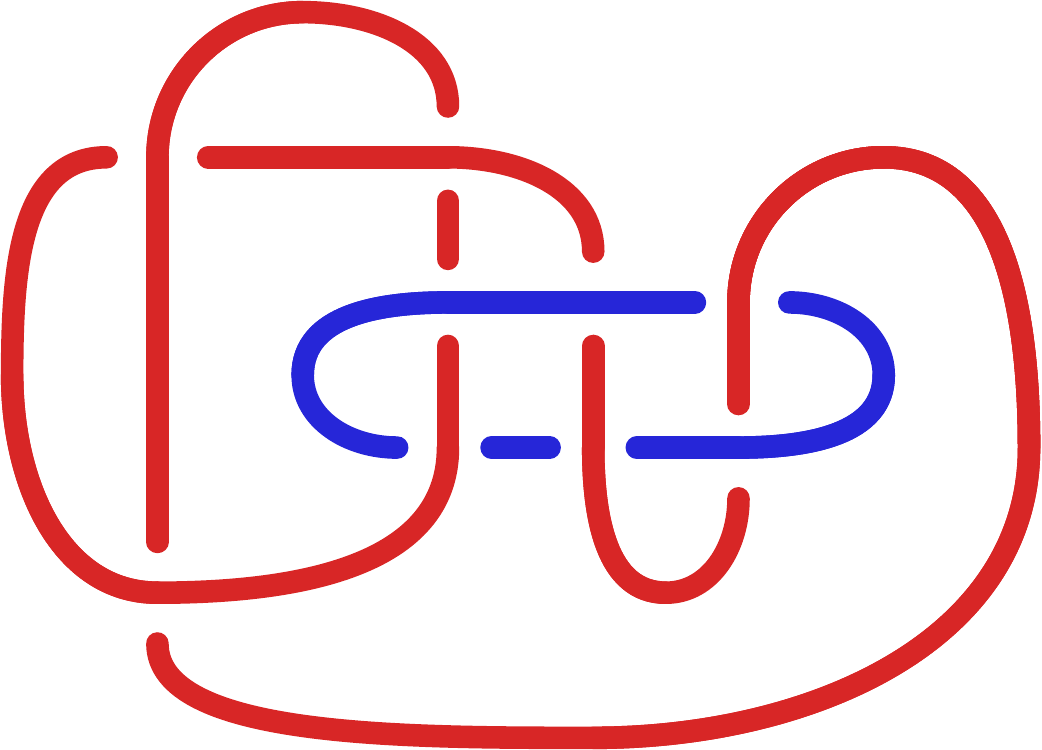} & \\ 
 \quad & & \quad & \\ 
 \hline  
\quad & \multirow{6}{*}{\Includegraphics[width=1.8in]{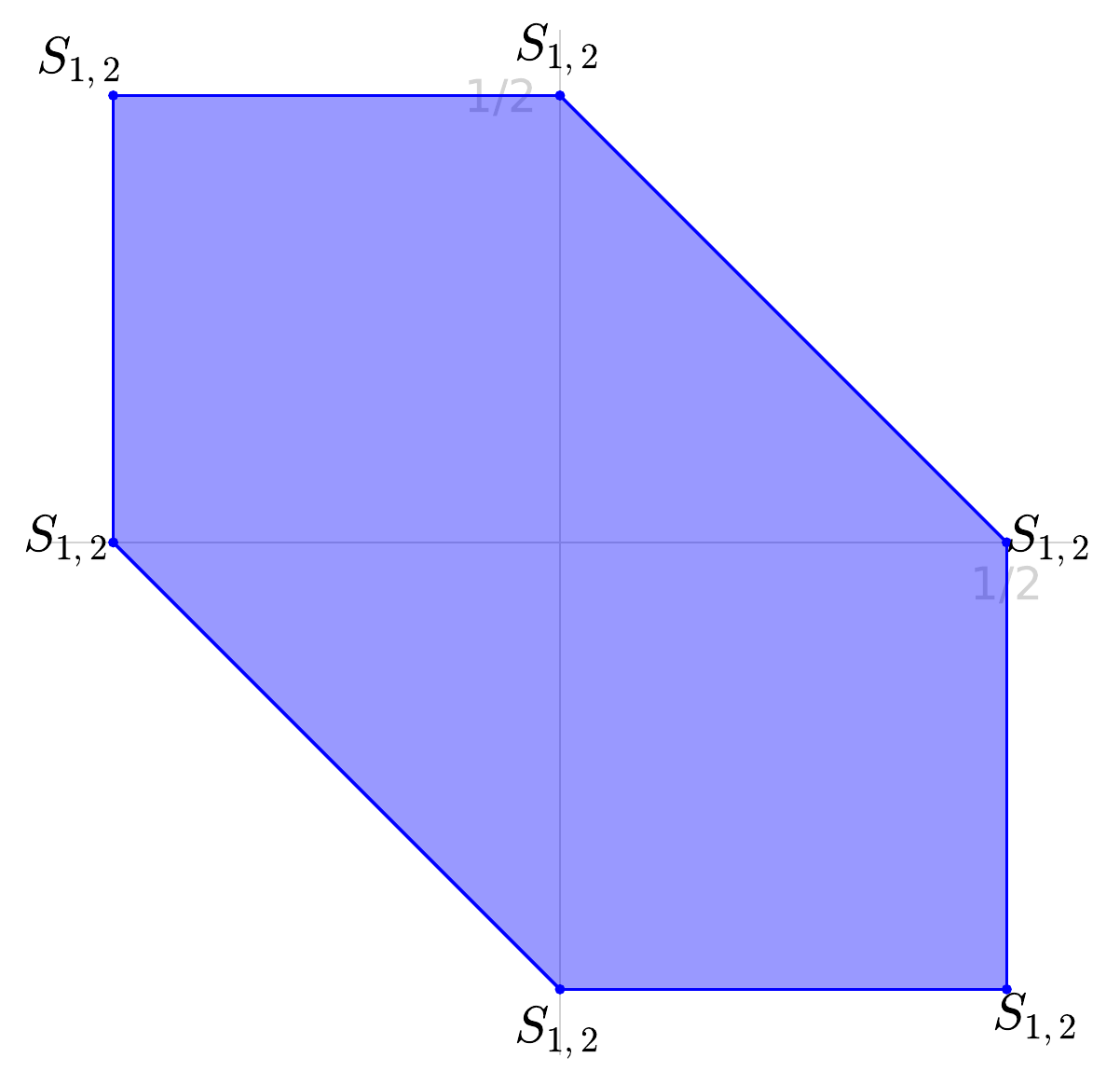}} & \quad & \multirow{6}{*}{\Includegraphics[width=1.8in]{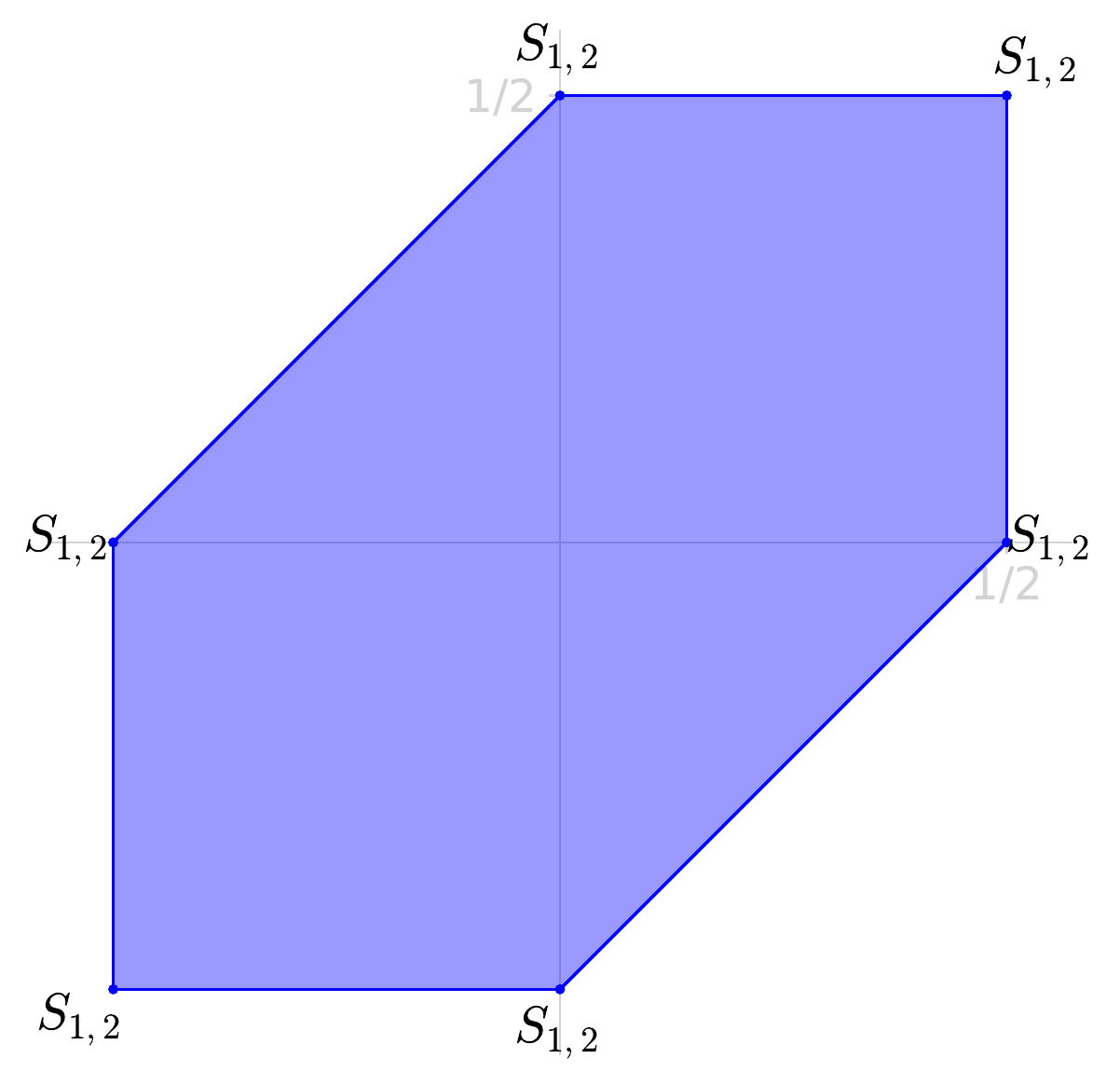}} \\ 
 $L=9^{{2}}_{{52}}$ & & $L=9^{{2}}_{{54}}$ & \\ 
 \quad & & \quad & \\ $\mathrm{Isom}(\mathbb{S}^3\setminus L) = \mathbb{{Z}}_2$ & & $\mathrm{Isom}(\mathbb{S}^3\setminus L) = \mathbb{{Z}}_2\oplus\mathbb{{Z}}_2$ & \\ 
 \quad & & \quad & \\ 
 \includegraphics[width=1in]{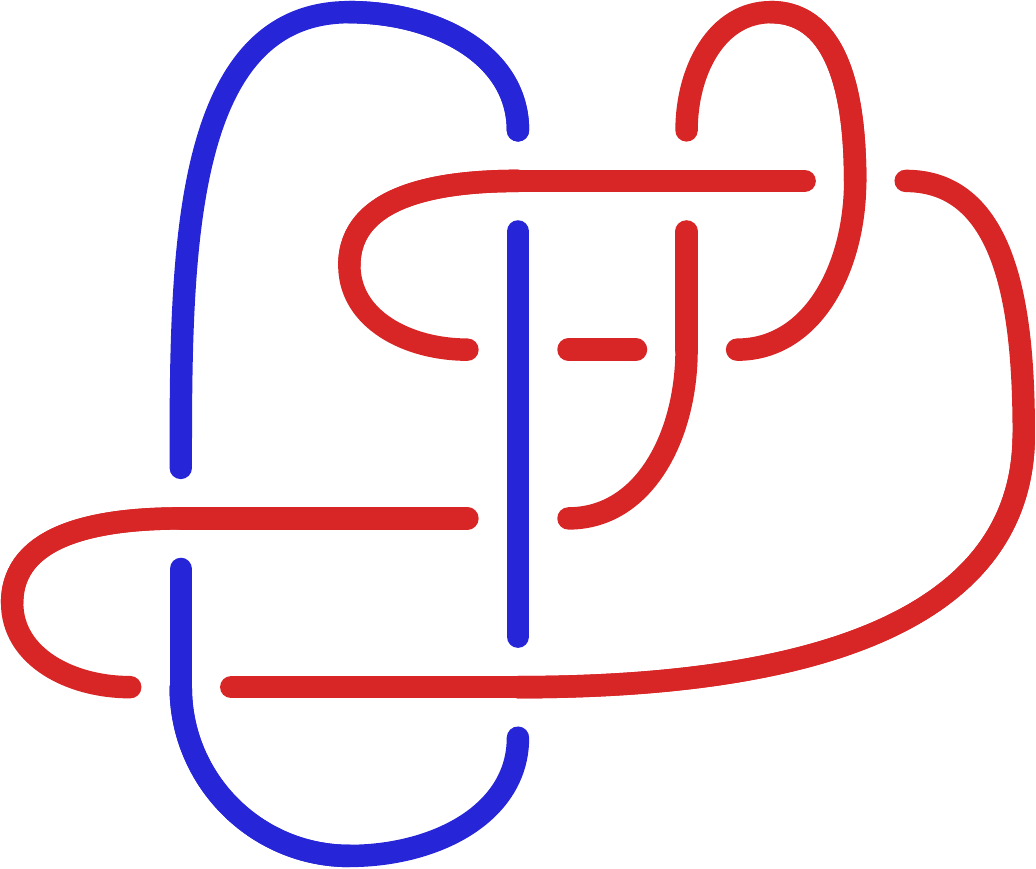}  & & \includegraphics[width=1in]{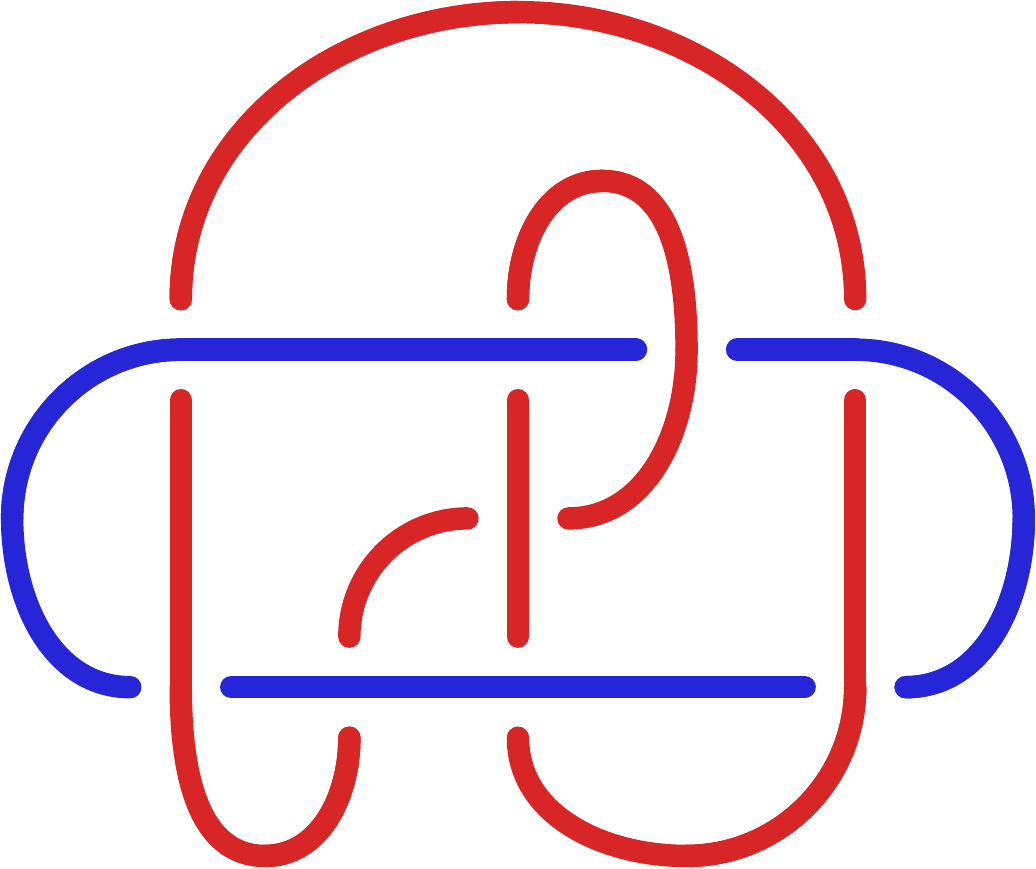} & \\ 
 \quad & & \quad & \\ 
 \hline  
\end{tabular} 
 \newpage \begin{tabular}{|c|c|c|c|} 
 \hline 
 Link & Norm Ball & Link & Norm Ball \\ 
 \hline 
\quad & \multirow{6}{*}{\Includegraphics[width=1.8in]{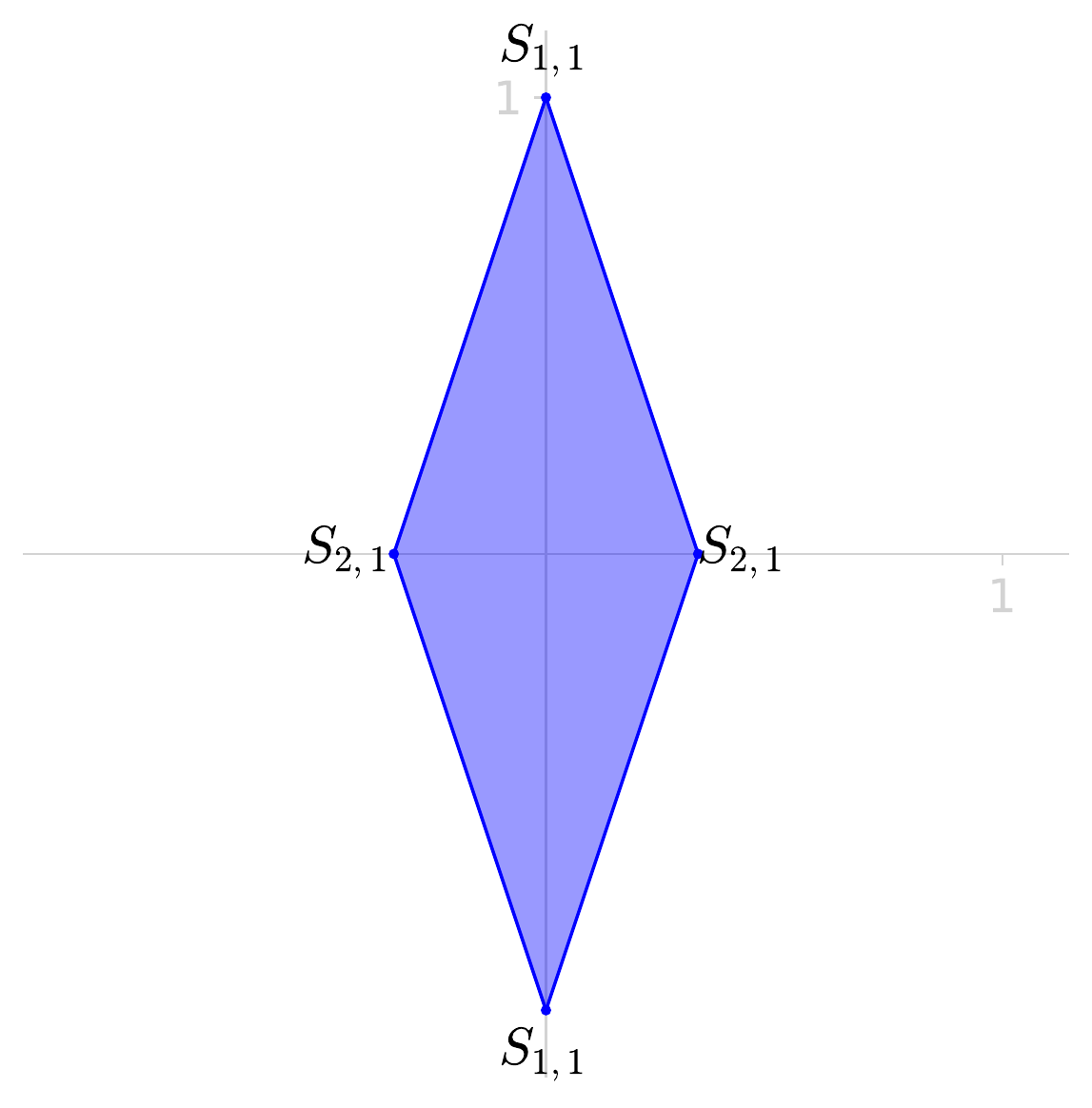}} & \quad & \multirow{6}{*}{\Includegraphics[width=1.8in]{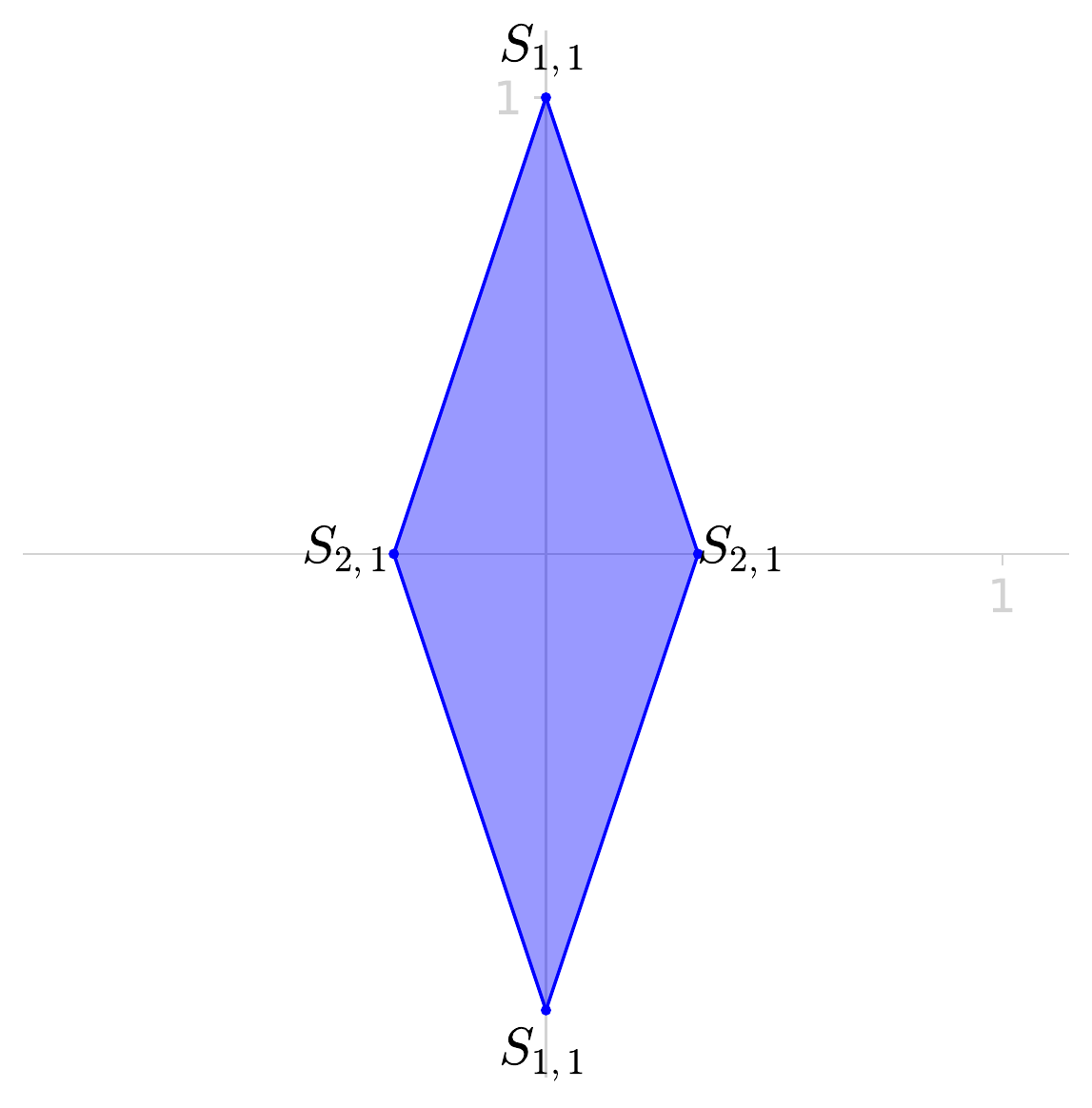}} \\ 
 $L=9^{{2}}_{{55}}$ & & $L=9^{{2}}_{{56}}$ & \\ 
 \quad & & \quad & \\ $\mathrm{Isom}(\mathbb{S}^3\setminus L) = \mathbb{{Z}}_2\oplus\mathbb{{Z}}_2$ & & $\mathrm{Isom}(\mathbb{S}^3\setminus L) = \mathbb{{Z}}_2\oplus\mathbb{{Z}}_2$ & \\ 
 \quad & & \quad & \\ 
 \includegraphics[width=1in]{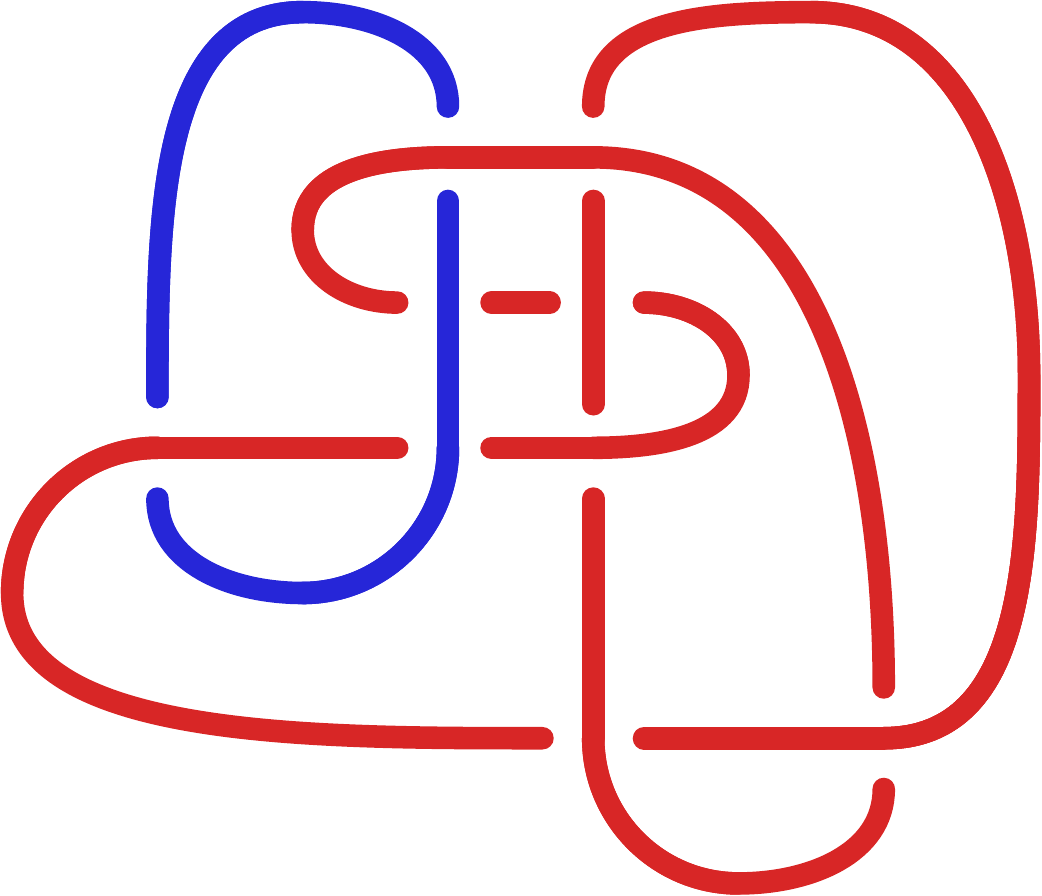}  & & \includegraphics[width=1in]{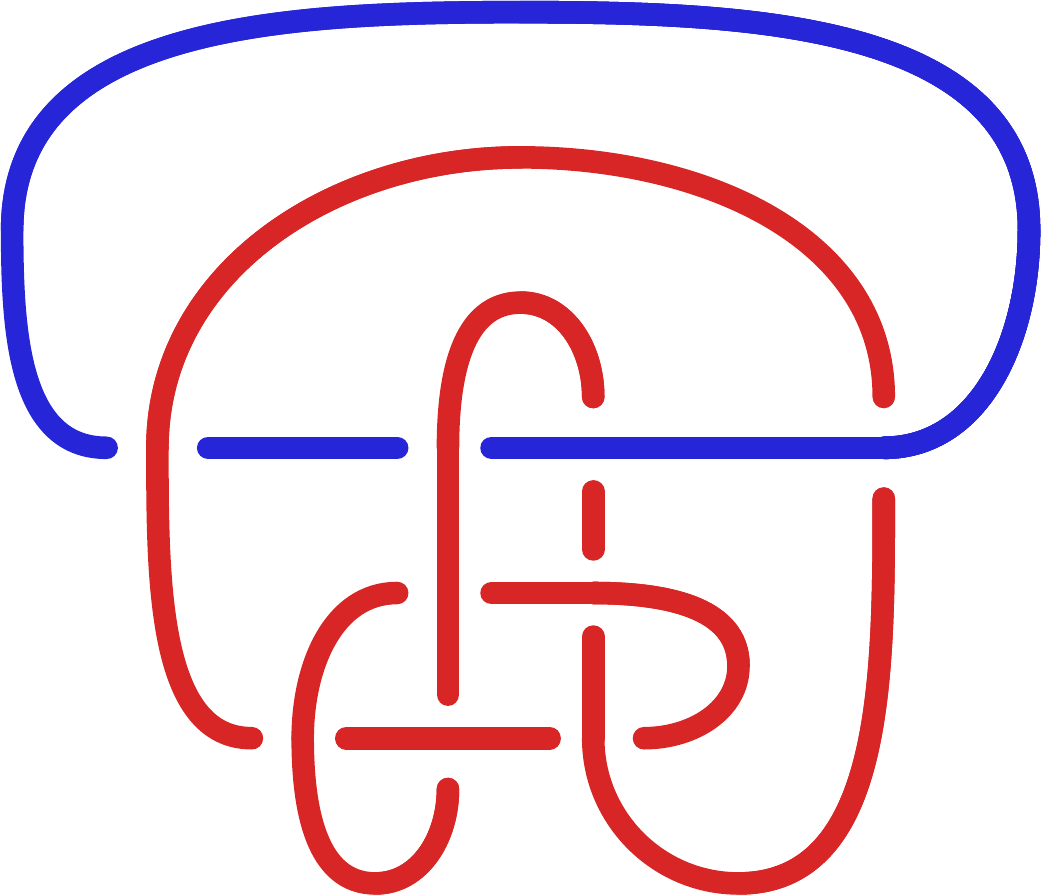} & \\ 
 \quad & & \quad & \\ 
 \hline  
\quad & \multirow{6}{*}{\Includegraphics[width=1.8in]{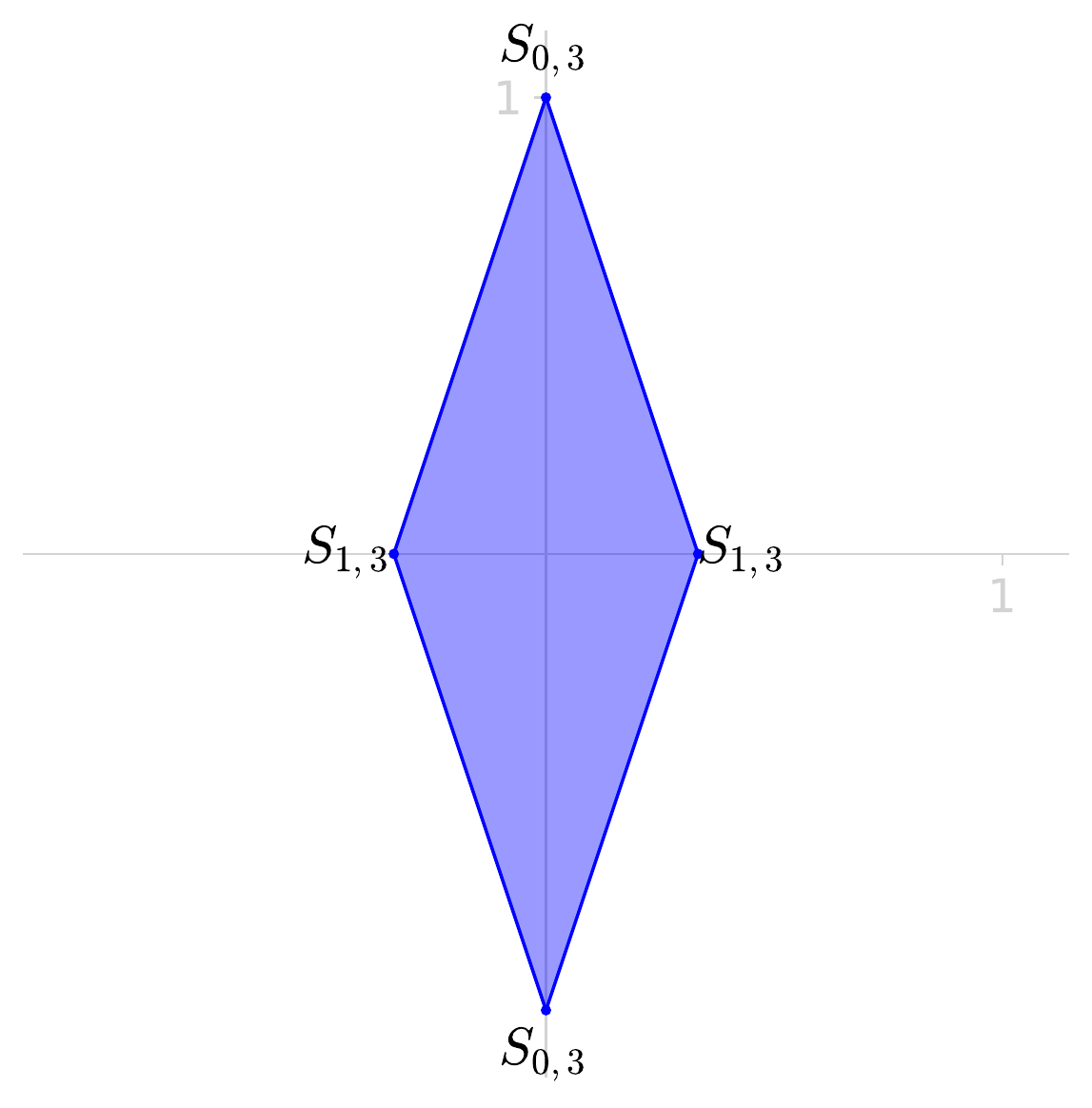}} & \quad & \multirow{6}{*}{\Includegraphics[width=1.8in]{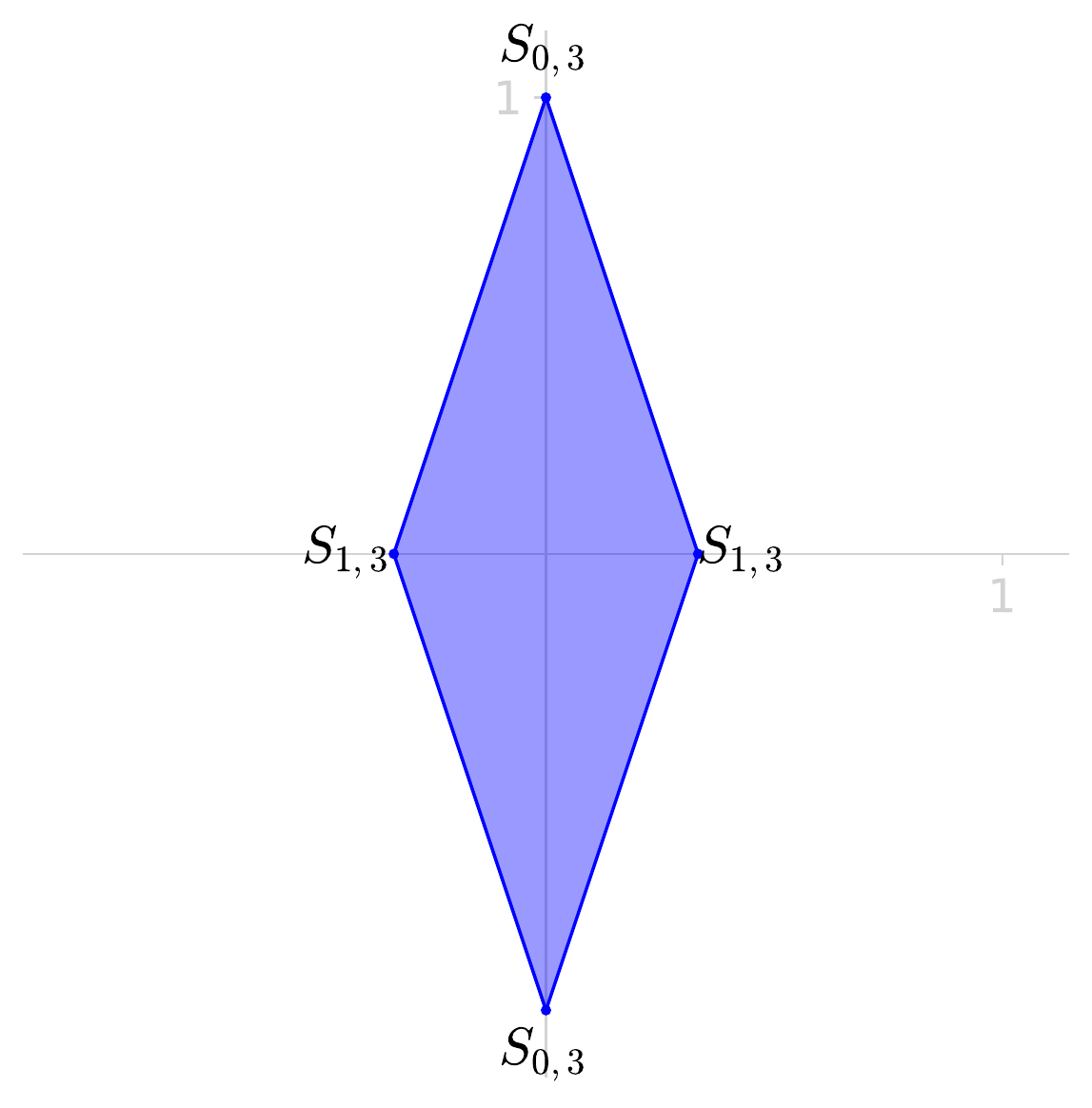}} \\ 
 $L=9^{{2}}_{{57}}$ & & $L=9^{{2}}_{{58}}$ & \\ 
 \quad & & \quad & \\ $\mathrm{Isom}(\mathbb{S}^3\setminus L) = \mathbb{{Z}}_2$ & & $\mathrm{Isom}(\mathbb{S}^3\setminus L) = \mathbb{{Z}}_2$ & \\ 
 \quad & & \quad & \\ 
 \includegraphics[width=1in]{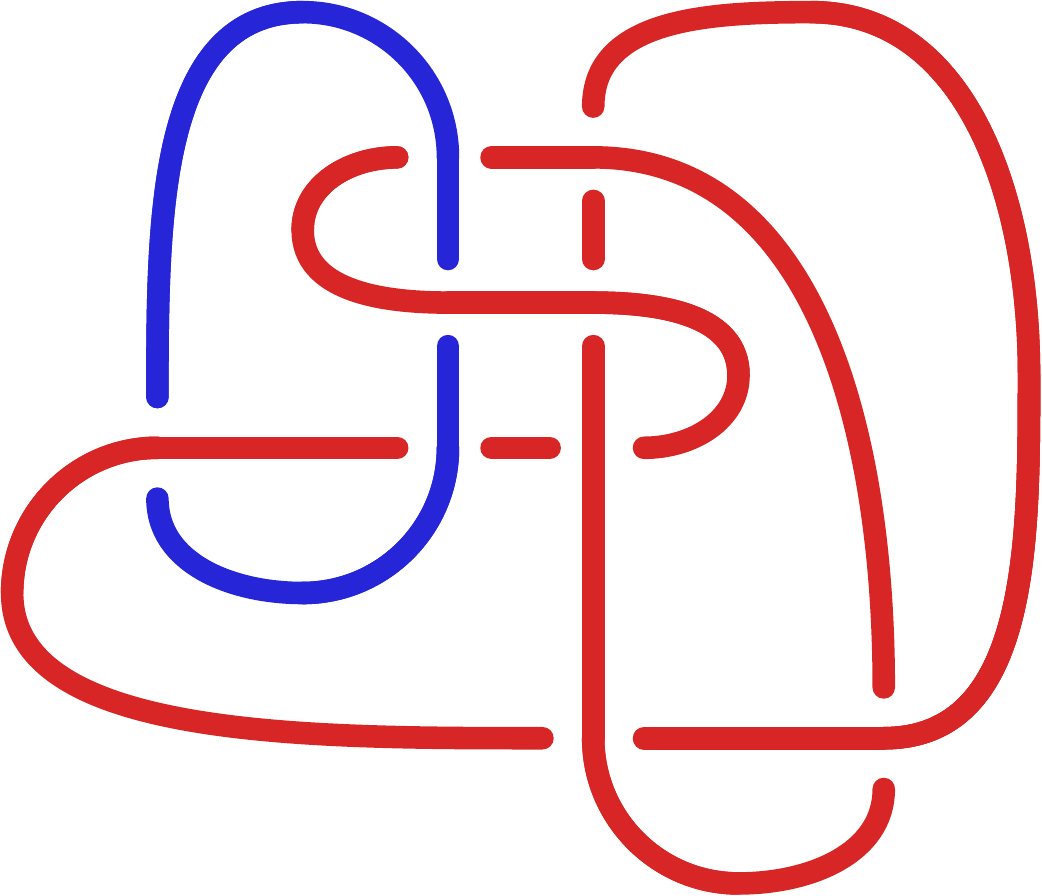}  & & \includegraphics[width=1in]{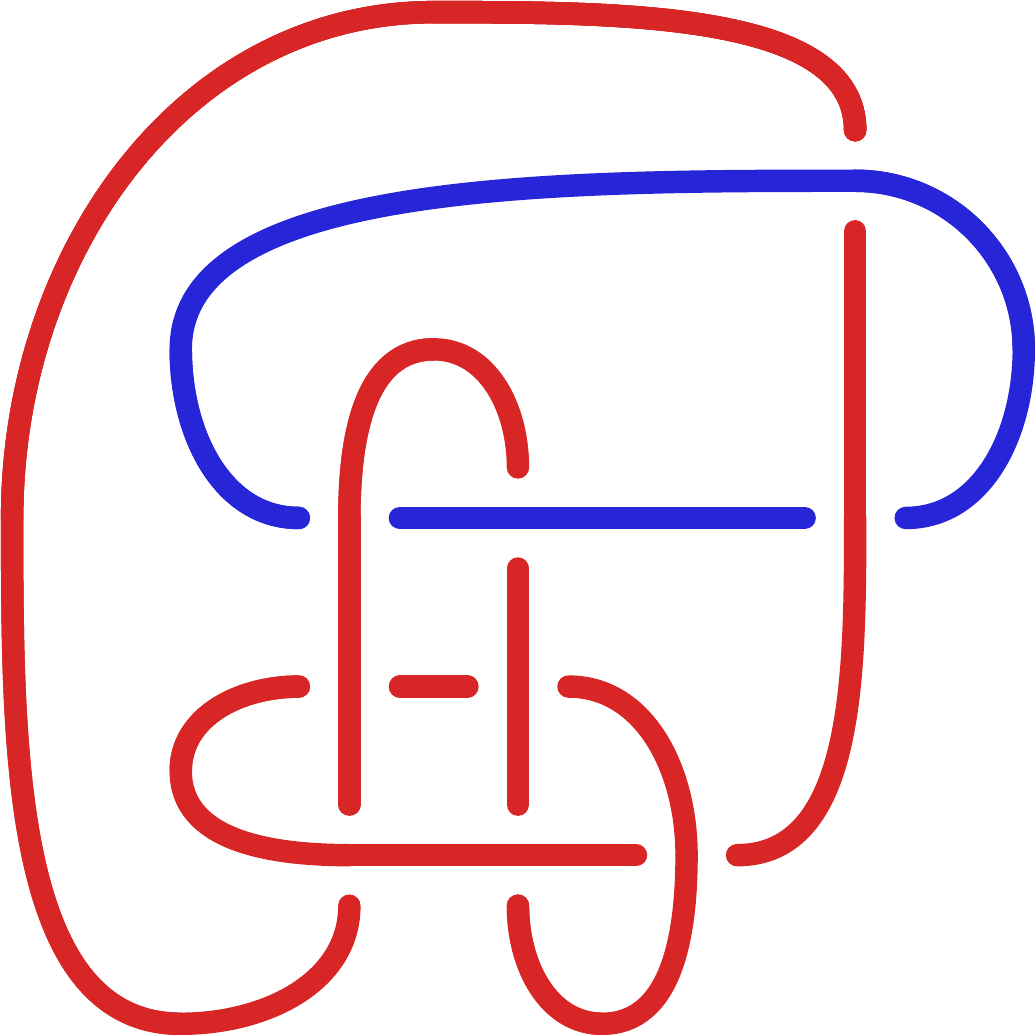} & \\ 
 \quad & & \quad & \\ 
 \hline  
\quad & \multirow{6}{*}{\Includegraphics[width=1.8in]{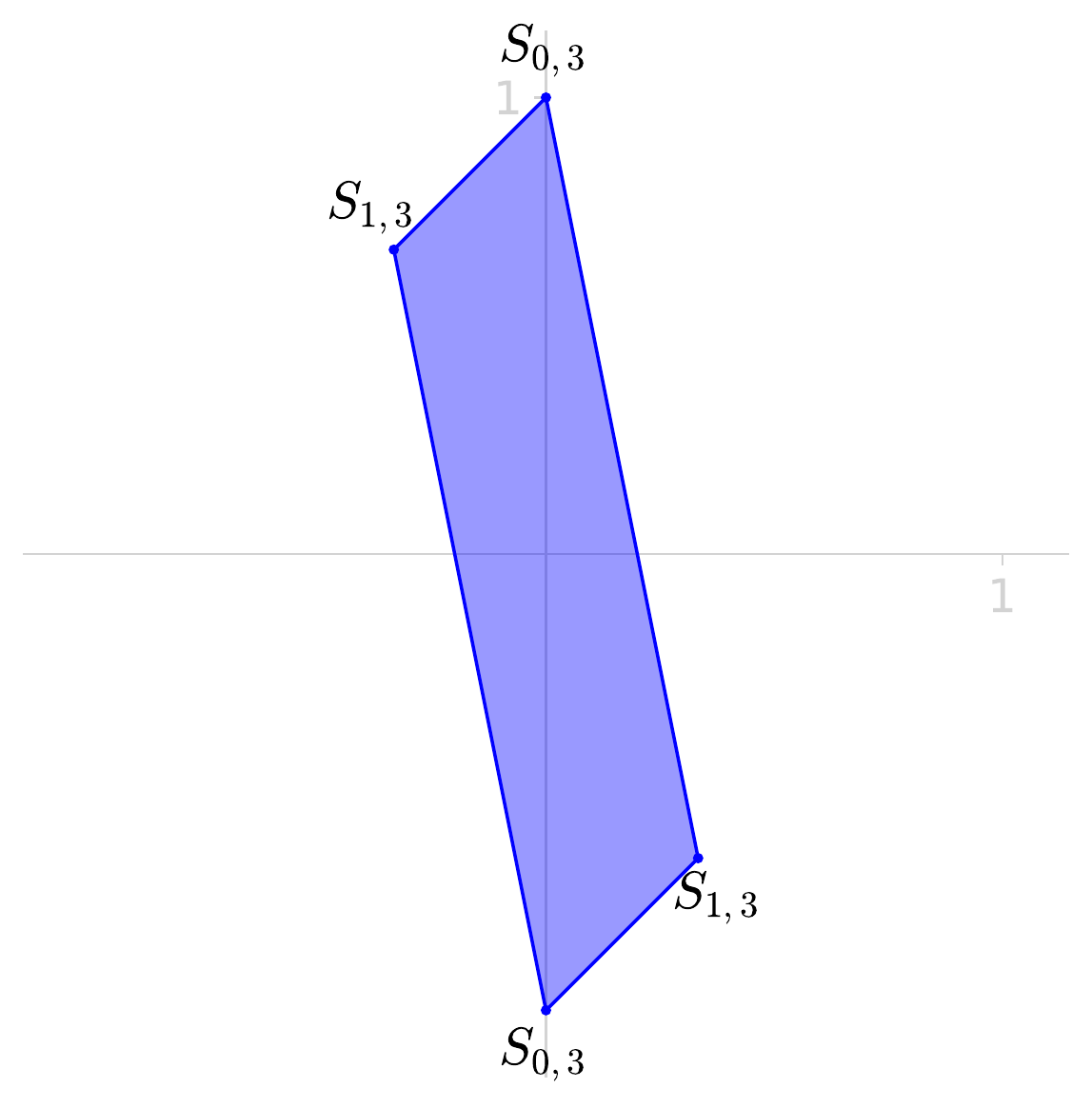}} & \quad & \multirow{6}{*}{\Includegraphics[width=1.8in]{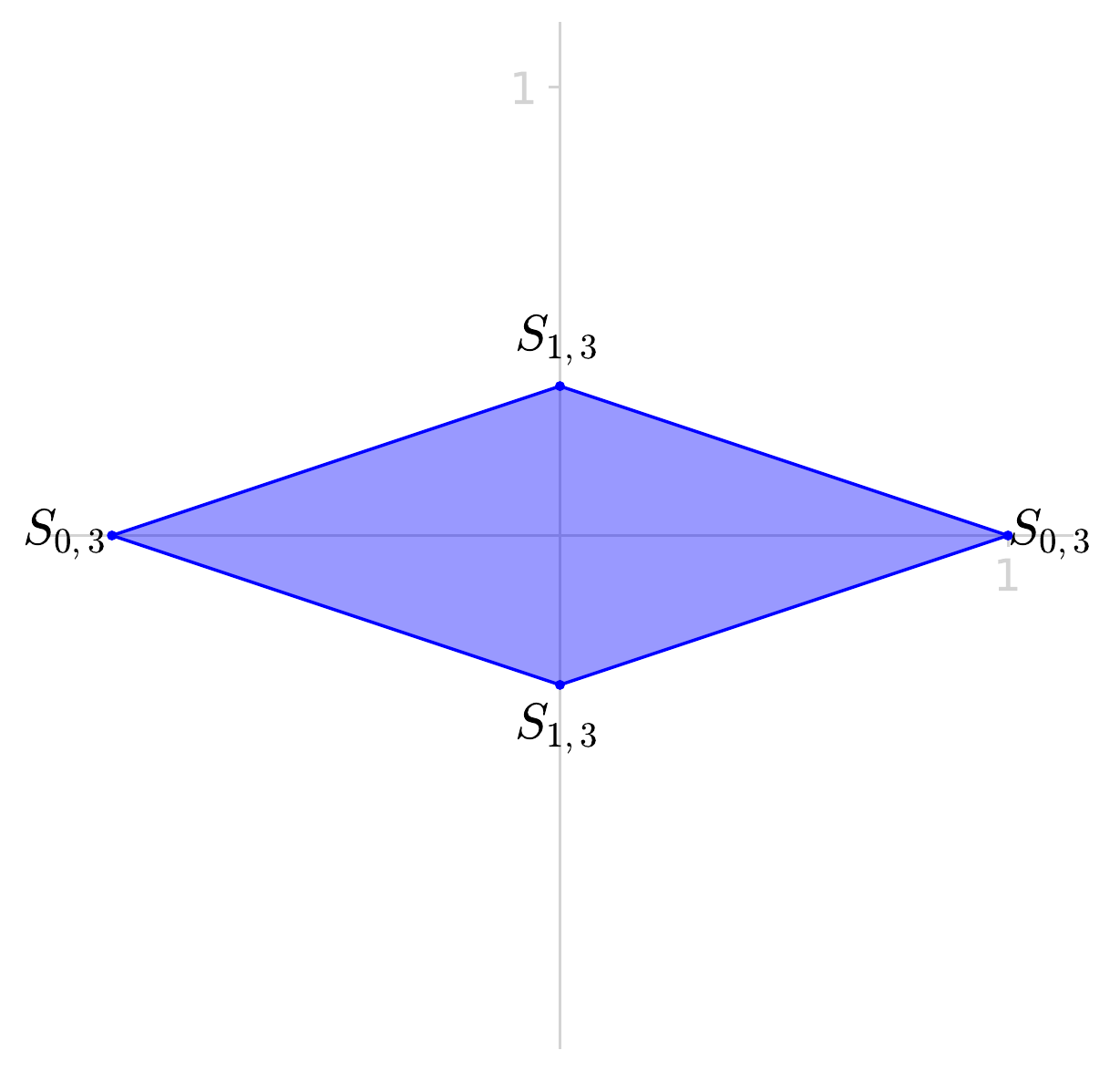}} \\ 
 $L=9^{{2}}_{{59}}$ & & $L=9^{{2}}_{{60}}$ & \\ 
 \quad & & \quad & \\ $\mathrm{Isom}(\mathbb{S}^3\setminus L) = \mathbb{{Z}}_2$ & & $\mathrm{Isom}(\mathbb{S}^3\setminus L) = \mathbb{{Z}}_2$ & \\ 
 \quad & & \quad & \\ 
 \includegraphics[width=1in]{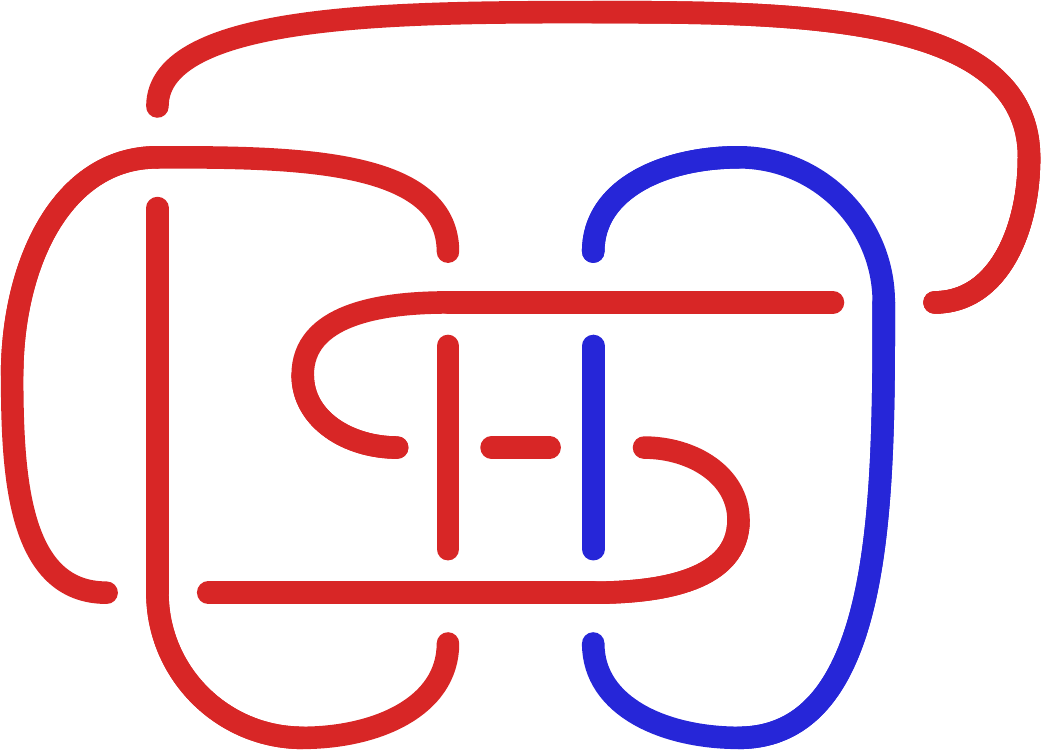}  & & \includegraphics[width=1in]{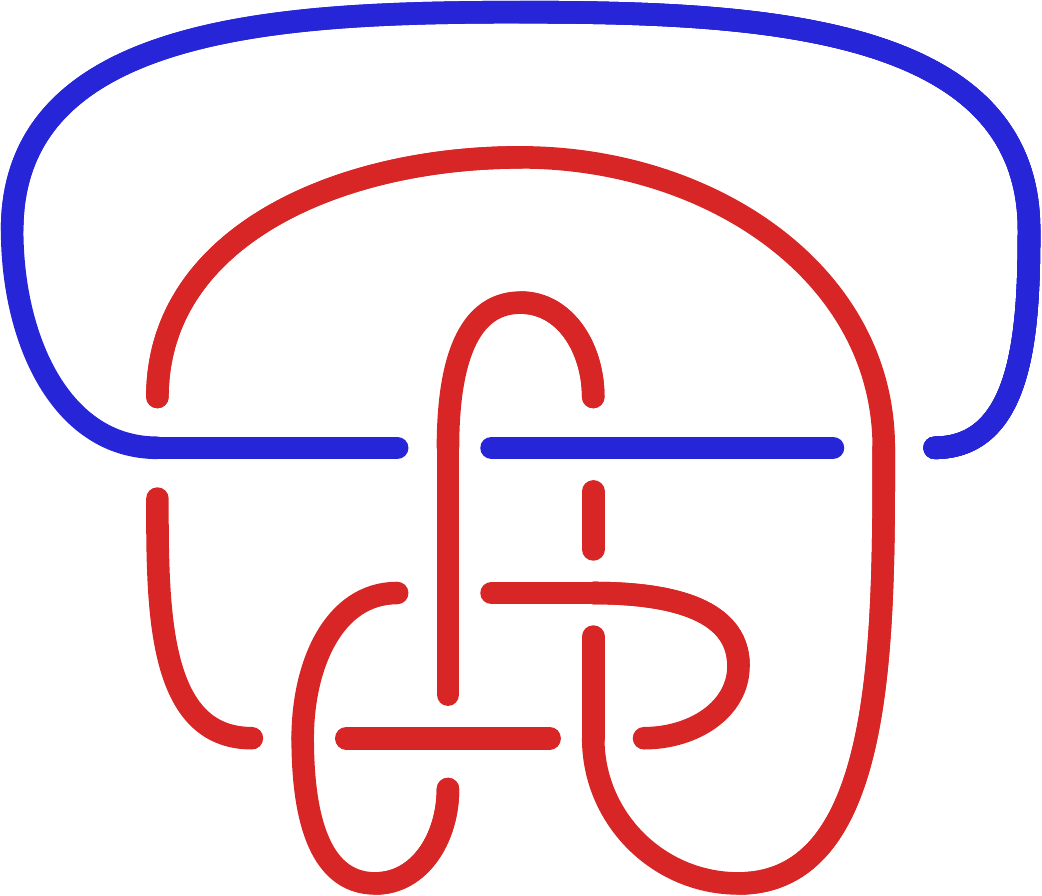} & \\ 
 \quad & & \quad & \\ 
 \hline  
\quad & \multirow{6}{*}{\Includegraphics[width=1.8in]{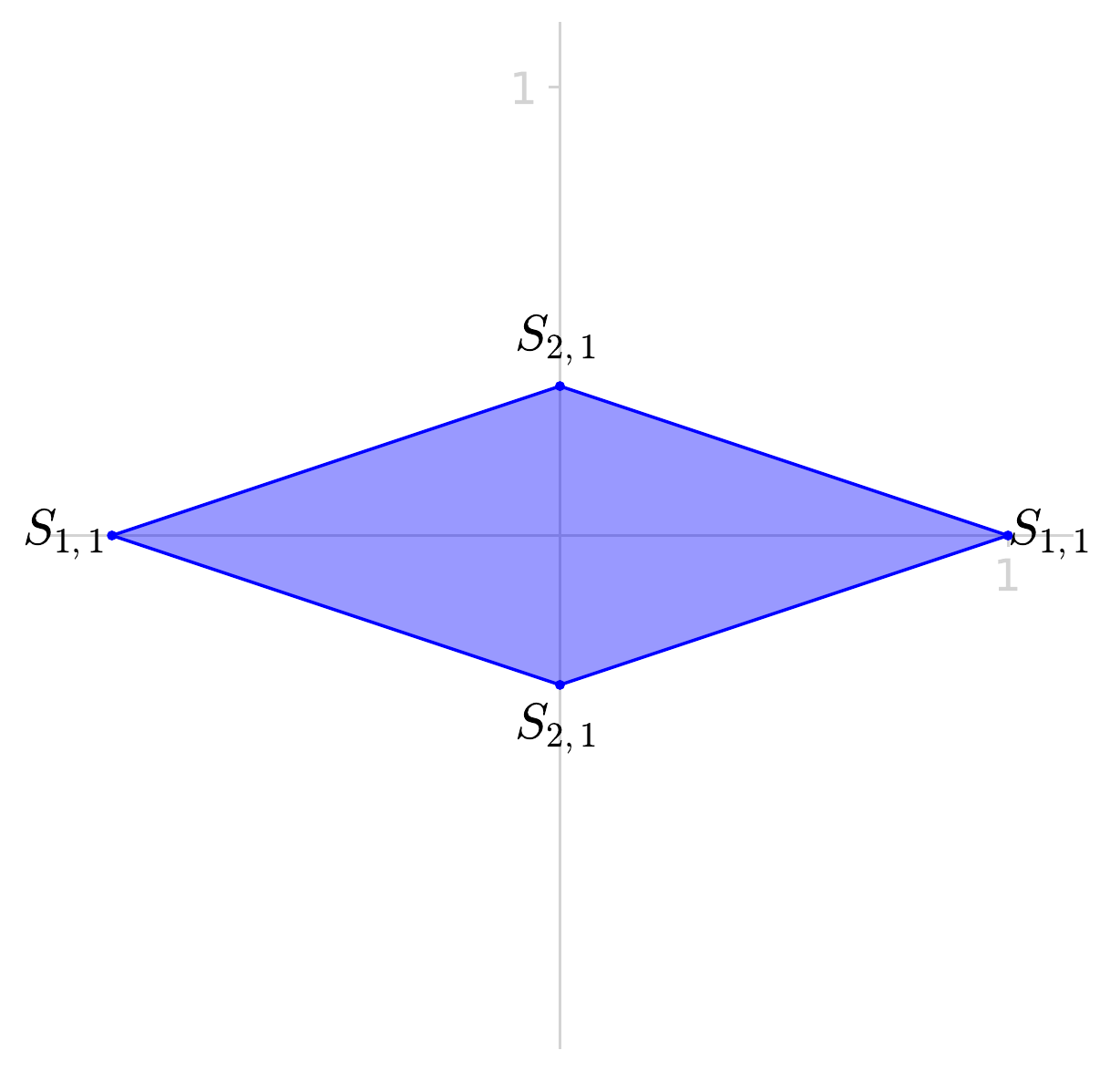}} & \quad & \multirow{6}{*}{\Includegraphics[width=1.8in]{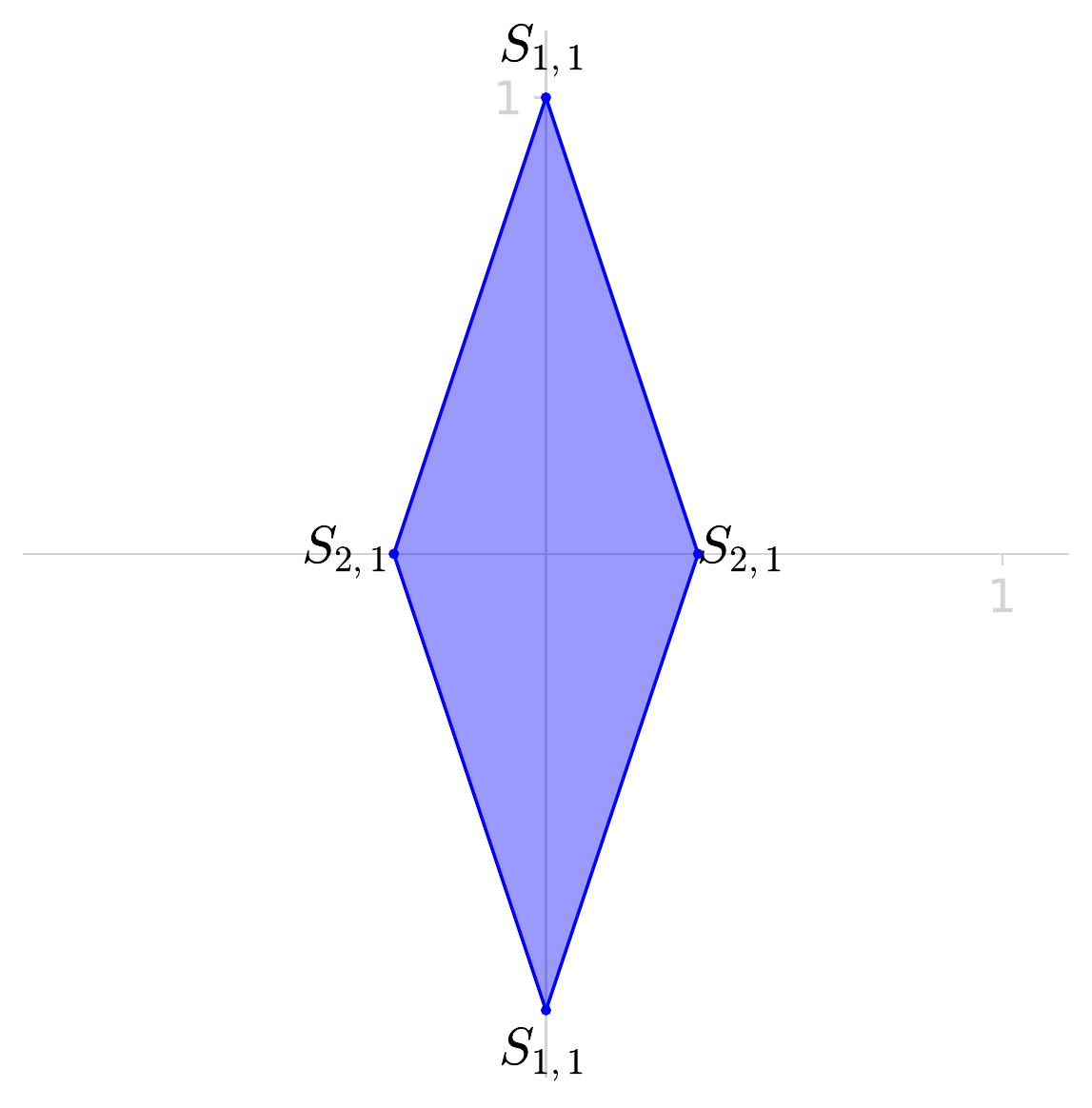}} \\ 
 $L=10^{{2}}_{{1}}$ & & $L=10^{{2}}_{{2}}$ & \\ 
 \quad & & \quad & \\ $\mathrm{Isom}(\mathbb{S}^3\setminus L) = \mathbb{{Z}}_2$ & & $\mathrm{Isom}(\mathbb{S}^3\setminus L) = \mathbb{{Z}}_2$ & \\ 
 \quad & & \quad & \\ 
 \includegraphics[width=1in]{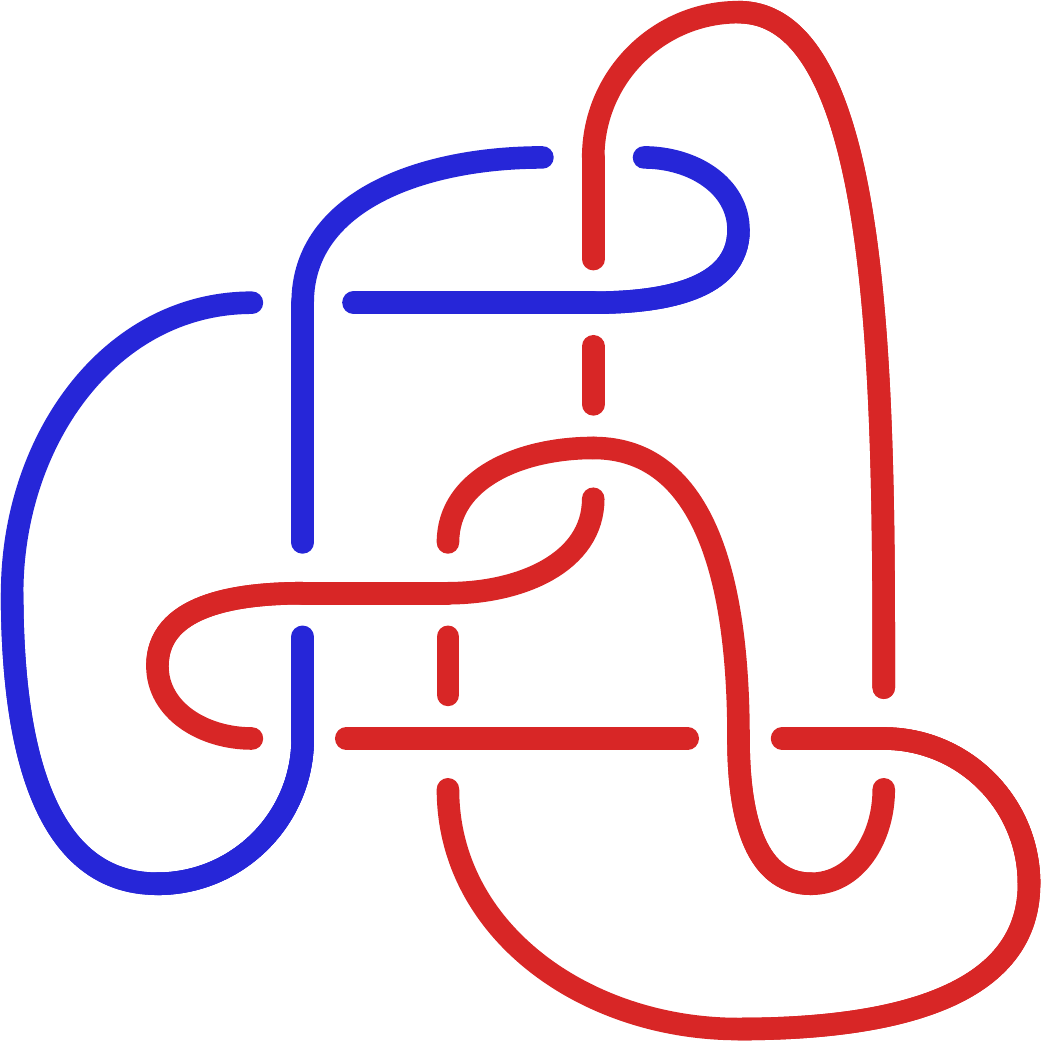}  & & \includegraphics[width=1in]{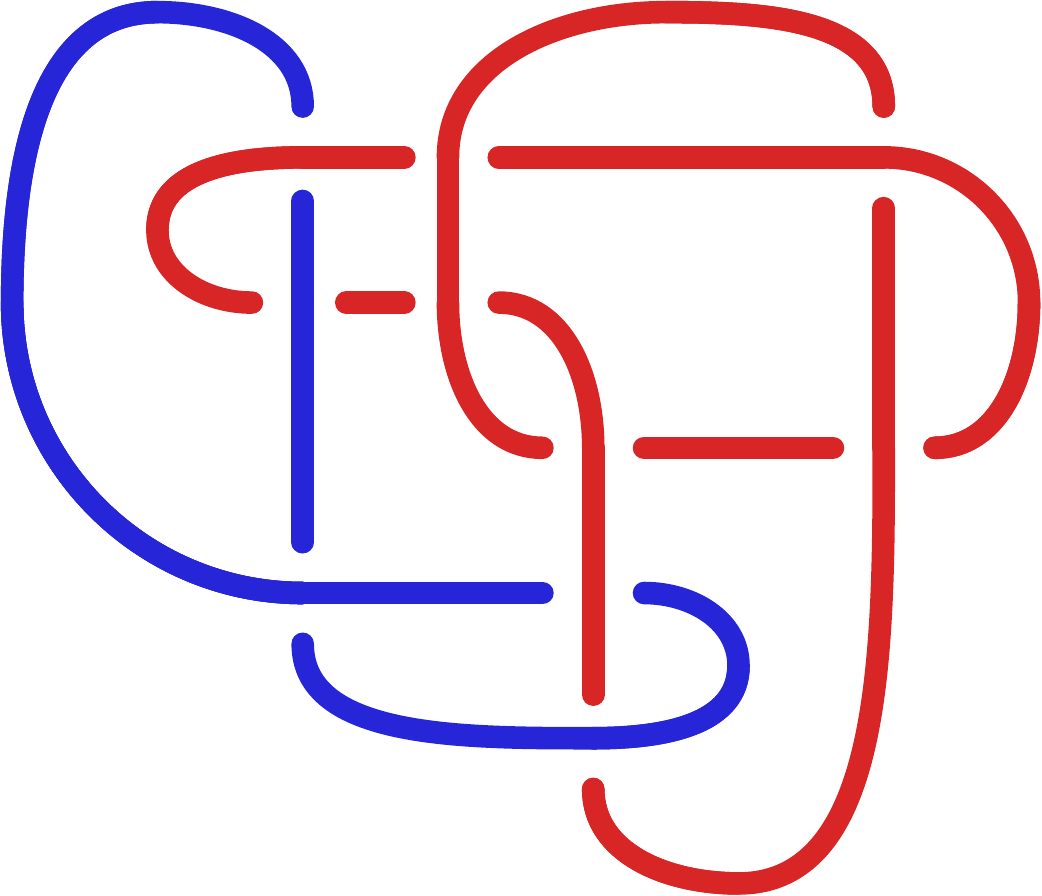} & \\ 
 \quad & & \quad & \\ 
 \hline  
\quad & \multirow{6}{*}{\Includegraphics[width=1.8in]{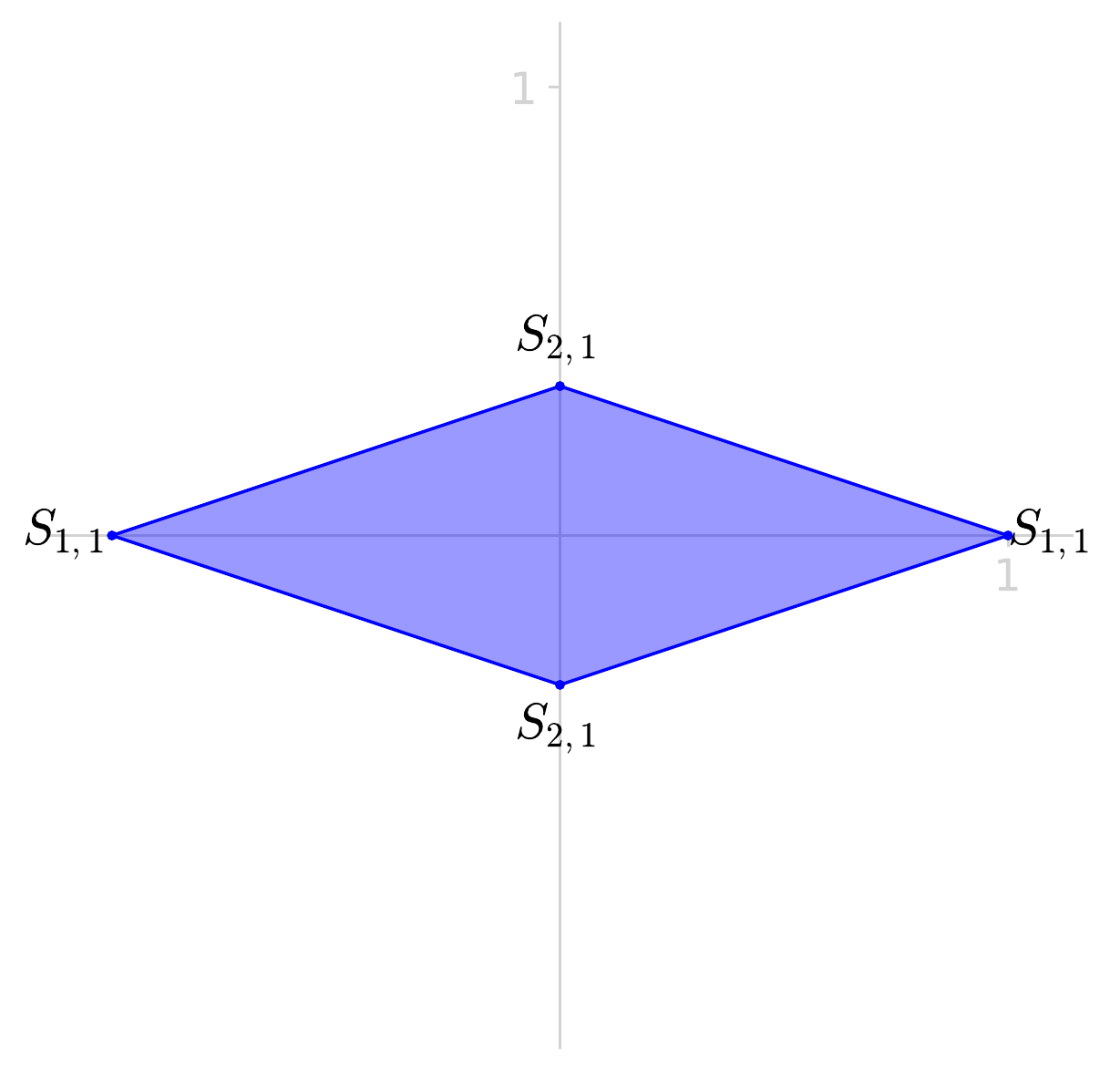}} & \quad & \multirow{6}{*}{\Includegraphics[width=1.8in]{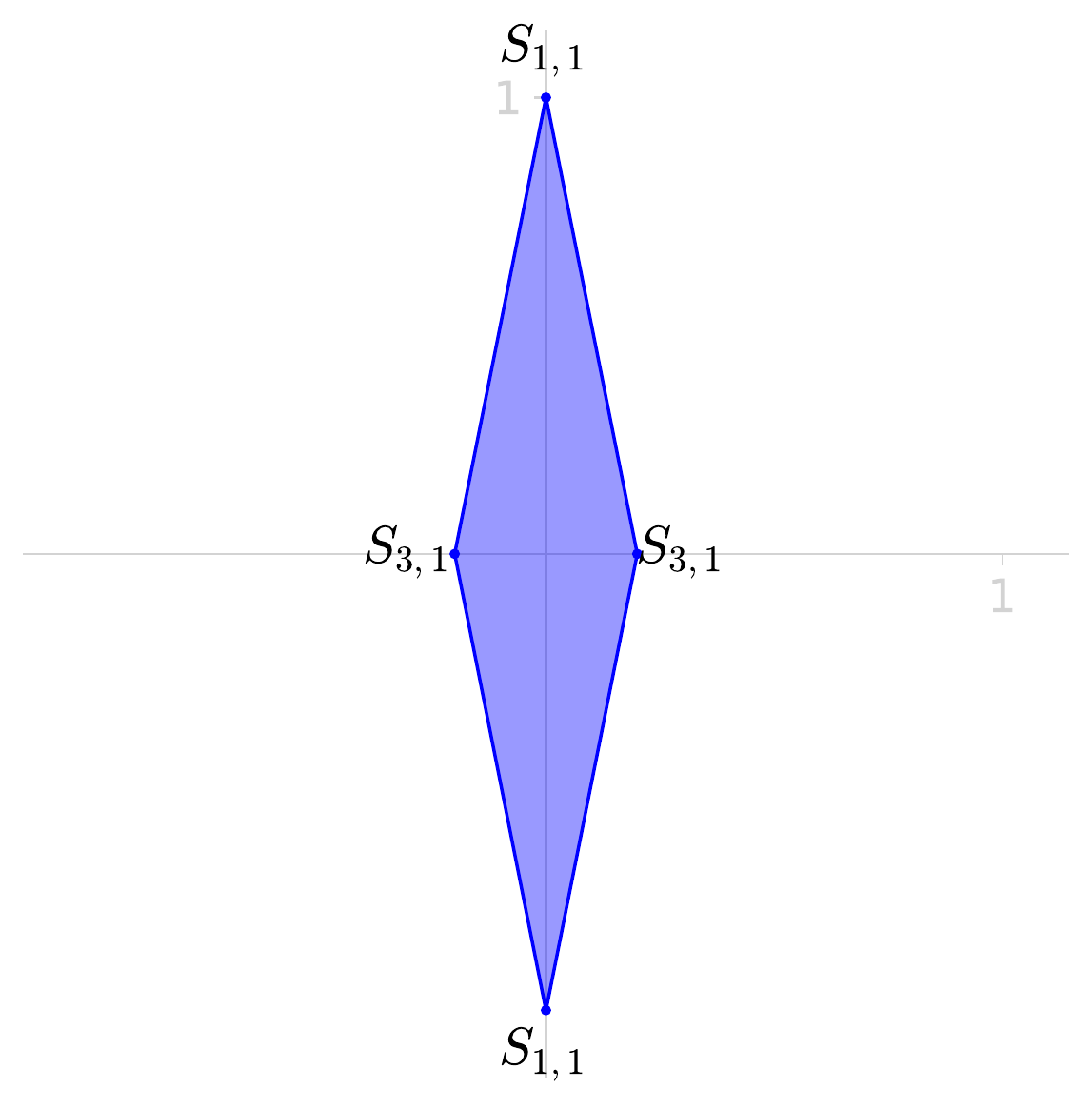}} \\ 
 $L=10^{{2}}_{{3}}$ & & $L=10^{{2}}_{{4}}$ & \\ 
 \quad & & \quad & \\ $\mathrm{Isom}(\mathbb{S}^3\setminus L) = \mathbb{{Z}}_2$ & & $\mathrm{Isom}(\mathbb{S}^3\setminus L) = \mathbb{{Z}}_2$ & \\ 
 \quad & & \quad & \\ 
 \includegraphics[width=1in]{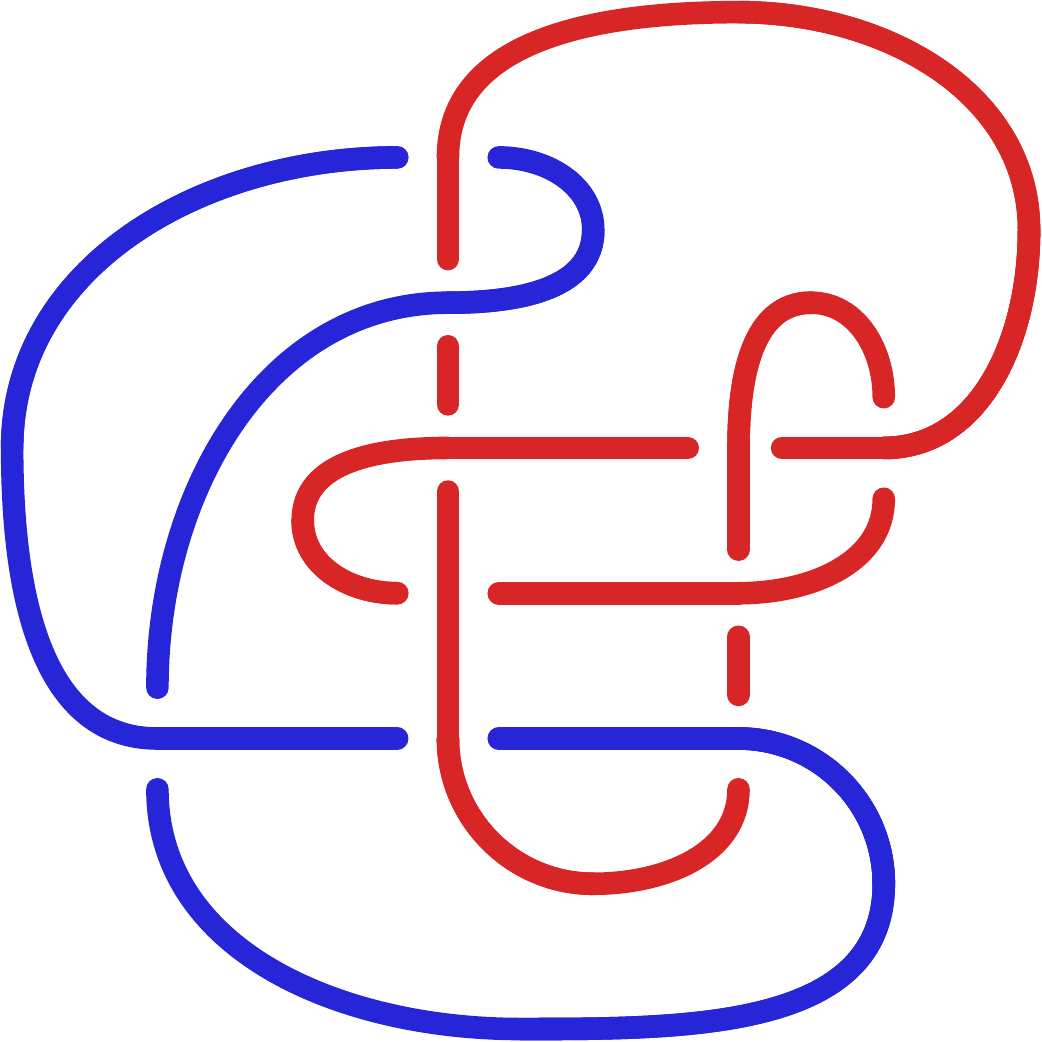}  & & \includegraphics[width=1in]{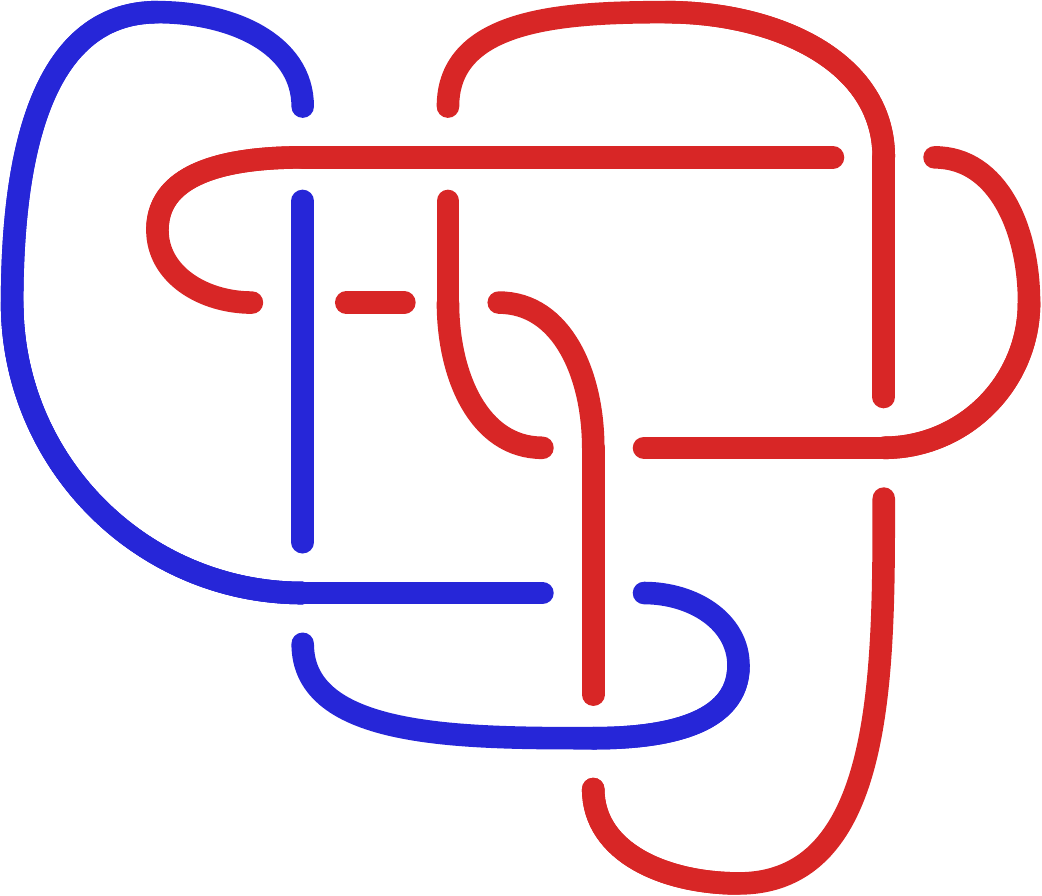} & \\ 
 \quad & & \quad & \\ 
 \hline  
\end{tabular} 
 \newpage \begin{tabular}{|c|c|c|c|} 
 \hline 
 Link & Norm Ball & Link & Norm Ball \\ 
 \hline 
\quad & \multirow{6}{*}{\Includegraphics[width=1.8in]{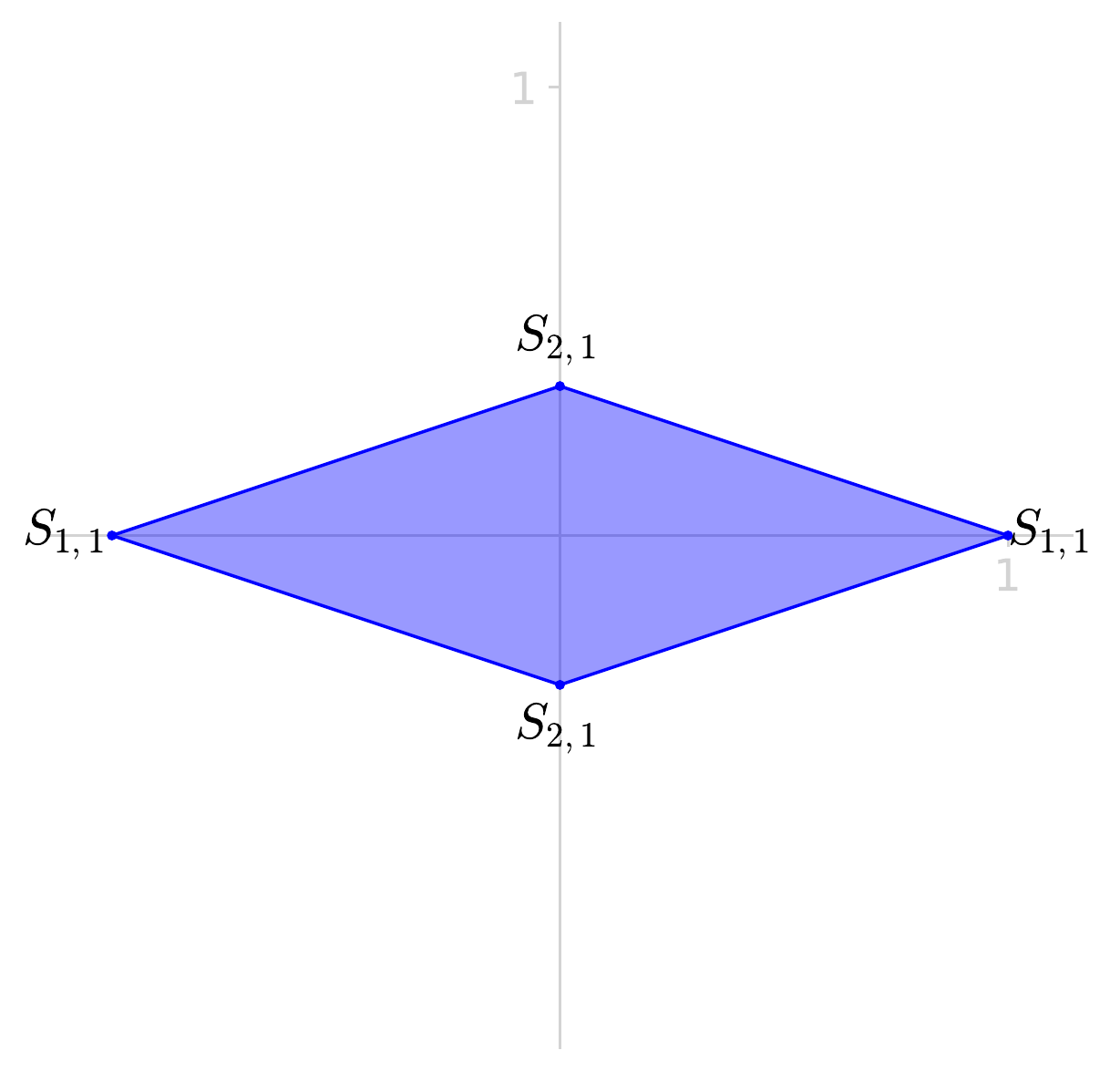}} & \quad & \multirow{6}{*}{\Includegraphics[width=1.8in]{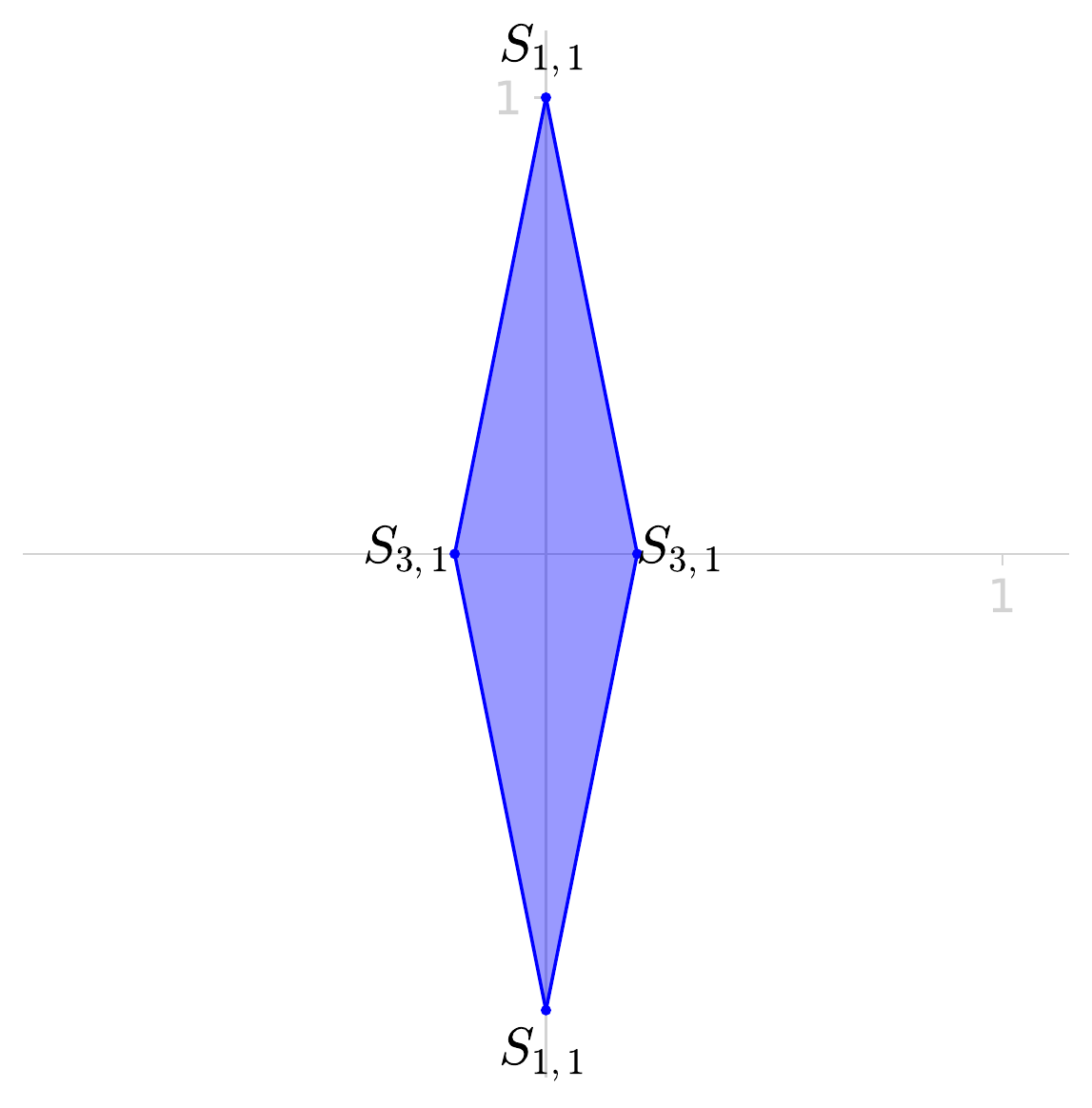}} \\ 
 $L=10^{{2}}_{{5}}$ & & $L=10^{{2}}_{{6}}$ & \\ 
 \quad & & \quad & \\ $\mathrm{Isom}(\mathbb{S}^3\setminus L) = \mathbb{{Z}}_2$ & & $\mathrm{Isom}(\mathbb{S}^3\setminus L) = \mathbb{{Z}}_2$ & \\ 
 \quad & & \quad & \\ 
 \includegraphics[width=1in]{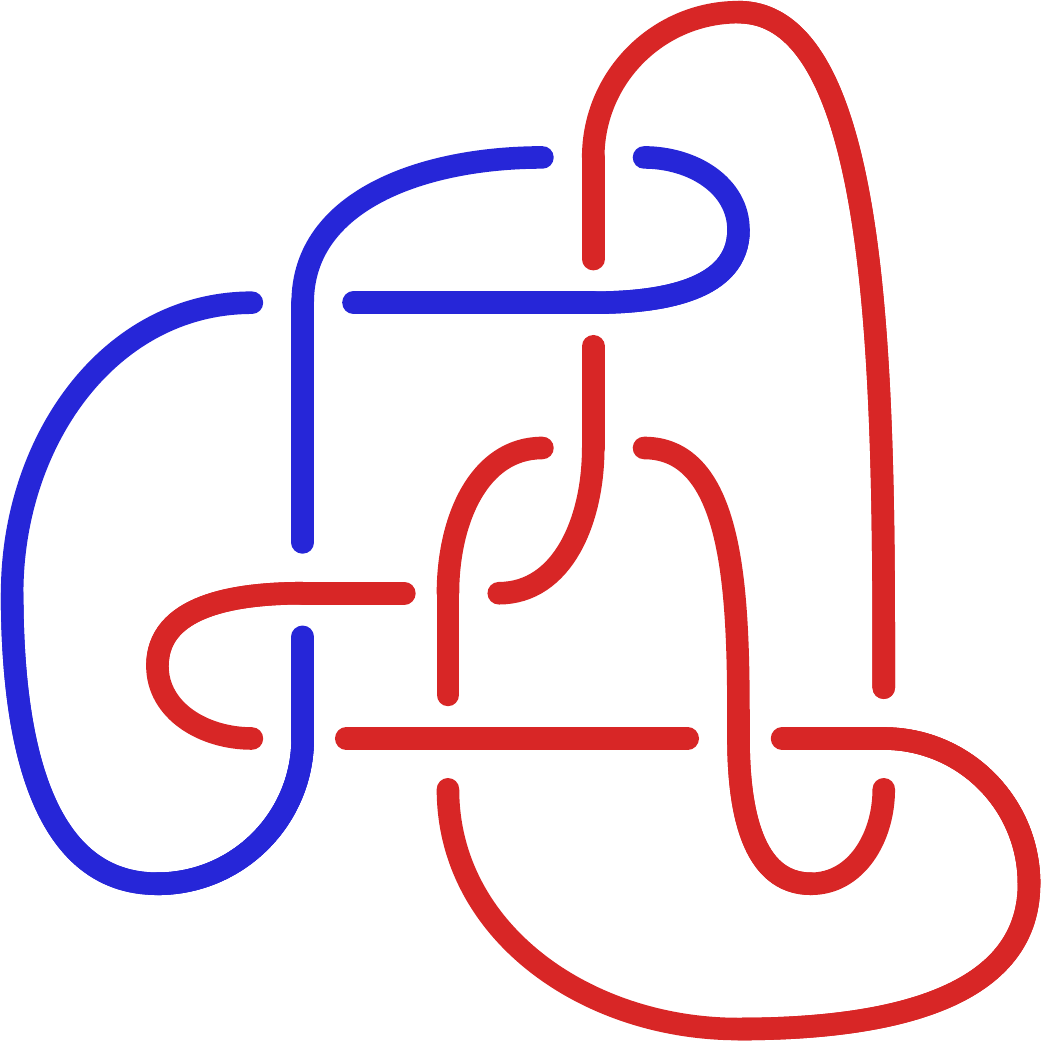}  & & \includegraphics[width=1in]{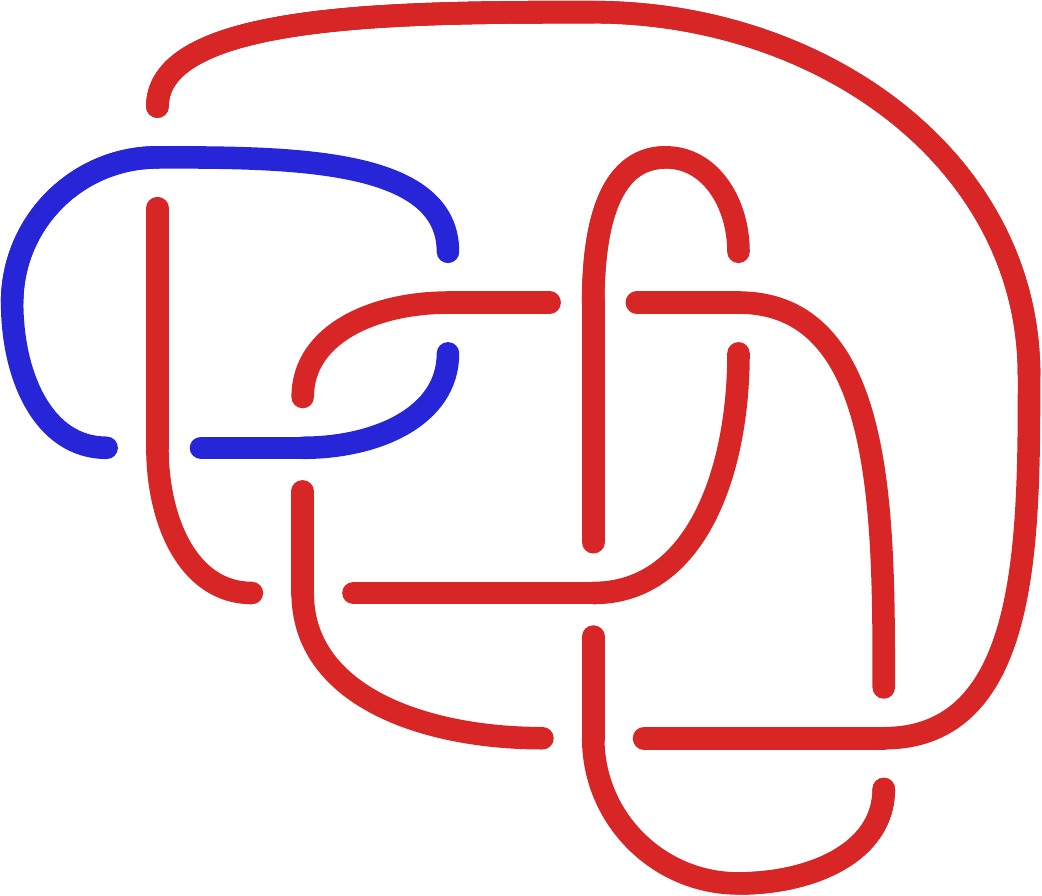} & \\ 
 \quad & & \quad & \\ 
 \hline  
\quad & \multirow{6}{*}{\Includegraphics[width=1.8in]{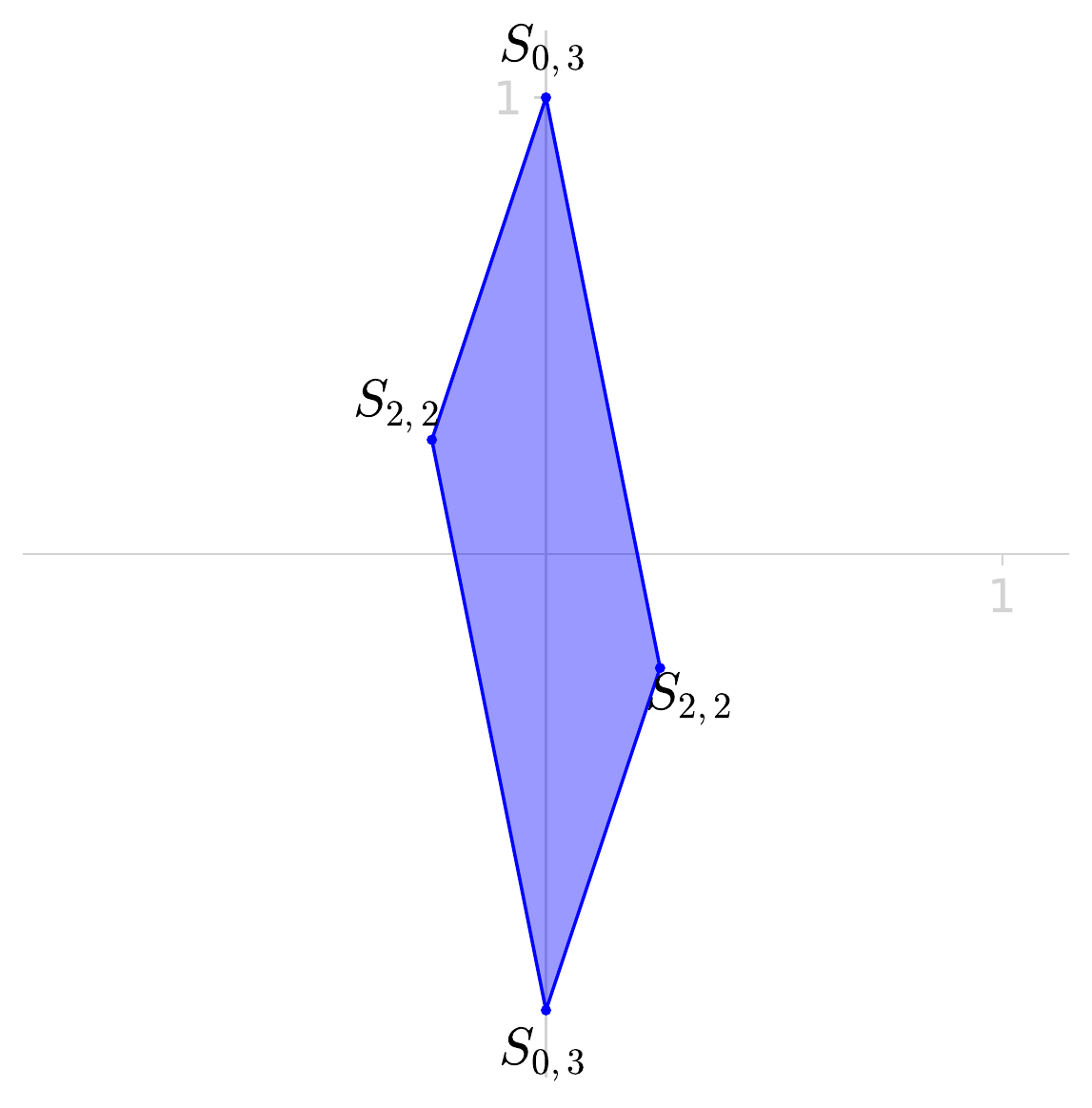}} & \quad & \multirow{6}{*}{\Includegraphics[width=1.8in]{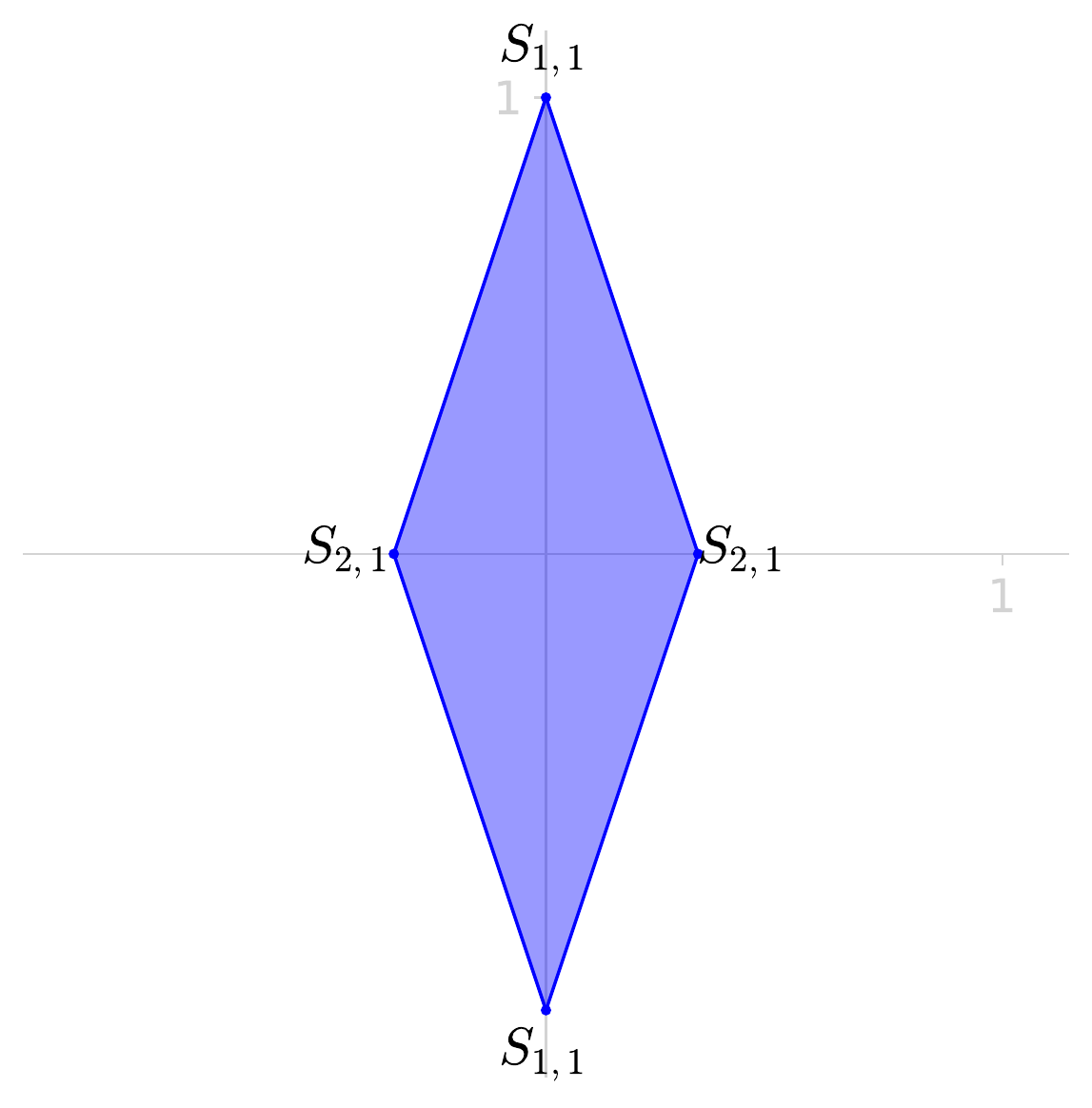}} \\ 
 $L=10^{{2}}_{{7}}$ & & $L=10^{{2}}_{{8}}$ & \\ 
 \quad & & \quad & \\ $\mathrm{Isom}(\mathbb{S}^3\setminus L) = \mathbb{{Z}}_2$ & & $\mathrm{Isom}(\mathbb{S}^3\setminus L) = \mathbb{{Z}}_2\oplus\mathbb{{Z}}_2$ & \\ 
 \quad & & \quad & \\ 
 \includegraphics[width=1in]{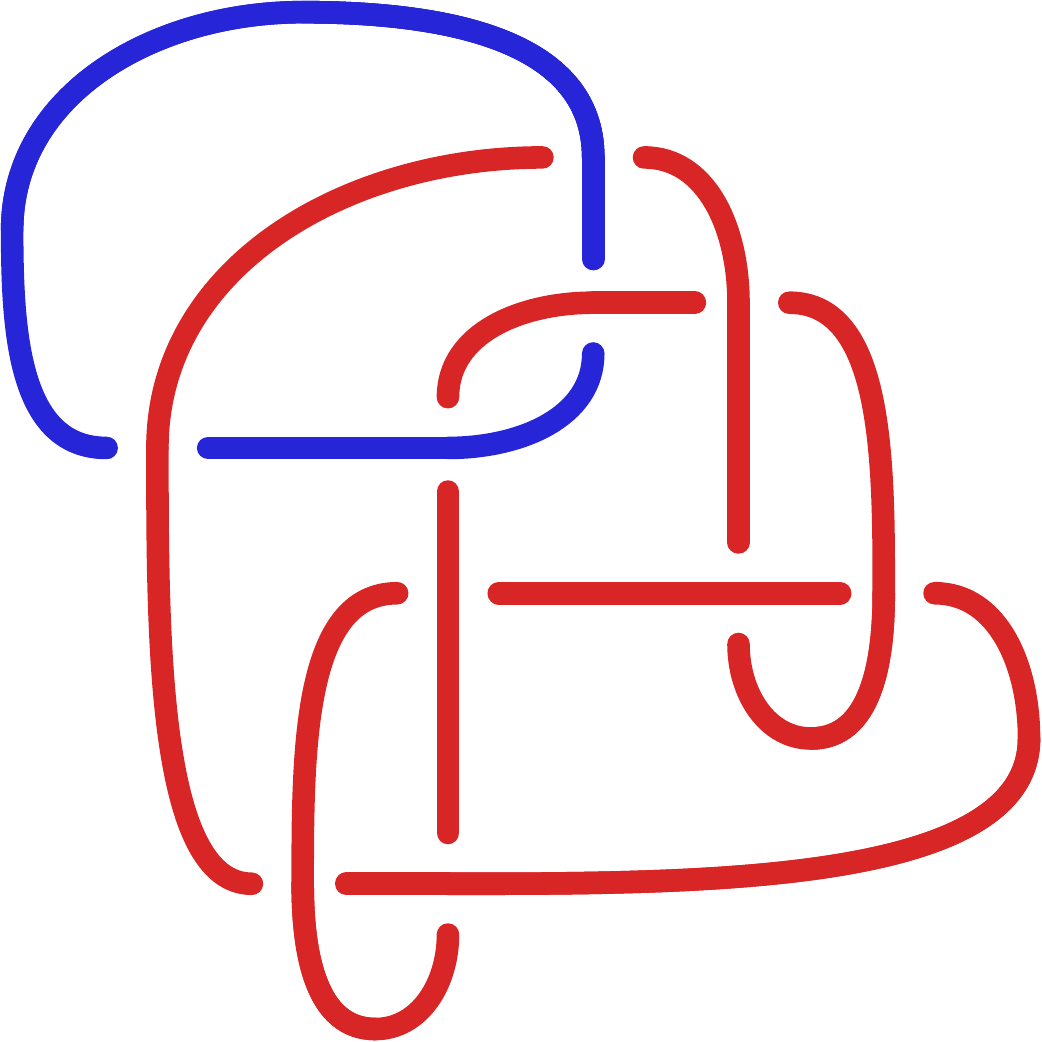}  & & \includegraphics[width=1in]{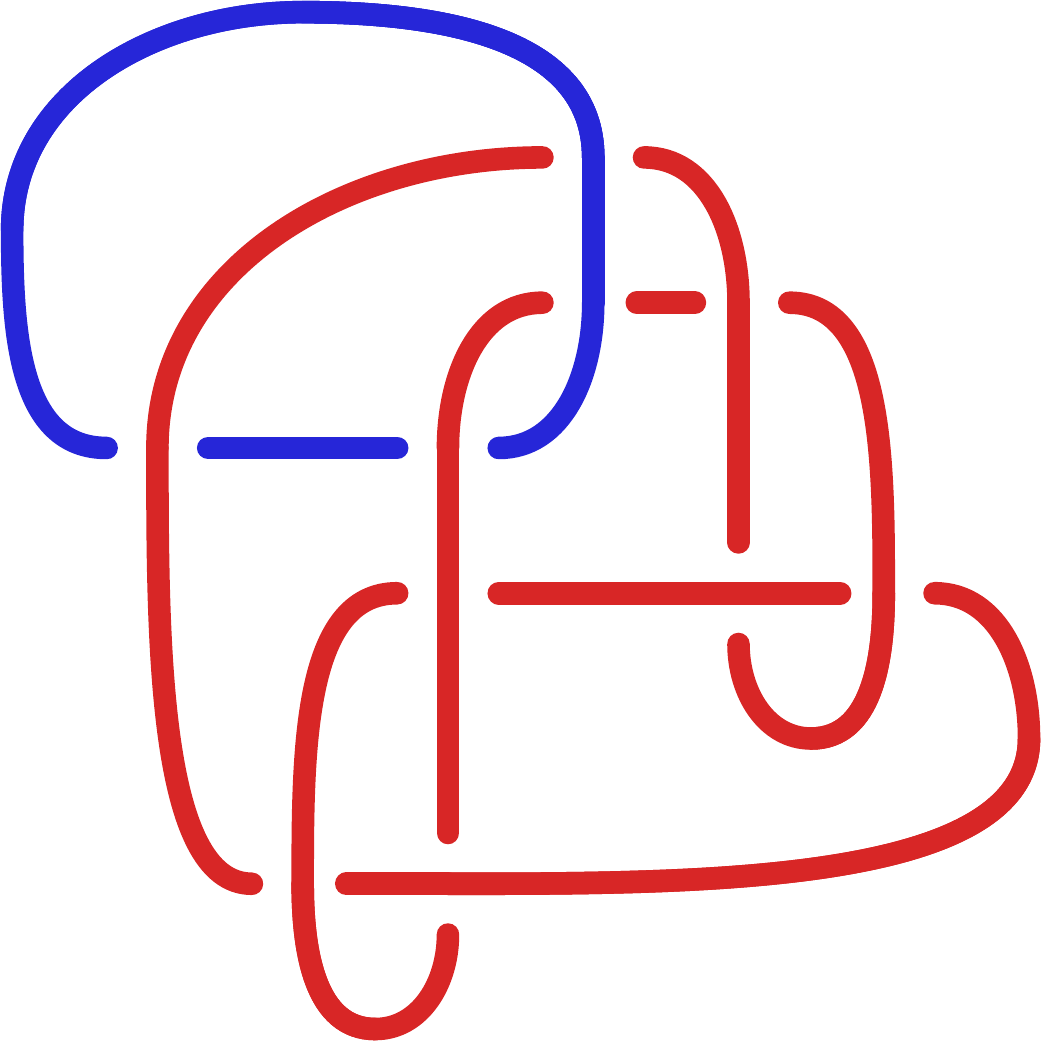} & \\ 
 \quad & & \quad & \\ 
 \hline  
\quad & \multirow{6}{*}{\Includegraphics[width=1.8in]{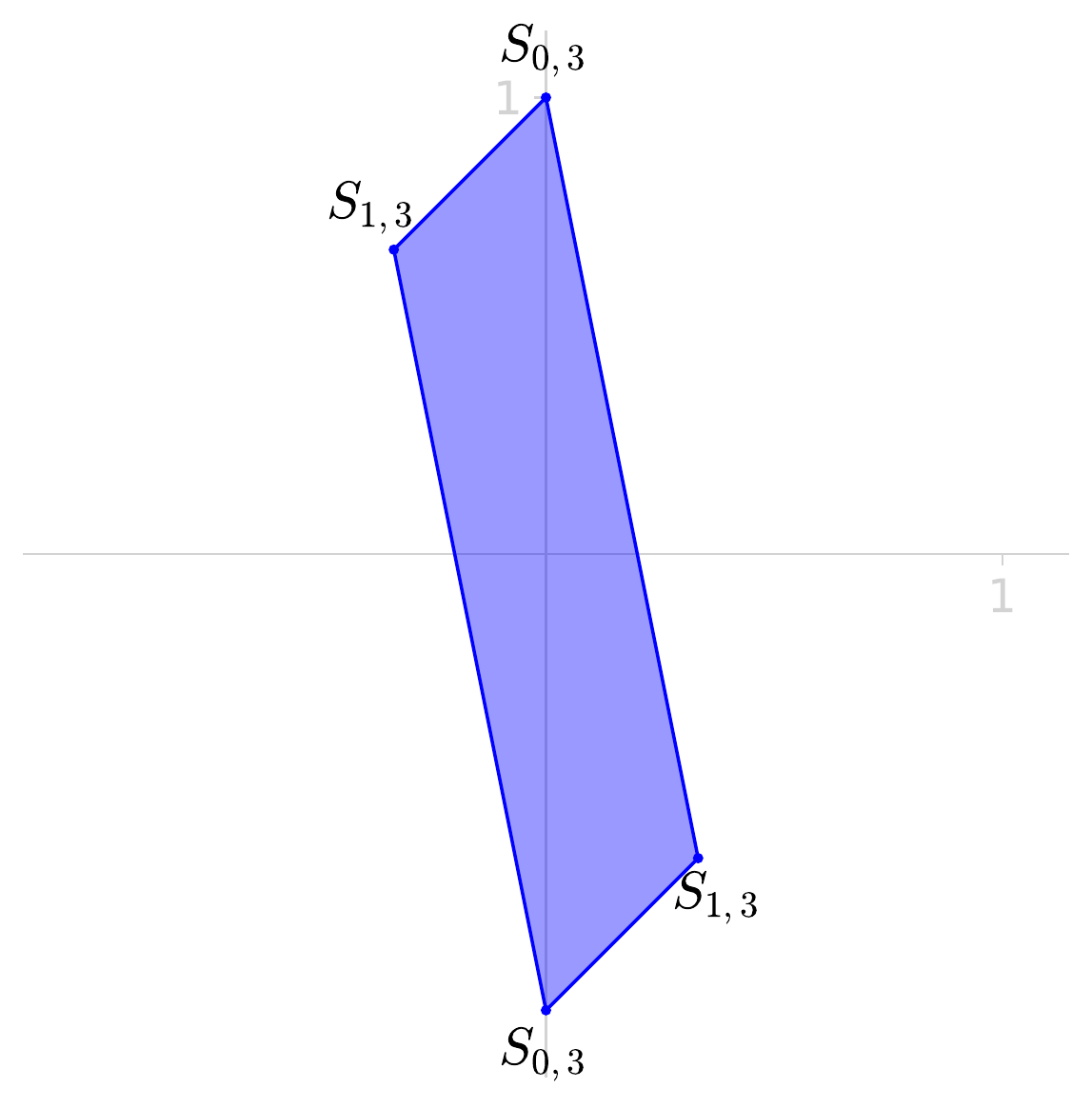}} & \quad & \multirow{6}{*}{\Includegraphics[width=1.8in]{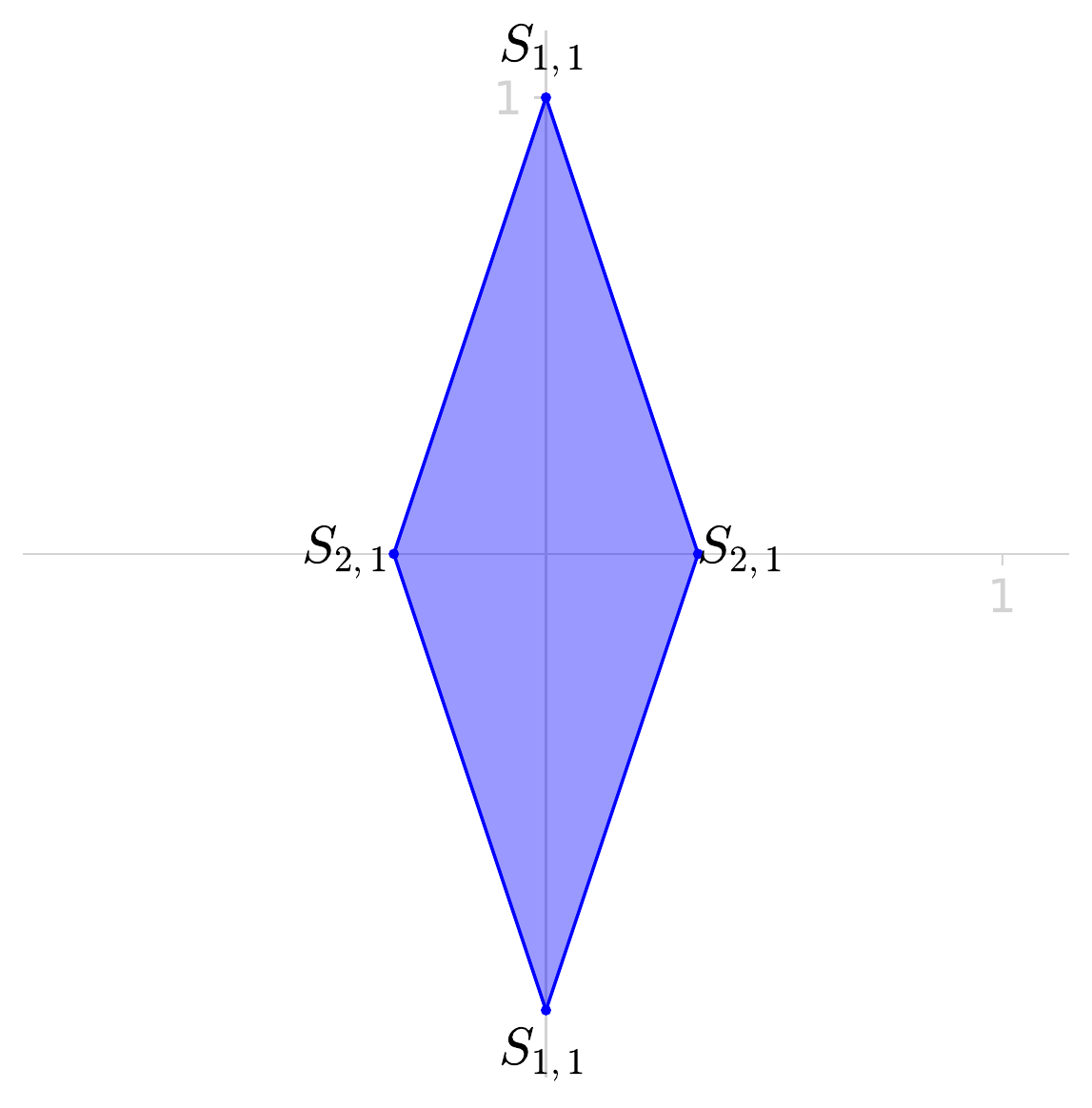}} \\ 
 $L=10^{{2}}_{{9}}$ & & $L=10^{{2}}_{{10}}$ & \\ 
 \quad & & \quad & \\ $\mathrm{Isom}(\mathbb{S}^3\setminus L) = \mathbb{{Z}}_2$ & & $\mathrm{Isom}(\mathbb{S}^3\setminus L) = \mathbb{{Z}}_2\oplus\mathbb{{Z}}_2$ & \\ 
 \quad & & \quad & \\ 
 \includegraphics[width=1in]{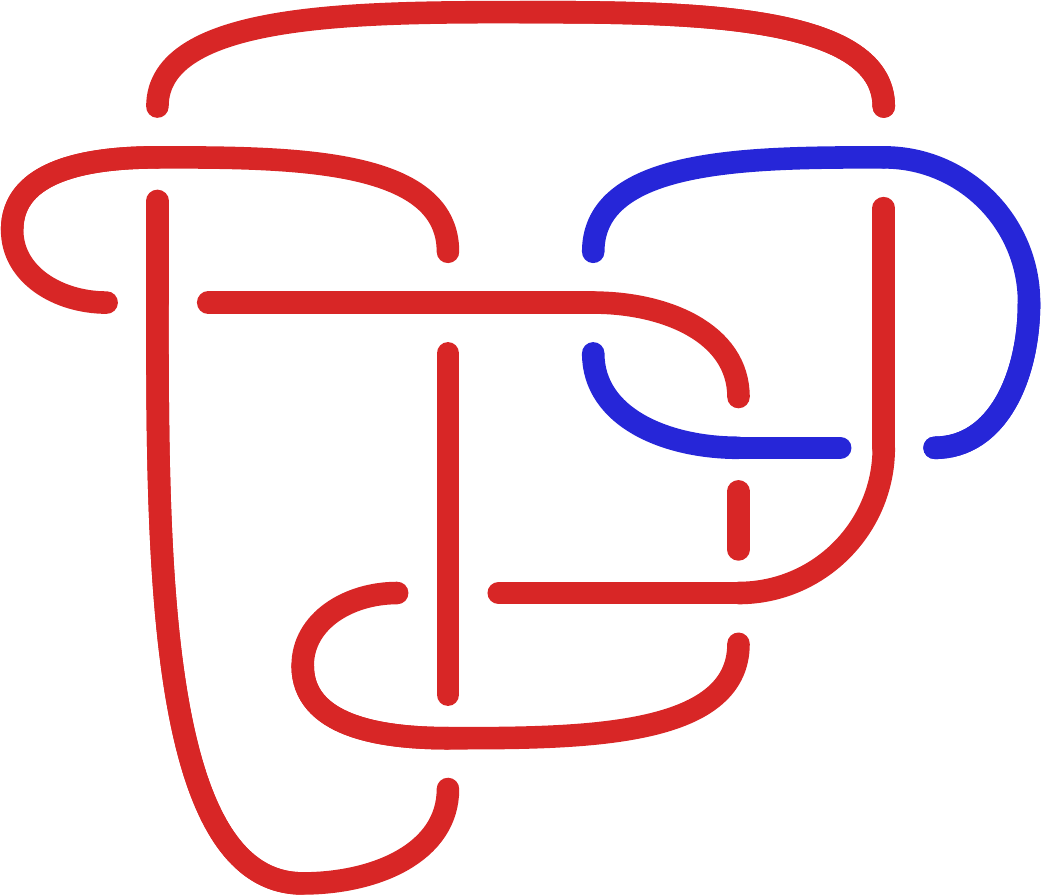}  & & \includegraphics[width=1in]{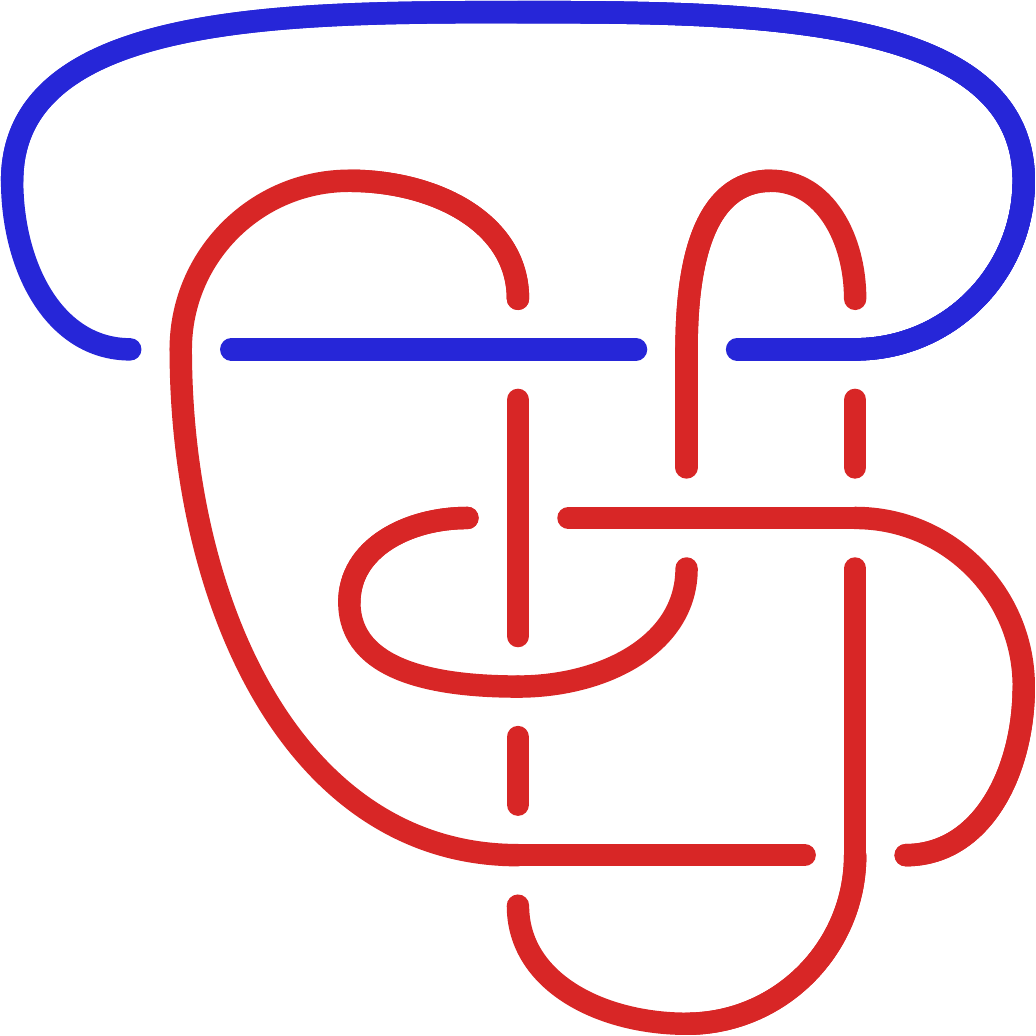} & \\ 
 \quad & & \quad & \\ 
 \hline  
 \quad & \multirow{6}{*}{\Includegraphics[width=1.8in]{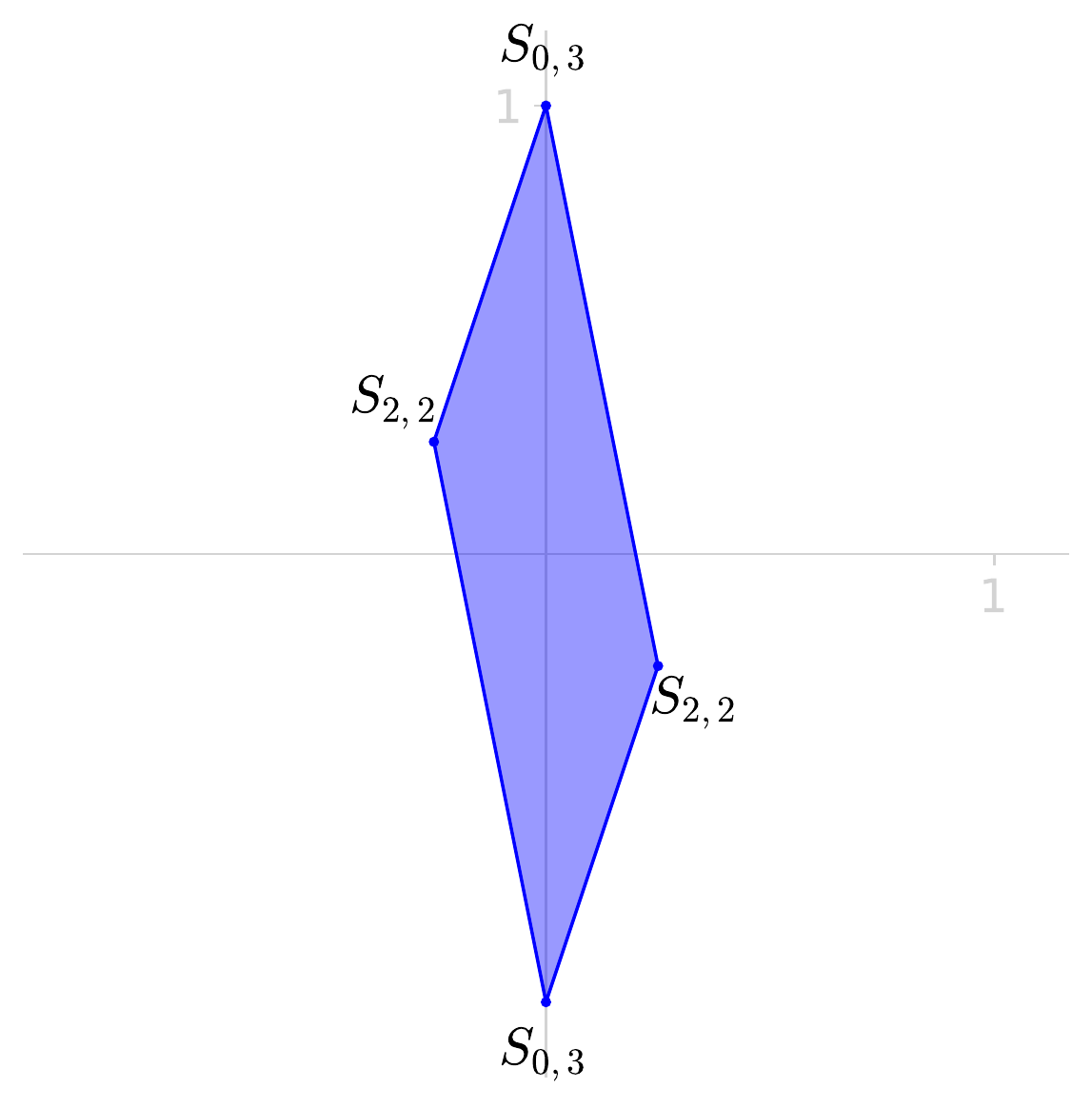}} & \quad & \multirow{6}{*}{\Includegraphics[width=1.8in]{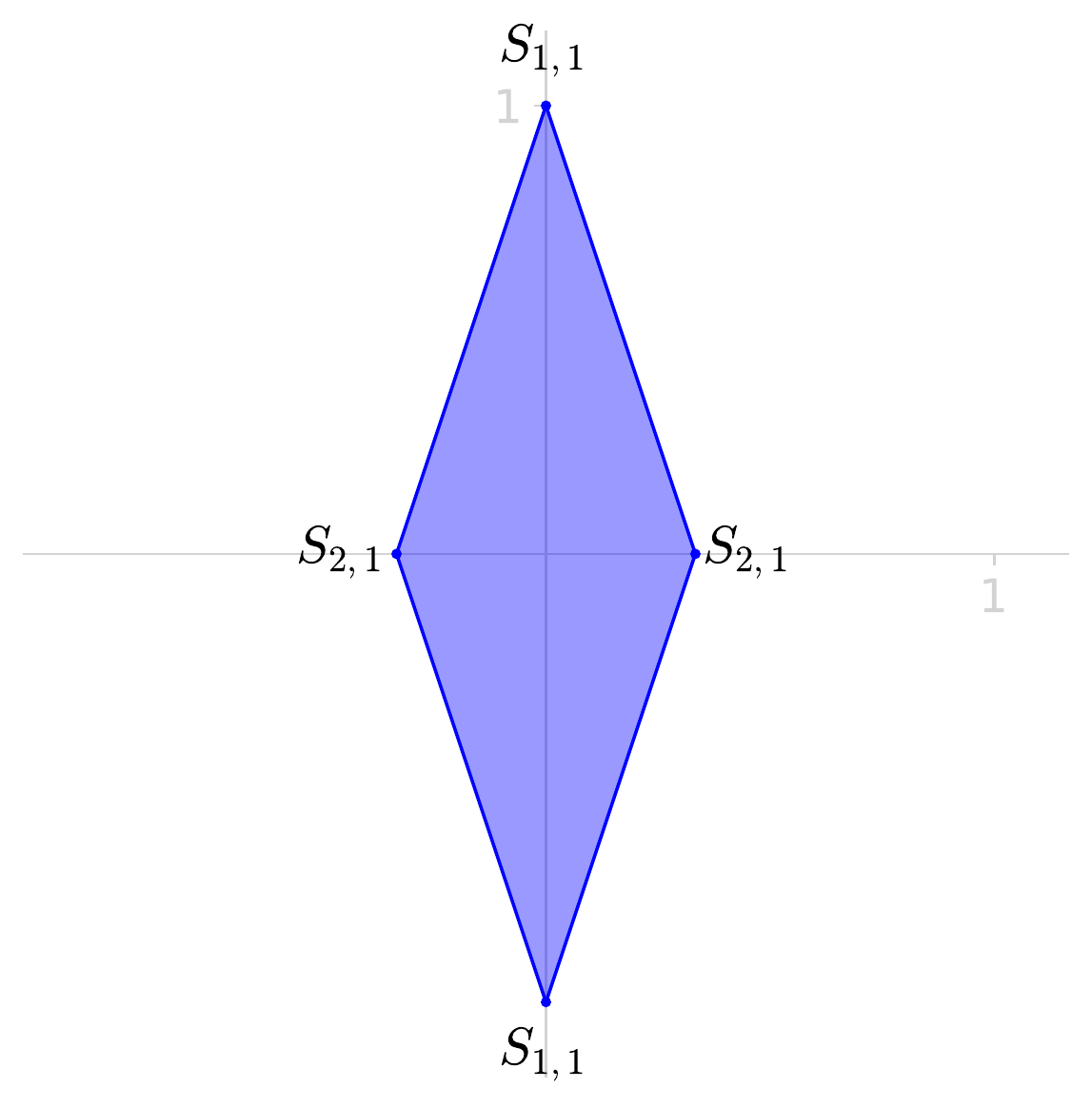}} \\ 
 $L=10^{{2}}_{{11}}$ & & $L=10^{{2}}_{{12}}$ & \\ 
 \quad & & \quad & \\ $\mathrm{Isom}(\mathbb{S}^3\setminus L) = \mathbb{{Z}}_2$ & & $\mathrm{Isom}(\mathbb{S}^3\setminus L) = \mathbb{{Z}}_2\oplus\mathbb{{Z}}_2$ & \\ 
 \quad & & \quad & \\ 
 \includegraphics[width=1in]{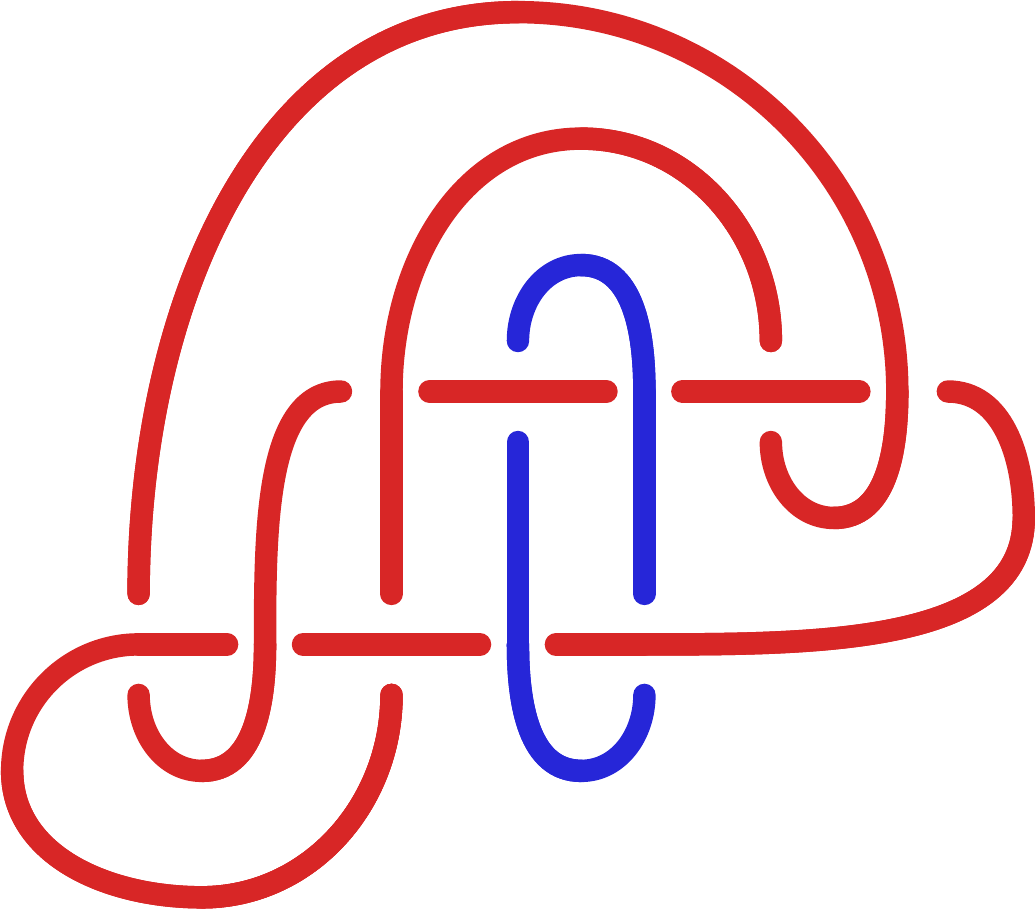}  & & \includegraphics[width=1in]{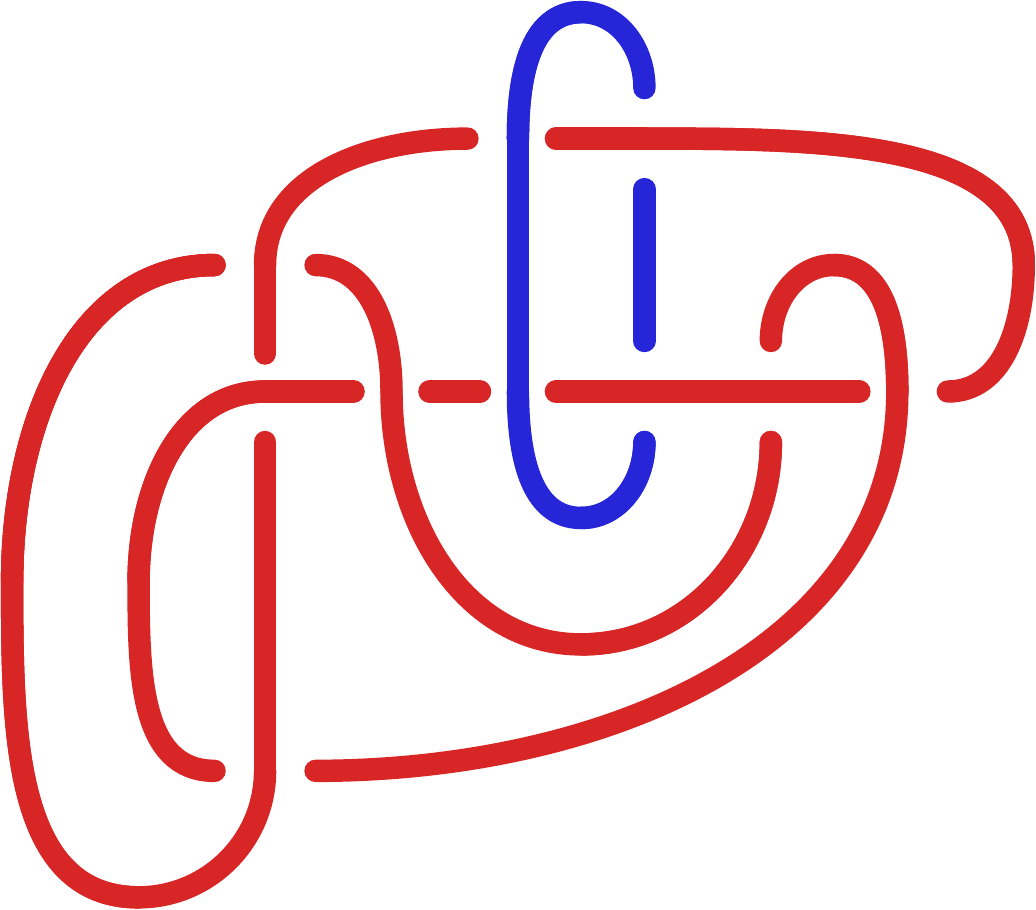} & \\ 
 \quad & & \quad & \\ 
 \hline  
\quad & \multirow{6}{*}{\Includegraphics[width=1.8in]{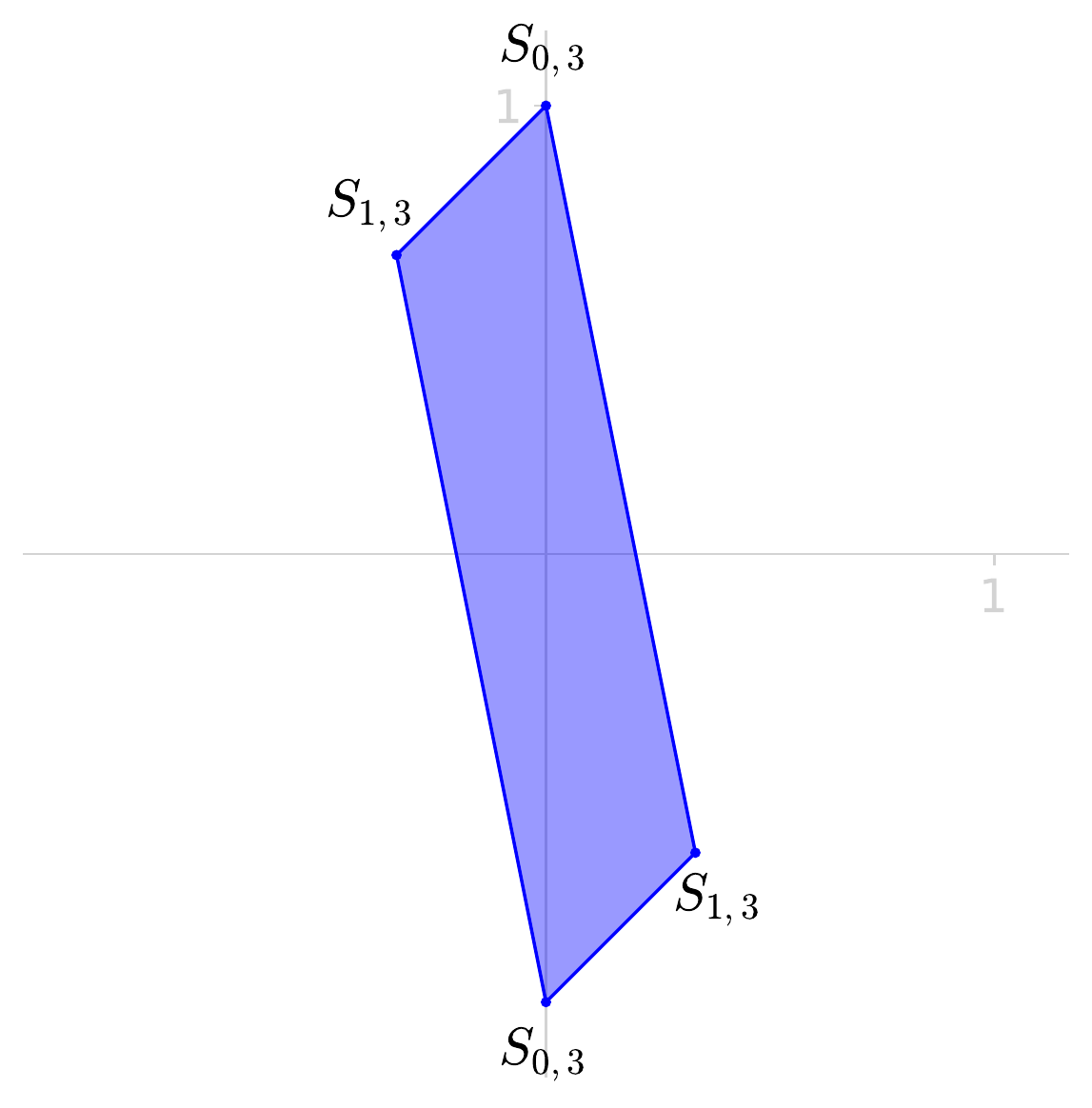}} & \quad & \multirow{6}{*}{\Includegraphics[width=1.8in]{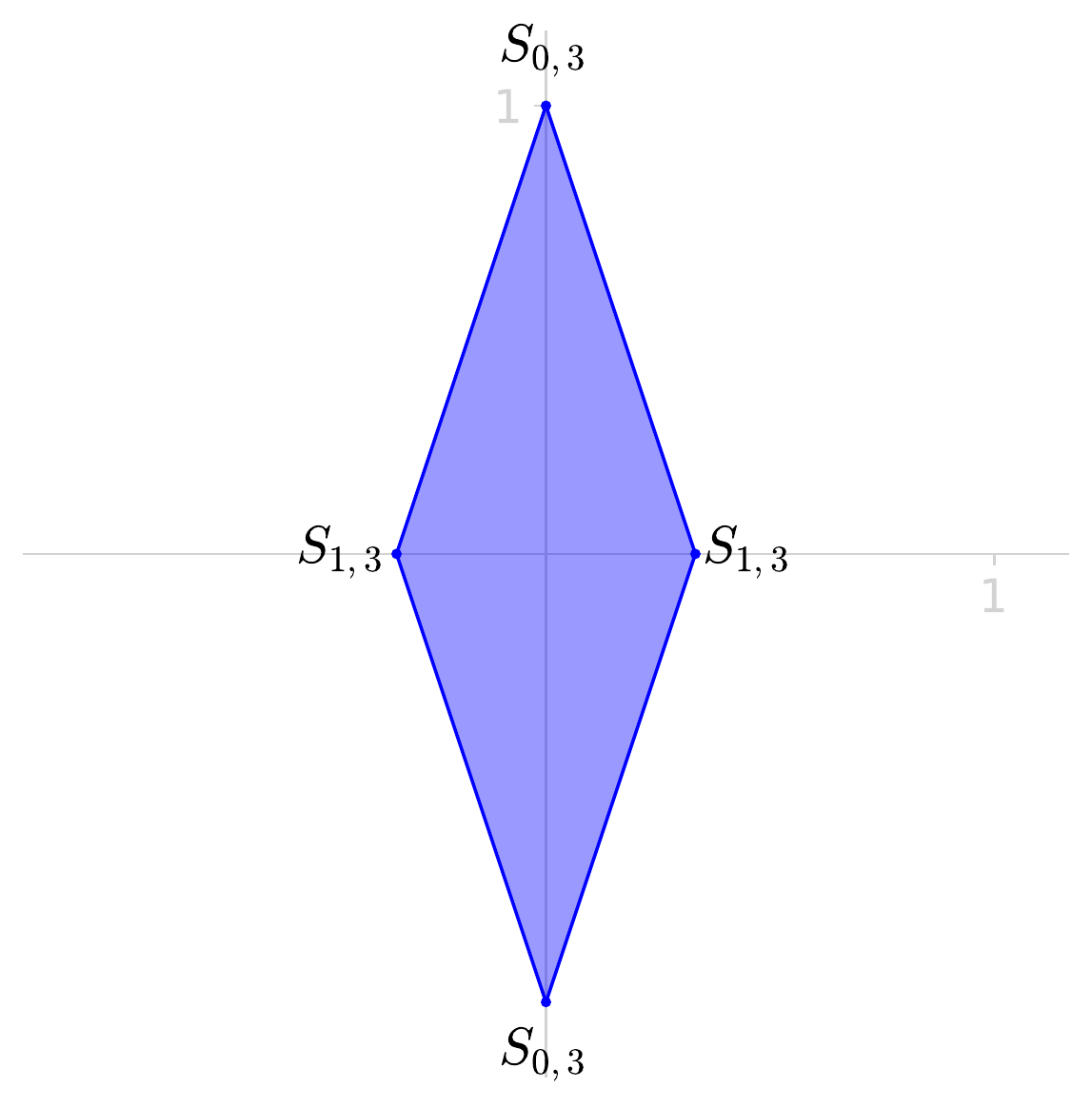}} \\ 
 $L=10^{{2}}_{{13}}$ & & $L=10^{{2}}_{{14}}$ & \\ 
 \quad & & \quad & \\ $\mathrm{Isom}(\mathbb{S}^3\setminus L) = \mathbb{{Z}}_2$ & & $\mathrm{Isom}(\mathbb{S}^3\setminus L) = \mathbb{{Z}}_2$ & \\ 
 \quad & & \quad & \\ 
 \includegraphics[width=1in]{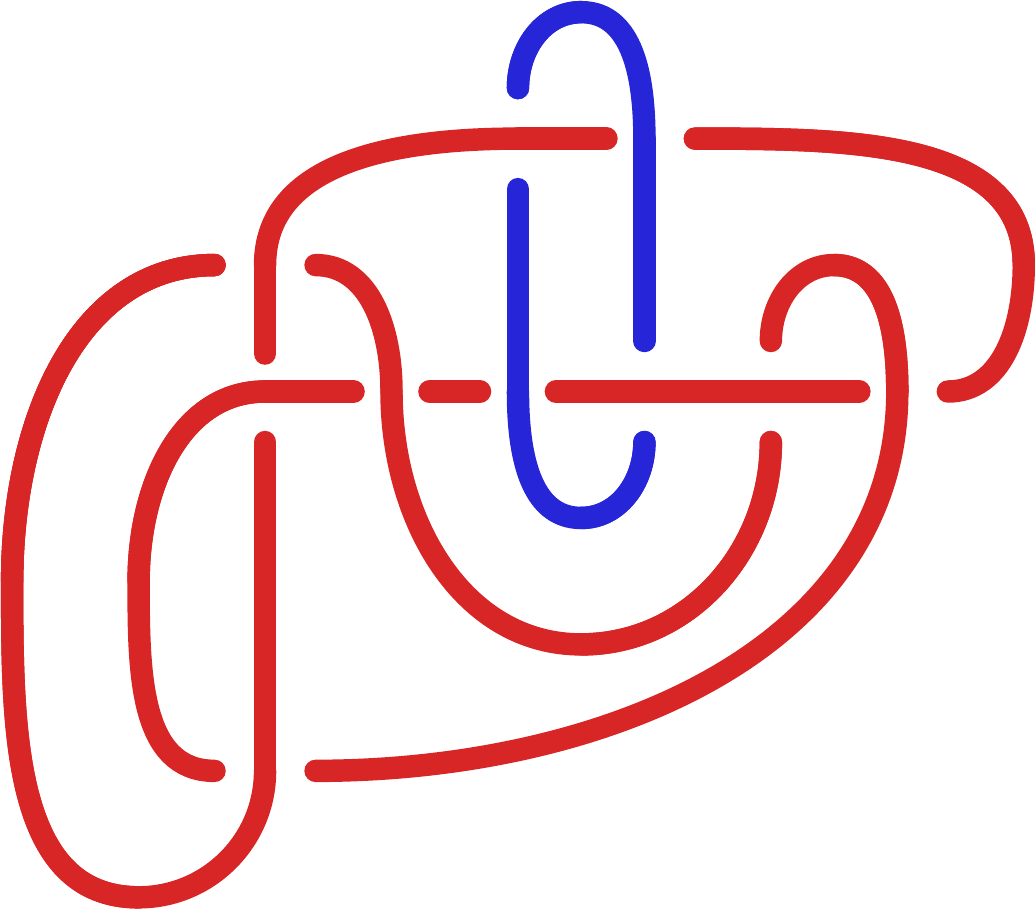}  & & \includegraphics[width=1in]{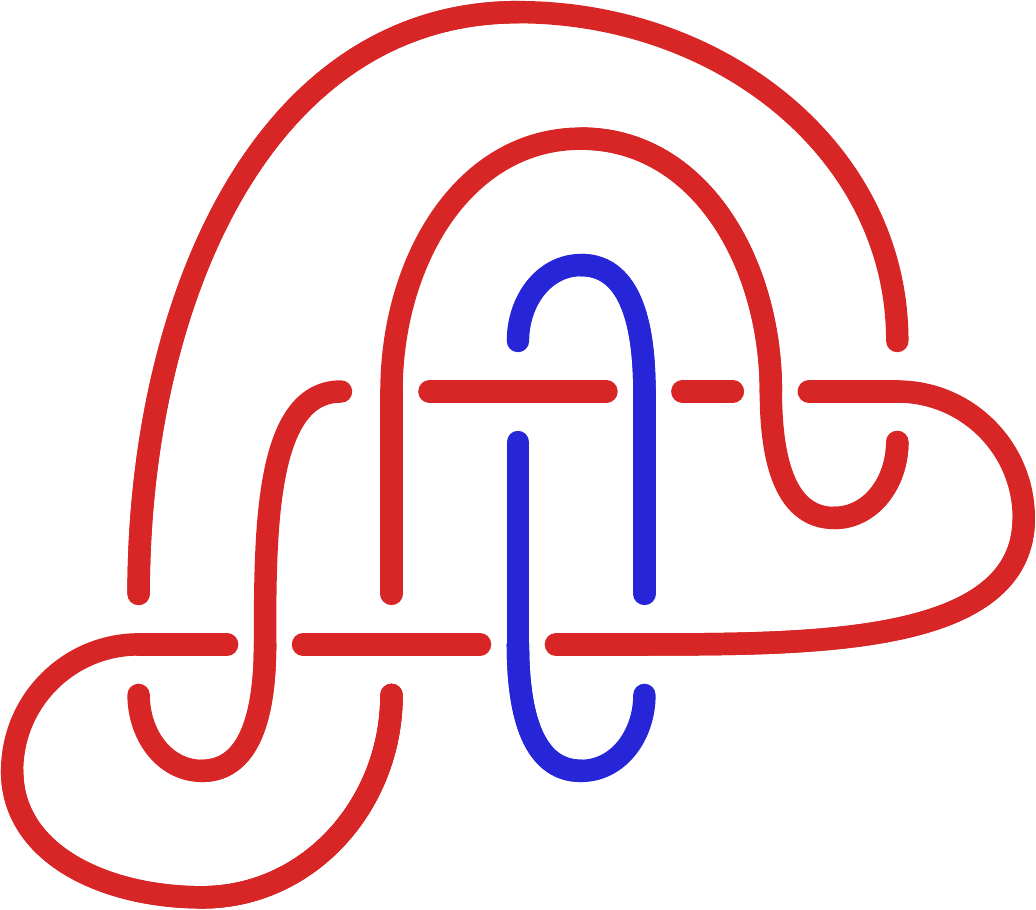} & \\ 
 \quad & & \quad & \\ 
 \hline
\end{tabular}
 \newpage \begin{tabular}{|c|c|c|c|} 
 \hline 
 Link & Norm Ball & Link & Norm Ball \\ 
 \hline 
\quad & \multirow{6}{*}{\Includegraphics[width=1.8in]{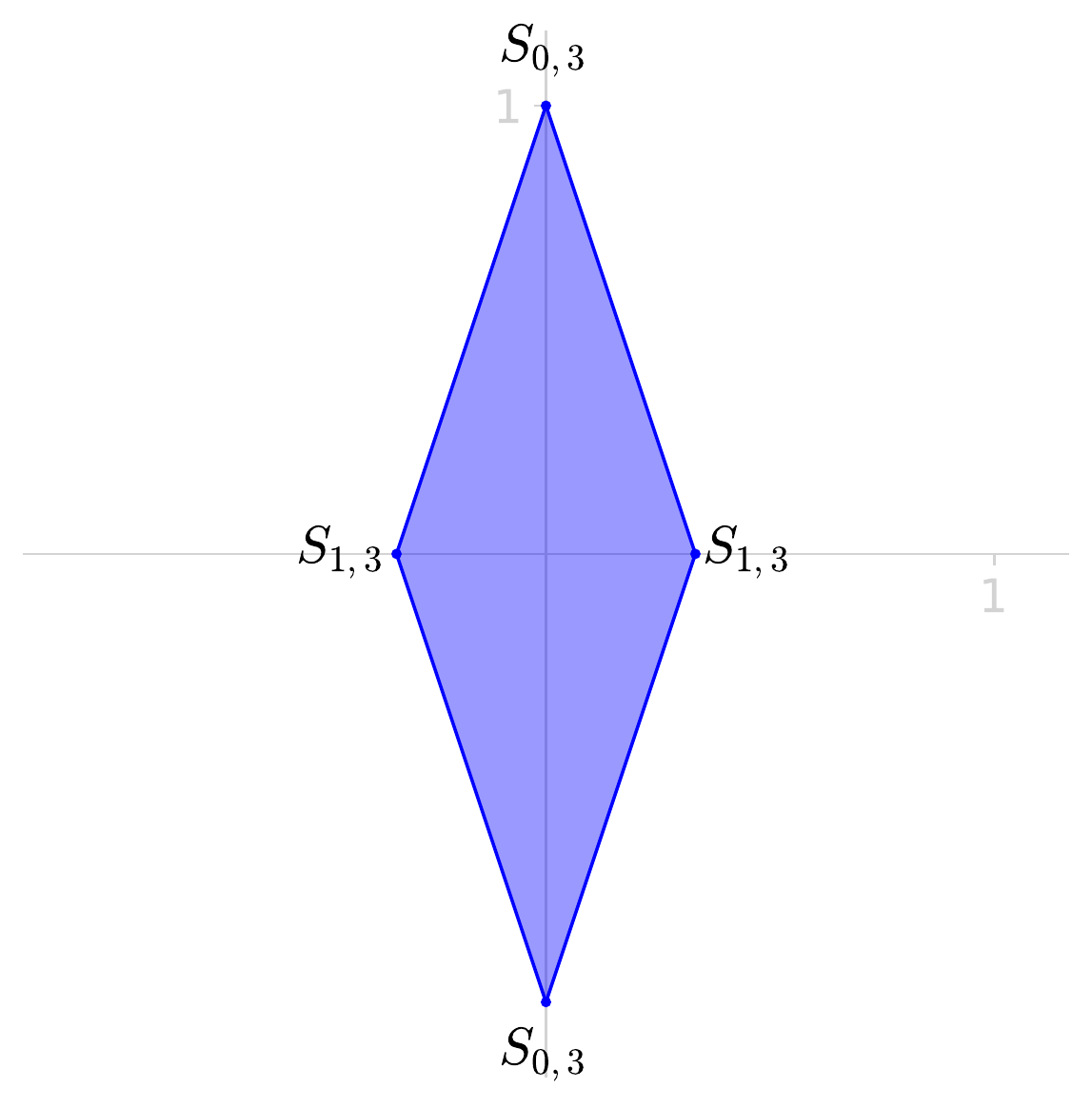}} & \quad & \multirow{6}{*}{\Includegraphics[width=1.8in]{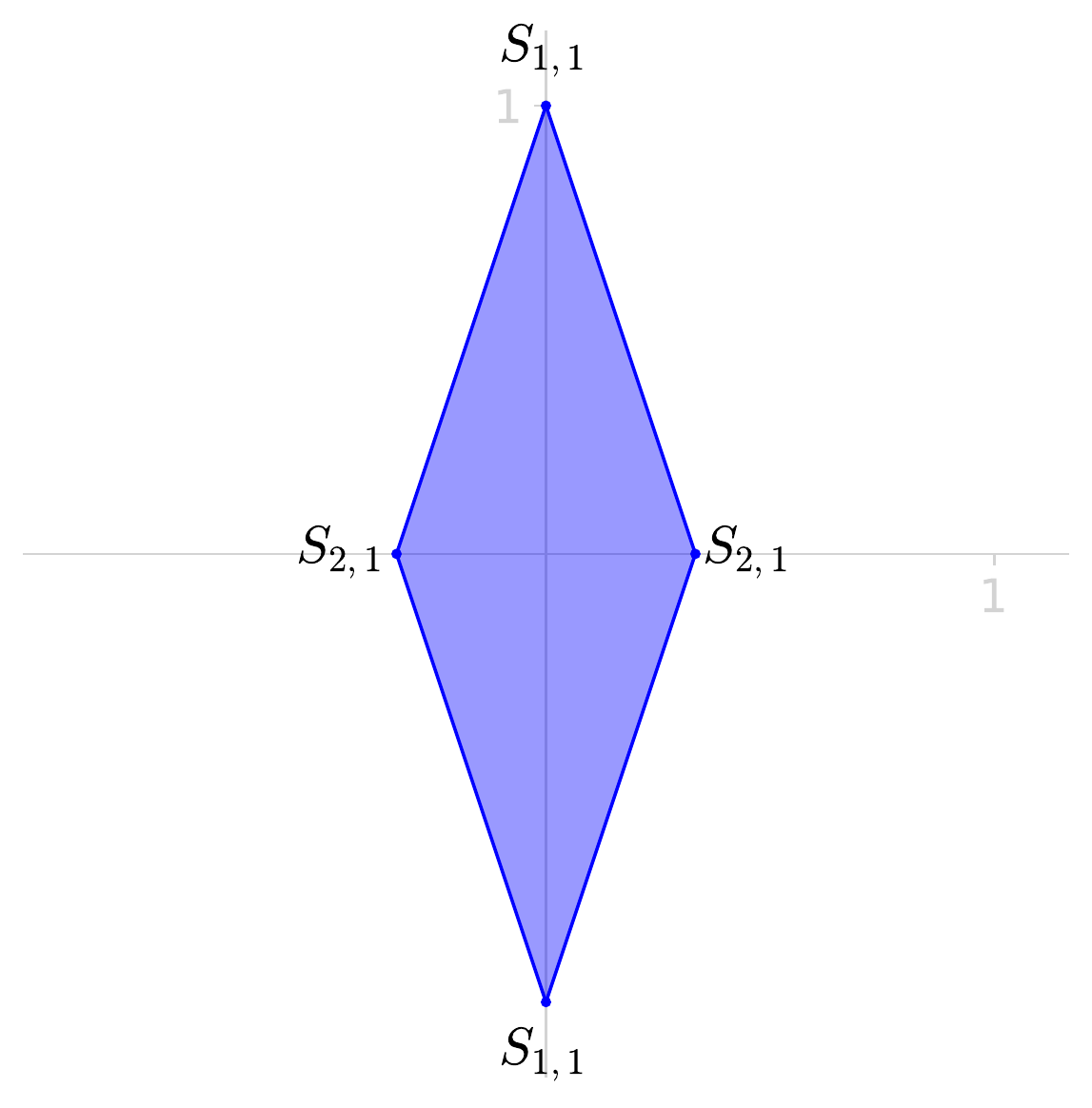}} \\ 
 $L=10^{{2}}_{{15}}$ & & $L=10^{{2}}_{{16}}$ & \\ 
 \quad & & \quad & \\ $\mathrm{Isom}(\mathbb{S}^3\setminus L) = \mathbb{{Z}}_2$ & & $\mathrm{Isom}(\mathbb{S}^3\setminus L) = \mathbb{{Z}}_2\oplus\mathbb{{Z}}_2$ & \\ 
 \quad & & \quad & \\ 
 \includegraphics[width=1in]{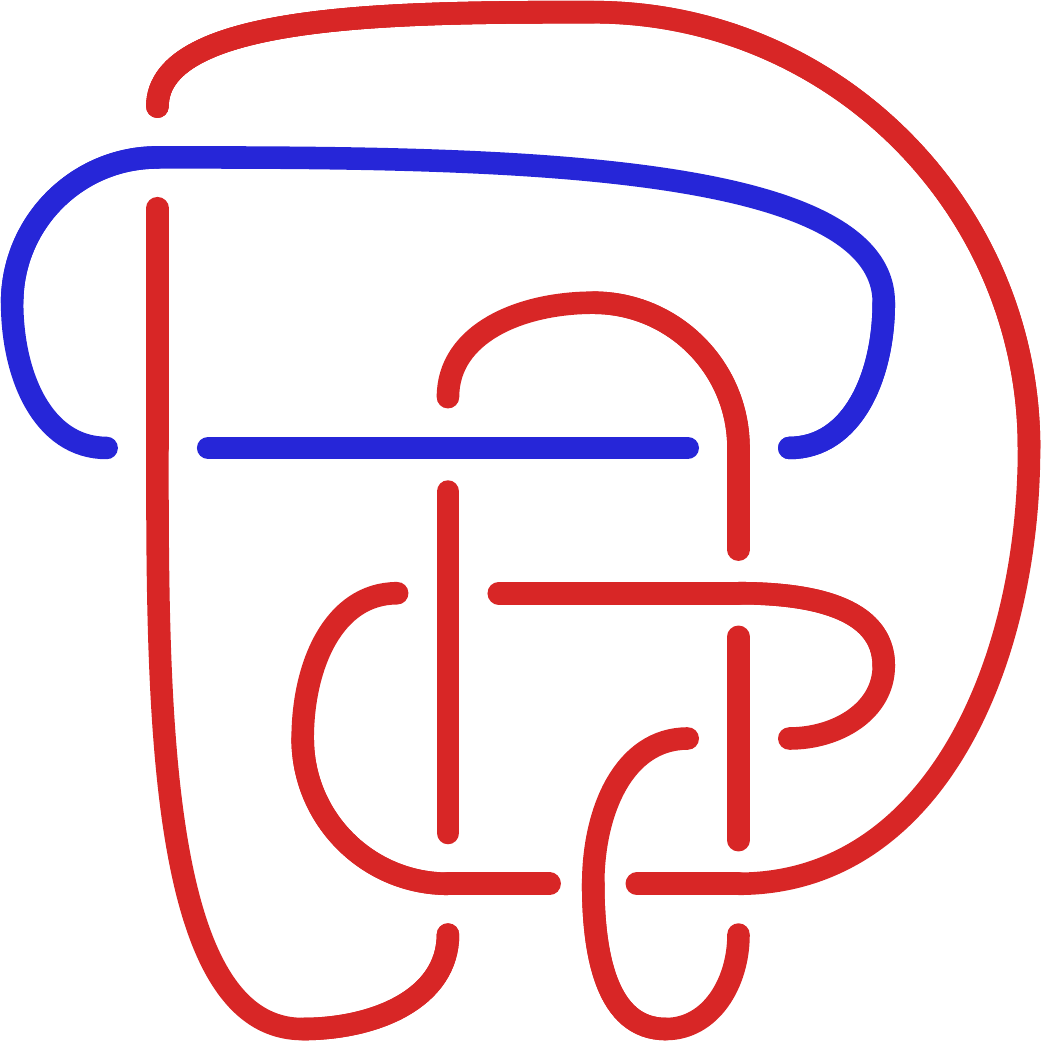}  & & \includegraphics[width=1in]{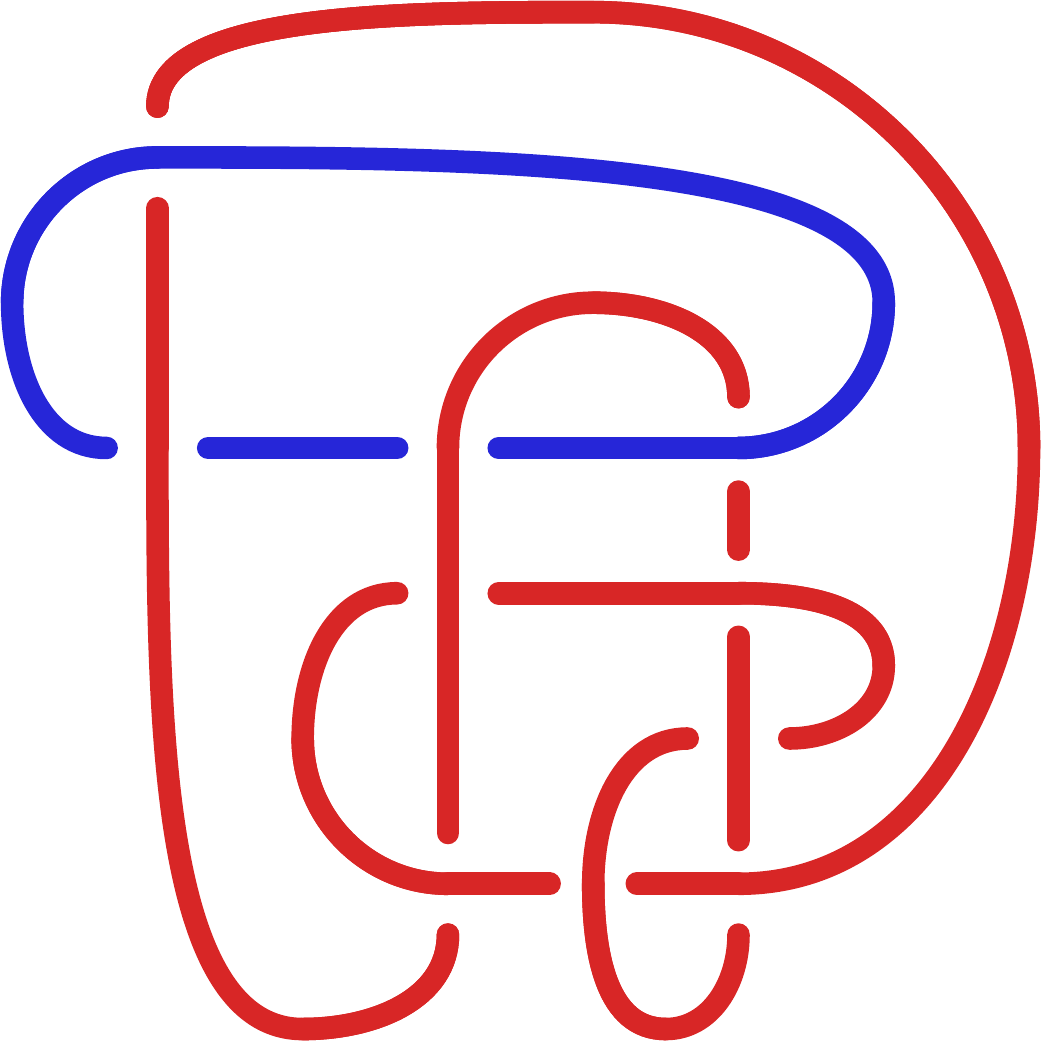} & \\ 
 \quad & & \quad & \\ 
 \hline  
\quad & \multirow{6}{*}{\Includegraphics[width=1.8in]{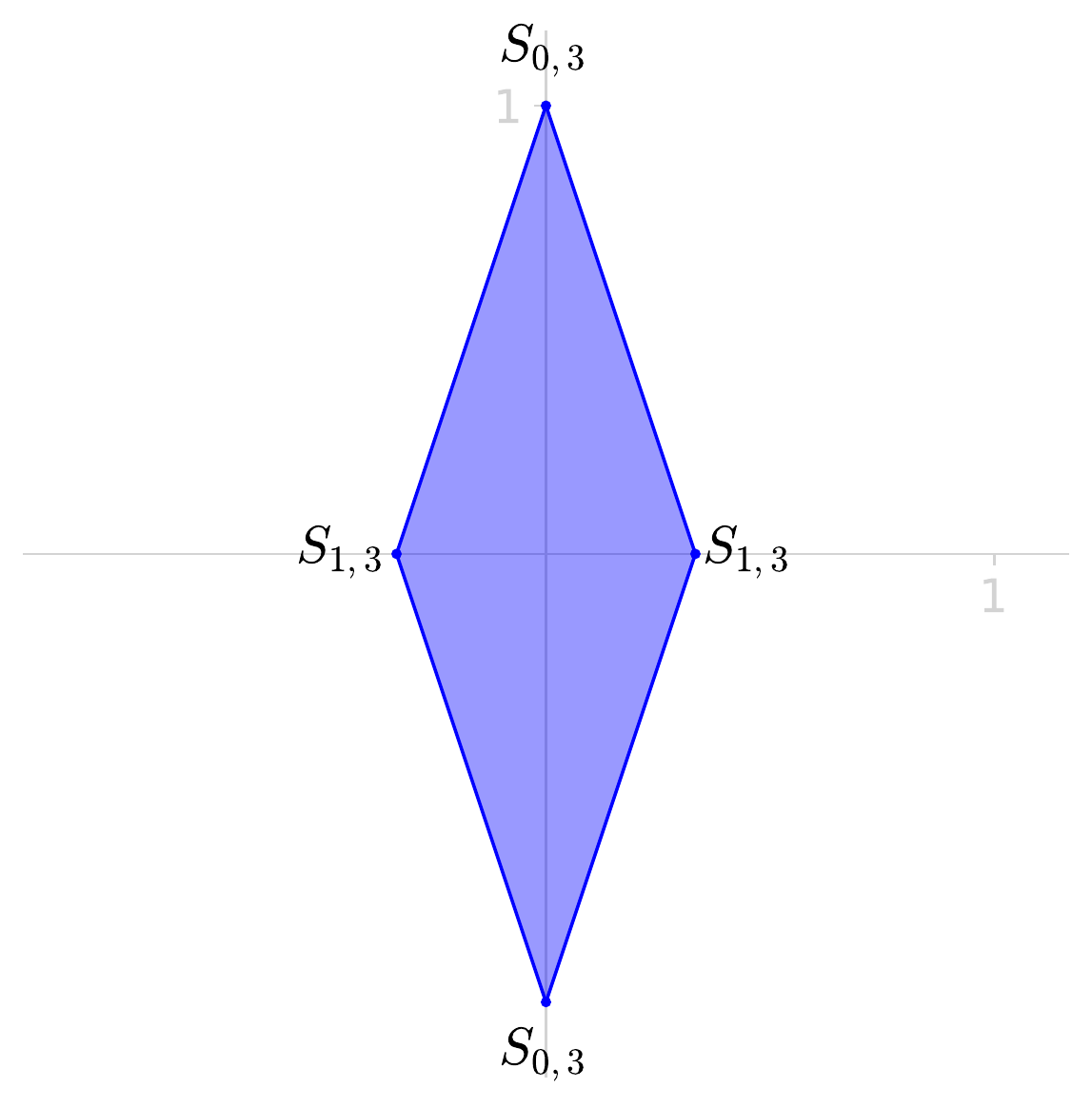}} & \quad & \multirow{6}{*}{\Includegraphics[width=1.8in]{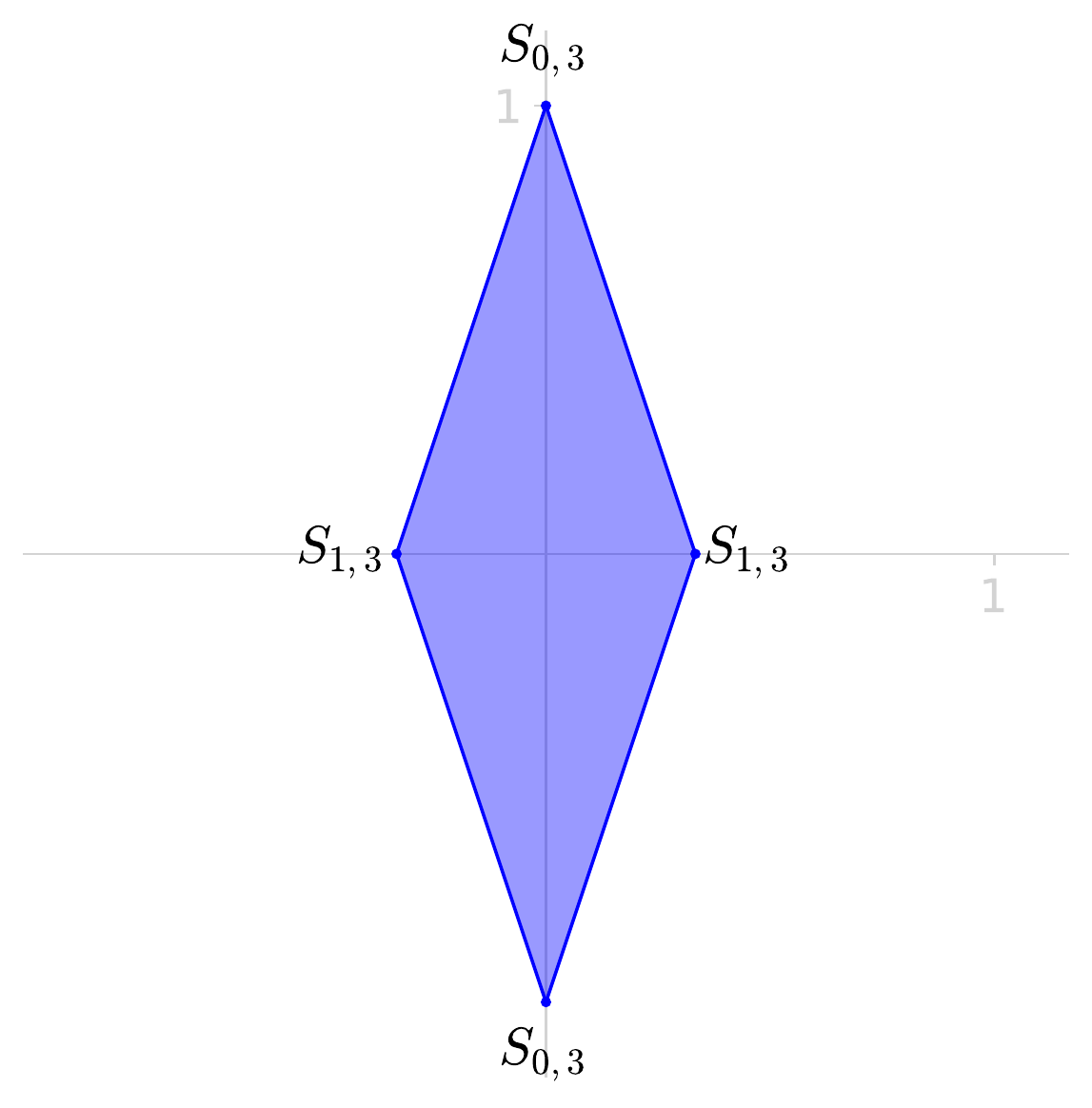}} \\ 
 $L=10^{{2}}_{{17}}$ & & $L=10^{{2}}_{{18}}$ & \\ 
 \quad & & \quad & \\ $\mathrm{Isom}(\mathbb{S}^3\setminus L) = \mathbb{{Z}}_2$ & & $\mathrm{Isom}(\mathbb{S}^3\setminus L) = \mathbb{{Z}}_2$ & \\ 
 \quad & & \quad & \\ 
 \includegraphics[width=1in]{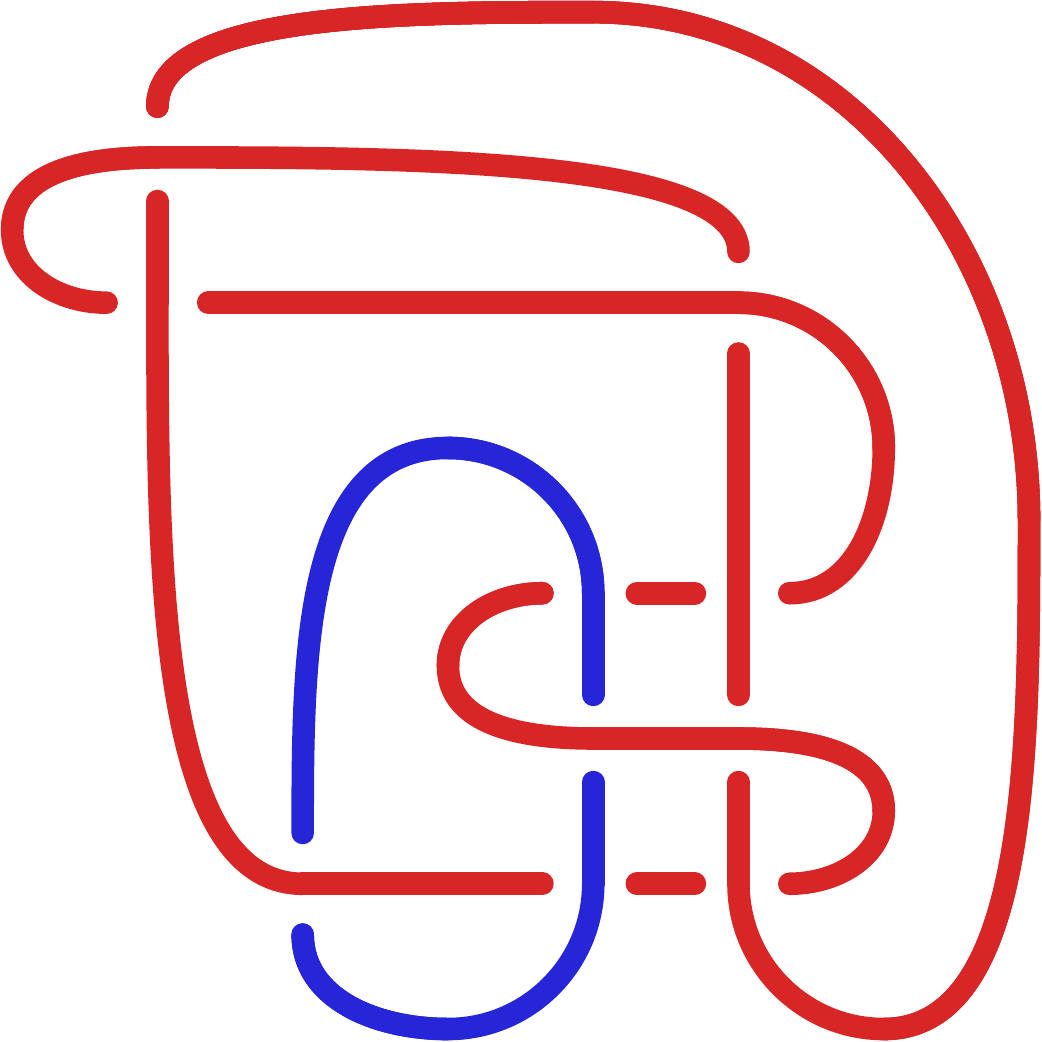}  & & \includegraphics[width=1in]{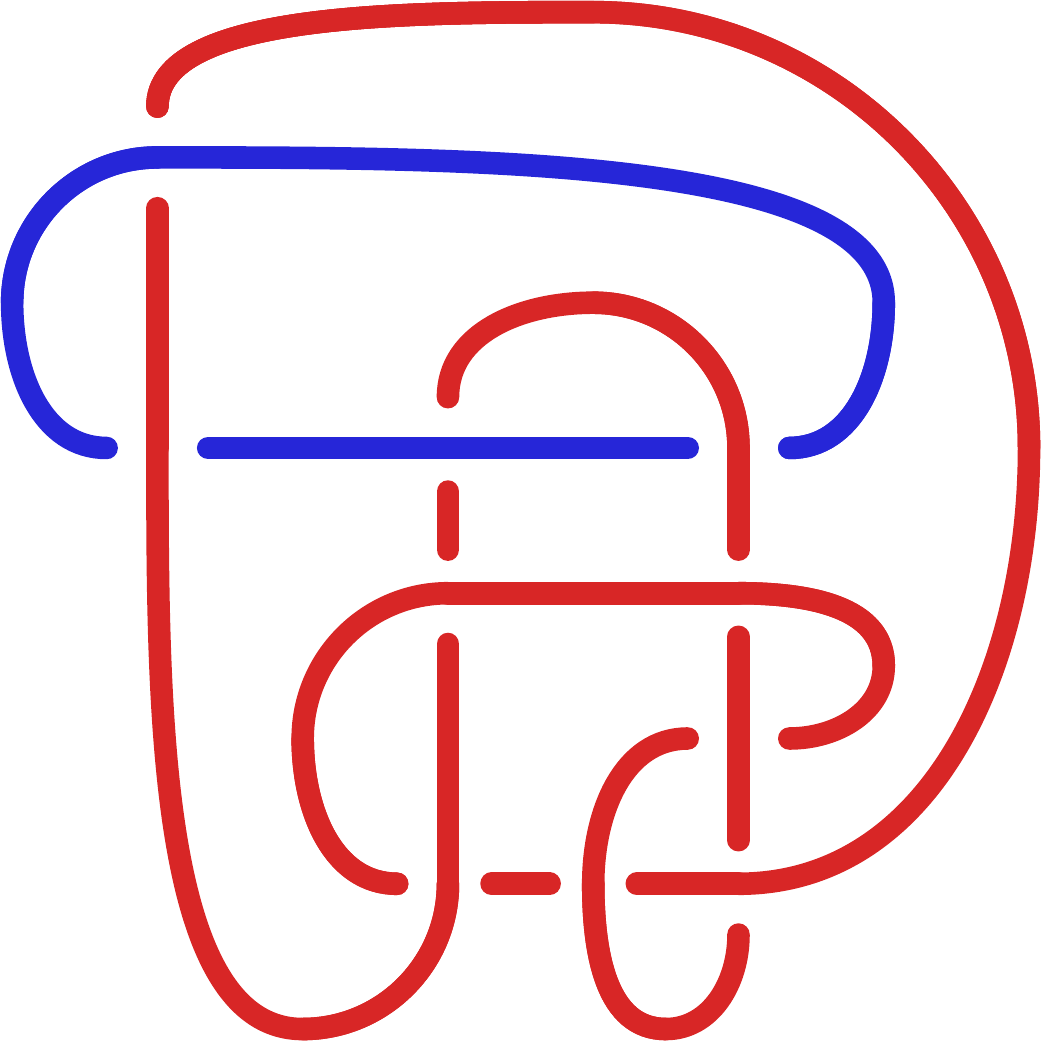} & \\ 
 \quad & & \quad & \\ 
 \hline  
\end{tabular}